\documentclass[11pt,leqno]{article}
\usepackage[all]{xy}
\usepackage{epsfig}

\usepackage{amsfonts,amssymb,amsthm,amsmath,amscd}
\usepackage[mathscr]{eucal}
\usepackage{mathrsfs}
\usepackage{hyperref}
\usepackage{url}

      \makeatletter

\usepackage[OT2,T1]{fontenc}
\DeclareSymbolFont{cyrletters}{OT2}{wncyr}{m}{n}
\DeclareMathSymbol{\Sha}{\mathalpha}{cyrletters}{"58}

\usepackage{titlesec}

\titleformat{\subsection}[runin]
{\normalfont\large\bfseries}{\thesubsection}{1em}{}

\titleformat{\subsubsection}[runin]
{\normalfont\large\bfseries}{\thesubsubsection}{1em}{}

\theoremstyle{plain}
\newtheorem{thm}[subsubsection]{Theorem}
\newtheorem{thm*}{Theorem}
\newtheorem{cor}[subsubsection]{Corollary}
\newtheorem{lem}[subsubsection]{Lemma}
\newtheorem{prop}[subsubsection]{Proposition}
\newtheorem{conj}[subsubsection]{Conjecture}

\theoremstyle{definition}
\newtheorem{defn}[subsubsection]{Definition}

\theoremstyle{remark}
\newtheorem{rem}[subsubsection]{Remark}

\numberwithin{equation}{subsubsection}

\newcommand{\N}{\mathbb N}
\newcommand{\Z}{\mathbb Z}
\newcommand{\Q}{\mathbb Q}
\newcommand{\R}{\mathbb R}
\newcommand{\C}{\mathbb C}
\newcommand{\A}{\mathbb A}

\newcommand{\F}{\mathbb F}
\newcommand{\Fp}{\mathbb{F}_{p}}
\newcommand{\Fpb}{\bar{\mathbb{F}}_{p}}
\newcommand{\Zp}{{\mathbb Z}_p}
\newcommand{\Qp}{{\mathbb Q}_p}
\newcommand{\Ql}{{\mathbb Q}_l}
\newcommand{\Qv}{{\mathbb Q}_v}
\newcommand{\Qb}{\overline{\mathbb Q}}
\newcommand{\Qpb}{\overline{\mathbb{Q}}_p}
\newcommand{\Qlb}{\overline{\mathbb{Q}}_l}
\newcommand{\Qvb}{\overline{\mathbb{Q}}_v}
\newcommand{\kb}{\overline{k}} %algebraic closure of k
\newcommand{\Lb}{\overline{L}} %algebraic closure of L=Frac W(k)
\newcommand{\Qpnr}{\Q_p^{\mathrm{ur}}} %maximal unramified extn of Qp
\newcommand{\Zpnr}{\Z_p^{\mathrm{ur}}}
\newcommand{\nr}{\mathrm{ur}} %(maxiam) unramified
\newcommand{\Gal}{\mathrm{Gal}} 

\newcommand{\Hom}{\mathrm{Hom}}
\newcommand{\Aut}{\mathrm{Aut}}
\newcommand{\im}{\mathrm{im}}
\newcommand{\Nm}{\mathrm{N}} %norm
\newcommand{\Tr}{\mathrm{tr}} %trace
\newcommand{\et}{\text{\'et}} %etale
\newcommand{\Int }{\mathrm{Int}} %interior automorphism
\newcommand{\Inn}{\mathrm{Int}} %Inner automorphism
\newcommand{\Cent}{\mathrm{Cent}} %Centralizer
\newcommand{\lisom}{\stackrel{\sim}{\longrightarrow}} %long isomorphism
\newcommand{\isom}{\stackrel{\sim}{\rightarrow}} %isomorphism

\newcommand{\tors}{\mathrm{tors}} %torsion

\newcommand{\ra}{\rightarrow}
\newcommand{\lra}{\longrightarrow}
\newcommand{\hra}{\hookrightarrow}
\newcommand{\thra}{\twoheadrightarrow}

\newcommand{\Spec}{\mathrm{Spec}}
\newcommand{\Res}{\mathrm{Res}} %Restriction of Scalars
\newcommand{\cO}{\mathcal{O}} %mathcal O
\newcommand{\sS}{\mathscr{S}} %integral canonical model of Shimura variety
\newcommand{\dS}{\mathbb{S}} %Deligne torus
\newcommand{\Sh}{\mathrm{Sh}} %Shimura variety
\newcommand{\sF}{\mathscr{F}} %automorphic sheaf

\newcommand{\mrU}{\mathrm{U}}
\newcommand{\SU}{\mathrm{SU}}
\newcommand{\Sp}{\mathrm{Sp}}
\newcommand{\SO}{\mathrm{SO}}
\newcommand{\Spin}{\mathrm{Spin}}
\newcommand{\Gm}{\mathbb{G}_{\mathrm{m}}} 
\newcommand{\ab}{\mathrm{ab}} %maximal abelian quotient
\newcommand{\ad}{\mathrm{ad}} %adjoint
\newcommand{\der}{\mathrm{der}} %derived
\newcommand{\uc}{\mathrm{sc}} %universal covering

\mathchardef\mhyphen="2D

\newcommand{\mcB}{\mathcal{B}} %Bruhat-Tits building
\newcommand{\mcA}{\mathcal{A}} %apartment of BT building
\newcommand{\mbfa}{\mathbf{a}} %alcove of BT building
\newcommand{\mbff}{\mathbf{f}} %facet of BT building
\newcommand{\mbfo}{\mathbf{0}} %special point of BT building
\newcommand{\mbfv}{\mathbf{v}}
\newcommand{\mcG}{\mathcal{G}} %group scheme in BT building
\newcommand{\mro}{\mathrm{o}}
\newcommand{\mbfK}{\mathbf{K}} 
\newcommand{\mbfKt}{\tilde{\mathbf{K}}} 

\newcommand{\mfk}{\mathfrak{k}}

\newcommand{\fG}{\mathfrak{G}} %neutral Gerbe attached to G
\newcommand{\fE}{\mathfrak{E}} %Galois gerbe E 
\newcommand{\fD}{\mathfrak{D}} %Dieudonne gerb
\newcommand{\fQ}{\mathfrak{Q}} %quasi-motivic Galois gerb
\newcommand{\fP}{\mathfrak{P}} %pseudo-motivic groupoid
\newcommand{\Adm}{\mathrm{Adm}} %admissible subset in the sense of Kottwitz-Rapoport

\newcommand{\bfab}{\mathbf{ab}}

\newcommand{\sI}{\mathscr{I}} %isogeny class
\newcommand{\uvA}{\mathcal{A}}
\newcommand{\xb}{\bar{x}} %geometric point induced by x
\newcommand{\cris}{\mathrm{cris}}
\newcommand{\dR}{\mathrm{dR}}
\newcommand{\Betti}{\mathrm{B}}
\newcommand{\Fr}{\mathrm{Fr}}

\newcommand{\ST}{\mathrm{ST}}
\newcommand{\elp}{\mathrm{ell}}
\newcommand{\loccit}{\textit{loc. cit.}}
\newcommand{\WgF}{W(\bar{K}/F)}

\begin{document}

\title{Galois gerbs and Lefschetz number formula for Shimura varieties of Hodge type}
\author{Dong Uk Lee}
\date{}
\maketitle

\begin{abstract}
For any Shimura variety of Hodge type with hyperspecial level at a prime $p$ and any automorphic lisse sheaf on it, we prove a formula, conjectured by Kottwitz \cite{Kottwitz90}, for the Lefschetz number of an arbitrary Frobenius-twisted Hecke correspondence acting on the compactly supported \'etale cohomology and verify another conjecture of Kottwitz \cite{Kottwitz90} on stabilization of that formula.
The main ingredients of our proof of the formula are a recent work of Kisin \cite{Kisin17} on Langlands-Rapoport conjecture and the theory of Galois gerbs developed by Langlands and Rapoport \cite{LR87}. Especially, we use the Galois gerb theory to establish an effectivity criterion of Kottwitz triple, and mimic the arguments of Langlands and Rapoport of deriving the Kottwitz formula from their conjectural description of the $\Fpb$-point set of Shimura variety (Langlands-Rapoport conjecture). 
We do not assume that the derived group is simply connected, and also obtain partial results at (special) parahoric levels under some condition at $p$. For that, in the first part of our work we extend the results of Langlands and Rapoport to such general cases. 
\end{abstract}

\tableofcontents

\section{Introduction}

\subsection{Statement of the main results}

The main results of this work are the following two descriptions of the Lefschetz numbers of Frobenius-twisted Hecke correspondences on the compactly supported cohomology of Shimura varieties of Hodge type with hyperspecial level, which were conjectured by Kottwitz  \cite{Kottwitz90}, (3.1) and Thm. 7.2. 
The first one is a consequence of the second one.

Fix two distinct primes $p$, $l$ of $\Q$.
Let $(G,X)$ be a Shimura datum of Hodge type, $K_p$ a hyperspecial subgroup of $G(\Qp)$, $K^p$ a (sufficiently small) compact open subgroup of $G(\A_f^p)$; put $K=K_pK^p$. Let $Sh_K(G,X)$ be (the canonical model over the reflex field $E(G,X)$ of) the associated Shimura variety, and $\sF_K$ a $\lambda$-adic lisse sheaf on it defined by a finite-dimensional algebraic representation $\xi$ of $G$ over a number field $L$ and choice of a place $\lambda$ of $L$ over $l$. For $i\in\N_{\geq0}$, let $H_c^i(Sh_K(G,X)_{\Qb},\sF_K)$ be the compactly supported cohomology of $\sF_K$.
Let $\wp$ be a prime of $E=E(G,X)$ above $p$ and $\Phi\in \Gal(\Qpb/E_{\wp})$ be a geometric Frobenius for an embedding $E_{\wp}\hookrightarrow \Qpb$. It is known \cite{Kisin10}, \cite{Vasiu99} that $Sh_K(G,X)$ has a canonical smooth integral model over the ring of integers $\mathcal{O}_{E_{\wp}}$ and $\sF_K$ extends over it.

\begin{thm} \label{thm_intro:Kottwitz_conj}
Assume that the center $Z(G)$ of $G$ has same ranks over $\Q$ and $\R$.%%
\footnote{In this work, for Hodge-type $(G,X)$, we always assume that $G$ is the smallest algebraic $\Q$-group such that every $h\in X$ factors through $G_{\R}$; then this assumption holds.}
For every $f^p$ in the Hecke algebra $\mathcal{H}(G(\A_f^p)/\!\!/ K^p)$, there exists $m(f^p)\in\N$, depending on $f^p$, such that for each $m\geq m(f^p)$, the Lefschetz number of $\Phi^m\times f^p$ is given by
\[\sum_i (-1)^i \mathrm{tr}(\Phi^m\times f^p | H_c^i(Sh_K(G,X)_{\Qb},\sF_K))=\sum_{\underline{H}\in \mathscr{E}_{\mathrm{ell}}(G)} \iota(G,\underline{H}) \ST_{\elp}^{H_1}(f^{H_1}),\]
where the right sum is over a set $\mathscr{E}_{\mathrm{ell}}(G)$ of representatives of the isomorphism classes of elliptic endoscopic data of $G$ with a choice of a $z$-pair $(H_1,\xi_1)$ for each $\underline{H}\in \mathscr{E}_{\mathrm{ell}}(G)$,  and $\ST_{\elp}^{H_1}(f^{H_1})$ is the elliptic part of the geometric side of the stable trace formula for a suitable function $f^{H_1}$ on $H_1(\A)$, $\iota(G,\underline{H})$ is a certain constant depending only on the pair $(G,\underline{H})$.

Moreover, if the adjoint group $G^{\ad}$ is $\Q$-anisotropic or $f^p$ is the identity, we can take $m(f^p)$ to be $1$.
\end{thm}

In particular, we obtain a formula for the local zeta functions of the same Shimura varieties. 
We refer to Thm. \ref{thm:EP_GS_STF} for explanation of the objects appearing here. 
We remark that this formula is the first step in the Langlands's program for establishing the celebrated conjecture that the Hasse-Weil zeta function of an arbitrary Shimura variety is a product of automorphic $L$-functions. 
This formula is deduced from the following formula, via the machinery of stabilization of the trace formula.

\begin{thm} \label{thm_intro:Kottwitz_formula} 
For any $f^p$ in the Hecke algebra $\mathcal{H}(G(\A_f^p)/\!\!/ K^p)$, there exists $m(f^p)\in\N$, depending on $f^p$, such that for each $m\geq m(f^p)$, the Lefschetz number of $\Phi^m\times f^p$ is given by
\begin{align} \label{eq_intro:Kottwitz_formula}
\sum_{i}(-1)^i\mathrm{tr}( & \Phi^m\times f^p | H^i_c(Sh_{K}(G,X)_{\Qb},\sF_K)) \\
& = \sum_{(\gamma_0;\gamma,\delta)} c(\gamma_0;\gamma,\delta)\cdot \mathrm{O}_{\gamma}(f^p)\cdot \mathrm{TO}_{\delta}(\phi_p) \cdot \mathrm{tr}\xi(\gamma_0), \nonumber
\end{align}
where the sum is over a set of representatives $(\gamma_0;\gamma,\delta)$ of \emph{all} equivalence classes of \emph{stable} Kottwitz triples of level $n$ having trivial Kottwitz invariant and $c(\gamma_0;\gamma,\delta)$ is a certain constant (defined in terms of Galois cohomology). Moreover, if $G^{\ad}$ is $\Q$-anisotropic or $f^p$ is the identity, we can take $m(f^p)$ to be $1$.
\end{thm}

See Thm. \ref{thm:Kottwitz_formula:Kisin} for a precise description of the terms involved and more details. We emphasize that in both theorems, we do \emph{not} assume that the derived group $G^{\der}$ of $G$ is simply connected. In such case that $G^{\der}=G^{\uc}$, $G^{\uc}$ being the simply connected cover of $G^{\der}$, the expression on the right-hand side of (\ref{eq_intro:Kottwitz_formula}) equals the one in the Kottwitz's conjecture \cite[(3.1)]{Kottwitz90}.
Regarding this generality, we remark on two points. First, we take the sum only over certain Kottwitz triples which we call \emph{stable} (Def. \ref{defn:stable_Kottwitz_triple}). This notion of stableness has the same meaning as in the stable conjugacy (e.g. every Kottwitz triple is stable in our sense if $G^{\der}=G^{\uc}$), so is a natural condition in this context. But there is also some subtle point at $p$ (cf. Remark \ref{rem:Kottwitz_triples} and \autoref{subsubsec:w-stable_sigma-conjugacy}).
Secondly, the condition of ``\textit{having trivial Kottwitz invariant}'' should be also taken with a grain of salt: again unless $G^{\der}=G^{\uc}$, Kottwitz invariant itself is not a notion well-defined by a (stable) Kottwitz triple alone, although its vanishing is so (cf. \autoref{subsubsec:Kottwitz_invariant}).

To have a complete picture of the Hasse-Weil zeta function (and for applications to construction of Galois representations, cf. \cite{ScholzeShin13}), one also needs similar descriptions at bad reductions (cf. \cite[$\S$10]{Rapoport05}, \cite[Conj.4.30, 4.31]{Haines14}). For that, we remark that we also obtain some partial results towards Thm. \ref{thm_intro:Kottwitz_formula} for (special) parahoric levels under the assumption that $G_{\Qp}$ is quasi-split and is tamely ramified, and together with such results, the same methods proving the above theorems (based on \cite{Kottwitz84b}, \cite{LR87}) should allow the wanted descriptions, once those results in \cite{Kisin17} which were geometric ingredients in our arguments (and further some relevant results in harmonic analysis) are also available for general parahoric levels.

\subsection{Strategy of proof}

We begin with some general comments on our proofs. 
First, Thm. \ref{thm_intro:Kottwitz_conj} is a consequence of the stabilization of the right-hand side of the identity (\ref{eq_intro:Kottwitz_formula}), the key tools in this machinery being the endoscopic transfer conjectures which are now established by Ngo, Waldspurger. Such stabilization was carried out by Kottwitz \cite[$\S$4-$\S$7]{Kottwitz90}, \cite{Kottwitz10} in the case $G^{\der}=G^{\uc}$. We follow closely his arguments which however need occasionally modifications due to our generality. Our main contribution is the proof of Thm. \ref{thm_intro:Kottwitz_formula} and now we give a quick overview of the historical origin of our strategy. 
After previous successes with examples (by Eichler, Shimura, Kuga, Sato, and Ihara), a systematic approach to a formula for the Lefschetz number in question (and its stabilization) was conceived by Langlands based on Lefschetz-Grothendieck trace formula, and a set of ideas and techniques was developed by Langlands \cite{Langlands76}, \cite{Langlands79} and Kottwitz \cite{Kottwitz84b} (``Langlands-Kottwitz method''). Meanwhile, it was also realized that this method itself needed some refinement. Then, Langlands and Rapoport \cite{LR87} recast the problem by developing the theory of Galois gerbs, modeled on the theory of Grothendieck motives. Its purpose was to provide a general framework for formulating a conjectural description of the $\Fpb$-point set of Shimura variety which is \emph{precise} enough to allow one to derive the Kottwitz formula from it.
 
Our proof of the Kottwitz formula (\ref{eq_intro:Kottwitz_formula}) imitates these arguments of Langlands and Rapoport of deriving the Kottwitz formula from their conjecture. As such, their theory of Galois gerbs is a major ingredient in this work. Another essential ingredient is a recent work of Kisin \cite{Kisin17} on the aforementioned conjecture of Langlands and Rapoport.
Previously, Kottwitz \cite{Kottwitz92} proved the formula (\ref{eq_intro:Kottwitz_formula}) in PEL-type cases (of simple Lie type $A$ or $C$) by a method which is based on the Honda-Tate theory. This method however cannot be applied in general Hodge-type situations, and indeed our proof is different from his proof in his PEL-type cases.

In the next, we give more detailed discussion of the idea of our proof. We begin by introducing the work of Langlands and Rapoport \cite{LR87}.

\subsubsection{Langlands-Rapoport conjecture}
The conjecture of Langlands and Rapoport, which was stated in \cite{LR87} and a significant progress towards which was recently made by Kisin \cite{Kisin17}, aims to give a group-theoretic description of the set of $\Fpb$-points of the mod-$p$ reduction of a Shimura variety, as provided with Hecke operators and Frobenius automorphism. 

Let $(G,X)$ be a (general) Shimura datum and $\mbfK^p\subset G(\A_f^p)$, $\mbfK_p\subset G(\Qp)$ compact open subgroups, and set $\mbfK:=\mbfK_p\times \mbfK^p\subset G(\A_f)$.
The original conjecture mainly concerned good reduction cases, where $\mbfK_p$ is \textit{hyperspecial}, i.e. $\mbfK_p=G_{\Zp}(\Zp)$ for a reductive $\Zp$-group scheme with generic fiber $G_{\Qp}$. 
We also choose a place $\wp$ of the reflex field $E(G,X)$ dividing $p$, and let $\cO_{\wp}$, $\kappa(\wp)$ denote respectively the integer ring of the local field $E(G,X)_{\wp}$ and its residue field.
Then, Langlands and Rapoport conjectured that there exists an integral model $\sS_{\mbfK_p}(G,X)$ of $\Sh_K(G,X)$ over $\cO_{\wp}$, for which there is a bijection 
\begin{equation} \label{eqn:LRconj-ver1}
\sS_{\mbfK_p}(G,X)(\Fpb)\isom \bigsqcup_{[\phi]}S(\phi)
\end{equation}
where
\[S(\phi)=\varprojlim_{\mbfK^p} I_{\phi}(\Q)\backslash X_p(\phi)\times X^p(\phi)/\mbfK^p.\]
To give an idea of what these objects look like, suppose that our Shimura variety $\Sh_{\mbfK_p}(G,X)$ is a moduli space of abelian varieties endowed with a certain prescribed set of additional structures defined by $G$ (called $G$-structure, for short), and that there exists an integral model whose reduction affords a similar moduli description (at least over $\Fpb$). Then, roughly speaking, each $\phi$ is supposed to correspond to an isogeny class of abelian varieties with $G$-structure, and the set $S(\phi)$ is to parameterize the isomorphism classes in the corresponding isogeny class. 
More precisely, $X_p(\phi)$ and $X^p(\phi)$ should correspond, repsectively, to the isogenies of $p$-power order and prime-to-$p$ order (say, leaving from a fixed member in the isogeny class $\phi$) preserving $G$-structure, and $X_p(\phi)$ can be also identified with a suitable affine Deligne-Lusztig variety $X(\{\mu_X\},b)_{\mbfK_p}$. 
The term $I_{\phi}(\Q)$ is to be the automorphism group of the isogeny class attached to $\phi$, and thus acts naturally on $X_p(\phi)$ and $X^p(\phi)$.
Moreover, each of the sets $S(\phi)$ carries a compatible action of the group $G(\A_f^p)$ and the Frobenius automorphism $\Phi$ (which is an element of $\Gal(\Fpb/\kappa(\wp))$), and the bijection (\ref{eqn:LRconj-ver1}) should be compatible with these actions.

When it comes to the precise definition, the most tricky object is the parameter $\phi$. Its precise definition makes use of the language of \textit{Galois gerb}: A Galois gerb is a gerb, in the sense of \textit{Cohomologie non ab\'elienne} \`a la Giraud, on the \'etale site of a field (with choice of a neutralizing object). This is motivated by the fact \cite{Milne94} that there is a well-determined class of Shimura varieties (i.e. Shimura varieties of abelian type) which, in characteristic zero, have a description of their point sets similar to (\ref{eqn:LRconj-ver1}) with the parameter $\phi$ being an abelian motive (constructed using absolute Hodge cycles). For $\Fpb$-points, the parameter $\phi$, called an \emph{admissible morphism}, is to represent ``a motive over $\Fpb$ with $G$-structure''. A proper definition of such object involves a Tannakian-theoretic description of the category of motives. But, the Tannakian category of Grothendieck motives over $\Fpb$, being non-neutral, is identified (after choice of a fiber functor) with the representation category of a certain Galois gerb (``motivic Galois gerb''), not an affine group scheme, and ``a motive over $\Fpb$ with $G$-structure'' should be a morphism from this motivic Galois gerb to the neutral Galois gerb attached to $G$.

\subsubsection{From Langlands-Rapoport conjecture to Kottwitz formula} \label{subsubsec:fromLRtoKF}
The conjectural description (\ref{eqn:LRconj-ver1}) allows one to obtain a manageable description of the fixed point set of any Hecke correspondence twisted by a Frobenius $\Phi^m$, and eventually a purely group-theoretic formula for its cardinality; in particular, it gives a formula for the cardinalities of the finite sets $\sS_{\mbfK}(G,X)(\F_{q^m})=[\sS_{\mbfK}(G,X)(\Fpb)]^{\Phi^m=\mathrm{id}}$ for each finite extension $\F_{q^m}$ of the base field $\F_q=\kappa(\wp)$ whose knowledge amounts to that of the local zeta function of $\sS_{\mbfK}(G,X)_{\kappa(\wp)}$. 

Such deduction arguments were provided by Kottwitz \cite{Kottwitz84b} and Langlands-Rapoport \cite{LR87}.
Here, we sketch its basic idea in the trivial correspondence case (cf. \cite{Milne92}).
For simplicity, we assume (only here) that the derived group $G^{\der}$ of $G$ is simply connected.
By an elementary argument \cite[$\S$1.4]{Kottwitz84b}, one readily sees that for each admissible morphism $\phi$, the corresponding fixed-point set $S(\phi)^{\Phi^m=\mathrm{Id}}$ breaks up further as a disjoint union of subsets $S(\phi,\epsilon)$ indexed by (the equivalence class of) a pair $(\phi,\epsilon)$, where $\epsilon$ is an automorphism of $\phi$ which exists as an element of $G(\Qb)$ and is to be regarded as a Frobenius descent datum of $\phi$: 
\[\sS_{\mbfK}(G,X)(\F_{q^m})\isom \bigsqcup_{[\phi]}S(\phi)^{\Phi^m=\mathrm{id}} = \bigsqcup_{[\phi,\epsilon]} S(\phi,\epsilon)\]
Among such pairs $(\phi,\epsilon)$, we are interested only in the pairs, called \emph{admissible}, satisfying some natural conditions that are necessary (but, not sufficient in general!) for the corresponding set $S(\phi,\epsilon)$ to be non-empty (Def. \ref{defn:admissible_pair}). 
Then, with any admissible pair $(\phi,\epsilon)$, one can associate a triple of group elements (\textit{Kottwitz triple})
\[(\gamma_0;\gamma=(\gamma_l)_{l\neq p},\delta)\in G(\Q)\times G(\A_f^p)\times G(L_n)\] 
satisfying certain compatibilities among themselves, where $L_n$ is the unramified extension of $\Qp$ of degree $n=m[\kappa(\wp):\Fp]$, and express the cardinality of the set $S(\phi,\epsilon)$ in terms of it as 
\[|S(\phi,\epsilon)|=\mathrm{vol}(G_{\gamma_0}(\Q)\backslash G_{\gamma_0}(\A_f))\cdot O_{\gamma}(f^p)\cdot TO_{\delta}(\phi_p),\] where $G_{\gamma_0}$ is the centralizer subgroup of $\gamma_0$ in $G$, $O_{\gamma}(f^p)$ and $TO_{\delta}(\phi_p)$ are some orbital integral and twisted orbital integral, respectively. Thus, the final formula for $|\sS_{\mbfK}(G,X)(\F_{q^m})|$ (for any sufficiently small $K^p$) takes the form of a sum, indexed by (the equivalence classes of) Kottwitz triples, of a product of quantities that can be defined purely group theoretically:
\begin{equation} \label{eqn:formulra_for_number_of_pts}
|\sS_{\mbfK}(G,X)(\F_{q^m})|=\sum_{(\gamma_0;\gamma,\delta)}\iota(\gamma_0;\gamma,\delta)\cdot \mathrm{vol}(G_{\gamma_0}(\Q)\backslash G_{\gamma_0}(\A_f))\cdot O_{\gamma}(f^p)\cdot TO_{\delta}(\phi_p),
\end{equation}
where $\iota(\gamma_0;\gamma,\delta)$ is by definition the number of equivalence classes of admissible pairs giving rise to a fixed Kottwitz triple $(\gamma_0;\gamma,\delta)$. 

With an explicit cohomological description for $\iota(\gamma_0;\gamma,\delta)$, this is the formula conjectured (and proved in certain PEL-type cases \cite[(19.6)]{Kottwitz92}) by Kottwitz \cite[(3.1)]{Kottwitz90}, except that here the sum is only over the \emph{effective} Kottwitz triples, namely those Kottwitz triples attached to some admissible pair $(\phi,\epsilon)$, while in the original formula, one takes \emph{all} Kottwitz triples with trivial Kottwitz invariant (it is known that every effective Kottwitz triple has trivial Kottwitz invariant).
We bring the reader's attention to this usuage of the terminology \emph{effectivity} and the possible confusion that a Kottwitz triple $(\gamma_0;\gamma,\delta)$ which is effective in this sense may not appear as a summation index ``effectively'' in the sum (\ref{eqn:formulra_for_number_of_pts}) (i.e. the corresponding summand could be zero): one could also have defined $\iota(\gamma_0;\gamma;\delta)$ to be zero if is not effective, but then it is not the definition $\iota(\gamma_0;\gamma,\delta):=|\ker[\ker^1(\Q,G_{\gamma_0})\rightarrow \ker^1(\Q,G)]|$ used by Kottwitz in his formula (which is always non-zero).
We note that to reconcile the formula (\ref{eqn:formulra_for_number_of_pts}) with the Kottwitz's formula \cite[(3.1)]{Kottwitz90}, it suffices to establish the following \emph{effectivity criterion of Kottwitz triple}: a Kottwitz triple with trivial Kottwitz invariant is effective (in the sense just defined) if the corresponding summand $O_{\gamma}(f^p)\cdot TO_{\delta}(\phi_p)$ is non-zero (one also needs the fact that $\iota(\gamma_0;\gamma,\delta)=|\ker[\ker^1(\Q,G_{\gamma_0})\rightarrow \ker^1(\Q,G)]|$ which is however proved in \cite[Satz 5.25]{LR87}).

\subsubsection{The works of Langlands and Rapoport \cite{LR87}  and of Kisin \cite{Kisin17}}
A substantial part of the work \cite{LR87} of Langlands and Rapoport is devoted to constructing all these objects and establishing their basic properties, especially those facts that are needed to carry out the deduction just sketched of the formula (\ref{eqn:formulra_for_number_of_pts}) from the (conjectural) description (\ref{eqn:LRconj-ver1}). 
The only missing ingredient in completing this deduction arguments was the effectivity criterion of Kottwitz triple stated above.
In fact, Langlands and Rapoport also suggest one such effectivity criterion (condition $\ast(\epsilon)$ in \cite[Satz.5.21]{LR87}), but fail to relate it to the more natural one of non-vanishing of $O_{\gamma}(f^p)\cdot TO_{\delta}(\phi_p)$.
In this work, we prove this effectivity criterion (Thm. \ref{thm:LR-Satz5.25}) and thereby (with a little more work allowing general Hecke correspondences) complete the arguments of Langlands and Rapoport of deriving the Kottwitz formula from their conjecture.%%
\footnote{In \cite{Milne92}, Milne uses this criterion which was at that time unjustified, for the same purpose of completing this deduction arguments, cf. see Remark \ref{rem:comments_on_Milne92}.}
We emphasize that our proof of the effectivity criterion (even though it could be formulated without the language of Galois gerbs) uses the full force of the theory of Galois gerbs and admissible morphisms from \cite{LR87}.

On the other hand, recently Kisin \cite{Kisin17} obtained a description of $\sS_{\mbfK_p}(G,X)(\Fpb)$ quite similar to (\ref{eqn:LRconj-ver1}). But, in his version of the description (\ref{eqn:LRconj-ver1}) (i.e. Theorem 0.3 of \textit{loc. cit.}), the action of the group $I_{\phi}(\Q)$ on $X_p(\phi)\times X^p(\phi)$ (which is somewhat artificial) is not the \emph{natural} one specified in the original conjecture.
An unfortunate consequence of this is that the deduction argument above \emph{per se} does not work for such imprecise description.%%
\footnote{In fact, this is an issue that has already been known for some time (namely, the kind of ambiguity as appearing in the Kisin's description becomes a problem when deriving a point-counting formula from it along the line of \cite{Langlands76}, \cite{Kottwitz84b}), and removing such ambiguity was one of the very motivations for Langlands and Rapoport introducing their conjecture \cite[p.116,line+15]{LR87}.}
In fact, in our proof of Kottwitz formula, we do \emph{not} use this weaker form (of Langlands-Rapoport conjecture) itself. Instead, we will just emulate the Langlands-Rapoport arguments above. Fortunately, this is possible due to the ``geometrical'' results of \cite{Kisin17} which are often quite strong (``geometrical'' in the sense that it can be formulated without using the language of Galois gerbs). Here, we mention one such result, \cite[Cor.2.2.5]{Kisin17} which we have dubbed  ``strong CM lifting theorem''  to distinguish it from the usual CM lifting theorem \cite[Thm.2.2.3]{Kisin17} (which only says that every isogeny class $\sI$ in a SV of Hodge type over $\Fpb$ contains a point which lifts to a CM point).

\subsection{Generalization of results of \cite{LR87}}
The main results of \cite{LR87} assume that the level subgroup $\mbfK_p$ is hyperspecial  (so, in particular that $G_{\Qp}$ is unramified) and that the derived group is simply connected.
To extend the scope of our method of proof beyond these cases, 
our first primary task in this article is to generalize their works to more general parahoric levels (so as to allow possibly bad reductions) under the assumption that $G_{\Qp}$ is quasi-split and spits over a tamely ramified extension of $\Qp$ (in fact, we will assume less than that, see Thm. \ref{thm:LR-Satz5.3} for a precise condition), and also to remove the restriction on the derived group; some of our results will further assume that $\mbfK_p$ is special maximal parahoric. 
When one works with this general level subgroups, the first notable change occurs in definitions, especially that of admissible morphism, where instead of a single affine Deligne-Lusztig variety $X(\{\mu_X\},b)_{\mbfK_p}$ which defined the set $X_p(\phi)$, one needs to use a \emph{finite union} of \emph{generalized} affine Deligne-Lusztig varieties $X(w,b)_{\mbfK_p}$:
\[X(\{\mu_X\},b)_{\mbfK_p}:=\bigsqcup_{w\in\Adm_{\mbfKt_p}(\{\mu_X\})} X(w,b)_{\mbfK_p}\]
Here, $\Adm_{\mbfKt_p}(\{\mu_X\})$ is a certain subset of the double coset $\mbfKt_p\backslash G(\mfk)/\mbfKt_p$ determined by the datum $(G,X)$. 
When $\mbfK_p$ is hyperspecial, this specializes to the previous definition.
Meanwhile, the conjecture (\ref{eqn:LRconj-ver1}) itself was extended by Rapoport \cite{Rapoport05} to cover general parahoric levels. 
Secondly, to work beyond the restriction of simply-connected derived group, we use the definition of admissible morphism generalized by Kisin for this purpose (cf. Def. \ref{defn:admissible_morphism}, \cite[(3.3.6)]{Kisin17}).

To keep statements short, we will call the following condition the \textit{Serre condition for the Shimura datum $(G,X)$}%%
\footnote{to distinguish this from two other similar conditions: first, from the original Serre condition which is applied to a $\Q$-torus $T$ endowed with a cocharacter $\mu\in X_{\ast}(T)$ (cf. Lemma \ref{lem:defn_of_psi_T,mu}), and secondly from the condition that $Z(G)$ has the same ranks over $\Q$ and $\R$ (equiv. the anisotropic kernel of $Z(G)$ remains anisotropic over $\R$); if $(G,X)$ is of Hodge type and $G$ is the Mumford-Tate group, it satisfies all these conditions.}
: \textit{the center $Z(G)$ of $G$ splits over a CM field and the weight homomorphism $w_X$ is defined over $\Q$.}
Note that this condition holds if the Shimura datum is of Hodge type.

Our main results in the first part are generalizations (sometimes, including improvements) of the key properties of admissible morphisms and admissible pairs:

%%%%%%%%%%%%%%%%%%%%
%%%%%%%%%%%%%%%%%%%%
\begin{thm} \label{thm_intro:1st_Main_thm} Let $p>2$ be a rational prime.
Let $(G,X)$ be a Shimura datum satisfying the Serre condition. Assume that $G$ is of classical Lie type, and that $G_{\Qp}$ is quasi-split and splits over a tamely ramified extension of $\Qp$. Let $\mbfK_p$ be a parahoric subgroup of $G(\Qp)$. Then, we have the followings.

(1) Any admissible morphism $\phi:\fP\rightarrow \fG_G$ is \emph{special}, namely there exists a special Shimura sub-datum $(T,h)$ and $g\in G(\Qb)$ such that $\Int g\circ\phi=i\circ\psi_{T,h}$, where $i:\fG_T\rightarrow \fG_G$ is the canonical morphism of neutral Galois gerbs induced by the inclusion $T\hra G$. If $\mbfK_p$ is special maximal parahoric, then every such morphism $i\circ\psi_{T,h}$ is admissible.
Every admissible pair $(\phi,\epsilon)$ is conjugate to $(i\circ\psi_{T,h},\epsilon'\in T(\Q))$ for a special Shimura sub-datum $(T,h)$.

(2) Suppose that $\mbfK_p$ is special maximal parahoric. For any $\gamma_0\in G(\Q)$ that is elliptic over $\R$, there exists an admissible pair $(\phi,\gamma_0)$ if and only if there exists $\epsilon\in G(\Q)$ stably conjugate to $\gamma_0$ and satisfying condition $\ast(\epsilon)$ of \autoref{subsubsec:pre-Kottwitz_triple}. If the latter condition holds, there exists a $\mbfK_p$-effective admissible pair $(\phi,\epsilon)$ with $\epsilon$ stably conjugate to $\gamma_0^t$ for some $t\in\N$.
\end{thm}

The tame condition on $G_{\Qp}$ can be relaxed significantly (see Thm. \ref{thm:LR-Satz5.3} for a precise condition).
These statements (and their counterpart results in \cite{LR87}) are found respectively in Theorem \ref{thm:LR-Satz5.3}, Lemma \ref{lem:LR-Lemma5.2}, \ref{lem:LR-Lemma5.23} (for (1)) and Theorem \ref{thm:LR-Satz5.21} (for (2)).
The statement (1) and the first claim of (2) are generalizations of the same results in \textit{loc. cit.}, while the second claim of statement (2) is new.
We also remark that the condition $\ast(\epsilon)$ in (2) is our generalization of the original condition in \cite[Satz 5.21]{LR87} of the same name $\ast(\epsilon)$, in our general set-up that the level subgroup is parahoric and $G_{\Qp}$ is not necessarily unramified nor $G^{\der}$ is simply connected.

The statement (1) is a fundamental fact about admissible morphisms, and underlies the CM lifting theorem \cite[Thm. 0.4]{Kisin17} (in the hyperspecial level case) that every isogeny class in $\sS_{\mbfK_p}(G,X)(\Fpb)$ contains a point that is the reduction of a special(=CM) point.

The following theorem is the first version of the aforementioned effectivity criterion of Kottwitz triple which is established for the first time here (even in the original set-up of \cite{LR87}, cf. Remark \ref{rem:comments_on_Milne92}).

%%%%%%%%%%%%%%%%%%%%
%%%%%%%%%%%%%%%%%%%%
\begin{thm} \label{thm_intro:2st_Main_thm} [Thm. \ref{thm:LR-Satz5.25}]
Under the same assumptions as in (2) of Thm. \ref{thm_intro:1st_Main_thm}, for every Kottwitz triple $(\gamma_0;(\gamma_l)_{l\neq p},\delta)$ with trivial Kottwitz invariant (Def. \ref{defn:Kottwitz_triple}), if $X(\{\mu_X\},\delta)_{\mbfK_p}\neq\emptyset$, there exists an admissible pair $(\phi,\epsilon)$ giving rise to it. There exists an explicit cohomological expression for the number of non-equivalent pairs $(\phi,\epsilon)$ producing a given triple $(\gamma_0;(\gamma_l)_{l\neq p},\delta)$.
\end{thm}

This is one of the key ingredients in our proof of the Kottwitz formula as well as in the Langlands-Rapoport's arguments of deriving it from Langlands-Rapoport conjecture.
There is also a similar, second version of effectivity criterion of Kottwitz triple (Thm. \ref{thm:LR-Satz5.25b2}) which will be needed in our (unconditional) proof of Thm. \ref{thm_intro:Kottwitz_formula}.

\subsubsection{Some further results and comments}

According to the discussion in \autoref{subsubsec:fromLRtoKF}, Langlands-Rapoport conjecture implies that
to any $\F_{p^n}$-point ($n$ being a multiple of $[\kappa(\wp):\F_p]$), one should be able to attach a Kottwitz triple $(\gamma_0;(\gamma_l)_{l\neq p},\delta)$ of level $n$ (i.e. with $\delta\in G(L_n)$). For Hodge-type Shimura varieties with hyperspecial level, this was done by Kisin \cite[Cor. 2.3.1]{Kisin17}.%%
\footnote{But for general parahoric $\mbfK_p$, without having control on the level, one can still attach a Kottwitz triple to any point that is the reduction of a CM point (thus in this case the triple is well-defined only up to powers).}
Then, the formula (\ref{eqn:formulra_for_number_of_pts}) implies that only those triples whose corresponding summation term is non-zero should be \textit{geometrically effective}, i.e. arises from an $\F_{p^n}$-point (where $n$ is the level of the triple). 
Thus, a natural question arises whether such necessary condition for geometric effectivity is also a sufficient condition. Closely related questions are which ($\R$-elliptic) stably conjugacy classes in $G(\Q)$ and which $\sigma$-conjugacy classes in $G(L_n)$  are ``effective'' (i.e. can be the classes of the elements $\gamma_0$ and $\delta$ attached to some $\F_{p^n}$-point, respectively). 
This question for (the stable conjugacy class of) an $\R$-elliptic rational element $\gamma_0\in G(\Q)$ can be regarded as Honda-Tate theorem in the context of Shimura varieties, while the question for $\delta$ is known as non-emptiness problem of Newton strata. 

For hyperspecial level, the proof of Theorem \ref{thm_intro:Kottwitz_formula} answers the first question affirmatively: the natural necessary condition is also sufficient. We also answer the second question too. 
%Here, we just state the result on non-emptiness of Newton strata.
For that, we consider Shimura varieties of Hodge type and fix an integral model $\sS_{\mbfK}$ over $\cO_{\wp}$ of the canonical model $\Sh_{\mbfK}(G,X)_{E_{\wp}}$ with the extension property that every $F$-point of $\Sh_{\mbfK}(G,X)$ for a finite extension $F$ of $E_{\wp}$ extends uniquely to $\sS_{\mbfK}$ over its local ring (for example, integral model constructed by a suitable normalization); see \cite{KisinPappas15} for a construction of such integral model in general parhoric levels.

\begin{cor} [Cor. \ref{cor:geom_effectivity_of_K-triple}] 
Keep the assumptions of Theorem \ref{thm_intro:1st_Main_thm}.
a Kottwitz triple $(\gamma_0;\gamma,\delta)$ of level $n$ with trivial Kottwitz invariant is \emph{geometrically effective} in the sense that it arises from a $\F_{p^n}$-valued point of $\sS$ as (\ref{eq:K-triple_for_isogeny_adm.pair}) if and only if $\mathrm{O}_{\gamma}(f^p)\cdot\mathrm{TO}_{\delta}(\phi_p)$ is non-zero. 
%In particular, an elliptic stable conjugacy class of $\gamma_0\in G(\Q)$ arises from a $\F_{p^n}$-valued point of $\sS$ for some $n\in\N$ if and only if there exists $\delta\in G(L_n)$ such that $\gamma_0$ is stably conjugate to $\Nm_n\delta$ and $\mathrm{TO}_{\delta}(\phi_p)\neq0$.
\end{cor}

%%%%%%%%%%%%%%%%%%%%
%%%%%%%%%%%%%%%%%%%%
\begin{thm} [Thm. \ref{thm:non-emptiness_of_NS}]  \label{thm:3rd_Main_thm}
Keep the assumptions of Theorem \ref{thm_intro:1st_Main_thm}, and further assume that $G_{\Qp}$ splits over a \emph{cyclic} tame extension of $\Qp$.
Let $\mbfK_p$ be a (not necessarily special) parahoric subgroup of $G(\Qp)$ and put $\mbfK=\mbfK_p\mbfK^p$ for a compact open subgroup $\mbfK^p$ of $G(\A_f^p)$. 

(1) Then, for any $[b]\in B(G_{\Qp},\{\mu_X\}$) (\autoref{subsubsec:B(G,{mu})}), there exists a special Shimura sub-datum $(T,h\in\Hom(\dS,T_{\R})\cap X)$ such that for any $g_f\in\mbfK^p$, the reduction in $\sS_{\mbfK}\otimes\Fpb$ of the special point $[h,g_f\cdot\mbfK]\in \Sh_{\mbfK}(G,X)(\Qb)$ has the $F$-isocrystal represented by $[b]$.

(2) The reduction $\sS_{\mbfK}(G,X)\otimes\Fpb$ has non-empty ordinary locus if and only if $\wp$ has absolute height one (i.e. $E(G,X)_{\wp}=\Q_p$). 
\end{thm}

This theorem generalizes Theorem 4.3.1 and Corollary 4.3.2 of \cite{Lee16} in the hyperspecial cases.

Next, we give some comments on possible generalizations of Thm. \ref{thm_intro:Kottwitz_conj} and Thm. \ref{thm_intro:Kottwitz_formula} for bad reductions. In view of the recent works \cite[Thm. 4.7.11]{KisinPappas15} and \cite{Zhou17}, it seems very likely that our methods and results allow us to establish also 
the conjecture on the semisimple zeta function (\cite{Rapoport05}, \cite{Haines14}) in special parahoric level case, when $G_{\Qp}$ is unramified and $p\nmid |\pi_1(G^{\der})|$, and as a result Thm. \ref{thm_intro:Kottwitz_conj} (for which one might also need certain results on the endoscopic transfer of stable Bernstein center, cf. \cite{Haines14}). But, we can already use our partial results to extend the scope of some previous results, for example, we can relax the ramification condition imposed on the PEL datum in the main result of \cite{Scholze13}. 

Finally, we make some comments on the various assumptions appearing in this article.
The running assumption, which will be effective except in some general discussions, is that $G_{\Qp}$ is quasi-split and splits over a tamely ramified extension of $\Qp$ (as mentioned above, for the latter condition, in fact, we only need some less restrictive one: see Thm. \ref{thm:LR-Satz5.3} and Prop. \ref{prop:existence_of_elliptic_tori_in_special_parahorics} for a precise condition). Equally universal assumption, although it is not needed for the important Theorem \ref{thm_intro:1st_Main_thm}, (1), is that $\mbfK_p$ is special maximal parahoric.
These two assumptions are somewhat forced on us because we follow closely the original line of arguments \cite{LR87} for our proofs.
We however remark that the two conditions that $G_{\Qp}$ splits over a tamely ramified extension and the level subgroup is special maxima parahoric 
%although it appears in many places (including all the statements of Theorem \ref{thm_intro:1st_Main_thm}) 
are imposed only via Lemma \ref{lem:LR-Lemma5.11}, or via Prop. \ref{prop:existence_of_elliptic_tori_in_special_parahorics},
where these assumptions are of more techinical nature rather than of intrinsic nature (e.g. in their proofs,
we verify certain statements by case-by-case analysis, where the possible cases are restricted or reduce substantially under such assumptions). It would be interesting to know if one can remove either of these conditions in these statements.

We remark that our sign convention up to Section 6 is the same as that of Langlands-Rapoport in \cite{LR87}; so for example, it is opposite to that of Kisin in \cite{Kisin17}. In Section 7, we also need to fix some more sign conventions, especially sign normalization of the local Langlands correspondence for tori (see Footnote \ref{ftn:LLC_sign} for our choice).

This article is organized as follows. 
The second section is a preliminary discussion, devoted to a review of some basic objects, including Kottwitz and Newton maps (defined for algebraic groups over $p$-adic fields), parahoric groups (in the Bruhat-Tits theory), extended affine Weyl groups, and $\{\mu\}$-admissible set. 
In the third section, we attempt to give a self-contained overview of the notions of Galois gerbs, the pseudo-motivic Galois gerb, admissible morphisms, Kottwitz triples, and admissible pairs, following closely the original source \cite{LR87}. At the same time, we generalize these notions and establish their properties (notably, admissible morphism and Kottwitz triple) beyond the original assumption that the group has simply connected derived group.
We also give a statement of the Langlands-Rapoport conjecture, as formulated by Rapoport \cite[$\S$8]{Rapoport05} so as to cover parahoric levels. Along the way, we extend results on special admissible morphisms to (special maximal) parahoric levels, under the assumption that $G_{\Qp}$ is quasi-split. 
In Section 4, we prove Theorem \ref{thm_intro:1st_Main_thm}, (1) above, namely that every admissible morphism is conjugate to a special admissible morphism (in our case of general parahoric level), as well as the fact that every admissible pair is nested in a special Shimura datum. 
For some other potential applications in mind, we spilt the proof into a few steps and formalize each of them into a separate proposition (incorporating slight improvements). 
The results in this section are in large part translations of the original results, except for generalizations to our setting and reorganization (with small improvements). However, some of the generalizations, e.g. the proof of Lemma \ref{lem:LR-Lemma5.11}, are rather non-trivial.
In Section 5, we establish the first version of effectivity criterion of Kottwitz triple (Thm. \ref{thm:LR-Satz5.25}). 
In Section 6, we prove the Kottwitz formula, in two ways, one assuming validity of Langlands-Rapoport conjecture (Thm. \ref{thm:Kottwitz_formula:LR}) and another one being unconditional (Thm. \ref{thm:Kottwitz_formula:Kisin}), which uses the second version of effectivity criterion of Kottwitz triple (Thm. \ref{thm:LR-Satz5.25b2}).
In the last section, we stabilize the Kottwitz formula, thereby proving Thm. \ref{thm_intro:Kottwitz_conj}.

In this work, a certain result in the Bruhat-Tits theory, whose hyperspecial case was already used critically in the original work \cite{LR87} (cf. \cite[Lem. A.0.4]{Lee16}), plays a key role. We provide its proof in Appendix \ref{sec:elliptic_tori_in_special_parahorics}. 
Also, generalizations to our setting (i.e. $G^{\der}\neq G^{\uc}$ in general) of previous arguments occur frequently throughout the entire work. For that, it is necessary to work with abelianized cohomology groups which are cohomology groups of complexes of tori. In appendix \ref{sec:abelianization_complex}, we collect basic facts about complexes of tori attached to connected reductive groups.

%%%%%%%%%%%%%%%%%%%%
\textbf{Acknowledgement}
A part of this work was supported by IBS-R003-D1. The author would like to thank M. Rapoport and C.-L. Chai for their interests in this work and encouragement.

%%%%%%%%%%%%%%%%%%%%
\textbf{Notations}

Throughout this paper, $\Qb$ denotes the algebraic closure of $\Q$ inside $\C$. 

%By our convention, a reductive group is not necessarily connected, even though it is so in most of the cases. When we discuss a non-necessarily connected group (e.g. the centralizer of an element in a (connected) reductive group), we make it explicit.

For a connected reductive group $G$ over a field, we let $G^{\uc}$ be the universal covering of its derived group $G^{\der}$, and for a (linear algebraic) group $G$, $Z(G)$, and $G^{\ad}$ denote its center, and the adjoint group $G/Z(G)$, respectively. 

For a group $I$ and an $I$-module $A$, we let $A_I$ denote the quotient group of $I$-coinvariants: $A_I=A/\langle ia-a\ |\ i\in I, a\in A\rangle$. For an element $a\in A$, we write $\underline{a}$ for the image of $a$ in $A_I$. In case of need for distinction, sometimes we write $\underline{a}_A$.

For a finitely generated abelian group $A$, we denote by $A_{\mathrm{tors}}$ its subgroup of torsion elements. For a locally compact abelian group $A$, we let $X^{\ast}(A):=\Hom_{\mathrm{cont}}(A,\C^{\times})$ (continuous character group) and $A^D:=\Hom_{\mathrm{cont}}(A,S^1)$ (Pontryagin dual).
For a (commutative) algebraic group $A$ over a field $F$, $X_{\ast}(A):=\Hom_{\mathrm{alg}}(\Gm,A)$, $X^{\ast}(A):=\Hom_{\mathrm{alg}}(A,\Gm)$. So, for a diagonalizable $\C$-group $A$, we have $\pi_0(A)^D=X^{\ast}(A)_{\mathrm{tors}}$ (with the embeddings $\Q/\Z\subset \R/Z=S^1\subset \C^{\times}$ understood).

In this article, the german letter $\mfk$ denotes the completion of the maximal unramified extension (in a fixed algebraic closure $\Qpb$) of $\Qp$, and for $n\in\N$ $L_n$ will denote $\mathrm{Frac}(W(\F_{p^n}))$. 
%On few occasions, $L$ also denotes a finite CM extension of $\Q$, in which case, if one also needs a notation for the completion of the maximal unramified extension of $\Qp$, we will use the german letter $\mfk$ (the original notation of Langlands-Rapoport). 

%%%%%%%%%%%%%%%%%%%%%%%%%%%%%%%%%%%%%%%%
%%%%%%%%%%%%%%%%%%%%%%%%%%%%%%%%%%%%%%%%

%\section{Parahoric subgroups and $\mu$-admissible set}
\section{Preliminaries: Parahoric subgroups and $\mu$-admissible set}

%%%%%%%%%%%%%%%%%%%%
%%%%%%%%%%%%%%%%%%%%

\subsection{Kottwitz maps and Newton map}

In this section, we briefly recall the definitions of the Kottwitz maps and the Newton map. We refer to \cite{Kottwitz97}, \cite{Kottwitz85}, \cite{RR96}, and references therein for further details.

%%%%%%%%%%%%%%%%%%%%
\subsubsection{The Kottwitz maps $w_G$, $v_G$, $\kappa_{G}$} \label{subsubsec:Kottwitz_hom}

%Let $L$ be a strictly henselian discrete valued field and set $I:=\Gal(\overline{L}/L)$.
Let $L$ be a complete discrete valued field with algebraically closed residue field and set $I:=\Gal(\overline{L}/L)$. 
For any connected reductive group $G$ over $L$, Kottwitz \cite[$\S$7]{Kottwitz97} constructs a group homomorphism
\[w_G:G(L)\rightarrow X^{\ast}(Z(\widehat{G})^{I})=\pi_1(G)_{I}.\]
Here, $\widehat{G}$ denotes the Langlands dual group of $G$, $\pi_1(G)=X_{\ast}(T)/\Sigma_{\alpha\in R^{\ast}}\Z\alpha^{\vee}$ is the fundamental group of $G$ (\`a la Borovoi) (i.e. the quotient of $X_{\ast}(T)$ for a maximal torus $T$ over $F$ of $G$ by the coroot lattice), and $\pi_1(G)_I$ is the (quotient) group of coinvariants of the $I$-module $\pi_1(G)$. This map $w_G$ is sometimes denoted by $\widetilde{\kappa}_G$, e.g. in \cite{Rapoport05}. When $G^{\der}$ is simply connected (so that $\pi_1(G)=X_{\ast}(G^{\ab})$ for $G^{\ab}=G/G^{\der}$), $w_G$ factors through $G^{\ab}$: $w_G=w_{G^{\ab}}\circ p_G$, where $p_G:G\rightarrow G^{\ab}$ is the natural projection \cite[7.4]{Kottwitz97}.

There is also a homomorphism 
\[v_G:G(L)\rightarrow \Hom(X_{\ast}(Z(\widehat{G}))^I,\Z)\] 
sending $g\in G(L)$ to the homomorphism $\chi\mapsto \mathrm{val}(\chi(g))$ from $X_{\ast}(Z(\widehat{G}))^I=\Hom_L(G,\Gm)$ to $\Z$, where $\mathrm{val}$ is the usual valuation on $L$, normalized so that uniformizing elements have valuation $1$. It is clear from this definition that $v_G=v_{G^{\ab}}\circ p_G$ for \emph{any} $G$ (i.e. not necessarily having the property $G^{\der}=G^{\uc}$).

There is the relation: 
\[v_G=q_G\circ w_G,\] 
where $q_G$ is the natural surjective map 
\[q_G:X^{\ast}(Z(\widehat{G})^{I})=X^{\ast}(Z(\widehat{G}))_I \rightarrow \Hom(X_{\ast}(Z(\widehat{G}))^I,\Z).\]
The kernel of $q_G$ is the torsion subgroup of $X^{\ast}(Z(\widehat{G}))_I$, i.e. $\Hom(X_{\ast}(Z(\widehat{G}))^I,\Z)\cong \pi_1(G)_I/\text{torsions}$; in particular, $q_G$ is an isomorphism if the coinvariant group $X^{\ast}(Z(\widehat{G}))_I$ is free (e.g. the $I$-module $X^{\ast}(Z(\widehat{G}))$ is trivial or more generally \textit{induced}, i.e. has a $\Z$-basis permuted by $I$).

For example, when $G$ is a torus $T$, we have $\langle \chi,w_T(t)\rangle= \mathrm{val}(\chi(t))$ for $ t\in T(L)$, $\chi\in X^{\ast}(T)^I$, where $\langle\ ,\ \rangle$ is the canonical pairing between $X^{\ast}(T)^I$ and $X_{\ast}(T)_I$.

Now suppose that $G$ is defined over a local field $F$, i.e. a finite extension of $\Qp$ (in a fixed algebraic closure $\Qpb$). Let $L$%%
\footnote{In this case that the residue field is $\Fpb$, we will write $\mfk$ for $L$ more often.}
be the completion of the maximal unramified extension $F^{\nr}$ of $F$ in $\Qpb$ 
and let $\sigma$ denote the Frobenius automorphism on $L$ which fixes $F$ and induces $x\mapsto x^q$ on the residue field of $L$ ($\cong\Fpb$), where the residue field of $F$ is $\F_q$. In this situation, the maps $v_{G_L}$, $w_{G_L}$ each induce notable maps.

First, as $w_{G_L}$ (and $v_{G_L}$ too) commutes with the action of $\Gal(F^{\nr}/F)$, by taking $H^{\mathrm{o}}(\Gal(F^{\nr}/F),-)$ on both sides of $w_{G_L}$, we obtain a homomorphism
\[\lambda_G:G(F)\rightarrow X^{\ast}(Z(\widehat{G})^I)^{\langle\sigma\rangle},\]
where $I\cong\Gal(\overline{F}/F^{\nr})$. This map is introduced in \cite[$\S$3]{Kottwitz84b} (cf. \cite[7.7]{Kottwitz97}) when $G$ is \emph{unramified} over $F$ (in which case the canonical action of $I$ on $Z(\widehat{G})$ is trivial) and used in \cite{LR87} (with the same notation) under the additional assumption $G^{\der}=G^{\uc}$ so that $w_{G_L}=v_{G_L}$. We remark that in our general set-up that $G_{\Qp}$ is not necessarily unramified nor $G^{\der}=G^{\uc}$, to achieve what $\lambda_G$ did in \cite{LR87}, we use $v_G$, or $w_G$ depending on the situation. 
%, so the target becomes $X^{\ast}(Z(\widehat{G}))^{\Gamma_F}$ ($\Gamma_F:=\Gal(\overline{F}/F)$) (and also $w_{G_L}=v_{G_L}$ as ).

Next, let $B(G)$ denote the set of $\sigma$-conjugacy classes:
\[B(G):=G(L)/\stackrel{\sigma}{\sim},\]
where two elements $b_1$, $b_2$ of $G(L)$ are said to be \textit{$\sigma$-conjugate}, denoted $b_1\stackrel{\sigma}{\sim} b_2$, if there exists $g\in G(L)$ such that $b_2=gb_1\sigma(g)^{-1}$. 
Then, $w_{G_L}$ induces a map of sets 
\begin{equation} \label{eq:kappa_G}
\kappa_G:B(G)\rightarrow X^{\ast}(Z(\widehat{G})^{\Gamma_F})=\pi_1(G)_{\Gamma_F}: \kappa_G([b])=\overline{w_{G_L}(b)}.
\end{equation}
Here, for $b\in G(L)$, $[b]$ denotes its $\sigma$-conjugacy class, and for $x\in \pi(G)_{I}$, $\overline{x}$ denotes its image under the natural quotient map $\pi(G)_{I}\rightarrow \pi(G)_{\Gamma_F}$. For further details, see \cite[7.5]{Kottwitz97}.

All these maps are functorial in $G$ (i.e. for group homomorphisms). 
%% Change

%%%%%%%%%%%%%%%%%%%%
\subsubsection{The Newton map $\nu_G$}

Let $\mathbb{D}$ denote the protorus $\varprojlim\Gm$ with the character group $\Q=\varinjlim\Z$.
For an algebraic group $G$ over a $p$-adic local field $F$, we put
\[\mathcal{N}(G):=(\Hom_{L}(\mathbb{D},G)/\Int(G(L)))^{\sigma}\]
(the subset of $\sigma$-invariants in the set of $G(L)$-conjugacy classes of $L$-rational quasi-cocharacters into $G_L$). We will use the notation $\overline{\nu}$ for the the conjugacy class of $\nu\in\Hom_{L}(\mathbb{D},G)$. 

For every $b\in G(L)$, Kottwitz \cite[$\S$4.3]{Kottwitz85} constructs an element $\nu=\nu_G(b)=\nu_b\in\Hom_L(\mathbb{D},G)$%%
\footnote{We interchangeably write $\nu_G(b)$ or $\nu_b$.}
uniquely characterized by the property that there are an integer $s>0$, an element $c\in G(L)$ and a uniformizing element $\pi$ of $F$ such that:
\begin{itemize}\addtolength{\itemsep}{-4pt}
\item[(i)] $s\nu\in\Hom_L(\Gm,G)$.
\item[(ii)] $\Int (c)\circ s\nu$ is defined over the fixed field of $\sigma^s$ in $L$.
\item[(iii)] $c\cdot (b\sigma)^s\cdot c^{-1}=c\cdot (s\nu)(\pi)\cdot c^{-1}\cdot \sigma^{s}$.
\end{itemize}
In (iii), the product (and the equality as well) take place in the semi-direct product group $G(L)\rtimes\langle\sigma\rangle$. We call $\nu_b$ the \textit{Newton homomorphism} attached to $b\in G(L)$

When $G$ is a torus $T$, $\nu_b=\mathrm{av}\circ w_{T_L}(b)$, where $\mathrm{av}:X_{\ast}(T)_I\rightarrow X_{\ast}(T)_{\Q}^{\Gamma_F}$ is ``the average map'' $X_{\ast}(T)_I\rightarrow X_{\ast}(T)_{\Gamma_F}\rightarrow X_{\ast}(T)_{\Q}^{\Gamma_F}$ sending $\underline{\mu}\ (\mu\in X_{\ast}(T))$ to $|\Gamma_F\cdot\mu|^{-1} \sum_{\mu'\in \Gamma_F\cdot\mu}\mu'$ (cf. \cite[Thm. 1.15, (iii)]{RR96}). Hence, it follows that if $T$ is split by a finite Galois extension $K\supset F$, for $b\in T(L)$, $[K:F]\nu_b\in X_{\ast}(T)$ and that $\langle \chi,\nu_b\rangle= \mathrm{val}(\chi(b))$ (especially $\in\Z$) for every $F$-rational character $\chi$ of $T$.

The map $b\mapsto \nu_b$ has the following properties.
\begin{itemize}\addtolength{\itemsep}{-4pt}
\item[(a)] $\nu_{\sigma(b)}=\sigma(\nu_b)$.
\item[(b)] $gb\sigma(g)^{-1}\mapsto \Int (g)\circ \nu,\ g\in G(L)$.
\item[(c)] $\nu_b=\Int (b)\circ\sigma(\nu_b)$.
\end{itemize}
It follows from (b) and (c) that $\nu_G:G(L)\rightarrow \Hom_L(\mathbb{D},G)$ gives rise to a map $\overline{\nu}_G:B(G)\rightarrow \mathcal{N}(G)$, which we call the \textit{Newton map}. This can be also regarded as a functor from the category of connected reductive groups to the category of sets (endowed with partial orders defined as below):
\begin{equation*}\overline{\nu}:B(\cdot)\rightarrow \mathcal{N}(\cdot)\ ;\
\overline{\nu}_{G}([b])=\overline{\nu}_{b},\quad
b\in[b].\end{equation*}

%%%%%%%%%%%%%%%%%%%%
\subsubsection{} For a connected reductive group $G$ over an arbitrary (i.e. not necessarily $p$-adic) field $F$, let $\mathcal{BR}(G)=(X^{\ast},R^{\ast},X_{\ast},R_{\ast},\Delta)$ be the based root datum of $G$: we may take $X^{\ast}=X^{\ast}(T)$, $X_{\ast}=X_{\ast}(T)$ for a maximal $F$-torus $T$ of $G$ and $R^{\ast}\subset X^{\ast}(T)$, $R_{\ast}\subset X_{\ast}(T)$ are respectively the roots and the coroots for the pair $(G,T)$ with a choice of basis $\Delta$ of $R^{\ast}$ (whose choice corresponds to that of a Borel subgroup $B\subset G_{\overline{F}}$ containing $T_{\overline{F}}$).
Let $\overline{C}\subset (X_{\ast})_{\Q}$ denote the closed Weyl chamber associated with the root base $\Delta$. It comes with a canonical action of $\Gamma_F:=\Gal(\overline{F}/F)$ on $\overline{C}$. 

For a cocharacter $\mu\in\Hom_{\overline{F}}(\Gm,G)$ lying in $\overline{C}$, we set 
\[\overline{\mu}:=|\Gamma_F\cdot\mu|^{-1} \sum_{\mu'\in\Gamma_F\cdot\mu}\mu'\quad\in\overline{C}.\]
Here, the orbit $\Gamma_F\cdot\mu$ is obtained using the canonical Galois action on $\overline{C}$. Once a Weyl chamber $\overline{C}$ (equivalently, a Borel subgroup $B$ or a root base $\Delta$) is chosen, $\overline{\mu}$ depends only on the $G(\overline{F})$-conjugacy class $\{\mu\}$ of $\mu$.

%%%%%%%%%%%%%%%
Suppose that $\mu\in X_{\ast}(T)\cap\overline{C}$. 
As $X_{\ast}(T)=X^{\ast}(\widehat{T})$ for the dual torus $\widehat{T}$ of $T$, regarded as a character on $\widehat{T}$, we can restrict $\mu$ to the subgroup $Z(\widehat{G})^{\Gamma_F}$ of $\widehat{T}$, obtaining an element 
\begin{equation} \label{eqn:mu_natural}
\mu^{\natural}\in X^{\ast}(Z(\widehat{G})^{\Gamma_F})=\pi_1(G)_{\Gamma_F}.
\end{equation}
Again, $\mu^{\natural}$ depends only on the $G(\overline{F})$-conjugacy class $\{\mu\}$ of $\mu$.
Alternatively, $\mu^{\natural}$ equals the image (sometimes, also denoted by $\underline{\mu}$) of $\mu\in X_{\ast}(T)$ under the canonical map $X_{\ast}(T)\rightarrow \pi_1(G)_{\Gamma_F}$.

%%%%%%%%%%%%%%%%%%%%
\subsubsection{The set $B(G,\{\mu\})$} \label{subsubsec:B(G,{mu})} 
Again, let us return to a $p$-adic field $F$. We fix a closed Weyl chamber $\overline{C}$ (equiv. a Borel subgroup $B$ over $\overline{F}$). 
Suppose given a $G(\overline{F})$-conjugacy class $\{\mu\}$ of cocharacters into $G_{\overline{F}}$.
Let $\mu$ be the representative of $\{\mu\}$ in $\overline{C}$; so we have $\overline{\mu}=\overline{\mu}(G,\{\mu\})\in\overline{C}$ and $\mu^{\natural}\in X^{\ast}(Z(\widehat{G})^{\Gamma_F})$. 
We define a finite subset $B(G,\{\mu\})$ of $B(G)$ (cf. \cite[Sec.6]{Kottwitz97}, \cite[Sec.4]{Rapoport05}): 
\[B(G,\{\mu\}):=\left\{\ [b]\in B(G)\ |\quad \kappa_{G}([b])=\mu^{\natural},\quad \overline{\nu}_{G}([b])\preceq \overline{\mu}\ \right\},\]
where $\preceq$ is the natural partial order on the closed Weyl chamber $\overline{C}$ defined by that $\nu\preceq \nu'$ if $\nu'-\nu$ is a nonnegative linear combination (with \emph{rational} coefficients) of simple coroots in $R_{\ast}(T)$ \cite{RR96}, Lemma 2.2). 
One knows \cite[4.13]{Kottwitz97} that the map 
\[(\overline{\nu},\kappa):B(G)\rightarrow \mathcal{N}(G)\times X^{\ast}(Z(\widehat{G})^{\Gamma_F})\] 
is injective, hence $B(G,\{\mu\})$ can be identified with a subset of $\mathcal{N}(G)$.

%%%%%%%%%%%%%%%%%%%%
%%%%%%%%%%%%%%%%%%%%

\subsection{Parahoric subgroups}

Our references here include \cite{Rapoport05}, \cite{HainesRapoport08}, \cite{HainesRostami10}, in addition to the original sources \cite{BT72}, \cite{BT84}, \cite{Tits79}.

\subsubsection{}  \label{subsubsec:parahoric}
Let $G$ be a connected reductive group $G$ over a strictly henselian discrete valued field $L$. Let $\mcB(G,L)$ be the Bruhat-Tits building of $G$ over $L$ (cf. \cite{Tits79}, \cite{BT72}, \cite{BT84}). Then, a \textit{parahoric subgroup} of $G(\mfk)$ is a subgroup of the form
\[ K_{\mbff}=\mathrm{Fix}\ \mbff \cap \ker\ w_G\]
for a facet $\mbff$ of $\mcB(G,L)$. Here, $\mathrm{Fix}\ \mbff$ denotes the subgroup of $G(\mfk)$ fixing $\mbff$ pointwise and $w_G$ is the Kottwitz map (\autoref{subsubsec:Kottwitz_hom}). When $\mbff$ is an alcove of $\mcB(G,L)$ (i.e. a maximal facet), the parahoric subgroup is called an \textit{Iwahori} subgroup. A \textit{special maximal parahoric subgroup} of $G(\mfk)$ is the parahoric subgroup attached to a special point in $\mcB(G,L)$.
More precisely, choose a maximal split torus $A$ of $G$ and let $\mcA(A,L)$ be the associated apartment; let $\mcA(A^{\ad},L)$ be the apartment in $\mcB(G^{\ad},L)$ corresponding to the image $A^{\ad}$ of $A$ in $G^{\ad}$. Then, there exists a canonical simplicial isomorphism \cite[1.2]{Tits79}
\[\mcA(A,L)\cong \mcA(A^{\ad},L)\times X_{\ast}(Z(G))_{\Gamma_F}\otimes\R.\]
Then, every special point in $\mcA(A,L)$ is of the form $\{\mbfv\}\times x$ for a unique special \emph{vertex} $\mbfv$ of $\mcA(A^{\ad},L)$ (in the sense of \cite[1.9]{Tits79}) and some $x\in X_{\ast}(Z(G))_{\Gamma_F}\otimes\R$.

The original definition of pararhoic subgroups by Bruhat-Tits \cite[5.2.6]{BT84}, cf. \cite[3.4]{Tits79} uses group schemes. With every facet $\mbff$ of $\mcB(G,L)$ they associate a smooth group scheme $\mcG_{\mbff}$ over $\Spec(\cO_L)$ with generic fiber $G$ such that $\mcG_{\mbff}(\cO_L)=\mathrm{Fix}\ \mbff$. Also, there exists an open subgroup $\mcG_{\mbff}^{\mathrm{o}}$ with the same generic fiber $G$ and the connected special fiber. Then, the parahoric subgroup attached to $\mbff$ by Bruhat-Tits is $\mcG_{\mbff}^{\mathrm{o}}(\cO_L)$. It is known \cite[Prop.3]{HainesRapoport08} that they coincide:
\[K_{\mbff}=\mcG_{\mbff}^{\mathrm{o}}(\cO_L).\]

Now suppose that $G$ is defined over a local field $F$, as before given as a finite extension of $\Qp$ in $\Qpb$.
Again, $L$ denotes the completion of the maximal unramified extension of $F$ in $\Qpb$ and let $\sigma$ be the Frobenius automorphism of $L$ fixing $F$. Let $\mcB(G,L)$ (resp. $\mcB(G,F)$) be the Bruhat-Tits building of $G$ over $L$ (resp. over $F$); as $G$ is defined over $F$, $\mcB(G,L)$ carries an action of $G(\mfk)\rtimes\langle\sigma\rangle$ and $\mcB(G,F)$ is identified with the set of fixed points of $\mcB(G,L)$ under $\langle\sigma\rangle$ \cite[5.1.25]{BT84}. This procedure of taking $\sigma$-fixed points $\mbff\mapsto \mbff^{\sigma}$ gives a bijection from the set of $\sigma$-stable facets in $\mcB(G,L)$ to the set of facets in $\mcB(G,F)$.

A \textit{parahoric subgroup} of $G(F)$ is by definition $\mcG_{\mbff}^{\mathrm{o}}(\cO_L)^{\sigma}(:=\mcG_{\mbff}^{\mathrm{o}}(\cO_L)\cap G(F))$ for a $\sigma$-stable facet $\mbff$ of $\mcB(G,L)$. A \textit{special maximal parahoric subgroup} of $G(F)$ is $\mcG_{\mbff}^{\mathrm{o}}(\cO_L)^{\sigma}$ for a special point $\mbff\in\mcB(G,F)$. 
%Change

%%%%%%%%%%%%%%%%%%%%
\subsubsection{Extended affine Weyl group} \label{subsubsec:EAWG}

%Let $G$ be a connected reductive group over a strictly henselian discrete valued field $L$. 
Let $G$ be a connected reductive group over a complete discrete valued field $L$ with algebraically closed residue field.
Let $S$ be a maximal split $L$-torus of $G$ and $T$ its centralizer; $T$ is a maximal torus since $G_L$ is quasi-split by a well-known theorem of Steinberg. Let $N=N_G(T)$ be the normalizer of $T$. The \textit{extended affine Weyl group} (or \textit{Iwahori Weyl group}) associated with $S$ is the quotient group
\[\tilde{W}:=N(L)/T(L)_1,\]
where $T(L)_1$ is the kernel of the Kottwitz map $w_T:T(L)\rightarrow X_{\ast}(T)_{I}$. As $w_T$ is surjective, $\tilde{W}$ is an extension of the relative Weyl group $W_0:=N(L)/T(L)$ by $X_{\ast}(T)_I$: 
\begin{equation} \label{eqn:EAWG1}
0\rightarrow X_{\ast}(T)_I\rightarrow \tilde{W}\rightarrow W_0\rightarrow 0.
\end{equation}
The normal subgroup $X_{\ast}(T)_I$ is called the \textit{translation subgroup} of $\tilde{W}$, and any $\lambda\in X_{\ast}(T)_I$, viewed as an element in $\tilde{W}$ in this way, will be denoted by $t^{\lambda}$ (\textit{translation element}).

This extension splits by choice of a special vertex $\mbfv$ in the apartment corresponding to $S$, namely if $K=K_{\mbfv}\subset G(L)$ is the associated parahoric subgroup, the subgroup
\[\tilde{W}_K:=(N(L)\cap K)/T(L)_1\]
of $\tilde{W}$ projects isomorphically to $W_0$, and thus gives a splitting
\[\tilde{W}=X_{\ast}(T)_I\rtimes \tilde{W}_K.\]

For two parahoric subgroups $K$ and $K'$ associated with facets in the apartment corresponding to $S$, there exists an isomorphism
\begin{equation} \label{eqn:parahoric_double_coset}
K\backslash G(L)/ K'\cong \tilde{W}_K\backslash \tilde{W}/ \tilde{W}_{K'}.
\end{equation}

Let $S^{\uc}$ (resp. $T^{\uc}$, $N^{\uc}$) be the inverse image of $S$ (resp. $T$, $N$) in the universal covering $G^{\uc}$ of $G^{\der}$; then, $S^{\uc}$ is a maximal split torus of $G^{\uc}$ and $T^{\uc}$ (resp. $N^{\uc}$) is its centralizer (resp. the normalizer). 
The natural map $N^{\uc}(L)\rightarrow N(L)$ induces an injection $X_{\ast}(T^{\uc})_I\hra X_{\ast}(T)_I$ and presents the extended affine Weyl group associated with $(G^{\uc},S^{\uc})$
\[W_a:=N^{\uc}(L)/ T^{\uc}(L)_1\]
as a normal subgroup of the extended affine Weyl group $\tilde{W}$ (attached to $S$) such that the translation subgroup $X_{\ast}(T)_I$ maps onto the quotient $\tilde{W}/W_a$ with kernel $X_{\ast}(T^{\uc})_I$:
\begin{equation} \label{eqn:EAWG2}
0\rightarrow W_a\rightarrow \tilde{W}\rightarrow X_{\ast}(T)_I/X_{\ast}(T^{\uc})_I\rightarrow 0.
\end{equation}
The quotient group $X_{\ast}(T)_I/X_{\ast}(T^{\uc})_I$ is identified in a natural way with $\pi_1(G)_I$ \cite[p.196]{HainesRapoport08}.
The group $W_a$ can be also regarded as an affine Weyl group attached to some reduced root system (\textit{loc. cit.}, p.195).%%
\footnote{With $(G,S)$, Tits \cite{Tits79} define an affine root system $\Phi_{\mathrm{af}}$ in the affine space $A$ under $V:=X_{\ast}(S)_{\R}$, and a homomorphism $\nu:\tilde{W}\rightarrow \mathrm{Aff}(A)$, where $\mathrm{Aff}(A)$ is the group of affine transformations of $A$. When $G$ is semi-simple, the affine Weyl group of the reduced root system in question equals the Weyl group of $\Phi_{\mathrm{af}}$, i.e. the group generated by the reflections about the hyperplanes that are zero sets of affine functions in $\Phi_{\mathrm{af}}$.
When we choose a special vertex $\mbfv$ (so, identify the affine space $A$ with $V$, and $W_0$ with the stabilizer subgroup of $\mbfv$), there exists a reduced root system ${}^{\mbfv}\Sigma$ in question having roots in $V$, and $W_a\cong \nu(W_a)$ is isomorphic to the affine Weyl group $Q^{\vee}({}^{\mbfv}\Sigma)\rtimes W({}^{\mbfv}\Sigma)$.}
%But it can happen that $\nu(W_a)$ may not be the same as (but just isomorphic to) the affine Weyl group $Q^{\vee}({}^{\mbfv}\Sigma)\rtimes W({}^{\mbfv}\Sigma)$, as subgroups of $\mathrm{Aff}(V)$ (under the embedding $X_{\ast}(T^{\uc})_I\subset (X_{\ast}(T^{\uc})_I)_{\R}=X_{\ast}(S)_{\R}$). }
%%

This extension (\ref{eqn:EAWG2}) also splits by choice of an alcove in the apartment $\mcA(S,L)$ of $S$. More precisely, the extended affine Weyl group $\tilde{W}$ (resp. the affine Weyl group $W_a$) acts transitively (resp. simply transitively) on the set of alcoves in $\mcA(S,L)$, hence when we choose a base alcove $\mbfa$ in $\mcA(S,L)$, 
\begin{equation} \label{eqn:splitting_of_EAWG2}
\tilde{W}=W_a\rtimes \Omega_{\mbfa},
\end{equation}
where $\Omega_{\mbfa}$ is the normalizer of $\mbfa$; $\Omega_{\mbfa}$ will be often identified with $X_{\ast}(T)_I/X_{\ast}(T^{\uc})_I$.

Finally, suppose that there is an automorphism $\sigma$ of $L$ such that $L$ is the strict henselization of its fixed field $L^{\natural}$ and that $G$ is defined over $L^{\natural}$.
%Finally, suppose that there is an automorphism $\sigma$ of $L$ such that $L$ is the completion of a strict henselization of its fixed field $L^{\natural}$ and that $G$ is defined over $L^{\natural}$. 
Then, we can find a $L^{\natural}$-torus $S$ such that $S_L$ becomes a maximal split $L$-torus, and a maximal $L^{\natural}$-torus $T$ containing $S$; set $N$ to be the normalizer of $T$. Then $\sigma$ acts on the extended Weyl group $\tilde{W}$ in an obvious way. Moreover, if $K_{\mbfv}\subset G(L)$ is the parahoric subgroup attached to a $\sigma$-stable facet $\mbfv$, then the subgroup $\tilde{W}_{K_{\mbfv}}$ is stable under $\sigma$. We refer the reader to \cite[Remark 9]{HainesRapoport08} for a ``descent theory'' in this situation.

%%%%%%%%%%%%%%%%%%%%
\subsubsection{The $\{\mu\}$-admissible set} \label{subsubsec:mu-admissible_set}
As before, let $G$ be a connected reductive group $G$ over a complete discrete valued field $L$ with algebraically closed residue field. Let $W=N(\Lb)/T(\Lb)$ be the absolute Weyl group. Let
$\{\mu\}$ be a $G(\Lb)$-conjugacy class of cocharacters of $G$ over $\Lb$. We use $\{\mu\}$ again to denote the corresponding $W$-orbit in $X_{\ast}(T)$. Let us choose a Borel subgroup $B$ over $L$ containing $T$ (which exists as $G_{/L}$ is automatically quasi-split), and let $\mu_B$ be the unique representative of $\{\mu\}$ lying in the associated absolute closed Weyl chamber in $X_{\ast}(T)_{\R}$. Then, the $W_0$-orbit of the image $\underline{\mu_B}$ of $\mu_B$ in $X_{\ast}(T)_I$ is well-determined, since any two Borel subgroups over $L$ containing $T$ are conjugate under $G(L)$. We denote this $W_0$-orbit by $\Lambda(\{\mu\})$:
\[ \Lambda(\{\mu\}):=W_0\cdot\underline{\mu_B}\ \subset\ X_{\ast}(T)_I.\]
 It is known \cite[Lem. 3.1]{Rapoport05} that the image of $\Lambda(\{\mu\})$ in the quotient group $X_{\ast}(T)_I/X_{\ast}(T^{\uc})_I$ consists of a single element, which we denote by $\tau(\{\mu\})$. 

Let us now fix an alcove $\mbfa$ in the apartment corresponding to $S$. This determines a Bruhat order on the affine Weyl group $W_a$ which further extends to the extended Weyl group $\tilde{W}=W_a\rtimes \Omega_{\mbfa}$ (\ref{eqn:splitting_of_EAWG2}), \cite[$\S$1]{KR00}. Also, when $K\subset G(\mfk)$ is a parahoric subgroup associated with a facet of $\mbfa$, it induces a Bruhat order on the double coset space $\tilde{W}_K\backslash \tilde{W}/\tilde{W}_K$ \cite[$\S$8]{KR00}. We will denote all these orders by $\leq$; this should not cause much confusion.

%%%%%%%%%%%%%%%%%%%%
\begin{defn} \label{defn:mu-admissible_subset}
The \textit{$\{\mu\}$-admissible subset} of $\tilde{W}$ is 
\[\mathrm{Adm}(\{\mu\})=\{w\in\tilde{W}\ |\ w\leq t^{\lambda}\text{ for some }\lambda\in\Lambda(\{\mu\})\},\]
and the \textit{$\{\mu\}$-admissible subset} of $\tilde{W}_K\backslash \tilde{W}/\tilde{W}_K$ is 
\[\mathrm{Adm}_K(\{\mu\})=\{w\in\tilde{W}_K\backslash \tilde{W}/\tilde{W}_K\ |\ w\leq \tilde{W}_Kt^{\lambda}\tilde{W}_K\text{ for some }\lambda\in\Lambda(\{\mu\})\}.\]
\end{defn}
One knows \cite[(3.8)]{Rapoport05} that $\mathrm{Adm}_K(\{\mu\})$ is the image of $\mathrm{Adm}(\{\mu\})$ under the natural map $\tilde{W}\rightarrow \tilde{W}_K\backslash \tilde{W}/\tilde{W}_K$.

%%%%%%%%%%%%%%%%%%%%
\begin{prop}
Suppose that $G$ splits over $L$ (thus $S=T$) and $K$ is a special maximal parahoric subgroup. Then, 
\[\mathrm{Adm}_K(\{\mu\})=\{\nu\in X_{\ast}(S)\cap\overline{C}\ |\ \nu\stackrel{!}{\leq} \mu\},\]
where $\mu$ denotes the representative in $\overline{C}$ of $\{\mu\}$.
If furthermore $\{\mu\}$ is minuscule, $\mathrm{Adm}_K(\{\mu\})$ consists of a single element, i.e. $\{\mu\}\in X_{\ast}(T)/W$ itself.
\end{prop}

Here, $\nu\stackrel{!}{\leq} \mu$ means that $\mu-\nu$ is a sum of simple coroots with non-negative \emph{integer} coefficients. See Prop. 3.11 and Cor. 3.12 of \cite{Rapoport05} for a proof.

%%%%%%%%%%%%%%%%%%%%%%%%%%%%%%%%%%%%%%%%
%%%%%%%%%%%%%%%%%%%%%%%%%%%%%%%%%%%%%%%%

\section{Pseudo-motivic Galois gerb and admissible morphisms}

This section is devoted to a review of the theory of the pseudo-motivic Galois gerb and admissible morphisms, as explained in \cite{LR87}. In addition to this original source \cite{LR87}, we also refer readers to \cite{Milne92},  \cite{Kottwitz92}, \cite{Reimann97}.

%%%%%%%%%%%%%%%%%%%%
%%%%%%%%%%%%%%%%%%%%
\subsection{Galois gerbs} \label{subsec:Galois_gerbs}
We review the notion of Galois gerbs as used by Langlands-Rapoport in \cite[$\S$2]{LR87} (cf. \cite[$\S$4]{Breen94}, \cite[$\S$8]{Rapoport05}, \cite[Appendix B]{Reimann97}).
 
Let $k$ be a field of characteristic zero (which will be for us either a global or a local field) and $\kb$ an algebraic closure. For an affine group scheme $G=\Spec A$ over a Galois extension $k'\subset\kb$ of $k$ and $\sigma\in\Gal(k'/k)$, 
an automorphism $\kappa$ of $G(k')$ is said to be \textit{$\sigma$-linear} if there is a $\sigma$-linear automorphism $\kappa'$ of the algebra $A$ such that
\[\kappa'(f)(\kappa(g))=\sigma(f(g)),\quad f\in A,\ g\in G(k').\]
The simplest example is given by the natural action of $\Gal(k'/k)$ on $G(k')$, when $G$ is defined over $k$. In this article, we will be concerned mainly with the following kind of Galois gerbs, which will be called \textit{algebraic}. 
For $\sigma\in\Gal(k'/k)$, let
\[\sigma_{k'}:G(k')\rightarrow (\sigma^{\ast}G)(k')\] 
be the unique map for which $f\otimes 1(\sigma_{k'}(g))=\sigma(f(g))$ holds for $f\in A,\ g\in G(k')$, where $f\otimes1\in A\otimes_{k',\sigma}k'$.
Then, for any algebraic isomorphism $\theta$ of $k'$-group schemes from $\sigma^{\ast}G$ to $G$, 
the automorphism $\theta\circ\sigma_{k'}$ of $G(k')$ is $\sigma$-linear, since then one can take $\kappa:=\theta\circ\sigma_{k'}$ and $\kappa'(f):=(\theta^{\ast})^{-1}(f\otimes1)$ (Here, $\theta^{\ast}:A\isom A\otimes_{k',\sigma}k'$ denotes the associated map on the structure sheaf). 
We will call such $\sigma$-linear automorphism of $G(k')$ \textit{algebraic}. 
Hence, one can identify an algebraic $\sigma$-linear isomorphism $\kappa(\sigma)$ with an algebraic $k'$-isomorphism $\theta(\sigma):\sigma^{\ast}(G)\isom G$ via $\kappa(\sigma)=\theta(\sigma)\circ\sigma_{k'}$.

%%%%%%%%%%%%%%%%%%%%
\begin{defn} \label{defn:Galois_gerb}
Let $k'\subset \kb$ be a Galois extension of $k$. A \textit{$k'/k$-Galois gerb} is an extension of topological groups
\[1\lra G(k')\lra \fG\lra \Gal(k'/k)\lra 1,\]
where $G$ is an affine smooth group scheme (i.e. a linear algebraic group) over $k'$ and $G(k')$ (resp. $\Gal(k'/k)$) has the discrete (resp. the Krull) topology, such that
\begin{itemize}
\item[(i)] for every representative $g_{\sigma}\in\fG$ of $\sigma\in \Gal(k'/k)$, the automorphism $\kappa(\sigma):g\mapsto g_{\sigma}gg_{\sigma}^{-1}$ of $G(k')$ is algebraic $\sigma$-linear.
\item[(ii)] for some finite sub-extension $k\subset K\subset k'$, there exists a continuous section
\[\Gal(k'/K)\lra \fG\ :\ \sigma\mapsto g_{\sigma}\] 
which is a group homomorphism.
\end{itemize}
\end{defn}

In the presence of (i), condition (ii) means that the family $\{\theta(\sigma):\sigma^{\ast}(G)\isom G\}$ of isomorphisms associated with $\Int (g_{\sigma})$ is a $k'/K$-descent datum on $G$: the homomorphism property of (ii) gives the cocycle condition of descent datum. 
Thus the section $\sigma\mapsto g_{\sigma}\ (\sigma\in\Gal(k'/K))$ determines a $K$-structure on $G$ and accordingly an action of $\Gal(k'/K)$ on $G(k')$. 
This Galois action is nothing other than $\theta(\sigma)\circ\sigma_{k'}$,%%
\footnote{Suppose that the $K$-structure is given by an isomorphism $\alpha:G_0\otimes_{K}k'\isom G$ for an algebraic $K$-group $G_0$. Then, it gives rise to a descent isomorphism $\alpha\circ\sigma^{\ast}(\alpha^{-1}):\sigma^{\ast}(G)\isom G$ and a Galois action $\sigma(g)=\alpha\circ \sigma(\alpha^{-1}(g))$ on $G(k')=G_0(k')$. Since the descent isomorphism was $\theta(\sigma)$, we have $\sigma(g)=\alpha\circ \sigma^{\ast}(\alpha^{-1})(\sigma_{k'}(g))=\theta(\sigma)\circ\sigma_{k'}(g)$.}
namely, we have the relation 
\begin{equation} \label{eq:conjugation=Galois_action}
g_{\sigma}gg_{\sigma}^{-1}=\sigma(g),\quad \sigma\in \Gal(k'/K),
\end{equation}
where $\sigma(g)$ is the just mentioned action of $\sigma\in\Gal(k'/K)$ on $G(k')$. In other words, the conditions (i), (ii) imply that over some finite Galois extension $K\subset k'$ of $k$, there exists a group-theoretic section $\sigma\mapsto \rho_{\sigma}$, via which the pull-back to $\Gal(k'/K)$ of $\fG$ becomes a semi-direct product $G(k')\rtimes \Gal(k'/K)$, with the action of $\Gal(k'/K)$ on $G(k')$ (i.e. the conjugation action of $\Gal(k'/K)$ on $G(k')$ via the section) being the natural Galois acton resulting from a $K$-structure on $G$.

We remark that our definition of Galois gerb is equivalent to that of affine smooth gerb%% 
\footnote{in the sense of Giruad, or in the sense of the theory of Tannakian categories, namely, a stack in groupoids over an \'etale site which is locally nonempty and locally connected, cf. \cite[Appendix]{Milne92}, \cite[Ch.II]{DMOS82}, \cite[2.2]{Breen94}.}
on the \'etale site $\Spec(k)_{\et}$ \textit{equipped with a neutralizing object over $\Spec(K)$}.%% 
\footnote{The category of such affine gerbs endowed with a distinguished neutralizing object is equivalent to the category of affine $\Spec(K)/\Spec(k)$-groupoid schemes, acting transitively on $\Spec(K)$, \cite{Milne92}, Appendex, Prop. A.15). 
For this reason, Milne insists to call Galois gerbs in our sense groupoids \cite{Milne92}, \cite{Milne03}. 
But any two neutralizing local objects become isomorphic over $\kb$, thus a gerb $\fG$ (as a stack) is uniquely determined by its associated groupoid $(\fG,x\in \mathrm{Ob}(\fG(\kb)))$, up to conjugation by an element of $\mathrm{Aut}(x)=\fG^{\Delta}(\kb)$. Hopefully, this justifies our decision to stick to the original terminology of Langlands-Rapoport.}
For a detailed discussion of this relation, we refer to \cite[p.152-153]{LR87}, \cite[$\S$4]{Breen94}.

We call the group scheme $G$ the \textit{kernel} of $\fG$ and write $G=\fG^{\Delta}$. A \textit{morphism} between $k'/k$-Galois gerbs $\varphi:\fG\rightarrow\fG'$ is a continuous map of extensions which induces the identity on $\Gal(k'/k)$ and an algebraic homomorphism on the kernel groups. Two morphisms $\phi_1$ and $\phi_2$ are said to be \textit{conjugate} if there exists $g'\in G'(k')$ with $\phi_2=\Int (g')\circ \phi_1$.
With every linear algebraic group $G$ over $k$, the semi-direct product gives a gerb
\[\fG_G=G(k')\rtimes\Gal(k'/k).\]
We call it the \textit{neutral gerb} attached to $G$.

For two successive Galois extensions $k\subset k'\subset k''\subset \kb$, 
any $k'/k$-Galois gerb $\fG$ gives rise to a $k''/k$-Galois gerb $\fG'$, by first pulling-back the extension $\fG$ by the surjection $\Gal(k''/k)\twoheadrightarrow\Gal(k'/k)$ and then pushing-out via $G(k')\rightarrow G(k'')$. In this situation, we will call $\fG'$ the \textit{inflation} to $k''$ of the $k'/k$-Galois gerb $\fG$; this terminology will be justified when we relate Galois gerbs with commutative kernels to Galois cohomology. 
We also call a $\kb/k$-Galois gerb, simply a Galois gerb over $k$. It follows from definition that any Galois gerb over $k$ is the inflation of a $k'/k$-Galois gerb for some finite Galois extension $k\subset k'\subset \kb$ and that every morphism between $k'/k$-Galois gerbs induces a morphism between their inflations to $k''$ for any sub-extension $k''\subset \kb$.

For two morphisms of $k'/k$-Galois gerbs $\phi_1,\phi_2:\fG\rightarrow\fG'$, there exists a $k$-scheme $\underline{\mathrm{Isom}}(\phi_1,\phi_2)$, whose set of $R$-points, for a $k$-algebra $R$, is given by
\begin{equation} \label{eq:Isom(phi_1,phi_2)}
\underline{\mathrm{Isom}}(\phi_1,\phi_2)(R)=\{g\in\fG'^{\Delta}(k'\otimes_kR)\ |\ \Int (g)\circ\phi_{1R}=\phi_{2R}\},
\end{equation}
where $\phi_{1R}$ and $\phi_{2R}$ are the induced maps $\fG_R\rightarrow \fG'_R$ between the push-outs of $\fG$ and $\fG'$ via $\fG^{\Delta}(k')\rightarrow \fG^{\Delta}(k'\otimes_kR)$ and the same map for $\fG'$. 
When $\phi_1=\phi_2$, we denote this $k$-group scheme by $I_{\phi_1}=\underline{\mathrm{Aut}}(\phi_1)$. For $g\in \fG'(k')$, one readily sees that $\Int(g)$ induces a $k$-isomorphism of $k$-groups $I_{\phi_1}\isom I_{\Int(g)\circ\phi_1}$.

%%%%%%%%%%%%%%%%%%%%
\subsubsection{Galois gerbs defined by $2$-cocyles with values in commutative affine group schemes}

In our work (as well as in the work of Langlands-Rapoport), 
besides the neutral gerbs attached to arbitrary algebraic groups, all the nontrivial Galois gerbs have as associated kernel \textit{commutative affine group schemes (in fact, (pro-)tori) defined over base fields}. In such cases, a Galois gerb has an explicit description in terms of (continuous) $2$-cocycles (in the usual sense) on the absolute Galois group of the base field with values in the geometric points of given commutative affine group scheme endowed with the natural Galois action.

Recall that for a group $H$ and an $H$-module $A$ (i.e. an abelian group with $H$-action), a normalized%
\footnote{i.e. $e_{h,1}=e_{1,h}=1$ for all $h\in H$.}
$2$-cocyle (or ``factor set'') $(e_{h_1,h_2})$ on $H$ with values in $A$ gives rise to an extension of $H$ by $A$:
\[1\rightarrow A\rightarrow E\stackrel{p}{\rightarrow}H\rightarrow1\] 
with property 
\begin{equation} \label{eq:extension_with_given_conjugation_action}
e\cdot a\cdot e^{-1}=p(e)(a)\quad (e\in E,a\in A),
\end{equation}
where the right action of $p(e)$ on $A$ is the given one.
Explicitly, $E$ is generated by $A$ and $\{e_h\}_{h\in H}$ ($a\mapsto 0,e_h\mapsto h$ giving the projection $E\rightarrow H$) with relations 
\[e_h\cdot a\cdot e_h^{-1}=h(a)\ (a\in A,h\in H),\qquad e_{h_1,h_2}=e_{h_1}\cdot e_{h_2}\cdot e_{h_1h_2}^{-1},\quad e_1=1\] ($e_{h_1,h_2}\in Z^2(H,A)$ guarantees the associativity of the resulting composition law). 
Two extensions $E$, $E'$ of $H$ by $A$ with property (\ref{eq:extension_with_given_conjugation_action}) are said to be \textit{isomorphic} if there exists an isomorphism $E\isom E'$ which restricts to identity on $A$ and also induces identity on $H$. 
Then, this construction gives a bijection of pointed sets between $H^2(H,A)$ and the set of isomorphisms classes of group extensions of $H$ by $A$ with the induced conjugation action of $H$ on $A$ being the given one. Here, the distinguished points are the cohomology class of the trivial $2$-cocycle and the semi-direct product, respectively. For two isomorphic extensions $E,E'$ with property (\ref{eq:extension_with_given_conjugation_action}), we say that two isomorphisms $f_1,f_2:E\isom E'$ 
which induce identities on $A$ and $H$ are \textit{equivalent} (or \textit{conjugate}) if $f_2=\Int  a\circ f_1$ for some $a\in A$, where $\Int  a$ is the conjugation automorphism of $E'$. Then, there is a natural action of $H^1(H,A)$ on the set of equivalence classes of isomorphisms from $E$ to $E'$, which makes the latter set into a torsor under $H^1(H,A)$. 

Now, suppose that $G$ is a \textit{separable} commutative affine group scheme over $k$: then $G$ is the inverse limit of a strict system of commutative algebraic groups indexed by $(\N,\leq)$ (cf. \cite[2.6]{Milne03}). If $k'\subset\kb$ is a Galois extension of $k$, $G(k')$ is endowed with the inverse limit topology (for algebraic group $Q$, $Q(k')$ is given the discrete topology) and we get a continuous action of $\Gal(k'/k)$ on $G(k')$ provided by the $k$-structure of $G$. 
Then, for any continuous $2$-cocycle $(e_{\rho,\tau})\in Z^2_{cts}(\Gal(k'/k),G(k'))$, the resulting extension 
\[1\rightarrow G(k')\rightarrow \fE_{k'}\rightarrow \Gal(k'/k)\rightarrow 1\] 
is a $k'/k$-Galois gerb.
Indeed, condition (i) of Definition \ref{defn:Galois_gerb} is obvious, and for (ii), we note that
since $G(k')$ has discrete topology, any class in $H^2_{cts}(\Gal(k'/k),G(k'))$ lies in 
$H^2(\Gal(K/k),G(K))$ for a \emph{finite} Galois extension $K$ of $k$, so becomes trivial when restricted to $\Gal(k'/K)$.
Furthermore, by pulling-back along $\Gal(\kb/k)\twoheadrightarrow\Gal(k'/k)$ and push-out via $G(k')\hra G(\kb)$, we obtain a Galois gerb $\fE$ over $k$
\[1\rightarrow G(\kb)\rightarrow \fE\rightarrow \Gal(\kb/k)\rightarrow 1,\]
which we called the inflation of $\fE_{k'}$ to $\kb$. 
Now, one can verify that the corresponding cohomology class in $H^2_{cts}(\Gal(\kb/k),G(\kb))$ is indeed the image of $(e_{\rho,\tau})\in H^2_{cts}(\Gal(k'/k),G(k'))$ under the inflation map $H^2_{cts}(\Gal(k'/k),G(k'))\rightarrow H^2_{cts}(\Gal(\kb/k),G(\kb))$.

%%%%%%%%%%%%%%%%%%%%
%%%%%%%%%%%%%%%%%%%%
\subsection{Pseudo-motivic Galois gerb}

%%%%%%%%%%%%%%%%%%%%
\subsubsection{Local Galois gerbs} \label{subsubsec:Local_Galois_gerbs}
Here, we define a Galois gerb $\fG_v$ over $\Q_v$ for each place $v$ of $\Q$. 

For $v\neq p,\infty$, we define $\fG_v$ to be the trivial Galois gerb $\Gal(\overline{\Q}_v/\Q_v)$:
\[\begin{array} {ccccccccc}
1 & \rightarrow & 1 & \rightarrow & \Gal(\overline{\Q}_v/\Q_v) & \rightarrow & \Gal(\overline{\Q}_v/\Q_v) & \rightarrow & 1.
\end{array}\]

For $v=\infty$, the cocycle $(d_{\rho,\gamma})\in Z^2(\Gal(\C/\R),\C^{\times})$ 
\[d_{1,1}=d_{1,\iota}=d_{\iota,1}=1,\quad d_{\iota,\iota}=-1,\]
where $\Gal(\C/\R)=\{1,\iota\}$, represents the fundamental class in $H^2(\Gal(\C/\R),\C^{\times})$ \cite{Milne13}. We set $\fG_{\infty}$ to be the (isomorphism class of) Galois gerb defined by this cocycle (or its cohomology class):
\[\begin{array} {ccccccccc}
1 & \rightarrow & \C^{\times} & \rightarrow & \fG_{\infty} & \rightarrow & \Gal(\C/\R) & \rightarrow & 1,
\end{array}\]
So,
$\fG_{\infty}$ has generators $\C^{\times}$ and $w=w(\iota)$ (lift of $\iota$) satisfying that
\[w(\iota)^2=-1\in\C^{\times},\ \text{ and}\quad wzw^{-1}=\iota(z)=\overline{z}\ (z\in\C^{\times}).\]

For $v=p$ also, for any finite Galois extension $K$ of $\Qp$ in $\Qpb$, there is the fundamental class in $H^2(\Gal(K/\Qp),K^{\times})$ \cite{Milne13}.
For unramified extension $L_n/\Qp$, it is represented by the cocycle: for $0\leq i,j<n$, 
\begin{equation} \label{eq:canonical_fundamental_cocycle}
d_{\overline{\sigma}^i,\overline{\sigma}^j}=\begin{cases}
p^{-1} & \text{if } i+j\geq n, \\
1 & \text{ otherwise, }
\end{cases}
\end{equation} 
where $\overline{\sigma}\in \Gal(L_n/\Qp)$ is the \emph{arithmetic} Frobenius. 

We let $\fG_{p,K}^K$ be the corresponding (isomorphism class of) $K/\Qp$-Galois gerb and $\fG_p^K$ the Galois gerb over $\Qp$ obtained from $\fG_{p,K}^K$ by inflation:%%
\footnote{Note that our notations for these gerbs differ from those of \cite{Kisin17}: our $\fG^K_{p,K}$ (resp. $\fG^K_p$) is his $\fG^K_p$ (resp. $\tilde{\fG}^K_p$).}
\[\xymatrix{
1 & \rightarrow & K^{\times} & \rightarrow & \fG_{p,K}^K & \rightarrow & \Gal(K/\Qp) & \rightarrow & 1\\ 
1 & \rightarrow & K^{\times} \ar@{=}[u] \ar@{^{(}->}[d] & \rightarrow & \pi_K^{\ast}\fG_{p,K}^K \ar[u] \ar[d] & \rightarrow & \Gal(\Qp/\Qp) \ar@{->>}[u]^{\pi_K} \ar@{=}[d] & \rightarrow & 1\\
1 & \rightarrow & \Gm(\Qpb) & \rightarrow & \fG_p^K & \rightarrow & \Gal(\Qp/\Qp) & \rightarrow & 1.}
\]
Here, $\pi_K^{\ast}\fG_{p,K}^K$ is the pull-back of $\fG_{p,K}^K$ along $\pi_K:\Gal(\Qp/\Qp)\twoheadrightarrow \Gal(K/\Qp)$, and $\fG_p^K$ is the push-out of $\pi_K^{\ast}\fG_{p,K}^K$ along $K^{\times}\hra \Gm(\Qpb)$.

For each $K\subset K'\subset\Qpb$ ($K'$ being a Galois extension of $\Qp$ containing $K$), there exists a homomorphism
\[\fG_p^{K'}\rightarrow \fG_p^K\]
which, on the kernel, is given by $z\mapsto z^{[K':K]}$ \cite[p.119]{LR87}, 
\cite[Remark B1.2]{Reimann97}. By passing to the inverse limit over $K\supset \Qp$, we obtain a pro-Galois gerb $\fG_p$ over $\Qp$ with kernel $\mathbb{D}=\varprojlim\Gm$ (the protorus over $\Qp$ with character group $X^{\ast}(\mathbb{D})=\Q$). 

For each Galois extension $K\subset\Qpb$ of $\Qp$, we make a choice of a normalized cocycle $(d_{\tau_1,\tau_2}^K)$ on $\Gal(K/\Qp)$ with values in $K^{\times}$ defining $\fG_{p,K}^K$, and fix a section $\tau\mapsto s_{\tau}^K$ to the projection $\fG_{p,K}^K\rightarrow\Gal(K/\Qp)$ with property that
\[s_{\tau_1}^Ks_{\tau_2}^K=d_{\tau_1,\tau_2}^Ks_{\tau_1\tau_2}^K,\qquad s_{1}^K=1.\]
Since $\fG_p^K$ is obtained from $\fG_{p,K}^K$ by inflation, this gives rise to a section to $\fG_p^K \rightarrow \Gal(\Qpb/\Qp)$, which we also denote by $s^K$.%
\footnote{At the moment, we do not require that one can choose the sections $\tau\mapsto s_{\tau}^K$ in a compatible way for extensions $K\subset K'\subset\Qpb$, which will be however the case (cf. the proof of Theorem \ref{thm:pseudo-motivic_Galois_gerb}).}
By Hilbert 90, any such section $s^K$ is uniquely determined up to conjugation by an element of $K^{\times}$.

%%%%%%%%%%%%%%%%%%%%
\subsubsection{Dieudonn\'e gerb} \label{subsubsec:Dieudonne_gerb}
We also need an unramified version of the Galois gerbs $\fG^K_{p,K}$, $\fG_p$. 
Let $\Qpnr$ be the maximal unramified extension of $\Qp$ in $\Qpb$.
For $n\in\N$, we denote by $\fD_n=\fD_{L_n}$ the inflation to $\Qpnr$ of the $L_n/\Qp$-Galois gerb $\fG^{L_n}_{p,L_n}$.
As before, for every pair $m|n$, there exists a homomorphism $\fD_n\rightarrow\fD_m$ which, on the kernel, is given by $z\mapsto z^{n/m}$ (cf. 
\cite{Reimann97}, Remark B1.2). By passing to the inverse limit, we get a pro-$\Qpnr/\Qp$-Galois gerb $\fD$ over $\Qp$ with kernel $\mathbb{D}$. We call $\fD$ the \textit{Dieudonn\'e gerb}. Obviously, the Galois gerb $\fG_p^{L_n}$ (resp. the pro-Galois gerb $\fG_p$) is (equivalent to) the inflation to $\Qpb$ of $\fD_n$ (resp. $\fD$). Again, a choice of a section to $\fG_{p,L_n}^{L_n} \rightarrow \Gal(L_n/\Qp)$ made above gives us a section to $\fD_n \rightarrow \Gal(\Qpnr/\Qp)$ which is again denoted by $s^{L_n}$.

%%%%%%%%%%%%%%%%%%%%
\subsubsection{Unramified morphisms} \label{subsubsec:cls}

For any (connected) reductive group $H$ over $\Qp$, there exists a canonical map
\[\mathrm{cls}_H:\Hom_{\Qpb/\Qp}(\fG_p,\fG_H)\rightarrow B(H),\]
where $B(H)$ is the set of $\sigma$-conjugacy classes of elements in $H(\mfk)$. 

Let $K\subset \Qpb$ be a finite Galois extension. Recall that we fixed a normalized cocycle $(d_{\tau_1,\tau_2}^K)\in Z^2(\Gal(K/\Qp),K^{\times})$ defining $\fG^K_{p,K}$ as well as a section $s^K$ to the projection $\fG^K_{p,K}\rightarrow\Gal(K/\Qp)$ with the property that $s_{\tau_1}^Ks_{\tau_2}^K=d_{\tau_1,\tau_2}^Ks_{\tau_1\tau_2}^K$ and $s_{1}^K=1$; one uses the same notations for the induced cocycle defining $\fG_p^K$ and the induced section $\Gal(\Qpb/\Qp)\rightarrow \fG_p^K$.
A morphism $\theta:\fG_p^K\rightarrow \fG_H$ is said to be \textit{unramified} (with respect to the chosen section $s^K$) if $\theta(s_{\tau}^K)=1\rtimes\tau$ for all $\tau\in\Gal(\Qpb/\Qpnr)$. Note that if $K$ is unramified and $\theta^{\Delta}:\mathbb{G}_{\mathrm{m},\Qpb}\rightarrow H_{\Qpb}$ is defined over $\Qpnr$, this definition does not depend on the choice of the section $s^K$.
A morphism $\theta:\fG_p\rightarrow \fG_H$ is then said to be \textit{unramified} if $\theta$ factors through $\fG_p^K$ for some finite Galois extension $K$ of $\Qp$ such that the induced map $\fG_p^K\rightarrow\fG_H$ is unramified in the just defined sense. 

For a connected reductive group $H$ over $\Qp$, we introduce the associated neutral $\Qpnr/\Qp$-Galois gerb by
\[\fG_H^{\nr}:=H(\Qpnr)\rtimes\Gal(\Qpnr/\Qp).\]

%%%%%%%%%%%%%%%%%%%%
\begin{lem} \label{lem:unramified_morphism}
(1) For any morphism $\theta^{\nr}:\fD\rightarrow\fG_H^{\nr}$ of $\Qp^{\nr}/\Qp$-Galois gerbs, its inflation $\overline{\theta^{\nr}}:\fG_p\rightarrow \fG_H$ to $\Qpb$ is an unramified morphism.

(2) For every morphism $\theta:\fG_p\rightarrow \fG_H$ of $\Qpb/\Qp$-Galois gerbs, there is a morphism $\theta^{\nr}:\fD\rightarrow\fG_H^{\nr}$ of $\Qp^{\nr}/\Qp$-Galois gerbs whose inflation to $\Qpb$ is conjugate to $\theta$. More precisely, if $\theta$ factors through $\fG_p^K$ for a finite extension $K$ of $\Qp$, there is a morphism $\theta^{\nr}:\fD_n\rightarrow\fG_H^{\nr}$ with $n=[K:\Qp]$ whose inflation to $\Qpb$ is conjugate to $\theta$.

(3) A morphism $\theta^{\nr}$ in (2) is determined uniquely up to conjugation by an element of $H(\Qpnr)$.

(4) For every unramified morphism $\theta:\fG_p\rightarrow \fG_H$ of $\Qpb/\Qp$-Galois gerbs, the element $b\in H(\Qpb)$ defined by $\theta(s_{\widetilde{\sigma}})=b \widetilde{\sigma}$ for an element $\widetilde{\sigma}\in \Gal(\Qpb/\Qp)$ lifting $\sigma$ lies in $H(\Qpnr)$ and moreover does not depend on the choice of the lifting $\widetilde{\sigma}$.
\end{lem}

\begin{proof} 
(1) Suppose that $\theta^{\nr}$ factors through $\fD_n$ so that $\overline{\theta^{\nr}}$ factors through $\fG_p^{L_n}$. The $\Qpb/\Qp$-Galois gerb $\fG_p^{L_n}$ is obtained from $\Qpnr/\Qp$-Galois gerb $\fD_n$ by pull-back along $\pi:\Gal(\Qpb/\Qp)\thra \Gal(\Qpnr/\Qp)$, followed by push-out along $\Gm(\Qpnr)\hra\Gm(\Qpb)$. 
To show that $\overline{\theta^{\nr}}$ is unramified, we may consider the morphism $\pi^{\ast}\theta^{\nr}: \pi^{\ast}\fD_n\rightarrow \pi^{\ast}\fG_H^{\nr}$ obtained by pull-back only, as the section to $\fG_p^{L_n}\thra \Gal(\Qpb/\Qp)$ induced, via inflation, from a section to $\fD_n\thra \Gal(\Qpnr/\Qp)$  lands in (the image in the push-out of) the pull-back $\pi^{\ast}\fD_n$. 
But, the pull-back $\pi^{\ast}\fD_n$ is also obtained as the pull-back of the $L_n/\Qp$-Galois gerb $\fG^{L_n}_{p,L_n}$ along the surjection $\Gal(\Qpb/\Qp)\thra \Gal(L_n/\Qp)$, followed by push-out along $\Gm(L_n)\hra\Gm(\Qpnr)$. Then, as the section $s^{L_n}:\Gal(\Qpb/\Qp)\rightarrow \fG_p^{L_n}$ is induced from a section $\Gal(L_n/\Qp)\rightarrow \fG^{L_n}_{p,L_n}$, we have $s^{L_n}_{\tau}=1$ for all $\tau\in\Gal(\Qpb/L_n)$.
This proves the claim, since by definition $\pi^{\ast}\fD_n\subset \fD_n\times \Gal(\Qpb/\Qp)$ and the pull-back $\pi^{\ast}\theta^{\nr}$ is defined on the second factor $\Gal(\Qpb/\Qp)$ as the identity.

(2) This is Lemma 2.1 of \cite{LR87} (cf. first paragraph on p.167 of loc.cit). The second assertion is shown in the proof of \textit{loc. cit.}

(3) In general, for any two unramified morphisms $\theta,\theta':\fG_p\rightarrow\fG_H$, if $\theta'=\mathrm{Int}(g_p)\circ \theta$ for some $g_p\in H(\Qpb)$, then it must be that $g_p\in H(\Qpnr)$, since for every $\tau\in\Gal(\Qpb/\Qpnr)$, 
\[1\rtimes\tau=\theta'(s_{\tau}^{K})=g_p\theta(s_{\tau}^{K})g_p^{-1}=g_p(1\rtimes\tau)g_p^{-1}=g_p\tau(g_p^{-1})\cdot\tau.\]
Here, $K\subset \Qpb$ is some finite Galois extension of $\Qp$ for which both $\theta$ and $\theta'$ factor through $\fG_p^K$.

(4) Let $\theta^{\nr}:\fD\rightarrow\fG_H^{\nr}$ be a morphism of $\Qp^{\nr}/\Qp$-Galois gerbs whose inflation to $\Qpb$ is conjugate to $\theta$. By the proof of (3), we have $\theta=g_p\overline{\theta^{\nr}}g_p^{-1}$ for some $g_p\in H(\Qpnr)$. So, if $\theta^{\nr}$ factors through $\fD^{L_n}$ for $n\in\N$, $\theta(s_{\widetilde{\sigma}}^{L_n})=g_p\overline{\theta^{\nr}}(s_{\widetilde{\sigma}}^{L_n})g_p^{-1}=g_p\theta^{\nr}(s_{\sigma}^{L_n})g_p^{-1}\in H(\Qpnr)$. Here, $s^{L_n}$ denotes both the section to $\fD^{L_n}\rightarrow \Gal(\Qpnr/\Qp)$ chosen before and the induced section to $\fG_p^{L_n}\rightarrow \Gal(\Qpb/\Qp)$. The second equality is easily seen to follow from the definition of the inflation $\overline{\theta^{\nr}}$ of a morphism $\theta^{\nr}$.
If $(g_p',{\theta^{\nr}}')$ is another pair with $\theta=g_p'\overline{{\theta^{\nr}}'}(g_p')^{-1}$, then we have the equalities 
\[g_p\theta^{\nr}(s_{\sigma}^{L_n})g_p^{-1}=\theta(s_{\widetilde{\sigma}}^{L_n})=g_p'{\theta^{\nr}}'(s_{\sigma}^{L_n})(g_p')^{-1},\]
so $\theta(s_{\widetilde{\sigma}}^{L_n})=g_p\theta^{\nr}(s_{\sigma}^{L_n})g_p^{-1}$ is independent of the choice of $\widetilde{\sigma}$ as well as that of the pair $(g_p,\theta^{\nr})$.
\end{proof}

%%%%%%%%%%%%%%%%%%%%
\begin{rem}
(1) For a morphism $\theta:\fG_p\rightarrow \fG_H$ of $\Qpb/\Qp$-Galois gerbs, if one chooses a morphism $\theta^{\nr}$ as in (2) and $\theta^{\nr}(s_{\sigma})=b\sigma$ for $b\in H(\Qp^{\nr})$, then by (3) the $\sigma$-conjugacy class of $b$ in $H(\mfk)$ is uniquely determined by $\theta$. Also, any other choice $s_{\sigma}'$ of the section $s_{\sigma}$ gives the same $\sigma$-conjugacy class, since $s_{\sigma}'=us_{\sigma}\sigma(u^{-1})$ for some $u\in \cO_L^{\times}$ (cf. \cite{LR87}, second paragraph on p.167).

(2) Suppose that $\theta$ is itself unramified, and let $b_1\in H(\Qpnr)$ be defined by $\theta(s_{\widetilde{\sigma}})=b_1\widetilde{\sigma}$ for \emph{some} lift $\widetilde{\sigma}\in \Gal(\Qpb/\Qp)$ of $\sigma$. 
Also, let $b\in H(\Qp^{\nr})$ be defined as in (1) for some choice of $\theta^{\nr}:\fD\rightarrow\fG_H^{\nr}$ ($\Qp^{\nr}/\Qp$-Galois gerb morphism). Then (again by Lemma \ref{lem:unramified_morphism}, (3)) the $\sigma$-conjugacy classes of $b$ and $b_1$ are equal.

(3) In \cite[Remark B1.2]{Reimann97}, Reimann uses some specific $s_{\sigma}(n)\in \fD_n$, namely  there exists a unique $s_{\sigma}(n)\in \fD$ such that for every $n$, the image of $s_{\sigma}(n)^n$ in $\fD_n$ is $p^{-[1/n]}\in \Gm(\Qp^{\nr})\subset \fD_n$ (i.e. equals $p^{-1}$ if $n=1$, or otherwise is $1$) and maps to $\sigma$ under $\fD\rightarrow\Gal(\Qpnr/\Qp)$.  There exists a compatible family of such elements $\{s_{\sigma}(n)\}$; we denote by $s_{\sigma}$ the corresponding element of $\fD$ (i.e. the image of $s_{\sigma}$ under the natural map $\fD\rightarrow \fD_n$ is $s_{\sigma}(n)$). 
\end{rem}

Now, the map $\mathrm{cls}_H$ in question is $\theta\mapsto \overline{b(\theta)}\in B(H)$. Note that this map $\mathrm{cls}$ gives the same element in $B(H)$ for all morphisms $\fG_p\rightarrow\fG_H$ lying in a single equivalence class.

%%%%%%%%%%%%%%%%%%%%
\begin{lem} \label{lem:Newton_hom_attached_to_unramified_morphism}
Let $H$ be a connected reductive group over $\Qp$ and $\theta:\fD\rightarrow \fG_H^{\nr}$ a morphism of $\Qpnr/\Qp$-Galois gerbs. Let $b\in H(\Qpnr)$ be defined by $\theta(s_{\widetilde{\sigma}})=b \widetilde{\sigma}$ as in Lemma \ref{lem:unramified_morphism}, (4). Suppose that $\theta$ factors through $\fD^{L_n}$. Then,
the Newton homomorphism $\nu_{b}$ attached to $b\in H(\Qpnr)$ (in the sense of \cite{Kottwitz85}, $\S$4.3) is equal to the quasi-cocharcter
\[-\frac{1}{n}\theta^{\Delta}\quad \in\Hom_{\Qpnr}(\mathbb{D},G),\]
where $\theta^{\Delta}$ denotes the restriction of $\theta$ to the kernel $\Gm$ of $\fD^{L_n}$.
\end{lem}

\begin{proof} 
See \textit{Anmerkung} on p.197 of \cite{LR87}.
\end{proof}

%%%%%%%%%%%%%%%%%%%%
\subsubsection{The Weil-number protorus and the pseudo-motivic Galois gerb} \label{subsubsec:pseudo-motivic_Galois_gerb}

In \cite{LR87}, Langlands and Rapoport work with two kinds of ``motivic Galois gerbs'', the quasi-motivic Galois gerb and the pseudo-motivic Galois gerb. The latter is the Galois gerb whose associated Tannakian category is supposed to be, with a suitable choice of a $\Qb$-fibre functor, the Tannakian category of Grothendieck motives over $\Fpb$ (\textit{loc. cit.}, $\S$4). The former's major role in \textit{loc. cit.} is for formulation of the conjecture for the most general Shimura varieties (beyond those satisfying the Serre condition). Here, we will work mainly with the pseudo-motivic Galois gerb. According to \cite[Lem. B3.9]{Reimann97}, this is harmless, at least when the Serre condition for $(G,X)$ holds (i.e. $Z(G)$ splits over a CM field and the weight homomorphism $w_X$ is defined over $\Q$), e.g. if the Shimura datum $(G,X)$ is of Hodge-type and $G$ is the Mumford-Tate group of a generic element $h\in X$.

Since this Serre condition will be assumed largely in most of the statements and our use of the quasi-motivic Galois gerb will be limited to formulation of certain definitions, here we discuss the pseudo-motivc Galois gerb in detail and refer the readers to \cite[Appendix B]{Reimann97} for the definition of the quasi-motivic Galois gerb.

The pseudo-motivic Galois gerb is a Galois gerb over $\Q$, which is also the projective limit of Galois gerbs $\fP(K,m)$ over $\Q$, indexed by CM fields $K\subset\Qb$ Galois over $\Q$ and $m\in\N$. The kernel $P(K,m)$ of $\fP(K,m)$ is a torus over $\Q$ whose character group consists of certain Weil numbers. Here, we give a brief review of their constructions. We begin with $P(K,m)$. 
As before, we fix embeddings $\Qb\hra\Qlb$, for every place $l$ of $\Q$. 

Recall that for a power $q$ of a rational prime $p$ and an integer $\nu\in\Z$, a \textit{Weil $q$-number of weight $\nu=\nu(\pi)\in \Z$} is an algebraic number $\pi$ such that $\rho(\pi)\overline{\rho(\pi)}=q^{\nu}$ for every embedding $\rho:\Q(\pi)\hra\C$. When $K$ is a field containing $\pi$, then for every archimedean place $v$ of $K$, one has
\begin{equation} \label{eq:Weil-number_archimedean_condition}
|\pi|_v=|\prod_{\sigma\in\Gal(K_v/\Q_{\infty})}\sigma\pi|_{\infty}=q^{\frac{1}{2}[K_v:\R]\nu}.
\end{equation}
Here, $|x+\sqrt{-1}y|_v=x^2+y^2$ if $K_v=\C$, while if $K_v=\R=\Q_{\infty}$, $|x|_v$ is the usual absolute value $|x|_{\infty}$ on $\R$ (hence the first equality always holds for any $\pi\in K$).

%%%%%%%%%%%%%%%%%%%%
\begin{defn} \label{defn:Weil-number_torus}
Let $K\subset\Qb$ be a CM-field which is finite, Galois over $\Q$ and $m\in\N$.

(1) The group $X(K,m)$ consists of the Weil $q=p^m$-numbers $\pi$ in $K$ (for some weight $\nu=\nu_1(\pi)$) with the following properties.
\begin{itemize}
\item[(a)] For each prime $v$ of $K$ above $p$, there is $\nu_2(\pi,v)\in\Z$ with
\[|\pi|_v=|\prod_{\sigma\in\Gal(K_v/\Qp)}\sigma\pi|_p=q^{\nu_2(\pi,v)}.\]
\item[(b)] At all finite places outside $p$, $\pi$ is a unit.
\end{itemize}

(2) Let $X^{\ast}(K,m)$ be the quotient of $X(K,m)$ (which is finitely generated by Dirichlet unit theorem) by the finite group of roots of unity contained therein (so that $X^{\ast}(K,m)$ is torsion free).  
Let $P(K,m)$ be the $\Q$-torus whose character group $X^{\ast}(P(K,m))$ is $X^{\ast}(K,m)$.
\end{defn}

The point of condition (a), while the first equality is always true (for any $\pi\in K$), is that $|\pi|_v$ is an \emph{integral} power of $q$ (which however may well depend on $v$). One also has 
\[\nu_2(\pi,v)+\nu_2(\pi,\overline{v})=-[K_v:\Qp]\nu_1(\pi),\]
since $\pi\overline{\pi}=q^{\nu_1}$ ($K$ being a CM field, the complex conjugation $\overline{\cdot}$ of $K$ lies in the center of $\Gal(K/\Q)$).

If necessary, to avoid any misunderstandings, we write $\chi_{\pi}$ for the character of $P(K, m)$ which corresponds to a Weil number $\pi\in X(K, m)$. Recall that we fixed embeddings $\Qb\rightarrow\Qpb$, $\Qb\rightarrow\C$. Let $K\subset\Q$ be a Galois CM-field and $v_1$, $v_2$ the thereby determined archimedean and $p$-adic places of $K$, respectively. Then, one can readily see that there exist cocharacters $\nu_1^K$, $\nu_2^K$ in $X_{\ast}(K, m)=X_{\ast}(P(K, m))$ with following properties: 
\begin{eqnarray} \label{eqn:cocharacters_nu^K}
\langle\chi_{\pi},\nu_1^K\rangle&=&\nu_1(\pi), \\
\langle\chi_{\pi},\nu_2^K\rangle&=&\nu_2(\pi,v_2). \nonumber
\end{eqnarray}
A priori, $\nu_1^K$ and $\nu_2^K$ are defined over respectively $K_{v_1}=\C$ and $K_{v_2}\subset\Qpb$, but one readily sees from their definition that
they are defined over respectively $\Q$ and $\Qp$.
Furthermore, for $K\subset K'$ and $m|m'$ (divisible), there exist maps of tori over $\Q$
\[\phi_{K,K'}:P(K',m)\rightarrow P(K,m),\quad \phi_{m,m'}:P(K,m')\rightarrow P(K,m)\]
induced by $\phi_{K,K'}^{\ast}(\pi)=\pi$ and $\phi_{m,m'}^{\ast}(\pi)=\pi^{m'/m}$ for $\pi\in X^{\ast}(K,m)$, and they satisfy that $\phi_{m,m'}(\nu_i')=\nu_i$ and $\phi_{K,K'}(\nu_i')=[K'_{v_i'}:K_{v_i}]\nu_i$ \cite[p.141]{LR87}. Let $P^K:=\varprojlim_{m|m'}P(K,m)$. This protorus over $\Q$ is in fact a torus \cite[Lem. 3.8]{LR87}, with character group $X^{\ast}(P^K)=\varinjlim X^{\ast}(K,m)$ and which splits over $K$. Let $\nu_1^K$,$\nu_2^K$ be the induced cocharacters of $P^K$. 

The triple $(P^K,\nu_1^K,\nu_2^K)$ is characterized by a universal property:

%%%%%%%%%%%%%%%%%%%%
\begin{lem}  \label{lem:Reimann97-B2.3}
(1) For every CM-field $K\subset\Qb$ which is Galois over $\Q$, $(P^K,\nu_1^K,\nu_2^K)$ is an initial object in the category of all triples $(T,\nu_{\infty},\nu_p)$ where $T$ is a $\Q$-torus which splits over $K$, and, $\nu_{\infty}$ and $\nu_p$ are cocharacters of $T$ defined over $\Q$ and $\Qp$, respectively, and such that 
\[\Tr_{K/K_0}(\nu_p)+[K_{v_2}:\Q_p]\nu_{\infty}=0,\]
where $K_0$ is the totally real subfield of $K$ of index $2$. 

(2) There exists a set $\{\delta_n\}$ with $m|n$, $n$ sufficiently large, of distinguished elements in $P(K,m)(\Q)$ such that for every $\pi\in X^{\ast}(K,m)=X^{\ast}(P(K,m))$, 
\[\chi_{\pi}(\delta_n)=\pi^{\frac{n}{m}},\]
($\frac{n}{m}$ should be divisible by the torsion order of $X(K,m)$) and that, when $K\subset K'$ and $m|m'$ (divisible),
\[\phi_{m,m'}(\delta_n)=\delta_n,\quad \phi_{K,K'}(\delta_n)=\delta_n.\]
Moreover, the set $\{\delta_m^k\ |\ k\in\Z\}$ is Zariski-dense in $P(K,m)$.
\end{lem}

\begin{proof} 
For (1), see \cite[B2.3]{Reimann97}. For (2), if the subset $\{\pi_1,\cdots,\pi_r\} \subset X(K,m)$ forms a basis of $X^{\ast}(K,m)$ (up to torsions) with dual basis $\{\pi_1^{\vee},\cdots,\pi_r^{\vee}\} \subset X(K,m)^{\vee}=X_{\ast}(P(K,m))$, we set $\delta_n:=\sum \pi_i^{n/m}\otimes\pi_i^{\vee}\in \Gm(\Qb)\otimes X_{\ast}(P(K,m))=P(K,m)(\Qb)$, then it clearly satisfies the required properties, cf. \cite[p.142]{LR87}. The last property is Lemma 5.5 of \cite{LR87}; it is stated for a different torus $Q(K,m)$, but the proof carries over to $P(K,m)$. 
\end{proof} 

Set $P:=\varprojlim_K P^K$ (protorus). It is equipped with two morphisms $\nu_1:=\varprojlim_K\nu_1^K:\mathbb{G}_{\mathrm{m}}\rightarrow P$ (defined over $\Q$), $\nu_2:=\varprojlim_K\nu_2^K:\mathbb{D}\rightarrow P_{\Qp}$ (defined over $\Qp$). Often, $\nu_1$ and $\nu_2$ are also denoted by $\nu_{\infty}$ and $\nu_p$, respectively.

%%%%%%%%%%%%%%%%%%%%
%%%%%%%%%%%%%%%%%%%%
\begin{thm} \label{thm:pseudo-motivic_Galois_gerb}
(1) There exists a Galois gerb $\fP$ over $\Q$ together with morphisms $\zeta_v:\fG_v\rightarrow \fP(v)$ for all places $v$ of $\Q$ such that
\begin{itemize}
\item[(i)] $(\fP^{\Delta},\zeta_{\infty}^{\Delta},\zeta_{p}^{\Delta})=(P_{\Qb},(\nu_1)_{\C},(\nu_2)_{\Qpb})$, the identifications being compatible with the Galois actions of $\Gal(\overline{\Q}/\Q)$, $\Gal(\C/\R)$, and $\Gal(\Qpb/\Qp)$ respectively;
\item[(ii)] the morphisms $\zeta_v$, for all $v\neq \infty,p$, are induced by a section of $\fP$ over $\Spec(\overline{\A_f^p}\otimes_{\A_f^p}\overline{\A_f^p})$;
\end{itemize}
where $\overline{\A_f^p}$ denotes the image of the map $\Qb\otimes_{\Q}\A_f^p\rightarrow \prod_{l\neq\infty,p}\Qlb$. 

(2) If $(\fP',(\zeta_v'))$ is another such system, there exists an isomorphism $\alpha:\fP\rightarrow \fP'$ such that, for all $v$, $\zeta_v'$ is isomorphic to $\alpha\circ\zeta_v$, and any two $\alpha$'s arising in this way are isomorphic.

(3) There is a surjective morphism $\pi:\fQ\rightarrow\fP$ such that, for all $l$, $\zeta_l^P$ is algebraically equivalent to $\pi\circ\zeta_l^Q$ \cite[Def. B1.1]{Reimann97}, where $\fQ$ is the quasi-motivic Galois gerb \cite[Appendix B]{Reimann97}.
\end{thm}

\begin{proof} 
(1) and (2): In \cite[$\S$3]{LR87}, Langlands and Rapoport first define, for each CM Galois field $K$, a Galois gerb $\fP^K$ with kernel $P^K$ which, for every place $v$ of $\Q$, is equipped with morphisms  $\zeta_v=\zeta_v^{K_w}:\fG_v^{K_w}\rightarrow \fP^K(v)$ whose restrictions to the kernels are $\nu_v^K$ for $v=\infty,p$, where $w$ is a place of $K$ above $v$. Then, they define $\fP$ as the projective limit of $\fP^K$'s; this requires choosing a place of $\Qb$ above each place $v$ of $\Q$. 
The construction of $\fP^K$ is a direct consequence of their Satz 2.2, but that of $\fP$ is more delicate: for example, for a projective system of algebraic tori $\{T_n\}_{n\in\N}$ over a field $F$, the natural map $H^2_{cts}(F,\varprojlim T_n)\rightarrow \varprojlim_n H^2(F,T_n)$ is not bijective in general (cf. \cite[Prop. 2.8]{Milne03}).
A proper treatment of construction of $\fP$ can be found in \cite{Milne03}, 
(see also the proof of Theorem B 2.8 of \cite{Reimann97}, where Reimann constructs the quasi-motivic Galois gerb $\fQ$, but the whole arguments should carry over to $\fP$ too, since all the relevant cohomological facts remain valid).
In more detail, for $v=p,\infty$, let $d_v^K$ be the image in $H^2(\Qv,P^K)$ of the fundamental class of the field extension $K_v/\Qv$ under the map $\nu_v^K$, where $K_v$ denotes (by abuse of notation) the completion of $K$ at the place induced by the embedding $\Qb\hra\Qvb$ (so, $K_{\infty}=\C$). Then, 
the Galois gerb $\fP^K$ corresponds to a cohomology class in $H^2(\Q,P^K)$ with image
\[(0,d_p^K,d_{\infty}^K)\in H^2(\A^{\{p,\infty\}},P^K)\times H^2(\Qp,P^K)\times H^2(\Q_{\infty},P^K).\]
The same statement holds for $\fP$, too (cf. \cite[$\S$4]{Milne03}). 
The work of Langland and Rapoport \cite[$\S$3]{LR87} and Milne \cite[(3.5b)]{Milne03} show that there exists a unique element in $H^2(\Q,P^K)$ with that property.
Then, by showing that the canonical maps 
\[H^2_{cts}(\Q,P)\rightarrow \varprojlim_K H^2(\Q,P^K),\quad H^1_{cts}(\Q,P)\rightarrow H^1_{cts}(\A,P)\] are isomorphisms \cite{Milne03}, Prop. 3.5, Prop. 3.10), 
Milne concludes the existence of $\fP$ as required.  
The statement (3) is proved in \cite{Reimann97}, Theorem B 2.8. 
\end{proof}

%%%%%%%%%%%%%%%%%%%%
\begin{rem} \label{rem:comments_on_zeta_v}
(1) As was remarked in the proof, to construct $\zeta_v$ (for a place $v$ of $\Q$), we need to choose a place $w$ of $K$ for each CM field $K$ Galois over $\Q$, in a compatible manner. From now on, when we talk about the pair $(\fP,(\zeta_v)_v)$, we will understand that such choice was already made. Clearly, it is enough to fix embeddings $\Qb\hra\Qvb$ for all $v$'s.

(2) As was also pointed out in the proof, for every CM field $K$ Galois over $\Q$ and each place $v$ of $\Q$, by construction, $\zeta_v$ induces a morphism $\fG_p^{K_w}\rightarrow \fP^K(v)$ of Galois $\Qv$-gerbs, where $w$ is the pre-chosen place of $K$ above $v$ (cf. \cite[Satz 2.2]{LR87}).

(3) The proof also establishes the existence of a (constinous) section to the projection $\fP\rightarrow\Gal(\Qb/\Q)$. We fix one and denote it by $\rho\mapsto q_{\rho}$.
\end{rem}

%%%%%%%%%%%%%%%%%%%%
%%%%%%%%%%%%%%%%%%%%
\subsection{The morphism $\psi_{T,\mu}$ and admissible morphisms}  

We consider the $\Q$-pro-torus $R:=\varprojlim_{L}\mathrm{Res}_{L/\\Q}(\mathbb{G}_{\mathrm{m},L})$ ($L$ running through the set of all Galois extensions of $\Q$ inside $\Qb$). Its character group $X^{\ast}(R)$ is naturally identified with the set of all continuous maps $f:\Gal(\Qb/\Q)\rightarrow\Z$, where the Galois action is given by $\rho(f)(\tau)=f(\rho^{-1}\tau),\ \forall\rho,\tau\in\Gal(\Qb/\Q)$.

%%%%%%%%%%%%%%%%%%%%
\begin{lem} \label{lem:defn_of_psi_T,mu}
(1) For any $\Q$-torus $T$ and every cocharacter $\mu$ of $T$, there exists a unique homomorphism $\xi:R\rightarrow T$ such that $\mu=\xi\circ\mu_0$, where $\mu_0\in X_{\ast}(R)$ is defined by $\langle f, \mu_0\rangle =f(id)\in\Z$ for $f\in X^{\ast}(R)$. 

Let $\psi:\fQ\rightarrow\fG_R$ be the morphism of Galois gerbs over $\Q$ in \cite[B.2.8]{Reimann97}. We define 
\[\psi_{T,\mu}:\fQ\rightarrow\fG_T\]
to be the composite of it and the morphism $\fG_R\rightarrow \fG_T$ induced by $\xi:R\rightarrow T$.

(2) $\psi_{T,\mu}$ factors through $\fP$ if $\mu$ satisfies the Serre condition: if
\[(\rho-1)(\iota+1)\mu=(\iota+1)(\rho-1)\mu=0,\quad \forall\rho\in \Gal(\Qb/\Q)\]
(e.g. if $T$ splits over a CM-field and the \textit{weight} $\mu\cdot\iota(\mu)$ of $\mu$ is defined over $\Q$).
If furthermore a CM field $K$ Galois over $\Q$ splits $T$, $\psi_{T,\mu}$ factors through $\fP^K$.

(3) The restriction $\psi_{T,\mu}^{\Delta}: \fQ^{\Delta}=Q_{\Qb}\rightarrow \fG_T^{\Delta}=T_{\Qb}$ of $\psi_{T,\mu}$ to the kernels is defined over $\Q$.
\end{lem}

\begin{proof} 
For (1) and (2), see \cite{Reimann97}, Definition B 2.10 and Remark B 2.11. The last statement of (2) follows from the very construction of $\psi_{T,\mu}$ in Satz 2.2, 2.3 of \cite{LR87} (which is equivalent to that of Reimann, \cite{Reimann97}, at least when it factors through $\fP$).

For (3), it is enough to show that the morphism $\psi^{\Delta}:\fQ^{\Delta}\rightarrow \fG_R^{\Delta}$ is defined over $\Q$. This morphism is constructed explicitly on p.117 of \cite{Reimann97}.
\end{proof}

Let $v$ a place of $\Q$ (mainly, one of $p,\infty$), $T$ a torus over $\Qv$, and $\mu\in X_{\ast}(T)$. 
Suppose that $T$ splits over a finite Galois extension $F$ of $\Qv$. Set 
\[\nu^F:=\sum_{\tau\in\Gal(F/\Q_{v})}\tau\mu,\]and let 
\[1\rightarrow F^{\times}\rightarrow W_{F/\Q_{v}}\rightarrow\Gal(F/\Q_{v})\rightarrow 1\] be the Weil group extension of $F/\Q_{v}$ (cf. \cite{Tate79}); we fix a section $s^{F}_{\rho}$ to the projection $W_{F/\Q_{v}}\rightarrow \Gal(F/\Q_{v})$ so that $d^F_{\rho,\tau}:=s_{\rho}\rho(s_{\tau})s_{\rho\tau}^{-1}$ is a cocycle defining $W_{F/\Q_{v}}$.

\begin{defn}  \label{defn:psi_T,mu}
We define $\xi_{\mu,F}^F:W_{F/\Q_{v}}\rightarrow T(F)\rtimes\Gal(F/\Q_{v})$ by
\begin{eqnarray*}
\xi_{\mu,F}^F(z)&=&\nu^F(z)\quad (z\in F^{\times}),\\
\xi_{\mu,F}^F(s^F_{\rho})&=&\prod_{\tau\in\Gal(F/\Q_{v})}\rho\tau\mu(d^F_{\rho,\tau})\rtimes\rho.
\end{eqnarray*}
\end{defn}

One easily checks that $\xi_{\mu,F}^F$ is a homomorphism (cf. \cite[p.134]{LR87}, \cite{Milne92}, Lemma 3.30 - Example 3.32). By obvious pull-back and push-out, one gets a morphism of Galois gerbs over $\Q_v$: 
\[\xi_{\mu}^F:\fG_v^{F}\rightarrow\fG_T,\] 
(where for $v\neq p, \infty$, we set $\fG_v^{F}$ to be $\fG_v=\Gal(\Qvb/\Qv)$) and further, by passing to the projective limit, a morphism of Galois gerbs over $\Q_v$:
\[\xi_{\mu}:\fG_v\rightarrow\fG_T,\]
which does not depend on the choice of a field $F$ splitting $T$.
These maps are independent, up to conjugation by an element of $T(\Qvb)$, of the choice of section $s_{\rho}$.

%%%%%%%%%%%%%%%%%%%%
\begin{lem} \label{lem:properties_of_psi_T,mu}
(1) If $v=p$ and $F$ is unramified, $\xi_{\mu}$ is unramified (in the sense of \autoref{subsubsec:cls}), and if $\xi_{\mu}(s_{\sigma})=b_{\mu} \sigma$ for $b_{\mu}\in T(\mfk)$, one has $\overline{b_{\mu}}=\overline{\mu(p^{-1})}$ in $B(T_{\Qp})$. 

(2) Suppose that $T$ is a torus defined over $\Q$, split over a finite Galois extension $K$ of $\Q$. For each $v=\infty,p$, let $\xi_{\pm\mu}$ be the morphism defined above for $(T_{\Qv},F=K_w,\pm\mu)$.
Then, $\psi_{T,\mu}(\infty)\circ \zeta_{\infty}$ is conjugate to $\xi_{\mu}$, and $\psi_{T,\mu}(p)\circ \zeta_{p}$ is conjugate to $\xi_{-\mu}$. For $v\neq \infty,p$, $\psi_{T,\mu}(\infty)\circ \zeta_{v}$ is conjugate to the canonical neutralization of $\fG_T(v)$. 
\end{lem}

\begin{proof}
(1) See Lemma 4.3 of \cite{Milne92}. (2) This follows from the construction of $\psi_{T,\mu}$ and $\fP(K,m)$, cf. \cite{LR87}, Satz 2.3 and $\S$3 (esp. (3.i)). 
\end{proof}

%%%%%%%%%%%%%%%
\subsubsection{Shimura data} Let $(G,X)$ be a Shimura datum. 
For a morphism $h:\dS\rightarrow G_{\R}$ in $X$, the associated \emph{Hodge cocharacter}
\[\mu_{h}:\mathbb{G}_{\mathrm{m}\C}\rightarrow G_{\C}\]
is the composite of $h_{\C}:\dS_{\C}\rightarrow G_{\C}$ and the cocharacter of $\dS_{\C}\cong\mathbb{G}_{\mathrm{m}\C}\times \mathbb{G}_{\mathrm{m}\C}$ corresponding to the identity embedding $\C\hookrightarrow\C$. Let $\{\mu_X\}$ denote the $G(\C)$-conjugacy class of cocharacters of $G_{\C}$ containing $\mu_h$ (for any $h\in X$). For a maximal torus $T$ of $G_{\Qb}$, we can consider $\{\mu_X\}$ as an element of $X_{\ast}(T)/W$. Alternatively, when we fix a based root datum $\mathcal{BR}(G,T,B)$, $c(G,X)$ has a unique representative in the associated closed Weyl chamber $\overline{C}(T,B)$, hence will be also identified with this representative: 
\begin{equation*} \label{eq:representataive_Hodge_cocharacter}
\{\mu_X\}\in \overline{C}(T,B).
\end{equation*}
The \emph{reflex field} $E(G,X)$ of a Shimura datum $(G,X)$ is the field of definition of $c(G,X)\in\mathcal{C}_G(\Qb)$, i.e. the fixed field of the stabilizer of $c(G,X)$ in $\Gal(\Qb/\Q)$; so the reflex field, which is a finite extension of $\Q$, is always a subfield of $\C$. When $T$ is a torus, the reflex field $E(T,\{h\})$ is just the smallest subfield of $\Qb\subset\C$ over which the single
morphism $\mu_h$ is defined.

For each $j\in\N$, we denote by $L_j$ the unramified extension of degree $j$ of $\Qp$ in $\Qpb$, and by $\Qpnr$ the maximal unramified extension of $\Qp$ in $\Qpb$. We let $L$ and $\sigma$ denote the completion of $\Qpnr$ and the absolute Frobenius on it, respectively.

%%%%%%%%%%%%%%%%%%%%
\subsubsection{Strictly monoidal categories $G/\widetilde{G}(\kb)$, $\fG_{G/\widetilde{G}}$}

In order to have a satisfactory formalism without the condition that the derived group is simply connected, Kisin \cite[(3.2)]{Kisin17} introduced certain strictly monoidal categories. Recall (cf. \cite[App. B]{Milne92}) that a \textit{crossed module} is a group homomorphism $\alpha:\tilde{H}\rightarrow H$ together with an action of $H$ on $\tilde{H}$, denoted by ${}^h\tilde{h}$ for $h\in H$, $\tilde{h}\in \tilde{H}$, which lifts the conjugation action on itself (i.e. 
$\alpha({}^h\tilde{h})=h\alpha(\tilde{h})h^{-1}$ for $h\in H$, $\tilde{h}\in \tilde{H}$) and such that the induced action of $\tilde{H}$ on itself is also the conjugation action (i.e. ${}^{\alpha(\tilde{g})}\tilde{h}=\tilde{g}\tilde{h}\tilde{g}^{-1}$ for $\tilde{g},\tilde{h}\in \tilde{H}$).  A crossed module $\tilde{H}\rightarrow H$ gives rise to a strictly monoidal category $H/\tilde{H}$. Its underlying category is the groupoid whose objects are the elements of $H$ and whose morphisms are given by $\Hom(h_1,h_2)=\{\tilde{h}\in\tilde{H}\ |\ h_2=\alpha(\tilde{h})h_1\}$; thus the set of morphisms is identified with the set $\tilde{H}\times H$.
The monoidal structure $\otimes$ on this groupoid is given on the objects by the group multiplication on $H$ and on the set of morphisms $\tilde{H}\times H$ by the semi-direct product for the action of $H$ on $\tilde{H}$:
\[(\tilde{h}_1,h_1)\otimes (\tilde{h}_2,h_2):=(\tilde{h}_1{}^{h_1}\tilde{h}_2,h_1h_2).\]

We may regard any group $H$ as the strictly monoidal category $H=H/\{1\}$.

For a strictly monoidal category $C$ and a crossed module $H/\tilde{H}$, two functors $\phi_1, \phi_2:C\rightarrow H/\tilde{H}$ of strictly monoidal categories are said to be \textit{conjugate-isomorphic} if $\phi_1$ is conjugate to another functor that is isomorphic to $\phi_2$.

Let $k$ be a field with an algebraic closure $\kb$, and $G$ a connected reductive group over $k$. Here, we will use the notation $\tilde{G}$ for the simply connected cover of $G^{\der}$ (which was denoted previously by $G^{\uc}$). Then, the commutator map $[\ ,\ ]:G\times G\rightarrow G$ factors through $[\ ,\ ]:G^{\ad}\times G^{\ad}\rightarrow G$. In particular, as $G^{\ad}=\tilde{G}^{\ad}$, we get a map $[\ ,\ ]:G^{\ad}\times G^{\ad}\rightarrow \tilde{G}$ \cite[2.0.2]{Deligne79}.
It follows that the conjugation action of $\tilde{G}$ on itself extends to an action of $G$, and thus
the natural map $\tilde{G}\rightarrow G$ has a canonical crossed module structure.
We write $G/\widetilde{G}(\kb)$ for the resulting strictly monoidal category $G(\kb)/\widetilde{G}(\kb)$, and 
$\fG_{G/\widetilde{G}}$ for the strictly monoidal category $\fG_G/\tilde{G}(\kb)$.

%%%%%%%%%%%%%%%%%%%%
\subsubsection{Admissible morphisms}
Let $(G,X)$ be a Shimura datum with reflex field $E\subset\C$. We fix an embedding $\Qb\hra\Qvb$ for every  place $v$. Suppose given a parahoric subgroup $\mbfK_p\subset G(\Qp)$; there exists a unique $\sigma$-stable parahoric subgroup $\mbfKt_p$ of $G(\mfk)$ such that $\mbfK_p=\mbfKt_p\cap G(\Qp)$.

Fix $h\in X$. Then, there exists a homomorphism of $\C/\R$-Galois gerbs 
\[\xi_{\infty}:\fG_{\infty}\rightarrow\fG_G(\infty)\]
defined by $\xi_{\infty}(z)=w_h^{-1}(z)=\mu_h\cdot\overline{\mu_h}(z),\ z\in\C^{\times}$ and $\xi_{\infty}(w)=\mu_h(-1)\rtimes \iota$, where $w=w(\iota)$. Clearly, its equivalence class depends only on $X$.

For $v\neq\infty, p$, we have the canonical section $\xi_v$ to $\fG_G(v)\rightarrow\Gal(\overline{\Q}_v/\Q_v)$:
\[\xi_v:\fG_v=\Gal(\overline{\Q}_v/\Q_v) \rightarrow\fG_G(v)\ :\ \rho\mapsto 1\rtimes\rho.\] 

For a cocharacter $\mu$ of $G$, we consider the composite of morphisms of strictly monoidal categories \[\mu_{\widetilde{\ab}}:\Gm(\Qb)\stackrel{\mu}{\rightarrow} G(\Qb) \rightarrow G/\tilde{G}(\Qb).\]

For a cocharacter $\mu$ of $G$ factoring through a maximal $\Q$-torus $T$, the composite \[\psi_{\mu_{\widetilde{\ab}}}:\fQ \stackrel{i\circ\psi_{T,\mu}}{\rightarrow}\fG_G\rightarrow \fG_{G/\tilde{G}}.\] 
(of morphisms of strictly monoidal categories) depends only on the $G(\Qb)$ conjugacy class of $\mu$; One easily verifies that this is isomorphic to the morphism denoted by the same symbol in \cite[(3.3.1)]{Kisin17}.

%%%%%%%%%%%%%%%%%%%%
\begin{defn} \label{defn:admissible_morphism} \cite[p.166-168]{LR87}, \cite[(3.3.6)]{Kisin17}
A morphism $\phi:\fQ\rightarrow\fG_G$ is called \textit{admissible} if
\begin{itemize}
\item[(1)] The composite
\[\phi_{\widetilde{\ab}}:\fQ\stackrel{\phi}{\rightarrow}\fG_G\rightarrow \fG_{G/\tilde{G}}\] 
is conjugate-isomorphic to the composite $\psi_{\mu_{\widetilde{\ab}}}:\fQ \stackrel{i\circ\psi_{\mu}}{\rightarrow}\fG_G\rightarrow \fG_{G/\tilde{G}}$.
\item[(2)] For every place $v\neq p$ (including $\infty$), the composite $\phi(v)\circ\zeta_v$ is conjugate to $\xi_v$ (by an element of $G(\Qlb)$).
\item[(3)] For some (equiv. any) $b\in G(\mfk)$ in the $\sigma$-conjugacy class $\mathrm{cls}(\phi(p)\circ\zeta_p)\in B(G)$ (\autoref{subsubsec:cls}),
the following set (which is a union of affine Deligne-Lusztig varieties) $X(\{\mu_X\},b)_{\mbfK_p}$ is non-empty:
\[X(\{\mu_X\},b)_{\mbfK_p}:=\{g\in G(\mfk)/\mbfKt_p\ |\ \mathrm{inv}_{\mbfKt_p}(g,b\sigma(g))\in\Adm_{\mbfKt_p}(\{\mu_X\})\}.\]
Here, $\Adm_{\mbfKt_p}(\{\mu_X\})$ is the $\{\mu_X\}$-admissible subset (Def. \ref{defn:mu-admissible_subset}) defined for the parahoric subgroup $\mbfKt_p\subset G(\mfk)$ attached to $\mbfK_p$, and
\[\mathrm{inv}_{\mbfKt_p}: G(\mfk)/\mbfKt_p \times G(\mfk)/\mbfKt_p \rightarrow \mbfKt_p\backslash G(\mfk)/\mbfKt_p \cong \tilde{W}_{\mbfKt_p}\backslash \tilde{W}/ \tilde{W}_{\mbfKt_p}\] 
is defined by $(g_1\mbfKt_p,g_2\mbfKt_p)\mapsto \mbfKt_pg_1^{-1}g_2\mbfKt_p$,
cf. (\ref{eqn:parahoric_double_coset}).
\end{itemize}
\end{defn}

Suppose that an admissible morphism $\phi:\fQ\rightarrow\fG_G$ factors through the pseudo-motivic Galois gerb $\fP$. 
As $G$ is an algebraic group, it further factors through $\fP(K,m)$ for some CM field $K$ Galois over $\Q$ and $m\in\N$.

%%%%%%%%%%%%%%%%%%%%
\begin{rem} \label{rem:Kisin's_defn_of_admissible_morphism}
This definition of admissible morphism is slightly different from the original definition by Langlands and Rapoport \cite[p.166]{LR87}, in that, instead of condition (1) here, which was introduced by Kisin \cite[(3.3.6)]{Kisin17}, they require the equality $pr\circ\phi=pr\circ i\circ\psi_{T,\mu}$, where $pr:\fG_G\rightarrow \fG_{G^{\ab}}$ is the natural map; so, these two conditions differ only when $G^{\der}$ is not simply connected. The original condition however turns out to be adequate only in their set-up assuming that $G^{\der}=G^{\uc}$. For example, Satz 5.3 of \cite{LR87} shows that under that assumption on $G$, every admissible morphism (in the original sense) is conjugate to a special admissible morphism (cf. Lemma \ref{lem:LR-Lemma5.2} below). But, they also give an example of $G$ with $G^{\der}\neq G^{\uc}$ (\cite{LR87}, $\S$6, the first example) for which this statement fails to be true (cf. \cite{Milne92}, Remark 4.20). In contrast, with the new condition here, this property always holds (as shown by Kisin for hyperspecial levels, and by Thm. \ref{thm:LR-Satz5.3} below for general parahoric levels when $G_{\Qp}$ is quasi-split).
\end{rem}

The following lemma was proved by Langlands-Rapoport for unramified $T$ (cf. \cite[4.3]{Milne92}). 

%%%%%%%%%%%%%%%%%%%%
\begin{lem} \label{lem:unramified_conj_of_special_morphism}
Let $T$ be a torus over $\Qp$, split by a finite Galois extension $K$ of $\Qp$, say of degree $n$, and $\mu\in X_{\ast}(T)$. Let $K_1$ be the composite in $\Qpb$ of $K$ and $L_n$, where $L_n$ is the unramified extension of $\Qp$ in $\Qpb$ of degree $n=[K:\Qp]$. Then, $\xi_{\mu}^{K_1}$ factors through $\fG^{L_n}_p$: let $\xi_{\mu}^{L_n}:\fG^{L_n}_p\rightarrow \fG_{T}$ be the resulting morphism.
When $\xi_p'$ is an unramified conjugate of $\xi_{\mu}^{L_n}$, we have
\[\xi_p'(s_{\rho}^{L_n})=\Nm_{K/K_0}(\mu(\pi^{-1}))\rtimes \rho,\] 
up to conjugation by an element of $T(\Qpnr)$.
Here, $K_0=K\cap L_n$ is the maximal subfield of $K$ unramified over $\Qp$, $\pi$ is a uniformizer of $K$, and $\rho$ is any element in $\Gal(\Qpb/\Qp)$ whose restriction to $L_n$ is the Frobenius automorphism $\sigma$.

Moreover, we have the equality in $X_{\ast}(T)_{\Gal(\Qpb/\Qpnr)}$:
\[w_{T}(\Nm_{K/K_0}(\mu(\pi^{-1})))=-\underline{\mu},\] 
where $\underline{\mu}$ is the image of $\mu\in X_{\ast}(T)$ in $X_{\ast}(T)_{\Gal(\Qpb/\Qpnr)}$.
\end{lem}

\begin{proof} 
Let $d^{L_n}_{\rho,\tau}$ denote the canonical fundamental $2$-cocycle defined in (\ref{eq:canonical_fundamental_cocycle}) which represents the fundamental class $u_{L_n/\Qp}=[1/n]\in H^2(L_n/\Qp)\cong \frac{1}{n}\Z/\Z$. Also, for each of $F=K_1$ and $L_n$, fix a section $s_{\rho}^{F}:\Gal(F/\Qp)\rightarrow W_{F/\Qp}$ to $1\rightarrow F^{\times}\rightarrow W_{F/\Qp}\rightarrow \Gal(F/\Qp)\rightarrow 1$ whose induced $2$-cocycle on $\Gal(F/\Qp)$
\[d^{F}_{\rho,\tau}:=s^{F}_{\rho}\rho(s^{F}_{\tau})(s^{F}_{\rho\tau})^{-1}\in F^{\times}\] 
represents the fundamental class $u_{F/\Qp}\in H^2(F/\Qp)\cong \frac{1}{[F:\Qp]}\Z/\Z$. In the case $F=L_n$, we further require that such induced $2$-cocycle is the canonical one.
Thus there exists a function $b:\Gal(K_1/\Qp)\rightarrow K_1^{\times}$ such that
\begin{equation} \label{eqn:inflations_of_two_cocycles}
(d^{K_1}_{\rho,\tau})^{[K_1:L_n]}\cdot\partial(b)_{\rho,\tau}=d^{L_n}_{\rho|_{L_n},\tau|_{L_n}},
\end{equation}
where $\partial(b)_{\rho,\tau}:=b_{\rho}\rho(b_{\tau})b_{\rho\tau}^{-1}$.
In terms of these generators and the function $b_{\rho}$, we obtain a homomorphism $p_{K_1,L_n}:\fG^{K_1}_{\Qp}\rightarrow \fG^{L_n}_{\Qp}$ defined by
\[z\mapsto z^{[K_1:K]}\ (z\in\Qpb^{\times}),\quad s^{K_1}_{\rho}\mapsto \ b_{\rho}^{-1}s^{L_n}_{\rho}.\]
Then, the morphisms $\xi_{\mu}^{K_1}, \xi_{\mu}^{L_n}\circ p_{K_1,L_n}:\fG^{K_1}_{\Qp}\rightarrow\fG_T$ differ from each other by conjugation with an element of $T(\Qpb)$. 

Recall that for $(T,\mu,K_1)$, $\xi_{\mu}^{K_1}:\fG_p\rightarrow \fG_T$ is induced, via obvious pull-back and push-out, from a map $\xi_{\mu,K_1}^{K_1}:W_{K_1/\Qp}\rightarrow T(K_1)\rtimes \Gal(K_1/\Qp)$: for $a\in K_1^{\times}$ and $\rho\in \Gal(K_1/\Qp)$, \[\xi_{\mu,K_1}^{K_1}:a\cdot s_{\rho}^{K_1} \mapsto \nu^{K_1}(a)\cdot c^{K_1}_{\rho}\rtimes\rho,\] where $\nu^{K_1}=\Nm_{K_1/\Qp}\mu\in \Hom_{\Qp}(\Gm,T)$ and $c^{K_1}_{\rho}=\prod_{\tau_1\in\Gal(K_1/\Qp)}(\rho\tau_1\mu)(d^{K_1}_{\rho,\tau_1})$.
Now, for any $x\in T(\Qpb)$, if we define $\psi'_x:\fG^{L_n}_p\rightarrow \fG_T$ by 
\[z\mapsto \nu^K(z),\quad s^{L_n}_{\rho}\mapsto \nu^K(b_{\rho})\cdot c^{K_1}_{\rho}\cdot x\cdot\rho(x^{-1}),\]
where $\nu^K=\Nm_{K/\Qp}\mu$, then it is clear that $\psi'_x\circ p_{K_1,L_n}=\Int (x)\circ \xi^{K_1}_{\mu}$. This proves the first claim. Since $\psi'_x=\Int x\circ\psi'_1$, the second statement will follow if there exists $x\in T(\Qpb)$ such that $\psi'_x(s^{L_n}_{\rho})$ equals $\Nm_{K/K_0}(\mu(\pi^{-1}))$ whenever $\rho|_{L_n}=\sigma$. According to (the proof of) Lemma \ref{lem:equality_restrictions_to_kernels_imply_conjugacy}, this will follow if the two elements $\nu^K(b_{\rho})\cdot c^{K_1}_{\rho}$, $\Nm_{K/K_0}(\mu(\pi^{-1}))$ of $T(K^{\nr})$ have the same image under $\kappa_{T_K}:B(T_K)\isom X_{\ast}(T)_{\Gal(\Qpb/K)}$, where $K^{\nr}$ is the maximal unramified extension of $K$ in $\Qpb$ with completion $L'$ and $B(T_K)$ is the set of $\sigma$-conjugacy classes in $T(L')$ (with respect to the Frobenius automorphism of $L'/K$).
But, as $\kappa_{T_K}$ is induced from $w_{T_{L'}}:T(L')\rightarrow X_{\ast}(T)_{\Gal(\overline{L'}/L')}=X_{\ast}(T)$ (\autoref{subsubsec:Kottwitz_hom}), in turn it suffices to show equality of the images under $w_{T_{L'}}$ of $c^{K_1}_{\rho}\cdot\nu^{K}(b_{\rho})$ and $\Nm_{K/K_0}(\mu(\pi^{-1}))$ when $\rho|_{\Qpnr}=\sigma$.

Choose a set of representatives $\Gamma_1\subset\Gal(K_1/\Qp)$ for the family of left cosets $\Gal(K_1/\Qp)/\Gal(K_1/L_n)$ (so that restriction to $L_n$ gives a bijection $\Gamma_1\isom\Gal(L_n/\Qp)$) and $\rho\in\Gal(K_1/\Qp)$ such that $\rho|_{L_n}=\sigma$. Then, we get
\begin{eqnarray*}
\prod_{\tau_1\in\Gal(K_1/\Qp)}(\rho\tau_1\mu)(\mathrm{Inf}_{L_n}^{K_1}(d^{L_n/\Qp})_{\rho,\tau_1})&=&\prod_{\tau\in\Gamma_1}\prod_{\gamma\in\Gal(K_1/L_n)}(\rho\tau\gamma\mu)(\mathrm{Inf}_{L_n}^{K_1}(d^{L_n/\Qp})_{\rho,\tau\gamma})\\
&=&\rho\prod_{\tau\in\Gal(L_n/\Qp)}(\tau(\Nm_{K_1/L_n}\mu))(d^{L_n}_{\rho|_{L_n},\tau}) \\
&=&\prod_{0\leq i\leq n-1}(\sigma^{i+1}(\Nm_{K_1/L_n}\mu))(d^{L_n}_{\sigma,\sigma^i}) \\
&=&(\Nm_{K_1/L_n}\mu)(p^{-1})=(\Nm_{K/K_0}\mu)(p^{-1}).
\end{eqnarray*}
Here, the last equality $\Nm_{K_1/L_n}\mu=\Nm_{K/K_0}\mu$ (in $X_{\ast}(T)$) holds since $\mu$ is defined over $K$ and restriction to $K$ is a bijection $\Gal(K_1/L_n)\isom\Gal(K/K_0)$.
Then, by taking $\prod_{\tau\in\Gal(K_1/\Qp)}(\rho\tau\mu)$ on both sides of (\ref{eqn:inflations_of_two_cocycles}), we obtain 
\[(c^{K_1}_{\rho}\cdot\nu^{K}(b_{\rho}))^{[K_1:K]}\cdot\rho(f)f^{-1}=(\Nm_{K/K_0}\mu)(p^{-1}),\]
where $f=\prod_{\tau\in\Gal(K_1/\Qp)}\tau\mu(b_{\tau})$. 
Now applying $w_{T_{L'}}$ to both sides, we get 
\begin{eqnarray*}
[K_1:K] w_{T_{L'}}(c^{K_1}_{\rho}\cdot\nu^{K}(b_{\rho}))&=& w_{T_{L'}}((\Nm_{K/K_0}\mu)(p^{-1}))\\
&\stackrel{(\ast)}{=}&[K:K_0] w_{T_{L'}}(\Nm_{K/K_0}(\mu(\pi^{-1}))).
\end{eqnarray*}
Due to the property \cite[(7.3.2)]{Kottwitz97} of the map $w$, the equality $(\ast)$ is deduced from the following stronger formula (comparing the images under $w_{T_L}$, instead of $w_{T_{L'}}$):
\begin{equation} \label{eqn:comparison_of_two_norms}
 w_{T_L}(\Nm_{K/K_0}(\mu)(p))=[K:K_0] w_{T_L}(\Nm_{K/K_0}(\mu(\pi))).
\end{equation}
Here, $\Nm_{K/K_0}(\mu)\in X_{\ast}(T)^{\Gal(K/K_0)}$ so $\Nm_{K/K_0}(\mu)(p)\in \im(K_0^{\times}\rightarrow T(K_0))$, while $\Nm_{K/K_0}(\mu(\pi))$ is the image of $\mu(\pi)\in T(K)$ under the norm map $T(K)\rightarrow T(K_0)$. 
To show this formula, by functoriality for tori $T$ endowed with a cocharacter $\mu$, it is enough to prove this formula in the universal case $T=\Res_{K/\Qp}\Gm$ and $\mu=\mu_K$, the cocharacter of $T_K=(\Gm)^{\oplus \Hom(K,K)}$ corresponding to the identity embedding $K\hra K$. Note that in this case $w_{T_L}=v_{T_L}$ as $X_{\ast}(T)$ is an induced $\Gal(K/\Qp)$-module (\autoref{subsubsec:Kottwitz_hom}).
For any extension $E\supset K$, Galois over $\Qp$, there exists a canonical isomorphism
$T_{E}\cong (\mathbb{G}_{\mathrm{m},E})^{\oplus\Hom(K,E)}$ (product of copies of $\mathbb{G}_{\mathrm{m},E}$, indexed by $\Hom_{\Qp}(K,E)$) such that $\tau\in\Gal(E/\Qp)$ acts on $T(E)=(E^{\times})^{\oplus\Hom(K,E)}$ by
 $\tau (x_{\rho})_{\rho\in \Hom(K,E)}=(\tau (x_{\rho}))_{\tau\circ\rho}$. Then, $\mu_K=(f_{\rho})_{\rho}\in\prod_{\rho\in\Hom(K,K)}X_{\ast}(\Gm)$, where $f_{\rho}=1\in X_{\ast}(\Gm)=\Z$ if $\rho$ is the inclusion $K\hra E$, and $f_{\rho}=0$ otherwise. So, $\Nm_{K/K_0}(\mu_K)$ is $(f_{\rho})_{\rho}\in\prod_{\rho}X_{\ast}(\Gm)$, where $f_{\rho}=1$ if $\rho|_{K_0}$ is the inclusion $K_0\hra K$, and $f_{\rho}=0$ otherwise, and similarly
$\Nm_{K/K_0}(\mu_K(\pi))=(x_{\rho})_{\rho}$, where $x_{\rho}=\rho(\pi)$ if $\rho|_{K_0}=(K_0\hra K)$, and $x_{\rho}=1$ otherwise.
It follows that the element of $T(K_0)$
\[\Nm_{K/K_0}(\mu_K)(p)\cdot \Nm_{K/K_0}(\mu_K(\pi))^{-[K:K_0]},\]
lies in the maximal compact subgroup of $T(K_0)$,  which is nothing but $\ker(v_{T_L})\cap T(K_0)$. This proves the equation (\ref{eqn:comparison_of_two_norms}).

Finally, the equality $w_{T}(\Nm_{K/K_0}(\mu(\pi^{-1})))=-\underline{\mu}$ in $X_{\ast}(T)_{\Gal(\Qpb/\Qpnr)}$ follows from commutativity of diagram (7.3.1) of \cite{Kottwitz97}.
This completes the proof. 
\end{proof}

%%%%%%%%%%%%%%%%%%%%
\begin{lem} \label{lem:LR-Lemma5.2}
Suppose that $G_{\Qp}$ is quasi-split and $\mbfK_p$ is a special maximal parahoric subgroup. Then,
for any special Shimura sub-datum $(T,h:\dS\rightarrow T_{\R})$ satisfying the Serre condition (e.g. if $T$ splits over a CM field and the weight homomorphism $w_X:=(\mu_h\cdot\iota(\mu_h))^{-1}$ is defined over $\Q$), the morphism $i\circ\psi_{T,\mu_h}:\fP\rightarrow \fG_T\hra\fG_G$ (Lemma \ref{lem:defn_of_psi_T,mu}) is admissible, where $i:\fG_T\rightarrow\fG_G$ is the canonical morphism defined by the inclusion $i:T\hra G$.
\end{lem}

Such admissible morphism $i\circ\psi_{T,\mu_h}$ will be said to be \textit{special}; in our use of this notation, $i$ will be often spared its explanation (or sometimes will be even omitted). This fact was proved in \cite[Lem. 5.2]{LR87} for hyperspecial $\mbfK_p$.

\begin{proof} 
The only nontrivial condition in Def. \ref{defn:admissible_morphism} is (3). Let $L$ be a finite Galois extension of $\Q$ splitting $T$ and $v_2$ the place of $L$ induced by the chosen embedding $\Qb\hra\Qpb$.
Put 
\[\nu_p:=(\xi_{-\mu_h}^{L_{v_2}})^{\Delta}=-\sum_{\sigma\in\Gal(L_{v_2}/\Qp)}\sigma\mu_h\quad (\in\Hom_{\Qp}(\Gm,T_{\Qp})),\]
and let $J$ be the centralizer in $G_{\Qp}$ of the image of $\nu_p$. Then, $J$ is a semi-standard $\Qp$-Levi subgroup of $G_{\Qp}$ (i.e. the centralizer of a $\Qp$-split torus) which is also quasi-split as $G_{\Qp}$ is so (Lemma \ref{lem:specaial_parahoric_in_Levi}, (1)). Hence, according to Lemma \ref{lem:specaial_parahoric_in_Levi}, there exists $g\in G(\Qp)$ such that $gJ(\mathfrak{k})g^{-1}\cap \mbfKt_p$ is a special maximal parahoric subgroup of $gJ(\mathfrak{k})g^{-1}$, where $\mbfKt_p\subset G(\mathfrak{k})$ is the special maximal parahoric subgroup associated with $\mbfK_p$: it is enough that for a maximal split $\Qp$-torus $S$ of $G_{\Qp}$ contained in $J$, the apartment $\mcA(\Int g(S),\Qp)$ contains a special point in $\mcB(G,\Qp)$ giving $\mbfK_p$. Set $J':=\Int g(J)$.
Then, by Prop. \ref{prop:existence_of_elliptic_tori_in_special_parahorics} (cf. Remark \ref{rem:properties_of_certain_elliptic_tori_in_special_parahorics}), there exists an elliptic maximal $\Qp$-torus $T'$ of $J'$ such that $T'_{\Qpnr}$ contains (equiv. is the centralizer of) a maximal $\Qpnr$-split $\Qpnr$-torus, say $S'_1$, of $J'_{\Qpnr}$ and that the (unique) parahoric subgroup $T'(\mathfrak{k})_1=\ker\ w_{T'_{\mathfrak{k}}}$ of $T'(\mathfrak{k})$ is contained in $J'(\mathfrak{k})\cap \mbfKt_p$. Let $\mu'$ be the cocharacter of $T'$ that is conjugate to $\Int g(\mu_h)$ under $J'(\Qpb)$ and such that it lies in the closed Weyl chamber of $X_{\ast}(T')$ associated with a Borel subgroup of $G_{\mathfrak{k}}$ (defined over $\mathfrak{k}$) containing $T'_{\mathfrak{k}}$.

Then, $\Int g\circ\xi_{-\mu_h}=\xi_{-\Int g(\mu_h)}$ and $\xi_{-\mu'}$ are equivalent as homomorphisms from $\fG_p$ to $\fG_{J'}$. This can be proved by the original argument in \cite[Lem. 5.2]{LR87}: 
the key fact is that the two cocharacters of $\Int g(J)$, $\Int g(\nu_p)\in X_{\ast}(\Int g(T_{\Qp}))$, $\nu_p'\in X_{\ast}(T')$ are equal and factor through the center of $\Int g(J)$:
\begin{equation} \label{eqn:equality_of_two_cochar}
\Int g(\nu_p):=-\sum_{\sigma\in\Gal(L_{v_2}/\Qp)}\sigma (g\mu_hg^{-1})\qquad =\qquad \nu_p':=-\sum_{\sigma\in\Gal(L_{v_2}/\Qp)}\sigma\mu',
\end{equation}
where $L$ is taken to be large enough so that $L_{v_2}$ splits $T'$ (as well as $T$).
Indeed, they both map into the center of $\Int g(J)$: this is clear for $\Int g(\nu_p)$ as $J=\Cent(\nu_p)$, while $\nu_p'$ maps into a split $\Qp$-subtorus of $T'$, so into $Z(J')$ (as $T'$ is elliptic in $J'$). So, their equality can be checked after composing them with the natural projection $J'\rightarrow J'^{\ab}=J'/J'^{\der}$, but this is obvious since $\Int g(\mu_h)$ is conjugate to $\mu'$ under $J'(\Qpb)$.

But, by Lemma \ref{lem:unramified_conj_of_special_morphism} and commutativity of the diagram (7.3.1) of \cite{Kottwitz97}, we see that for an unramified conjugate $\xi_p'$ of $\xi_{-\mu'}:\fG_p\rightarrow\fG_{T'}$ under $T'(\Qpb)$, under the map $\mathrm{inv}_{T'(\mfk)_1}:T'(\mfk)/T'(\mfk)_1\times T'(\mfk)/T'(\mfk)_1\rightarrow T'(\mfk)/T'(\mfk)_1\cong X_{\ast}(T')_{I}$ ($I:=\Gal(\Qpb/\Qpnr)$), we have the relation
\[\mathrm{inv}_{T'(\mfk)_1}(1,\xi_p'(s_{\rho}))=\underline{\mu'},\]
where $\rho\in\Gal(\Qpb/\Qp)$ is a lift of the Frobenius automorphism $\sigma$.
Hence, as $T'(\mfk)_1\subset \mbfKt_p$, $\mathrm{inv}_{\mbfKt_p}(x_0,\xi_p'(s_{\rho})x_0)$ ($x_0:=1\cdot\mbfKt_p$) equals the image of $t^{\underline{\mu'}}$ in $\tilde{W}_{\mbfK_p}\backslash \tilde{W}/ \tilde{W}_{\mbfK_p}\cong X_{\ast}(T')_{I}/\tilde{W}_{\mbfK_p}$. It remains to show that $\tilde{W}_{\mbfK_p}t^{\underline{\mu'}} \tilde{W}_{\mbfK_p}\in\Adm_{\mbfKt_p}(\{\mu_X\})$. 

For that, let $\tilde{W}=N(\mfk)/T'(\mfk)_1$ denote the extended affine Weyl group, where $N$ is the normalizer of $T'$ (note that $T'_{\mfk}$ contains a maximal split $\mfk$-torus $(S'_1)_{\mfk}$). By our choice of $T'$ (and as $\mbfK_p$ is special), $\tilde{W}$ is a semi-direct product $X_{\ast}(T')_{I}\rtimes \tilde{W}_{\mbfKt_p}$, where $\mbfKt_p$ is the special maximal parahoric subgroup of $G(\mfk)$ corresponding to $\mbfK_p$ and $\tilde{W}_{\mbfKt_p}=(N(\mfk)\cap \mbfK_p)/T'(\mfk)_1$ (which maps isomorphically onto the relative Weyl group $W_0=N(\mfk)/T'(\mfk)$). To fix a Bruhat order on $\tilde{W}$, we choose a $\sigma$-stable alcove $\mbfa$ in the apartment $\mcA(S'_1,\mfk)$ containing a special point, say $\mbfo$, corresponding to $\mbfK_p$. The choice of $\mbfa$ and $\mbfo$ give the semi-direct product decomposition $\tilde{W}=W_a\rtimes\Omega_{\mbfa}$, where $W_a$ is the extended affine Weyl group of $(G^{\uc},T'^{\uc})$ and $\Omega_{\mbfa}\subset \tilde{W}$ is the normalizer of $\mbfa$, and a reduced root system ${}^{\mbfo}\Sigma$ (whose roots $R({}^{\mbfo}\Sigma)$ are elements of $X^{\ast}(S')_{\Q}$) such that $W_a$, as a subgroup of the group of affine transformations of $X_{\ast}(S')_{\R}$, equals the affine Weyl group $Q^{\vee}({}^{\mbfo}\Sigma)\rtimes W({}^{\mbfo}\Sigma)$, cf. \autoref{subsubsec:EAWG}. Also, the choice of $\mbfa$ fixes a set of simple affine roots for $W_a$, in particular a set ${}^{\mbfo}\Delta$ of simple roots for ${}^{\mbfo}\Sigma$.
Then, as each root $\alpha$ of ${}^{\mbfo}\Delta$ is proportional to a relative root of $(G^{\uc},S'^{\uc})$ \cite[1.7]{Tits79}, we can find a set $\Delta\subset X^{\ast}(T')$ of simple roots for the root datum of $(G,T')$ with the property that every $\alpha\in{}^{\mbfo}\Delta$ is the restriction of a multiple of some $\tilde{\alpha}\in\Delta$. Let $\overline{C}_{\mbfa}\subset X_{\ast}(S'_1)_{\R}$ and $\overline{C}\subset X_{\ast}(T')_{\R}$ be the associated closed Weyl chambers. It follows that $\pi(\overline{C})\subset \overline{C}_{\mbfa}$, where $\pi$ is the natural surjection $X_{\ast}(T')_{\R}\rightarrow  (X_{\ast}(T')_{I})_{\R}=X_{\ast}(S'_1)_{\R}$. Now, let $\mu_0\in X_{\ast}(T')$ be the conjugate of $\mu_h$ lying in $\overline{C}$ (so its image $\underline{\mu_0}$ in $(X_{\ast}(T')_{I})_{\R}$ lies in $\overline{C}_{\mbfa}$). 
Then, it suffices to show that $t^{\underline{\mu'}}\leq t^{\underline{\mu_0}}$ in $\tilde{W}$. But, as $\mu'=w\mu_0$ for some $w\in N(\Qpb)/T'(\Qpb)$ and $\mu_0\in \overline{C}$, we see that $\mu_0-\mu'=\sum_{\tilde{\alpha}\in\Delta}n_{\tilde{\alpha}}\tilde{\alpha}^{\vee}\in X_{\ast}(T'^{\uc})$ with $n_{\tilde{\alpha}}\in\Z_{\geq0}$ (cf. \cite[2.2]{RR96}). Since for each $\tilde{\alpha}\in\Delta$, we have $\pi(\tilde{\alpha}^{\vee})\in \Q_{\geq0}\alpha^{\vee}$ for some $\alpha\in {}^{\mbfo}\Delta$, $t^{\underline{\mu_0}}- t^{\underline{\mu'}}=\pi(\mu_0)-\pi(\mu')\in C_{\mbfa}^{\vee}:=\{\sum_{\beta\in{}^{\mbfo}\Delta}c_{\beta}\beta^{\vee}\ |\ c_{\beta}\in\R_{\geq0}\}$. By \cite[Thm.4.10]{Stembridge05}, this implies the asserted inequality in $\tilde{W}$.
\end{proof}

%%%%%%%%%%%%%%%%%%%%
%%%%%%%%%%%%%%%%%%%%
\subsection{Langlands-Rapoport conjeture} \label{subsec:Langalnds-Rapoport conjeture}
In this subsection, we give a formulation of the Langlands-Rapoport conjecture for parahoric levels, following Kisin \cite[(3.3)]{Kisin17} and Rapoport \cite[$\S$9]{Rapoport05}.

For $v\neq p,\infty$, set $X_v(\phi):=\{g_v\in G(\Qvb)\ |\ \phi(v)\circ\zeta_v=\Int(g_v)\circ\xi_v\}$,
and \[X^p(\phi):=\sideset{}{'}\prod_{v\neq\infty,p} \ X_v(\phi),\]
where $'$ denotes the restricted product of $X_v(\phi)$'s as defined in the line 15-26 on p.168 of \cite{LR87}. By condition (2) of Def. \ref {defn:admissible_morphism}, $X^p(\phi)$ is non-empty  (cf. \cite[B 3.6]{Reimann97}), and  
is a right torsor under $G(\A_f^p)$. 

To define the component at $p$, put $\mbfK_p(\Qpnr):=\mcG_{\mbff}^{\mathrm{o}}(\Zpnr)$ for the parahoric group scheme $\mcG_{\mbff}^{\mathrm{o}}$ over $\Zpnr$ attached to $\mbfK_p$ (\autoref{subsubsec:parahoric}) and for $b\in G(\mfk)$, let $\mbfKt_p\cdot b\cdot \mbfKt_p$ denote the invariant $\mathrm{inv}_{\mbfKt_p}(1,b)$. We also recall that for $\theta_g^{\nr}:\fD\rightarrow\fG_{G_{\Qp}}^{\nr}$ (a morphism of $\Qpnr|\Qp$-Galois gerbs), $\overline{\theta_g^{\nr}}$ denotes its inflation to $\Qpb$, and $s_{\sigma}\in \fD$ is the lift of $\sigma\in\Gal(\Qpnr/\Qp)$ chosen in (\autoref{subsubsec:cls}). Then, we set
\begin{eqnarray*}
X_p(\phi):=\{ g\mbfK_p(\Qpnr) \in G(\Qpb)/\mbfK_p(\Qpnr) & | & \phi(p)\circ \zeta_p=\Int(g)\circ \overline{\theta_g^{\nr}}\text{ for some }\theta_g^{\nr}:\fD\rightarrow\fG_{G_{\Qp}}^{\nr} \text{ s.t. } \\
& &\mbfKt_p\cdot b_g\cdot \mbfKt_p\in \Adm_{\mbfKt_p}(\{\mu_X\}),\text{ where }\theta_g^{\nr}(s_{\sigma})=b_g\sigma \}.
\end{eqnarray*}
This set is equipped with an action of a $p^r$-Frobenius $\Phi$ ($r:=[\kappa(\wp):\Fp]$) defined by 
\begin{equation} \label{eq:Frob_Phi_1}
\Phi(g\mbfK_p(\Qpnr)):=g\cdot \Nm_rb_g\cdot\mbfK_p(\Qpnr)
\end{equation}
(this definition does not depend on the choice of a representative $g$ in the coset $g\mbfK_p(\Qpnr)$, as $b_{gk}=k^{-1}b_g\sigma(k)$ for $k\in G(\Qpnr)$). To see that this action on $G(\Qpb)/\mbfK_p(\Qpnr)$ leaves $X_p(\phi)$ stable, we use a more explicit description of the set $X_p(\phi)$.
When we choose $g_0\in G(\Qpb)$ with $g_0\mbfK_p(\Qpnr)\in X_p(\phi)$ and use it as a reference point, we obtain a bijection
\begin{equation} \label{eqn:X_p(phi)=ADLV}
X_p(\phi)\isom X(\{\mu_X\},b_{g_0})_{\mbfK_p}\ :\ h\mbfK_p(\Qpnr)\mapsto g_0^{-1}h\mbfKt_p,
\end{equation}
where $b_{g_0}\sigma=\theta_{g_0}^{\nr}(s_{\sigma})$ for $\theta_{g_0}^{\nr}$ with $\phi(p)\circ \zeta_p=\Int(g_0)\circ \overline{\theta_{g_0}^{\nr}}$.
Indeed, for $h\in G(\Qpb)$, if $\Int h^{-1}\circ\phi(p)\circ\zeta_p=\Int(g_0^{-1}h)^{-1}\circ \overline{\theta_{g_0}^{\nr}}$ is unramified, $g:=g_0^{-1}h\in G(\Qpnr)$ (cf. proof of Lemma \ref{lem:unramified_morphism}, (3)), and  
\[b_{h}\sigma=\theta_{h}^{\nr}(s_{\sigma})=\Int(g^{-1})\circ \theta_{g_0}^{\nr}(s_{\sigma})=g^{-1}b_{g_0}\sigma(g)\sigma,\]
and by definition, $g\mbfKt_p\in X(\{\mu_X\},b_{g_0})_{\mbfK_p}$ if and only if $\mbfKt_p\cdot g^{-1}b_{g_0}\sigma(g)\cdot\mbfKt_p\in \Adm_{\mbfKt_p}(\{\mu_X\})$. 
So, $h\in X_p(\phi)$ if and only if $g_0^{-1}h \mbfKt_p\in X(\{\mu_X\},b_{g_0})_{\mbfK_p}$. 
Then, each $g_0\in G(\Qpb)$ such that $\Int(g_0^{-1})\circ \phi(p)\circ\zeta_p$ is unramified, say inflation of $\theta_{g_0}^{\nr}$ (Lemma \ref{lem:unramified_morphism}) gives an absolute Frobenius automorphism $F=\theta_{g_0}^{\nr}(s_{\sigma})$ acting on $G(\mfk)$ (sending $g\in G(\mfk)$ to $b_{g_0}\sigma(g)$). This also induces an action on $G(\mfk)/\mbfKt_p$ as the facet in $\mcB(G_L)$ defining $\mbfKt_p$ is stable under $\sigma$, and we readily see that its $r$-th iterate 
\begin{equation} \label{eq:Frob_Phi_2}
\Phi=F^r\  :\ g\mapsto (b_{g_0}\sigma)^r(g)=b_{g_0} \sigma(b_{g_0}) \cdots \sigma^{r-1}(b_{g_0})\cdot\sigma^r(g)
\end{equation}
is identified with the above Frobenius automorphism of $X_p(\phi)$ under (\ref{eqn:X_p(phi)=ADLV}).
This $\Phi$ leaves $X(\{\mu_X\},b_{g_0})_{\mbfK_p}$ stable,
because $(\Phi g)^{-1}\cdot b_{g_0}\cdot\sigma(\Phi g)=\sigma^r(g^{-1}b_{g_0}\sigma(g))$ and $\Adm_{\mbfKt_p}(\{\mu_X\})$ is stable under the action of $\sigma^r$ on $\tilde{W}_{\mbfKt_p}\backslash \tilde{W}/\tilde{W}_{\mbfKt_p}$ (as $\mu$ is defined over $E_{\wp}$).
Note that there are natural left (or right) actions of $Z(\Qp)$ on $X_p(\phi)$ and $X(\{\mu_X\},b_{g_0})_{\mbfK_p}$ compatible with the above bijection: $z\in Z(\Qp)$ sends $g\in X(\{\mu_X\},b_{g_0})_{\mbfK_p}$ (resp. $g\mbfK_p(\Qpnr)\in X_p(\phi)$) to $zg$ (resp. to $zg\mbfK_p(\Qpnr)$).
Obviously, $\Phi$ also commutes with the actions of $Z(\Qp)$.

Let 
\begin{equation*}
I_{\phi}(\Q):=\{g\in G(\Qb)\ |\ \mathrm{Int}(g)\circ\phi=\phi\}
\end{equation*}
(as the notation suggests, this is the $\Q$-points of an algebraic $\Q$-group $I_{\Q}$, cf. (\ref{eq:inner-twisting_by_phi})).
This group naturally acts on $X^p(\phi)$ and $X_p(\phi)$ from the left and commutes with the (right) action of $G(\A_f^p)$.
Finally, we define
\[S(\phi):=\varprojlim_{\mbfK^p} I_{\phi}(\Q)\backslash (X^p(\phi)/\mbfK^p)\times X_p(\phi),\]
where $\mbfK^p$ runs through the compact open subgroups of $G(\Q^p)$. 
This set is equipped with an action of $G(\A_f^p)\times Z(\Qp)$ and a commuting action of $\Phi$, and as such is determined, up to isomorphism, by the equivalence class of $\phi$.

We note that under the bijection (\ref{eqn:X_p(phi)=ADLV}) (provided by a choice of $g_0\in G(\Qpb)$ with $g_0\mbfK_p(\Qpnr)\in X_p(\phi)$), the action of $I_{\phi}(\Q)$ transfers to $X(\{\mu_X\},b_{g_0})_{\mbfK_p}$ via the map 
\begin{equation} \label{eq:action_of_I_{phi}_on_AffDL}
\Int(g_0^{-1}):I_{\phi}(\Q)=\mathrm{Aut}(\phi) \rightarrow \mathrm{Aut}(\overline{\theta_{g_0}^{\nr}})=\{g\in G(\Qpnr) |\ \mathrm{Int}(g)\circ\theta_{g_0}^{\nr}=\theta_{g_0}^{\nr}\} 
\end{equation}
i.e. $i\in I_{\phi}(\Q)$ sends $g\mbfKt_p\in X(\{\mu_X\},b_{g_0})_{\mbfK_p}$ to $\Int(g_0^{-1})(i)\cdot g\mbfKt_p$.

%%%%%%%%%%%%%%%%%%%%
\begin{conj} \label{conj:Langlands-Rapoport_conjecture_ver1} [Langlands-Rapoport conjecture, 1987]
Suppose that $\mbfK_p\subset G(\Qp)$ is a parahoric subgroup.
Then, there exists an integral model $\sS_{\mbfK_p}(G,X)$ of $\Sh_{\mbfK_p}(G,X)$ over $\cO_{E_{\wp}}$ for which there exists a bijection 
\[\sS_{\mbfK_p}(G,X)(\Fpb)\isom \bigsqcup_{[\phi]}S(\phi)\]
compatible with the actions of $Z(\Qp)\times G(\A_f^p)$ and $\Phi$, where $\Phi$ acts on the left side as the $r$-th (geometric) Frobenius. Here, $\phi$ runs through a set of representatives for the equivalence classes of admissible morphisms $\fQ\rightarrow\fG_G$.
\end{conj}

%%%%%%%%%%%%%%%%%%%%
\begin{rem}
(1) The original conjecture was made under the assumption that $G^{\der}$ is simply connected
 (due to the expectation that only special admissible morphisms are to contribute to the $\Fpb$-points and the existence,  when $G^{\der}\neq G^{\uc}$, of a non-special morphism that is admissible in the original sense, cf. Remark \ref{rem:Kisin's_defn_of_admissible_morphism}).
 
(2) In \cite[Conj. 9.2]{Rapoport05}, Rapoport gave another version of this conjecture, using a different definition of admissible morphisms, where condition (3) of Def. \ref{defn:admissible_morphism} is replaced by the more natural (from group-theoretical viewpoint) and a priori weaker condition (3') that \textit{the $\sigma$-conjugacy class $\mathrm{cls}_{G_{\Qp}}(\phi(p))$ of $b_{\phi}$ lies in $B(G,\{\mu_X\})$}. Our theorem \ref{thm:non-emptiness_of_NS} together with Theorem A of \cite{He15} establishes equivalence of these two versions: previously, it was known that (3) $\Rightarrow$ (3'). 
\end{rem}

%%%%%%%%%%%%%%%%%%%%
%%%%%%%%%%%%%%%%%%%%
\subsection{Kottwitz triples and Kottwitz invariant} 
Our main references for the material covered here are \cite[p.182-183]{LR87}, \cite[$\S2$]{Kottwitz90}, and \cite{Kottwitz92}.

\subsubsection{Galois hypercohomology of crossed modules}
In order to extend the main results of Langlands and Rapoport to the general case that $G^{\der}$ is not necessarily simply connected, (especially, Satz 5.25 of \cite{LR87}, namely Theorem \ref{thm:LR-Satz5.25} here), we need to consider Galois cohomology groups of crossed modules and (length-$2$) complexes of tori quasi-isomorphic to them (for example, $H^1_{\ab}(\Q,G)$ instead of the cohomology of the quotient $G^{\ab}:=G/G^{\der}$). 
Here we give a (very) brief review of the theory of Galois hypercohomology of crossed modules; for details, see \cite{Borovoi98}, \cite[Ch.1]{Labesse99}. 
For a connected reductive group $H$ over a field $k$, we denote by 
\[\rho_H:H^{\uc}\rightarrow H\] 
the canonical map from the simply connected cover $H^{\uc}$ of $H^{\der}$ to $H$. Unless stated otherwise, every bounded complex of groups considered in this paper will be concentrated in non-negative homological degrees.

For a connected reductive group $G$ over a field $k$ (of characteristic zero), if $T$ is a maximal $k$-torus of $G$, the complex $(\rho^{-1}(T)\rightarrow T)$ of $k$-tori, where $\rho^{-1}(T)$ and $T$ are placed in degree $-1$ and $0$ respectively, is quasi-isomorphic to its sub-complex $(\rho^{-1}(Z(G))=Z(G^{\uc})\rightarrow Z(G))$ (cf. \cite[$\S$2, $\S$3]{Borovoi98}), hence, as an object in the derived category of complexes of commutative algebraic $k$-group schemes, depends only on $G$. We denote it by $G_{\bfab}$:
\[G_{\bfab}:=\rho^{-1}(T)\rightarrow T.\]
Then, following Borovoi \cite{Borovoi98}, we define the abelianized Galois cohomology group $H^i_{\ab}(k,G)\ (i\in\Z)$ of $G$ to be the hypercohomology group of $G_{\bfab}(\bar{k}):=(\rho^{-1}(T)(\bar{k})\rightarrow T(\bar{k}))$, two-term complex of discrete $\Gal(\bar{k}/k)$-modules:
\[H^i_{\ab}(k,G):=\mathbb{H}^i(k,G_{\bfab}).\] 
For $-1\leq i\leq 1$, this group is also equal to the cohomology (group) 
\[\mathbb{H}^i(k,G^{\uc}\stackrel{\rho}{\rightarrow} G)\] 
of the crossed module $G^{\uc}(\bar{k})\stackrel{\rho}{\rightarrow} G(\bar{k})$ of $\Gal(\bar{k}/k)$-groups (\textit{loc. cit.} (3.3.2)): for a proof, see the proof of the next lemma.

%%%%%%%%%%%%%%%%%%%%
\begin{lem} \label{lem:abelianization_exact_seq}
Let $H\subset G$ be a (not necessarily connected) $k$-subgroup containing a maximal $k$-torus of $G$ and $(a_{\tau})_{\tau}$ a cochain on $\Gal(\bar{k}/k)$ valued in $H$ whose coboundary $a_{\tau_1}\tau_1(a_{\tau_2})a_{\tau_1\tau_2}^{-1}$ belongs to $Z(H)^{\mathrm{o}}$. Then, for the (simultaneous) inner-twist $\rho_1:\tilde{H}_1\rightarrow H_1$ of the canonical map $\rho:\tilde{H}:=\rho^{-1}(H)\rightarrow H$ via the cocyle $a_{\tau}^{\ad}\in Z^1(\Q,H^{\ad})$ (the image of $a_{\tau}$ in $H^{\ad}$), there exists a natural exact sequence 
\begin{equation} \label{eq:abelianization_from_Levi}
H^1(k,\tilde{H}_1) \stackrel{\rho_{1\ast}}{\longrightarrow} H^1(k,H_1) \stackrel{ab_1}{\longrightarrow} H^1_{\ab}(k,G).
\end{equation}

Note that the condition on (the coboundary of) $a_{\tau}$ allows us to twist $\tilde{H}$ via  $a_{\tau}^{\ad}$.
In the applications, $H$ will be the centralizer of a semi-simple element of $G(k)$.

%%%%%%%%%%%%%%%%%%%%
\end{lem}
\begin{proof}
The first map is the obvious one (induced by $\rho_1$) and the second map is the composite of the natural map 
\[H^1(k,H_1)\rightarrow \mathbb{H}^1(k,\tilde{H}_1\rightarrow H_1)\] 
resulting from the map $(1\rightarrow H_1)\rightarrow (\tilde{H}_1\rightarrow H_1)$ of crossed modules of $k$-groups, and the isomorphisms 
\begin{equation*} \label{eq:isom_of_abelianized_coh}
\mathbb{H}^1(k,\tilde{H}_1\rightarrow H_1)\isom \mathbb{H}^1(k,\tilde{Z}(G)\rightarrow Z(G))\isom \mathbb{H}^1(k,G^{\uc}\rightarrow G)
\end{equation*}
($\tilde{Z}(G):=\rho^{-1}(Z(G))=Z(G^{\uc})$) resulting from the quasi-isomorphisms of crossed modules of $k$-groups 
\begin{equation*}
(\tilde{H}_1\rightarrow H_1)\leftarrow (\tilde{Z}(G)\rightarrow Z(G)) \rightarrow (G^{\uc}\rightarrow G)
\end{equation*}
(cf. \cite{Borovoi98}, Lem. 2.4.1 and its proof: the key point is that $H_1$ contains $Z(G)$). 
Now, the exactness follows from \cite[Cor. 3.4.3]{Borovoi98}.
\end{proof}

%%%%%%%%%%%%%%%%%%%%
\subsubsection{Kottwitz triple} \label{subsubsec:pre-Kottwitz_triple}
A Kottwitz triple is a triple $(\gamma_0;\gamma=(\gamma_l)_{l\neq p},\delta)$ of elements satisfying certain conditions, where
\begin{itemize} \addtolength{\itemsep}{-4pt} 
\item[(i)] $\gamma_0$ is a semi-simple element of $G(\Q)$ that is elliptic in $G(\R)$, defined up to conjugacy in $G(\overline{\Q})$;
\item[(ii)] for $l\neq p$, $\gamma_l$ is a semi-simple element in $G(\Q_l)$, defined up to conjugacy in $G(\Q_l)$, which is conjugate to $\gamma_0$ in $G(\overline{\Q}_l)$;
\item[(iii)] $\delta$ is an element of $G(L_n)$ (for some $n$), defined up to $\sigma$-conjugacy in $G(\mfk)$, such that the norm $\Nm_n\delta$ of $\delta$  is conjugate to $\gamma_0$ under $G(\bar{\mfk})$, where $\Nm_n\delta:=\delta\cdot\sigma(\delta)\cdots\sigma^{n-1}(\delta)\in G(L_n)$.
\end{itemize}

There are two conditions to be satisfied by such triple. 
To explain the first one, put $G_{\gamma_0}:=Z_G(\gamma_0)$
(centralizer of $\gamma_0$ in $G$); if $G^{\der}$ is simply connected, this is a connected reductive group. 
Then, for every place $v$ of $\Q$, we now construct an algebraic $\Q_v$-group $H_0(v)$ and an inner twisting 
\[\psi_v:(G_{\gamma_0})_{\Qvb}\rightarrow H_0(v)_{\Qvb}\] 
over $\Qvb$. 
First, for each finite place $v\neq p$ of $\Q$, set 
\[H_0(v):=Z_{G_{\Q_v}}(\gamma_v).\] 
For any $g_v\in G(\overline{\Q}_v)$ with $g_v\gamma_0g_v^{-1}=\gamma_v$, the restriction of $\Int (g_v)$ to $(G_{\gamma_0})_{\Qv}$ gives us an inner twisting $\psi_v:(G_{\gamma_0})_{\Qv}\rightarrow H_0(v)$, which is well defined up to inner automorphism of $G_{\gamma_0}$.
At $v=p$, we define an algebraic $\Qp$-group $H_0(p)$ 
to be the $\sigma$-centralizer of $\delta$:
\[H_0(p):=G_{\delta\theta}:=\{x\in \Res_{L_n/\Qp}(G_{L_n})\ |\ x\delta \theta(x^{-1})=\delta\},\]
where $\theta$ is the $\Qp$-automorphism of $\Res_{L_n/\Qp}(G_{L_n})$ induced by the restriction of $\sigma$ to $L_n$. Then, there exists an inner twisting $\psi_p:(G_{\gamma_0})_{\Qpb}\rightarrow H_0(p)_{\Qpb}$, which is canonical up to inner automorphism of $G_{\gamma_0}$. For detailed discussion, we refer to \autoref{subsubsec:w-stable_sigma-conjugacy} and \cite[$\S$5]{Kottwitz82} (where $\theta$ and $H(p)$ are denoted respectively by $s$ (on p. 801) and $I_{s\delta}$ (on p. 802)).
Finally, at the infinite place, we choose an elliptic maximal torus $T_{\R}$ of $G_{\R}$ containing $\gamma_0$ and $h\in X\cap \Hom(\dS,T_{\R})$. We twist $G_{\gamma_0}$ using the Cartan involution $\Int (h(i))$ on $G_{\gamma_0}/Z(G)$, and get an inner twisting $\psi_{\infty}:(G_{\gamma_0})_{\C}\rightarrow H_0(\infty)_{\C}$. So, $H_0(\infty)/Z(G)$ is anisotropic over $\R$.

%%%%%%%%%%%%%%%%%%%%
\begin{defn}  \label{defn:Kottwitz_triple}
A triple $(\gamma_0;(\gamma_l)_{l\neq p},\delta)$ as in (i) - (iii) (with some $n\in N$) is called a Kottwitz triple of level $n$ if it further satisfies the following two conditions (iv), $\ast(\delta)$:
\begin{itemize} \addtolength{\itemsep}{-4pt}
\item[(iv)] There exists a triple $(H_0,\psi,(j_v))$ consisting of a $\Q$-group $H_0$, an inner twisting $\psi:G_{\gamma_0}\rightarrow H_0$ and for each place $v$ of $\Q$, an isomorphism $j_v:(H_0)_{\Qv}\rightarrow H_0(v)$ over $\Q_v$, unramified almost everywhere, such that $j_v\circ\psi$ and $\psi_v$ differ by an inner automorphism of $G_{\gamma_0}$ over $\Qb_v$. 
\item[$\ast(\delta)$] the image of $\overline{\delta}$ under the Kottwitz homomorphism $\kappa_{G_{\Q_p}}:B(G_{\Q_p})\rightarrow \pi_1(G_{\Q_p})_{\Gamma(p)}$ (\autoref{subsubsec:Kottwitz_hom}) is equal to $\mu^{\natural}$ (defined in (\ref{eqn:mu_natural})).
\end{itemize}
\end{defn}

Two triples $(\gamma_0;(\gamma_l)_{l\neq p},\delta)$, $(\gamma_0';(\gamma_l')_{l\neq p},\delta')$ as in (\autoref{subsubsec:pre-Kottwitz_triple}) with $\delta,\delta'\in G(L_n)$ for (iii)
are said to be \textit{equivalent}, if $\gamma_0$ is $G(\Qb)$-conjugate to $\gamma_0'$, $\gamma_l$ is $G(\Ql)$-conjugate to $\gamma_l'$ for each finite $l\neq p$, and $\delta$ is $\sigma$-conjugate to $\delta'$ in $G(L_n)$ (i.e. there exists $d\in G(L_n)$ such that $\delta'=d\delta\sigma(d^{-1})$). Clearly, for two such equivalent triples, one of them is a Kottwitz triple of level $n$ if and only if the other one is so. Normally, we consider Kottwitz triples having level $n=m[\kappa(\wp):\Fp]$ for some $m\geq1$.

We will also consider the following condition:
\begin{itemize}
\item[$\ast(\epsilon)$] Let $H$ be the centralizer in $G_{\Qp}$ of the maximal $\Qp$-split torus in the center of $Z_{G}(\gamma_0)_{\Qp}$.
Then, there exists a cocharacter of $\mu$ of $H$ lying in the $G(\Qpb)$-conjugacy class $\{\mu_X\}$
such that
\begin{equation} \label{eq:lambda(gamma_0)}
w_H(\gamma_0)=\sum_{i=1}^{n}\sigma^{i-1}\underline{\mu}
\end{equation}
for some $n\in\N$, where $w_H:H(\mfk)\rightarrow \pi_1(H)_I$ is the map from \autoref{subsubsec:Kottwitz_hom} and $\underline{\mu}$ denotes the image of $\mu$ in $\pi_1(H)_I$ ($I=\Gal(\Qpb/\Qpnr)$).
\end{itemize}

We refer to the number $n$ appearing in $\ast(\epsilon)$ as the \textit{level} of this condition (e.g., we will say that the condition $\ast(\epsilon)$ holds for $\gamma_0$ with level $n$).

%%%%%%%%%%%%%%%%%%%%
\begin{rem} \label{rem:condition_(ast(gamma_0))}
1) Note that $H$ is a semi-standard $\Qp$-Levi subgroup of $G_{\Qp}$ with its center containing the maximal $\Qp$-split torus $A_{\epsilon}$ in the center of $Z_{G}(\gamma_0)_{\Qp}$, hence $\Nm_{K/\Qp}\mu$, being a $\Qp$-rational cocharacter of the center of  $Z_{G}(\gamma_0)_{\Qp}$, maps into $A_{\epsilon}$, thus a posteriori into the center of $H$.

2) As one can readily check, this condition generalizes the condition introduced by Langlands and Rapoport \cite[p.183]{LR87} with the same name $\ast(\epsilon)$ in their set-up that the level subgroup is hyperspecial and $G^{\der}=G^{\uc}$. However, in our formulation, even if $G_{\Qp}$ is unramified, we do not (unlike \cite{LR87}) require $\mu$ to be defined over $L_n$ or even over an unramified extension of $\Qp$.
\end{rem}

When the derived group of $G$ is not simply connected, a stable version of the notion of Kottwitz triple is more relevant.
Recall \cite[$\S$3]{Kottwitz82} that for a connected reductive group $F$ over a perfect field $F$, two rational elements $x,y\in G(F)$ are said to be \emph{stably conjugate} if there exists $g\in G(\bar{F})$ such that $gxg^{-1}=y$ and $g^{-1}{}^{\tau}g\in G_s^{\mathrm{o}}$ for all $\tau\in \Gal(\bar{F}/F)$, where $s$ is the semi-simple part of $x$ in its Jordan decomposition. When we just refer to the relation $gxg^{-1}=y$ for some $g\in G(\bar{F})$, we will say that $x,y$ are $\bar{F}$-(or $G(\bar{F})$-)conjugate or \emph{geometrically} conjugate. By definition, a \textit{stable conjugacy class} in $G(F)$ is an equivalence class in $G(F)$ with respect to this stable conjugation relation.

%%%%%%%%%%%%%%%%%%%%
\begin{defn} \label{defn:stable_Kottwitz_triple}
A Kottwitz triple $(\gamma_0;\gamma=(\gamma_l)_{l\neq p},\delta)$, say of level $n\in\N$, is \emph{stable} if it satisfies the following conditions (in addition to (i) - (iv) and $\ast(\delta)$):
\begin{itemize} \addtolength{\itemsep}{-4pt} 
\item[(i$'$)] $\gamma_0\in G(\Q)$ is defined up to stable conjugacy;
\item[(ii$'$)] for each $l\neq p$, $\gamma_0$ is stably conjugate to $\gamma_l$;
\item[(iii$'$)] there exists $c\in G(\mfk)$ such that $c\gamma_0 c^{-1}=\Nm_n\delta$ and $b:=c^{-1}\delta\sigma(c)$ lies in $I_0(\mfk)$ for $I_0:=G_{\gamma_0}^{\mathrm{o}}$ (a priori, one only has $b\in G_{\gamma_0}(\mfk)$: $\delta^{-1}c\gamma_0 c^{-1}\delta=\sigma(\delta)\cdots\sigma^n(\delta)=\sigma(c\gamma_0 c^{-1})=\sigma(c)\gamma_0\sigma(c^{-1})$);
\item[(iv$'$)] Set $I(v):=H(v)^{\mathrm{o}}$ for each place $v$.
There exists a triple $(I,\psi,(j_v))$ consisting of a $\Q$-group $I$, an inner twisting $\psi:I_0\rightarrow I$ and for each place $v$ of $\Q$, an isomorphism $j_v:(I)_{\Qv}\rightarrow I(v)$ over $\Q_v$, unramified almost everywhere, such that $j_v\circ\psi$ and $\psi_v$ differ by an inner automorphism of $I_0$ over $\Qb_v$.
\end{itemize} 
\end{defn}

If $G^{\der}$ is simply connected, by Steinberg's theorem ($G_{\gamma_0}=I_0$), every Kottwitz triple is stable.

%%%%%%%%%%%%%%%%%%%%
\begin{rem} \label{rem:Kottwitz_triples}
(1) It is easy to verify the following fact: let $(\gamma_0;\gamma,\delta)$ be a stable Kottwitz triple of level $n$. If $\gamma_0'\in G(\Q)$ is stably conjugate to $\gamma_0$, $\gamma$ is $G(\A_f^p)$-conjugate to $\gamma'$, and $\delta'$ is $\sigma$-conjugate to $\delta$ in $G(L_n)$, then $(\gamma_0';\gamma',\delta')$ is also a stable Kottwitz triple (of same level). We say that two stable Kottwitz triples $(\gamma_0;\gamma,\delta)$, $(\gamma_0';\gamma',\delta')$ are \emph{stably equivalent} if they are equivalent as Kottwitz triples and $\gamma_0$, $\gamma_0'$ are stably conjugate. A priori, two stable Kottwitz triples can be equivalent without being stably equivalent (see, however, Prop. \ref{prop:triviality_in_comp_gp}, (2)).

(2) By Hasse principle for $H^1(\Q,I_0^{\ad})$, a Kottwitz triple satisfying (i$'$) - (iii$'$) will also fulfill condition (iv$'$) if \emph{some} attached Kottwitz invariant (to be recalled below) is trivial \cite[p.172]{Kottwitz90}.

(3) The condition (iii$'$) should be distinguished from the following stronger condition: 
\begin{itemize}
\item[(iii$''$)] there exists $c\in G(\Qpnr)$ (not just in $G(\mfk)$) fulfilling the same condition as (iii$'$) (i.e. $c\gamma_0 c^{-1}=\Nm_n\delta$ and $b:=c^{-1}\delta\sigma(c)\in I_0(\Qpnr)$). 
\end{itemize}
The two conditions (iii$'$), (iii$''$) are the same if $G_{\gamma_0}=I_0$ (by Steinberg's theorem $H^1(\Qpnr,G_{\gamma_0})=\{1\}$), in particular when $G^{\der}=G^{\uc}$. In general, condition (iii$''$) seems to be strictly stronger than condition (iii$'$).

(4) As $G_{\Qp}$ is quasi-split, there exists a norm mapping $\mathscr{N}=\mathscr{N}_n$ from $G(L_n)$ to the set of stable conjugacy classes in $G(\Qp)$ \cite[$\S$5]{Kottwitz82}: if $G^{\der}=G^{\uc}$, for any $\delta\in G(L_n)$, the $G(\Qpb)$-conjugacy class of $\Nm_n\delta$, being defined over $\Qp$, contains a rational element (i.e. lying in $G(\Qp)$) by \cite[Thm.4.1]{Kottwitz82} and $\mathscr{N}(\delta)$ is defined to be its stable (=geometric) conjugacy class. For general $G$, one uses a $z$-extension to reduce to the former situation (see \textit{loc. cit.} for a detailed argument).
We claim that the new condition (iii$''$) is the same as that \emph{$\mathscr{N}(\delta)$ is the stable conjugacy class of $\gamma_0$} (in which case, we say that $\gamma_0$ is the \emph{stable norm} of $\delta$). Indeed, we first note that condition (iii$''$) holds for a rational element $\gamma_0$ (and $\delta\in G(L_n)$) if and only if it holds for any element of $G(\Qp)$ stably conjugate to it. Also, it follows immediately from definition that there exists a representative $\gamma_s\in G(\Qp)$ of the stable conjugacy class $\mathscr{N}(\delta)$ for which (and $\delta$) condition (iii$''$) holds. In particular, we see that the implication $\Leftarrow$ holds.
Conversely, if condition (iii$''$) holds for $\gamma_0$ and $\delta$, $\gamma_0$ is stably conjugate to any (semi-simple) rational representative $\gamma_s$ of $\mathscr{N}(\delta)$: for any choice of $c$, $c_s\in G(\Qpnr)$ satisfying  (iii$''$) for $\gamma_0$, $\gamma_s$ respectively, we have $\gamma_s=\Int(c_s^{-1}c)(\gamma_0)$ and
\[ c^{-1}c_s\sigma(c_s^{-1}c)= c^{-1}c_s\cdot b_s^{-1}c_s^{-1}\delta \cdot \delta^{-1}cb=\Int(c^{-1}c_s)(b_s)\cdot b\in I_0(\Qpnr).\]
We remark that if $\delta$ satisfies the condition (iii$'$), say $c\gamma_0 c^{-1}=\Nm_n\delta$ for $c\in G(\mfk)$, then there is found in $G(\mfk)$ an element $g$ such that $\gamma_s=\Int(g)(\gamma_0)$ and $g^{-1}\tau(g)\in I_0$ for every $\tau\in W_{\Qp}$ (i.e. $g:=c_s^{-1}c$), but it is not clear whether one can find such $g$ from $G(\Qpb)$ (here, we also demand that $g^{-1}\tau(g)\in I_0$ for every $\tau\in\Gamma_p$), that is, whether $\gamma_s$ and $\gamma_0$ are stably conjugate in the usual sense; to distinguish the two situations, we will say that  $\gamma_s$ and $\gamma_0$ are \textit{w-stably conjugate} if the former condition holds.
This difference in stable conjugacy relation (occurring only for $p$-adic fields) is harmless in that the kind of results that we need and which were previously established based on the usual stable conjugacy (typically, those of the next susubbsection) continue to hold for the w-stable conjugacy.
But one needs to be careful when applying some classical arguments involving stably conjugacy to w-stably conjugacy.
\end{rem}

%%%%%%%%%%%%%%%%%%%%
\subsubsection{The sets $\mathfrak{C}_p^{(n)}(\gamma_0)$, $\mathfrak{D}_p^{(n)}(\gamma_0)$.} \label{subsubsec:w-stable_sigma-conjugacy}

For a (not necessarily connected) group $H$ over a $p$-adic local field $F$, we let $B(H)$ denote the set of $\sigma$-conjugacy classes of elements in $H(L)$ (see \autoref{subsubsec:Kottwitz_hom} for the notation $L$, $\sigma$). We write $[b]_H$ for the $\sigma$-conjugacy class of an element $b\in H(L)$ (the subscript $H$ is inserted when we want to stress the group $H$).

%%%%%%%%%%%%%%%%%%%%
\begin{defn} \label{defn:D_p^{(n)}(gamma_0)}
For a semi-simple $\gamma_0\in G(\Qp)$ and $n\in\N$, 

$\mathfrak{C}_p^{(n)}(\gamma_0)$ denotes the set of $\sigma$-conjugacy classes in $G(L_n)$ of elements $\delta\in G(L_n)$ satisfying conditions (iii$'$) and $\ast(\delta)$ in Def. \ref{defn:Kottwitz_triple} and Def. \ref{defn:stable_Kottwitz_triple}.

$\mathfrak{D}_p^{(n)}(\gamma_0)$ denotes the subset of $B(I_0)$ consisting of $\sigma$-conjugacy classes $[b=c^{-1}\delta\sigma(c)]_{I_0}$ arising from pairs $(\delta;c)$ satisfying conditions (iii$'$) and $\ast(\delta)$.
\end{defn}

Note that for any $\delta\in G(L_n)$ such that there exists $c'\in G(\mfk)$ with $\Nm_n\delta=c'\gamma_0c'^{-1}$, the $\sigma$-conjugacy class in $B(G_{\gamma_0})$ of $b':=c'^{-1}\delta\sigma(c')\in G_{\gamma_0}(\mfk)$ does not depend on the choice of such $c'$. Thus, this gives a well-defined map $\mathfrak{C}_p^{(n)}(\gamma_0) \rightarrow B(G_{\gamma_0}):[\delta]\mapsto [b=c^{-1}\delta\sigma(c)]_{G_{\gamma_0}}$. In fact, this map is injective. Indeed, suppose that there exist $\delta_i\in G(L_n)$ and $c_i\in G(\mfk)$ ($i=1,2$) such that $c_i\gamma_0 c_i^{-1}=\Nm_n \delta_i$; set $b_i:=c_i^{-1}\delta_i\sigma(c_i)\in G_{\gamma_0}$ as usual. If $b_2=d^{-1}b_1\sigma(d)$ for some $d\in G_{\gamma_0}(\mfk)$, then as 
$\Nm_n b_i=c_i^{-1}\Nm_n \delta_i\sigma^n(c_i)=\gamma_0\cdot c_i^{-1}\sigma^n(c_i)$, we see that
\[ \gamma_0\cdot c_2^{-1}\sigma^n(c_2)=d^{-1}(\gamma_0\cdot c_1^{-1}\sigma^n(c_1))\sigma^n(d)=\gamma_0\cdot (c_1d)^{-1}\sigma^n(c_1d),\]
namely $x:=c_1dc_2^{-1}\in G(L_n)$ and $\delta_2=x^{-1}\delta_1\sigma(x)$.
In particular, $\mathfrak{C}_p^{(n)}(\gamma_0)$ is identified with the subset of $B(G_{\gamma_0})$ consisting of the $\sigma$-conjugacy classes $[b=c^{-1}\delta\sigma(c)]_{G_{\gamma_0}}$ arising from pairs $(\delta;c)$ satisfying conditions (iii$'$) and $\ast(\delta)$.
Hence, there exists a cartesian diagram
\[\xymatrix{ \mathfrak{C}_p^{(n)}(\gamma_0) \ar@{^(->}[r] & B(G_{\gamma_0}) \\ 
\mathfrak{D}_p^{(n)}(\gamma_0) \ar@{^{(}->}[r] \ar@{->>}[u] & B(I_0), \ar[u] } \]
where the left vertical map is defined by $[b]_{I_0}\mapsto [b]_{G_{\gamma_0}}$ (this is clearly surjective).

The main goal of this subsection is to give a cohomological description of the sets $\mathfrak{D}_p^{(n)}(\gamma_0)$, $\mathfrak{C}_p^{(n)}(\gamma_0)$ (when $G_{\Qp}$ is quasi-split); see Prop. \ref{prop:B(gamma_0)=D(I_0,G;Qp)}. For $\delta\in G(L_n)$, let $\mathscr{S}\mathscr{C}(\delta)$ denote the set of the $\sigma$-conjugacy classes in $G(L_n)$ of elements $\delta'\in G(L_n)$ that are \emph{stably $\sigma$-conjugacte} to it (\textit{stable $\sigma$-conjugacy} is to \textit{$\sigma$-conjugacy} as \textit{stable conjugacy} is to \textit{rational conjugacy}, cf. \cite[$\S$5]{Kottwitz82}). This set has a cohomological description:
\begin{equation} \label{eq:sigma-conj_in_stable-sigma-conj}
\mathscr{S}\mathscr{C}(\delta)= \im\left(\ker[H^1(\Qp,G_{\delta\theta}^{\mathrm{o}})\rightarrow H^1(\Qp,R)]\rightarrow \ker[H^1(\Qp,G_{\delta\theta})\rightarrow H^1(\Qp,R)]\right)
\end{equation}
where $R=\Res_{L_n/\Qp}G$  \cite[p.806]{Kottwitz82} (see also the discussion below).
If $\delta$ satisfies condition (iii$''$) of Remark \ref{rem:Kottwitz_triples} (for a fixed $\gamma_0\in G(\Qp)$), it also parametrizes the $\sigma$-conjugacy classes of elements $\delta'\in G(L_n)$ satisfying the same condition (iii$''$), because an element $\delta'\in G(L_n)$ satisfies (iii$''$) if and only if its stable norm $\mathscr{N}(\delta')$ is equal to the stable conjugacy class of $\gamma_0$ (Remark \ref{rem:Kottwitz_triples}, (4)), while two elements of $G(L_n)$ have the same stable norms if and only if they are stably $\sigma$-conjugate \cite[Prop.5.7]{Kottwitz82}.
So, when the two conditions (iii$'$), (iii$''$) are the same (which occurs normally when $G^{\der}=G^{\uc}$), our set $\mathfrak{C}_p^{(n)}(\gamma_0)$ becomes a subset of $\mathscr{S}\mathscr{C}(\delta)$ and this leads immediately to its simple cohomological description.
It turns out (Prop. \ref{prop:B(gamma_0)=D(I_0,G;Qp)}) that the same description is still valid in the general case (even if the two conditions (iii$'$), (iii$''$) are not necessarily equivalent). However, our proof of it will be rather indirect: in that regards, note that one cannot appeal to the same argument above, since it is not clear whether the $\sigma$-conjugacy classes in $\mathfrak{C}_p^{(n)}(\gamma_0)$ are stably $\sigma$-conjugate in the usual sense (cf. Remark \ref{rem:Kottwitz_triples}, (4)) while in the identification just cited of stable $\sigma$-conjugacy in terms of stable norm one uses the usual definition of stable $\sigma$-conjugacy.
Instead, we first give a cohomological description for $\mathfrak{D}_p^{(n)}(\gamma_0)$ and use it to obtain the description for $\mathfrak{C}_p^{(n)}(\gamma_0)$.
In our next discussion, we follow \cite[$\S$5]{Kottwitz82} (but will use slightly different notations and conventions). We begin with a review of some basic definitions.

Let $E/F$ be a cyclic extension of degee $n$ contained in $\bar{F}$ (of characteristic zero), and let $\sigma$ be a generator of $\Gal(E/F)$. Let $G$ be a connected reductive group over $F$, and $R:=\Res_{E/F}G_E$ the Weil restriction of the base-change of $G$ (this group was denoted by $I$ in \loccit). 
There exists a natural isomorphism $R_E\isom G_E\times \cdots G_E$, where the factors are ordered such that the $i$-th factor corresponds to $\sigma^i\in \Gal(E/F)\ (i=1,\cdots,l)$ (note the convention different from that of \loccit).
The element $\sigma\in \Gal(E/F)$ determines an automorphism $\theta\in\Aut_F(R)$, which on $R(E)$ takes the form
\[(x_1,\cdots,x_{n-1},x_n)\mapsto (x_2,\cdots,x_n,x_1)\]
Note that the composition
\[\Delta_{\sigma}:G(E)=R(F)\rightarrow R(E)=G(E)\times \cdots\times G(E)\]
is given by
\[x\mapsto (x,x^{\sigma},\cdots, x^{\sigma^{n-1}}).\]
We define an $F$-morphism $N:R\rightarrow R$ by $Nx=x\cdot x^{\theta}\cdots x^{\theta^{n-1}}$; clearly, one has $N\circ \Delta_{\sigma}=\Delta_{\sigma}\circ\Nm$ on $G(E)=R(F)$ ($\Nm$ denotes the map $\Nm_n$ on $G(E)$).

For $x\in G(E)=R(F)$, the \emph{$\sigma$-centralizer} of $x$ is by definition the $F$-subgroup of $R$:
\[G_{x\theta}=\{g\in R\ :\ gxg^{-\theta}=x\}.\]
(We sometimes write $g^{-\theta}$ for $\theta(x^{-1})$).
If $p:R_E=G_E\times\cdots G_E\rightarrow G_E$ denotes the projection onto the factor indexed by the identity element of $\Gal(E/F)$, by restriction $p$ induces an isomorphism 
\begin{equation} \label{eq:p_x}
p_x:(G_{x\theta})_E\rightarrow (G_E)_{\Nm x}
\end{equation}
(\loccit, Lemma 5.4). Define $G_{x\theta}^{\ast}$ to be the inverse image under the $E$-isomorphism $p_x$ of the subgroup $G_{\Nm x}^{\ast}$ of $G_{\Nm x}$. This is an $F$-subgroup of $G_{x\theta}$ (\loccit, Lemma 5.5) which equals the neutral component $G_{x\theta}^{\mathrm{o}}$ of $G_{x\theta}$ if $\Nm x$ is semi-simple, and equals $G_{x\theta}$ when $G^{\der}=G_{\uc}$.

%%%%%%%%%%%%%%%%%%%%
\begin{defn} 
(1) Two elements $x,y$ of $R(\bar{F})$ are \emph{$\bar{F}$-$\sigma$-conjugate} if there exists  $g\in R(\bar{F})$ such that $y=gxg^{-\theta}$.

(2) Two elements $x,y$ of $G(E)=R(F)$ are \emph{stably $\sigma$-conjugate} if there exists $g\in R(\bar{F})$ such that $gxg^{-\theta}=y$ and $g^{-1}\cdot {}^{\tau}g\in G_{x\theta}^{\ast}$ for all $\tau\in\Gamma:=\Gal(\bar{F}/F)$.

For the next definition (of w-stable $\sigma$-conjugacy relation), we work in the following $p$-adic set-up:
Let $k$ be an algebraically closed field of char. $p>0$ and $K=W(k)[\frac{1}{p}]$.
Suppose that $F$ is a finite extension of $\Qp$ in a fixed algebraic closure $\bar{K}$ of $K$ and let $L$ be the composite of $K$ and $F$ in $\bar{K}$. We also assume that $E\subset L$ and use $\sigma$ again to denote the Frobenius automorphism of $L/F$ (as well as its restriction to $E$).

(3) Two elements $x,y$ of $G(E)=R(F)$ are \emph{w-stably $\sigma$-conjugate} if there exists $g\in R(L)$ such that $gxg^{-\theta}=y$ and $g^{-1}\cdot {}^{\sigma}g\in G_{x\theta}^{\ast}$.
\end{defn}

In the set-up of (3), when $\Nm x$ is semi-simple, two elements $x,y\in G(E)=R(F)$ are stably $\sigma$-conjugate if and only if there exists $g\in R(F^{\nr})$ such that $gxg^{-\theta}=y$ and $g^{-1}\cdot {}^{\sigma}g\in G_{x\theta}^{\ast}$, by Steinberg's theorem $H^1(\bar{F}/F^{\nr},G_{x\theta}^{\ast})=\{1\}$.
In particular, in this case, we have the implications: stably $\sigma$-conjugacy $\Rightarrow$ w-stably $\sigma$-conjugacy $\Rightarrow$ $\bar{F}$-$\sigma$-conjugacy.%%
\footnote{The notion of w-stable $\sigma$-conjugacy is introduced here by the author, to bring attention to the subtle point in the notion of $\sigma$-stable conjugacy, parallel to the difference in (ordinary) stable conjugacy discussed in Remark \ref{rem:Kottwitz_triples}, (4), rather than for some explicit use.
The author does not know whether the first implication is an equivalence in general.}

Now, we give an explicit relation between the equations appearing in condition (iii) and in $\bar{F}$-$\sigma$-conjugacy; this will provide another proof of the lemma. 
Let $W(\bar{K}/F)$ be the Weil group, i.e. the group of continuous automorphisms of $\bar{K}$ which fix $F$ pointwise and induce on the residue field $k$ an integral power of the Frobenius automorphism (cf. \cite[$\S$1]{Kottwitz85}); when $k=\Fpb$, the restriction homomorphism $W(\bar{K}/F)\rightarrow \Gal(\bar{F}/F)$ identifies $W(\bar{K}/F)$ with the usual absolute Weil group $W_F$ of $F$ (\loccit, (1.4)). There exists an exact sequence $1\rightarrow \Gal(\bar{K}/L)\rightarrow W(\bar{K}/F)\rightarrow \langle\sigma\rangle \rightarrow 1$ which endows $W(\bar{K}/F)$ with a natural topology such that the injection identify $\Gal(\bar{K}/L)$ with an open subgroup. For a connected reductive group $G$ over $F$ and  $b\in G(L)$, we define a cocyle $b_{\tau}\in Z^1(W(\bar{K}/F), G(\bar{K}))$ by $b_{\tau}:=\Nm_{i(\tau)}b$, where $i(\tau)\in\N$ is determined by $\tau|_L=\sigma^{i(\tau)}$: its cohomology class is the image of $[b]$ under the isomorphism $B(G)=H^1(\langle\sigma\rangle,G(L)) \isom H^1(W(\bar{K}/F), G(\bar{K}))$.

%%%%%%%%%%%%%%%%%%%%
\begin{lem} \label{lem:stable-Norm-conjugacy=stable-sigma-conjugacy}
Suppose given $\delta_1$, $\delta_2\in G(E)$, and let $x:=\Delta_{\sigma}(\delta_1)$, $y:=\Delta_{\sigma}(\delta_2)\in R(F)$.

(1) If $g_0\Nm\delta_1 g_0^{-1}=\Nm \delta_2$ for some $g_0\in G(\bar{K})$, then for every $\tau\in W(\bar{K}/F)$, one has
\begin{equation*} %\label{eq:Norm-conjugacy_cocycle}
\alpha_{\tau}:=g_0^{-1}\cdot \Nm_l\delta_2\cdot {}^{\tau}g_0\cdot \Nm_l\delta_1^{-1}\in G_{\Nm\delta_1},
\end{equation*}
where $l=l(\tau)\in\N$ is determined by $\tau|_E=\sigma^l\ (0\leq l<n)$, and there exists $g\in R(\bar{K})$ such that 
\[gxg^{-\theta}=y,\ \text{ and }\ p_{\delta_1}(g^{-1}\cdot {}^{\tau}g)=\alpha_{\tau},\]
where $p_{\delta_1}:(G_{\delta_1\theta})_E\rightarrow (G_E)_{\Nm \delta_1}$ is the $E$-isomorphism (\ref{eq:p_x}). Moreover, if $g_0$ lies in $G(\bar{F})$ (resp. in $G(L)$), then one can choose such $g$ in $R(\bar{F})$ (resp. in $R(L)$).

(2) Conversely, if $gxg^{-\theta}=y$ for $g\in R(\bar{K})$, then one has $g_0\Nm\delta_1 g_0^{-1}=\Nm \delta_2$, where $g_0\in G(\bar{K})$ is the image of $g$ under the map $\Xi^{(n)}:R(\bar{K})=G(E\otimes\bar{K})\rightarrow G(\bar{K})$ induced by the canonical $E$-algebra homomorphism $E\otimes\bar{K}\rightarrow \bar{K}:l\otimes x\mapsto lx$ (so that $\Xi^{(n)}|_{R(E)}=p$); clearly, if $g$ lies in $R(\bar{F})$ (resp. in $R(L)$), then $g_0$ also lies $G(\bar{F})$ (resp. in $G(L)$).
\end{lem}

\begin{proof}
(1) The first claim is a formal check (by applying $\tau$ to both sides of $g_0\Nm\delta_1 g_0^{-1}=\Nm \delta_2$). For the second claim, we fix an element $\tilde{\sigma}\in \WgF$ lifting $\sigma$. Since $Nx=N\Delta_{\sigma}\delta_1=\Delta_{\sigma}(\Nm \delta_1)$, we have the relation in $R(E)=G(E)\times\cdots\times G(E)$: 
\[Ny=(\Nm \delta_2,\cdots,\Nm \delta_2^{\sigma^{n-1}})=f (\Nm \delta_1,\cdots,\Nm \delta_1^{\sigma^{n-1}})f^{-1}=fNx f^{-1}\]
for $f:=(g_0,g_0^{\tilde{\sigma}},\cdots,g_0^{\tilde{\sigma}^{n-1}})$ (for a moment, we write $g^{\tau}$ for $\tau\in \WgF$ instead of ${}^{\tau}g$).
Next, we find $h\in R(\bar{K})$ satisfying that 
\[ hxh^{-\theta}=y':=f^{-1}yf^{\theta}, \] 
namely, that $(h_1,\cdots,h_n)(x_1,\cdots,x_n)(h_2^{-1},\dots,h_n^{-1},h_1^{-1})=(y_1',\cdots,y_n')$: the condition $Nx=Ny'$ holds (\loccit, Lemma 5.2), and this guarantees that this system of equations in $h_i$'s are consistent and we may choose $h_1$ arbitrarily.
By putting $h_1=1$, we obtain
\[h_{i+1}=(y_1'\cdots y_i')^{-1}(x_1\cdots x_i)=(y_{i+1}'\cdots y_n')(x_{i+1}\cdots x_n)^{-1}.\]
for $i=1,\cdots,n-1$. Hence, we have $gxg^{-\theta}=y$ for $g:=fh\in R(\bar{K})$ and it is clear that if $g_0\in G(E')$ for $E'=\bar{F}$ or $L$ (so that $y'$, $h\in R(E')$ too), one also has $g\in R(E')$.

On the other hand, one checks that any $\tau\in\WgF$ acts on $R(\bar{K})=G(\bar{K})\times \cdots\times G(\bar{K})$ by:
\begin{equation} \label{eq:Galois_action_on_R(Kbar)}
(x_1,\cdots,x_n)\mapsto (\tau(x_{n-l+1}),\cdots,\tau(x_{n-l-1}),\tau(x_{n-l})),
\end{equation}
where $\tau|E=\sigma^l$ with $0\leq l<n$ (``after applying $\tau$, move each component $l$-times to the right"); we set $x_{n+1}:=x_1$, $y_{n+1}:=y_1$. Hence, since $y_i'=f_i^{-1}y_if_{i+1}$ for $1\leq i\leq n$ ($f_{n+1}:=f_1$), it follows that
\begin{align*}
p_{\delta_1}(g^{-1}\cdot {}^{\tau}g) &=(f_1h_1)^{-1}\tau(f_{n-l+1}h_{n+l+1})=f_1^{-1}\tau(f_{n-l+1}(y_{n-l+1}'\cdots y_n')(x_{n-l+1}\cdots x_n)^{-1}) \\
&= f_1^{-1}\tau((y_{n-l+1}\cdots y_n) f_1(x_{n-l+1}\cdots x_n)^{-1}) \\
&= f_1^{-1}\tau(\sigma^{n-l}(\delta_2\cdots \delta_2^{\sigma^{l-1}}) f_1 \sigma^{n-l}(\delta_1\cdots \delta_1^{\sigma^{l-1}})^{-1}) \\
&=g_0^{-1}\cdot \Nm_l\delta_2\cdot \tau(g_0)\cdot \Nm_l \delta_1^{-1}
\end{align*}
as was asserted. The last claim is clear from the construction.

(2) Apply $N:R\rightarrow R$ followed by $\Xi^{(n)}:R(\bar{K})\rightarrow G(\bar{K})$ \cite[Lem.5.2]{Kottwitz82}.
\end{proof}

From this lemma, we make an observation (which itself, however, will not be used in the rest of this work): 

%%%%%%%%%%%%%%%%%%%%
\begin{prop}
Suppose given $(\gamma_0,\delta)\in G(\Qp)\times G(L_n)$ satisfying condition (iii$'$) of Def. \ref{defn:stable_Kottwitz_triple}. Then, an element $\delta'\in G(L_n)$ satisfies condition (iii$'$) for the same $\gamma_0$ if and only if $\delta$ and $\delta'$ are w-stably $\sigma$-conjugate.
\end{prop}

Next, we recall  \cite[4.13, 3.3-3.5]{Kottwitz97}  that for any connected reductive $F$-group $H$, the map 
\begin{equation} \label{eq:kappa_times_nu} \kappa_{H}\times \bar{\nu}_{H}: B(H) \rightarrow X^{\ast}(Z(\hat{H})^{\Gamma_p})\times\mathcal{N}(H) \end{equation} is injective and for any $b\in H(L)$, the image under $\kappa_{H}$ of $\bar{\nu}_{H}^{-1}(\bar{\nu}_H([b]))$ is a torsor under $H^1(F,J_b)$, where $J_b=J_b^H$ is the $F$-group such that for any $F$-algebra $R$, 
\[J_b(R)=\{g\in G(L\otimes R)\ |\ b\sigma(g)=gb\}\] 
($\sigma$ acting via its action on $L$): more explicitly, if 
\begin{equation} \label{eq:Xi}
\Xi:H(L\otimes\bar{K})\rightarrow H(\bar{K})
\end{equation}
denotes the map induced by the canonical $L$-algebra homomorphism $L\otimes\bar{K}\rightarrow \bar{K}:l\otimes x\mapsto lx$, the cocycle $b_{\tau}\in Z^1(W,H(\bar{K}))$ ($W:=W(\bar{K}/F)$) defined by $b$ gives rise to a map \[j_b=j_b^H:Z^1(W,J_b(\bar{K}))\rightarrow  Z^1(W,H(\bar{K})) : x_{\tau} \mapsto \Xi(x_{\tau})\cdot b_{\tau},\] which in turn induces an injection (again denoted by $j_b=j_b^H$)
\begin{equation} \label{eq:j_b}
j_b=j_b^H:H^1(F,J_b)\hookrightarrow B(H)=H^1(W,H)
\end{equation}
whose image is $j_b(H^1(F,J_b))=\bar{\nu}_{H}^{-1}(\bar{\nu}_H([b]))$ (cf. \cite[3.3 - 3.5]{Kottwitz97}).

Now, we assume that $F\subset\Qpb$ is a $p$-adic local field. Then, for any connected reductive $F$-groups $I_0\subset G$ and a basic $b\in I_0(\mfk)$, there exist canonical bijections
\[ \ker[H^1(F,J_b^{I_0})\rightarrow H^1(F,J_b^G)]=\ker[H^1(F,I_0)\rightarrow H^1(F,M)]  =\ker[H^1(F,I_0)\rightarrow H^1(F,G)], \]
where for $H=I_0, G$, $J_b^H$ denotes the corresponding group $J_b$ and $M:=Z_G(\nu)$ for $\nu:=\nu_{I_0}(b)$. 
The second equality follows from the fact that as $M$ is a Levi-subgroup of $G$, the map $H^1(F,M)\rightarrow H^1(F,G)$ is injective (cf. \cite[(4.13.3)]{Kottwitz97}).
To see the first equality, we choose $h\in I_0(L)$ and $n\in\N$ satisfying $\Nm_n(b')=\nu(p)^n$ for $b':=hb\sigma(h^{-1})$ \cite[(4.3)]{Kottwitz85}. Then, the maps $\Int(h):J_b^H\isom J_{b'}^H$ are $F$-isomorphisms and the inclusion $J_{b'}^{I_0}\hookrightarrow J_{b'}^G$ is a simultaneous inner-twist of the inclusion $I_0\hookrightarrow M$, via the cochain $b'_{\tau}\in C^1(W,I_0)$ defined by $b'$ which becomes cocycles in both $I_0^{\ad}$ and $M^{\ad}$ (more precisely, as $\nu$ maps to $Z(I_0)$, the map $\Xi:I_0(\mfk\otimes\bar{\mfk})\rightarrow I_0(\bar{\mfk})$ (\ref{eq:Xi}) restricts to an isomorphism 
\[ \Xi_{b'}:J_{b'}^{I_0}(\bar{\mfk})\isom I_0(\bar{\mfk})\] 
\cite[3.4]{Kottwitz97}, which identifies $J_{b'}^{I_0}$ with the inner twist of $I_0$ via the cocycle in $Z^1(F^{\nr}/F,I_0^{\ad})$ defined by the image of $b'$ in $I_0^{\ad}(F^{\nr})$). 
In view of these, the first equality is due to the canonical bijection $H^1(F,H)=H^1(F,H_{\bfab})$ (for any connected reductive group $H$ over arbitrary non-archimedean local $F$), which also implies that this equality is independent of the choice of a particular simultaneous inner twist $((J_b^{I_0})_{\Qpb}\hookrightarrow (J_b^G)_{\Qpb}) \isom ((I_0)_{\Qpb}\hookrightarrow M_{\Qpb})$. The resulting injection
\begin{equation} \label{eq:j_[b]^{I_0}}
j_{[b]}^{I_0}:\ker[H^1(F,I_0)\rightarrow H^1(F,G)] \hookrightarrow B(I_0)
\end{equation}
depends only on the $\sigma$-conjugacy class $[b]_{I_0}$, not on the choice of its representative in $I_0(\mfk)$ (which justifies the notation $j_{[b]}^{I_0}$).

Returning to the original set-up (where further $F=\Qp$ and $W(\bar{\mfk}/\Qp)=W_{\Qp}$), we know that:
\begin{itemize}
\item[(i)] For any pair of elements $(\delta,c)\in G(L_n)\times G(\mfk)$ satisfying (iii$'$) of Def. \ref{defn:stable_Kottwitz_triple}, the associated element $b:=c^{-1}\delta\sigma(c)$ has Newton quasi-cocharacter $\nu_{I_0}(b)=\frac{1}{n}\nu_{I_0}(\gamma_0)$  (Lemma \ref{lem:equality_of_two_Newton_maps}, (2), cf. \cite[Lem. 5.15]{LR87}). In particular, its $\sigma$-conjugacy class $[b]_{I_0}\in B(I_0)$ is basic.
\item[(ii)] The functorial map $\Int(c^{-1}):G(\mfk\otimes R)\isom G(\mfk\otimes R)$ (for $\Qp$-algebras $R$) restricts to a $\Qp$-isomorphism
\[ G_{\delta\theta}^{\mathrm{o}}\isom J_b^{I_0},\]
Indeed, one has $\gamma_0=(b\sigma)^n\cdot c^{-1}\cdot \sigma^n\cdot c$ (identity in $G(\mfk)\rtimes\langle\sigma\rangle$), which implies that an element $g\in G(\mfk\otimes R)$ commutes with $\gamma_0$ and $b\sigma$ if and only if $cgc^{-1}$ does with $\delta\sigma=c(b\sigma)c^{-1}$ and $\sigma^n$ ($\sigma$ acts on $G(\mfk\otimes R)$ via its action on $\mfk$). The composite of this with $\Xi_b:J_b^{I_0}(\bar{\mfk})\isom J_{b'}^{I_0}(\bar{\mfk})\isom I_0(\bar{\mfk})$ above
\[ G_{\delta\theta}^{\mathrm{o}}(\bar{\mfk}) \isom J_b^{I_0}(\bar{\mfk})\isom I_0(\bar{\mfk})\] 
equals the restriction to the neutral components of the composite isomorphism 
\begin{equation} \label{eq:psi}
\psi:=\Int(c^{-1})\circ p_{\delta}: (G_{\delta\theta})_{\mfk}\isom (G_{\mfk})_{\Nm \delta} \isom (G_{\gamma_0})_{\mfk}
\end{equation} which satisfies $\psi\cdot{}^{\sigma}\psi^{-1}=\Int(b)$.
\item[(iii)]
If the pair $(\delta,c)$ is required further to satisfy condition $\ast(\delta)$ of Def. \ref{defn:Kottwitz_triple}, the resulting $\sigma$-conjugacy class $[b]_{G}\in B(G)$ is a constant one depending only on $\gamma_0$. 
\end{itemize}
Therefore, with choice of a reference element $b=g_p^{-1}\delta\sigma(g_p)\in I_0(\mfk)$, 
there exist inclusions
\begin{equation} \label{eq:B(G(L_n);gamma_0)}
\mathfrak{D}_p^{(n)}(\gamma_0) \subset j_{[b]}^{I_0}(\ker[H^1(\Qp,I_0)\rightarrow H^1(\Qp,G)])\ \subset B(I_0).
\end{equation}

On the other hand, the natural diagram
\begin{equation} \label{eq:stable_sigma_conj_diagm}
 \xymatrix{ H^1(\Qp,G_{\delta\theta}^{\mathrm{o}}) \ar[r]  \ar[d]_{\simeq} & H^1(\Qp,R) \\
H^1(\Qp,I_0) \ar[r] & H^1(\Qp,G), \ar[u] }
\end{equation}
is commutative, where the map $H^1(\Qp,G)\rightarrow H^1(\Qp,R)=H^1(L_n,G)$ is the restriction map. To see this, we replace these cohomology groups by the respective abelianized cohomology groups (Appendix \ref{sec:abelianization_complex}). Then, the isomorphism $H^1(\Qp,(G_{\delta\theta}^{\mathrm{o}})_{\bfab})\isom H^1(\Qp,(I_0)_{\bfab})$ is induced by the restriction of the inner twist $\Int(c^{-1})\circ p_{\delta}$ (\ref{eq:psi}) to the abelianization complexes (which is then a $\Qp$-isomorphism). Since $p_{\delta}$ is the restriction of $\Xi^{(n)}:G(L_n\otimes\Qpb)\rightarrow G(\Qpb)$ and $\Int(c)$ is the identity on $G_{\bfab}=(Z(G^{\uc})\rightarrow Z(G))$, we see that $H^1(\Qp,(G_{\delta\theta}^{\mathrm{o}})_{\bfab})\isom H^1(\Qp,G_{\bfab})$ is induced by the restriction of $\Xi^{(n)}$ to the abelianization complexes. But the identification $H^1(\Qp,\Res_{L_n/\Qp}G)\isom H^1(L_n,G)$ (the Shapiro lemma) is also induced by $\Xi^{(n)}$, which proves the commutativity.

%%%%%%%%%%%%%%%%%%%%
\begin{prop} \label{prop:B(gamma_0)=D(I_0,G;Qp)}
Suppose given a semi-simple $\gamma_0\in G(\Qp)$ and $(\delta,c)\in G(L_n)\times G(\mfk)$ satisfying conditions (iii$'$) and $\ast(\delta)$. Then, the inclusion (\ref{eq:B(G(L_n);gamma_0)}) is an equality:
\[\mathfrak{D}_p^{(n)}(\gamma_0)= j_{[b]}^{I_0}(\ker[H^1(\Qp,I_0)\rightarrow H^1(\Qp,G)]). \] 
Also, we have
\[ \mathfrak{C}_p^{(n)}(\gamma_0)=\im\left(\ker[H^1(\Qp,G_{\delta\theta}^{\mathrm{o}})\rightarrow H^1(\Qp,G)]\rightarrow \ker[H^1(\Qp,G_{\delta\theta})\rightarrow H^1(\Qp,R)]\right). \]
\end{prop}

\begin{proof}
Let $\mathfrak{D}_p^{(n)}(\gamma_0)'$ be the subset of $B(I_0)$ consisting of $\sigma$-conjugacy classes $[b=c^{-1}\delta\sigma(c)]_{I_0}$ arising from the pairs $(\delta;c)$ satisfying condition (iii$'$) (but not necessarily $\ast(\delta)$). Then, we claim that 
\[\mathfrak{D}_p^{(n)}(\gamma_0)'=j_{b}^{I_0}(\ker[H^1(\Qp,G_{\delta\theta}^{\mathrm{o}})\rightarrow H^1(\Qp,R)])\] 
Clearly, this is the same as $\mathfrak{D}_p^{(n)}(\gamma_0)'=j_{b}^{I_0}(\ker[H^1(\Qp,G_{\delta\theta}^{\mathrm{o}})\rightarrow H^1(W_{\Qp},R(\bar{\mfk}))])$, which we now show.
The inclusion $\subset$ follows from the facts that any $b'\in \mathfrak{D}_p^{(n)}(\gamma_0)'$ has the property $\nu_{I_0}(b')=\nu_{I_0}(b)$, and that $j_b^{I_0}(H^1(\Qp,J_b^{I_0}))=\bar{\nu}_{I_0}^{-1}(\bar{\nu}_{I_0}([b]))$.
 
Conversely, any $\beta\in \ker[H^1(\Qp,G_{\delta\theta}^{\mathrm{o}})\rightarrow H^1(\Qp,R)]$ is of the form $[g^{-1}\cdot{}^{\tau}g]$ for some $g\in R(\Qpnr)$. Then, we have $\delta':=g\delta g^{-\theta}\in R(\Qp)$ and $gN\delta g^{-1}=N\delta'$ (in $R(\Qpnr)$). Hence, if $g_0:=\Xi^{(n)}(g)$ for the map $\Xi^{(n)}:R(\Qpnr)\rightarrow G(\Qpnr)$ from Lemma \ref{lem:stable-Norm-conjugacy=stable-sigma-conjugacy}, (2), we have $\Nm\delta'=g_0\Nm\delta g_0^{-1}$ (in $G(\Qpnr)$), that is, $c'\gamma_0c'^{-1}=\Nm\delta'$ for $c':=g_0c\in G(\mfk)$. The inclusion $\supset$ will be established, if we show that
$b':=c'^{-1}\delta'\sigma(c')\in I_0$ and 
\[j_b^{I_0}(\beta)=[b_{\tau}']\] 
for the cocycle $b'_{\tau}$ attached to $b'$ (i.e. $\tau\mapsto \Nm_{l(\tau)}b'$, where $l(\tau)\in\N$ is given by $\tau|_{L_n}=\sigma^{l(\tau)}$ with $0\leq l(\tau)<n$).
If $g=(g_1,\cdots,g_n)$ in $R(\Qpnr)=G(\Qpnr)\times\cdots\times G(\Qpnr)$ (then, $g_0=\Xi^{(n)}(g)=g_1$), it follows from $\Nm\delta'=g_0\Nm\delta g_0^{-1}$ that $g_i=(\Nm_{i-1}\delta')^{-1}\cdot g_1\cdot \Nm_{i-1}\delta$, and thus, as $p_{\delta}({}^\tau g)=\tau(g_{n-l(\tau)+1})$ (\ref{eq:Galois_action_on_R(Kbar)}), one has
\[p_{\delta}(g^{-1}\cdot {}^{\tau}g)=g_0^{-1}\cdot \Nm_{l(\tau)}\delta'\cdot \tau(g_0)\cdot (\Nm_{l(\tau)}\delta)^{-1}=c\cdot \Nm_{l(\tau)}b'\cdot (\Nm_{l(\tau)}b)^{-1}\cdot c^{-1}\]
for any $\tau\in W$; in particular, $b'\in I_0$. Then, since the map $\Xi:G_{\delta\theta}^{\mathrm{o}}(\bar{\mfk})\isom J_b^{I_0}(\bar{\mfk})\isom I_0(\bar{\mfk})$ equals $\Int(c^{-1})\circ p_{\delta}$ (\ref{eq:psi}), we see that
$ j_b^{I_0}(g^{-1}\cdot{}^{\tau}g)=\Xi_b(g^{-1}\cdot{}^{\tau}g)\cdot b_{\tau}=\Nm_{l(\tau)}b'$. 

From the injectivity of (\ref{eq:kappa_times_nu}), it follows that an element $[b'=c'^{-1}\delta'\sigma(c')]_{I_0}$ of $\mathfrak{D}_p^{(n)}(\gamma_0)'$ lies in the subset $\mathfrak{D}_p^{(n)}(\gamma_0)$ if (and only if) $[b']=[b]$ in $B(G)$, namely 
if $\beta\in\ker[H^1(\Qp,J_b^{I_0})\rightarrow H^1(\Qp,J_b^G)]$ when $[b']_{I_0}=j_{[b]}^{I_0}(\beta)$ (so that $[b']_{G}=j_{[b]}^{G}(\beta)$, as $j_b$ is functorial with respect to $I_0\rightarrow G$).  But, from the commutative diagram (\ref{eq:stable_sigma_conj_diagm}), we see that any class in $\ker[H^1(\Qp,I_0)\rightarrow H^1(\Qp,G)]$ already lies in $\mathfrak{D}_p^{(n)}(\gamma_0)'$, so the first statement follows. The second statement is a consequence of the first statement.
\end{proof}

%%%%%%%%%%%%%%%%%%%%

\subsubsection{Kottwitz invariant} \label{subsubsec:Kottwitz_invariant}

Here, we extend, to general groups, the definition of Kottwitz invariant which was previously used (at least in the context of Shimura varieties, as in \cite[$\S$2]{Kottwitz90}, \cite[$\S$5]{Kottwitz92}) under the assumption $G^{\der}=G^{\uc}$.
When $G^{\der}=G^{\uc}$, the notion of Kottwitz invariant is well-defined by a Kottwitz triple only. For general groups, however, this is not the case any longer, and we generalize the defintion such that it depends on some auxiliary choices as well as on stable Kottwitz triple; in fact, as we will see, it seems better to regard Kottwitz invariant as being defined on the set $\ker[H^1(\A_f,I_0)\rightarrow H^1(\A_f,G)]$.
Recall our convention for the notation $X_{\ast}(A)$, $X^{\ast}(A)$, and $A^D$: for a locally compact abelian group $A$, they denote the (co)character groups $\Hom(\C^{\times},A)$, $\Hom(A,\C^{\times})$, and the Pontryagin dual group $\Hom(A,S^1)$, respectively. Let $\Gamma:=\Gal(\Qb/\Q)$. 
%We use the canonical exact sequence of $\Gamma$-modules \begin{equation} \label{eq:pi_1-es} 1\rightarrow \pi_1(\tilde{I}_{\phi_1}) \rightarrow  \pi_1(I_{\phi_1})  \rightarrow  \pi_1(G)  \rightarrow 1. \end{equation}

Let $(\gamma_0;\gamma,\delta)$ be a \emph{stable} Kottwitz triple and put $I_0:=G_{\gamma_0}^{\mathrm{o}}$, $\tilde{I}_0:=\rho^{-1}(I_0)$ for the canonical homomorphism $\rho=\rho_G:G^{\uc}\rightarrow G$.
The exact sequence
\[1\rightarrow Z(\hat{G})\rightarrow Z(\hat{I}_0)\rightarrow Z(\hat{\tilde{I}}_0) \rightarrow 1\]
induces a homomorphism \cite[Cor.2.3]{Kottwitz84a} 
\begin{equation} \label{eq:boundary_map_for_center_of_dual} 
\partial: \pi_0(Z(\hat{\tilde{I}}_0)^{\Gamma})\rightarrow H^1(\Q,Z(\hat{G})). 
\end{equation}
Let $\ker^1(\Q,Z(\hat{G}))$ denote the kernel of the localization map $H^1(\Q,Z(\hat{G}))\rightarrow \prod_v H^1(\Q_v,Z(\hat{G}))$. Then we define $\mathfrak{K}(I_0/\Q)$ by
\[\mathfrak{K}(I_0/\Q) :=\{ a\in \pi_0(Z(\hat{\tilde{I}})^{\Gamma}) \ |\ \partial(a)\in \ker^1(\Q,Z(\hat{G}))\}.\]
This is known to be a finite group. Since $\gamma_0$ is elliptic, there is also an identification
\[\mathfrak{K}(I_0/\Q)=\biggl( \bigcap_v Z(\hat{I}_0)^{\Gamma_v}Z(\hat{G}) \biggl)/Z(\hat{G}).\]

The Kottwitz invariant $\alpha(\gamma_0;\gamma,\delta;(g_v)_v)$ attached to the stable Kottwitz triple $(\gamma_0;\gamma,\delta)$ and certain auxiliary elements $((g_l)_{l\neq p},g_p)\in G(\bar{\A}_f^p)\times G(\mfk)$ is then a character of $\mathfrak{K}(I_0/\Q)$. 
It is defined as a product over all places of $\Q$ of the restrictions to $\mathfrak{K}(I_0/\Q)$ of some
characters $\tilde{\alpha}_v(\gamma_0;\gamma,\delta;(g_v)_{v})\in  X^{\ast}(Z(\hat{I}_0)^{\Gamma_v}Z(\hat{G}))$:
\[\alpha(\gamma_0;\gamma,\delta;(g_v)_v)=\prod_v \tilde{\alpha}_v|_{\bigcap_v Z(\hat{I}_0)^{\Gamma_v}Z(\hat{G})}\]
(it factors through the quotient $( \bigcap_v Z(\hat{I}_0)^{\Gamma_v}Z(\hat{G}) )/Z(\hat{G})$).
For each place $v$, $\tilde{\alpha}_v$ is itself defined as the unique extension $\tilde{\alpha}_l(\gamma_0;\gamma_l;g_l)$ ($l\neq p,\infty$), $\tilde{\alpha}_p(\gamma_0;\delta;g_p)$, $\tilde{\alpha}_{\infty}(\gamma_0)$ of another character on $Z(\hat{I}_0)^{\Gamma_v}$: 
\begin{equation} \label{eq:alpha(gamma_0;gamma,delta)}
\alpha_l(\gamma_0;\gamma_l;g_l)\ (l\neq p,\infty),\quad \alpha_p(\gamma_0;\delta;g_p),\quad \alpha_{\infty}(\gamma_0)
\end{equation} 
The restriction of $\tilde{\alpha}_v$ to $Z(\hat{G})$ is defined to be%%
\footnote{The sign here is opposite to the sign used by Kottwitz in \cite{Kottwitz90}, \cite{Kottwitz92}.}
\begin{equation} \label{eq:restriction_of_alpha_to_Z(hatG)}
\tilde{\alpha}_v|_{Z(\hat{G})}=
\begin{cases}
\ -\mu_{\natural} &\text{ if }\quad  v=\infty \\
\quad \mu_{\natural} &\text{ if }\quad  v=p \\
\text{ trivial } &\text{ if }\quad  v\neq p,\infty.
\end{cases}
\end{equation}
(Such extension of $\alpha_v\in X^{\ast}(Z(\hat{I}_0)^{\Gamma_v}Z(\hat{G}))$ will be possible, since $\alpha_v$'s have the same restrictions to $Z(\hat{I}_0)^{\Gamma_v}\cap Z(\hat{G})$ as these characters (\ref{eq:restriction_of_alpha_to_Z(hatG)}) on $Z(\hat{G})$, as we will see now).
Finally, the characters $\alpha_v\in X^{\ast}(Z(\hat{I}_0)^{\Gamma_v})$ (\ref{eq:alpha(gamma_0;gamma,delta)}) are defined as follows: 

For each finite place $l\neq p$, we choose $g_l\in G(\Qlb)$ such that 
\begin{equation} \label{eq:stable_g_l}
g_l\gamma_0g_l^{-1}=\gamma_l,\ \text{ and }\ g_l^{-1}\tau(g_l)\in I_0(\Qlb)\quad \forall\tau\in\Gamma_l,
\end{equation}
so that $\tau\mapsto g_l^{-1}\tau(g_l)$ is a cocycle in $Z^1(\Ql,I_0)$. In view of the canonical isomorphism $H^1(\Ql,I_0)\isom\pi_0(Z(\hat{I}_0)^{\Gamma_l})^D$ \cite[Thm.1.2]{Kottwitz86}, the corresponding cohomology class gives a character $\alpha_l(\gamma_0;\gamma_l;g_l)$ of $Z(\hat{I}_0)^{\Gamma_l}$. We note that this class $\alpha_l(\gamma_0;\gamma_l;g_l)$ lives in $\ker[H^1(\Ql,I_0)\rightarrow H^1(\Ql,G)]$ and in general depends on the choice of $g_l$ as well as on the pair $(\gamma_0,\gamma_l)$, while its image in $H^1(\Ql,G_{\gamma_0})$ does not. More precisely, it depends only on the left $I_0(\Qlb)$-coset $g_lI_0(\Qlb)$.  In \cite[p.41]{Labesse99}, the set of the left $I_0(\Qlb)$-cosets of such elements $g_l$ (i.e. $g_l^{-1}\tau(g_l)\in I_0(\Qlb),\forall\tau\in\Gamma_l$) is denoted by $\mathbb{H}^0(\Ql,I_0\backslash G)$ and the set $\ker[H^1(\Ql,I_0)\rightarrow H^1(\Ql,G)]$ by $\mathfrak{D}(I_0,G;\Ql)$,%%
\footnote{By definition \cite[3.1.1]{Borovoi98}, $\mathbb{H}^0(\Ql,I_0\backslash G)$ equals the hypercohomology set $\mathbb{H}^0(\Q,I_0\rightarrow G)$ of the complex of $\Q$-groups $I_0\rightarrow G$ ($I_0$ being located in degree -1, as always).}
and $\alpha_l(\gamma_0;\gamma_l;g_l)$ is the image of $g_l I_0(\Qlb)$ under the natural surjective map 
\[\mathbb{H}^0(\Ql,I_0\backslash G)\rightarrow \ker[H^1(\Ql,I_0)\rightarrow H^1(\Ql,G)]\ :\ g_l I_0(\Qlb)\mapsto [g_l^{-1}\tau(g_l)].\]
Namely, the invariant $\alpha_l(\gamma_0;-;-)$ can be regarded as a such map. Also for $\eta\in \mathfrak{D}(I_0,G;\Ql):=\ker[H^1(\Ql,I_0)\rightarrow H^1(\Ql,G)]$, the fiber of $\eta$ is identified with the set $G(\Ql)/I_{\eta}(\Ql)$, where $I_{\eta}=xI_0x^{-1}$ (inner form of $I_0$) for any $x\in \mathbb{H}^0(\Ql,I_0\backslash G)$ with image $\eta$.

On the other hand, the set $\ker[H^1(\Ql,G_{\gamma_0})\rightarrow H^1(\Ql,G)]$ classifies the $G(\Ql)$-conjugacy classes in the $G(\Qlb$)-conjugacy class of $\gamma_0\in G(\Ql)$, and its subset
\begin{equation} \label{eq:C_l(gamma_0)}
\mathfrak{C}_l(\gamma_0):= \im\left(\ker[H^1(\Ql,I_0)\rightarrow H^1(\Ql,G)]\rightarrow \ker[H^1(\Ql,G_{\gamma_0})\rightarrow H^1(\Ql,G)]\right)
\end{equation}
classifies the $G(\Ql)$-conjugacy classes in the \emph{stable} conjugacy class of $\gamma_0\in G(\Ql)$.
The invariant $\alpha_l(\gamma_0;\gamma_l;g_l)$ is then a lifting in $\ker[H^1(\Ql,I_0)\rightarrow H^1(\Ql,G)]$  of the $G(\Ql)$-conjugacy class of $\gamma_l$.

For $v=p$, when we choose $g_p\in G(\mfk)$ satisfying condition (iii$'$) of Def. \ref{defn:stable_Kottwitz_triple} (with $g_p=c$),
the element $b:=g_p^{-1}\delta\sigma(g_p)\in I_0(\mfk)$ and its $\sigma$-conjugacy class $[b]_{I_0}$ in $B(I_0)$ in general depends on the choice of $g_p$ as well as on the pair $(\gamma_0,\delta)$ while its $\sigma$-conjugacy class in $B(G_{\gamma_0})$ does not. We define $\alpha_p(\gamma_0;\delta;g_p)$ to be the image of this class under the canonical map $\kappa_{I_0}:B(I_0)\rightarrow X^{\ast}(Z(\hat{I}_0)^{\Gamma(p)})$ (or, its restriction $B(I_0)_{basic}\isom X^{\ast}(Z(\hat{I}_0)^{\Gamma(p)})$ \cite[Prop.6.2]{Kottwitz85}). 

We have seen (Prop. \ref{prop:B(gamma_0)=D(I_0,G;Qp)}) that with choice of a reference element $[b]_{I_0}$ in $\mathfrak{D}_p^{(n)}(\gamma_0)\subset B(I_0)_{basic}$, the map $\kappa_{I_0}-\kappa_{I_0}([b])$ provides an identification
\[\mathfrak{D}_p^{(n)}(\gamma_0)=\ker[H^1(\Qp,I_0)\rightarrow H^1(\Qp,G)] \]
sending $[b]_{I_0}$ to the distinguished element of the target (whose inverse is given by $j_{[b]}$ (\ref{eq:B(G(L_n);gamma_0)})) and that the set $\mathfrak{C}_p^{(n)}(\gamma_0)$ of $\sigma$-conjugacy classes of elements in $G(L_n)$ satisfying conditions (iii$'$) and $\ast(\delta)$ is identified with
\begin{equation} \label{eq:C_p(gamma_0)}
\im\left(\ker[H^1(\Qp,G_{\delta\theta}^{\mathrm{o}})\rightarrow H^1(\Qp,G)]\rightarrow \ker[H^1(\Qp,G_{\delta\theta})\rightarrow H^1(\Qp,R)]\right).
\end{equation}
Further, we can regard the invariant $\alpha_p$ as being valued in $\ker[H^1(\Qp,G_{\delta\theta}^{\mathrm{o}})\rightarrow H^1(\Qp,G)]$; or, as in the case $l\neq p$, one can also view $\alpha_p$ as a map 
\[\mathbb{H}^0(\Qp,I_0\backslash G)\rightarrow \mathfrak{D}(I_0,G;\Qp). \]

For $v=\infty$, we define $\alpha_{\infty}(\gamma_0)$ to be the image of $\mu_h$ in $\pi_1(I_0)_{\Gamma(\infty)}\cong X^{\ast}(Z(\hat{I}_0)^{\Gamma(\infty)})$ for some $h\in X$ factoring through a maximal torus $T$ of $G$ containing $\gamma_0$ and elliptic over $\R$. This character is independent of the choice of $(T,h)$ \cite[Lem.5.1]{Kottwitz90}.

Since the image of $\alpha_p$ in $X^{\ast}(Z(\hat{G})^{\Gamma(p)})$ equals that of the $\sigma$-conjugacy class of $b$ in $B(G_{\Qp})$ under $\kappa_{G_{\Qp}}$, it is equal to $\mu^{\natural}$ by condition $\ast(\delta)$, which makes the construction of $\tilde{\alpha}_v$ above possible. It also follows that the product $\prod_v \tilde{\alpha}_v$ is trivial on $Z(\hat{G})$.
Clearly, the Kottwitz invariant $\alpha(\gamma_0;\gamma,\delta;(g_v)_v)$ is determined by the adelic class $(\alpha_v(\gamma_0;\gamma,\delta;g_v))_v\in \ker[H^1(\A_f,I_0)\rightarrow H^1(\A_f,G)]$ 
(once we fix a reference $\sigma$-conjugacy class $[b]=[g_p^{-1}\delta\sigma(g_p)]\in B(I_0)$). Here and thereafter, to have uniform notation for all $v$'s, we write $\alpha_v(\gamma_0;\gamma,\delta;g_v)$ and $\tilde{\alpha}_v(\gamma_0;\gamma,\delta;g_v)$, although for specific $v$, these depend only on either $\gamma_l\ (l\neq p)$ or $\delta$.

Later, we will often use the expression ``\textit{stable Kottwitz triple with trivial Kottwitz invariant}". By this, we will mean that for given stable Kottwitz triple $(\gamma_0;(\gamma_l)_{l\neq p},\delta)$ there exist elements $(g_v)_v\in G(\bar{\A}_f^p)\times G(\mfk)$ satisfying conditions (\ref{eq:stable_g_l}), (\ref{eq:stable_g_l}) such that the associated Kottwitz invariant $\alpha(\gamma_0;\gamma,\delta;(g_v)_v)$ vanishes.

%%%%%%%%%%%%%%%%%%%%
%%%%%%%%%%%%%%%%%%%%
\subsection{Admissible pairs and associated Kottwitz triples}

Recall that every parahoric subgroup $\mbfK_p\subset G(\Qp)$ is defined by a $\sigma$-stable facet $\mbfa$ in the building $\mcB(G,\mfk)$ which also gives rise to a parahoric subgroup $\mbfKt_p\subset G(\mfk)$ with $\mbfK_p=\mbfKt_p\cap G(\Qp)$. To formulate the definition of admissible pair for general parahoric level, we need to consider the coset space $G(\mfk)/\mbfKt_p$ endowed with the obvious action of the semi-direct product $G(\mfk)\rtimes\langle\sigma\rangle$, instead of the Bruhat-Tits building $\mcB(G,\mfk)$ which was used in the hyperspecial level case; note that unless $\mbfK_p$ is hyperspecial (in which case $\mbfKt_p$ equals the whole stabilizer $\mathrm{Stab}(\mbfo)\subset G(\mfk)$), $G(\mfk)/\mbfKt_p$ is not a subset of $\mcB(G,\mfk)$ in any natural manner.

%%%%%%%%%%%%%%%%%%%%
\begin{defn} \cite[p.189]{LR87} \label{defn:admissible_pair}
A pair $(\phi,\epsilon)$ is called \textit{admissible (of level $n=[\kappa(\wp):\Fp]m$)} if 
\begin{itemize}
\item[(1)] $\phi:\fQ\rightarrow \fG_G$ is admissible in the sense of Def. \ref{defn:admissible_morphism}.
\item[(2)] $\epsilon\in I_{\phi}(\Q)(=\{g\in G(\Qb)\ |\ \mathrm{Int}(g)\circ\phi=\phi\})$.
\item[(3)] There are $\gamma=(\gamma(v))\in G(\A_f^p)$, $y\in X^p(\phi)$, and $x$ in $G(\mfk)/\mbfKt_p$ 
such that
\[\epsilon y=y\gamma\quad \text{ and }\quad \epsilon x=\Phi^mx.\]
\end{itemize}
\end{defn}

Comments on condition (3) are in order: the equation $\epsilon x=\Phi^mx$ requires actions of $I_{\phi}(\Q)$ and $\Phi$ on $G(\mfk)/\mbfKt_p$. Such actions are specified by choice of $u\in G(\Qpb)$ such that $\xi_p':=\Int(u)\circ\xi_p$ is unramified ($\xi_p:=\phi(p)\circ\zeta_p:\fG_p\rightarrow \fG_G(p)$).
First, as an automorphism of $G(\mfk)/\mbfKt_p$, $\Phi$ acts by left multiplication by $\Phi_b:=(b\sigma)^{[\kappa(\wp):\Fp]}$ when $\xi_p'(s_{\widetilde{\sigma}})=b\widetilde{\sigma}$ for a (fixed) lifting $\widetilde{\sigma}\in\Gal(\Qpb/\Qp)$ of the Frobenius automorphism $\sigma$ of $\Qpnr$ (recall that $\tau\mapsto s_{\tau}$ is the previously chosen section to the projection $\fG_p\rightarrow\Gal(\Qpb/\Qp)$). 
Moreover, the choice of $u$ also fixes an action of $I_{\phi}(\Q)$ on $G(\mfk)/\mbfKt_p$ via the homomorphism 
\[\Int(u):I_{\phi}(\Qp)=\mathrm{Aut}(\phi(p))\hookrightarrow \mathrm{Aut}(\phi(p)\circ\zeta_p)\isom \mathrm{Aut}(\xi_p')\subset G(\Qpnr).\]
(one has $\mathrm{Aut}(\xi_p')\subset G(\Qpnr)$, since $\xi_p'(s_{\tau})=1\rtimes\tau$ for every $\tau\in\Gal(\Qpb/\Qpnr)$, so that $\epsilon'\tau=\epsilon'\cdot \xi_p'(s_{\tau})=\xi_p'(s_{\tau})\cdot \epsilon'=\tau(\epsilon')\tau$ for all $\tau\in\Gal(\Qpb/\Qpnr)$).
Then, by existence of $x$ in $G(\mfk)/\mbfKt_p$ satisfying $\epsilon x=\Phi^mx$, we mean that for some (equiv. any) $u\in G(\Qpb)$ such that $\Int(u)\circ\xi_p$ is unramified, there exists $x$ in the coset space $G(\mfk)/\mbfKt_p$ satisfying that 
\begin{equation} \label{defn:admissible_pair_(3)a}
\epsilon'x=(b\sigma)^{[\kappa(\wp):\Fp]m}x,
\end{equation}
where $\epsilon':=u\epsilon u^{-1}\in G(\Qpnr)$. 
Here, to define $F=b\sigma$, instead of $\xi_p'$, we could also have used a $\Qpnr/\Qp$-Galois gerb morphism $\theta^{\nr}$ whose inflation to $\Qpb$ is $\xi_p'$: it does not change $b$ (Lemma \ref{lem:unramified_morphism}).

%%%%%%%%%%%%%%%%%%%%
\begin{rem} \label{rem:admissible_pair}
(1) Suppose that the anisotropic kernel of $Z(G)$ (maximal anisotropic subtorus of $Z(G)$) remains anisotropic over $\R$ (then $Z(G)(\Q)$ is discrete in $Z(G)(\A_f)$); take $\mbfK^p$ to be small enough so that $Z(\Q)\cap \mbfK_p\mbfK^p=\{1\}$ and also condition (1.3.7) of \cite{Kottwitz84b} holds.

Then, conjecture (\ref{conj:Langlands-Rapoport_conjecture_ver1}) leads to the following description of the finite sets $\sS_{\mbfK}(G,X)(\F_{q^m})=[\sS_{\mbfK}(G,X)(\Fpb)]^{\Phi^m=\mathrm{id}}$ for each finite extension $\F_{q^m}$ of $\F_q=\kappa(\wp)$, where $\Phi:=F^r$ is the $r=[\kappa(\wp):\Fp]$-th relative Frobenius (cf. \cite[1.3-1.5]{Kottwitz84b}, \cite[$\S$5]{Milne92}): There exists a bijection (forming a compatible family for varying $\mbfK^p$)
\begin{equation} \label{eqn:LRconj-ver2} 
\sS_{\mbfK}(G,X)(\F_{q^m})\isom \bigsqcup_{(\phi,\epsilon)} S(\phi,\epsilon), \end{equation} where $S(\phi,\epsilon)=I_{\phi,\epsilon}(\Q)\backslash X_p(\phi,\epsilon)\times X^p(\phi,\epsilon)_{\mbfK^p}/\mbfK^p$ with 
\begin{align} \label{eqn:LRconj-ver2_constituents} 
I_{\phi,\epsilon}&:= \text{ centralizer of }\epsilon\text{ in }I_{\phi}(\Q), \\  
X^p(\phi,\epsilon)_{\mbfK^p}&:= \{ x^p\in X^p(\phi) \ |\ \epsilon x^p=x^p\mod \mbfK^p\}, \nonumber \\  
X_p(\phi,\epsilon)&:= \{ x_p\in X_p(\phi) \ |\ \epsilon x_p=\Phi^mx_p\} \nonumber \\
&\simeq \{ x_p\in X(\{\mu_X\},b)_{\mbfK_p} \ |\ \epsilon' x_p=(b\sigma)^{rm}x_p\}. \nonumber
\end{align}
In (\ref{eqn:LRconj-ver2}), $(\phi,\epsilon)$ runs through a set of representatives for the conjugacy classes of pairs consisting of an admissible morphism $\phi$ and an element $\epsilon\in I_{\phi}(\Q)(\subset G(\Qb))$. 
We emphasize that the second identification of $X_p(\phi,\epsilon)$ requires choice of $g_0\in G(\Qpb)$ with $g_0\mbfK_p(\Qpnr)\in X_p(\phi)$ (so that especially, $\Int(g_0^{-1})\circ\phi(p)\circ\zeta_p$ is unramified), which also specifies an action of $(b\sigma)^r$ on $X(\{\mu_X\},b)_{\mbfK_p}$ (for the Frobenius $\Phi=b\sigma$ defined in (\ref{eq:Frob_Phi_2})) as well as an action of $\epsilon\in I_{\phi}(\Q)$ via $\epsilon':=\Int g_0^{-1}(\epsilon)$ (\ref{eq:action_of_I_{phi}_on_AffDL}),  cf. (\ref{defn:admissible_pair_(3)a}).

(2) According to this conjecture (\ref{eqn:LRconj-ver2}), an admissible pair $(\phi,\epsilon)$ will contribute to the set $\sS_{\mbfK}(G,X)(\F_{q^m})$ if and only if in condition (3) one can find $y\in X^p(\phi)(\subset G(\bar{\A}_f^p))$ and a solution $x\in G(\mfk)/\mbfKt_p$ to the equation $\epsilon x=\Phi^mx$ (i.e. to (\ref{defn:admissible_pair_(3)a})), further satisfying that
\begin{equation} \label{eqn:effectivity_admssible_pair_condition}
y^{-1}\epsilon y\in \mbfK^p,\ \text{ and }\ x\in X(\{\mu_X\},b)_{\mbfK_p}
\end{equation}

Regarding this, we emphasize that in condition (3) of the definition of an admissible pair, $y\in X^p(\phi)$ and $x\in G(\mfk)/\mbfKt_p$ are \emph{not} demanded to satisfy this (seemingly natural) condition (\ref{eqn:effectivity_admssible_pair_condition}).%%
\footnote{So, Definition 5.8 of \cite{Milne92} is not correct, thus some proofs based on it in that work, especially the proof of Cor. 7.10, are incomplete; the first complete proof of (a substantial generalization of) that corollary is given in this work.}.
Indeed, Langlands and Rapoport made this point explicit: \textit{``Es wird \"ubrigens nicht verlangt, dass $x$ in $X_p$ liegt.''} \cite[p. 189, line +18]{LR87}. We remark that the main function of this notion of admissible pair is to attach a Kottwitz triple, but the usefulness of this definition comes from its flexibility that $x$ in (\ref{defn:admissible_pair_(3)a}) is \emph{not} required to be found in $X(\{\mu_X\},b)_{\mbfK_p}$.

(3) We will call an admissible pair $(\phi,\epsilon)$ $\mbfK^p$-\textit{effective}, respectively $\mbfK_p$-\textit{effective}  if there exist $y\in X^p(\phi)$, respectively $x\in G(\mfk)/\mbfKt_p$ with $\epsilon x=\Phi^mx$, satisfying the conditions (\ref{eqn:effectivity_admssible_pair_condition}). Note that if  $(\phi,\epsilon)$ is $\mbfK^p$-effective and $\epsilon\in G(\A_f^p)$, $\epsilon$ itself lies in a compact open subgroup of $G(\A_f^p)$ (cf. proof of Lemma \ref{lem:invariance_of_(ast(gamma_0))_under_transfer_of_maximal_tori}, (2)).
\end{rem}

Two admissible pairs $(\phi,\epsilon)$, $(\phi',\epsilon')$ are said to be \textit{equivalent} (or \textit{conjugate}) if there exists $g\in G(\Qb)$ such that $(\phi',\epsilon')=g(\phi,\epsilon)g^{-1}$.

%%%%%%%%%%%%%%%%%%%%
\subsubsection{}  \label{subsubsec:K-triple_attached_to_adm.pair}
Now, we explain how with any admissible pair $(\phi,\epsilon)$, one can associate a Kottwitz triple $(\gamma_0;\gamma:=(\gamma_l)_{l\neq p},\delta)$, cf. \cite{LR87}, p189. 

First, we set $\gamma_l:=\gamma(l)$, i.e. $\gamma=y^{-1}\epsilon y\in G(\A_f^p)$ for some $y\in X^p(\phi)$. When we replace $y$ by $yh, h\in G(\A_f^p)$, $\gamma$ goes over to $h^{-1}\gamma h$, thus the $G(\A_f^p)$-conjugacy class of $\gamma$ is well-determined by the pair $(\phi,\epsilon)$. Also, one easily sees that if $(\phi',\epsilon')=g(\phi,\epsilon)g^{-1}\ (g\in G(\Qb))$, the corresponding $G(\A_f^p)$-conjugacy classes of $(\gamma_l')_{l\neq p}$, $(\gamma_l)_{l\neq p}$ are the same, since $y\mapsto gy$ gives a bijection $X^p(\phi)\isom X^p(\Int g\circ\phi)$.

Next, to find $\gamma_0$, we use Lemma \ref{lem:LR-Lemma5.23} below, which is a generalization of Lemma 5.23 of \cite{LR87}. If an admissible pair $(\phi,\epsilon)$ satisfies that $\phi=i\circ\psi_{T,\mu_h}$ and $\epsilon\in T(\Q)$ for a special Shimura datum $(T,h)$, then we will say that it is \textit{nested} %%
\footnote{This is our translation of the German word \textit{eingeschachtelt} used by Langlands-Rapoport \cite[p.190, line19]{LR87}.}
in $(T,h)$ and that the (admissible) pair $(i\circ\psi_{T,\mu_h},\epsilon)$ is a \emph{special} (admissible) pair. Hence, this lemma says that every admissible pair is conjugate to an admissible pair that is nested (in particular, well-located) in a special Shimura sub-datum of $(G,X)$.
Then, we define $\gamma_0$ to be any (semi-simple) rational element of $G$ that is $G(\Qb)$-conjugate to $\epsilon$ and also lies in a maximal $\Q$-torus $T'$ of $G$ that is elliptic at $\R$, which we now know exists; note that any such $\gamma_0$ given by Lemma \ref{lem:LR-Lemma5.23} is always semi-simple.
%since $T'_{\R}$ is elliptic ($\Inn(h(i))$ defines a Cartan involution of $G^{\ad}_{\R}$ and so of $T'_{\R}/Z_{\R}$ as well). 
By definition, this rational element is well-defined up to $G(\Qb)$-conjugacy.

Finally, to define $\delta$, we choose $u\in G(\Qpb)$ such that $\Int u\circ \xi_p\ (\xi_p:=\phi(p)\circ\zeta_p)$ is unramified. When $\Int u\circ \xi_p$ is the inflation of a morphism $\theta^{\nr}:\fD\rightarrow \fG_{G_{\Qp}}^{\nr}$ of Galois $\Qpnr/\Qp$-gerbs, we define $b\in G(\Qpnr)$ by $\theta^{\nr}(s_{\sigma})=b\sigma$, and put $\epsilon'=u\epsilon u^{-1}\in G(\Qpnr)$.

%%%%%%%%%%%%%%%%%%%%
\begin{lem} \label{lem:Kottwitz84-a1.4.9_b3.3}
(1) Let $\Psi=b\sigma^n\in G(\mfk)\rtimes\langle\sigma\rangle$ such that $n\neq0$. Then, $\Psi$ is conjugate under $G(\mfk)$ to $\sigma^n$ if and only if $\Psi$ fixes some point in $G(\mfk)/\mbfKt_p$.

(2) Let $H$ be a quasi-split group over a $p$-adic field, either a local field $F$ or the completion $L$ of its maximal unramified extension $F^{\nr}$ in $\overline{F}$. Then, the map $v_{H_L}:H(\mfk)\rightarrow \Hom(X_{\ast}(Z(\widehat{H}))^I,\Z)$ (\autoref{subsubsec:Kottwitz_hom}) vanishes on any special maximal bounded subgroup of $H(\mfk)$ (not just on special parahoric subgroups).
\end{lem}

\begin{proof}
(1) This is Lemma 1.4.9 of \cite{Kottwitz84b} when $x^{\mathrm{o}}$ is a hyperspecial point, and its proof continues to work in our setup: note that the parahoric group scheme over $\Zp^{\nr}$ attached to $x^{\mathrm{o}}$ has connected special fiber, so the result of Greenberg \cite[Prop. 3]{Greenberg63} still applies.

(2) When $H$ is unramified and $F$ is a local field, this is (stated and) proved in Lemma 3.3 of \cite{Kottwitz84a} (where $\lambda_H$ is the same as the restriction of $w_{H_L}$ to $H(F)$). But one easily sees that for general quasi-split $H$, the same argument works with the map $v_{H_L}$ and any special maximal bounded subgroup (instead of special parahoric subgroup). More explicitly, as the relation (3.3.4) in \textit{loc. cit.} continues to hold, one just needs to note (in the last step of the original argument) that for a maximal split torus $S$ of $G$ whose apartment contains a given special point $x^{\mathrm{o}}\in\mcB(G,L)$ and $T=Z_G(S)$, the subgroup $T(\mfk)_0:=T(\mfk)\cap \mathrm{Fix} (x^{\mathrm{o}})$ equals $\ker(v_{T})$ \cite{HainesRapoport08}; it is the maximal bounded subgroup of $T(\mfk)$.  
\end{proof}

%%%%%%%%%%%%%%%%%%%%
\begin{lem} \label{lem:delta_from_b&gamma_0}
There exists $c\in G(\mfk)$ such that $\delta:=cb\sigma(c^{-1})$ lies in $G(L_n)$ and that $\Nm_n\delta$ and $\gamma_0$ are conjugate under $G(\Qpb)$. 
\end{lem}

\begin{proof} 
By definition (cf. \ref{defn:admissible_pair_(3)a}) and Lemma \ref{lem:Kottwitz84-a1.4.9_b3.3}, there exists $c\in G(\mfk)$ such that 
\[c(\epsilon'^{-1}F^n)c^{-1}=\sigma^n.\] 
If we define $\delta\in G(\mfk)$ by 
\[\delta\sigma:=cFc^{-1}=c(b\sigma)c^{-1},\] 
i.e. $\delta=cb\sigma(c^{-1})$, then it follows (as shown on p. 291 of \cite{Kottwitz84b}) that $\delta\in G(L_n)$ and $\sigma^n=c(\epsilon'^{-1}F^n) c^{-1}=c\epsilon'^{-1}c^{-1}(\delta\sigma)^n=(c\epsilon'^{-1}c^{-1}\Nm_n\delta)\sigma^n$, i.e. $\Nm_n\delta=c\epsilon'c^{-1}$, so that 
\begin{equation} \label{eqn:stable_conjugacy_reln_betwn_Nmdelta_and_gamma_0}
\Nm_n\delta=c\epsilon'c^{-1}=(cu)\epsilon (cu)^{-1}=(cug)\gamma_0(cug)^{-1},
\end{equation}
where $\epsilon= g\gamma_0 g^{-1}$ for $g\in G(\Qb)$. This implies that $\Nm_n\delta$ and $\gamma_0$ are conjugate under $G(\Qpb)$.
\end{proof}

Choose $\delta\in G(L_n)$ as in Lemma \ref{lem:delta_from_b&gamma_0}: its $\sigma$-conjugacy class in $G(L_n)$ is uniquely determined by the admissible pair $(\phi,\epsilon)$. 
We will say that the admissible pair $(\phi,\epsilon)$ \textit{gives rise to} the triple $(\gamma_0;(\gamma_l)_{l\neq p},\delta)$ (or, the triple $(\gamma_0;(\gamma_l)_{l\neq p},\delta)$ \textit{is attached to} the pair $(\phi,\epsilon)$); then $\epsilon$ and $\gamma_0$ are $G(\Qb)$-conjugate, among others. Also note that by definition (cf. (\autoref{subsubsec:cls})), the $\sigma$-conjugacy class $\mathrm{cls}(\phi)\in B(G_{\Qp})$ attached to $\phi(p)$ (the pull-back of $\phi$ to $\Gal(\Qpb/\Qp)$) is $[\delta]\in B(G_{\Qp})$.

%%%%%%%%%%%%%%%%%%%%
\begin{rem} \label{rem:two_different_b's}
For $\gamma_0\in G(\Qp)$, let $\mathfrak{C}_p^{(n)}(\gamma_0)'$ be the set of $\sigma$-conjugacy classes in $G(L_n)$ of elements $\delta\in G(L_n)$ for which there exists $c'\in G(\mfk)$ with $\Nm_n\delta=c'\gamma_0c'^{-1}$ (i.e. it is the same set as $\mathfrak{C}_p^{(n)}(\gamma_0)$ of Def. \ref{defn:D_p^{(n)}(gamma_0)} except that we do not demand the condition $\ast(\delta)$); then the map
\[ \mathfrak{C}_p^{(n)}(\gamma_0)' \hookrightarrow B(G_{\gamma_0}) \quad ;\quad \delta\mapsto [b']=[c'^{-1}\delta\sigma(c')]\]
is well-defined and injective (cf. \autoref{subsubsec:w-stable_sigma-conjugacy}).

Suppose that a pair $(\gamma_0,\delta)\in G(\Q)\times G(L_n)$ arises from an admissible pair $(\phi,\epsilon)$, that is, that $\gamma_0=\Int(g)(\epsilon)$ for some $g\in G(\Qb)$, $\delta=cb\sigma(c^{-1})$ for some $c\in G(\mfk)$ and $b\sigma=\theta^{\nr}(s_{\sigma})$, where $\theta^{\nr}$ is a $\Qpnr/\Qp$-Galois gerb morphism whose inflation to $\Qpb$ equals $\Int(u)\circ\xi_p$ for some $u\in G(\Qpb)$.
Then, in general, $b':=c'^{-1}\delta\sigma(c')$ defined by $c'\in G(\mfk)$ with $\Nm_n\delta=c'\gamma_0 c'^{-1}$ is different from $b$ defined by $\theta^{\nr}(s_{\sigma})=b\sigma$ (which thus depends on the choice of $u$). 

But, if $\epsilon\in G(\Q)$ and $\phi$ factors through $\fG_{I_0}\subset\fG_G$ (for example, if $(\phi,\epsilon)$ is well-located in a maximal $\Q$-torus $T$ or if $G_{\epsilon}$ is already connected), we can take $\gamma_0=\epsilon$ and also pick an element $u$ making $\Int(u)\circ\phi(p)\circ\zeta_p$ unramified (say, inflation of $\theta^{\nr}$) from $I_0(\Qpb)$.
With such choice of $\gamma_0$ and $u$, we have the equality 
\[[b]=[b']\text{ in }B(G_{\gamma_0})\] 
(these $\sigma$-conjugacy classes $[b]$, $[b']$ in $B(G_{\gamma_0})$ do not depend on the choices involved, i.e. choices of any of $u\in I_0(\Qpb)$, $c, c'\in G(\mfk)$). 
In fact, we can choose $c'$ such that $b':=c'^{-1}\delta\sigma(c')$ equals $b$ (as elements in $G_{\gamma_0}(\mfk)$). Indeed, by construction (cf. proof of Lemma \ref{lem:delta_from_b&gamma_0}), we have $\delta=cb\sigma(c^{-1})$ for some $c\in G(\mfk)$ such that $c(\epsilon'^{-1}(b\sigma)^n)c^{-1}=\sigma^n$ ($\epsilon':=u\epsilon u^{-1}=\epsilon$), so that $\Nm_n\delta=c\epsilon'c^{-1}=c\gamma_0c^{-1}$ in (\ref{eqn:stable_conjugacy_reln_betwn_Nmdelta_and_gamma_0}). Then we can take $c$ for $c'$, so that $b'=c^{-1}\delta\sigma(c)=b$. 
In particular, there exists a representative in $I_0(\mfk)$ of $[b']\in B(G_{\gamma_0})$.
\end{rem}

It is immediate that equivalent admissible pairs give rise to equivalent Kottwitz triples. 
However, non-equivalent admissible pairs can also give equivalent Kottwitz triples, and there is a cohomological expression for the number of non-equivalent admissible pairs giving rise to a given Kottwitz triple (see Theorem \ref{thm:LR-Satz5.25}).

%%%%%%%%%%%%%%%%%%%%
\begin{prop} \label{prop:Kottwitz_triple}
Every special admissible pair $(\phi,\epsilon)$ of level $n=[\kappa(\wp):\Fp]m$ gives rise to a stable Kottwitz triple $(\gamma_0;(\gamma_l)_{l\neq p},\delta)$ of level $n$ (Def. \ref{defn:Kottwitz_triple}) for which there exist $(g_v)_v\in G(\bar{\A}_f^p)\times G(\mfk)$ satisfying (\ref{eq:stable_g_l}), (\ref{eq:stable_g_l}) such that the associated Kottwitz invariant $\alpha(\gamma_0;\gamma,\delta;(g_v)_v)$ vanishes.
\end{prop}

Recall that an admissible pair $(\phi,\epsilon)$ is special if there exists a special Shimura sub-datum $(T,h)$ such that $\phi=\psi_{T,\mu_h}$ and $\epsilon\in T(\Q)$, and that every admissible pair is conjugate to a special one (Lemma \ref{lem:LR-Lemma5.23}).

\begin{proof} 
First, condition (i) of \autoref{subsubsec:pre-Kottwitz_triple} (and (i$'$) of Def. \ref{defn:stable_Kottwitz_triple} as well) is clear. By Lemma \ref{lem:properties_of_psi_T,mu}, we have $(g_v)_{v\neq p,\infty}\in T(\bar{\A}_f^p)\cap X^p(\phi)$ which implies that we may take $\gamma_l$ to be $\gamma_0\in T(\Ql)$, thereby establishing condition (ii$'$): note that condition (ii) itself just follows from condition (2) of Def. \ref{defn:admissible_morphism}. 
Condition (iii) was proved in Lemma \ref{lem:delta_from_b&gamma_0}. For condition (iii$'$), we have seen in Remark \ref{rem:two_different_b's} that one can find $\delta\in G(L_n)$ and $c\in G(\mfk)$ satisfying that $c\gamma_0c^{-1}=\Nm_n\delta$ and $b:=c^{-1}\delta\sigma(c)\in T(\mfk)$: there exist $u\in T(\Qpb)$ and $c\in G(\mfk)$ such that $\Int(u)\circ\phi(p)\circ\zeta_p$ is the inflation of a $\Qpnr/\Qp$-Galois gerb morhism $\theta^{\nr}$ and $c\gamma_0^{-1}\theta^{\nr}(s_{\sigma})^nc^{-1}=\sigma^n$, and one defines $\delta:=cb\sigma(c^{-1})$ for $b\sigma:=\theta^{\nr}(s_{\sigma})$.  
That the condition $\ast(\delta)$ holds is a consequence of $\phi$ being admissible (specifically, condition (3) of Def. \ref{defn:admissible_morphism}) in view of \cite[Thm. A]{He15}.
Therefore, the triple $(\gamma_0;\gamma,\delta)$ just defined is a stable Kottwitz triple. 
Finally, it remains to verify vanishing of the Kottwitz invariant $\alpha(\gamma_0;(\gamma_l)_{l\neq p},\delta;(g_v)_v)$ for suitable elements $(g_v)_v\in G(\bar{\A}_f^p)\times G(\mfk)$; as observed before \cite[p.172]{Kottwitz90}, this will also establish condition (iv) (alternatively, we can take $I_{\phi,\epsilon}$ for $I$ in condition (iv), \cite[Lem6.10]{Milne92}). We take $g_v=1$ for $v=l\neq p$ and $g_p=c$ above. Then, we have $\alpha_l=1$ for $l\neq p$, and $\alpha_v$'s for $v=p,\infty$ are $\pm1$ times the restrictions of the character of $\hat{T}$ defined by $\mu_h$ with different signs (Lemma \ref{lem:unramified_conj_of_special_morphism}), from which the vanishing of $\alpha(\gamma_0;(\gamma_l)_{l\neq p},\delta;(g_v)_v)$ is obvious.
\end{proof}

%%%%%%%%%%%%%%%%%%%%%%%%%%%%%%%%%%%%%%%%
%%%%%%%%%%%%%%%%%%%%%%%%%%%%%%%%%%%%%%%%

\section{Every admissible morphism is conjugate to a special admissible morphism}

We make some comments on various assumptions called upon in this section.
First, there are two running assumptions: (1) $G_{\Qp}$ is quasi-split, and (2) $G_{\Qp}$ splits over a tamely ramified extension of $\Qp$ (more precisely, $G^{\uc}_{\Qp}$ is a product $\prod_i \Res_{F_i/\Qp}G_i$ of simple groups each of which is the restriction of scalars $\Res_{F_i/\Qp}G_i$ of an absolutely simple group $G_i$ over a field $F_i$ such that $G_i$ splits over a tamely ramified extension of $F_i$): condition (2) is directly invoked only in Lemma \ref{lem:LR-Lemma5.11} which however is used in all the main results of this section.
Secondly, the level subgroup can be arbitrary parahoric subgroup (except in Thm. \ref{thm:non-emptiness_of_NS}, (2)), and $G^{\der}$ is not assumed to be simply connected, unlike in the original work \cite{LR87}.
Finally, every statement involving the pseudo-motivic Galois gerb $\fP$ (instead of the quasi-motivic Galois gerb $\fQ$) will (for safety) assume that the Serre condition holds for the Shimura datum $(G,X)$ (i.e. the center $Z(G)$ splits over a CM field and the weight homomorphism $w_X$ is defined over $\Q$). 
Of course, in every statement we will make explicit the assumptions that we impose.

%%%%%%%%%%%%%%%%%%%%
%%%%%%%%%%%%%%%%%%%%
\subsection{Transfer of maximal tori and strategy of proof of Theorem \ref{thm:LR-Satz5.3}}

%%%%%%%%%%%%%%%%%%%%
\subsubsection{}
In the work of Langlands-Rapoport, a critical role is played by the notion of \textit{transfer} (or \textit{admissible embedding}) \textit{of maximal torus}. Recall \cite[$\S$9]{Kottwitz84b} that for a connected reductive group $G$ over a field $F$, if $\psi:H_{\overline{F}}\isom G_{\overline{F}}$ is an inner-twisting and $T$ is a maximal $F$-torus of $H$, an $F$-embedding $i:T\rightarrow G$ is called \textit{admissible} if it is of the form $\Int g\circ\psi|_T$ for some $g\in G(\overline{F})$ (equiv. of the form $\psi\circ \Int  h|_T$ for some $h\in H(\overline{F})$). Whether an $F$-embedding $i:T\rightarrow G$ is admissible or not depends only on the conjugacy class of the inner-twisting $\psi$. We will also say that $T$ \textit{transfers to} $G$ (with a conjugacy class of inner twistings $\psi:H_{\overline{F}}\isom G_{\overline{F}}$ understood) if there exists an admissible $F$-embedding $T\rightarrow G$ (with respect to the same conjugacy class of inner twistings). Usually, this notion is considered when $H$ is a quasi-split inner form of $G$, due to the well-known fact \cite[p.340]{PR94} that \textit{every maximal torus in a reductive group transfers to its quasi-split inner form}, but here we do not necessarily restrict ourselves to such cases. In fact, we will consider more often the identity inner twisting $\mathrm{Id}_G:G_{\overline{F}}=G_{\overline{F}}$, in which case if $\Int g:T_{\overline{F}}\hra G_{\overline{F}}\stackrel{=}{\rightarrow}G_{\overline{F}}$ is an admissible embedding, we will also let $\Int g$ denote both the $\Q$-isomorphism $T\rightarrow \Int g(T)$ of $\Q$-tori and the induced morphism of Galois gerbs $\fG_T\rightarrow \fG_{\Int g(T)}$; of course the latter is also the restriction of $\Int g$, regarded as an automorphism of the Galois gerb $\fG_G$, to $\fG_T$.

%%%%%%%%%%%%%%%%%%%%
\subsubsection{} \label{subsubsec:well-located_admissible_morphism}
We say that an admissible morphism $\phi:\fP\rightarrow\fG_G$ is \textit{well-located}%% 
\footnote{This is our translation of the German word \textit{g\"unstig gelegen} used by Langlands-Rapoport \cite[p.190, line8]{LR87}.}
if $\phi(\delta_n)\in G(\Q)$ for all sufficiently large $n$. Here, $\delta_n$ is the element in $P(K,m)(\Q)$ introduced in Lemma \ref{lem:Reimann97-B2.3}, (2) for any $\fP(K,m)$ which $\phi$ factors through: $\phi(\delta_n)$ does not depend on the choice of such $\fP(K,m)$ by the compatibility property of $\delta_n$ (\textit{loc. cit.}). More generally, for a $\Q$-subgroup $H$ of $G$, we will say that an admissible morphism $\phi:\fP\rightarrow\fG_G$ is \textit{well-located in $H$} if $\phi(\delta_n)\in H(\Q)$ for all sufficiently large $n$ and $\phi$ factors through $\fG_H \subset \fG_G$; note the second additional requirement. For example, every special admissible morphism $i\circ\psi_{T,\mu_h}$ (for a special Shimura sub-datum $(T,h)$) is always well-located (in $T$). Indeed, $\delta_n\in P(K,m)(\Q)$ (for sufficiently large $n$ divisible by $m$) and the restriction of $i\circ\psi_{T,\mu_h}$ to kernels is defined over $\Q$ \cite[p.143, second paragraph]{LR87}. 

Recall the $\Q$-group scheme $I_{\phi}=\underline{\mathrm{Isom}}(\phi,\phi)$ (\ref{eq:Isom(phi_1,phi_2)}).
When $\phi:\fP\rightarrow\fG_G$ is well-located, this group scheme is identified in a natural way with the inner twist via $\phi$ of the $\Q$-group scheme
\[I:=Z_G(\phi(\delta_n)),\] 
the centralizer $\Q$-group of $\phi(\delta_n)\in G(\Q)$ for any sufficiently large $n$ ($I$ is connected and does not depend on the choice of such $n$). More precisely, if $\phi(q_{\rho})=g_{\rho}\rtimes\rho$ for the chosen section $\rho\mapsto q_{\rho}$ to $\fP\rightarrow\Gal(\Qb/\Q)$ (Remark \ref{rem:comments_on_zeta_v}, (3)), the $1$-cochain $\rho\mapsto g_{\rho}$ takes values in $I(\Qb)$ and its image in $I^{\ad}(\Qb)$ is a cocycle (although the original cochain $\rho\mapsto g_{\rho}$ itself is not). The group $I_{\phi}$ is then defined by its cohomology class in $H^1(\Q,I^{\ad})$; so, there exists an inner twisting $\Qb$-isomorphism
\begin{equation} \label{eq:inner-twisting_by_phi}
\psi:I_{\Qb}\isom (I_{\phi})_{\Qb},
\end{equation}
under which
\[I_{\phi}(\Q)=\{g\in I(\Qb)\ |\ g_{\rho}\rho(g)g_{\rho}^{-1}=g\}=\{g\in G(\Qb)\ |\ \Inn(g)\circ\phi=\phi\}.\]

We will also say that an admissible pair $(\phi,\epsilon)$ is \textit{well-located} if $\phi$ is a well-located admissible morphism and $\epsilon\in I_{\phi}(\Q)$ lies in $G(\Q)$ (a priori, $\epsilon\in I_{\phi}(\Q)$ is only an element of $G(\Qb)$ via $I_{\phi}(\Q)\subset I_{\phi}(\Qb)=I(\Qb)\subset G(\Qb)$). If $\phi$ is well-located in $H$ and $\epsilon\in I_{\phi}(\Q)\cap H(\Q)$ for a $\Q$-subgroup $H$ of $G$, the admissible pair $(\phi,\epsilon)$ will be said to be \textit{well-located in $H$}. Clearly, any admissible pair $(\phi,\epsilon)$ nested in some special Shimura datum $(T,h)$ is well-located (in $T$). 

We also note that for any admissible morphism $\phi$ mapping into $\fG_T$ for a maximal $\Q$-torus $T$, the torus $T$ is also a $\Q$-subgroup of $I_{\phi}$ (so $T(\Q)\subset I_{\phi}(\Q)=\{g\in G(\Qb)\ |\ \Inn(g)\circ\phi=\phi\}$). Indeed, suppose that $\phi:\fP\rightarrow \fG_T$ factors through $\fP(K,m)$ for a CM field $K$ Galois over $\Q$ and $m\in\N$. We need to check that any $\epsilon\in T(\Q)$ commutes with $\phi(\delta_n)$ and $\phi(q_{\rho})$ (for all $n\gg 1$ and $\rho\in\Gal(\Qb/\Q)$). But, the first is obvious
%since $\im(\phi^{\Delta})\subset T(\Qb)(\subset \fG_T=T(\Qb)\rtimes\Gal(\Qb/\Q))$,
and for the second, in general, for $x\in G(\Qb)$, we have
\[\phi(q_{\rho})x\phi(q_{\rho})^{-1}=(g_{\rho}\rho)x(g_{\rho}\rho)^{-1}=(g_{\rho}\rho)x(\rho^{-1}(g_{\rho}^{-1})\rho^{-1})=g_{\rho}\rho(x)g_{\rho}^{-1}\cdot1,\]
so that $\epsilon\in T(\Q)$ and $g_{\rho}\in T(\Qb)$ together imply that $\epsilon\in I_{\phi}(\Q)$.

%As mentioned earlier in the introduction to this section, we assume that the Serre condition for $(G,X)$ holds.
%%%%%%%%%%%%%%%%%%%%
%%%%%%%%%%%%%%%%%%%%
\begin{thm} \label{thm:LR-Satz5.3}
Assume that $G_{\Qp}$ is quasi-split, and that $G^{\uc}_{\Qp}$ is a product $\prod_i \Res_{F_i/\Qp}G_i$ of simple groups each of which is the restriction of scalars $\Res_{F_i/\Qp}G_i$ of an absolutely simple group $G_i$ over a field $F_i$ such that $G_i$ splits over a tamely ramified extension of $F_i$.
Let $\mbfK_p$ be a parahoric subgroup of $G(\Qp)$. 

(1) Every admissible morphism $\phi:\fP\rightarrow \fG_G$ is conjugate to a special admissible morphism $i\circ\psi_{T,\mu_h}:\fP\rightarrow \fG_T\hra\fG_G$ (for some special Shimura sub-datum $(T,h\in\Hom(\dS,T_{\R})\cap X)$ and $i:\fG_T\rightarrow\fG_G$ the canonical morphism defined by the inclusion $i:T\hra G$).

(2) 
%Assume that $G^{\der}$ is simply connected. 
Let $\phi:\fP\rightarrow \fG_G$ be a \emph{well-located} (i.e. $\phi(\delta_n)\in G(\Q)$ for all $n\gg1$) admissible morphism. Then, for any maximal torus $T$ of $G$, elliptic over $\R$, such that
\begin{itemize} \addtolength{\itemsep}{-4pt}
\item[(i)] $\phi(\delta_n)\in T(\Q)$ for a sufficiently large $n$ (i.e. $T\subset I:=Z_G(\phi(\delta_n))$), and
\item[(ii)] $T_{\Qp}\subset I_{\Qp}$ is elliptic,
\end{itemize}
$\phi$ is conjugate to a special admissible morphism $\psi_{T',\mu_{h'}}$, where $T'=\Int g'(T)$ for some transfer of maximal torus $\Int g':T\hra G\ (g'\in G(\Qb))$ (with respect to the identity inner twisting, i.e. the composite map $T_{\Qb}\hra G_{\Qb}\stackrel{\Int  g'}{\lra}G_{\Qb}$ is defined over $\Q$).

\begin{itemize} \addtolength{\itemsep}{-4pt}
\item[(iii)] If furthermore $T_{\Ql}\subset G_{\Ql}$ is elliptic at some prime $l\neq p$ and $H^1(\Qp,T)=0$,
\end{itemize}
we may find such transfer of maximal torus $\Int g':T\hra G$ which is also conjugation by an element of $G(\Qp)$ (i.e. $\Int g':T_{\Qpb}\rightarrow G_{\Qpb}$ equals $\Int y$ for some $y\in G(\Qp)$).
\end{thm}

In the case where the level subgroup $\mbfK_p$ is hyperspecial and $G^{\der}=G^{\uc}$, the first statement is Satz 5.3 of \cite{LR87}. The second statement is new; it was motivated by gaining control on special admissible morphism (especially on the torus $T$) that is conjugate to a given admissible morphism. 
%In contrast, the point of the second statement is to shed light on such special admissible morphisms, focusing on the tori in the special Shimura subdata. 

To establish these statements, we will adapt the arguments from the proof of \textit{loc. cit.}, which we now review briefly. It can be divided into three steps: 
\begin{itemize}
\item[I.] The first step is to replace a given admissible morphism $\phi$ by its conjugate that is well-located (i.e. when we denote the conjugate again by $\phi$, we have $\phi(\delta_n)\in G(\Q)$).
\item[II.] The second step is to show that there is a conjugate $\Int  g\circ\phi:\fP\rightarrow\fG_G\ (g\in G(\Qb))$ of $\phi$ in step I which factors through $\fG_T$ for \emph{some} maximal $\Q$-torus $T$ (of $G$) that is elliptic over $\R$.
Again, let us denote this conjugate $\Int  g\circ\phi$ by $\phi$. 
\item[III.]
The final step is to find an admissible embedding $\Int g':T \hra G$ (with respect to the identity inner twisting $\mathrm{id}:G_{\Qb}=G_{\Qb}$, $g'\in I(\Qb)$) such that $\Int  g'\circ\phi:\fP\rightarrow \fG_{\Int  g'(T)}$ becomes a special admissible morphism. 
\end{itemize}

We point out that while the first two steps will be established by arguments in Galois cohomology which do not use level subgroups, we need to validate the third step for parahoric level subgroups (see Lemma \ref{lem:LR-Lemma5.11} below) and also that in the original argument (i.e. proof of  \cite{LR87}, Satz 5.3), the assumption $G^{\der}=G^{\uc}$ was needed only in the third step.
For the finer claim (2), it is also necessary to strengthen the second step in addition to the third one. 
For that purpose (and some other potential applications), we formalize (with some improvements incorporated) these steps into two propositions respectively: Prop. \ref{prop:existence_of_admissible_morphism_factoring_thru_given_maximal_torus}, Prop. \ref{prop:equivalence_to_special_adimssible_morphism}. We first introduce these propositions, postponing their proofs to the next subsection.

%%%%%%%%%%%%%%%%%%%%
\begin{prop} \label{prop:existence_of_admissible_morphism_factoring_thru_given_maximal_torus}
Let $\phi:\fP\rightarrow\fG_G$ be an admissible morphism and $T$ a maximal $\Q$-torus of $G$, elliptic over $\R$, having properties (i) and (ii) of Thm. \ref{thm:LR-Satz5.3}, (2). Then, there exists $g\in I(\Qb)$ such that $\Int  g\circ\phi$ maps $\fP$ into $\fG_T(\hra \fG_G)$. 
\end{prop}

Notice that here we only modify $\phi$ (while keeping $T$).
This proposition can be regarded as a souped-up version of the second step above, in the following sense. In the original second step (i.e. \cite{LR87}, p. 176, from line 1 to -5), we are given an admissible morphism $\phi$ and want to find a maximal $\Q$-torus $T$ such that some conjugate of $\phi$ maps into $\fG_{T}\subset \fG_G$. There, $\phi$ is considered somewhat as given and one looks for $T$ with this property. As such, the choice of $T$ is restricted by $\phi$, namely, for arbitrary $\phi$ there is very little room of choice for such $T$. In our case, however, we start with some fixed torus $T$ and demand that a conjugate of $\phi$ factors through $\fG_T$. It turns out that this becomes possible if $T$ satisfies properties (i) and (ii) of Thm. \ref{thm:LR-Satz5.3}, (2).

Also we note that the new pair $(\Int g\circ\phi,T)$ still enjoys properties (i) and (ii),
since $\Int g\circ\phi(\delta_n)=g\phi(\delta_n)g^{-1}=\phi(\delta_n)$ (as $g\in I(\Qb)$). 

The next proposition is also an enhanced version of the third step above.

%%%%%%%%%%%%%%%%%%%%
\begin{prop} \label{prop:equivalence_to_special_adimssible_morphism}
If an admissible morphism $\phi:\fP\rightarrow\fG_G$ is well-located in a maximal $\Q$-torus $T$ of $G$ that is elliptic over $\R$, there exists an admissible embedding of maximal torus $\Int g'|_{T}:T\hra G$ such that $\Int  g'\circ\phi:\fP\rightarrow\fG_{T'}$ is special, where $T':=\Int  g'(T)$ ($\Q$-torus), i.e. $\Int  g'\circ\phi=\psi_{T',\mu_{h'}}$ for some $h'\in X\cap\Hom(\dS,T'_{\R})$.

If $T$ further fulfills condition (iii) of Thm. \ref{thm:LR-Satz5.3}, one can find such transfer of maximal torus $\Int g'$ (i.e. $\Int g'\circ\phi$ becomes special admissible) which also satisfies that
\begin{itemize}
\item $\Int g'|_{T_{\Qpb}}:T_{\Qpb}\hra G_{\Qpb}$ equals $\Int y|_{T_{\Qp}}:T_{\Qp}\hra G_{\Qp}$ for some $y\in G(\Qp)$. 
\end{itemize}
In particular, $T_{\Qp}$ and $T'_{\Qp}$ are conjugate under $G(\Qp)$.
\end{prop}

In the original case of hyperspecial $\mbfK_p$, the first statement of this proposition is established in the course of proving Satz 5.3 (more precisely, in the part beginning from Lemma 5.11 until the end of the proof of that theorem; this is also the only part in the proof of Satz 5.3 that requires the assumption $G^{\der}=G^{\uc}$, cf. Footnote \ref{ftn:sc-assumption1}).
On the other hand, we remark that the idea of exploiting a transfer of maximal torus $i:T\isom T'$ which becomes $G(\Qp)$-conjugacy first appeared in our previous work \cite[Thm. 4.1.1]{Lee16}, and will find similar application here later.

Finally, we point out that the three properties (i) - (iii) of Theorem \ref{thm:LR-Satz5.3} remain intact under any transfer of maximal torus. Indeed, an inner automorphism does not interfere with the center and a transfer of maximal torus $i:T\rightarrow G$ also restricts to a $\Q$-isomorphism $i:Z(Z_{G}(t))\isom Z(Z_{G}(i(t)))$ for any $t\in T(\Q)$.

In the next subsection, we present the proofs of Prop.
\ref{prop:existence_of_admissible_morphism_factoring_thru_given_maximal_torus}, \ref{prop:equivalence_to_special_adimssible_morphism}, and Theorem \ref{thm:LR-Satz5.3}.%%
\footnote{The original arguments in \cite{LR87} use the quasi-motivic Galois gerb $\fQ$ (instead of the pseudo-motivic Galois gerb $\fP$) whose definition is however wrong (cf. \cite[Appendix B]{Reimann97}). Fortunately, the whole arguments remain valid with $\fQ$ replaced by $\fP$, as long as the (admissible) morphisms in question factor through $\fP$.}
%%

%%%%%%%%%%%%%%%%%%%%
%%%%%%%%%%%%%%%%%%%%
\subsection{Proofs of Propositions \ref{prop:existence_of_admissible_morphism_factoring_thru_given_maximal_torus},  \ref{prop:equivalence_to_special_adimssible_morphism}, and Theorem \ref{thm:LR-Satz5.3}}

Recall that we have fixed a continuous section $\rho\mapsto q_{\rho}$ to the projection $\fP\thra\Gal(\Qb/\Q)$ (Remark \ref{rem:comments_on_zeta_v}, (3)).

%%%%%%%%%%%%%%%%%%%%
\begin{lem} \label{lem:criterion_for_admissible_morphism_to_land_in_torus}
Let $\phi:\fP\rightarrow\fG_G$ be a well-located admissible morphism. Let $T$ be a maximal $\Q$-torus of $I:=Z_G(\phi(\delta_n))\ (n\gg1)$ and $\psi:I_{\Qb}\isom (I_{\phi})_{\Qb}$ the inner twisting (\ref{eq:inner-twisting_by_phi}).

If for some $a\in G(\Qb)$, $\Int a^{-1}:T_{\Qb}\hra G_{\Qb}$ induces a $\Q$-rational map from $T$ to $I_{\phi}$, the latter being regarded as a $\Qb$-subgroup of $G_{\Qb}$ via the inner twisting $\psi$, $\Int  a\circ\phi$ factors through $\fG_T\subset \fG_G$. 

f $a\in I(\Qb)$, the converse also holds: if $\Int  a\circ\phi$ factors through $\fG_T$, $\psi\circ\Int  a^{-1}|_T:(T)_{\Qb}\hookrightarrow I_{\Qb}\isom (I_{\phi})_{\Qb}$ is $\Q$-rational (i.e. is a transfer of the maximal torus $T\subset I$ into $I_{\phi}$ with respect to $\psi$). 
\end{lem}

In particular, a conjugate of $\phi$ maps into $\fG_T$ if $T\subset I$ transfers to $I_{\phi}$ with respect to the inner twisting $I_{\Qb}\isom (I_{\phi})_{\Qb}$ (\ref{eq:inner-twisting_by_phi}). 

\begin{proof} %Since $\psi:I_{\Qb}\isom (I_{\phi})_{\Qb}$ is an inner twist, it restricts to a $\Q$-isomorphism $Z(I)\isom Z(I_{\phi})$. 
As $\phi$ is well-located, $\phi(\delta_n)\in Z(I_{\phi})(\Q)$ and thus $\Int a\circ\phi(\delta_n)\in  T(\Q)$. This is equivalent to that $(\Int a\circ\phi)^{\Delta}$ maps $P$ to $T$, as $\{\delta_n^k\}_{k\in\N}$ is Zariski-dense in $P(K,n)$ (for any suitable CM field $K$) (Lemma \ref{lem:Reimann97-B2.3}, (2)).
Next, via $\psi$ we identify $I_{\phi}(\Qb)$ with $I(\Qb)\subset G(\Qb)$ endowed with the twisted Galois action $g\mapsto g_{\rho}\rho(g)g_{\rho}^{-1}$, where $g\mapsto\rho(g)$ is the original Galois action on $G(\Qb)$. 
Let $\phi(q_{\rho})=g_{\rho}\rho$. Then, the condition means that 
\begin{equation} \label{eqn:transfer_from_I_to_I_phi}
g_{\rho}\rho(a^{-1}ta)g_{\rho}^{-1}=a^{-1}\rho(t)a
\end{equation}
for all $t\in T(\Qb)$, which is the same as that $ag_{\rho}\rho(a)^{-1}\in T(\Qb)$, as $Z_G(T)=T$. As $a\phi a^{-1}(q_{\rho})=ag_{\rho}\rho(a)^{-1} \rho$, this implies the assertion. 
\end{proof}

%%%%%%%%%%%%%%%%%%%%
\subsubsection{Proof of Proposition \ref{prop:existence_of_admissible_morphism_factoring_thru_given_maximal_torus}.}

\begin{proof}
Let $I=Z_G(\phi(\delta_n))$ and $\psi:I_{\Qb}\isom (I_{\phi})_{\Qb}$ the inner twisting defined by $\phi$ (\ref{eq:inner-twisting_by_phi}); by our assumption that $\phi(\delta_n)\in T(\Q)$, $T$ is a $\Q$-subgroup of $I$. So, by Lemma \ref{lem:criterion_for_admissible_morphism_to_land_in_torus}, it suffices to prove that $T$ transfers to $I_{\phi}$ (with respect to the conjugacy class of $\psi$). 
Since $T_{\Qv}$ is elliptic in $I_{\Qv}$ at a place $v$ (i.e. at $v=\infty, p$), according to \cite[Lem. 5.6]{LR87}
(in which $G^{\ast}$ does not need to be quasi-split; see also the discussion in $\S$9 of \cite{Kottwitz84a}, particularly 9.4.1, 9.5), it suffices to check that $T$ transfers to $I_{\phi}$ everywhere locally (with respect to $\psi_{\Qvb}$). At $v=\infty, p$, this already follows from the condition that $T_{\Qv}$ is an elliptic torus of $I_{\Qv}$ \cite[Lem. 5.8, 5.9]{LR87}. At a finite place $v\neq p$, since $\phi\circ\zeta_v$ is conjugate to the canonical trivialization $\xi_v:\fG_v\rightarrow \fG_G(v)=G(\Qvb)\rtimes\fG_v$ (Def. \ref{defn:admissible_morphism}, (2)), the inner-twisting $\psi_{\Qlb}:I_{\Qvb}\isom (I_{\phi})_{\Qvb}$ (via the chosen embedding $\Qb\hra\Qvb$), after conjugation, descends to an $\Qv$-isomorphism \[I_{\Qv}\isom (I_{\phi})_{\Qv}.\qedhere \]
\end{proof}

%%%%%%%%%%%%%%%%%%%%
As was explained after the statement of Theorem \ref{thm:LR-Satz5.3}, Proposition \ref{prop:equivalence_to_special_adimssible_morphism} is a strengthening of the third step in the proof of Satz 5.3 in \cite{LR87}. 
The proof of this step in \textit{loc. cit.} itself proceeds in three steps: Lemma 5.11, Lemma 5.12, and the rest of the proof of Satz 5.3 (p.181, line 1-19 of \cite{LR87}). 
Again we will prove our proposition along the same line. First, we need some facts from Bruhat-Tits theory.

%%%%%%%%%%%%%%%%%%%%
\begin{lem} \label{lem:specaial_parahoric_in_Levi}
(1) Let $G$ be a (connected) reductive group over a field $F$. Then, for any $F$-split $F$-torus $A_M$ in $G$, its centralizer $M:=Z_G(A_M)$ is an $F$-Levi subgroup of $G$ (i.e. a Levi factor defined over $F$ of an $F$-parabolic subgroup of $G$). If $G$ is quasi-split, then so is $M$. 

From now on, we suppose that $F$ is a complete discrete valued field with perfect residue field (mainly, local fields or $\mfk=\mathrm{Frac}(W(\Fpb))$), and $G$ a connected reductive group over $F$.
As before, let $M=Z_G(A_M)$ for a split $F$-torus $A_M$, and fix a maximal split $F$-torus $S$ of $G$ containing $A_M$. Let $\mcA^G=\mcA^G(S,F)$ and $\mcA^M=\mcA^M(S,F)$ denote respectively the apartments in the buildings $\mcB(G,F)$ and $\mcB(M,F)$ corresponding to $S$.

(2) Every affine root $\alpha$ of $\mcA^G$ whose vector part $a=v(\alpha)\in \Phi(G,S)$ is a root in $\Phi(M,S)$ is also an affine root of $\mcA^M$.

(3) For any special maximal parahoric subgroup $K$ of $G(F)$ associated with a special point in $\mcA^G$, the intersection $K\cap M(F)$ is also a special maximal parahoric subgroup of $M(F)$. 
\end{lem}

\begin{proof} (1) These are standard. For the first claim, see \cite[Thm. 4.15]{BT65}.
The second claim is easily seen. We use the fact that for a (connected) reductive group $H$ over a field $F$, $H$ is quasi-split if and only if for a (equiv. any) maximal $F$-split torus of $H$, its centralizer in $H$ is a (maximal) torus.

Now, as any torus containing $A_M$ is a subgroup of $M=Z_{G}(A_M)$, so is any maximal $F$-split torus of $G$ containing $A_M$; choose one and call it $S$. As $G$ is quasi-split, the centralizer $T:=Z_{G}(S)$ of $S$ is a torus of $G$, thus is itself contained in $M$. Now, $T$ is also the centralizer of $S$ in $M$. 

(2) This also follows readily from definition. First, we recall that the relative root datum $\Phi(M,S)=(X_{\ast}(S),R_{\ast}(M),X^{\ast}(S),R^{\ast}(M))$ for $(M,S)$ is a closed sub-datum of the root datum $\Phi(G,S)=(X_{\ast}(S),R_{\ast}(G),X^{\ast}(S),R^{\ast}(G))$ defined by a subset $I$ of the set $\Delta=\{a_1,\cdots,a_n\}$ of simple roots (for some ordering on $R^{\ast}(G)$) \cite[Thm. 4.15]{BT65}: 
\begin{equation} \label{eqn:basis_of_simple_roots_for_Levi}
R^{\ast}(M)=R^{\ast}(G)\cap \sum_{a_i\in I}\Z a_i,
\end{equation} 
and $A_M=(\cap_{\alpha\in I}\ker(\alpha))^{\mathrm{o}}$ (the largest split $F$-torus in the center $Z(M)$). 
Next, for an affine function $\alpha$ on $\mathcal{A}(S,F)\cong X_{\ast}(S)_{\R}$ (regarded as a common affine space without any apartment structure) whose vector part belongs to $\Phi$, let $X_{\alpha}^G$ be defined as in \cite[1.4]{Tits79} with respect to $G$, i.e.
\[X_{\alpha}^G=\{u\in U_{v(\alpha)}(F)\ |\ u=1\text{ or }\alpha(v(\alpha),u)\geq \alpha\}.\]
Here, for $a\in R^{\ast}(G,S)$, $U_a$ refers to the associated root group. This an unipotent $F$-group, which was denoted by ${}_FU_a$ or $U_{(a)}$ in \cite{BT65}, 5.2.%%
\footnote{In turn, this is the group that was denoted by $G^{\ast (S)}_{(a)}$ (or $G^{\ast}_{(a)}$) in \textit{loc. cit.} 3.8. Namely, when we choose  a maximal $\overline{F}$-torus $T$ of $G_{\overline{F}}$ containing $S$, it is the group generated by ${}_{\overline{F}}U_b$ (the ``absolute'' root group in $G_{\overline{F}}$ defined with respect to $(G_{\overline{F}},T)$) for the absolute roots in $R^{\ast}(G_{\overline{F}},T)$ whose restriction to $S$ belong to $(a)$, the set of relative roots in $R^{\ast}(G,S)$ that are positive integer multiples of $a$.}
When $a\in R^{\ast}(M,S)$, as $U_{a}\subset M$ for $a\in R^{\ast}(M)$, it follows from definition that this $U_a$ is the same group as that defined regarding $a$ as a root for $(M,S)$. Similarly, if $v(\alpha)\in R^{\ast}(M)$, the same is also true of the affine function $\alpha(v(\alpha),u)$. In more detail, its definition uses only the properties of \textit{root (group) datum} (of type some root system) in the sense of \cite[6.1]{BT72}. For $S$ and $U_{a}\ (a\in R^{\ast}(G,S))$ as above, there exist certain $S$-right cosets $\{M_a\}_{a\in R^{\ast}(G,S)}$ such that the family of subgroups 
\[\{S,\{U_{a},M_{a}\}_{a\in \Phi(G,S)}\}\] 
becomes a root group datum of type $\Phi(G,S)$ (in $G$) (in fact $M_a$ is then a subset of the group generated by $\{S,U_{a},U_{-a}\}$, cf. \cite[(6.1.2), (9)]{BT72}). In particular, the element $m(u)$ (for each $u\in U_a(F)\backslash\{1\}$) appearing in \cite[1.4]{Tits79} belongs to $M_a$ and is determined solely by the root group datum $\{S,\{U_{a},M_{a}\}_{a\in \Phi(G,S)}\}$ \cite[(6.1.2), (2)]{BT72}. But, as $\Phi(M,S)$ is \textit{(quasi-)closed} in $\Phi(G,S)$ and $U_{a}\subset M$ for $a\in R^{\ast}(M)$, the subset $\{T,\{U_{\alpha},M_{\alpha}\}_{\alpha\in \Phi(M,S)}\}$ is also a root datum of type $\Phi(M,S)$ (in $G$), cf. \cite[7.6]{BT72}. Hence we can drop the superscript $G$ in $X_{\alpha}^G$ without ambiguity.
Now, we recall the definition \cite[1.6]{Tits79} that an affine function $\alpha$ is an \textit{affine root} of $G$ (relative to $S$ and $F$) if $X_{\alpha}$ is not contained in $X_{\alpha+\epsilon}\cdot U_{2v(\alpha)}\ (=X_{\alpha+\epsilon}\text{ if }2v(\alpha)\notin\Phi$) for any strictly positive constant $\epsilon$. The claim in question is immediate from this definition and the above discussions. 

(3) It is shown in \cite[Lem.  4.1.1]{HainesRostami10} that $K\cap M(F)$ is a parahoric subgroup of $M(F)$. So we just have to show that it is special maximal parahoric. 
Using the special point $\mbfo\in \mcA^G(S,F)$, we may embed $\mcB(M,F)$ into $\mcB(G,F)$ such that $\mbfo$ lies in the image \cite[7.6.4]{BT72}, \cite[4.2.17-18]{BT84}. Let $\mcA((M^{\der}\cap S)^{\mathrm{o}},F)$ be the apartment corresponding to the maximal split $F$-torus $(M^{\der}\cap S)^{\mathrm{o}}$ of $M^{\der}$.
As the affine hyperplanes in $\mcA^{M}$ form a subset of those in $\mcA^{G}$, it is obvious that $\mbfo$ is contained in a unique facet $\mbfa_{\mbfo}^{M}$ in $\mcA^{M}$, i.e. in
\[\mbfa^{M}_{\mbfo}\cong X_{\ast}(Z(M))_{\R}\times \{v_{\mbfo}\}.\]
for some unique facet $v_{\mbfo}$ in the apartment $\mcA((M^{\der}\cap S)^{\mathrm{o}},F)$ (recall that $M$ is the centralizer of a split torus $Z(M)$ which then must be the center). 
Now, we claim that $v_{\mbfo}$ is a vertex of $\mcA((M^{\der}\cap S)^{\mathrm{o}},F)$.
Indeed, as $\mbfo$ is a special point, we may identify the affine space $\mcA^{G}$ with the vector space $X_{\ast}(S)_{\R}$ ($\mbfo$ becoming the origin) so that the root hyperplanes $\{H_{\alpha}\}_{\alpha\in R^{\ast}(G,S)}$ are all affine hyperplanes. Clearly, $\mbfo=\cap_{\alpha\in R^{\ast}(G,S)}H_{\alpha}\cap X_{\ast}((G^{\der}\cap S)^{\mathrm{o}})_{\R} $. Let $I\subset \Delta$ be the subset defining the root datum $\Phi(M,S)$ as in (2); so, the center $Z(M)$ is equasl to $\cap_{\alpha\in I}H_{\alpha}$ (intersection of root hyperplanes in $X_{\ast}(S)_{\R}$).
But, by (2), $\{H_{\alpha}\}_{\alpha\in I}$ is also a subset of affine hyperplanes in the apartment $\mcA^M$ of $(M,S)$, and $\mbfo$ is contained in the intersection of these linearly independent affine hyperplanes in $\mcA(M,S)$, whose dimension is thus equal to rank $r_M$ of $Z(M)$. Hence, $\mbfo$ is contained in a facet of $\mcA^M$ of dimension at most $r_M$, which implies that the facet $v_{\mbfo}$ is of zero-dimension, i.e. a vertex in the building for $M^{\der}$.

Once we know that $v_{\mbfo}$ is a vertex, the fact that it is a special vertex also follows readily from (2). Indeed, 
by definition \cite[1.9]{Tits79}, we need to check that every root of $a\in\Phi(M,S)$ is the vector part of an affine roof of $(M,S)$ vanishing at $v_{\mbfo}(\in \mcA^{M^{\der}})$. We know that $a$ also belongs to $\Phi(G,S)$, thus since $\mbfo$ is special, there exists an affine root $\alpha$ of $(G,S)$ with vector part $v(\alpha)=a$ and vanishing at $\mbfo$. But, by (2), such $\alpha$ is also an affine root of $\mcA^M$, and as such, it must vanish on the facet in $\mcA^M$ containing $\mbfo$, i.e. on $\mbfa^M_{\mbfo}\cong X_{\ast}(Z(M))_{\R}\times \{v_{\mbfo}\}$. \end{proof}

%%%%%%%%%%%%%%%%%%%%
\begin{lem} \label{lem:LR-Lemma5.11}
Assume that $G_{\Qp}$ is quasi-split and that $G^{\uc}_{\Qp}$ is a product $\prod_i \Res_{F_i/\Qp}G_i$ of simple groups each of which is the restriction of scalars $\Res_{F_i/\Qp}G_i$ of an absolutely simple group $G_i$ over a field $F_i$ such that $G_i$ splits over a tamely ramified extension of $F_i$.

Let $T_1\subset G_{\Qp}$ be a maximal $\Qp$-torus, split by a finite Galois extension $K$ of $\Qp$, $b\in T_1(\mfk)$, and $\{\mu\}$ a $G(\Qpb)$-conjugacy class of minuscule cocharacters of $G_{\Qpb}$. Let $\mbfK_p$ be a (not necessarily special maximal) parahoric subgroup of $G(\Qp)$.

If $X(\{\mu\},b)_{\mbfK_p}\neq\emptyset$, there exists $\mu\in X_{\ast}(T_1)\cap \{\mu\}$ such that
\begin{equation} \label{eq:equality_on_the_kernel}
\Nm_{K/\Qp}\mu=[K:\Qp]\nu_b,
\end{equation}
where $\nu_b\in X_{\ast}(T_1)_{\Q}$ is the Newton homomorphism attached to $b$. 

In particular, if $\phi$ is an admissible morphism well-located in a maximal $\Q$-torus $T$ of $G$ that is elliptic over $\R$,  there exists $\mu\in X_{\ast}(T)\cap\{\mu_X\}$ such that $\phi$ and $\psi_{T,\mu}$ coincide on the kernel of $\fP$. 
\end{lem}

Here, $i\circ\psi_{T,\mu}$ is not necessarily admissible, because $\mu$ may not be $\mu_{h'}$ for some $h'\in X$. 

\begin{proof} This is proved in Lemma 5.11 of \cite{LR87} when $\mbfK_p$ is a hyperspecial subgroup. 
We will adapt its argument for general parahoric subgroups. 
We first show how the first statement implies the second one.
Since the kernel of the Galois gerb $\fP$ is the projective limit of $P(L,m)(\Qb)$, where $L$ runs through CM Galois extensions of $\Q$ and $m\in\N$ varies with respect to divisibility (cf. \autoref{subsubsec:pseudo-motivic_Galois_gerb}), we only need to show it after restricting $\phi$ and $\psi_{T,\mu}$ to $P(L,m)$ (for all sufficiently large Galois CM field $L$ and $m$). Let $L$ be a CM-field splitting $T$. Then, the Galois-gerb morphisms $\psi_{T,\mu}:\fP\rightarrow \fG_T$ and $\zeta_p:\fG_p\rightarrow\fP(p)$ factor through $\fP^L$ and $\fG_p^{L_{v_2}}$, respectively (Lemma \ref{lem:defn_of_psi_T,mu}, (2) and Remark \ref{rem:comments_on_zeta_v}), where as usual $v_2$ denotes (by abuse of notation) the place of $L$ induced by the fixed embedding $L\hra \Qpb$. Let $\zeta_p^{L_{v_2}}:\fG_p^{L_{v_2}}\rightarrow\fP(p)$ denote the induced morphism.
Then, according to the definition of $\psi_{T,\mu}$ (cf. \cite{LR87}, p. 143-144), when $\mu+\iota\mu$ is defined over $\Q$, for sufficiently large $m\in\N$, $\psi_{T,\mu}(\delta_m)$ is the unique element $t$ in $T(\Q)$ such that for all $\lambda\in X^{\ast}(T)$, $\lambda(t)$ is a Weil number and 
\[|\prod_{\sigma\in\Gal(L_{v_2}/\Q_p)}\sigma\lambda(t)|_p=q^{-\langle\lambda,\Nm_{L_{v_2}/\Qp}\mu\rangle} \]
holds with $q=p^m$.
Further, $\psi_{T,\mu}(p)\circ\zeta_p:\fG_p^{L_{v_2}}\rightarrow \fG_T(p)$ is conjugate to $\xi_{-\mu}^{L_{v_2}}$ by an element of $T(\Qpb)$ (Lemma \ref{lem:properties_of_psi_T,mu}, (2)), where $\xi_{-\mu}^{L_{v_2}}:\fG_p^{L_{v_2}}\rightarrow \fG_T(p)$ is the morphism defined in Definition \ref{defn:psi_T,mu} for $(T_{\Qp},\mu,L_{v_2})$.
On the other hand, by enlarging $L$ if necessary, we may assume that $\phi$ also factors through $\fP^L$, and that there exists a Galois $\Qpnr/\Qp$-gerb morphism $\xi_p':\fD_l\rightarrow\fG_{T_{\Qp}}^{\nr}$ whose inflation to $\Qpb$ is $T(\Qpb)$-conjugate to $\xi_p=\phi(p)\circ\zeta_p$, where $l:=[L_{v_2}:\Qp]$, (Lemma \ref{lem:unramified_morphism}, (2), or Lemma \ref{lem:unramified_conj_of_special_morphism}). Then, for $b\sigma:=\xi_p'(s_{\sigma}^l)$, we have 
\[-[L_{v_2}:\Qp]\nu_b=\nu_p:=\phi(p)^{\Delta}\circ(\zeta_p^{L_{v_2}})^{\Delta}.\]
Next, recall (cf. \autoref{subsubsec:pseudo-motivic_Galois_gerb}) that each character of $P(L,m)$ is regarded as a Weil $q=p^m$-number in $L$, with the correspondence being realized in terms of $\delta_m$ by $\chi\mapsto \chi(\delta_m)$ (Lemma \ref{lem:Reimann97-B2.3}, (2)), and that for a Weil $q$-number $\pi$, $\chi_{\pi}$ is the notation regarding it as a character of $P(L,m)$. For $\lambda\in X^{\ast}(T)$, writing $\lambda\circ\phi^{\Delta}(\delta_m)$ as $\pi_{\lambda}$ for short (so that $\chi_{\pi_{\lambda}}=\lambda\circ\phi^{\Delta}$), we see that
\begin{eqnarray*}
|\prod_{\sigma\in\Gal(L_{v_2}/\Q_p)}\sigma\lambda(\phi^{\Delta}(\delta_m))|_p&=&
|\prod_{\sigma\in\Gal(L_{v_2}/\Q_p)}\sigma\pi_{\lambda}|_p\\
&=&q^{\langle \chi_{\pi_{\lambda}},\nu_2^{L_{v_2}}\rangle}=q^{\langle\lambda\circ\phi^{\Delta},(\zeta_p^{L_{v_2}})^{\Delta}\rangle} \\
&=&q^{\langle \lambda,\nu_p\rangle}.
\end{eqnarray*}
Here, the third equality is the property of $\nu_2^{L_{v_2}}=(\zeta_p^{L_{v_2}})^{\Delta}$ (Def. \ref{defn:Weil-number_torus}, (\ref{eqn:cocharacters_nu^K}), cf. \autoref{subsubsec:pseudo-motivic_Galois_gerb}). 
This shows that the first statement implies the second claim.

Now, we establish the first statement. We remark that we will reduce the general parahoric subgroup case to a situation involving only a special maximal parahoric subgroup.

Let $T_1^{\textrm{split}}$ be the maximal $\Qp$-split subtorus of $T_1$, $M$ the centralizer of $T_1^{\textrm{split}}$; thus $T_1\subset M$, and $M$ is a semi-standard $\Qp$-Levi subgroup (Lemma \ref{lem:specaial_parahoric_in_Levi}, (1)).  Below, there will be given a special point $\mbfo$ of the Bruhat-Tits building $\mcB(G,\mfk)$. Then, one can find a $\Qp$-torus $S'$ in $M$ whose extension to $\Qpnr$ becomes a maximal $\Qpnr$-split torus of $M_{\Qpnr}$, and a $M(\mfk)\rtimes\Gal(\Qpnr/\Qp)$-equivariant embedding $\mcB(M,\mfk)\hra\mcB(G,\mfk)$ such that $\mbfo$ lies in the image of the apartment $\mcA^{M}_{\mfk}\subset\mcB(M,\mfk)$ corresponding to $S'$.
Note that the centralizer $T':=Z_{G_{\Qp}}(S')$ is a maximal torus of $G_{\Qp}$ (thus, a maximal torus of $M$ as well), as $G_{\mfk}$ is quasi-split by a theorem of Steinberg.

We recall that for a facet $\mbff^{\sigma}$ in $\mcB(G_{\Qp},\Qp)$, there exists a unique $\sigma$-stable facet $\mbff$ in $\mcB(G_{\Qp},\mfk)$ with $\mbff^{\langle\sigma\rangle}=\mbff^{\sigma}$ \cite[5.1.28]{BT84}.
Let $\mcG_{\mbff}^{\mro}$ be the smooth $\cO_{\mfk}$-group scheme canonically attached to $\mbff$, so that it has connected geometric fibers and the elements of $\mcG_{\mbff}^{\mro}(\cO_{\mfk})$ fixes $\mbff$ pointwise (cf. \cite[5.2]{BT84}). Then, 
\[K_{\mbff}(\mfk):=\mcG_{\mbff}^{\mro}(\cO_{\mfk}),\quad K_{\mbff}(\Qp):=\mcG_{\mbfo}^{\mro}(\cO_{\mfk})^{\sigma}\]
are the pararhoric groups associated with the facet $\mbff$ (or $\mbff^{\sigma}$) (cf. \cite{HainesRapoport08}, Prop. 3). 

Let $\mcA^{G_{\Qp}}_{\mfk}$ be the apartment corresponding to $S'$. By conjugation, we assume that the given $\sigma$-stable facet $\mbff$ defining $\mbfK_p$ (i.e. $\mbfK_p=K_{\mbff}(\Qp)$) lies in $\mcA^{G_{\Qp}}_{\mfk}$.
We fix a $\sigma$-stable alcove $\mbfa$ in $\mcA^{G_{\Qp}}_{\mfk}$ whose closure contains $\mbff$, and let $K_{\mbfa}(\mfk)$ be the corresponding Iwahoric subgroup of $G(\mfk)$. The alcove $\mbfa$ then must contain some (not necessarily $\sigma$-stable) special point $\mbfo$ in its closure. If $K_{\mbfo}(\mfk)\subset G(\mfk)$ denotes the associated special maximal parahoric subgroup, we have the inclusion $K_{\mbfa}(\mfk)\subset K_{\mbfo}(\mfk)$ since $\mbfo$ is in the closure of $\mbfa$.
Now, as both  $\mbfa$ and $\mbff$ are $\sigma$-stable, according to \cite[Thm.1.1]{He15}, the condition $X(\{\mu\},b)_{K_{\mbff}(\Qp)}\neq\emptyset$ implies that $X(\{\mu\},b)_{K_{\mbfa}(\Qp)}\neq\emptyset$. 
Let $\mu_B$ be a dominant representative of $\{\mu\}$, where we choose the dominant Weyl chamber \emph{opposite} to the unique Weyl chamber containing the base alcove $\mbfa$ with apex at the special vertex $\mbfo$ (following the convention of \cite{HeRapoport15}). Also, recall (\ref{eqn:splitting_of_EAWG2}) that the choice of a base alcove $\mbfa$ presents the extended affine Weyl group $\widetilde{W}$ as the semidirect product $W_a\rtimes \Omega_{\mbfa}$ of the affine Weyl group $W_a$ (attached to $S'$) with the normalizer subgroup $\Omega_{\mbfa}\subset \widetilde{W}$ of $\mbfa$, thereby fixes a Bruhat order $\leq$ on $\widetilde{W}$ as well.

Let $g_1\in G(\mfk)$ be such that $g_1 K_{\mbfa}(\mfk)\in X(\{\mu_X\},b)_{K_{\mbfa}(\mfk)}$, i.e. if 
\begin{equation} \label{eqn:Iwahori_invariant} 
\mathrm{inv}_{K_{\mbfa}(\mfk)}(g_1,b\sigma(g_1))=\widetilde{W}_{K_{\mbfa}(\mfk)}\cdot w_1\cdot \widetilde{W}_{K_{\mbfa}(\mfk)}\quad (w_1\in \widetilde{W}),
\end{equation}
under the isomorphism $K_{\mbfa}(\mfk)\backslash G(\mfk)/K_{\mbfa}(\mfk)\simeq \widetilde{W}_{K_{\mbfa}(\mfk)}\backslash\widetilde{W}/\widetilde{W}_{K_{\mbfa}(\mfk)}$, there exists $\mu'\in X_{\ast}(T')\cap W_0\cdot\mu_B$ that 
\begin{eqnarray} \label{eqn:Iwahori_inequality} 
\widetilde{W}_{K_{\mbfa}(\mfk)}\cdot w_1\cdot \widetilde{W}_{K_{\mbfa}(\mfk)} &\leq&\widetilde{W}_{K_{\mbfa}(\mfk)}\cdot t^{\underline{\mu'}}\cdot \widetilde{W}_{K_{\mbfa}(\mfk)}. 
\end{eqnarray} 
Here, we used the notations from (\autoref{subsubsec:mu-admissible_set}); namely, $\underline{\mu'}$ is the image of $\mu'$ in $X_{\ast}(T')_{\Gamma_{\mfk}}$, and for $\lambda\in X_{\ast}(T')_{\Gamma_{\mfk}}$, $t^{\lambda}$ denotes the corresponding element of $\widetilde{W}$ via $X_{\ast}(T')_{\Gamma_{\mfk}}\cong T'(\mfk)/T'(\mfk)_1\subset \widetilde{W}$.
Since $K_{\mbfa}(\mfk)\subset K_{\mbfo}(\mfk)$, the same relations (\ref{eqn:Iwahori_invariant}), (\ref{eqn:Iwahori_inequality}) continue to hold with $K_{\mbfa}(\mfk)$ replaced by $K_{\mbfo}(\mfk)$ (see \cite[(3.5)]{Rapoport05} for (\ref{eqn:Iwahori_inequality})). Be warned that this does not mean that $X(\{\mu\},b)_{K_{\mbfo}(\mfk)}\neq\emptyset$: the latter definition makes sense only when the point $\mbfo$ is $\sigma$-stable.

Therefore, we are given a string of $\Qp$-subgroups of $G_{\Qp}$:
\[S'\subset T'\subset M \supset T_1,\]
where
\begin{itemize}
\item[(a)] $M$ is a semi-standard $\Qp$-Levi subgroup of $G_{\Qp}$ (i.e. $M$ is the centralizer of a $\Qp$-split torus of $G_{\Qp}$);
\item[(b)] $S'$ is a $\Qp$-torus in $M$ whose extension to $\Qpnr$ becomes a maximal $\Qpnr$-split torus of $M_{\Qpnr}$ (thus $S'$ is also such a torus for $G_{\Qp}$);
\item[(c)] $T'=Z_{G_{\Qp}}(S')$ (thus, a maximal torus of $M$ and also of $G$);
\item[(d)] $T_1$ is an elliptic maximal torus of $M$ which $\nu_b$ factors through.
\end{itemize}
These satisfy the following properties: There exists a special point $\mbfo$ of $\mcB(G,\mfk)$ which lies in the image of the apartment $\mcA^{M}_{\mfk}\subset\mcB(M,\mfk)$ corresponding to $S'$, under a suitable embedding $\mcB(M,\mfk)\hra\mcB(G,\mfk)$. Also, the relations (\ref{eqn:Iwahori_invariant}), (\ref{eqn:Iwahori_inequality}) hold with $K_{\mbfa}(\mfk)$ replaced by $K_{\mbfo}(\mfk)$ (for some $g_1$, $w_1$, $\mu'$ as in there). In this set up, we establish the existence of $\mu\in X_{\ast}(T_1)\cap \{\mu\}$ satisfying (\ref{eq:equality_on_the_kernel}). We proceed in the following steps; we remark that in the first three steps, it is not necessary that $b\in T_1(\mfk)$, and it suffices that $b\in M(\mfk)$.

\begin{itemize}
\item[(1)] Let $Q$ be a $\Qp$-parabolic subgroup of $G_{\Qp}$ of which $M$ is a Levi factor. Then, 
by Iwasawa decomposition $G(\mfk)=Q(\mfk) K_{\mbfo}(\mfk)$, we may assume $g_1\in Q(\mfk)$: this follows from the classical Iwasawa decomposition $G(\mfk)=Q(\mfk) \mathrm{Fix}(\mbfo)$ \cite[3.3.2]{Tits79} and that $\mathrm{Fix}(\mbfo)\subset T'(\mfk)\cdot K_{\mbfo}(\mfk)$. Indeed, $\mathrm{Fix}(\mbfo)\subset G(\mfk)=K_{\mbfo}(\mfk)T'(\mfk) K_{\mbfo}(\mfk)$, so any $g\in \mathrm{Fix}(\mbfo)$ is written as $k_1 tk_2$ with $k_1,k_2\in K_{\mbfo}(\mfk)$ and $t\in T'(\mfk)\cap \mathrm{Fix}(\mbfo)$. But, as $K_{\mbfo}(\mfk)$ is normal in $\mathrm{Fix}(\mbfo)$, we see that $g\in T'(\mfk)\cdot K_{\mbfo}(\mfk)$.
When one writes $g_1=nm$ with $m\in M(\mfk)$ and $n\in N_Q(\mfk)$ ($N_Q$ being the unipotent radical of $Q$), one has that
\[g_1^{-1}b\sigma(g_1)=m^{-1}b\sigma(m)n',\]
where 
\[n'=\sigma(m)^{-1}b^{-1}n^{-1}b\sigma(n)\sigma(m).\]
One readily checks that $n'\in N(\mfk)$.

\item[(2)] Define $\mu''\in X_{\ast}(T')$ by
\[m^{-1}b\sigma(m)\in (K_{\mbfo}(\mfk)\cap M(\mfk))\ t^{\underline{\mu''}}\ (K_{\mbfo}(\mfk)\cap M(\mfk)),\]
using the Cartan decomposition for $(M,K_{\mbfo}(\mfk)\cap M(\mfk))$ (as $K_{\mbfo}(\mfk)\cap M(\mfk)$ is a special maximal parahoric subgroup of $M(\mfk)$, by Lemma \ref{lem:specaial_parahoric_in_Levi}, (3)). Note the equality: 
\[w_{G}(t^{\underline{\mu''}})=w_{G}(m^{-1}b\sigma(m))=w_{G}(m^{-1}b\sigma(m)n')=w_{G}(w_1)=w_{G}(t^{\underline{\mu'}}),\]
where the last equality holds since by definition of the Bruhat order on $\widetilde{W}=W_a\rtimes \Omega_{\mbfa}$, $w_1$ and $t^{\underline{\mu'}}$ have the same component in $\Omega_{\mbfa}$ and $w_{G}(=w_{G_{\mfk}})$ is trivial on the image in $G_{\Qp}(\mfk)$ of $G_{\Qp}^{\uc}(\mfk)$ ($G_{\Qp}^{\uc}$ being the universal covering of $G_{\Qp}^{\der}$) \cite[7.4]{Kottwitz97}. It follows that the images of $\mu',\mu''\in X_{\ast}(T')$ in $\pi_1(G)_{\Gamma_{\mfk}}$ are the same. 

On the other hand, if $k_1,k_2\in K_0(\mfk)\cap M(\mfk)$ are such that $m^{-1}b\sigma(m)=k_1t^{\underline{\mu''}}k_2$, then
$m^{-1}b\sigma(m)n'=k_1t^{\underline{\mu''}}k_2n'=k_1t^{\underline{\mu''}}n''k_2$
for some $n''\in N(\mfk)$.
Consequently, we see that
\[K_{\mbfo}(\mfk)\ t^{\underline{\mu''}}\ N_Q(\mfk)\cap K_{\mbfo}(\mfk)w_1 K_{\mbfo}(\mfk)\neq\emptyset.\]
By (\ref{eqn:Iwahori_inequality}), this implies that with respect to the Bruhat order on $\widetilde{W}_{K_{\mbfo}(\mfk)}\backslash\widetilde{W}/\widetilde{W}_{K_{\mbfo}(\mfk)}$, \[\widetilde{W}_{K_{\mbfo}(\mfk)}\cdot t^{\underline{\mu''}}\cdot \widetilde{W}_{K_{\mbfo}(\mfk)} \leq \widetilde{W}_{K_{\mbfo}(\mfk)}\cdot t^{\underline{\mu'}}\cdot \widetilde{W}_{K_{\mbfo}(\mfk)}.\] 
Indeed, the argument of the proof of \cite{HainesRostami10}, Lemma 10.2 establishes the following fact: for $x,y\in \widetilde{W}$, if $K_{\mbfo}(\mfk)yN_Q(\mfk)\cap K_{\mbfo}(\mfk)x K_{\mbfo}(\mfk)\neq\emptyset$, then $y\leq x'$ for some $x'\in\widetilde{W}_{K_{\mbfo}(\mfk)}\cdot x\cdot \widetilde{W}_{K_{\mbfo}(\mfk)}$. 
Also, if $x\leq y$ in the Bruhat order on $\widetilde{W}$, then $\widetilde{W}_{K_{\mbfo}(\mfk)}\cdot x \cdot\widetilde{W}_{K_{\mbfo}(\mfk)} \leq \widetilde{W}_{K_{\mbfo}(\mfk)}\cdot y\cdot \widetilde{W}_{K_{\mbfo}(\mfk)}$ \cite[8.3]{KR00}.

\item[(3)] 
By \cite[1.7]{Tits79}, with our choice of the special vertex $\mbfo$, the affine space $\mcA(S',\mfk)$ is identified with the real vector space $V:=X_{\ast}(S')_{\R}=X_{\ast}(T')_I\otimes\R$ (with the origin being $\mbfo$), and there exists a reduced root system ${}^{\mbfo}\Sigma$ whose roots belong to $X^{\ast}(S'^{\uc})_{\R}$ and
such that $W_a$ is isomorphic to its affine Weyl group 
\[W_a \cong Q^{\vee}({}^{\mbfo}\Sigma)\rtimes W({}^{\mbfo}\Sigma).\]
%(but, it may happen that $Q^{\vee}({}^{\mbfo}\Sigma)\neq X_{\ast}(T'^{\uc})_I$ as lattices in $X_{\ast}(S'^{\uc})_\R=(X_{\ast}(T'^{\uc})_I)_{\R}$).
Also, the choice of the alcove $\mbfa$ containing $\mbfo$ determines a set $\mathrm{S}_{\mbfa}$ of simple affine roots on $\mcA(S',\mfk)$, and $\widetilde{W}_{K_{\mbfo}(\mfk)}\simeq W_0$ is the subgroup of $W_a$ generated by the subset ${}^{\mbfo}\Delta$ consisting of the simple affine roots whose corresponding affine hyperplanes pass through $\mbfo$ (thus, ${}^{\mbfo}\Delta$ is a set of simple roots for ${}^{\mbfo}\Sigma$).

Let $\underline{\mu'}_0$, $\underline{\mu''}_0$ denote the dominant representatives of $W_0\cdot\underline{\mu'}$, $W_0\cdot\underline{\mu''}\subset X_{\ast}(T')_I$, where $I=\Gamma_{\mfk}=\Gal(\overline{\mfk}/\mfk)$. Then, we claim that 
\[\underline{\mu''}_0\leq \underline{\mu'}_0.\]
for the dominance order on $Q^{\vee}({}^{\mbfo}\Sigma)=X_{\ast}(T'^{\uc})_I$ (determined by the choice of the alcove $\mbfa$); in particular, $\underline{\mu'}_0- \underline{\mu''}_0$ is a linear combination of positive coroots in ${}^{\mbfo}\Sigma$ with non-negative \emph{integer} coefficients \cite[Lem. 4.11]{Stembridge05}.

Indeed, in this set-up of the Coxeter group $W_a$ endowed with a set of generators $\mathrm{S}_{\mbfa}$, the claim follows from \cite{Stembridge05} (more precisely, Prop. 1.1, Prop. 1.5, Prop. 1.8), applied with the choice $\theta=\mbfo$, noting the following two facts: 
First, for $w=t^{\underline{\nu}}\in W_a$ with $\underline{\nu}\in X_{\ast}(T^{\uc})_I$, $w\theta$ is identified with $\underline{\nu}\in V$. Secondly, $W_0 t^{\underline{\mu''}} W_0\leq W_0 t^{\underline{\mu'}} W_0$ if and only if $w''\leq w'$, where $w''$ (resp. $w'$) is the (unique) element of minimal length in the coset $W_0 t^{\underline{\nu''}} W_0$ (resp. $W_0 t^{\underline{\nu'}} W_0$) with $\underline{\nu''}\in X_{\ast}(T^{\uc})_I$ being the component of $\underline{\mu''}\in X_{\ast}(T)_I(\subset \widetilde{W}=W_a\rtimes \Omega_{\mbfa})$ (resp. $\underline{\nu'}\in X_{\ast}(T^{\uc})_I$ being the component of $\underline{\mu'}\in X_{\ast}(T)_I$).

On the other hand, we claim that under our assumption on $G_{\Qp}$,
\[\underline{\mu''}\in \underline{W\cdot\mu'}\]
for the absolute Weyl group $W$ of $(G,T')$.
Indeed, as this statement concerns only the differences $w\cdot\underline{\mu'}_0-\underline{\mu''}_0$ ($w\in W$) which always lie in $X_{\ast}(T'^{\uc})_I$ and the root system ${}^{\mbfo}\Sigma$ which is determined by $G^{\uc}_{\Qp}$, 
one can assume that $G_{\Qp}^{\uc}=\Res_{F/\Qp}H$ for an absolutely simple, quasi-split, semi-simple group $H$ over a finite extension $F$ of $\Qp$ which splits over a tamely ramified extension of $F$. Then, since $G_{\Qpnr}^{\uc}=\prod_{\iota\in\Hom(F_0,\Qpnr)}\Res_{F_{\iota}^{\nr}/\Qpnr}H_{F_{\iota}^{\nr}}$ where $F_0$ is the maximal subfield of $F$ unramified over $\Qp$ and $F_{\iota}^{\nr}:=F\otimes_{F_0,\iota}\Qpnr$, and for a maximal torus $T^{\nr}$ of $H_{F_{\iota}^{\nr}}$, we have 
\[X_{\ast}(\Res_{F_{\iota}^{\nr}/\Qpnr}T^{\nr})_I=X_{\ast}(T^{\nr})_{\Gal(\bar{F}_{\iota}^{\nr}/F_{\iota}^{\nr})},\] we may reduce to $H_{F^{\nr}}$ for such $H$, where $F^{\nr}$ is the maximal unramified extension (in $\bar{F}$) of $F$, and thus further assume that $H$ is absolutely simple, residually split, and splits over a tamely ramified extension of $F$. We may exclude the split case which is trivial. Then, according to the list of such groups provided in the proof of Prop. \ref{prop:existence_of_elliptic_tori_in_special_parahorics}, we are left with the groups of type $B\operatorname{-}C_m$, $C\operatorname{-}BC_m$, $C\operatorname{-}B_m$. 

In the first two cases, one can readily check (cf. \cite{Tits79}) that each root $\beta$ in ${}^{\mbfo}\Sigma$ is also a relative root for $(H_{F^{\nr}},S')$ (thus lifts to a root $\tilde{\beta}\in X^{\ast}(T')$ for $(H_{F^{\nr}},T')$),
where $S'$ now refers to a maximal $F^{\nr}$-split torus in $H_{F^{\nr}}$ and $T'$ is its centralizer (a maximal torus of $H_{F^{\nr}}$).
Hence as $\mu'$ is a minuscule coweight for $(G,T')$, for each root $\beta$ in ${}^{\mbfo}\Sigma$, we have
\[|\langle\beta,\underline{\mu'}\rangle|=|\frac{1}{[K':\Qpnr]}\sum_{\tau\in\Gal(K'/\Qpnr)}\langle\tilde{\beta},\tau\mu'\rangle|\leq 1,\]
where $K'/\Qpnr$ is a finite Galois extension splitting $T'_{\Qpnr}$ (as element of $Q^{\vee}({}^{\mbfo}\Sigma)_{\Q}=X_{\ast}(S')_{\Q}$, one has $\underline{\mu'}=\frac{1}{[K'/\Qpnr]}\sum_{\tau\in\Gal(K'/\Qpnr)}\tau\mu'$). As $\langle\beta,\underline{\mu'}\rangle\in\Z$, we see that $\underline{\mu'}\in Q^{\vee}({}^{\mbfo}\Sigma)$ is minuscule for ${}^{\mbfo}\Sigma$, and thus $\underline{\mu''}_0\in W_0\cdot\underline{\mu'}_0$ by \cite[Lem. 2.3.3]{Kottwitz84b}, \cite[Lem. 2.2]{RR96}. 

In the remaining case ($C\operatorname{-}B_m$), it is not true any longer that the roots of ${}^{\mbfo}\Sigma$ are also relative roots for $(H_{F^{\nr}},S')$: see below. But, we claim that one still has $\underline{\mu''}\in \underline{W\cdot\mu'}$ (but, not necessarily $\underline{\mu''}_0\in W_0\cdot\underline{\mu'}_0$). In this case $H$ is the universal covering of the special orthogonal group $\SO$ attached to a non-degenerate orthogonal space of even dimension $2n+2$ over $F$ of Witt index $n$, as described in \cite[Example 1.16]{Tits79} whose notations we use: there exists a maximal split torus $S$ of $\SO$ with a basis $\{a_1,\cdots,a_n\}$ of $X^{\ast}(S)$ such that if $a_{-i}:=-a_i$, $a_{ij}:=a_i+a_j$, the relative roots of $H$ are
\[\Phi=\{a_{ij}\ |\ i,j\in I,\ j\neq\pm i\}\cup \{a_i\ |\ i\in I\},\] 
where $I=\{\pm1,\cdots,\pm n\}$. As we assume that $\SO$ is ramified, the reduced root system ${}^{\mbfo}\Sigma$ is of type $C_n$ with a set of simple roots $\{a_{1,-2},\cdots,a_{n-1,-n},2a_{n}\}$ (\textit{loc. cit.}).  Let $\{e_1,\cdots,e_n\}\subset X_{\ast}(S)$ be the dual basis of $\{a_1,\cdots,a_n\}$ and $T\subset \SO$ the centralizer of $S$ (maximal torus); there exists a basis of $X_{\ast}(T)$ consisting of $\{e_1,\cdots,e_n,e_{n+1}\}$ with the non-trivial element $\Gal(E/F)$ acting on $e_{n+1}$ by multiplication with $-1$, where $E$ is the splitting field of $H$ (which we assume to be a ramified extension of $F$). Also, we may assume $S'^{\uc}=S^{\uc}$, $T'^{\uc}=T^{\uc}$, where $S^{\uc}=\pi^{-1}(S)$, $T^{\uc}=\pi^{-1}(T)$ for the surjection $\pi:H\rightarrow \SO$. 
%Then, we see that as a lattice in $X_{\ast}(S^{\uc})_{\Q}=X_{\ast}(S)_{\Q}$ equipped with an inner product, the coroot lattice $Q^{\vee}({}^{\mbfo}\Sigma)$ equals $X_{\ast}(S)$ (a sub-lattice with the inherited inner product of $X_{\ast}(T)$ 
%which itself contains the coroot lattice $Q^{\vee}(\SO,T)=\langle e_i\pm e_{i+1}\ (1\leq i\leq n)\rangle=X_{\ast}(T^{\uc})$). Indeed, 
Then as a lattice in $X_{\ast}(S^{\uc})_{\Q}=(X_{\ast}(T^{\uc})_I)_{\Q}\ (I=\Gal(\bar{F}/F^{\nr}))$, $Q^{\vee}({}^{\mbfo}\Sigma)$ equals $X_{\ast}(T^{\uc})_I$, which, as a free $\Z$-module, is generated by $\{e_i\pm e_j, 1\leq i\neq j\leq n\}$ and (the image in $X_{\ast}(T^{\uc})_I$ of) $e_{n}-e_{n+1}$ (in the quotient $X_{\ast}(T^{\uc})_I$, $e_{i}-e_{n+1}$ and $e_{i}+e_{n+1}$ are the same):
for each $1\leq i\leq n-1$, $e_i-e_{i+1}$ equals the coroot corresponding to the root $a_{i,-i-1}$,  and  $e_{n}-e_{n+1}$ equals ``the shortest simple coroot'' corresponding to $2a_n$ (as element of $X_{\ast}(S)_{\Q}$, $e_{i}-e_{n+1}\in X_{\ast}(T^{\uc})_I$ is $e_{i}$ for each $1\leq i\leq n$). 
%With this identification, it is clear that the positive coroots of $Q^{\vee}({}^{\mbfo}\Sigma)$ are also positive with respect to $\Delta(\SO,T)$. Hence, we can regard the dominant element $\underline{\mu'}_0-\underline{\mu''}_0\in Q^{\vee}({}^{\mbfo}\Sigma)_+$ also as a dominant element in $Q^{\vee}(\SO,T)_+$. 
To prove the claim, we may replace $\underline{\mu'}$ and $\underline{\mu''}$ by their $W_0$-orbits so that
$\underline{\mu'}_0=\underline{\mu'}$ and $\underline{\mu''}_0=\underline{\mu''}$. If $\mu'-\mu''=\sum_{i=1}^{n}c_i(e_i-e_{i+1})+c_{n+1}(e_n+e_{n+1})\ (c_i\in\Z)$, its image in $Q^{\vee}({}^{\mbfo}\Sigma)=X_{\ast}(T^{\uc})_I$ is $\sum_{i=1}^{n-1}c_i(e_i-e_{i+1})+(c_n+c_{n+1})(e_n-e_{n+1})$, so $c_i\geq0$ for $1\leq i\leq n-1$ and $c_n+c_{n+1}\geq0$ by assumption. Then replacing $\mu''$ by $\nu:=\mu''+c_n((e_n-e_{n+1})-\iota(e_n-e_{n+1}))$ does not change $\underline{\mu''}$ and now $\mu'-\nu$ lies in the dominant semigroup of the coroot lattice $Q^{\vee}(\SO,T)$. Since $\mu'\in X_{\ast}(T)$ is a minuscule cocharacter of $(G,T)$, it follows again (\cite[Lem. 2.3.3]{Kottwitz84b}, \cite[Lem. 2.2]{RR96}) that $\underline{\mu''}=\underline{\nu}\in \underline{W\cdot\mu'}$.

Therefore, there exist $w\in W$ and $\mu_1\in \langle \tau x-x\ |\ \tau\in\Gal(\overline{\mfk}/\mfk), x\in X_{\ast}(T')\rangle$
such that 
\begin{equation} \label{eq:LR-Lemma5.11_Step_3}
\mu''=w\mu'\cdot\mu_1
\end{equation}
(multiplicative notation).

\item[(4)] So far, we have not used $T_1$ at all and in fact only used that $b\in M$. Now, we will use the condition that $b\in M(\mfk)$ is basic and $T_1$ is elliptic in $M$; as will be clear below, it is not even necessary that $b\in T_1(\mfk)$. Let us put $\nu_p:=[K:\Qp]\nu_b\in X_{\ast}(T_1)$.
For an (arbitrary, for a moment) cocharacter $\mu\in X_{\ast}(T_1)$, as $\Nm_{K/\Qp}\mu$ and $\nu_p$ are both $\Qp$-rational, the equation (\ref{eq:equality_on_the_kernel}) holds if and only if
\begin{equation} \label{eqn:defining_property_of_mu}
[K:\Qp]\langle\chi,\mu\rangle=\langle\chi,\nu_p\rangle
\end{equation}
for every $\Qp$-rational character $\chi$ of $T_1$, and for that matter, we may consider only the cocharaters $\chi$ lying in the submodule $mX^{\ast}(T_1)^{\Gal(\Qpb/\Qp)}\subset X^{\ast}(T_1)^{\Gal(\Qpb/\Qp)}$ for any fixed $m\in\N$. In the following, we take $m:=|\pi_0(T_1\cap M^{\der})|$.

Then, since $T_1$ is \emph{elliptic} in $M$, any $\chi\in mX^{\ast}(T_1)^{\Gal(\Qpb/\Qp)}$, vanishing on $T_1\cap M^{\der}$, can be regarded as a $\Qp$-rational character $\chi^{\ab}$ of $M^{\ab}=M/M^{\der}\cong T_1/(T_1\cap M^{\der})$, thus also as that of $M$ (via the canonical projection $p:M\rightarrow M^{\ab}$) such that $\langle\chi,\mu\rangle=\langle\chi^{\ab},p\circ\mu\rangle$
(the first pairing is defined for $X^{\ast}(T_1)\times X_{\ast}(T_1)$ and the second one for $X^{\ast}(M^{\ab})\times X_{\ast}(M^{\ab})$). For a cocharacter $\nu$ of $M$, we will often write $\langle\chi^{\ab},\nu\rangle$ for $\langle\chi^{\ab},p\circ\nu\rangle$. 

Now, as $T_1$, $T'$ are both maximal tori of $M$, there exists $g\in M(\overline{\mfk})$ with $T_1=gT'g^{-1}$. Set
\begin{equation}
\mu:=gw(\mu')g^{-1}\in X_{\ast}(T_1)\cap W\cdot\{\mu\}
\end{equation}
Then, as $\chi$ is $\Qp$-rational, it holds that $\langle\chi,g\mu_1 g^{-1}\rangle=\langle\chi^{\ab},p\circ g\mu_1 g^{-1}\rangle=\langle\chi^{\ab},p\circ\mu_1\rangle=0$, thus
\[\langle\chi,\mu\rangle=\langle\chi^{\ab},p\circ\mu\rangle=\langle\chi^{\ab},p\circ gw(\mu')g^{-1}\cdot g\mu_1g^{-1}\rangle=\langle\chi^{\ab},p\circ\mu''\rangle.\]

On the other hand, for any $\Qp$-rational character $\lambda$ of $M$, by definition of $\mu''$ we have
\[|\lambda(m^{-1}b\sigma(m))|=p^{-\langle\lambda,\mu''\rangle},\]
and by definition of $\nu_b$ \cite[(4.3)]{Kottwitz85} there exists $n\in\N$ such that $|\lambda(b\cdot\sigma(b)\cdots \sigma^{n-1}(b))|=p^{-\langle\lambda,n\nu_b\rangle}$. Hence, we find that with $l:=n[K:\Qp]$, 
\[ p^{-l\langle\lambda,\mu''\rangle} = |\lambda(m^{-1}b\sigma(m)\cdots\sigma^{-(l-1)}(m)\sigma^{l-1}(b)\sigma^l(m))|  =p^{-\langle\lambda,n\nu_p\rangle}. \]
Thus by substituting $\lambda=\chi^{\ab}$ (for $\chi\in mX^{\ast}(T_1)^{\Gal(\Qpb/\Qp)}$) and using that $\langle\chi^{\ab},\nu_p\rangle=\langle\chi,\nu_p\rangle$, one obtains the equality (\ref{eqn:defining_property_of_mu}).
\end{itemize}

This completes the proof. 
\end{proof}

%%%%%%%%%%%%%%%%%%%%
\begin{rem}
In our proof, we have a maximal $\Qp$-torus $T'=Z_{G_{\Qp}}(S')$. In the proof of \cite{LR87}, Lemma 5.11, there appears (on p. 177, line -5) a maximal torus of $M$ which is also denoted by $T'$. The role of their $T'$ is played by our $S'$, namely, being a maximal $\mfk$-split torus, is to provide apartments in $\mcB(G_{\Qp},\mfk)$ and $\mcB(M,\mfk)$ (i.e. whose underlying affine spaces are both $X_{\ast}(T')_{\R}$) which share the given (hyper)special point. Meanwhile, our $T'_{\mfk}$ is the centralizer of a maximal $\mfk$-split torus $S'_{\mfk}$ and enters into the proof as such, for example, via the Iwasawa and Cartan decompositions (cf. \cite[3.3.2]{Tits79}).
On the other hand, in the original proof where $G_{\Qp}$ is unramified, the unramified conditions (or words) show up for the simple reason that their $T'$ is an unramified $\Qp$-torus.
\end{rem}

The next lemma is our strengthening of Lemma 5.12 of \cite{LR87}. Its proof does not involve the level subgroup $\mbfK_p$.

%%%%%%%%%%%%%%%%%%%%
\begin{lem} \label{lem:LR-Lemma5.12}
Let $\phi$, $\psi_{T,\mu}$ be as in Lemma \ref{lem:LR-Lemma5.11}.
Then, there exists an admissible embedding of maximal torus $\Int g':T\hra G\ (g'\in G(\Qb))$ (with respect to the identity inner twisting $G_{\Qb}=G_{\Qb}$) such that 
\begin{itemize} 
\item[(i)] $\Int g'\circ\phi$ equals $\psi_{T',\mu_{h'}}$ on the kernel of $\fP$, for some $h'\in X_{\ast}(T')\cap X$.
\end{itemize}
Moreover, if $T_{\Ql}$ is elliptic in $G_{\Ql}$ for some prime $l\neq p$,
there exist $g'\in G(\Qb)$ and $h'\in X$ satisfying, in addition to (i), that 
\begin{itemize}
\item[(ii)] there exists $y\in G(\Qp)$ such that $(T'_{\Qp},\mu_{h'})=\Int y(T_{\Qp},\mu)$.
\end{itemize}
\end{lem}

The first statement, i.e. existence of a transfer of maximal torus $\Int g':T\hra G$ with the property (i) is the assertion of Lemma 5.12 of \cite{LR87} which we reproduce now, while the existence of such element with the additional property (ii) (under the given assumption on $T_{\Ql}$) is due to the author.

\begin{proof}
Let $T^{\uc}$ denote the inverse image of $T\cap G^{\der}$ under the isogeny $G^{\uc}\rightarrow G^{\der}$; it is a maximal torus of $G^{\uc}$. Choose $w\in N_{G^{\uc}}(T^{\uc})(\C)$ such that $\mu=w(\mu_h)$. We have the cocycle $\alpha^{\infty}\in Z^1(\Gal(\C/\R),G^{\uc}(\C))$ defined by 
\[\alpha^{\infty}_{\iota}=w\cdot\iota(w^{-1}).\] 
One readily checks that this has values in $T^{\uc}(\C)$. Indeed, according to \cite[Prop. 2.2]{Shelstad79}, the automorphism $\Int(w^{-1})$ of $T^{\uc}_{\C}$ is defined over $\R$, so $\Int(w^{-1})(\iota(t))=\iota(\Int(w^{-1})t)=\Int(\iota(w^{-1}))(\iota(t))$ for all $t\in T^{\uc}(\C)$, i.e. $\iota(w)w^{-1}\in Z_{G^{\uc}}(T^{\uc})(\C)=T^{\uc}(\C)$, and so is $w\iota(w^{-1})=\iota(\iota(w)w^{-1})$.
Let $\phi$, $\psi_{T,\mu}$ be as in Lemma \ref{lem:LR-Lemma5.11}. Then, according to Lemma 7.16 of \cite{Langlands83}, one can find a global cocycle $\alpha\in Z^1(\Q,T^{\uc})$ mapping to $\alpha^{\infty}\in  H^1(\Q_{\infty},T^{\uc})$. If furthermore $T_{\Ql}$ is elliptic in $G_{\Ql}$ for some prime $l\neq p$, 
we can choose $\alpha\in Z^1(\Q,T^{\uc})$ mapping to $\alpha^{\infty}\in  H^1(\Q_{\infty},T^{\uc})$ and having trivial image in $H^1(\Qp,T^{\uc})$, according to \cite[Lem. 4.1.2]{Lee16}, which we now recall with its proof. It is a variant of the original argument of \cite[Lem. 5.12]{LR87}.

%%%%%%%%%%%%%%%%%%%%
\begin{lem} \cite[Lem. 4.1.2]{Lee16} \label{lem:Lee14-lem.4.1.2}
Let $T$ be a maximal $\Q$-torus of $G$ which is elliptic at some finite place $l\neq p$.

(1) The natural map $(\pi_1(T^{\uc})_{\Gamma(l)})_{\mathrm{tors}} \rightarrow (\pi_1(T^{\uc})_{\Gamma})_{\mathrm{tors}}$
is surjective.

(2) The diagonal map $H^1(\Q,T^{\uc})\rightarrow H^1(\R,T^{\uc})\oplus H^1(\Qp,T^{\uc})$ is surjective.
\end{lem}

\begin{proof} 
(1) This map equals the composite:
\[(\pi_1(T^{\uc})_{\Gamma(l)})_{\mathrm{tors}}\hra \pi_1(T^{\uc})_{\Gamma(l)}\twoheadrightarrow \pi_1(T^{\uc})_{\Gamma}\twoheadrightarrow (\pi_1(T^{\uc})_{\Gamma})_{\mathrm{tors}},\]
where the last two maps are obviously surjective. 
Therefore, it is enough to show that $\pi_1(T^{\uc})_{\Gamma(l)}$ is a torsion group. 
But, as $T^{\uc}_{\Q_l}$ is anisotropic, $\widehat{T^{\uc}}^{\Gamma(l)}$ is a finite group, and so is $\pi_1(T^{\uc})_{\Gamma(l)}=X^{\ast}(\widehat{T^{\uc}}^{\Gamma(l)})=\Hom(\widehat{T^{\uc}}^{\Gamma(l)},\C^{\times})$.

(2) For every place $v$ of $\Q$, non-archimedean or not, there exists a canonical isomorphism \cite[(3.3.1)]{Kottwitz84a}
\[H^1(\Qv,T^{\uc})\isom \pi_0(\widehat{T^{\uc}}^{\Gamma_v})^D=\Hom(\pi_0(\widehat{T^{\uc}}^{\Gamma_v}),\Q/\Z)\cong X^{\ast}(\widehat{T^{\uc}}^{\Gamma_v})_{\mathrm{tors}}\cong(X_{\ast}(T^{\uc})_{\Gamma_v})_{\mathrm{tors}},\]
and a short exact sequence \cite[Prop.2.6]{Kottwitz86}
\[H^1(\Q,T^{\uc})\rightarrow H^1(\Q,T^{\uc}(\overline{\A}))=\oplus_v H^1(\Qv,T^{\uc})
\stackrel{\theta}{\rightarrow} \pi_0(Z(\widehat{T^{\uc}})^{\Gamma})^D=(\pi_1(T^{\uc})_{\Gamma})_{\mathrm{tors}},\]
where $\theta$ is the composite 
\[\oplus_v H^1(\Qv,T^{\uc})\isom \oplus_v \pi_0(Z(\widehat{T^{\uc}})^{\Gamma_v})^D\rightarrow \pi_0(Z(\widehat{T^{\uc}})^{\Gamma})^D\] (the second map is the direct sum of the maps considered in (1)). 
Let $(\gamma^{\infty},\gamma^p)\in H^1(\R,T^{\uc})\oplus H^1(\Qp,T^{\uc})$. By (1), there exists a class $\gamma^l\in H^1(\Q_l,T^{\uc})$ with $\sum_{v=l,\infty,p}\theta(\gamma^{v})=0$. Then, the element $(\beta^v)_v\in H^1(\Q,T^{\uc}(\overline{\A}))$ such that $\beta^{v}=\gamma^{v}$ for $v=l,\infty,p$ and $\beta^{v}=0$ for $v\neq l,\infty,p$ goes to zero in $ \pi_0(Z(\widehat{T^{\uc}})^{\Gamma})^D$. By exactness of the sequence, we find a class $\gamma$ in $H^1(\Q,T^{\uc})$ which maps to the class $(\beta^v)_v$. 
\end{proof}

Now, by changing $w$ (to another $w'\in N_G(T)(\C)$) if necessary, we may further assume that $\alpha^{\infty}$ is equal (as cocycles) to the restriction of $\alpha$ to $\Gal(\C/\R)$. Then, since the restriction map $H^1(\Q,G^{\uc})\rightarrow H^1(\Q_{\infty},G^{\uc})$ is injective (the Hasse principle), $\alpha$ becomes trivial as a cohomology class in $G^{\uc}(\Qb)$: 
\[\xymatrix{
H^1(\Q_{\infty},T^{\uc})\ar[r] & H^1(\Q_{\infty},G^{\uc}) & \alpha^{\infty}_{\iota}=w\iota(w^{-1}) \ar@{|->}[r] & 0 & \\
H^1(\Q,T^{\uc})\ar[r] \ar[u] & H^1(\Q,G^{\uc}) \ar@{^{(}->}[u] & \alpha \ar@{|->}[u]   \ar@{|->}[r] & \alpha'\ \ar@{^{(}->}[u] \ar@{=>}[r] & \alpha'=0} 
\]
In other words, there exists $u\in G^{\uc}(\Qb)$ such that
 \[\alpha_{\rho}=u^{-1}\rho(u),\]
for all $\rho\in\Gal(\Qb/\Q)$. It then follows that $\Int  u: (T)_{\Qb}\rightarrow G_{\Qb}$ is an admissible embedding of maximal torus with respect to the identity inner twisting of $G_{\Qb}$ (i.e. the homomorphism $\Int u$ and thus the torus $T'=\Int(T)$ as well are defined over $\Q$). We also note that since the restriction of $\Int u$ to $Z(G)$ is the identity, $T'_{\R}$ is also elliptic in $G_{\R}$. Then, for $\phi':=\Int  u\circ\phi$ and $\mu':=\Int u(\mu)\in X_{\ast}(T')$, we have $\psi_{T',\mu'}=\Int u\circ \psi_{T,\mu_h}=\phi'$ on the kernel $P$, by functoriality of the construction of $\psi_{T,\mu}$ with respect to the pair $(T,\mu)$ \cite[Satz 2.3]{LR87}. Furthermore, since 
\[u^{-1}\iota(u)=\alpha_{\iota}=\alpha^{\infty}_{\iota}=w\iota(w^{-1})\] 
for $\iota\in\Gal(\C/\R)$, one has that $uw\in G^{\uc}(\R)$ and $\mu'=\Int  u(\mu)=\mu_{h'}$ for $h':=\Int  (uw)(h)\in X$. This establishes the first claim (existence of a transfer of maximal torus $\Int g':T\hra G$ with the property (i)). 

Next, when we assume that $T_{\Ql}$ is elliptic for some $l\neq p$, by Lemma \ref{lem:Lee14-lem.4.1.2}, 
we may choose $\alpha\in Z^1(\Q,T^{\uc})$ such that it maps to $\alpha^{\infty}\in H^1(\Q_{\infty},T^{\uc})$ and to zero in $H^1(\Qp,T^{\uc})$. Then, by repeating the argument above, we find $u\in G^{\uc}(\Qb)$ such that $\alpha_{\rho}=u^{-1}\rho(u)$ for all $\rho\in\Gal(\Qb/\Q)$.
As $\alpha|_{\Gal/\Qpb/\Qp)}$ is trivial, there exists $x\in T(\Qpb)$ such that $x\rho(x^{-1})=\alpha_{\rho}=u^{-1}\rho(u)$ for all $\rho\in \Gal(\Qpb/\Qp)$, in other words. $y:=ux\in G(\Qp)$. But, the homomorphism $\Int u:T_{\Qpb}\rightarrow T'_{\Qpb}$ also equals $\Int u=\Int y$; in particular, it is defined over $\Qp$. This proves (ii) and finishes the proof of Lemma \ref{lem:LR-Lemma5.12}. 
\end{proof}

%%%%%%%%%%%%%%%%%%%%
%%%%%%%%%%%%%%%%%%%%
\subsubsection{Proof of Proposition \ref{prop:equivalence_to_special_adimssible_morphism}.}

We proceed in parallel with the arguments on p.181, line 1-19 of \cite{LR87}. By Lemma \ref{lem:LR-Lemma5.11} and Lemma \ref{lem:LR-Lemma5.12}, after some transfer of tori (always with respect to the identity inner twist $\mathrm{Id}_G$) whose restriction to the torus becomes a conjugation by an element of $G(\Qp)$ when $T_{\Ql}$ is elliptic in $G_{\Ql}$ for some $l\neq p$, we may assume that $\phi$ coincide with $i\circ\psi_{T,\mu_h}$ on the kernel for some $h\in X$ factoring through $T_{\R}$. If further $H^1(\Qp,T)=0$, the cocharacter $\mu$ in Lemma \ref{lem:LR-Lemma5.11} satisfies $[b]_{T_{\Qp}}=\underline{\mu}$ under the isomorphism $\kappa_{T_{\Qp}}:B(T_{\Qp})\isom X_{\ast}(T_{\Qp})_{\Gal(\Qpb/\Qp)}$, so we may assume that $\phi(p)\circ\zeta_p$ is conjugate to $i\circ\psi_{T,\mu_h}(p)\circ\zeta_p$ under $T(\Qpb)$.
Given this, one readily checks that the map $\Gal(\Qb/\Q)\rightarrow T(\Qb):\rho\mapsto b_{\rho}$ defined by
\[\phi(q_{\rho})=b_{\rho}\cdot i\circ\psi_{T,\mu_h}(q_{\rho})\]
is a cocycle, where $\rho\mapsto q_{\rho}$ is the chosen section to the projection $\fP\rightarrow\Gal(\Qb/\Q)$ (Remark \ref{rem:comments_on_zeta_v}). We claim that its image in $H^1(\Q,G)$ under the natural map $H^1(\Q,T)\rightarrow H^1(\Q,G)$ is trivial.  As before, a diagram helps to visualize the proof:
\[\xymatrix{
& H^1(\Q_{\infty},G') & & & 0 &  \\
H^1(\Q_{\infty},T)\ar@{^{(}->}[ur] \ar[r]  & H^1(\Q_{\infty},G) & f(H^1(\Q_{\infty},G^{\uc})) \ar@{_{(}->}[l] & 0=b^{\infty}_{\rho} \ar@{^{(}->}[ur] \ar@{|->}[rr] & & 0 \\
H^1(\Q,T)\ar[r] \ar[u] & H^1(\Q,G) \ar[u] & f(H^1(\Q,G^{\uc})) \ar@{_{(}->}[l]  \ar@{^{(}->}[u] & b_{\rho} \ar@{|->}[u]  \ar@{|->}[rr]  & & b_{\rho}'=0\  \ar@{^{(}->}[u] } 
\]
Here, $G'$ is the inner twist of $G_{\R}$ by $\phi(\infty)\circ\zeta_{\infty}$ (i.e. by the cocycle $\iota\mapsto h_{\iota}\in Z^1(\Q_{\infty},G^{\ad}_{\R})$ with $\phi(\infty)\circ\zeta_{\infty}(w(\iota))=h_{\iota}\iota$) and 
$f$ is the canonical homomorphism from the universal covering $G^{\uc}$ of $G^{\der}$ to $G$ (or the map induced on the cohomology sets). 
The restriction of $[b_{\rho}]\in H^1(\Q,T)$ to $\R=\Q_{\infty}$ is trivial, since it maps to zero in $H^1(\Q_{\infty},G')$ and that map is injective \cite[Lem. 5.14]{LR87}. According to Lemma \ref{lem:isom_of_monoidal_functors_into_croseed_modules} below, condition (1) of Def. \ref{defn:admissible_morphism} implies that the image $b_{\rho}'$ of $b_{\rho}$ under the map $H^1(\Q,T) \rightarrow H^1(\Q,G)$ lies in the image of $H^1(\Q,G^{\uc})$ in $H^1(\Q,G)$.
But, the Hasse principle holds for such image according to \cite[Thm.5.12]{Borovoi98},%%
\footnote{This result of Borovoi generalizes Lemma 5.13 of \cite{LR87} (attributed to Deligne) which concerns the case when $G^{\der}=G^{\uc}$. As remarked there \cite[p.180, line5]{LR87}, this is also the only part in the proof of Satz 5.3 of \textit{loc. cit.} that uses the assumption that $G^{\der}=G^{\uc}$. Hence, thanks to that result of Borovoi and the generalized definition of admissible morphisms (in terms of strict monoidal category $\fG_{G/\tilde{G}}$), in Prop 4.1.5 here, thus Thm. 4.1.3 as well, that assumption ($G^{\der}=G^{\uc}$) is unnecessary. \label{ftn:sc-assumption1} }
so we deduce that $b_{\rho}'\in H^1(\Q,G)$ is zero. 
If \[b_{\rho}=v\rho(v^{-1}),\quad v\in G^{\uc}(\Qb),\]
then, $\Int v^{-1}:T_{\Qb}\hra G_{\Qb}$ is an admissible embedding of maximal torus (with respect to the identity twisting $G_{\Qb}=G_{\Qb}$), i.e. the image $T':=v^{-1}Tv$ and the isomorphism $\Int v^{-1}:T_{\Qb}\hra T'_{\Qb}$ are all defined over $\Q$. One has to check that $\mu'$ is $\mu_{h'}$ for some $h'\in X$. This can be seen as follows. Since the cohomology class $[b_{\sigma}]\in H^1(\Q_{\infty},T)$ is trivial, there exists $t_{\infty}\in T(\C)$ such that 
$t_{\infty}^{-1}\iota(t_{\infty})=b_{\iota}=v\iota(v^{-1})$, which implies that
$t_{\infty}v\in G(\R)$. Then, 
\[\mu'=v^{-1}\mu_h v=(t_{\infty}v)^{-1}\cdot\mu_h\cdot (t_{\infty}v)=\mu_{h'}\] 
for $h':=(t_{\infty}v)^{-1}\cdot h\cdot (t_{\infty}v)\in X$. 

Finally, when $H^1(\Qp,T)=0$, 
there is $x\in T(\Qpb)$ such that $i\circ\psi_{T,\mu_h}(p)\circ\zeta_p=x(\phi(p)\circ\zeta_p)x^{-1}$ (as $\Qpb/\Qp$-Galois gerb morphisms $\fG_p\rightarrow\fG_{T_{\Qp}}$). But, as $\phi$ and $i\circ\psi_{T,\mu_h}$ are the same on the kernel $P$, the two $\Qpb/\Qp$-Galois gerb morphisms 
\[i\circ\psi_{T,\mu_h}(p),\ x\phi(p)x^{-1}\ :\fP(p)\rightarrow \fG_{T_{\Qp}}\]
agree on the kernel $P_{\Qp}(\Qpb)$. It follows that $i\circ\psi_{T,\mu_h}(p)$ and $x\phi(p)x^{-1}$ are equal on the whole $\fP(p)$. In other words, the restriction of $b_{\rho}$ to $\Gal(\Qpb/\Qp)$ is zero: 
\[x^{-1}\rho(x)=b_{\rho}=v\rho(v^{-1})\] 
for all $\rho\in \Gal(\Qpb/\Qp)$, and $xv\in G(\Qp)$. But, the homomorphism $\Int v^{-1}:T_{\Qpb}\rightarrow T'_{\Qpb}$ also equals $\Int v^{-1}=\Int (xv)^{-1}$. 
It follows from this and the discussion in the beginning of the proof that if the initial torus $T$ satisfies that $T_{\Ql}\subset G_{\Ql}$ is elliptic at some $l\neq p$, we may find a transfer of maximal torus $\Int g':T\hra G$ such that $\Int g'\circ\phi$ is special admissible and that $\Int g'|_{T_{\Qp}}=\Int y|_{T_{\Qpb}}$ for some $y\in G_{\Qp}$.
This completes the proof of Proposition \ref{prop:equivalence_to_special_adimssible_morphism}. $\square$

%%%%%%%%%%%%%%%%%%%%
\begin{lem} \label{lem:isom_of_monoidal_functors_into_croseed_modules}
Let $I$ be a (not necessarily connected) $\Q$-subgroup of $G$ containing $Z(G)$.
Let $\phi,\psi:\fP\rightarrow \fG_I$ be Galois gerb morphisms into the neutral Galois gerb of $I$ such that $\phi^{\Delta}=\psi^{\Delta}$ and maps into $Z(I)$; thus, the cochain $\tau\in\Gal(\Qb/\Q)\mapsto a_{\tau}\in I(\Qb)$ defined by $\psi(q_{\tau})=a_{\tau}\phi(q_{\tau})$ becomes a coccyle in $Z^1(\Q,I_{\phi})$, where $I_{\phi}$ is the twist of $I$ via $\phi$.

Then, the induced morphisms $\phi_{\widetilde{\ab}}, \psi_{\widetilde{\ab}}:\fP\rightarrow \fG_{G/\tilde{G}}=\fG_G/\tilde{G}$ are conjugate-isomorphic if and only if the cohomology class $[a_{\tau}]\in H^1(\Q,I_{\phi})$ lies in the image of the natural map $H^1(\Q,\tilde{I}_{\phi})\rightarrow H^1(\Q,I_{\phi})$, where $\tilde{I}_{\phi}$ is the twist of $\tilde{I}:=\rho^{-1}(I)$ via $\phi$ ($\rho:\tilde{G}\rightarrow G$ being the canonical morphism).%% 
\footnote{Here, we use the redundant notation $\tilde{G}=G^{\uc}$.}

If $I$ is a maximal $\Q$-torus of $G$ (so that $I_{\phi}=I$), the image of $[a_{\tau}]$ in $H^1(\Q,G)$ lies in the image of the canonical map $H^1(\Q,\tilde{G})\rightarrow H^1(\Q,G)$.
\end{lem}

%As will be clear from the proof, $I$ does not need to be connected.

\begin{proof}
For $g=z\cdot \tilde{g}\ \in G(\Qb)$ with $z\in Z(G)(\Qb)$ and $\tilde{g}\in \tilde{G}(\Qb)$, $(\Int(g)\circ\psi)_{\widetilde{\ab}}$ and $\phi_{\widetilde{\ab}}$ are isomorphic if and only if $(\phi_1)_{\widetilde{\ab}}$ and $\phi_{\widetilde{\ab}}$ are so, where $\phi_1:=\Int(z)\circ\psi$. 
In the latter case, by definition, there exists a family of elements $\{\tilde{h}_x\in  \mathrm{Mor}(\phi(x),\phi_1(x))=\tilde{G}(\Qb)\}_{x}$ indexed by $x\in\fP$, such that 
\[\phi_1(x)=\rho(\tilde{h}_{x})\phi(x),\] 
namely when we define $a_x\in I_{\phi}(\Qb)=I(\Qb)$ by $\psi(x)=a_x\phi(x)$ for $x\in\fP$, one has 
\begin{equation} \label{eq:equality_in_cohomology_of_crossed_modules_I}
a_{x}\cdot z x(z^{-1})=\rho(\tilde{h}_{x}). 
\end{equation}
In particular, we have $\tilde{h}_{x}\in \rho^{-1}(I)$ (by our assumption $Z(G)\subset I$).
Here, $x\in \fP$ acts on $z\in Z(G)(\Qb)$ via the projection $\pi:\fP\rightarrow \Gal(\Qb/\Q)$; as $\phi^{\Delta}=\psi^{\Delta}$, $x\mapsto a_x$ factors through the projection $\pi:\fP\rightarrow \Gal(\Qb/\Q)$ (i.e. $a_{x}=a_{y}$ when $\pi(x)=\pi(y)$) and induces the cocycle $a_{\tau}\in Z^1(\Q,I_{\phi})$ in the statement. We claim that in this situation, the cochain $x\mapsto \tilde{h}_{x}$ on the group $\fP$ valued in $\tilde{G}(\Qb)$ is a cocycle with respect to the action of $\fP$ on $\tilde{G}(\Qb)$ defined by conjugation via $\phi$, i.e. with respect to the action:
\[\tilde{h}\mapsto {}^{\phi(x)}\tilde{h}:=\Int(g_x)(\sigma_x(\tilde{h})),\] 
where $\phi(x)=g_x \sigma_x\in I(\Qb)\rtimes\Gal(\Qb/\Q)$.
This follows from the fact that the faimly $\{\tilde{h}_x\in\tilde{G}(\Qb)\}_{x\in\fP}$ is an isomorphism between strict monoidal functors. Indeed, for every $x,y\in\fP$, there are two morphisms from $\phi(xy)$ to $\phi_1(x)\phi_1(y)$ arising from the famly $\{\tilde{h}_x\}$: one is $\tilde{h}_{xy}$ (via $\phi_1(x)\phi_1(y)=\phi_1(xy)$) and the other one is $\tilde{h}_x\cdot {}^{\phi(x)}\tilde{h}_y$:
\[ \phi_1(x)\phi_1(y)=\tilde{h}_x\phi(x)\tilde{h}_y\phi(y)=\tilde{h}_x{}^{\phi(x)}\tilde{h}_y\phi(x)\phi(y)=\tilde{h}_x{}^{\phi(x)}\tilde{h}_y\phi(xy). \]
That the family $\{\tilde{h}_x\}$ is an isomorphism between strict monoidal functors means the equality of these two morphisms: 
\begin{equation} \label{eq:coycle_on_fP_valued_in_Iuc}
\tilde{h}_{xy}=\tilde{h}_x\cdot {}^{\phi(x)}\tilde{h}_y.
\end{equation}
In particular, as $\im(\phi^{\Delta})\subset Z(I)(\Qb)$ and $\rho(\tilde{h}_y)\in I(\Qb)$, the restriction of $\tilde{h}_x$ to the torus $\fP^{\Delta}=P(\Qb)$ is an algebraic homomorphism $P(\Qb)\rightarrow \rho^{-1}(I)(\Qb)$, which in view of (\ref{eq:equality_in_cohomology_of_crossed_modules_I}) further factors through $\ker\rho$, thus is trivial. Hence, it follows that the cochain $\tau\mapsto \tilde{h}_{\tau}:=\tilde{h}_{q_{\tau}}$ is a cocycle in $Z^1(\Q,\tilde{I}_{\phi})$. 
Then, the equation (\ref{eq:equality_in_cohomology_of_crossed_modules_I}) in turn shows that the cohomology class $[a_{\tau}]\in H^1(\Q,I_{\phi})$ equals the image of $[\tilde{h}_{\tau}]\in H^1(\Q,\tilde{I}_{\phi})$ in $H^1(\Q,I_{\phi})$.

Conversely, suppose that $[a_{\tau}]\in H^1(\Q,I_{\phi})$ lies in the image of the canonical map $H^1(\Q,\tilde{I}_{\phi})\rightarrow H^1(\Q,I_{\phi})$. Using the decomposition $I_{\phi}=Z(G)\cdot \rho(\tilde{I}_{\phi})$ (which holds as $Z(G)\subset I_{\phi}$), we may find $z\in Z(G)(\Qb)$ and a normalized cocycle $\tilde{h}_{\tau}\in Z^1(\Q,\tilde{I}_{\phi})$ such that $a_{\tau}\cdot z\tau(z^{-1})=\rho(\tilde{h}_{\tau})$ holds for all $\tau\in\Gal(\Qb/\Q)$. If we extend $a_{\tau}$, $\tilde{h}_{\tau}$ to cochains on $\fP$ by setting $a_x:=a_{\pi(x)}$, $\tilde{h}_{x}:=\tilde{h}_{\pi(x)}$ for $x\in \fP$ (which implies the relations (\ref{eq:coycle_on_fP_valued_in_Iuc}) for $x\in P(\Qb)$, $y=q_{\tau}$), one easily checks that $a_x$ and $\tilde{h}_x$ are cocylces in $Z^1(\fP,I_{\phi})$ and $Z^1(\fP,\tilde{I}_{\phi})$, respectively. Then, the discussion above implies that the relation (\ref{eq:equality_in_cohomology_of_crossed_modules_I}), i.e. $\Int(z)\circ\psi(x)=\rho(\tilde{h}_x)\phi(x)$ holds for all $x\in\fP$, which says that $(\Int(z)\circ\psi)_{\widetilde{\ab}}$ and $\phi_{\widetilde{\ab}}$ are isomorphic.

Next, if $I$ is a maximal $\Q$-torus, we have $I_{\phi}=I$, and the conjugation action $\tilde{h}\mapsto {}^{\phi(x)}\tilde{h}$ of $x\in\fP$ on $\tilde{I}(\Qb)$ becomes the usual Galois action via $\pi$. Namely, now $\tau\mapsto \tilde{h}_{\tau}:=\tilde{h}_{q_{\tau}}$ is a cocycle in $Z^1(\Q,\tilde{I})$. With these facts, the relation (\ref{eq:equality_in_cohomology_of_crossed_modules_I}) gives the conclusion as before.
\end{proof}

%%%%%%%%%%%%%%%%%%%%
\begin{lem} \label{lem:equality_restrictions_to_kernels_imply_conjugacy}
Let $T$ be a $\Qp$-torus, and for $i=1,2$, $\theta_i:\fG_p\rightarrow\fG_T$ an unramified morphism of $\Qpb/\Qp$-Galois gerbs; let $\theta_i^{\nr}:\fD\rightarrow\fG_{T_{\Qp}}^{\nr}$ be a morphism of $\Qpnr/\Qp$-Galois gerbs such that $\theta_i$ is the inflation $\overline{\theta_i^{\nr}}$ of $\theta_i^{\nr}$. 
If $\mathrm{cls}(\theta_1^{\nr})=\mathrm{cls}(\theta_2^{\nr})$ in $B(T_{\Qp})$, then $\theta_1$ and $\theta_2$ are conjugate under $T(\Qpb)$.
\end{lem}

\begin{proof}
To ease the notations, we continue to use $\theta_i$ for such $\theta_i^{\nr}$. The condition implies that the restrictions $\theta_i^{\Delta}$ of $\theta_i$ to the kernel $\mathbb{D}$ are equal, hence the two maps $\theta_i\ (i=1,2)$ are conjugate under $T(\mfk)$.
We will show that there exists $x_p\in T(\Qpnr)$ with  $\theta_2=\Int(x_p)\circ\theta_1$, i.e. such that 
\[\theta_2(s_{\tau})=x_p\theta_1(s_{\tau})x_p^{-1}\]
for all $\tau\in\Gal(\Qpnr/\Qp)$. The map $\tau\mapsto b_{\tau}:\Gal(\Qpnr/\Qp)\rightarrow T(\Qpnr)$ defined by 
\[\theta_2(s_{\tau})=b_{\tau}\theta_1(s_{\tau})\] 
is a cocycle in $Z^1(\Gal(\Qpnr/\Qp),T(\Qpnr))$ with $b_{\sigma}=b_2b_1^{-1}=t_p\sigma(t_p^{-1})$,
where $t_p\in T(\mfk)$ is  such that $b_2=t_pb_1\sigma(t_p^{-1})$.
Let $\langle\sigma\rangle$ be the infinite cyclic group $\langle\sigma\rangle$ generated by $\sigma$ (endowed with discrete topology). 
The surjections $\langle\sigma\rangle \twoheadrightarrow \Gal(L_n/\Qp)\ (n\in\N)$ combined with the inclusions $T(L_n)\hra T(\mfk)$ induce, via inflations, an injective map
\[H^1(\Gal(\Qpnr/\Qp),T(\Qpnr)) \hookrightarrow B(T_{\Qp})=H^1(\langle\sigma\rangle,T(\mfk)). \]
As the image of our cohomology class $[b_{\tau}]\in H^1(\Gal(\Qpnr/\Qp),T(\Qpnr))$ under this map is $[b_{\sigma}]=[1]$, the claim follows. 
\end{proof}

\textsc{Proof of Theorem \ref{thm:LR-Satz5.3}.} 
(1) We follow the original proof of Satz 5.3, as explained after the statement of Theorem \ref{thm:LR-Satz5.3}.%%
\footnote{If one wants, by using \cite[Lem. 3.7.7]{Kisin17}, one can also reduce the general case to the case where $G^{\der}=G^{\uc}$ (cf. proof of Thm. 3.7.8 of \textit{loc. cit.}). However, we already established all the facts/lemmas required for our proof to work without that assumption.}
The first step is to replace given $\phi$ by a conjugate $\phi_0=\Int g_0\circ \phi\ (g_0\in G(\Qb))$ of it whose restriction to the kernel $\phi_0^{\Delta}:P_{\Qb}\rightarrow G_{\Qb}$ is defined over $\Q$ (which amounts to that $\phi_0(\delta_n)\in G(\Q)$ for all sufficiently large $n\in\N$, \cite[Lem. 5.5]{LR87}). This is Lemma 5.4 of \cite{LR87}. This lemma is a statement just concerned with the restriction $\phi^{\Delta}$, whose proof only requires that $G_{\Qp}$ is \emph{quasi-split} and does not use the level subgroup at all.

The second step is to find a conjugate $\phi_1=\Int g_1\circ\phi_0$ of $\phi_0$ (produced in the first step) that factors through $\fG_{T_1}$ for some maximal $\Q$-torus $T_1$ (elliptic over $\R$, as usual). As discussed before (after statement of Prop. \ref{prop:existence_of_admissible_morphism_factoring_thru_given_maximal_torus}), this is shown on p. 176, from line 1 to -5 of \textit{loc. cit.}, and the arguments given there again do not make any use of the level (hyperspecial) subgroup and thus carries over to our situation. The basic idea is, in view of Lemma \ref{lem:criterion_for_admissible_morphism_to_land_in_torus}, to find a maximal torus $T_1$ of $I=Z_G(\phi_0(\delta_n))$ that can transfer to $I_{\phi}$. (An argument in similar style appears in the proof of Prop. \ref{prop:existence_of_admissible_morphism_factoring_thru_given_maximal_torus}). 

The final step is to find a conjugate $\phi:\fP\rightarrow\fG_{T}$ of $\phi_1:\fP\rightarrow\fG_{T_1}$ which becomes a special admissible morphism $i\circ\psi_{T,\mu_h}$ (for some special Shimura sub-datum $(T,h)$ and the canonical morphism $i:\fG_T\rightarrow\fG_G$ defined by the inclusion $i:T\hra G$). This is accomplished by successive admissible embeddings of maximal tori. It begins with showing existence of $\mu_1\in X_{\ast}(T_1)$ lying in the conjugacy class $\{\mu_h\}$ such that $\phi_1:\fP\rightarrow\fG_{T_1}$ coincides with $\psi_{T_1,\mu_1}$ on the \emph{kernel} of $\fP$. In the original proof, this is done in Lemma 5.11 which is also the only place in the proof of Satz 5.3 where the level subgroup is involved in an explicit manner (through non-emptiness of the set $X_p(\phi)\simeq X(\{\mu_X\},b)_{\mbfK_p}$). But, it continues to hold, as Lemma \ref{lem:LR-Lemma5.11} here, for general parahoric subgroup $\mbfK_p$ with the general definition of $X(\{\mu_X\},b)_{\mbfK_p}$. 
After this, one performs two admissible embeddings. First, we need to find an admissible embedding of maximal torus $\Int g_2: T_1\hra G$ such that $\Int g_2\circ\phi_1$ equals a special admissible morphism $\psi_{T_2,\mu_{h_2}}$ again just on the \emph{kernel} of $\fP$ (here $T_2=\Int g_2(T_1)$ and $(T_2,h_2)$ is a special Shimura sub-datum); note the difference from the previous step where $\mu_1$ did not need to be $\mu_h$ for some $h\in X$. In \textit{loc. cit.}, this is shown in Lemma 5.12 whose argument we adapted to prove Lemma \ref{lem:LR-Lemma5.12} here (which also refines the original lemma a bit). Let $\phi_2:=\Int g_2\circ\phi_1:\fP\rightarrow\fG_{T_2}$ be the admissible morphism just obtained. Then, one looks for the last admissible embedding of maximal torus $\Int(g_3):T_2\hra G$ making finally $\Int(g_3)\circ\phi_2:\fP\rightarrow\fG_{\Int g_3(T_2)}$ special admissible. In \textit{loc. cit.}, this is carried out from after Lemma 5.12 to the rest of the proof of Satz 5.3; this part of the argument was adapted to prove Prop. \ref{prop:equivalence_to_special_adimssible_morphism} here. Now, we see that
\[(T,\phi):=(\Int g_3(T_2),\Int(g_3)\circ\phi_2)=(\Int (\prod_{i=2}^3g_i)(T_1),\Int(\prod_{i=0}^3g_i)\circ\phi)\]
is a special admissible morphism which is a conjugate of $\phi$.

(2) This is proved by the same argument from (1) with applying Prop. \ref{prop:existence_of_admissible_morphism_factoring_thru_given_maximal_torus} in the second step and then Prop. \ref{prop:equivalence_to_special_adimssible_morphism} in the third step. Note that as explained before, the three properties (i) - (iii) continue to hold under any transfer of maximal torus.
$\square$

%%%%%%%%%%%%%%%%%%%%
\begin{lem} \label{lem:LR-Lemma5.23}
Retain the assumptions of Theorem \ref{thm:LR-Satz5.3}. 
For any pair $(\phi,\epsilon)$ consisting of an admissible morphism $\phi$ and $\epsilon\in I_{\phi}(\Q)$, there exists an equivalent pair $(\phi',\epsilon')=\Int(g)(\phi,\epsilon)\ (g\in G(\Qb))$ and a special Shimura sub-datum $(T',h')$, such that $(\phi',\epsilon')$ is nested in $(T',h')$ (\autoref{subsubsec:K-triple_attached_to_adm.pair}).
\end{lem}
 
In the original setting of \cite{LR87} (i.e. $\mbfK_p$ is hyperspeiclal and $G^{\der}=G^{\uc}$), this is their Lemma 5.23, where however Langlands and Rapoport assume that $(\phi,\epsilon)$ is admissible.
As we will see, this assumption is not necessary and their proof works in our situation without essential change. Most importantly, the level subgroup $\mbfK_p$ enters the proof only through Lemma 5.11 of \textit{loc. cit.} which is generalized by our Lemma \ref{lem:LR-Lemma5.11}. 

\begin{proof} 
By Lemma \ref{lem:criterion_for_admissible_morphism_to_land_in_torus}, it suffices to show that for any maximal $\Q$-torus $T$ of $I_{\phi}$ containing $\epsilon$ (i.e. $\epsilon\in T(\Q)\subset I_{\phi}(\Q)$), there exists $g\in G(\Qb)$ such that the composite map 
\[ G_{\Qb} \stackrel{\Int g}{\leftarrow} G_{\Qb}\supset I_{\Qb} \stackrel{\psi^{-1}}{\leftarrow} (I_{\phi})_{\Qb} \hookleftarrow T_{\Qb}\] 
is defined over $\Q$, where $\psi:I_{\Qb}\isom (I_{\phi})_{\Qb}$ is the inner twisting (\ref{eq:inner-twisting_by_phi}): 
Indeed, then $T_1:=\Int g(T)$ is a $\Q$-subgroup of $Z_G(\Int g(\epsilon))\subset G$, and $\phi_1:=\Int g\circ\phi$ maps into $\fG_{T_1}$ by Lemma \ref{lem:criterion_for_admissible_morphism_to_land_in_torus} (applied to $(g,T_1)$ for the role of $(a,T)$), and so the admissible pair $(\Int g\circ\phi,\Int g(\epsilon))$ is well-located in $\Int g(T)$. Note that $\Int g(T)_{\R}$ is elliptic in $G_{\R}$ since $Z(G)$ is a $\Q$-subgroup of $I_{\phi}$ and $(I_{\phi}/Z(G))_{\R}$ is anisotropic \cite[Lem. 5.1]{LR87}. Therefore, we can apply Prop. \ref{prop:equivalence_to_special_adimssible_morphism} to $(\Int g\circ\phi, \Int g(T))$ and obtain a desired pair $(\phi'=\psi_{T',\mu_{h'}},T')$ by an admissible embedding; as $\epsilon\in \Int g(T)(\Q)$, the image of $\epsilon$ under that admissible embedding belongs to $T'(\Q)$. 

To find $g\in G(\Qb)$ with the required property, we choose an element $\epsilon_1\in I_{\phi}(\Q)$ whose centralizer in $G$ is $T$ (i.e. a strongly regular semi-simple element of $G(\Qb)$ lying in $T(\Q)\subset I_{\phi}(\Q)$, so its centralizer in $I_{\Qb}=(I_{\phi})_{\Qb}$ is also $T$). The conjugacy class of $\epsilon_1$ in $I(\Qb)$ is rational, as $I$ is an inner form of $I_{\phi}$. When one chooses an inner twist $\psi_0: G_{\Qb}\rightarrow G^{\ast}_{\Qb}$ with $G^{\ast}$ being quasi-split, by the same reason, the conjugacy class of $x:=\psi_0(\epsilon_1)$ in $G^{\ast}(\Qb)$ is also rational. Then, since the centralizer of $x$ is connected, such conjugacy class contains a rational element by \cite{Kottwitz82}, Thm. 4.7 (2) (the obstruction to the existence of such a rational element lives in $H^2(\Q,C_x)$, where $C_x$ is the group of connected components of $Z_{G^{\ast}}(x)$, cf. Lemma 4.5 of \textit{loc. cit.}).%% 
\footnote{In the original argument (i.e. proof of Lemma 5.23 of \cite{LR87}), the authors simply appeal to a well-known theorem of Steinberg (or its extension by Kottwitz \cite[Thm.4.4]{Kottwitz82}), since they work under the assumption that $I^{\der}=I^{\uc}$. However, as we have just seen, this is unnecessary, and this allows us to remove that assumption in this lemma, because, as we now see, all other parts of the proof (especially, Prop. 4.1.5) do not require that assumption. \label{ftn:sc-assumption2} }
Thus, we can find an inner twist $\psi_1: G_{\Qb}\rightarrow G^{\ast}_{\Qb}$ such that $\epsilon_1^{\ast}:=\psi_1(\epsilon_1)\in G^{\ast}(\Q)$; let $T^{\ast}$ be the centralizer of $\epsilon_1^{\ast}$ in $G^{\ast}$ (a maximal $\Q$-torus of $G^{\ast}$).

Now, by the proof of \cite[Lem. 5.23]{LR87} (more precisely by the argument in the last paragraph on p.190, where the only necessary condition is that $G_{\Qp}$ is quasi-split), there exists a transfer of maximal torus $\Int g:T^{\ast}\hra G$ with respect to the inner twist $\psi_1^{-1}$, namely the map 
\[\Int g\circ \psi_1^{-1}|_{T^{\ast}}:T^{\ast}_{\Qb}\hra G^{\ast}_{\Qb}\stackrel{\psi_1^{-1}}{\rightarrow} G_{\Qb}\stackrel{\Int g}{\rightarrow} G_{\Qb}\] is defined over $\Q$, and thus $g\epsilon_1 g^{-1}\in G(\Q)$. 
We claim that then the (a priori $\Qb$-rational) embedding
\[\varphi:=\Int g\circ\psi^{-1}|_{T_{\Qb}}: T_{\Qb}\hookrightarrow (I_{\phi})_{\Qb}\isom I_{\Qb} \hookrightarrow G_{\Qb}\stackrel{\Int g}{\rightarrow} G_{\Qb}\] is in fact $\Q$-rational.
Indeed, as the centralizer $T=Z_{I_{\phi}}(\epsilon_1)$ is a maximal $\Q$-torus, this $\Q$-torus is also the $\Q$-subgroup of $I_{\phi}$ generated by $\epsilon_1$, and thus $\{\epsilon_1^n\}_{n\in\N}$ are also Zariski dense in $T_{\Qb}$ (Lemma \ref{lem:Zariski_group_closure} below).
So, for any $\rho\in\Gal(\Qb/\Q)$, we have ${}^{\rho}\varphi=\varphi$, as these maps coincide on $\{\epsilon_1^n\}_{n\in\N}$.
\end{proof}

We record some useful elementary facts on the subgroup generated by a semi-simple element in a reductive group. 

%%%%%%%%%%%%%%%%%%%%
\begin{lem} \label{lem:Zariski_group_closure}
Let $G$ be a reductive group $G$ over a field $k$ and $\epsilon\in G(k)$ a semi-simple element; let $S$ be the $k$-subgroup of $G$ generated by $\epsilon$.

(1) If $\epsilon':=g\epsilon g^{-1}\in G(k)$ for $g\in G(\bar{k})$, the $k$-subgroup of $G$ generated by $\epsilon'$ equals $\Int(g)(S)$.

(2) For any field extension $k'/k$, the $k'$-subgroup of $G_{k'}$ generated by $\epsilon$ equals $S_{k'}$,
and the elements $\{\epsilon^n\}_{n\in\N}$ are Zariski dense in $S_{k'}$.
\end{lem}

\begin{proof} 
(1) As $G_{\epsilon}=Z_G(S)$, the map $\Int(g):S\rightarrow \Int(g)(S)$ is a $k$-isomorphism of $k$-groups, and thus $\Int(g)(S)$ is the $k$-subgroup of $G$ generated by $\epsilon'$.
(2) This follows from the following easy fact:
if one takes a $k$-torus $T\subset G$ containing $\epsilon$, the $k$-subgroup $S$ (of multiplicative type) of $T$ generated by $\epsilon$ is the kernel of the surjective homomorphism $T\rightarrow T'$ of $k$-tori, where $T'$ is defined by $X^{\ast}(T')=\{\chi\in X^{\ast}(T)\ |\ \chi(\epsilon)=1\}$.
\end{proof}

At this point, we give an application of the results obtained thus far, namely we establish non-emptiness of Newton strata for general parahoric levels, under the condition that $G_{\Qp}$ is quasi-split. To talk about the reduction of a Shimura variety at a prime, we need to choose an integral model, i.e. a flat model $\sS_{\mbfK}$ over $\cO_{E_{\wp}}$ with generic fiber being the canonical model $\Sh_{\mbfK}(G,X)_{E_{\wp}}$. For the following result, it is enough to fix an integral model over $\cO_{E_{\wp}}$ having the extension property that every $F$-point of $\Sh_{\mbfK}(G,X)$ for a finite extension $F$ of $E_{\wp}$ extends uniquely to $\sS_{\mbfK}$ over its local ring (for example, a normal integral model); see \cite{KisinPappas15} for a construction of such integral model.

%%%%%%%%%%%%%%%%%%%%
%%%%%%%%%%%%%%%%%%%%
\begin{thm} \label{thm:non-emptiness_of_NS}
Suppose that $G_{\Qp}$ is quasi-split and that $G^{\uc}_{\Qp}$ is a product $\prod_i \Res_{F_i/\Qp}G_i$ of simple groups each of which is the restriction of scalars $\Res_{F_i/\Qp}G_i$ of an absolutely simple group $G_i$ over a field $F_i$ such that $G_i$ splits over a tamely ramified extension of $F_i$. Let $\mbfK_p$ be a parahoric subgroup of $G(\Qp)$ and $\mbfK=\mbfK_p\mbfK^p$ for a compact open subgroup $\mbfK^p$ of $G(\A_f^p)$. 

(1) Then, for any $[b]\in B(G_{\Qp},\{\mu_X\}$) (\autoref{subsubsec:B(G,{mu})}), there exists a special Shimura sub-datum $(T,h\in\Hom(\dS,T_{\R})\cap X)$ such that the Newton homomorphism $\nu_{G_{\Qp}}([b])$ equals the $G(\mfk)$-conjugacy class of
\[\frac{1}{[K_{v_2}:\Qp]}\Nm_{K_{v_2}/\Qp}\mu_h\quad (\in X_{\ast}(T)_{\Q}),\]
where $K_{v_2}\subset\Qpb$ is any finite extension of $\Qp$ splitting $T$.

In particular, if $(G,X)$ is of Hodge type, for $g_f\in G(\A_f)$, the reduction in $\sS_{\mbfK}\otimes\Fpb$ of the special point $[h,g_f\cdot\mbfK]\in \Sh_{\mbfK}(G,X)(\Qb)$ has the $F$-isocrystal represented by
\[\Nm_{K_{v_2}/K_0}(\mu_h(\pi)),\]
where $K_0\subset K_{v_2}$ is the maximal unramified subextension and $\pi$ is a uniformizer of $K_{v_2}$.

(2) Suppose that $\mbfK_p$ is \emph{special maximal} parahoric. Moreover, assume that $G$ splits over a tamely ramified cyclic extension of $\Qp$ and is of classical Lie type. Then, one can choose a special Shimura datum $(T,h\in\Hom(\dS,T_{\R})\cap X)$ as in (1) such that furthermore the unique parahoric subgroup of $T(\Qp)$ is contained in $\mbfK_p$.

(3) Suppose that $(G,X)$ is a Shimura datum of Hodge type. Then the reduction $\sS_{\mbfK}(G,X)\otimes\Fpb$ has non-empty ordinary locus if and only if $\wp$ has absolute height one (i.e. $E(G,X)_{\wp}=\Q_p$). 
\end{thm}

\begin{proof}
(1) We follow the strategy of our proof of the corresponding result in the hyperspecial case given in \cite{Lee16}, Thm. 4.1.1 and Thm. 4.3.1. Let $[b]\in B(G_{\Qp},\{\mu\}$). 
Since $G_{\Qp}$ is quasi-split, there exist a representative $b\in G(\mfk)$ of $[b]$ and a maximal torus $T_p$ of $G_{\Qp}$ such that the Newton homomorphism $\nu_b:\mathbb{D}\rightarrow G_{\Qpnr}$ is $\Qp$-rational and factors through $T_p$ \cite[Prop. 6.2]{Kottwitz86}. By the argument of Step 1 in the proof of \cite{Lee16}, Thm. 4.1.1, we may further assume that $T_p=(T_0)_{\Qp}$ for a maximal $\Q$-torus $T_0$ of $G$ such that $(T_0)_{\Qv}\subset G_{\Qv}$ is elliptic maximal for $v=\infty$ and some prime $v=l\neq p$. Then, Lemma \ref{lem:LR-Lemma5.11} tells us that there exists 
$\mu'\in X_{\ast}(T_0)\cap \{\mu\}$ such that the relation (\ref{eq:equality_on_the_kernel}) holds in $X_{\ast}(T_0)$:
\[\Nm_{K_{v_2}/\Qp}\mu'=[K_{v_2}:\Qp]\ \nu_b,\]
where $K$ is a finite Galois extension of $\Q$ splitting $T_0$ and $v_2$ is the place of $K$ induced by the pre-chosen embedding $\Qb\hra\Qpb$ (here, the sign is correct by Lemma \ref{lem:Newton_hom_attached_to_unramified_morphism}).
Next, by the argument of Step 2 in \textit{loc. cit.} (which corresponds to that of Lemma \ref{lem:LR-Lemma5.12} here), we can find a transfer of maximal torus $\Int u:T_0\hra G$ such that 
$\Int u (\mu')=\mu_h$ for some $h\in X\cap \Hom(\dS,T_{\R})$, where $T=\Int u(T_0)$ (again, be wary of the sign difference from \cite{Lee16}), and that $\Int u|_{(T_0)_{\Qpb}}=\Int y$ for some $y\in G(\Qp)$. By the latter property, for $(T,\mu_h, \Int u (b))$ we still have the equality
\[\Nm_{K_{v_2}/\Qp}\mu_h=[K_{v_2}:\Qp]\ \nu_{yb\sigma(y)^{-1}}\]
(here, $\Nm_{K_{v_2}/\Qp}$ is taken on $X_{\ast}(T)$). This proves the first statement of (1).
According to Lemma \ref{lem:unramified_conj_of_special_morphism} and \cite[Thm. 1.15]{RR96}, the element of $T(\mfk)$
\[b_T:=\Nm_{K_{v_2}/K_0}(\mu_h(\pi))\]
has the Newton homomorphism $\nu_{b_T}=\frac{1}{[K_{v_2}:\Qp]}\Nm_{K_{v_2}/\Qp}\mu_h=\nu_{yb\sigma(y)^{-1}}$.
As $\kappa_{T_{\Qp}}(b_T)=\mu^{\natural}\in X_{\ast}(T)_{\Gal(\Qpb/\Qp)}$ and $(\overline{\nu},\kappa):B(G_{\Qp})\rightarrow \mathcal{N}(G_{\Qp})\times\pi_1(G)_{\Gal(\Qpb/\Qp)}$ is injective \cite[4.13]{Kottwitz97}, we see the equality of isocrystals $[b]=[b_T]\in B(G_{\Qp})$. 
Given this, the second statement is proved in the same fashion as in the hyperspecial case, using 
\cite[Lem. 3.24]{Lee16}.

(2) Let $(T_1,h_1)$ be a special Shimura sub-datum produced in (1). Thanks to our additional assumptions and Prop. \ref{prop:existence_of_elliptic_tori_in_special_parahorics}, in its construction, we could have started with a maximal torus $T_p$ of $G_{\Qp}$ such that the unique parahoric subgroup of $T_p(\Qp)$ is contained in a $G(\Qp)$-conjugate of $\mbfK_p$. Then, also by the fine property of our methods (it uses only transfers of maximal tori which become conjugacy by $G(\Qp)$-elements), the torus $T_1$ produced in (1) can be assumed to further satisfy that the unique parahoric subgroup of $T_1(\Qp)$ is contained in $g_p\mbfK_pg_p^{-1}$ for some $g_p\in G(\Qp)$. As $G_{\Qp}$ splits over a cyclic extension of $\Qp$, $G(\Q)$ is dense in $G(\Qp)$ by a theorem of Sansuc \cite[Cor.3.5,(ii)]{Sansuc81}, thus there exists $g_0\in G(\Q)\cap \mbfK_p\cdot g_p^{-1}$. Then, one easily checks that the new special Shimura datum $(T,h):=\Int(g_0)(T_1,h_1)$ satisfies the required properties.

(3) Again. the proof is the same as that in the hyperspecial case given in \cite[Cor. 4.3.2]{Lee16}. In more detail, as was observed in \textit{loc. cit.}, it suffices to construct a special Shimura sub-datum $(T,\{h\})$ with the property that
there exists a $\Qp$-Borel subgroup $B$ of $G_{\Qp}$ containing $T_{\Qp}$ such that $\mu_h\in X_{\ast}(T)$ lies in the closed Weyl chamber determined by $(T_{\Qp},B)$. Indeed, then we have $E(T,h)_{\mathfrak{p}}=E(G,X)_{\wp}$, where $\mathfrak{p}$ and $\wp$ denote respectively the places of each reflex field induced by the given embedding $\Qb\hra\Qpb$. We remark that this is the property (ii) found in the proof of \textit{loc. cit.}, and for our conclusion one does not really need the property (i) from it.
But, since $G_{\Qp}$ is quasi-split, there exists a Borel subgroup $B'$ defined over $\Qp$. Moreover, by the same argument as was used in (1) (i.e. Step 1 in the proof of \cite[Thm. 4.1.1]{Lee16}), we may assume that $B$ contains $T'_{\Qp}$ for a maximal $\Q$-torus $T'$ of $G$ such that $T'_{\Qv}\subset G_{\Qv}$ is elliptic for $v=\infty$ and some prime $v=l\neq p$. Let $\mu'\in \{\mu_X\}\cap X_{\ast}(T')$ be the cocharacter lying in the closed Weyl chamber determined by $(T'_{\Qp},B')$. Then, the argument in (1) again produces a special Shimura sub-datum $(T,\{h\})$ such that $(T,\mu_h)=\Int y(T',\mu_{h'})$ for some $y\in G(\Qp)$, and $(T,\{h\})$ is the looked-for special Shimura sub-datum. Note that as we do not need the property (i) in the original proof of \cite[Thm. 4.1.1]{Lee16}, the condition in (2) on splitting of $G_{\Qp}$ is no longer necessary. 
\end{proof}

For more on the Newton stratification, we refer to the recent survey article \cite{Viehmann15}.

%%%%%%%%%%%%%%%%%%%%%%%%%%%%%%%%%%%%%%%%
%%%%%%%%%%%%%%%%%%%%%%%%%%%%%%%%%%%%%%%%

\section{Admissible pairs and Kottwitz triples}
 
We retain the same assumptions as in the previous section. 
  
\subsection{Criterion for an $\R$-elliptic rational element to arise from an admissible pair}

%%%%%%%%%%%%%%%%%%%%
\begin{lem} \label{lem:invariance_of_(ast(gamma_0))_under_transfer_of_maximal_tori}
Let $\gamma_0\in G(\Q)$ be a semi-simple element. 

(1) Suppose $\gamma_0\in T(\Q)$ for a maximal torus $T\subset G$. Then condition $\ast(\epsilon)$ in \autoref{subsubsec:pre-Kottwitz_triple} holds for $\gamma_0$ if and only if it holds for $\gamma_0':=\Int g(\gamma_0)$ for a transfer of maximal torus $\Int g:T\hra G$.

(2) If $\gamma_0':=\Int g(\gamma_0)\in G(\Q)$ for some $g\in G(\Qb)$, for each place $v$ of $\Q$, the image of $\gamma_0$ in $G^{\ad}(\Ql)$ lies in a compact open subgroup of $G^{\ad}(\Ql)$ if and only if $\gamma_0'$ is so.
\end{lem}

\begin{proof}
(1) Let $H$ be the centralizer of the maximal $\Qp$-split torus in the center of $(G_{\epsilon})_{\Qp}$ and $H'$ the similarly defined group for $\gamma_0'$. Choose $\mu\in X_{\ast}(T)\cap \{\mu_X\}$ which is conjugate under $H$ to a cocharacter satisfying condition $\ast(\epsilon)$.
Let $K$ be a finite Galois extension of $\Qp$ over which $\mu$ is defined, $K_0\subset K$ its maximal unramified sub-extension, and $\pi$ a uniformizer of $K$.
In view of the equality 
\[w_{H}(\Nm_{K/K_0}(\mu(\pi)))=\underline{\mu},\]
where $\underline{\mu}$ is the image of $\mu$ in $\pi_1(H)_{\Gal(\Qpb/\Qp)}$ (commutativity of diagram (7.3.1) of \cite{Kottwitz97}), condition $\ast(\epsilon)$ holds for $\gamma_0$ if and only if
\begin{equation} \label{eq:ast(gamma_0)}
\gamma_0\cdot\Nm_{K/K_0}(\mu(\pi))^{-1}\in \ker(w_H)\cap T(\mfk).
\end{equation}
By the Steinberg's theorem $H^1(\Qpnr,T)=0$, we may find $g_p\in G(\Qpnr)$ such that $\Int(g)|_{T_{\Qp}}=\Int(g_p)|_{T_{\Qp}}$.
Hence, it follows from funtoriality (with respect to $\Int(g_p):H\rightarrow H':=\Int(g_p)(H)$) of the functor $w_H$ \cite[$\S$7]{Kottwitz97} that (\ref{eq:ast(gamma_0)}) holds if and only if the same condition holds for $(\gamma_0',T',\mu')=\Int(g_p)(\gamma_0,T,\mu)$.

(2) Let $P$ and $P'$ be the $\Q$-subgroups of $G$ generated by $\gamma_0$, $\gamma_0'$, respectively.
Then, the $\Qb$-isomorphism $\Int g:G_{\Qb}\isom G_{\Qb}$ restricts to a $\Q$-isomorphism $P\isom P'$, thus if the image of $\gamma_0$ in $G^{\ad}(\Ql)$ lies in a compact open subgroup of $G^{\ad}(\Ql)$, then as it also lies in a compact open subgroup of $Q(\Ql)$, where $Q$ is the image of $P$ in $G^{\ad}$, the same property holds for $\gamma_0'$.
\end{proof}

%%%%%%%%%%%%%%%%%%%%  
%%%%%%%%%%%%%%%%%%%%
\begin{thm} \label{thm:LR-Satz5.21}
Retain the same assumptions from Thm. \ref{thm:LR-Satz5.3}.
%we assume that $G_{\Qp}$ is quasi-split, and that $G^{\uc}_{\Qp}$ is a product $\prod_i \Res_{F_i/\Qp}G_i$ of simple groups each of which is the restriction of scalars $\Res_{F_i/\Qp}G_i$ of an absolutely simple group $G_i$ over a field $F_i$ such that $G_i$ splits over a tamely ramified extension of $F_i$.
Also assume that $G$ is of classical Lie type. Let $\mbfK_p$ be a special maximal parahoric subgroup of $G(\Qp)$ and $\gamma_0$ an element of $G(\Q)$ that is $\R$-elliptic. 

(1)  If $\gamma_0$ satisfies condition $\ast(\epsilon)$ of \autoref{subsubsec:pre-Kottwitz_triple}, there exists a \emph{special} admissible pair $(\phi,\epsilon)$ (i.e. nested in a special Shimura sub-datum) such that $\epsilon$ is stably conjugate to $\gamma_0$ (in fact, $\epsilon$ is the image of $\gamma_0$ under a transfer of a maximal torus $T_0$ containing $\gamma_0$). For some $t\in\N$, the admissible pair $(\phi,\gamma_0^t)$ is also $\mbfK_p$-effective (cf. Remark \ref{rem:admissible_pair}).

Moreover, if $H$ is the centralizer of the maximal $\Qp$-split torus in the center of $(G_{\epsilon})_{\Qp}$, the following property holds:

$(\heartsuit)$: there exist an unramified $H(\Qpb)$-conjugate $\xi_p'=\Int(g_p)\circ\xi_p\ (g_p\in H(\Qpb))$ of $\xi_p:=\phi(p)\circ\zeta_p$ and $x\in H(\mfk)$ such that $x(\epsilon'^{-1}\cdot \xi_p'(s_{\tilde{\sigma}})^n)x^{-1}=\tilde{\sigma}^n$, where $\epsilon':=\Int(g_p)(\epsilon)$ and $\tilde{\sigma}\in\Gal(\Qpb/\Qp)$ is any fixed lifting of the Frobenius automorphism $\sigma$.

(2) If there exists $\delta\in G(L_n)$ (with $[\kappa(\wp):\Fp]|n$) such that $\gamma_0$ and $\Nm_n(\delta)$ are $G(\mfk)$-conjugate 
and the set 
\begin{equation} \label{eq:Y_p(delta)}
Y_p(\delta):=\left\{\ x\in G(\mfk)/\mbfKt_p\ \vert \ \sigma^nx=x,\ \mathrm{inv}_{\mbfKt_p}(x,\delta\sigma x)\in \Adm_{\mbfKt_p}(\{\mu_X\}) \ \right\}
\end{equation}
is non-empty, condition $\ast(\epsilon)$ holds for the stable conjugacy class of $\gamma_0$ with level $n$.
\end{thm}

In the original set-up of \cite{LR87}, where $\mbfK_p$ is hyperspecial and $G^{\der}=G^{\uc}$, statement (1) is Satz 5.21 there (recall that in such case, condition $\ast(\epsilon)$ also recovers the original condition).
The statement (2) is due to us (even in the set-up of \cite{LR87}). This is a key to the proof of the aforementioned effectivity criterion of Kottwitz triple which was the missing ingredient in the arguments in \cite{LR87} deducing (the constant coefficient case of) the Kottwitz formula (Thm. \ref{thm_intro:Kottwitz_formula}) from the Langlands-Rapoport conjecture (Conj. \ref{conj:Langlands-Rapoport_conjecture_ver1}).

Before entering into the proof, we discuss an explicit expression of the Frobenius automorphism $\Phi=F^n$ attached to a special admissible morphism. Let $\phi=i\circ\psi_{T,\mu_h}$ be a special admissible morphism, where $(T,h)$ is a special Shimura sub-datum and $i:\fG_T\rightarrow\fG_G$ is as usual the canonical morphism induced by the inclusion $T\hra G$. Put $\xi_p:=\phi(p)\circ\zeta_p$; so $\xi_p$ and $\xi_{-\mu_h}$ are conjugate under $T(\Qpb)$ (and in particular, $\xi_p^{\Delta}=\xi_{-\mu_h}^{\Delta}$). Now assume the following condition:

($\dagger$) \emph{$\Nm_{K/\Qp}\mu_h$ maps into the center of a semi-standard $\Qp$-Levi subgroup $H$ of $G_{\Qp}$ containing $T_{\Qp}$}, 

where $K$ is a finite Galois extension of $\Qp$ splitting $T$; later, for $H$, we will take the centralizer of the maximal $\Qp$-split torus in the center of $(G_{\epsilon})_{\Qp}$ for some admissible pair $(\phi,\epsilon)$ well-located in $T$ (cf. condition $\ast(\epsilon)$). Next, suppose given an \emph{elliptic} maximal $\Qp$-torus $T'$ of $H$. We choose $j\in H(\Qpb)$ with $T'=\Int (j)(T_{\Qp})$, and set
\[ \mu':=\Int (j)(\mu_h)\ \ \in X_{\ast}(T').\]
We take $K$ to be big enough so that $K$ splits $T'$ as well. Let $\pi$ be a uniformizer of $K$ and $K_0$ the maximal subfield of $K$ unramified over $\Qp$. If $K_1$ is the composite of $K$ and $L_s$ with $s=[K:\Qp]$ and $\xi_{-\mu'}^{K_1}$ denotes the pull-back of $\xi_{-\mu'}^{K}$ to $\fG_p^{K_1}$, by Lemma \ref{lem:unramified_conj_of_special_morphism}, there exists $t_p\in T'(\Qpb)$ such that $\Int(t_p)(\xi_{-\mu'}^{K_1})$ is an unramified morphism mapping into $\fG_{T'}$ and as such also factors through $\fG_p^{L_s}$. Moreover, we can choose $t_p$ further such that if $\xi_p':\fG_p^{L_s}\rightarrow \fG_{T'}$ denotes the induced (unramified) morphism of Galois $\Qpb/\Qp$-gerbs,
\begin{equation} \label{eq:xi_p'}
F=\xi_p'(s_{\sigma}^{L_s})=\Nm_{K/K_0}(\mu'(\pi)) \sigma.
\end{equation}
From now on, we write $\xi_{-\mu_h}$, $\xi_{-\mu'}$ for $\xi_{-\mu_h}^{K}$, $\xi_{-\mu'}^{K}$, respectively. 
In this set-up, we have the following facts:

%%%%%%%%%%%%%%%%%%%%
\begin{lem} \label{lem:Phi_for_special_morphism}
(1) The following two $\Qp$-rational cocharacters of $H$ are equal:
\[ \Nm_{K/\Qp}\mu'=\Nm_{K/\Qp}\mu_h \]
We set $\nu_p':=-\Nm_{K/\Qp}\mu'\in X_{\ast}(T')^{\Gal(\Qpb/\Qp)}$.

(2) Let $[K:K_0]=e_K$, $[K_0:\Qp]=f_K$. For any $j\in \N$ divisible by $[K:\Qp]=e_Kf_K$, we have 
\[ F^j=(p^{-\nu_p'}\cdot u_0)^{\frac{j}{[K:\Qp]}}\sigma^{j} \]
for some $u_0\in T'(\Qp)_1(:=\ker (w_{T'})\cap T'(\Qp))$. In particular, $b_j:=F^j\sigma^{-j}\in T'(\Qp)$.
Moreover, in the case $K=K_0$, if we take $\pi=p$, we have $u_0=1$.
\end{lem}

\begin{proof}
(1) This was already noted in the proof of Lemma \ref{lem:LR-Lemma5.2}:
indeed, first we see that they both map into the center of $H$: for $\Nm_{K/\Qp}\mu_h$, this is by our assumption on it, while $\nu_p'=-\Nm_{K/\Qp}\mu'$ maps into a $\Qp$-split sub-torus of the elliptic maximal torus $T'$ of $H$, so factors through the center $Z(H)$. But, also their projections into $H^{\ab}=H/H^{\der}$ are the same, since $\mu'$, $\mu_h$ are conjugate under $H(\Qpb)$. Clearly this proves the claim.

(2) Put $u:=\pi^{e_K}p^{-1}\in \cO_K^{\times}$ and $t':=j/[K:\Qp]$. Using (\ref{eq:xi_p'}), we express $F^j$ in terms of $\nu_p'$:
\begin{align*}
F^{j}&=(\Nm_{K/K_0}(\mu'(\pi)) \sigma)^j=\Nm_{K/\Qp}(\mu'(\pi^{e_K}))^{t'} \sigma^{j} \\
&=(p^{\Nm_{K/\Qp}\mu'}\cdot \Nm_{K/\Qp}(\mu'(u)))^{t'} \sigma^{j} \\
&=p^{-t'\nu_p'}\cdot u_0^{t'} \sigma^{j}
\end{align*}
where $u_0:=\Nm_{K/\Qp}(\mu'(u))$.
A priori, $u_0\in T'(\Qp)_0(=\ker(v_{T'})\cap T'(\Qp))$ (maximal compact subgroup of $T'(\Qp)$), but in fact it belongs to $T'(\Qp)_1(=\ker (w_{T'})\cap T'(\Qp))$. To see that, by funtoriality for tori $T$ endowed with a cocharacter $\mu\in X_{\ast}(T)$, we can take $T'=\Res_{K/\Qp}\Gm$ and $\mu'=\mu_K$, the cocharacter of $T'_K=\Gm^{\Hom(K,K)}$ corresponding to the identity embedding $K\hra K$. But in this case, $X_{\ast}(T')$ is an induced $\Gal(K/\Qp)$-module, so $w_{T'}=v_{T'}$, and clearly $u_0\in \ker_{v_{T'}}$. 
\end{proof}

Now, let us prove Theorem \ref{thm:LR-Satz5.21}.

\begin{proof}
To a large extent, we follow the original strategy, but using some of those facts that were established in our general setting of (special maximal) parhoric level, especially Prop. \ref{prop:existence_of_elliptic_tori_in_special_parahorics} and Lemma \ref{lem:unramified_conj_of_special_morphism}.

(1) Suppose that $\gamma_0\in G(\Q)$ is $\R$-elliptic and satisfies $\ast(\epsilon)$ of \autoref{subsubsec:pre-Kottwitz_triple}. Set $I_0:=G_{\gamma_0}^{\mathrm{o}}$; by the well-known fact that every semisimple element in a connected reductive group lies in a maximal torus, we have $\gamma_0\in I_0(\Q)$.

First, we prove the existence of a maximal $\Q$-torus $T_0$ of $I_0$ such that $(T_0)_{\R}$ is elliptic in $G_{\R}$ and $(T_0)_{\Qp}$ is elliptic in $(I_0)_{\Qp}$. Indeed, we choose a maximal $\Q$-torus $T_{\infty}$ of $G$ that contains $\gamma_0$ (so $T_{\infty}\subset I_0$) and is elliptic in $G$ over $\R$ (which exists as $\gamma_0$ is elliptic over $\R$). We also choose an elliptic maximal $\Qp$-torus $T_p$ of $(I_0)_{\Qp}$ (which exists as $(I_0)_{\Qp}$ is reductive, \cite[Thm.6.21]{PR94}). Then, one can deduce (cf. Step 1 of the proof of Thm. 4.1.1 of \cite{Lee16}) that there exists a maximal $\Q$--torus $T_0$ of $I_0$ which is $I_0(\Q_v)$-conjugate to $T_v$ for each $v=\infty$ and $p$. 

Again, let $H$ be the centralizer of the maximal $\Qp$-split torus in the center of $(G_{\gamma_0})_{\Qp}$ (so, $(G_{\gamma_0})_{\Qp}\subset H$). Let $\mu_0$ be a cocharacter of $T_0$ that is conjugate under $H(\Qpb)$ to some $\mu$ satisfying condition $\ast(\epsilon)$. Clearly, condition $\ast(\epsilon)$ still holds for $\mu_0$.
Moreover, as $(T_0)_{\Qp}$ is elliptic in $(I_0)_{\Qp}$, we see that when $K$ is a finite Galois extension of $\Qp$ splitting $T_0$, the following property (which implies the condition ($\dagger$)) holds:

($\dagger\dagger$) \emph{$\Nm_{K/\Qp}(\mu_0)$ ($\Qp$-rational cocharacter of $T_0$) maps into the center of $I_0$, thus into the maximal $\Qp$-split torus of the center of $H$.}%%
%\footnote{Although this property is critical in the rest of the proof of sufficiency of condition ($\gamma_0$), in the original proof of \cite{LR87}, Satz 5.21, it is not explicitly invoked in the argument.}
%%

Then, with such $(T_0,\mu_0)$, one applies the argument of proof of Lemma \ref{lem:LR-Lemma5.12} to find an admissible embedding of maximal torus $\Int g_0:T_0\hra G$ such that $\Int(g_0)(\mu_0)=\mu_h$ for some special Shimura sub-datum $(T:=\Int(g_0)(T_0),h)$. By Lemma \ref{lem:invariance_of_(ast(gamma_0))_under_transfer_of_maximal_tori}, conditions $\ast(\epsilon)$ and ($\dagger\dagger$) continue to hold for $(\epsilon:=\Int g_0(\gamma_0),T,\mu_h)$.
Now, we check that the resulting pair 
\[(\phi,\epsilon):=(i\circ\psi_{T,\mu_h},\Int g_0(\gamma_0)\in T(\Q))\]
is admissible and also the pair $(\phi,\epsilon^t)$ is $\mbfK_p$-effective admissible for some $t\in\N$ (Def. \ref{defn:admissible_pair}, Remark \ref{rem:admissible_pair}); as $\Int(g_0)$ is a transfer of maximal torus, $\epsilon$ is stably conjugate to $\gamma_0$. 
As we are working with special maximal parahoric $\mbfK_p$, by Lemma \ref{lem:LR-Lemma5.2}, $\phi=i\circ\psi_{T,\mu_h}$ is admissible.%%
\footnote{In this proof, the assumption that $G_{\Qp}$ is tamely ramified is needed to invoke Lemma \ref{lem:LR-Lemma5.12} and Prop. \ref{prop:existence_of_elliptic_tori_in_special_parahorics} (the latter via Lemma \ref{lem:LR-Lemma5.2}). Note that the latter proposition further requires $G_{\Qp}$ to be of classical Lie type.}
Since $\epsilon\in T(\Q)$ and $\phi$ factors through $\fG_T$, condition (2) of Def. \ref{defn:admissible_pair} holds, and as $\phi$ is special admissible, we have $T(\bar{\A}_f^p)\cap X^p(\phi)\neq\emptyset$ by Lemma \ref{lem:properties_of_psi_T,mu}, so condition (3) at $l\neq p$ is satisfied. So, it remains to establish condition (3) at $p$ (i.e. existence of $x\in G(\mfk)/\mbfKt_p$ with $\epsilon x=\Phi^m x$).

Let us use $H$ again to denote the centralizer of the maximal $\Qp$-split torus in the center of $(G_{\epsilon})_{\Qp}$. As $H$ is a semi-standard $\Qp$-Levi subgroup of the quasi-split $G_{\Qp}$, there exists $g\in G(\Qp)$ such that $H(\Qp)\cap {}^g\mbfK_p$ is a special maximal parahoric subgroup of $H(\Qp)$ (Lemma \ref{lem:specaial_parahoric_in_Levi}, cf. proof of Lemma \ref{lem:LR-Lemma5.2}). Then, we apply Prop. \ref{prop:existence_of_elliptic_tori_in_special_parahorics} to $H(\Qp)\cap {}^g\mbfK_p$ and choose an elliptic maximal $\Qp$-torus $T'$ of $H$ such that $T'_{\Qpnr}$ contains (equiv. is the centralizer of) a maximal $\Qpnr$-split $\Qpnr$-torus of $H_{\Qpnr}$ and that the (unique) parahoric subgroup $T'(\mfk)_1$ of $T'(\mfk)$ is contained in $T'(\mfk)\cap {}^g\mbfKt_p$ (as usual, $\mbfKt_p$ being the parahoric subgroup of $G(\mfk)$ corresponding to $\mbfK_p$); in fact, $T'(\mfk)_1=T'(\mfk)\cap {}^g\mbfKt_p$. 
Next, with this choice of $T'\subset H$, we let
\[K,\ \mu',\ \xi_p',\ \cdots\] 
be defined as in the beginning of this discussion: recall that $\xi_p'$ is an unramified morphism from $\fG_p^{L_s}$ to $\fG_{T'}$ ($s=[K:\Qp]$) which is conjugate to $\xi_{-\mu'}$ under $T'(\Qpb)$ and satisfies (\ref{eq:xi_p'}). Since the property ($\dagger$) holds, we have the equality $\Nm_{K/\Qp}\mu'=\Nm_{K/\Qp}\mu_h$ of Lemma \ref{lem:Phi_for_special_morphism}, (1). This equality then implies (see the proof of Lemma \ref{lem:LR-Lemma5.2}) that the two Galois $\Qpb/\Qp$-gerb morphisms into $\fG_{H}$, 
\[\xi_{-\mu_h},\quad \xi_{-\mu'}\]
are conjugate under $H(\Qpb)$, thus so are $\xi_p=\phi(p)\circ\zeta_p$ and $\xi_{-\mu'}$. Hence, there exists $v\in H(\Qpb)$ such that $\xi_p'=\Int (v)(\xi_p)$. Set 
\[\epsilon':=\Int (v)(\epsilon)\in H(\Qpnr).\]
A priori, this is only an element of $H(\Qpb)$, but since it commutes with an unramified morphism $\xi_p'$, it belongs to $H(\Qpnr)$. Given this, we can even find $v'\in H(\Qpnr)$ with $\epsilon'=\Int (v')(\epsilon)$: the neutral component $T_{\epsilon}$ (resp. $T_{\epsilon'}$) of the group (of multiplicative type) generated by $\epsilon$ (resp. by $\epsilon'$) are tori, and $T_{\epsilon'}=\Int(v)(T_{\epsilon})$. Hence, by the theorem of Steinberg ($H^1(\Qpnr,T_{\epsilon})=0$), we can find $t\in T_{\epsilon}(\Qpb)$ such that $v':=vt\in H(\Qpnr)$.
As the last preparation, for each $j\in\N$, let us define $b_j\in T'(K_0)$ by (cf. (\ref{eq:xi_p'}))
\[ b_j \sigma^{j}:=F^j=(\Nm_{K/K_0}(\mu'(\pi))\sigma)^{j}. \]
We have $b_n = \prod_{i=1}^n\sigma^i(\Nm_{K/K_0}(\mu'(\pi)))$. We will also write $b$ for $b_1$.

Now, since one has
\begin{equation} \label{eq:Kottwitz97-(7.3.1)}
w_{H}(\Nm_{K/K_0}(\mu'(\pi)))=\underline{\mu'}.
\end{equation}
(commutativity of diagram (7.3.1) of \cite{Kottwitz97}), condition $\ast(\epsilon)$ implies that
\[w_{H}(\epsilon) =\sum_{i=1}^n\sigma^{i-1}\underline{\mu'}=\sum_{i=1}^n\sigma^{i-1}w_{H}(\Nm_{K/K_0}(\mu'(\pi))),\]
thus we have
\begin{align*}
[K:\Qp]w_{H}(\epsilon) &=\sum_{j=1}^{[K:\Qp]}\sigma^{j-1}(\sum_{i=1}^n\sigma^{i-1}w_{H}(\Nm_{K/K_0}(\mu'(\pi))))\\
&=nw_{H}(\prod_{j=1}^{[K:\Qp]}\sigma^{j-1}(\Nm_{K/K_0}(\mu'(\pi)))) \\
&=w_{H}(p^{n\Nm_{K/\Qp}\mu_h}).
\end{align*}
Here, the first equality holds as $\epsilon\in H(\Qp)$ (so $\sigma(w_H(\epsilon))=w_H(\epsilon)$) and the next two equalities follow from Lemma \ref{lem:Phi_for_special_morphism}, (2): $\prod_{j=1}^{[K:\Qp]}\sigma^{j-1}(\Nm_{K/K_0}(\mu'(\pi)))=\Nm_{K/\Qp}(\mu'(pu))\in T'(\Qp)$ ($\pi^{e_K}=pu$).%%
%\footnote{In the original setting, Langlands and Rapoport demanded $\mu'$ to be defined over $L_n$. As one can see, this is not really necessary: $\mu'$ is somehow only auxiliary.}
%% 

Let $Z_{\epsilon}:=Z(G_{\epsilon})$ and $Z_{\epsilon}^{\mathrm{o}}$ be the center of $G_{\epsilon}$ and its neutral component, respectively. 
By property ($\dagger\dagger$), $\Nm_{K/\Qp}\mu_h$ maps into $Z_{\epsilon}^{\mathrm{o}}$, hence the equation above shows that the element of $Z_{\epsilon}(\Qp)$:
\[k_0:=\epsilon^{-[K:\Qp]}\cdot p^{n\Nm_{K/\Qp}\mu_h}\]
lies in $\ker(v_H)\cap Z_{\epsilon}(\Qp)$. We claim that for some $a\in\N$,
\begin{equation} \label{eq:k_0_is_bounded}
k_0^{a}\in \ker(v_{Z_{\epsilon}^{\mathrm{o}}}),
\end{equation}
i.e. $k_0^a$ lies in the maximal compact subgroup of $T(\Qp)$.
First, take $a_1\in\N$ with $\epsilon^{a_1}\in Z_{\epsilon}^{\mathrm{o}}(\Q)$, so that $k_0^{a_1}\in Z_{\epsilon}^{\mathrm{o}}(\Qp)$.
Let $A_{\epsilon}$ and $B_{\epsilon}$ be respectively the isotropic kernel and the anisotropic kernel of $Z_{\epsilon}^{\mathrm{o}}$. Then, there exists $a_2\in\N$ such that $k_0^{a_1a_2}=x\cdot y$ with $x\in A_{\epsilon}(\Qp)$ and $y\in B_{\epsilon}(\Qp)$. 
Since $v_H=v_{H^{\ab}}\circ p_H$ for the quotient map $p_H:H\rightarrow H^{\ab}$,%%
\footnote{For this, we do not need to assume that $G^{\der}=G^{\uc}$.} 
there are the implications:
\[v_{B_{\epsilon}}(y)=0\ \Rightarrow\ v_{H^{\ab}}(p_H(y))=0\ \Rightarrow\ v_{H^{\ab}}(p_H(x))=0.\]
Then, since $A_{\epsilon}\subset Z(H)$ and the natural map $Z(H)\rightarrow H^{\ab}$ is an isogeny, it follows that $v_{A_{\epsilon}}(x)=0$. Therefore, $a:=a_1a_2$ satisfies that $k_0^{a}\in \ker(v_{Z_{\epsilon}^{\mathrm{o}}})$; in particular, $k_0^a$ lies in a compact (thus bounded) subgroup of $T(\Qp)$. We also note in passing that when $G^{\der}=G^{\uc}$ (so that $H^{\der}=H^{\uc}$ too), the same argument establishes that 
$k_0^{a}\in \ker(w_{Z_{\epsilon}^{\mathrm{o}}})$ (since in such case $w_H=w_{H^{\ab}}\circ p_H$).

Then, by (\ref{eq:k_0_is_bounded}) we see that for sufficiently large $t$ divisible by $a[K:\Qp]$, the element $k_t\in H(\Qpnr)$ defined by
\begin{align*} \label{eqn:epsilon'^{-1}F^n}
k_t  \sigma^{nt} & := (\epsilon'^{-1}F^n)^{t} \\
&= v' \epsilon^{-t} v'^{-1} \cdot (p^{-\nu_p'}\cdot u_0)^{\frac{nt}{[K:\Qp]}} \sigma^{nt} \nonumber \\
&= v' (\epsilon^{-[K:\Qp]}\cdot p^{n\Nm_{K/\Qp}\mu_h})^{\frac{t}{[K:\Qp]}} v'^{-1} \cdot u_0^{\frac{nt}{[K:\Qp]}} \sigma^{nt}  \nonumber \\
&= v' k_0^{\frac{t}{[K:\Qp]}} v'^{-1}\cdot (u_0)^{\frac{nt}{[K:\Qp]}} \sigma^{nt} \nonumber
\end{align*}
lies in any given neighborhood of $1$ in $H(\mfk)$.%
\footnote{Hence, in the hyperspecial cases where one may assume $u_0=1$, for sufficiently large $t$ such that $\nu':=-\frac{nt}{[K:\Qp]}\nu_p'\in X_{\ast}(T')$ and $k_t\in \mbfKt_p$, the decomposition \[(\epsilon')^t= p^{\nu'}\cdot k_t^{-1}\] is a Cartan decomposition, as was asserted in \cite[p.193, line 9]{LR87}. Note that here the fact that $\Nm_{K/\Qp}\mu_h$ maps into the center of $G_{\epsilon}$ is needed (to know that for some $a\in\N$ $k_0^a$ lies in a bounded subgroup of $H(\mfk)$).}

In particular, for sufficiently large $t$, $k_t$ lies in the special maximal parahoric subgroup $H(\mfk)\cap {}^g\mbfKt_p$ of $H(\mfk)$, which then implies existence of $h\in H(\mfk)\cap {}^g\mbfKt_p$ such that
\begin{equation*} 
%\label{eqn:epsilon'^{-1}F^n_has_fixed_pt}
(\epsilon'^{-1}\Phi^m)^{t}=h\sigma^{tn}(h^{-1})\rtimes \sigma^{tn}.
\end{equation*}
by \cite[Prop.3]{Greenberg63}.  
We fix such $t\in\N$. We see that $\epsilon'^{-t}\Phi^{mt}$ fixes $hg x^{\mathrm{o}}=g x^{\mathrm{o}}$ ($x^{\mathrm{o}}:=1\cdot \mbfKt_p$, the base point of $G(\mfk)/\mbfKt_p$).
Moreover, by Lemma \ref{lem:unramified_conj_of_special_morphism} and commutativity of the diagram (7.3.1) of \cite{Kottwitz97}, we have
\[\mathrm{inv}_{T'(\mfk)_1}(T'(\mfk)_1,F T'(\mfk)_1)=\underline{\mu'}. \]
Then, since $g^{-1}T'(\mfk)_1g \subset \mbfKt_p$, it follows (see the proof of Lemma \ref{lem:LR-Lemma5.2}) that
\[ \mathrm{inv}_{\mbfKt_p}(gx^{\mathrm{o}},F gx^{\mathrm{o}})=\tilde{W}_{\mbfK_p}\ t^{\underline{g^{-1}\mu'g}}\ \tilde{W}_{\mbfK_p}\in\Adm_{\mbfKt_p}(\{\mu_X\}) \]
(regarding $g^{-1}\mu'g$ as a cocharacter of $g^{-1}T'g$).
This proves that $(\phi,\epsilon^t)$ is $\mbfK_p$-effective admissible.

Next, we show that there exists $e\in H(\mfk)$ such that \[e^{-1}(\epsilon')^{-1}\Phi^m e=\sigma^n,\] 
which will establish the admissibility of $(\phi,\epsilon)$. 
We first claim that there exists $c\in H(\mfk)$ such that \[c^{-1}(\epsilon')^{-1}\Phi^m c\in \rho(H^{\uc}(\mfk))\times\sigma^n,\] 
where $\rho:H^{\uc}\rightarrow H$ is the canonical morphism. 
By Lemma \ref{lem:kernel_of_w} below, it suffices to show that $w_{H}(\epsilon'^{-1}b_n)=0$ (recall that $\Phi^m=F^n=b_n \sigma^n$). By (\ref{eq:xi_p'}), $b_n=\prod_{i=1}^n\sigma^{i-1}(\Nm_{K/K_0}(\mu'(\pi)))$ so 
\[w_H(b_n)=\sum_{i=1}^n\sigma^{i-1}\underline{\mu'}=\sum_{i=1}^n\sigma^{i-1}\underline{\mu_h}=w_H(\epsilon)=w_H(\epsilon')\] 
(the first equality is (\ref{eq:Kottwitz97-(7.3.1)})).

Next, we proceed as in the proof (on p. 193) of \cite{LR87}, Satz 5.21, to find $d\in \rho(H^{\uc}(\mfk))\cap {}^g\mbfKt_p$ such that $d^{-1}c^{-1}\epsilon'^{-1}\Phi^mcd=\sigma^n$. For that, when we pick $k'\in H^{\uc}(\mfk)$ mapping to $c^{-1}\epsilon'^{-1}b_n \sigma^{n}(c)$ (i.e. $\rho(k') \sigma^n=c^{-1}(\epsilon'^{-1}\Phi^m)c$), by \cite[Prop. 5.4]{Kottwitz85}, it suffices to show that $k'$ is basic (in $B(H^{\uc})$).
By definition \cite[(4.3.3)]{Kottwitz85}, this is the same as the existence of $d'\in H^{\uc}(\mfk)$ with $(k' \sigma^n)^{t}=d'(1\rtimes\sigma^{nt})d'^{-1}$ for some sufficiently large $t$.
But, we have
\[c^{-1}(\epsilon'^{-1}\Phi^m)^{t}c=c^{-1}k_t\sigma^{nt}(c) \sigma^{nt}=(c^{-1}k_tc)(c^{-1}\sigma^{nt}(c)) \sigma^{nt},\]
and for sufficiently large $t\in\N$, both $c^{-1}k_tc$ and $c^{-1}\sigma^{tn}(c)$ are contained in any neighborhood of $1$ in $H(\mfk)$, in particular, in the special maximal parahoric subgroup $H(\mfk)\cap {}^g\mbfKt_p$ of $H(\mfk)$. Thus, 
if $k'_t:=k'\cdot\sigma^n(k')\cdots \sigma^{n(t-1)}(k')\in H^{\uc}(\mfk)$ (i.e. 
$\rho(k_t') \sigma^{nt}=c^{-1}(\epsilon'^{-1}\Phi^m)^{t}c$), $\rho(k_t')$ lies in the special maximal parahoric subgroup $H^{\der}(\mfk)\cap {}^g\mbfKt_p$ of $H^{\der}(\mfk)$. In view of the canonical equality of reduced buildings $\mcB(H^{\uc},\mfk)=\mcB(H^{\der},\mfk)$, this implies that $k_t'$ also lies in the stabilizer in $H^{\uc}(\mfk)$ of the corresponding special vertex, which is itself a special maximal parahoric subgroup of $H^{\uc}(\mfk)$ (as $w_{H^{\uc}}$ is trivial). Hence, again by \cite[Prop.3]{Greenberg63} there exists $d'\in H^{\uc}(\mfk)$ such that $k'_t=d'\sigma^{nt}(d'^{-1})$, as required.

(2) Suppose that $c\gamma_0c^{-1}=\Nm_n(\delta)$ for $c\in G(\mfk)$; then, $b:=c^{-1}\delta\sigma(c)\in Z_{G_{\Qp}}(\gamma_0)(\mfk)$; thus, $\gamma_0,b\in H(\mfk)$. As a matter of fact, the arguments coming next work in general for \emph{any} semi-standard $\Qp$-Levi subgroup $H$ of $G_{\Qp}$ containing $\gamma_0$ and $b$.

We proceed in several steps:
\begin{itemize}
\item[(i)]
First, we claim existence of $c_1\in G(\mfk)$ such that for $\delta':=c_1b\sigma(c_1^{-1})$, one has
\begin{equation} \label{eq:stable_conjugacy_1}  
c_1\gamma_0c_1^{-1}=\Nm_n(\delta'), 
\end{equation}
and $c_1^{-1}\mbfKt_p \in X(\{\mu_X\},b)_{\mbfK_p}$, i.e.
\begin{equation} \label{eq:Deligne-Lustizg_1}
\mbfKt_p\cdot \delta'\cdot \mbfKt_p \in \Adm_{\mbfKt_p}(\{\mu_X\}).
\end{equation}
Indeed, pick $g\mbfKt_p\in Y_p(\delta)$. Then, as $g^{-1}\sigma^n(g)\in \mbfKt_p$, by \cite[Prop.3]{Greenberg63}, there exists $k_0\in \mbfKt_p$ with $g^{-1}\sigma^n(g)=k_0\sigma^n(k_0^{-1})$, i.e. $d:=(gk_0)^{-1}\in G(L_n)$. Clearly,  $c_1:=dc$ satisfies the required conditions: 
\[\mbfKt_p\cdot g^{-1}\delta\sigma(g)\cdot \mbfKt_p= \mbfKt_p\cdot c_1b\sigma(c_1^{-1})\cdot \mbfKt_p.\]
We observe that $\delta'=d\delta\sigma(d^{-1})$, i.e. $\delta'$ is $\sigma$-conjugate to $\delta$ under $G(L_n)$.

\item[(ii)]
From this point, we adapt the argument of the proof of Lemma \ref{lem:LR-Lemma5.11} (however, working with $H$ in place of $M$). We recall its set-up as we need it. First we choose a maximal $\Qp$-split torus in $H$ and a maximal $\Qp$-split torus $S$ in $G_{\Qp}$ containing it; as $H$ is the centralizer of a $\Qp$-split torus, we have $S\subset H$ and $S$ is also a maximal $\Qp$-split torus in $H$. Then, we pick a $\Qp$-torus $S'$ of $H$ whose extension to $\Qpnr$ becomes a maximal $\Qpnr$-split torus of $H_{\Qpnr}$ containing $S_{\Qpnr}$; thus $S'$ is again such a torus for $G_{\Qp}$ and the centralizer $T':=Z_{G_{\Qp}}(S')$ is a maximal torus of both $G$ and $H$:
\[S\subset S'\subset T'=Z_{G_{\Qp}}(S')\subset H.\]
As in the proof of Lemma \ref{lem:LR-Lemma5.11}, we may choose $S'$ such that the given special point $\mbfo$ of $\mcB(G,\mfk)$ lies in the image of the apartment $\mcA^{H}_{\mfk}\subset\mcB(H,\mfk)$ of $S'$ (regarding $S'$ as a maximal $\mfk$-split torus of $H$) under a suitable embedding $\mcB(H,\mfk)\hra\mcB(G,\mfk)$.

We also choose a $\Qp$-parabolic subgroup $Q$ of $G_{\Qp}$ of which $H$ is a Levi factor; let $N_Q$ be the unipotent radical of $Q$.
Now, we claim that there exists $m\in H(\mfk)$ such that 
\begin{equation} \label{eq:stable_conjugacy_2}
\Nm_n(m^{-1}b\sigma(m))^{-1}\cdot m^{-1}\gamma_0m \in H(\mfk)\cap (\mbfKt_p\cdot N_Q(\mfk))
\end{equation}
and that
\begin{equation} \label{eq:Deligne-Lustizg_2}
\mbfKt_p\cdot m^{-1}b\sigma(m)n' \cdot \mbfKt_p \in \Adm_{\mbfKt_p}(\{\mu_X\})
\end{equation}
for some $n'\in N_Q(\mfk)$.
Indeed, using the Iwasawa decomposition as presented in Step (1) of Lemma \ref{lem:LR-Lemma5.11}, we write 
\[c_1^{-1}=nmk\] 
with $n\in N_Q(\mfk)$, $m\in H(\mfk)$, and $k\in \mbfKt_p$, so that
\begin{align*}
\delta' & =c_1b\sigma(c_1)^{-1}=k^{-1}m^{-1}n^{-1}b\sigma(n)\sigma(m)\sigma(k) \\
& =k^{-1}m^{-1}b\sigma(m)n'\sigma(k)
\end{align*}
for $n':=\sigma(m)^{-1}b^{-1}n^{-1}b\sigma(n)\sigma(m)$ (which belongs to $N_Q(\mfk)$, as $H$ normalizes $N_Q$), and (\ref{eq:Deligne-Lustizg_1}) becomes
\[\mbfKt_p\cdot \delta'\cdot \mbfKt_p=\mbfKt_p\cdot m^{-1}b\sigma(m)n'\cdot \mbfKt_p.\] 
On the other hand, the left-hand side of (\ref{eq:stable_conjugacy_1}) becomes
\[ c_1\gamma_0c_1^{-1}=k^{-1}m^{-1}n^{-1}\gamma_0 nmk =k^{-1}m^{-1}\gamma_0m n_1 k \]
for $n_1:=m^{-1}\gamma_0^{-1}n^{-1}\gamma_0nm\in N_Q(\mfk)$, so that if $b_1:=m^{-1}b\sigma(m)\in H(\mfk)$, the right-hand side of (\ref{eq:stable_conjugacy_1}) becomes (again using that $H$ normalizes $N_Q$):
\begin{align*}
\Nm_n(\delta') &=\Nm_n(k^{-1}b_1n'\sigma(k)) \\
&= k^{-1}\cdot b_1n'\cdot \sigma(b_1)\sigma(n')\cdots\sigma^{n-1}(b_1)\sigma^{n-1}(n')\cdot \sigma^n(k) \\
& =k^{-1}\Nm_n(b_1)n_2\sigma^n(k).
\end{align*}
for some $n_2\in N_Q(\mfk)$.
Therefore, (\ref{eq:stable_conjugacy_1}) becomes:
\[m^{-1}\gamma_0m\cdot n_1\cdot k= \Nm_n(b_1)\cdot n_2 \cdot \sigma^n(k),\]
which reduces to: 
\[ \Nm_n(b_1)^{-1}\cdot m^{-1}\gamma_0m \cdot n_3=\sigma^n(k)k^{-1} \]
for some $n_3 \in N(\mfk)$ (i.e. $n_3$ is defined by $hn_3=\sigma^n(k)k^{-1}=n_2^{-1}h\cdot n_1$ for $h:=\Nm_n(b_1)^{-1}\cdot m^{-1}\gamma_0m\in H(\mfk)$), which establishes (\ref{eq:stable_conjugacy_2}). 

\item[(iii)]
Next, we claim that
\begin{equation} \label{eq:vanishing_of_w_H}
 w_H(H(\mfk)\cap (\mbfKt_p\cdot N_Q(\mfk)))=0.
\end{equation}
For that, recalling that $S'_{\mfk}$ is a maximal $\mfk$-split torus  of $G_{\mfk}$ such that $T'=Z_{G_{\Qp}}(S')\subset H$, we choose a Borel subgroup $B'$ of $G_{\mfk}$ containing $T'_{\mfk}$ and $N_Q$. Then, we pick a Borel subgroup $B$ of $H_{\mfk}$ containing $T'_{\mfk}$ and contained in $B'$; let $N_B$ be its unipotent radical. Thanks to the condition that $H(\mfk)\cap \mbfKt_p$ is also a special maximal parahoric subgroup of $H(\mfk)$ (cf. Lemma \ref{lem:specaial_parahoric_in_Levi}), it is enough to show that 
\[H(\mfk)\cap (\mbfKt_p\cdot N_Q(\mfk)) \subseteq (H(\mfk)\cap \mbfKt_p)\cdot N_B(\mfk).\]

Suppose given $m=kn$ for some $m\in H(\mfk)$, $k\in\mbfKt_p$, and $n\in N_Q(\mfk)$.
Since $H(\mfk)\cap \mbfKt_p$ is a special maximal parahoric subgroup of $H(\mfk)$, we use the Iwasawa decomposition for $(H(\mfk),H(\mfk)\cap \mbfKt_p,T',B)$, to write
\[m=k_0 t u\]
where $k_0\in H(\mfk)\cap \mbfKt_p$, $t\in T'(\mfk)$, and $u\in N_B(\mfk)$. So, we have
$k_0^{-1}k=t(un^{-1})$. Then, since one has 
\[ B'(\mfk)\cap\mbfKt_p=(T'(\mfk)\cap\mbfKt_p)\cdot  (N_{B'}(\mfk)\cap\mbfKt_p) \]
\cite{BT84}, 5.2.4, applied to $(G(\mfk),\mbfKt_p, T'_{\mfk},B')$), we must have
$t\in \mbfKt_p$, $un^{-1}\in \mbfKt_p$.

\item[(iv)]
Let $\mu''\in X_{\ast}(T')$ be the cocharacter defined in Step (2) of Lemma \ref{lem:LR-Lemma5.11}, i.e. defined by the Cartan decomposition
\[m^{-1}b\sigma(m)\in (K_{\mbfo}(\mfk)\cap H(\mfk))\ t^{\underline{\mu''}}\ (K_{\mbfo}(\mfk)\cap H(\mfk))\]
for $(m^{-1}b\sigma(m),H,K_{\mbfo}(\mfk)\cap H(\mfk))$ ($K_{\mbfo}(\mfk)=\mbfKt_p$).
Then, recalling the cocharacter $\mu'\in X_{\ast}(T')\cap \{\mu_X\}$ (\ref{eqn:Iwahori_inequality}),
the argument of Step (3) of Lemma \ref{lem:LR-Lemma5.11} (which uses condition (\ref{eq:Deligne-Lustizg_2})) established the relation (\ref{eq:LR-Lemma5.11_Step_3})
\[\mu''=w\mu'\cdot\mu_1\] 
for some $w\in W$ and $\mu_1\in \langle \tau x-x\ |\ \tau\in\Gal(\overline{\mfk}/\mfk), x\in X_{\ast}(T')\rangle$; in particular, one has $w_H(m^{-1}b\sigma(m))=\underline{\mu''}=\underline{w\mu'}$ in $\pi_1(H)_{\Gal(\overline{\mfk}/\mfk)}$.
Therefore, by (\ref{eq:vanishing_of_w_H}) we have
\begin{align*}
w_H(\gamma_0) &=w_H(m^{-1}\gamma_0m)=w_H(\Nm_n(m^{-1}b\sigma(m))) \\
&=\sum_{i=1}^n\sigma^{i-1}w_H(m^{-1}b\sigma(m))=\sum_{i=1}^n\sigma^{i-1}\underline{w\mu'}.
\end{align*} 
As $w\mu'\in X_{\ast}(T')\cap \{\mu_X\}$, the lemma is proved. 
\end{itemize}
This completes the proof of the theorem.
\end{proof}

%%%%%%%%%%%%%%%%%%%%
\begin{lem} \label{lem:kernel_of_w}
Let $H$ be a connected reductive group over $\Qp$. For any element $h\in H(\mfk)$ with $w_H(h)=0$, there exists $c\in H(\mfk)$ such that $ch\sigma^n(c^{-1})\in \rho(H^{\uc}(\mfk))$, where $\rho:H^{\uc}\rightarrow H$ is the canonical homomorphism 
\end{lem}

\begin{proof}
Using the fact \cite[(3.3.3)]{Kottwitz84b} that for any maximal $\mfk$-split torus $S$ of $H_{\mfk}$ and its centralizer $T$, one has
\[H(\mfk)=\rho(H^{\uc}(\mfk))T(\mfk),\]
we write $h=\rho(h')t$, where $h'\in H^{\uc}(\mfk)$ and $t\in T(\mfk)$.
Since $w_H$ vanishes on $\rho(H^{\uc}(\mfk))$ (cf. diagrams (7.4.1), (7.4.2) of \cite{Kottwitz97}), we have
$w_H(t)=0$. When we put $T^{\uc}:=\rho^{-1}(T)$ (maximal torus of $H^{\uc}$), as $X_{\ast}(T^{\uc})$ is an induced module for $I=\Gal(\bar{\mfk}/\mfk)$ \cite[4.4.16]{BT84} so that
one has
\[X_{\ast}(T)_I/X_{\ast}(T^{\uc})_I=\pi_1(H)_I\]
\cite[p.196]{HainesRapoport08}, and as $w_{T^{\uc}}$ is surjective, we can find $t'\in T^{\uc}(\mfk)$ with $w_{T}(\rho(t')t^{-1})=0$. Hence, by \cite[Prop.3]{Greenberg63} there exists $c\in T(\mfk)$ with $\rho(t')t^{-1}=c^{-1}\sigma^n(c)$, namely with $c\rho(t')\sigma^n(c^{-1})=t$.
Finally, as $\rho(H^{\uc}(\mfk))$ is a normal subgroup of $H(\mfk)$, this establishes the claim.
\end{proof}

%%%%%%%%%%%%%%%%%%%%
\begin{rem}  \label{rem:equality_of_two_ADLVs}
In \cite{LR87}, Satz 5.21, Langlands and Rapoport also claim that the converse of Theorem \ref{thm:LR-Satz5.21}, (1) holds (again when $\mbfKt_p$ is hyperspecial and $G^{\der}=G^{\uc}$):

\textit{If $(\phi,\epsilon)$ is an admissible pair and $\gamma_0\in G(\Qp)$ is stably conjugate to $\epsilon$, then $\gamma_0$ satisfies the condition $\ast(\epsilon)$ \cite[p.183]{LR87}.}

We believe that this statement becomes true only under some additional assumption, for example if \textit{the set}
\[X_p(\phi,\epsilon)=\{ x_p\in X_p(\phi)\ |\ \epsilon x_p=\Phi^mx_p\}\]
\textit{is non-empty}.

Indeed, this results from Theorem \ref{thm:LR-Satz5.21}, (2) and the following observation (cf. \cite[1.4]{Kottwitz84b}):

(1) if we choose $u\in G(\Qpb)$ with $u^{-1}\mbfK_p(\Qpnr)\in X_p(\phi)$ (so, $\xi_p'=\Int(u)\circ\phi(p)\circ\zeta_p$ is unramified), say the inflation of a $\Qpnr/\Qp$-Galois gerb morphism $\theta^{\nr}:\fD\rightarrow \fG_{G_{\Qp}}^{\nr}$, $\Int(u)$  (\ref{eqn:X_p(phi)=ADLV}) gives a bijection between $X_p(\phi,\epsilon)$ and the set
\begin{equation} \label{eq:X_p(b,epsilon')}
X_p(b,\epsilon'):=\{ x\in X(\{\mu_X\},b)_{\mbfK_p}\ |\ \epsilon' x=(b\sigma)^nx \},
\end{equation}
where $\epsilon':=u\epsilon u^{-1}\in G(\Qpnr)$ and $b\sigma:=F:=\theta^{\nr}(s_{\sigma})$; so, $\epsilon'\in J_b(\Qp)=\{h\in G(\mfk)\ |\ (b\sigma)h=h(b\sigma)\}$. 

(2) For $b\in G(\mfk)$ and $\epsilon'\in J_b(\Qp)$, if the equation 
\begin{equation} \label{eq:stably_conjugacy_rel'n}
y\epsilon'^{-1}(b\sigma)^n y^{-1}= \sigma^n
\end{equation}
has a solution $y=c\in G(\mfk)$, then $\delta:=cb\sigma(c^{-1})$ belongs to $G(L_n)$ and satisfies
\begin{equation} \label{eq:stably_conjugacy_rel'n'}
c\epsilon' c^{-1}=\Nm_n\delta,
\end{equation}
and the mapping $x\mapsto cx$ gives a bijection 
\[X_p(b,\epsilon')\rightarrow Y_p(\delta). \]
Conversely, for a pair $(\epsilon',\delta)\in G(\mfk)\times G(L_n)$ satisfying the relation (\ref{eq:stably_conjugacy_rel'n'}) for some $c\in G(\mfk)$, we have $\epsilon'\in J_b(\Qp)$ for 
$b:=c^{-1}\delta\sigma(c)$ (since then $\delta\sigma$ commutes with $\Nm_n\delta$), and the mapping $x\mapsto c^{-1}x$ gives a bijection $Y_p(\delta)\rightarrow X_p(b,\epsilon')$. 
\end{rem}

%%%%%%%%%%%%%%%%%%%%%%%%%%%%%%%%%%%%%%%%
%%%%%%%%%%%%%%%%%%%%%%%%%%%%%%%%%%%%%%%%

\subsection{Criterion for a Kottwitz triple to come from an admissible pair}

We continue to work in the same set-up from the previous subsection.

%%%%%%%%%%%%%%%%%%%%
\begin{prop} \label{prop:phi(delta)=gamma_0_up_to_center}
Let $(\phi,\gamma_0)$ be an admissible pair, say of level $n$ (cf. Def. \ref{defn:admissible_pair}). Suppose that some conjugate of $\gamma_0$ belongs to $G(\Q)$ and its image in $G^{\ad}(\A_f^p)$ lies in a compact subgroup of $G^{\ad}(\A_f^p)$. 
 
(1) For any sufficiently large $k\in\N$ divisible by $n$, the element $\phi(\delta_k)\cdot \gamma_0^{-\frac{k}{n}}$ of $G(\Qb)$ lies in the center of $G$. 

(2) Assume that the weight homomorphism $w_X=(\mu_h\cdot\iota(\mu_h))^{-1}\ (h\in X)$ is rational. 
If $\gamma_0$ is a Weil $q=p^n$-element of weight $-w=-w_X$, in the sense that for the $\Q$-subgroup $S$ (of multiplicative type) generated by $\gamma_0$ and any character $\chi$ of $S$, $\chi(\gamma_0)\in\Qb$ is a Weil $q=p^n$-number of weight $-\langle \chi,w_X\rangle\in\Z$ in the usual sense, then $\gamma_0^{\frac{k}{n}}=\phi(\delta_k)$ for any sufficiently large $k\in\N$ divisible by $n$.

(3) If the anisotropic kernel of $Z(G)$ remains anisotropic over $\R$ and $\gamma_0\in G(\A_f^p)$ lies in a compact subgroup of $G(\A_f^p)$, then $\gamma_0^{\frac{k}{n}}=\phi(\delta_k)$ for any sufficiently large $k\in\N$ divisible by $n$. In particular, $(\phi,\gamma_0)$ is already well-located.
\end{prop}

%Be reminded that as we do not assume that $G^{\der}$ is simply connected, the centralizer $G_{\gamma_0}$ of $\gamma_0\in G(\Q)$ is not necessarily connected.

\begin{proof} 
The first statement (1) is asserted in \cite{LR87}, p.194 (line -8) - p.195 (line 12) with a sketchy proof. Here we will give a detailed proof. To show the proposition, we need a fact which was stated in \cite{LR87}, p.195, line 5-9,  but without an explanation:

%%%%%%%%%%%%%%%%%%%%
\begin{lem} \label{lem:equality_of_two_Newton_maps}
(1) Let $\phi:\fP\rightarrow \fG_G$ be an admissible morphism such that $\phi(\delta_k)\in T(\Q)$ for a maximal $\Q$-torus $T$ of $G$ and for all $k\gg1$.
Let $I:=Z_G(\phi(\delta_k))\ (k\gg1)$ and $b\in I(\mfk)$ defined by an unramified conjugate of $\xi_p:=\phi(p)\circ\zeta_p:\fG_p\rightarrow \fG_{I}(p)$ under $I(\Qpb)$ as in Lemma \ref{lem:unramified_morphism}, (4).
Then, the Newton homomorphism $\nu_{\phi(\delta_k)}\in X_{\ast}(T)_{\Q}^{\Gal_{\Qp}}$ equals $k\nu_b\in \Hom_{\mfk}(\mathbb{D},I)$: $\nu_{\phi(\delta_k)}=k\nu_b$.

(2) Let $\epsilon\in G(\Qp)$ be a semi-simple element and suppose that there exists $\delta\in G(L_n)$ such that $\Nm_n\delta=c'\epsilon c'^{-1}$ for some $c'\in G(\mfk)$; let $I_0:=G_{\epsilon}^{\mathrm{o}}$. If $b':=c'^{-1}\delta\sigma(c')\in G_{\epsilon}(\mfk)$ belongs to $I_0$, the two Newton homomorphisms $\nu_{\epsilon}$, $\nu_{b'}\in \Hom_{\mfk}(\mathbb{D},I_0)$ are related by: $\nu_{\epsilon}=n\nu_{b'}$. 
In particular, in this case $[b']_{I_0}\in B(I_0)$ is basic \cite[(5.1)]{Kottwitz85}.

(3) For any admissible pair $(\phi,\epsilon)$ well-located in a maximal $\Q$-torus $T$ of $G$, we have equality of quasi-cocharacters of $T$: $\frac{1}{k}\nu_{\phi(\delta_k)}=\frac{1}{n}\nu_{\epsilon}\ (k\gg1)$.
\end{lem}

\begin{proof}
(1) Fix a CM field $K$ Galois over $\Q$ and $m\gg1$ such that $\phi^{\Delta}$ factors through $P^K$ and $P^K=P(K,m)$, and denote by $(\phi^{K})^{\Delta}$ the resulting morphism $P(K,m)\rightarrow G$. Then for every $\lambda\in X^{\ast}(T)$, the $\Qb$-character $\lambda\circ(\phi^{K})^{\Delta}$ of $P(K,m)$ is also a Weil $p^m$-number $\pi$, in which case if we write $\chi_{\pi}:=\lambda\circ(\phi^{K})^{\Delta}$, we have $\chi_{\pi}(\delta_k)=\pi^{\frac{k}{m}}$ for every $k\in\N$ divisible by $m$ (cf. \autoref{subsubsec:pseudo-motivic_Galois_gerb}). Next, let $w$ be the place of $K$ induced by $\iota_p$. When we regard any $\lambda\in X^{\ast}(T_{\Qp})^{\Gal_{\Qp}}$ as a $\Qb$-character of $T$ via the chosen embedding $\iota_p:\Qb\hra\Qpb$, we have 
%(writing $\phi^K(\delta_k)$ for $(\phi^K)^{\Delta}(\delta_k)$ for short)
\[|\lambda(\phi^K(\delta_k))|_p^{[K_w:\Qp]}=|\lambda(\phi^K(\delta_k))|_w=|\chi_{\pi}(\delta_k)|_w=|\pi|_w^{\frac{k}{m}}=p^{k\nu_2(\pi,w)}=p^{k\langle\chi_{\pi},\nu_2^{K}\rangle}=p^{k\langle\lambda,(\xi_p^{K_w})^{\Delta}\rangle},\]
where $(\xi_p^{K_w})^{\Delta}=(\phi^{K})^{\Delta}_{\Qp}\circ (\zeta_p^{K_w})^{\Delta}:(\fG_p^{K_w})^{\Delta}\rightarrow P(K,m)_{\Qp}\rightarrow I_{\Qp}$ (by assumption, $(\xi_p^{K_w})^{\Delta}\in X_{\ast}(T)^{\Gal_{\Qp}}$); note that here, $\lambda$ can be considered as a $\Qp$-character in the first three expressions, while the $3^{\text{rd}}$ equality makes sense only when $\chi_{\pi}$ is a $\Qb$-character.
This shows \cite[2.8, 4.4]{Kottwitz85} that the Newton homomorphism $\nu_{\phi(\delta_k)}\in X_{\ast}(T)_{\Q}^{\Gal_{\Qp}}=\Hom_{\Qp}(\mathbb{D},T_{\Qp})$ attached to $\phi(\delta_k)\in T(\Qp)$ is $-\frac{k}{[K_w:\Qp]}(\xi_p^{K_w})^{\Delta}=-k\xi_p^{\Delta}$.

On the other hand, let $(\xi_p^{K_w})'$ be an unramified conjugate of $\xi_p^{K_w}$ under $I(\Qpb)$. If $(\xi_p^{K_w})'$ factors through $\fD_n$ (for some $n\in\N$), the Newton homomorphism $\nu_b$ attached to $b$ is $\frac{1}{n}(\xi_p^{K_w})'^{-1}|_{\Gm}$ (Lemma \ref{lem:Newton_hom_attached_to_unramified_morphism}, cf. \cite{LR87}, Anmerkung). In our case, we may assume that $n=[K_w:\Qp]$, by Remark \ref{rem:comments_on_zeta_v}, (2) and Lemma \ref{lem:unramified_morphism}, (2). Since the restrictions of $\xi_p'$ and $\xi_p$ to the kernel of $\fD_n$ are the same, we have $\nu_b=-\frac{1}{[K_w:\Qp]}(\xi_p^{K_w})^{\Delta}$, which proves the claim.

(2) This is proved in Lemma 5.15 of \cite{LR87}. We briefly sketch its arguments. 
First, we observe that for any $c'\in G(\mfk)$ and $n'\in \N$, $c_{n'}:=c'^{-1}\sigma^{n'}(c')$ belongs to $G_{\epsilon}(\mfk)$ and lies in any small neighborhood of $1$ in $I_0(\mfk)$ as $n'$ becomes large (in fact, even becomes $1$ if $c'\in G(\Qpnr)$). Secondly, if we choose a maximal $\Qp$-torus $T_1$ of $G_{\Qp}$ containing $\epsilon$, then the Newton quasi-cocharacter $\nu_{\epsilon}$ of $\epsilon\in T_1(\Qp)$ satisfies \cite[4.4]{Kottwitz85} that for every $\Qp$-rational $\lambda\in X_{\ast}(T_1)$, one has
\[|\lambda(\epsilon)|_p=p^{-\langle\lambda,\nu_{\epsilon}\rangle}.\]
It follows from this equation that $\nu_1:=\frac{1}{n}\nu_{\epsilon}\in X_{\ast}(T_1)_{\Q}^{\Gal_{\Qp}}$ maps into the (connected) center of $G_{\epsilon}$ and that $p^{-n\nu_1}\epsilon\in T_1(\Qp)$ lies in the maximal compact subgroup of $T_1(\Qp)$. Especially, $(p^{-n\nu_1}\epsilon)^k$ also lies in any small neighborhood of $1$ in $I_0$ if $k$ becomes large. Therefore, according to \cite[Prop. 3]{Greenberg63}, for sufficiently large $k\in\N$, there exists $d\in I_0(\mfk)$ such that with $n'=nk$, 
\[\epsilon^k c_{n'}=p^{n'\nu_1}d^{-1}\sigma^{n'}(d).\]
Here, we used the fact that $(I_0)_{\mfk}$ admits a smooth $\cO_{\mfk}$-integral model with connected special fiber (e.g. parahoric group schemes, cf. \cite{HainesRapoport08}).
Finally, from $\Nm_n\delta=c'\epsilon c'^{-1}$, one easily checks that
\begin{equation} \label{eq:geom_conj_at_p}
\Nm_{n'}b'=\epsilon^k c_{n'}(=p^{n'\nu_1}d^{-1}\sigma^{n'}(d)).
\end{equation}
It follows from this equality and the definition \cite[4.3]{Kottwitz85} that when $b'\in I_0(\mfk)$, $\nu_1\in\Hom_{\mfk}(\mathbb{D},I_0)$ is the Newton homomorphism of $b'$. 

(3) Clearly, we can choose the element $b$ of (1) in $T(\mfk)$ (by using an unramified conjugate of $\xi_p$ under $T(\Qpb)$; so, the Newton homomorphism $\nu_b$ is regarded to be attached to $b\in T(\mfk)$, i.e. as an element of $X_{\ast}(T)_{\Q}^{\Gal_{\Qp}}$, and as such, for every $s\gg1$ there exists $e_{s}\in \ker(w_T)$ such that $\Nm_{s}b=e_sp^{s\nu_b}$.
Let $b'\in G_{\epsilon}(\Qpnr)$ be defined by the admissible pair $(\phi,\epsilon)$ as in (2), except that $b'$ does not necessarily lie in $G_{\epsilon}^{\mathrm{o}}(\Qpnr)$.
Then, according to Remark \ref{rem:two_different_b's}, their $\sigma$-conjugacy classes in $B(G_{\epsilon})$ are equal; in particular, if $G_{\epsilon}$ is connected, the claim in question follows immediately from (1), (2), as $b'$ is basic in $G_{\epsilon}$.
Hence, in view of (\ref{eq:geom_conj_at_p}), for any $k\gg1$, there exists $d'\in G_{\epsilon}(\mfk)$ such that 
\[p^{n'\nu_b}e_{n'}=\Nm_{n'}b=p^{k\nu_{\epsilon}}\cdot d'^{-1}\sigma^{n'}(d')\]
holds in $G_{\epsilon}(\mfk)$ with $n'=nk$. Since $d'^{-1}\sigma^{n'}(d')$ belongs to $T(\mfk)$ ($\nu_{\epsilon}$ maps into the center of $G_{\epsilon}^{\mathrm{o}}$) and lies in any neighborhood of $1$ in $G_{\epsilon}^{\mathrm{o}}(\mfk)$ for $k\gg1$, this is possible only when $n\nu_b=\nu_{\epsilon}$.
\end{proof}

%%%%%%%%%%%%%%%%%%%%
\textsc{Proof of Proposition \ref{prop:phi(delta)=gamma_0_up_to_center} continued.} 
Clearly, all the statements remain intact under any conjugation $\Int(g)\ (g\in G(\Qb))$ with the property that $\Int(g)(\gamma_0)\in G(\Q)$, thus we may and do assume (by Lemma \ref{lem:LR-Lemma5.23})
that $(\phi,\gamma_0)$ is well-located in a maximal $\Q$-torus $T$ of $G$ (which is necessarily elliptic over $\R$); in this case, the assumption says that $\gamma_0$ lies in the (unique maximal) compact subgroup of $T/Z(G)(\A_f^p)$.

(1) As $\delta_{kd}=\delta_k^{d}$, it is enough to show that for some $k\in\N$ (divisible by $n$), the image of $\phi(\delta_k)\cdot \gamma_0^{-\frac{k}{n}}$ in $G^{\ad}(\Q)$ is a torsion element. For that, we use the fact that for any linear algebraic group $G$ over a number field $F$, $G(F)$ is discrete in $G(\A_F)$, so for any compact subgroup $K\subset G(\A_F)$, $G(F)\cap K$ will be finite, particularly, a torsion group. 
We will check that for every place $v$ of $\Q$, the image of $\phi(\delta_k)\cdot \gamma_0^{-\frac{k}{n}}$ in $T/Z(G)(\Qv)$ lies in the maximal compact (open) subgroup of $T/Z(G)(\Qv)$. Recall that for an $F$-torus $T$ and any finite place $v$ of $F$, the maximal compact subgroup $H$ of $T(F_v)$ equals 
\[\bigcap_{\chi\in X^{\ast}(T),\ F_v-\text{rational}}\ker(\mathrm{val}_v\circ\chi),\]
where $\mathrm{val}_v$ is the (normalized) valuation on $F_v$. 
For every finite place $l\neq p$, the image of $\gamma_0$ in $T/Z(G)(\Ql)$ is a unit (i.e. lies in a compact subgroup) by assumption, and so is $\phi(\delta_k)$ by definition of $\delta_k$ (in fact, $\phi(\delta_k)$ is itself a unit in $T(\Ql)$ for every $l\neq p$). As $T/Z(G)$ is anisotropic over $\R$, the claim is trivial for the archimedean place. Hence, it suffices to show that for every $\Qp$-rational character $\chi$ of $T/Z(G)$, $|\chi(\phi(\delta_k))|_p=|\chi(\gamma_0^{\frac{k}{n}})|_p$. In fact, we will show this for $\Qp$-rational characters $\chi$ of $T$.
Choose a finite Galois CM-extension $L$ of $\Q$ and $m\in\N$ such that $\phi$ factors through $\fP(L,m)$. Then, for all sufficiently large $k\in\N$ divisible by $[L:\Q]n$ and for any $\Qp$-rational character $\chi$ of $T$, one has 
\[ |\chi(\phi(\delta_k))|_p=p^{-\langle\chi,\nu_{\phi(\delta_k)}\rangle}=p^{-\frac{k}{n}\langle\chi,\nu_{\gamma_0}\rangle}=|\chi(\gamma_0)|_p^{\frac{k}{n}}. \]
due to Lemma \ref{lem:equality_of_two_Newton_maps}, (3) (for the second equality). 

(2) The additional assumption tells us that $|\chi(\phi(\delta_k))|_{\infty}=|\chi(\gamma_0^{\frac{k}{n}})|_{\infty}$ for every $\Q_{\infty}$-rational character $\chi$ of $S$, and also implies that $\gamma_0\in G(\A_f^p)$ itself lies in a compact open subgroup of $G(\A_f^p)$.
Hence, by the argument of (1), $\phi(\delta_k)\cdot \gamma_0^{-\frac{k}{n}}\in G(\Q)$ is a torsion element.

(3) It is well-known that the stated condition implies that for any maximal $\Q$-torus $T_0$ of $G$, elliptic over $\R$, $T_0(\Q)$ is discrete in $T_0(\A_f)$; this is the condition what Kisin called \emph{the Serre condition for $T_0$}, \cite[(3.7.3)]{Kisin17}. Then, again we resort to the argument of (1).
\end{proof}

%%%%%%%%%%%%%%%%%%%%
\begin{lem}  \label{lem:canonical_decomp_of_epsilon} 
For any $\Q$-group $T$ of multiplicative type such that $\Q\operatorname{-}\mathrm{rk}(T)=\R\operatorname{-}\mathrm{rk}(T)$, there exist a positive integer $s$ and elements $\pi_0$, $t\in T(\Q)$ such that
\begin{align*} 
%\label{eq:virtual_admissible_pair} 
(a)\quad & \epsilon^s=\pi_0 t,\\ 
(b)\quad &\pi_0\in K_l\text{ for all }l\neq p, \nonumber \\ 
(c)\quad &t\in K_p, \nonumber 
\end{align*}
where for each finite place $v$ of $\Q$, $K_v$ denotes the maximal compact subgroup of $T^{\mathrm{o}}(\Q_v)$.
The pair $(\pi_0,t)$ is uniquely determined by $\epsilon$, up to taking simultaneous powers. In particular, the construction of $(\pi_0,t;s)$ is functorial in $(T,\epsilon)$: $f:(T,\epsilon) \rightarrow (T',\epsilon')$ is a morphism of pairs as above, $f$ matches the corresponding elements $(\pi_0,t)$, $(\pi_0',t')$ (for the same $s$)
\end{lem}

Clearly, if one wants, one may further assume that $\pi_0,t\in T^{\mathrm{o}}(\Q)$.

\begin{proof}
(cf. \cite[Lem. 10.12]{Kottwitz92})
The uniqueness up to taking simultaneous powers of a pair of elements $(\pi_0,t)$ satisfying the properties (a) - (c) is an easy consequence of the property that $T(\Q)$ is discrete in $T(\A_f)$ (implied by $\Q\operatorname{-}\mathrm{rk}(T)=\R\operatorname{-}\mathrm{rk}(T)$). For existence, it follows from the same property that the canonical map $\varphi:T(\Q)\rightarrow X:=\oplus_{v\neq\infty} (T(\Qv)/K_v)$ has finite kernel and cokernel, so if we consider the two elements $(a_v)_v,\ (b_v)_v\in X$ where $a_v=1$, $b_v=\epsilon\mod K_v$ for $v\neq p$, and $a_v=\epsilon\mod K_p$, $b_p=1$, then some (common) power $(a_v)^r$, $(b_v)^r$ are the images of elements $a$, $b$ of $T(\Q)$. As $\epsilon^r$ and $ab$ have the same images under $\varphi$, some powers of them are equal: $\epsilon^s=\pi_0t$ where $\pi_0:=a^{s/r}, t:=b^{s/r}$.
\end{proof}

%%%%%%%%%%%%%%%%%%%%
\begin{prop} \label{prop:canonical_decomp_of_epsilon}
Suppose that $\Q\operatorname{-}\mathrm{rk}(Z(G))=\R\operatorname{-}\mathrm{rk}(Z(G))$.
For an admissible pair $(\phi,\epsilon)$, let $T_{\epsilon}^{\phi}$ denote the $\Q$-subgroup (of multiplicative type) of $I_{\phi}$ generated by $\epsilon$ and $(\pi_0^{\phi},t^{\phi})$ elements of $T_{\epsilon}^{\phi}(\Q)$ attached to $(T_{\epsilon}^{\phi},\epsilon)$ by Lemma \ref{lem:canonical_decomp_of_epsilon} (for some $s^{\phi}\in\N$). 
Then the followings hold.

(1) For any field $k\supset\Q$ and $g\in G(\bar{k})$ with $\Int(g)(\epsilon)\in G(k)$, the elements $\Int(g)(\pi_0^{\phi})$, $\Int(g)(t^{\phi})$ belong to the $k$-subgroup of $G_{k}$ generated by $\Int(g)(\epsilon)$, and if $k\subset \Qb$, they are the elements (in $\Int g(T_{\epsilon}^{\phi}(\Q))=T_{\epsilon'}^{\phi'}(\Q)$) attached to $(\phi',\epsilon'):=\Int(g)(\phi,\epsilon)$ (for the same $s^{\phi}\in\N$).

(2) Suppose that $\epsilon\in G(\Q)$ and let $T_{\epsilon}^G$ be the $\Q$-subgroup of $G$ generated by $\epsilon$ with associated elements $\pi_0^G,t^G$ of $T_{\epsilon}^G(\Q)$ (for some $s^G\in\N$).
Then, as subsets of $G(\Qb)\ (\supset I_{\phi}(\Q))$, we have $T_{\epsilon}^{\phi}(\Q)=T_{\epsilon}^G(\Q)$, and $\pi_0^{\phi}=\pi_0^G$, $t^{\phi}=t^G$ when $s^{\phi}=s^{G}$; in this case, we simply write $T_{\epsilon}$, $\pi_0,t,s$.

(3) For any sufficiently large $k\in\N$, the pair $(\phi,(\pi_0^{\phi})^k)$ is also admissible. In particular, when $(\phi,\epsilon)$ is well-located, the map $\phi^{\Delta}:P\rightarrow G$ factors through $T_{\epsilon}$.
\end{prop}

\begin{proof}
We observe that when $\Q\operatorname{-}\mathrm{rk}(Z(G))=\R\operatorname{-}\mathrm{rk}(Z(G))$, for any admissible pair $(\phi,\epsilon)$, we also have $\Q\operatorname{-}\mathrm{rk}(T_{\epsilon})=\R\operatorname{-}\mathrm{rk}(T_{\epsilon})$, since $(I_{\phi}/Z(G))_{\R}$ is an $\R$-subgroup of the compact inner form of $G^{\ad}_{\R}$ defined by $\phi(\infty)\circ\zeta_{\infty}$. 

Now, let $(\phi,\epsilon)$ be an arbitrary admissible pair.
From Lemma \ref{lem:Zariski_group_closure}, we know the followings: when $\epsilon':=\Int g(\epsilon)\in G(k)$, $\Int g$ induces a $k$-isomorphism between the $k$-subgroups of $G_k$ generated by $\epsilon$ and $\epsilon'$, and when $k=\Q$, it also induces $\Q$-isomorphisms between the $\Q$-groups $T_{\epsilon}\subset I_{\phi}$ and $T_{\epsilon'}\subset I_{\phi'}$, where $\phi':=\Int g\circ\phi$; this proves (1) by uniqueness up to powers of the pair $(\pi_0,t)$ (Lemma \ref{lem:canonical_decomp_of_epsilon}).
If $(\phi,\epsilon)$ is also well-located, the $\Qb$-embedding $(T_{\epsilon})_{\Qb}\hookrightarrow (I_{\phi})_{\Qb}\isom I_{\Qb}\hookrightarrow G_{\Qb}$ (\ref{eq:inner-twisting_by_phi}) induces a $\Q$-isomorphism $T_{\epsilon}\isom T_{\epsilon}^G$.
Therefore, if $\epsilon\in G(\Q)$ and $(\phi',\epsilon'):=\Int g(\phi,\epsilon)\ (g\in G(\Qb))$ is well-located, we have $\Q$-isomorphisms
\[ T_{\epsilon}^{\phi} \isom  T_{\epsilon'}^{\phi'} \cong T_{\epsilon'}^G \isom T_{\epsilon}^G,\]
fixing the Zariski dense subset $\{\epsilon^n\}_{n\in\Z}$ (of both $T_{\epsilon}^{\phi}(\Q)$ and $T_{\epsilon}^G(\Q)$), thus when both are regarded as subgroups of $G(\Qb)$, the induced map $T_{\epsilon}^{\phi}(\Q) \isom T_{\epsilon}^G(\Q)$ is the identity.
By uniqueness up to powers of the pair $(\pi_0,t)$ (Lemma \ref{lem:canonical_decomp_of_epsilon}), this proves the claim of (2); in particular, $\pi_0^{\phi}$, $t^{\phi}$ (a priori, elements of $G(\Qb)$) both lie in $G(\Q)$. 

In view of this discussion, to prove statement (3), we may assume (by Lemma \ref{lem:LR-Lemma5.23}) that $(\phi,\epsilon)$ is \emph{nested} in a maximal $\Q$-torus $T$ of $G$ that is elliptic over $\R$  (i.e. $\phi=\psi_{T,\mu_h}$ for some $h\in X\cap\Hom(\dS,T_{\R})$ and $\epsilon\in T(\Q)$); thus, $T_{\epsilon}$ is a $\Q$-subgroup of $T$ via the inner-twisting $(I_{\phi})_{\Qb}\isom I_{\Qb}$ (\ref{eq:inner-twisting_by_phi}). 
In this situation, we establish statement (3), assuming that $t\in T(\Qp)_1=T(\Qp)\cap\ker(w_{T})$, which is allowed, as $T(\Qp)_1$ is a subgroup with finite index of the maximal compact subgroup of $T(\Qp)$. We only need to check condition (c) of Def. \ref{defn:admissible_pair}. For $v=l\neq p$, $\phi(l)\circ\zeta_l$ is conjugate under $T(\Qlb)$ to the canonical neutralization $\xi_l$ of $\fG_T(l)$ (Lemma \ref{lem:properties_of_psi_T,mu}), thus as $\epsilon\in T(\Q)$, condition (c) for $v=l$ holds.
At $v=p$, we choose $u\in T(\Qpb)$ such that $\Int(u)\circ\phi(p)\circ\zeta_p$ is unramified, say the inflation of a $\Qpnr/\Qp$-Galois gerb morphism $\theta^{\nr}:\fD\rightarrow\fG^{\nr}_{T_{\Qp}}$, from which we obtain $b\in T(\Qpnr)$ by $b\sigma=\theta^{\nr}(s_{\sigma})$. Since $(\phi,\epsilon)$ is admissible, by Lemma \ref{lem:Kottwitz84-a1.4.9_b3.3}, there exists $c\in G(\mfk)$ such that $c(\epsilon^{-1}\cdot(b\sigma)^n)c^{-1}=\sigma^n$, where $n$ is the level of $(\phi,\epsilon)$. Since $t$ lies in the parahoric subgroup $T(\Qp)_1$ of $T(\Qp)$, we can find $t_p\in T(\mfk)$ such that $t=t_p^{-1}\sigma^{ns}(t_p)$ \cite[Prop.3]{Greenberg63}. So, we have $(ct_p)(\pi_0^{-1}\cdot(b\sigma)^{ns})(ct_p)^{-1}=\sigma^n$, as was asserted. 
Next, since for every $k\gg1$, $(\phi,\pi_0^k\in G(\Q))$ is also an admissible pair with $\pi_0^k$ lying in a compact subgroup of $G(\A_f^p)$, we can apply Prop. \ref{prop:phi(delta)=gamma_0_up_to_center} by the assumption $\Q\operatorname{-}\mathrm{rk}(Z(G))=\R\operatorname{-}\mathrm{rk}(Z(G))$ to see that the image of $\phi^{\Delta}$ is generated by $\pi_0^k$ for any sufficiently large $k\in\N$. 
\end{proof}

%%%%%%%%%%%%%%%%%%%%
\subsubsection{}
For the next discussion, 
it is also necessary to use Tate hypercohomology groups $\widehat{\mathbb{H}}^i\ (i\in\Z)$ of (bounded) complexes of discrete $\mathcal{G}$-modules for a finite group $\mathcal{G}$: they are defined by means of either the complete (standard) resolution of the trivial $\Z[\mathcal{G}]$-module $\Z$ or hypercochains in the usual manner (cf. \cite[$\S$1]{Koya90}).%% 
\footnote{For finite $\mathcal{G}$, the Tate hypercohomology $\widehat{\mathbb{H}}^i(\mathcal{G},-)$ factors through the stable module category $\mathcal{T}(\mathcal{G})$ of the group algebra $\Z[\mathcal{\mathcal{G}}]$, and as such equals the cohomological functor $A^{\bullet}\mapsto \Hom_{\mathcal{T}(\mathcal{G})}(\Z,A^{\bullet}[i])$ on that triangulated category.}
In this work, we will be only interested in the bounded complexes $A^{\bullet}$ of discrete $\mathcal{G}$-modules \emph{whose positive terms are zero}, in which case one has 
\[ \widehat{\mathbb{H}}^i(\mathcal{G},A^{\bullet})=\begin{cases} 
\mathbb{H}^0(\mathcal{G},A^{\bullet})/\Nm_{\mathcal{G}}\mathcal{H}^0(A^{\bullet}) &\text{ if } i=0\\
\mathbb{H}^i(\mathcal{G},A^{\bullet}) &\text{ if } i>0, \end{cases} \]
where $\mathcal{H}^0(A^{\bullet})$ denotes the $0$-th cohomology $\mathcal{G}$-module of the complex $A^{\bullet}$ and $\Nm_\mathcal{G}$ is the norm map \cite[Prop.1.2]{Koya90}. 

For a diagonalizable $\C$-group $D$ with (algebraic) action of a finite group $\mathcal{G}$,
the  canonical surjection and injection 
\begin{equation} \label{eq:H^0_H^-1}
\pi_0(D^{\mathcal{G}})\twoheadrightarrow \widehat{\mathbb{H}}^{0}(\mathcal{G},D)=\widehat{\mathbb{H}}^{1}(\mathcal{G},X_{\ast}(D)),\quad \widehat{\mathbb{H}}^{-1}(\mathcal{G},X^{\ast}(D))\hookrightarrow X^{\ast}(D)_{\mathcal{G},\tors},
\end{equation}
identify the canonical duality $X^{\ast}(D)_{\mathcal{G},\tors} \isom \pi_0(D^{\mathcal{G}})^D$ with the duality induced from the cup-product pairing
\[\widehat{\mathbb{H}}^{1}(\mathcal{G},X_{\ast}(D))\otimes\widehat{\mathbb{H}}^{-1}(\mathcal{G},X^{\ast}(D))\rightarrow \widehat{\mathbb{H}}^{0}(\mathcal{G},\Z)=\frac{1}{|\mathcal{G}|}\Z.\]

For a field $k$, either global or local, we let 
\begin{align} \label{eq:C_k}
C_k:=\begin{cases} k^{\times} & \text{ if }k\text{ is local},\\
\A_k^{\times}/k^{\times} & \text{ if }k\text{ is global}.\end{cases}
\end{align}
For a number field $F$ and a bounded complex of $F$-tori $T^{\bullet}=(\cdots\rightarrow T^i \rightarrow \cdots)$ (concentrated in non-positive degrees) and $i\in\Z$, we define
\begin{align*}
\mathbb{H}^i(\A_F/F,T^{\bullet})&:=\mathbb{H}^i(F,T^{\bullet}(C_{\overline{F}})) \\
& =\varinjlim_{E}\mathbb{H}^i(E/F,T^{\bullet}(C_E)),
\end{align*}
where $E$ runs through finite Galois extensions of $F$, 
$T^{\bullet}(C_{\overline{F}})$ denotes the complex of discrete $\Gamma_F$-modules $\cdots\rightarrow T^i(\A_{\overline{F}})/T^i(\overline{F})\rightarrow\cdots$ and $T^{\bullet}(C_E)$ is defined similarly.

%%%%%%%%%%%%%%%%%%%%
\begin{lem} \label{lem:identification_of_Kottwitz_A(H)}
(1) Let $H$ be a (connected) reductive group over a field, either global or local, and $T$ a maximal $k$-torus of $H$; set $\tilde{T}:=\rho_H^{-1}(T)$. If $k'$ is a finite Galois extension of $k$ splitting $T$, the two pairings are compatible:
\[ \xymatrix{ \pi_0(Z(\hat{H})^{\mathcal{G}}) \ar@{->>}[d] & \bigotimes & \pi_1(H)_{\mathcal{G},\tors}  \ar[r] & \Q/\Z \ar@{=}[d] \\
\widehat{\mathbb{H}}^{0}(k'/k,X^{\ast}(H_{\bfab})) & \bigotimes & \widehat{\mathbb{H}}^{1}(k'/k,H_{\bfab}(C_{k'})) \ar[r] \ar@{^(->}[u] & \Q/\Z }, \]
where the bottom pairing is the local/global Tate-Nakayama pairings for the complex $H_{\bfab}=(\tilde{T}\rightarrow T)$ (cf. \cite[A.2.2]{KottwitzShelstad99}).%% 
\footnote{Recall that by our convention, in $H_{\bfab}=\tilde{T}\rightarrow T$ and $X^{\ast}(H_{\bfab})=X^{\ast}(T)\rightarrow X^{\ast}(\tilde{T})$, $T$ and $X^{\ast}(\tilde{T})$ are both placed in degree $0$.}%% 

Moreover, the two vertical maps are bijections if the extension $k'/k$ has sufficiently large degree.
\end{lem}

\begin{proof} 
(1) First, we note that the bottom local/global Tate-Nakayama pairings are perfect pairings: its proof can be easily reduced to the classical Tate-Nakayama duality for tori (cf. proof of (A.2.2) of \cite{KottwitzShelstad99}). 
In view of the canonical maps (\ref{eq:H^0_H^-1}), the existence of the vertical maps and the compatibility of the pairings deduced from the canonical isomorphisms
\begin{align*}
\widehat{\mathbb{H}}^{0}(k'/k,Z(\hat{H}))^D &=\widehat{\mathbb{H}}^{-1}(k'/k,\hat{H}_{\bfab})^D \\
&\stackrel{(a)}{=}\widehat{\mathbb{H}}^{0}(k'/k,X^{\ast}(T)\rightarrow X^{\ast}(\tilde{T}))^D \\
&\stackrel{(b)}{=}\widehat{\mathbb{H}}^{1}(k'/k,H_{\bfab}(C_{k'})) \\
&\stackrel{(c)}{=}\widehat{\mathbb{H}}^{-1}(k'/k,\pi_1(H)), 
%\\ &\stackrel{(e)}{\hookrightarrow}\pi_1(H)_{\mathcal{G},\tors}.
\end{align*}
where the equalities are as follows:

(a) For any two-term complex $D_{-1}\rightarrow D_0$ of diagonalizable $\C$-groups with $\Gal(k'/k)$-action, the exponential sequence $0\rightarrow X_{\ast}(D_i)\rightarrow \mathrm{Lie}(D_i) \rightarrow D_i\rightarrow 0$ gives an isomorphism, for all $i\in\Z$,
\begin{equation} \label{eq:connecting_isom_diagonal_gp}
\widehat{\mathbb{H}}^i(k'/k,D_{-1}\rightarrow D_0) = \widehat{\mathbb{H}}^{i+1}(k'/k,X_{\ast}(D_{-1})\rightarrow X_{\ast}(D_0)),
\end{equation} (First, establish this for a single term complex using the cohomology long exact sequence of the associated exponential sequence. Then, the general case follows from it by an appropriate application of the five lemma). 

(b) For a single term complex, this is the classical Tate-Nakayma duality (cf. \cite{Milne13}). Then, one uses the five lemma in an obvious way.

(c) Apply the five lemma to the isomorphisms from the long cohomology exact sequence attached to the short exact sequence $0\rightarrow X_{\ast}(\tilde{T})\rightarrow X_{\ast}(T)\rightarrow \pi_1(H)\rightarrow 0$ to the long exact sequence 
 \[ \cdots \rightarrow \widehat{\mathbb{H}}^{i}(k'/k,\tilde{T}(C_{k'})) \rightarrow  \widehat{\mathbb{H}}^{i}(k'/k,T(C_{k'})) \rightarrow \widehat{\mathbb{H}}^{i}(k'/k,H_{\bfab}(C_{k'})) \rightarrow  \widehat{\mathbb{H}}^{i+1}(k'/k,\tilde{T}(C_{k'})) \rightarrow \cdots, \]
provided by cup-product with the fundamental class in $H^2(k'/k,C_{k'})$, to see that the same cup-product gives an isomorphism
\[ \widehat{\mathbb{H}}^{i-1}(k'/k,\pi_1(H)) \isom \widehat{\mathbb{H}}^{i+1}(k'/k,H_{\bfab}(C_{k'})). \]

Finally, the two maps (\ref{eq:H^0_H^-1}) become isomorphisms  if the degree $[k':k]$ is sufficiently large \cite[Prop.B.4]{Milne92}.
\end{proof}

%%%%%%%%%%%%%%%%%%%%
\begin{prop} \label{prop:triviality_in_comp_gp} 
Suppose that $(G,X)$ is a Shimura datum of abelian type \cite{Milne94}. 

(1) For any semi-simple element $\epsilon\in G(\Q)$ with centralizer $G_{\epsilon}$, we have
\begin{equation} \label{eq:trivaility_of_Sha^{infty,p}}
\Sha^{\infty,p}(\Q,\pi_0(G_{\epsilon})):=\ker[H^1(\Q,\pi_0(G_{\epsilon}))\rightarrow \prod_{v\neq\infty,p}H^1(\Qv,\pi_0(G_{\epsilon}))]=0.
\end{equation}

(2) If two stable Kottwitz triples are (geometrically) equivalent, they are also stably equivalent.
\end{prop}

\begin{proof}
(1) For a general connected reductive group $G$, the finite $\Q$-group $\pi_0(G_{\epsilon})=G_{\epsilon}/G_{\epsilon}^{\mathrm{o}}$ is canonically isomorphic to a subgroup $C_{\epsilon}$ of $\ker(\rho)(\subset Z(G^{\uc}))$ (for the canonical morphism $\rho:G^{\uc}\rightarrow G^{\der}\hookrightarrow G$) \cite[Lem. 4.5]{Kottwitz82}. To prove the statement, we will use an explicit information on the possible $\Q$-groups for $\pi_0(G_{\epsilon})$ when $G$ comes from an abelian-type Shimura datum $(G,X)$.

We first show that when $G^{\uc}=\prod_{i\in I} G_i$ is the decomposition into simple factors, there exist subgroups $U_i\subset Z(G_i)\ (i\in I)$ such that $\pi_0(G_{\epsilon})=\prod_{i\in I}U_i$. For this, we choose $y\in G^{\uc}$ and $z\in Z(G)$ be such that $\epsilon=\rho(y)z$, and consider the map $\theta:G_{\epsilon}\rightarrow \ker(\rho)$ defined by $\theta(g)=hyh^{-1}y^{-1}$, where $h\in G^{\uc}$ is any element whose image in $G^{\ad}$ is the same as that of $g$. Then one knows (\textit{loc. cit.}) that $\theta(g)$ is independent of the choice of $y$ and $h$, and $\theta: G_{\epsilon}\rightarrow \ker(\rho)$ is a homomorphism with $\ker(\theta)=G_{\epsilon}^{\mathrm{o}}$. We have $\rho^{-1}(G_{\epsilon})=\{h\in G^{\uc}\ |\ hyhy^{-1}\in \ker(\rho)\}$ and $\im(\theta)=\{ hyhy^{-1}\ |\ h\in \rho^{-1}(G_{\epsilon})\}$. If $C_i\ (i\in I)$ denotes the projection to $G_i$ of $\ker(\rho)$, it is clear that $\rho^{-1}(G_{\epsilon})=\prod_{i\in I}\{h_i\in G_i\ |\ h_iy_ih_iy_i^{-1}\in C_i\}$, where $y_i\in G_i\ (i\in I)$ is the $i$-th component of $y$, from which the claim follows. 

Next, when $G^{\uc}=\prod_{i\in I} G_i$ as above, we have $Z(G^{\uc})=\prod_{i\in I}\Res_{F_i/\Q}Z_i$, where $Z_i$ is the center of some absolutely simple (simply-connected) group over a number field $F_i$; each $F_i$ is a totally real field as $(G,X)$ is a Shimura datum \cite{Deligne79}. For abelian-type Shimura datum $(G,X)$, Satake classified such $Z_i$'s \cite[$\S$5]{Satake67}, \cite[Appendix]{Satake80} (cf. \cite[Prop.2.6]{Lee12}): if $Z$, $F$ are $Z_i$, $F_i$ for some $i\in I$,
\begin{equation} \label{eq:simple_factors_abelian_type}
Z = \begin{cases} 
V_{E/F,N}:=\{x\in E^{\times}\ |\ \Nm_{E/F}(x)=1,\ x^N=1\} & \text{ if } G_i \text{ is of Lie type } A,\ D,\\
\qquad \mu_2 & \text{ if } G_i \text{ is of Lie type }  B,\ C,\\ 
\quad \mu_2\times\mu_2,\quad \mu_4,\quad \Res_{E/F}\mu_2 & \text{ if } G_i \text{ is of Lie type }  D,
\end{cases}
\end{equation}
Here, in the first and the third cases, $E/F$ is a quadratic extension (in the type $A$ case, $E/F$ is a CM-field with its totally real subfield of index $2$). One can readily check that any subgroup of $\Res_{F/\Q}Z$ (in particular, each $U_i$ in $\pi_0(G_{\epsilon})=\prod_{i\in I}U_i$) is the Weil restriction to $\Q$ of either again a group of the same kind in this list or $\mu_n$ (over a number field) for some $n\in\N$. Then, to prove (\ref{eq:trivaility_of_Sha^{infty,p}}), it suffices to establish the same statement for any such finite $\Q$-group $U$. This follows from the observation that for such $U$, there exists a finite place $l\neq p$ of $\Q$ such that the map $H^1(\Q,U)\rightarrow H^1(\Ql,U)$ is injective. Indeed, each such $U$ equals $\Res_{K/\Q}V$, where $K$ is a number field and $V$ is a finite $K$-group whose splitting field $K'$ is abelian over $K$ (more precisely, $V$ is either $\mu_n$ or $\mu_2\times\mu_2$ or $V_{E/K,N}$ in (\ref{eq:simple_factors_abelian_type})). Hence, there exists a finite place $v$ of $K$ not dividing $p$ that remains prime in $K'$ (by Cebotarev density theorem)  so that $H^1(K,V)\isom H^1(K_v,V)$; or for $\mu_n$, one can appeal to the stronger fact
\cite[Thm.9.1.9]{NSW08} (one applies the case (ii) with $S$ being the set of all places and $T=S\backslash \{\infty,p\}$, by noting that the ``special case'' there cannot occur in our situation).

(2) Suppose that two stable Kottwitz triples $(\gamma_0;\gamma,\delta)$, $(\gamma_0';\gamma',\delta')$ are (geometrically) equivalent, and choose $g\in G(\Qb)$ such that $\gamma_0'=g\gamma_0g^{-1}$. Then the cohomology class in $H^1(\Q,\pi_0(G_{\gamma_0}))$ of the cocycle $g^{-1}\cdot{}^{\tau}g\in Z^1(\Q,G_{\gamma_0})$ lies in $\Sha^{\infty,p}(\Q,\pi_0(G_{\epsilon}))$, so is trivial by (1), which amounts to saying that $g_0$ and $g_0'$ are stably conjugate.
\end{proof}

%%%%%%%%%%%%%%%%%%%%
\begin{lem} \label{lem:isom_of_H^1(Ql)_of_inner-twist}
Let $G$ be a connected reductive group over a field $k$ such that $H^1(k,H)=\{1\}$ for every simply-connected semi-simple $k$-group $H$. Let $\epsilon, \epsilon'\in G(k)$ be elements which are stably conjugate. Put $I_0:=G_{\epsilon}^{\mathrm{o}}$ and $I_0':=G_{\epsilon'}^{\mathrm{o}}$.
If $g\in G(\bar{k})$ satisfies that $\epsilon'=\Int(g)(\epsilon)$ and $g^{-1}\cdot{}^{\tau}g\in Z^1(k,I_0)$ for every $\tau\in\Gal(\bar{k}/k)$, $\Int(g)$ induces isomorphisms 
\[H^1(k,I_0)\isom H^1(k,I_0'),\quad H^1(k,G_{\epsilon})\isom H^1(k,G_{\epsilon'}).\] 
\end{lem}

\begin{proof}
Note that $I_0$ and $I_0'$ are inner-twists of each other and $\Int(g)$ induces a $k$-isomorphism $\pi_0(G_{\epsilon})\isom \pi_0(G_{\epsilon'})$. In particular, we have an isomorphism $(Z(I_0^{\uc})\rightarrow Z(I_0))\isom (Z(I_0'^{\uc})\rightarrow Z(I_0'))$ of crossed modules of algebraic groups.
Considering the maps induced by $\Int(g)$ between the exact sequence of Galois cohomology groups
$1=H^1(k,I_0^{\uc}) \rightarrow H^1(k,I_0) \rightarrow  \mathbb{H}^{i}(k,Z(I_0^{\uc})\rightarrow Z(I_0))$ to its analogue for $\epsilon'$, we see that $\Int(g)$ gives an isomorphism (of abelian groups) $H^1(k,I_0)\isom H^1(k,I_0')$. Then, it is easy to deduce the second isomorphism by examining the maps induced by $\Int(g)$ between the exact sequence 
\[ \cdots\rightarrow H^0(k,\pi_0(G_{\epsilon})) \rightarrow H^1(k,I_0) \rightarrow H^1(k,G_{\epsilon}) \rightarrow H^1(k,\pi_0(G_{\epsilon}))\]
and its analogue for $\epsilon'$.
\end{proof}

Following Kisin \cite[(4.4)]{Kisin17}, we introduce some Galois cohomology notations. For a crossed module of algebraic $\Q$-groups $H'\rightarrow H$, we set
\begin{equation} \label{eq:Sha^{infty}}
\Sha^{\infty}(\Q,H'\rightarrow H)=\ker\left[H^1(\Q,H'\rightarrow H)\rightarrow H^1(\R,H'\rightarrow H)\right].
\end{equation}
When $H$ is a \emph{connected} reductive $\Q$-group, the natural map 
\[ \Sha^{\infty}(\Q,H):=\Sha^{\infty}(\Q,\{e\}\rightarrow H)\ \rightarrow\ \Sha^{\infty}(\Q,H^{\uc}\rightarrow H) \] 
is a bijection, thus $\Sha^{\infty}(\Q,H)$ is an abelian group in a natural way. For a general \emph{$\Q$-subgroup} $H$ of $G$, we introduce a pointed set
\[ \Sha^{\infty}_G(\Q,H):=\ker\left[ \Sha^{\infty}(\Q,H) \rightarrow \Sha^{\infty}(\Q,G)\right].\]
If $H$ is connected, by the above fact this is an abelian group.
We also define $\Sha^{\infty}_G(\Q,H_1)$ for an inner-twist $H_1$ of a \emph{connected} reductive subgroup $H$ of $G$ as follows.
Let $I_0$ be a \emph{connected} reductive $\Q$-subgroup of $G$ and $I_1$ an inner twist of $I_0$.
Then, we define
\begin{equation} \label{eq:Sha^{infty}_G}
\Sha^{\infty}_G(\Q,I_1):=\ker\left[ \Sha^{\infty}(\Q,I_1)\isom \Sha^{\infty}(\Q,I_0) \rightarrow \Sha^{\infty}(\Q,G)\right],
\end{equation}
where the isomorphism $\Sha^{\infty}(\Q,I_1)\isom \Sha^{\infty}(\Q,I_0)$ arises from the canonical isomorphisms (see the proof of Lemma \ref{lem:abelianization_exact_seq}, cf. \cite[Lem.4.4.3]{Kisin17}):
\begin{align} \label{eq:Kisin17_Lem.4.4.3}
\Sha^{\infty}(\Q,I_1)\isom \Sha^{\infty}(\Q,I_{1\bfab}) \isom \Sha^{\infty}(\Q,I_{0\bfab})\isom \Sha^{\infty}(\Q,I_0)
\end{align}

The first statement of the next theorem is the effectivity criterion of Kottwitz triple alluded before.

%%%%%%%%%%%%%%%%%%%%
%%%%%%%%%%%%%%%%%%%%
\begin{thm} \label{thm:LR-Satz5.25} 
Keep the assumptions of Theorem \ref{thm:LR-Satz5.21}.
Let $(\gamma_0;\gamma=(\gamma_l)_{l\neq p},\delta)$ be a stable Kottwitz triple with trivial Kottwitz invariant.
Suppose that one of the following two conditions holds: (a) $Z(G)$ has same ranks over $\Q$ and $\R$, or (b) the weight homomorphism $w_X$ is rational and $\gamma_0$ is a Weil $p^n$-element of weight $-w_X$, where $n$ is the level of $(\gamma_0;(\gamma_l)_{l\neq p},\delta)$ (cf. Prop. \ref{prop:phi(delta)=gamma_0_up_to_center}).

If $Y_p(\delta)\neq \emptyset$ (\ref{eq:Y_p(delta)}), the Kottwitz triple $(\gamma_0;\gamma,\delta)$ is \emph{effective}, that is, there exists an admissible pair $(\phi,\epsilon)$ giving rise to it. In this case, the number $i(\gamma_0;\gamma,\delta)$ of non-equivalent admissible pairs giving rise to the triple $(\gamma_0;\gamma,\delta)$ equals the cardinality of the set 
\begin{align*}
\Sha_G(\Q,I_{\phi,\epsilon})^+
&:=\im \left[ \widetilde{\Sha}_G(\Q,I_{\phi,\epsilon})^+ \rightarrow H^1(\Q,I_{\phi,\epsilon}) \right]
\end{align*}
for \emph{any} such admissible pair $(\phi,\epsilon)$ (giving rise to the triple), where $I_{\phi,\epsilon}$ denotes the centralizer of $\epsilon$ in $I_{\phi}$ and $\widetilde{\Sha}_G(\Q,I_{\phi,\epsilon})^+:=\ker[\Sha^{\infty}_G(\Q,I_{\phi,\epsilon}^{\mathrm{o}})\rightarrow H^1(\A,I_{\phi,\epsilon})]$.
\end{thm}

Recall that ``having trivial Kottwitz invariant'' is our shortened expression for the condition that there exist elements $(g_v)_v\in G(\bar{\A}_f^p)\times G(\mfk)$ satisfying conditions (\ref{eq:stable_g_l}), (\ref{eq:stable_g_l}) such that the associated Kottwitz invariant $\alpha(\gamma_0;\gamma,\delta;(g_v)_v)$ vanishes.

%%%%%%%%%%%%%%%%%%%%
\begin{rem} \label{rem:LR-Satz5.25} 
(1) \emph{By this theorem}, we see that the number $|\Sha_G(\Q,I_{\phi,\epsilon})^+|$, being equal to $i(\gamma_0;\gamma,\delta)$, depends only on the effective Kottwitz triple $(\gamma_0;\gamma,\delta)$, not on the choice of an admissible pair giving rise to it. When $G_{\gamma_0}$ is connected so that $I_{\phi,\epsilon}$ is also connected (see below for why), the set $\Sha_G(\Q,I_{\phi,\epsilon})^+$ is identical to the group denoted by $\Sha_G(\Q,G_{\gamma_0})$ in \cite{Kisin17} or to the set appearing in \cite[Lem. 5.24]{LR87} if $G^{\der}=G^{\uc}$. For this set, one knows that it is a finite abelian group, depends only on $\gamma_0$, and remains invariant under inner twist of $G_{\gamma_0}$, but in the general case the author does not know either of these facts for our set $\Sha_G(\Q,I_{\phi,\epsilon})^+$ (except for finiteness which follows indirectly from its appearance in the formula in Thm. \ref{thm:Kottwitz_formula:Kisin}). 

(2) Compare this theorem with Satz 5.25 of \cite{LR87} which asserts (in the case where the level is hyperspecial and $G^{\der}=G^{\uc}$) that a Kottwitz triple with trivial Kottwitz invariant is effective if the condition $\ast(\epsilon)$ holds. Our condition $Y_p(\delta)\neq \emptyset$ is more natural and is what one really needs. Indeed, according to Remark \ref{rem:admissible_pair} (and the discussion in the introduction), Langlands-Rapoport conjecture implies that the cardinality of the set $\sS_{\mbfK}(\F_{q^m})$ (for each $m\in\N$) is a sum of certain quantities, the sum being over all \emph{effective} Kottwtiz triples (up to equivalence) with trivial Kottwitz invariant (or over the Kottwitz triples satisfying condition $\ast(\epsilon)$ with trivial Kottwitz invariant if one resorts to \cite[Satz5.25]{LR87}).
But, since one easily sees that $Y_p(\delta)\neq \emptyset$ when the corresponding summand is non-zero, the effectivity criterion of Kottwitz triple (Theorem \ref{thm:LR-Satz5.25}) tells us that in the summation one may as well take simply \emph{all} Kottwitz triples with trivial Kottwitz invariant (see Theorem \ref{thm:Kottwitz_formula:LR} and its proof).
\end{rem}

\begin{proof} 
By Theorem \ref{thm:LR-Satz5.21}, there exists an admissible pair $(\phi_1,\epsilon)$ that is nested in a special Shimura sub-datum $(T,h)$, i.e. $\phi_1=\psi_{T,\mu_h}$ and $\epsilon\in T(\Q)$ (and satisfies condition $(\heartsuit)$) and also such that $\epsilon$ is stably conjugate to $\gamma_0$; thus we may assume that $\gamma_0=\epsilon$. Then, the restriction of $\phi_1$ to the kernel of $\fP$ is determined by $\epsilon$ alone, and its image $\im(\phi_1^{\Delta})$ lies in the center of $G_{\epsilon}$. Indeed, if $T_{\epsilon}$ is the $\Q$-subgroup of $G$ generated by $\epsilon\in G(\Q)$ and $(\pi_0,t)$ are the elements of $T_{\epsilon}(\Q)$ determined by $\epsilon$ as in Prop. \ref{prop:canonical_decomp_of_epsilon} (for some $s\in\N$), we have $\phi_1^{\Delta}(\delta_k)=\pi_0^{k/n}\in T_{\epsilon}(\Q)$ for any sufficiently large $k$ (Prop. \ref{prop:phi(delta)=gamma_0_up_to_center}). Hence, one can twist $G_{\epsilon}$ and $G_{\epsilon}^{\mathrm{o}}$ via $\phi_1$ (they are the inner twists of these groups by the cocycle $(g_{\rho}^{\ad})_{\rho}\in Z^1(\Q,(G_{\epsilon}^{\mathrm{o}})^{\ad})$, where $\phi(q_{\rho})=g_{\rho} \rho$ and $g_{\rho}^{\ad}$ is the image of $g_{\rho}$ in $(G_{\epsilon}^{\mathrm{o}})^{\ad}$. Clearly, the resulting twist of $G_{\epsilon}$ equals $I_{\phi_1,\epsilon}$ (recall that $I_{\phi_1}$ itself is the twist of $I_{\phi_1^{\Delta}}:=Z_G(\im(\phi_1^{\Delta}))$ via $\phi_1$,  cf. (\ref{eq:inner-twisting_by_phi})). We also see that if $G_{\gamma_0}$ is connected, so is $I_{\phi_1,\epsilon}$.

The next step in the proof is to modify $\phi_1$ using a suitable cocycle in $Z^1(\Q,I_{\phi_1,\epsilon})$, to get a new admissible morphism $\phi$ such that $(\phi,\epsilon)$ becomes an admissible pair which produces the given Kottwitz triple $(\gamma_0;(\gamma_l)_{l\neq p},\delta)$. By construction, this modification (of an admissible morphism $\phi_1$ by a cocycle in $Z^1(\Q,I_{\phi_1,\epsilon})$) does not change the restriction of $\phi_1$ to the kernel; in particular, we still have $\epsilon\in \mathrm{Aut}(\phi)$. Under the assumption $G^{\der}=G^{\uc}$, such modification was carried out in the proof of Satz 5.25 in \cite{LR87} and this part of that proof does not require any condition on the level subgroup. In the next lemma, we provide its generalization that works without the restriction $G^{\der}=G^{\uc}$.

%%%%%%%%%%%%%%%%%%%%
\begin{lem} \label{lem:LR-Lem5.26,Satz5.25}
Let $(\phi_1,\epsilon)$ be a well-located admissible pair.
Let $I_1$ (resp. $H_1$) be the twist of $I_0:=G_{\epsilon}^{\mathrm{o}}$ (resp. of $H_0:=G_{\epsilon}$) by $\phi_1$ (i.e. $H_1=I_{\phi_1,\epsilon}$ and $I_1=I_{\phi_1,\epsilon}^{\mathrm{o}}$).

(1) For any cochain $a=\{a_{\sigma}\}$ on $\Gal(\Qb/\Q)$ with values in $H_0(\Qb)$, the map $\phi:\fP\rightarrow\fG_G$ defined by $\phi|_P=\phi_1|_P$ and $\phi(q_{\rho})=a_{\rho}\phi_1(q_{\rho})$ is a morphism of Galois gerbs over $\Q$ if and only if $a$ is a cocycle in $Z^1(\Q,H_1)$; in this case, we write $\phi=a\phi_1$.

(2) For $a\in Z^1(\Q,I_{\phi_1})$, the morphism $\phi=a\phi_1$ is admissible if and only if its cohomology class $[a]\in H^1(\Q,I_{\phi_1})$ lies in 
\begin{equation} \label{eq:Sha^{infty}_G} 
\Sha^{\infty}_G(\Q,I_{\phi_1}):=\ker[\Sha^{\infty}(\Q,I_{\phi_1})\isom \Sha^{\infty}(\Q,I_{\phi_1^{\Delta}})
\rightarrow H^1(\Q,G_{\bfab})], 
\end{equation}  
where $I_{\phi_1^{\Delta}}:=Z_G(\im(\phi_1^{\Delta}))$, the isomorphism $\Sha^{\infty}(\Q,I_{\phi_1})\isom \Sha^{\infty}(\Q,I_{\phi_1^{\Delta}})$ is (\ref{eq:Kisin17_Lem.4.4.3}) and the map $H^1(\Q,I_{\phi_1^{\Delta}})\rightarrow H^1(\Q,G_{\bfab})$ is (\ref{eq:abelianization_from_Levi}).

If $a\in Z^1(\Q,H_1)$, the morphism $\phi=a\phi_1$ is admissible if and only if the class $[a]\in H^1(\Q,H_1)$ 
lies in the subset
\begin{equation} \label{eq:Sha^{infty}_G'}
 \Sha^{\infty}_G(\Q,H_1):=\ker[H^1(\Q,H_1)\rightarrow H^1_{\ab}(\Q,G)\oplus H^1(\R,G')],
\end{equation}
where $G'$ is the inner form of $G_{\R}$ defined by $\phi_1(\infty)\circ\zeta_{\infty}$ (so $(G')^{\ad}$ is compact and $(I_{\phi_1})_{\R}$ is a subgroup of $G'$ in a natural way).
%Moreover, in this case, the pair $(\phi=a\phi_1,\epsilon)$ is admissible if further one of the following two conditions is satisfied: (i) the image of $[a]\in \Sha^{\infty}_G(\Q,H_1)$ in $H^1(\Qp,H_1)$ is trivial, or (ii) 
In this case, the pair $(\phi=a\phi_1,\epsilon)$ is admissible if moreover
the localization $a(p)\in Z^1(\Qp,H_1)$ lies in $Z^1(\Qp,I_1)$ and property $(\heartsuit)$ of Thm. \ref{thm:LR-Satz5.21} holds for $(\phi_1,\epsilon)$; then the admissible pair $(\phi,\epsilon)$ also enjoys the property $(\heartsuit)$.

(3) For any $a\in \Sha_G(\Q,H_1)^+$,
the admissible pairs $(\phi_1,\epsilon)$, $(a\phi_1,\epsilon)$ have equivalent associated Kottwitz triples. 

(4) If $(\epsilon;\gamma=(\gamma_l),\delta)$ is a stable Kottwitz triple with trivial Kottwitz invariant, there exists a cocycle $a=\{a_{\sigma}\}\in Z^1(\Q,I_1)$ such that the pair $(\phi=a\phi_1,\epsilon)$ is an admissible pair giving rise to $(\epsilon,(\gamma_l),\delta)$.
\end{lem}

When $H_1$ is connected (i.e. $H_1=I_1$), as $\ker[H^1(\R,I_1)\rightarrow H^1(\R,G')]=0$ \cite[Lem. 5.14, 5.28]{LR87}, the set (\ref{eq:Sha^{infty}_G'}) becomes $\Sha^{\infty}_G(\Q,I_1)$ (\ref{eq:Sha^{infty}_G}), which is known to be an abelian group \cite[4.4.2]{Kisin17}.

\begin{proof} 
(1) This is a straightforward verification: the cocycle condition (for the $\Q$-structure of $H_1$) amounts to $\phi:\fP\rightarrow \fG_G$ being a group homomorphism. 

(2) This statement is Lemma 5.26 of \cite{LR87} under the assumption $G^{\der}=G^{\uc}$ (so that $I_0=H_0$ and $I_1=H_1$, and $H^1(\Q,G^{\ab})$ is used instead of $H^1_{\ab}(\Q,G)=H^1(\Q,G_{\bfab})$).%%
\footnote{In fact, our statement is weaker than this lemma, where Langlands and Rapoport claim (\textit{loc. cit.} p.197, line +6) that the pair $(\phi=a\phi_1,\epsilon)$ is admissible if and only if $[a]\in \Sha^{\infty}_G(\Q,I_1)$ (i.e. they do not require our additional condition). Unfortunately we could not verify their claim. But, our weaker claim as given is sufficient for our purpose.}
We adapt its proof (on p196, line -14, \textit{loc. cit.}).
Note that $H_1$ is a $\Q$-subgroup of $I_{\phi_1}$ (as both are twists of $G_{\epsilon}\subset Z_G(\phi_1^{\Delta})$ via $\phi_1$). Set $\tilde{I}_{\phi_1^{\Delta}}:=\rho^{-1}(I_{\phi_1^{\Delta}})$, where $\rho:G^{\uc}\rightarrow G$ is the canonical morphism.

First, by Lemma \ref{lem:abelianization_exact_seq} (applied to the twist  $\tilde{I}_{\phi_1}\rightarrow I_{\phi_1}$ of $\tilde{I}_{\phi_1^{\Delta}}\rightarrow I_{\phi_1^{\Delta}}$ via $\phi_1$), the triviality of the image of $[a]$ in $H^1_{\ab}(\Q,G)$ is the same as that $[a]\in\im(H^1(\Q,\tilde{I}_{\phi_1})\rightarrow H^1(\Q,I_{\phi_1}))$ (and also the same as that $[a]\in\im(H^1(\Q,\tilde{H}_1)\rightarrow H^1(\Q,H_1))$ if $[a]\in H^1(\Q,H_1)$).
By Lemma \ref{lem:isom_of_monoidal_functors_into_croseed_modules}, this condition is equivalent to the condition that $\phi_{\widetilde{\ab}}$ is conjugate-isomorphic to $\psi_{\mu_{\widetilde{\ab}}}$, i.e. to condition (1) of Def. \ref{defn:admissible_morphism} for $\phi=a\phi_1$.

Secondly, it is easy to see that $\phi(\infty)\circ\zeta_{\infty}$ is conjugate to $\xi_{\infty}$ if and only if the image of $[a]\in H^1(\R,I_{\phi_1})$ in $H^1(\R,G')$ is trivial.
We note that since $I_{\phi_1}$ is connected, the natural map $H^1(\R,I_{\phi_1})\rightarrow H^1(\R,G')$ is injective \cite[Lem. 5.14, 5.28]{LR87}, \cite[Lem. 4.4.5]{Kisin17}.

Next, we show that for $a\in \Sha^{\infty}_G(\Q,I_{\phi_1})$, the morphism $\phi=a\phi_1$ is already admissible: we have to check the remaining conditions (2), (3) of Def. \ref{defn:admissible_morphism}. 
We fix $(g_l)_{l\neq p}\in X^p(\phi_1)$; then, for each finite $l\neq p$, one has $g_l^{-1}\phi_1(P(\Ql))g_l\in G(\Ql)$ as $g_l^{-1}\phi_1(P(\Ql))g_l$ commutes with $g_l^{-1}\phi_1(q_{\rho})g_l=\rho$ for all $\rho\in\Gamma_l$. Set $I_l:=Z_{G_{\Ql}}(\Int(g_l^{-1})\circ\phi_1^{\Delta})\subset G_{\Ql}$. Then, as $I_{\phi_1}$ is the inner twist of $I_{\phi_1^{\Delta}}$ by $\phi_1$ (i.e. for $g\in I_{\phi_1}(\Qb)=I_{\phi_1^{\Delta}}(\Qb)$ and $\rho\in\Gal(\Qb/\Q)$, $\rho(g)=\phi_1(q_{\rho})g \phi_1(q_{\rho})^{-1}$) and $g_l^{-1}I_{\phi_1^{\Delta}}g_l=I_l$ is a $\Ql$-group, $\Int(g_l^{-1})$ induces a $\Ql$-isomorphism from $(I_{\phi_1})_{\Ql}$ to the inner twist of $I_l$ via the canonical trivialization $\Int(g_l^{-1})\circ \phi_1(l)\circ\zeta_l=\xi_l$, i.e. to $I_l$ itself. Since the image of $[a]$ in $H^1_{\ab}(\Q,G)$ is trivial by assumption and $H^1(\Ql,G^{\uc})=\{1\}$, the localization of $a(l)$ at $l$ satisfies
\[[a(l)]\in \ker[H^1(\Ql,I_{\phi_1})\isom H^1(\Ql,I_l)\rightarrow H^1(\Ql,G)].\] 
Then, for any $h_l\in G(\Qlb)$ such that $g_l^{-1}a_{\tau}g_l=h_l\cdot{}^{\tau}h_l^{-1}\in I_l$ for $\tau\in\Gamma_l$, one has $g_lh_l\in X^p(\phi)$.
Further, by a similar argument it is easy to see that when $a\in Z^1(\Q,H_1)$ (not just in $Z^1(\Q,I_{\phi_1})$), the pair $(\phi,\epsilon\in I_{\phi}(\Q))$ satisfies condition (3) of admissible pair for every finite place $l\neq p$ (Def. \ref{defn:admissible_pair}).

At $p$, if $H$ denotes the centralizer (in $G_{\Qp}$) of the maximal $\Qp$-split torus in the center of $G_{\epsilon}$ and $\xi_p:=\phi_1(p)\circ\zeta_p$, then $\xi_p^{\Delta}$ maps into the center of $H$ by Prop. \ref{prop:canonical_decomp_of_epsilon}, (3), and thus we have 
\begin{equation} \label{eq:H_centralizes_J_0,epsilon} 
(G_{\epsilon})_{\Qp}\subset H\subset J:=Z_{G_{\Qp}}(\xi_p^{\Delta}). 
\end{equation} 
In particular, we can twist both $\tilde{H}:=\rho^{-1}(H)$ and $\tilde{J}:=\rho^{-1}(J)$ via $\xi_p$ to obtain $\tilde{H}_{\xi_p}$ and $\tilde{J}_{\xi_p}$. Then, we claim that 
\[H^1(\Qp,\tilde{H}_{\xi_p})=H^1(\Qp,\tilde{J}_{\xi_p})=\{1\}.\]
Of course, since $\tilde{H}$ and $\tilde{J}$ are connected reductive groups over a $p$-adic local field, this is equivalent to vanishing of $H^1(\Qp,\tilde{H})$ and $H^1(\Qp,\tilde{J})$.
The latter is proved in the discussion proceeding Lemma 5.18 of \cite{LR87} when $G^{\der}=G^{\uc}$, and the same argument works in our case, so we will just give a sketch of the proof for $H$ only. Let $A$ be the connected center of $H\cap G^{\der}_{\Qp}$ (a split $\Qp$-torus) and $T\subset G^{\der}_{\Qp}$ the centralizer of a maximal $\Qp$-split torus of $G^{\der}_{\Qp}$ containing $A$. The simple roots $\alpha_1,\cdots,\alpha_s$ of $T$ in $G$ can be divided into two sets, the roots $\alpha_1,\cdots,\alpha_r$ of $T$ in $H$ and the others $\alpha_{r+1},\cdots,\alpha_s$. Let $\{\omega_1,\cdots,\omega_s\}$ be the corresponding fundamental weights which form a basis of $X^{\ast}(\tilde{T})$ for $\tilde{T}=\rho^{-1}(T)$, and $R\subset \tilde{T}$ be the kernel of $\{\omega_1,\cdots,\omega_r\}$. Note that $R$ is an induced torus since the basis $\{\omega_{r+1},\cdots,\omega_s\}$ of $X^{\ast}(R)$ is permuted under the canonical action of $\Gal(\Qpb/\Qp)$ which equals the naive action as $G^{\der}_{\Qp}$ is quasi-split.
Now, the claim follows from the fact that $\tilde{H}$ is the semi-direct product of a simply connected semi-simple group over $\Qp$ (i.e. $H^{\uc}$) and $R$.
%cocharacter group has a basis consisting of the coroots of $(G^{\der}_{\Qp},T)$ such that if $\{\omega_1,\cdots,\omega_s\}$ is the set of fundamental weights do not vanish on $A$ (as $G^{\der}_{\Qp}$ is quasi-split, the naive action of $\Gal(\Qpb/\Qp)$ on $X_{\ast}(T)$ permutes such basis of coroots), from which the claim follows.
%The roots of $(G^{\der}_{\Qp},T)$ divide into two subsets $R_1=\{\alpha_1.\cdots,\alpha_r\}$, $R_2=\{\alpha_{r+1},\cdots,\alpha_s\}$, where $R_1$ is the set of roots of $H\cap G^{\der}_{\Qp}$, and $X^{\ast}(T_1)$ has a basis consisting of the fundamental weights $\{\omega_{r+1},\cdots,\omega_s\}$ which are permuted by $\Gal(\Qpb/\Qp)$.
% 

We can find $f_{1p}\in X_p(\phi_1)\cap J(\Qpb)$ so that $\xi_p':=\Int(f_{1p}^{-1})\circ\xi_p$ is an unramified morphism. Indeed, clearly it suffices to show this for any conjugate of $\xi_p$ that is well-located. So we may assume that $\phi_1$ is special admissible, in which case the claim is established in the proof of Lemma \ref{lem:LR-Lemma5.2}. Let $\tilde{J}_{\xi_p'}$ be the twist of $\tilde{J}$ by $\xi_p'$. By the same argument as above, we have $H^1(\Qp,\tilde{J}_{\xi_p'})=\{1\}$.
Therefore, there exists $h_p\in  \tilde{J}(\Qpb)$ such that for $\tau\in\Gal(\Qpb/\Qp)$,
\begin{equation} \label{eq:a_{tau}'}
a_{\tau}':=f_{1p}^{-1}a_{\tau}f_{1p}=h_p\cdot \xi_p'(s_{\tau})h_p^{-1}\xi_p'(s_{\tau})^{-1}.
\end{equation}
(Here and in the rest of proof, we write $h_p$ for $\rho(h_p)$ when there is no danger of confusion). 
Then, for $f_{p}:=f_{1p} h_p$, as $\phi(p)\circ\zeta_p=a\xi_p$, we have 
\[\Int(f_{p}^{-1})\circ\phi(p)\circ\zeta_p(s_{\tau})=\xi_p'(s_{\tau})h_p^{-1}\xi_p'(s_{\tau})^{-1}h_p\cdot\Int(h_p^{-1})\circ\xi_p'(s_{\tau})=\xi_p'(s_{\tau})\]
which shows that $\phi$ is an admissible morphism and $f_{p}\in X_p(\phi)$. 
This proves the claims in (2) on admissibility of $\phi=a\phi_1$ in the two cases $a\in Z^1(\Q,I_{\phi_1})$, $a\in Z^1(\Q,H_1)$.

%Now, suppose that we are in the situation (i), namely that the image of $[a]\in \Sha^{\infty}_G(\Q,H_1)$ in $H^1(\Qp,H_1)$ is trivial. Then, in the above, we could have chosen $h_p$ satisfying (\ref{eq:a_{tau}'}) from $H_0(\Qpb)=G_{\epsilon}(\Qpb)$. So, we have $\Int(f_{p}^{-1})(\epsilon)=\Int(f_{1p}^{-1})(\epsilon)$, and \[ \Int(f_{p}^{-1})(\epsilon)^{-1}\cdot \Int(f_{p}^{-1})\circ\phi(p)\circ\zeta_p(s_{\tilde{\sigma}})^n=\Int(f_{1p}^{-1})(\epsilon)^{-1}\cdot\xi_p'(s_{\tilde{\sigma}})^n,\] which shows that condition (3) of Def. \ref{defn:admissible_pair} holds at $p$ and $f_{p}=f_{1p}h_p\in X_p(\phi,\epsilon)$.

Now, we consider the case that $a\in Z^1(\Q,H_1)$ and its class in $H^1(\Q,H_1)$ lies in the subset 
$\Sha^{\infty}_G(\Q,H_1)$ (\ref{eq:Sha^{infty}_G'}).
We furthermore assume that the localization $a(p)\in Z^1(\Qp,H_1)$ lies in $Z^1(\Qp,I_1)$ and property $(\heartsuit)$ of Thm. \ref{thm:LR-Satz5.21} holds for $(\phi_1,\epsilon)$. 
Then, since $H^1(\Qpnr,I_1)=\{1\}$, we may assume (by replacing $a(p)$ with a cohomologous one) that $a(p)\in Z^1(\Qp,I_1)$ is unramified, i.e. $a_{\tau}$ is induced from a normalized cocycle in $Z^1(L_{n'},I_1)$ for some $n'\in\N$.

We choose $f_{1p}$ from $H(\Qpb)$ such that $\xi_p':=\Int(f_{1p}^{-1})\circ\xi_p$ is an unramified morphism, and define $\xi_p'$, $h_p$, $f_{p}\in H(\Qpb)$ by the same recipe as before from this $f_{1p}$ (using that $H^1(\Qp,\tilde{H}_{\xi_p})=\{1\}$); then, since both $\xi_p'$ and $a'_{\tau}$ are unramified, we must have $h_p\in \tilde{H}(\Qpnr)$. Also, we set
\[\epsilon_1':=\Int(f_{1p}^{-1})(\epsilon),\quad \epsilon':=\Int(f_{p}^{-1})(\epsilon)=\Int(h_p^{-1})(\epsilon_1');\] 
then, $\epsilon_1',\epsilon'$ are elements of $H(\Qpnr)$ (as they centralize unramified morphisms of $\Qpb/\Qp$-Galois gerbs) which are conjugate under $H(\Qpnr)$. Then for any lifting $\tilde{\sigma}\in \Gal(\Qpb/\Qp)$ of the Frobenius automorphism $\sigma\in \Gal(\Qpnr/\Qp)$, by admissibility of $(\phi_1,\epsilon)$ there exists  $c_1\in H(\mfk)$ such that 
\begin{equation} \label{eq:c_1}
c_1(\epsilon_1'^{-1}\xi_p'(s_{\tilde{\sigma}})^n)c_1^{-1}=\tilde{\sigma}^n 
\end{equation}
(as $\xi_p'$ is unramified, it is the inflation of a $\Qpnr/\Qp$-Galois gerb morphism $\theta^{\nr}$ and we have $\xi_p'(s_{\tilde{\sigma}})=\theta^{\nr}(s_{\sigma})$).
Let $\xi_p'(s_{\tilde{\sigma}})^n=b_n\rtimes\tilde{\sigma}^n$ with $b_n\in H(\Qpnr)$. Then, it follows from the existence of $c_1\in H(\mfk)$ satisfying (\ref{eq:c_1}) that $w_H(\epsilon_1')=w_H(b_n)$. Since $\epsilon'$ is conjugate to $\epsilon_1'$ under $H(\mfk)$, we also have $w_H(\epsilon')=w_H(b_n)$. If we can show that for any neighborhood $V$ of the identity in $H(\mfk)$, there exists $t\in\N$ such that $(\epsilon'^{-1}\xi_p'(s_{\tilde{\sigma}})^n)^t\in V\rtimes\tilde{\sigma}^{nt}$, then, we can repeat the arguments of the proof of Thm. \ref{thm:LR-Satz5.21}, (2), to conclude that $(\phi,\epsilon)$ is an admissible pair which again enjoys the property $(\heartsuit)$.

As $a\in  Z^1(\Qp,H_{\xi_p})$, for every $\tau\in \Gal(\Qpb/\Qp)$ and any $n'\in\N$, we have \[(a_{\tau}'\xi_p'(s_{\tau}))^{n'}=a_{\tau^{n'}}'\xi_p'(s_{\tau})^{n'},\] 
thus as both $a_{\tau}'$ and $\xi_p'$ commute with $\epsilon_1'$ and as $a_{\tau}'$ is unramified,
\begin{align} \label{eq:epsilon'{-1}xi'{n}} 
(\epsilon'^{-1}\xi_p'(s_{\tilde{\sigma}})^{n})^t &\stackrel{}{=} h_p^{-1}\epsilon_1'^{-t}(h_p\xi_p'(s_{\tilde{\sigma}})h_p^{-1})^{nt}h_p =h_p^{-1}\epsilon_1'^{-t}(a_{\tilde{\sigma}}'\xi_p'(s_{\tilde{\sigma}}))^{nt}h_p  \\ &\stackrel{}{=} h_p^{-1} a_{\sigma^{nt}}'(\epsilon_1'^{-1}\xi_p'(s_{\tilde{\sigma}})^{n})^th_p \nonumber. \end{align}
So, for $t\gg1$ with $a\in Z^1(L_{nt}/\Qp,I_1)$,
$(\epsilon'^{-1}\xi_p'(s_{\tilde{\sigma}})^{n})^t=(c_1h_p)^{-1}\sigma^{nt}(c_1h_p)$ lies in any given neighborhood of the identity in $H(\mfk)$, as was required.

\textsc{Proofs of (3) and (4)}. Under the assumption $G^{\der}=G^{\uc}$, this claim is established in the proof of Satz 5.25 of \cite{LR87} (more precisely, in the part from the third paragraph on p.195 to the end of the proof). 
This argument again carries over to our general case if one works systematically with $\tilde{I}_0$, $\tilde{I}_1$, and the abelianized cohomology groups $H^1_{\ab}(k,H(C_k))$, instead of, respectively, $I_0\cap G^{\der}$, its twist via $\phi_1$, and the groups $\pi_0(Z(\widehat{H})^{\Gal(\bar{k}/k)})^D$ (where $H$ is some reductive group over a field $k$, either global or local). As the original proof is quite sketchy and our general setting requires careful modifications of the original arguments, again we give a detailed proof.

Let $(\gamma_0;\gamma,\delta)$ be a stable Kottwitz triple such that there exist elements $(g_v)_v\in G(\bar{\A}_f^p)\times G(\mfk)$ satisfying (\ref{eq:stable_g_l}), (iii$'$) which gives trivial Kottwitz invariant, and that $\delta$ satisfies $Y_p(\delta)\neq\emptyset$. First, we look for an admissible pair $(\phi,\epsilon)$ giving rise to $(\gamma_0;\gamma,\delta)$. 
By Thm. \ref{thm:LR-Satz5.21}, after replacing $\gamma_0$ by a rational element stably conjugate to it, we can find an admissible pair $(\phi_1=\psi_{T,\mu_h},\gamma_0\in T(\Q))$ that is nested in a special Shimura sub-datum $(T,h)$ and satisfies the condition $(\heartsuit)$ there. Then there exists a stable Kottwitz triple $(\gamma_0;\gamma_1,\delta_1)$ attached to this special admissible pair $(\phi_1,\gamma_0)$ such that for every $l\neq p$, one may take $\gamma_{1l}:=\gamma_0$ (as $T(\bar{\A}_f^p)\cap X^p(\phi_1)\neq\emptyset$). We will find a cocycle $a\in Z^1(\Q,H_1)$ (in fact, in $Z^1(\Q,I_1)$) such that $(a\phi_1,\gamma_0)$ becomes an admissible pair giving rise to  $(\gamma_0;\gamma,\delta)$ as an attached stable Kottwitz triple. We note that since $(\phi_1,\gamma_0)$ has property $(\heartsuit)$, the pair $(a\phi_1,\gamma_0)$ will be admissible if $[a]\in \ker[H^1(\Q,H_1)\rightarrow H^1_{\ab}(\Q,G)\oplus H^1(\R,G')]$, by (2).

In the following, only assuming that $(\phi:=a\phi_1,\epsilon:=\gamma_0)$ is an admissible pair for $a\in Z^1(\Q,H_1)$, we give an explicit description of an associated (stable) Kottwitz triple $(\gamma_0';\gamma'=(\gamma'_l)_{l\neq p},\delta')$, in terms of $(\gamma_0;\gamma_1=(\gamma_{1l})_{l\neq p},\delta_1)$ and $[a]$ (cf.  \cite[Lem. 5.27]{LR87}); statement (3) will follow from this description. Moreover, we will see that if $a\in Z^1(\Q,I_1)$, there exists a stable Kottwitz triple $(\gamma_0';\gamma'=(\gamma'_l)_{l\neq p},\delta')$ attached to $(\phi=a\psi_{T,\mu_h},\epsilon\in T(\Q))$ such that $\gamma_0'=\epsilon(=\gamma_0)$.

Recall (\autoref{subsubsec:Kottwitz_invariant}) that any stable Kottwitz triple $(\gamma_0;\gamma,\delta)$ together with a choice of elements $(g_v)_v\in G(\bar{\A}_f^p)\times G(\mfk)$ satisfying (\ref{eq:stable_g_l}), (\ref{eq:stable_g_l}) gives rise to an adelic class 
\[(\alpha_v(\gamma_0;\gamma,\delta;g_v))_v\in  \bigoplus_v X^{\ast}(Z(\hat{I}_0)^{\Gamma_v})\] 
whose finite part $(\alpha_v)_{v\neq\infty}$ measures the difference between $\gamma_0$ and $(\gamma,\Nm_n\delta)$ ($n$ being the level of the triple).

For $l\neq p$, it was shown in the proof of (2) that for any $g_{1l}\in X_l(\phi_1,\epsilon)$, 
$[g_{1l}^{-1}a_{\tau}g_{1l}]\in H^1(\Ql,G)$ is trivial and if one chooses $h_l\in G(\Qlb)$ such that 
\[g_{1l}^{-1}a_{\tau}g_{1l}=h_l\cdot{}^{\tau}h_l^{-1}\] 
for every $\tau\in\Gamma_l$, one has $g_{l}':=g_{1l}h_l\in X_l(\phi,\epsilon)$, from which we obtain the $l$-components 
\[\gamma_{1l}:=\Int(g_{1l}^{-1})(\epsilon),\quad \gamma'_l:=\Int(g_{l}'^{-1})(\epsilon)\ \in G(\Ql)\] 
of some (not necessarily stable) Kottwitz triples attached to $(\phi_1,\epsilon)$ and $(\phi,\epsilon)$ respectively. The $G(\Ql)$-conjugacy classes of $\gamma_l$ and $\gamma_l'$ (in the geometric conjugacy classes of $\gamma_0$) are determined by the cohomology classes of the cocycles in $Z^1(\Ql,G_{\gamma_0})$: 
\begin{equation} \label{eq:cocycles_at_l}
g_{1l}\cdot{}^{\tau}g_{1l}^{-1},\quad g_{l}'\cdot{}^{\tau}g_{l}'^{-1}=a_{\tau}\cdot g_{1l}\cdot{}^{\tau}g_{1l}^{-1}.
\end{equation}
(these cohomology classes in $H^1(\Ql,G_{\gamma_0})$ do not depend on the choice of $g_{1l}$ and $h_l$.)

Now, we consider the case that $a(l)\in Z^1(\Ql,I_1)$. If there exists $g_{1l}\in X_l(\phi_1,\epsilon)$ further satisfying that 
\[ g_{1l}\cdot{}^{\tau}g_{1l}^{-1} \in I_0(\Qlb)\] 
for every $\tau\in\Gamma_l$, so that $g_{l}'\cdot{}^{\tau}g_{l}'^{-1}\in I_0(\Qlb)$ as well (this will be the case, for example, if $(\phi_1,\epsilon)$ is nested in a special Shimura sub-dtaum), as $(I_1)_{\Ql}$ is the inner twist of $(I_0)_{\Ql}$ via $\phi_1(l)\circ\zeta_l (\tau)$, i.e. via the cocycle $g_{1l}\cdot{}^{\tau}g_{1l}^{-1}\in Z_1(\Ql,I_0)$, there exists a natural isomorphism 
\[H^1(\Ql,I_1)=H^1(\Ql,(I_1)_{\bfab})\ \isom\ H^1(\Ql,(I_0)_{\bfab})=H^1(\Ql,I_0).\]
%More precisely, as $\Int(g_{1l}^{-1})\circ\phi_1(l)\circ\zeta=\xi_l$, $\Int(g_{1l}^{-1})$ induces a $\Ql$-isomorphism $(H_1)_{\Ql}\isom (G_{\gamma_{1l}})_{\Ql}$, and as $\gamma_{1l}$ and $\gamma_0$ are stably conjugate, there exists an isomorphism $H^1(\Ql,G_{\gamma_{1l}}^{\mathrm{o}}) \isom H^1(\Ql,G_{\gamma_0}^{\mathrm{o}})$ (Lemma \ref{lem:isom_of_H^1(Ql)_of_inner-twist}).
Then the cohomology classes $\alpha_l(\gamma_0;\gamma_{1l};g_{1l})$, $\alpha_l(\gamma_0;\gamma_{l}';g_{l}')$ in $H^1(\Ql,I_0)=H^1(\Ql,(I_0)_{\bfab})$ of the cocyles $g_{1l}\cdot{}^{\tau}g_{1l}^{-1}$, $g_{l}'\cdot{}^{\tau}g_{l}'^{-1}$ satisfy the relation:
\begin{equation} \label{eq:difference_of_alpha_l's}
\alpha_l(\gamma_0;\gamma'_l;g_{l}')=\alpha_l(\gamma_0;\gamma_{1l};g_{1l})+[a(l)]\ \in\ \ker[H^1(\Ql,I_0)\rightarrow H^1(\Ql,G)]
\end{equation}
(Apply \cite{Borovoi98}, Lemma 3.15.1 to $(I_1)_{\bfab}={}_{\psi}(I_0)_{\bfab}$, $\psi:=g_{1l}\cdot{}^{\tau}g_{1l}^{-1}$, $\psi':=a(l)$.)

At $v=p$, we first fix some choices.
As in the proof of (2), we choose $f_{1p}\in H(\Qpb)$ such that $\xi_p':=\Int(f_{1p}^{-1})\circ\xi_p$ is an unramified morphism, and let $h_p$, $f_p=f_{1p}h_p\in H(\Qpb)$ be defined as there. When $\xi_p'(s_{\tilde{\sigma}})=b\rtimes s_{\tilde{\sigma}}$ for $b\in H(\mfk)$,
by construction (of $\delta_1$), there exists $x_1\in G(\mfk)$ such that $x_1(\epsilon_1'^{-1}\xi_p'(s_{\tilde{\sigma}})^n)x_1^{-1}=\tilde{\sigma}^n$ and $\delta_1=x_1b\sigma(x_1^{-1})$; we set 
\[c_{1p}:=x_1f_{1p}^{-1}\in G(\bar{\mfk}),\]
so that $c_{1p}\epsilon c_{1p}^{-1}=x_1\epsilon_1'x_1^{-1}=\Nm_n(\delta_1)$.

Next, from this relation $c_{1p}\epsilon c_{1p}^{-1}=\Nm_n(\delta_1)$, we obtain a cocycle in $Z^1(W_{\Qp},G_{\epsilon}(\bar{\mfk}))$ ($\epsilon:=\gamma_0$):
\begin{equation} \label{eq:cocycle_b_{tau}}
b_{\tau}:=c_{1p}^{-1}\cdot\Nm_{i(\tau)}\delta_1\cdot {}^{\tau}c_{1p},
\end{equation}
where $W_{\Qp}$ is the Weil group of $\Qp$ and $i(\tau)\in\Z$ is defined by $\tau|_{L_n}=\sigma^{i(\tau)}\ (0\leq i(\tau)< n)$.
As $B(G_{\epsilon})$ is naturally identified with $H^1(\langle\sigma\rangle,G_{\epsilon}(\mfk))$, the inflation map induces an injection $B(G_{\epsilon})\hookrightarrow  H^1(W_{\Qp},G_{\epsilon}(\bar{\mfk}))$ which becomes a bijection if $G_{\epsilon}$ is connected  \cite[$\S$1]{Kottwitz85}. 

Similarly, since $\Int(f_p^{-1})\circ\phi(p)\circ\zeta_p=\xi_p'$, by construction of $\delta'$ (as we are assuming that $(\phi,\epsilon)$ is admissible), there exists $x'\in G(\mfk)$ such that $x'(\epsilon'^{-1}\xi_p'(s_{\tilde{\sigma}})^n)x'^{-1}=\tilde{\sigma}^n$ and $\delta'=x'b\sigma(x')^{-1}$. As before,  for $c_{p}':=x'f_{1p}^{-1}$, we obtain a cocycle in $Z^1(W_{\Qp},G_{\epsilon}(\bar{\mfk}))$
\begin{equation} \label{eq:cocycle_b_{tau}'}
b_{\tau}':=c_{p}'^{-1}\cdot\Nm_{i(\tau)}\delta'\cdot {}^{\tau}c_{p}'.
\end{equation}

Then, since $a_{\tau}\cdot f_{1p}\xi_p'(s_{\tau})f_{1p}^{-1}=\phi(p)\circ\zeta_p(s_{\tau})=f_p\xi_p'(s_{\tau})f_p^{-1}$ for any $\tau\in W_{\Qp}$, there is the relation 
\begin{equation} \label{eq:a_{tau}b_{tau}=b_{tau}'}
a_{\tau}\cdot b_{\tau}=b_{\tau}'.
\end{equation}

Now, suppose that $a\in Z^1(\Q,I_1)$ and that $[a]\in \Sha^{\infty}_G(\Q,I_1)$; then, since $(\phi_1,\epsilon)$ satisfies the condition $(\heartsuit)$ of Thm. \ref{thm:LR-Satz5.21}, we know that $(\phi=a\phi_1,\epsilon)$ is also admissible.
Because the admissible pair $(\phi_1,\epsilon)$ is nested in a special Shimura sub-datum, we can find $c_1\in G(\mfk)$ satisfying (iii$'$), i.e. such that
\[c_1\epsilon c_1^{-1}=\Nm_n\delta_1,\ \text{ and }\ b(\epsilon;\delta_1;c_1):=c_1^{-1}\delta_1\sigma(c_1)\in I_0(\mfk)\] 
(Remark \ref{rem:two_different_b's}); this $b(\epsilon;\delta_1;c_1)$ and its $\sigma$-conjugacy class $[b(\epsilon;\delta_1;c_1)]_{I_0}$ in $B(I_0)$ both depend on $c_1$ as well as on the pair $(\epsilon,\delta_1)$, while its $\sigma$-conjugacy class in $G_{\epsilon}(\mfk)$ does not. Also, the $\sigma$-conjugacy class of $\delta_1$ in $G(L_n)$ is completely determined by the $\sigma$-conjugacy class $[b(\epsilon;\delta_1;c_1)]\in B(G_{\epsilon})$.  
We see that under the composite map $B(I_0)\rightarrow B(G_{\epsilon})\hookrightarrow  H^1(W_{\Qp},G_{\epsilon}(\bar{\mfk}))$, the $\sigma$-conjugacy class $[b(\epsilon;\delta_1;c_1)]_{I_0}$ maps to the cohomology class $[b_{\tau}]$, since
\[b(\epsilon;\delta_1;c_1)=c_1^{-1}\delta_1\sigma(c_1)=y_1\cdot b_{\tilde{\sigma}}\cdot \tilde{\sigma}(y_1^{-1}),\] 
where $y_1=c_1^{-1}c_{1p}\in G_{\epsilon}(\bar{\mfk})$ and $\tilde{\sigma}$ is any lifting in $W_{\Qp}$ of $\sigma$.
Then, similarly we can also find $c'$ such that $b(\epsilon;\delta';c'):=c'^{-1}\delta'\sigma(c')$ lies in $I_0(\mfk)$. Indeed, if $v\in G_{\epsilon}(\bar{\mfk})$ is chosen such that $vb_{\tau}{}^{\tau}v^{-1}$ belongs to $I_0(\bar{\mfk})$ for every $\tau\in W_{\Qp}$, so does $vb_{\tau}'{}^{\tau}v^{-1}$ by (\ref{eq:a_{tau}b_{tau}=b_{tau}'}). 
Since $H^1(\langle\sigma\rangle,I_0(\mfk))\isom H^1(W_{\Qp},I_0(\bar{\mfk}))$, there exist $c'\in G(\bar{\mfk})$ and $b'\in I_0(\mfk)$ such that $c'^{-1}\cdot\Nm_{i(\tau)}\delta'\cdot {}^{\tau}c'=\Nm_{i(\tau)}b'$ for every $\tau\in W_{\Qp}$, which implies that $c'\in G(\mfk)$ and $b(\epsilon;\delta';c'):=c'^{-1}\delta'\sigma(c')(=b')\in I_0$.
In this case, it follows from Lemma \ref{lem:equality_of_two_Newton_maps}, (2) and (\ref{eq:a_{tau}b_{tau}=b_{tau}'}) that one has
\begin{equation} \label{eq:difference_of_alpha_p's}
 \kappa_{I_0}(b(\epsilon;\delta';c'))=\kappa_{I_0}(b(\epsilon;\delta_1;c_1))+a(p)\ \text{ and }\ \nu_{I_0}(b(\epsilon;\delta';c'))=\nu_{I_0}(b(\epsilon;\delta_1;c_1)).
\end{equation}
In particular, we conclude that when $a\in Z^1(\Q,I_1)$, with the elements $g_l$, $g_{l}'=g_{1l}h_l$, $\delta_1$, $c_1$, $\delta'$, $c'$ chosen above, 
the triple
\[(\gamma_0,\gamma':=(g_l'^{-1}\epsilon g_l')_l,\delta')\] 
is a stable Kottwitz triple attached to $(\phi=a\psi_{T,\mu_h},\epsilon=\gamma_0\in T(\Q))$ which further satisfies the relations (\ref{eq:difference_of_alpha_l's}), (\ref{eq:difference_of_alpha_p's}).
Also, note that these two relations imply statement (3).

Before continuing with the discussion, we recall the fact (Appendix \ref{sec:abelianization_complex}) that the abelianization complexes $I_{0\bfab}$ and $I_{1\bfab}$ of $I_0$ and $I_1$ are canonically isomorphic in the derived category of commutative algebraic $\Q$-group schemes (since they are both quasi-isomorphic to $Z(I_0^{\uc})\rightarrow Z(I_0)$), and the same is true of $\tilde{I}_{0\bfab}$ and $\tilde{I}_{1\bfab}$.
Now, we set 
\[(\alpha(v):=\alpha_v(\gamma_0;\gamma,\delta;g_v)\cdot \alpha_v(\gamma_0;\gamma_1,\delta_1;g_{1v})^{-1})_v\in  \bigoplus_{v} \mathbb{H}^1(\Qv,I_{0\bfab}), \]
where $(g_v)_v$ are some elements used to define $\alpha_v(\gamma_0;\gamma,\delta;g_v)$ for the given stable Kottwitz triple $(\gamma_0;\gamma,\delta)$ (their choice does not matter), and $g_{1l}$, $g_{1p}:=c_1$ are as described above.
A priori, $\alpha(v)$ is an element of $X^{\ast}(Z(\hat{I}_0)^{\Gamma_v})$, but one easily verifies that it belongs to the subgroup $\mathbb{H}^1(\Qv,I_{0\bfab})\subset X^{\ast}(Z(\hat{I}_0)^{\Gamma_v})_{\tors}$. Indeed, this is clear for $v\neq p$ by definition: for every finite prime $l\neq p$, each $\alpha_l$ itself belongs to $\mathbb{H}^1(\Ql,I_0)=\mathbb{H}^1(\Ql,I_{0\bfab})=X^{\ast}(Z(\hat{I}_0)^{\Gamma_l})_{\tors}$, and $\alpha_{\infty}=0$. For $v=p$, according to Lemma \ref{lem:equality_of_two_Newton_maps}, (2), the Newton quasi-cocharacters $\nu_{I_0}(b(\gamma_0;\delta;g_p))$, $\nu_{I_0}(b(\gamma_0;\delta_1;g_{1p}))\in \mathcal{N}(I_0)$ are equal and do not depend on the choice of $g_p$, $g_{1p}$. We have seen (cf. \cite[4.13, 3.5]{Kottwitz97}) that the set $B(I_0)$ is identified with a subset of $X^{\ast}(Z(\hat{I}_0)^{\Gamma_v})\times\mathcal{N}(I_0)$ via $\kappa_{I_0}\times \nu_{I_0}$ and the image under $\kappa_{I_0}$ of any fibre of $\nu_{I_0}$ is a torsor under $X^{\ast}(Z(\hat{I}_0)^{\Gamma_p})_{\tors}=\mathbb{H}^1(\Qp,I_{0\bfab})$. Clearly, the claim follows from these facts; we also see that $[\alpha(p)]=\kappa_{I_0}(b(\gamma_0;\delta;g_p))-\kappa_{I_0}(b(\gamma_0;\delta_1;g_{1p}))$.

Therefore, in view of these discussions, to prove effectivity of the given stable Kottwitz triple $(\gamma_0;\gamma,\delta)$, it suffices to find a global class $[\widetilde{a}]\in H^1(\Q,\tilde{I}_{1})$ whose localizations $[\tilde{a}(v)]$ go over to $(\alpha(v))$ under the map $\oplus f_v:=\oplus \mu_v\circ \bfab_v$ in the commutative diagram:
\begin{equation}  \label{eq:H^1_{ab}_for_(I'->I)}
 \xymatrix{ \bigoplus_v H^1(\Qv,I_{1}) \ar[r] & \bigoplus_v \mathbb{H}^1(\Qv,I_{1\bfab}) & \mathfrak{K}(I_{1}/\Q)^D  \\ 
 \bigoplus_v H^1(\Qv,\tilde{I}_{1}) \ar[u] \ar[r]^(.45){\bfab} \ar[ru]^{\oplus f_v} & \bigoplus_v \mathbb{H}^1(\Qv,\tilde{I}_{1\bfab}) \ar[u]_{\mu=\oplus_v\mu_v} \ar[r]^(.5){\lambda} & \mathbb{H}^1(\A/\Q,\tilde{I}_{1\bfab}) \ar@{->>}[u]_{\nu}  \\
 \bigoplus_v \mathbb{H}^0(\Qv,\tilde{I}_{1}\rightarrow I_{1}) \ar[r]^(.5){\overline{\bfab}} \ar[u] & \bigoplus_v \mathbb{H}^0(\Qv,G_{\bfab})  \ar[u]_{\xi} \ar@{->>}[r] & \im(\lambda\circ\xi) \ar@{^(->}[u] }
 \end{equation} 
Here, the first two columns (consisting of adelic cohomology groups) are each a part of the cohomology long exact sequence for the crossed module $\tilde{I}_1\rightarrow I_1$ \cite[(3.4.3.1)]{Borovoi98} and for the distinguished triangle (\ref{eq:DT_of_CX_of_tori}) attached to the morphism $\tilde{I}_{1\bfab}\rightarrow I_{1\bfab}$; so, they are exact. The two horizontal maps for $H^1$ between these two columns are the abelianization maps for $I_1$, $\tilde{I}_1$, and the left lower commutative diagram (containing the map $\overline{\bfab}$) is induced from an obvious commutative diagram of crossed modules (replace the complexes of tori $\tilde{I}_{1\bfab}[1]$, $G_{\bfab}$ by the quasi-isomorphic complexes $(\tilde{I}^{\uc}_1\rightarrow \tilde{I}_1)[1]$, $\tilde{I}_1\rightarrow I_1$); in particular, the map $\overline{\bfab}$ is an isomorphism. The map $\lambda$ comes from the long exact sequence for $\tilde{I}_{1\bfab}$:
\[ \cdots \longrightarrow \mathbb{H}^1(\Q,\tilde{I}_{1\bfab}) \longrightarrow \mathbb{H}^1(\A,\tilde{I}_{1\bfab}) \stackrel{\lambda}{\longrightarrow} \mathbb{H}^1(\A/\Q,\tilde{I}_{1\bfab})\longrightarrow \mathbb{H}^{2}(\Q,\tilde{I}_{1\bfab}) \longrightarrow \cdots, \] 
and $\nu$ is the dual of the inclusion $\mathfrak{K}(I_1/F) \hookrightarrow \pi_0(Z(\hat{\tilde{I}}_1)^{\Gamma})$ under the identification $\mathbb{H}^1(\A/\Q,\tilde{I}_{1\bfab})=\pi_0(Z(\hat{\tilde{I}}_1)^{\Gamma})^D$ (Lemma \ref{lem:identification_of_Kottwitz_A(H)}). 
Since $\im[H^1(\Q,\tilde{I}_{1})\rightarrow \bigoplus_v H^1(\Q_v,\tilde{I}_{1})]=\ker(\lambda\circ\bfab)$ \cite[Thm.5.16]{Borovoi98} (cf. \cite[Prop.2.6]{Kottwitz85}), hence it suffices to find an adelic cohomology class 
\[ [\tilde{a}(v)]\in \oplus_vH^1(\Qv,\tilde{I}_{1}) \] 
such that $\alpha(v)=f_v([\tilde{a}(v)])$ for each $v$ and $\lambda\circ\bfab([\tilde{a}(v)])=0$.
We find such class $([\tilde{a}(v)])_v$ as follows. We first note that $(\alpha(v))$ lifts to an element $\tilde{\alpha}=(\tilde{\alpha}(v))_v\in \bigoplus_v \mathbb{H}^1(\Qv,\tilde{I}_{1\bfab})$. In view of the cohomology exact sequence attached to the distinguished triangle (\ref{eq:DT_of_CX_of_tori}), this amounts to vanishing of the image of $(\alpha(v))_v$ in $\bigoplus_v \mathbb{H}^1(\Qv,G_{\bfab}) (\subset \bigoplus_v X^{\ast}(Z(\hat{G})^{\Gamma_v})_{\tors})$, which then follows from the fact (\ref{eq:restriction_of_alpha_to_Z(hatG)}) that for any Kottwitz triple $(\epsilon;\gamma,\delta)$, the image of the invariant $\alpha_v(\epsilon;\gamma,\delta;g_v)$ in $X^{\ast}(Z(\hat{G})^{\Gamma_v})$ is independent of the triple.
Now, by assumption (of vanishing of Kottwitz invariants), we have $\nu\circ\lambda(\tilde{\alpha})=0$ in $\mathfrak{K}(I_1/\Q)^D$. 
Hence, by Lemma \ref{lem:proof_of_Kottwitz86_Thm.6.6} below (which asserts exactness of the third column of (\ref{eq:H^1_{ab}_for_(I'->I)})), we may further assume that $(\tilde{\alpha}(v))_v\in \ker(\lambda)$.
Since the (local) abelianization map $\bfab_v$ is surjective for every place $v$ (even bijective for finite $v$) \cite[Thm.5.4]{Borovoi98}, we can find $([\tilde{a}(v)])_{v\neq\infty} \in \bigoplus_{v\neq\infty} H^1(\Qv,\tilde{I}_{1})$ mapping to $(\tilde{\alpha}(v))_{v\neq\infty}$. For $v=\infty$, we use the condition that $\alpha(\infty)=0$, by which and the exactness of the vertical sequence for $\tilde{I}_{1\bfab}\rightarrow I_{1\bfab}$, we find an element $\bar{\alpha}(\infty)\in \mathbb{H}^0(\R,\tilde{I}_{1\bfab}\rightarrow I_{1\bfab})$ mapping to $\tilde{\alpha}(\infty)$. Then, using that $\overline{\bfab}$ is an isomorphism, we define $\tilde{a}(\infty)$ to be the the image of $\bar{\alpha}(\infty)$ under the map $\mathbb{H}^0(\R,\tilde{I}_{1}\rightarrow I_{1})\rightarrow H^1(\R,\tilde{I}_1)$. The cohomology class $([\tilde{a}(v)])_v$ thus found is the one that we are looking for. 
\end{proof}

\textsc{Proof of Thm. \ref{thm:LR-Satz5.25} continued.} 
It remains to prove the statement on $i(\gamma_0;\gamma,\delta)$. Suppose given two admissible pairs $(\phi_1,\epsilon_1)$, $(\phi,\epsilon)$ whose associated Kottwitz triples are equivalent. By conjugation, we may assume that $(\phi_1,\epsilon_1)$ is nested in a special Shimura sub-datum and $\epsilon=\epsilon_1\in G(\Q)$; so, as was explained in the beginning of this proof, one has $\phi^{\Delta}=\phi_1^{\Delta}$ and $\phi=a\phi_1$ for a cocycle $a_{\tau}\in Z^1(\Q,H_1)$ whose cohomology class lies in $\ker[H^1(\Q,H_1)\rightarrow H^1_{\ab}(\Q,G)\oplus H^1(\R,G')]$, by Lemma \ref{lem:LR-Lem5.26,Satz5.25}, (2).

We claim that for every finite place $v$, the localization $[a(v)]\in H^1(\Qv,H_1)$ is trivial. Indeed, according to the previous discussion (cf. (\ref{eq:cocycles_at_l})), there exist $(g_l)_{l\neq p}\in X^p(\phi_1,\epsilon)$ and $(h_l)_{l\neq p}\in G(\bar{\A}_f^p)$ such that $g_l^{-1}a_{\tau}g_l=h_l\cdot{}^{\tau}h_l^{-1}$ for all $l\neq p$ and $\tau\in\Gamma_l$, thereby $(g_l':=g_lh_l)_{l\neq p}\in X^p(\phi,\epsilon)$ as well.
Also, $\gamma_{1l}=\Int(g_l^{-1})(\epsilon)$, $\gamma_l'=\Int(g_l'^{-1})(\epsilon)$ are the $l$-components of some Kottwitz triples attached to $(\phi_1,\epsilon)$, $(\phi,\epsilon)$.
Since $\gamma_{1l}$, $\gamma_l'$ are conjugate under $G(\Ql)$ by assumption, we may modify $h_l$ (with right translation by an element of $G(\Ql)$) so that $\Int(g_l^{-1})(\epsilon)=\Int(g_l'^{-1})(\epsilon)$ and so $x_l:=g_lh_l^{-1}g_l^{-1}\in G_{\epsilon}(\Qlb)$. Then, one easily checks that for every $\tau\in\Gamma_l$,
\[a_{\tau}=x_l^{-1}\cdot\phi_1(l)\circ\zeta_l(\tau)\cdot x_l\cdot \phi_1(l)\circ\zeta_l(\tau)^{-1}.\]
In particular, $[a]_{H_1}\in H^1(\Q,H_1)$ maps to zero in $H^1(\Qv,\pi_0(H_1))$ for every $v\neq p,\infty$, and thus by Prop. \ref{prop:triviality_in_comp_gp}, we may assume that $a\in Z^1(\Q,I_1)$ for $I_1=H_1^{\mathrm{o}}$ (note that as $\pi_0(G_{\epsilon})$ is commutative (proof of Prop. \ref{prop:triviality_in_comp_gp}), it is isomorphic to its inner twist $\pi_0(H_1)$).
At $v=p$, by standard recipe (Subsec. (\autoref{subsubsec:K-triple_attached_to_adm.pair}) and proof of Lemma \ref{lem:delta_from_b&gamma_0}), the admissible pairs $(\phi_1,\epsilon=\gamma_0)$, $(\phi,\epsilon=\gamma_0)$  give rise to $\delta_1,\delta'\in G(L_n)$ such that $(\gamma_0;(\gamma_{1l})_l,\delta_1)$, $(\gamma_0;(\gamma_l')_l,\delta')$ are associated Kottwiz triples. 
%% -------------------- A direct proof of the triviality of $[a(p)]\in H^1(\Qp,H_1)$
We recall the cocycles $b_{\tau}$ (\ref{eq:cocycle_b_{tau}}), $b_{\tau}'$ (\ref{eq:cocycle_b_{tau}'}) in $Z^1(W_{\Qp},H_0(\bar{\mfk}))$ whose constructions involve $\delta_1$, $\delta'$, and some other elements $c_{1p},c_p'\in G(\bar{\mfk})$. Now, if $\delta_1$ and $\delta'$ are $\sigma$-conjugate in $G(L_n)$, say $\delta'=d\delta_1\sigma(d^{-1})$, after replacing $x'$ by $d^{-1}x'$ (in the construction of $\delta'$) and $c_p'$ by $d^{-1}c_p'$, we may assume that $\delta'=\delta_1$ and $c_p'\epsilon c_p'^{-1}=\Nm_n(\delta')$, and it follows that $e:=c_{1p}^{-1}c_p'\in H_0(\bar{\mfk})$ and
\[b_{\tau}'=e^{-1}\cdot b_{\tau}\cdot {}^{\tau}e\quad (\tau\in W_{\Qp}).\] Hence, $a_{\tau}=e^{-1}\cdot (b_{\tau}\cdot {}^{\tau}e\cdot b_{\tau}^{-1})$ becomes trivial in $H^1(W_{\Qp},H_1(\bar{\mfk}))$. But, $[a_{\tau}]_{H_1}\in H^1(\Qp,H_1)$ lies in the subset $H^1(\Qpnr/\Qp,H_1)\hookrightarrow H^1(\Qpb/\Qp,H_1)$ since $[a_{\tau}]_{I_1}\in H^1(\Qp,I_1)$ is so, by the Steinberg's theorem $H^1(\Qpnr,I_1)=\{1\}$. Hence, the cohomology class $[a_{\tau}]\in H^1(\Qp,H_1)$ is trivial, since the inflation maps for $W_{\Qp}\twoheadrightarrow \Gal(L_{n'}/\Qp)$ for $n'\in\N$ induce an injection $H^1(\Qpnr/\Qp,H_1)\hookrightarrow H^1(W_{\Qp},H_1(\bar{\mfk}))$.

Also, $[a(\infty)]_{H_1}\in \ker[H^1(\Q,H_1)\rightarrow H^1(\R,G')]$ implies that $[a(\infty)]_{I_1}=0$ in $H^1(\R,I_1)$, since $(I_1)_{\R}\subset G'$ is connected so that $\ker[H^1(\Q,I_1)\rightarrow H^1(\R,G')]=\{1\}$
\cite[Lem. 5.14, 5.28]{LR87} (or, \cite[Lem.4.4.5]{Kisin17}). 
Hence, one has 
\[[a]_{I_1}\in \ker[\Sha^{\infty}_G(\Q,I_1)\rightarrow \prod_{v} H^1(\Qv,H_1)],\] 
where $v$ runs through \emph{all} places of $\Q$ (including $\infty$).
Finally, it follows from the previous discussion that the subset of $H^1(\Q,I_1)$ consisting of $[a]$'s such that the admissible pairs $(\phi_1,\epsilon)$ and $(a\phi_1,\epsilon)$ produce the same equivalence classes of Kottwitz triples is in bijection with 
\[\im[\ker[\Sha^{\infty}_G(\Q,I_1)\rightarrow \prod_{v} H^1(\Qv,H_1)]\rightarrow H^1(\Q,H_1)]\]
i.e. with $\Sha^{\infty}_G(\Q,H_1)^+$.
\end{proof}

%%%%%%%%%%%%%%%%%%%%
\begin{lem} \label{lem:proof_of_Kottwitz86_Thm.6.6}
Let $G$ be a connected reductive group over a number field $F$ and $I$ an $F$-Levi subgroup of $G$.%%
\footnote{A better notation (for consistency) for such $\Q$-subgroup of $G$ will be $I_0$, but here we used $I$ to simplify notation.}
Set $\tilde{I}:=\rho^{-1}(I)$ for the canonical homomorphism $\rho:G^{\uc}\rightarrow G$, $\Gamma:=\Gal(\overline{F}/F)$, and $\Gamma_v:=\Gal(\bar{F}_v/F_v)$ for any place $v$ of $F$.
The kernel of the natural map 
\begin{equation}
\nu:\mathbb{H}^1(\A_F/F,\tilde{I}_{\bfab})=\pi_0(Z(\hat{\tilde{I}})^{\Gamma})^D\rightarrow \mathfrak{K}(I/F)^D
\end{equation}
equals the image of $\lambda:\mathbb{H}^1(\A_F,\tilde{I}_{\bfab}) \rightarrow \mathbb{H}^1(\A_F/F,\tilde{I}_{\bfab})$ of the kernel of
\[\mu:\mathbb{H}^1(\A_F,\tilde{I}_{\bfab})\rightarrow  \mathbb{H}^1(\A_F,I_{\bfab}). \]
In other words, we have an equality
\[\mathrm{coker}[\mathbb{H}^0(\A_F,G_{\bfab})\rightarrow \mathbb{H}^1(\A_F/F,\tilde{I}_{\bfab})]=\mathfrak{K}(I/F)^D.\]
\end{lem}

\begin{rem} \label{eq:K_Labesse_Kottwitz}
(1) It is easy to see that for a reductive $H$ over $F$ and for each place $v$ of $F$, the map $\mathbb{H}^1(F_v,H)\rightarrow \pi_0(Z(\hat{H})^{\Gamma_v})^D$ that was constructed by Kottwitz in \cite[Thm.1.2]{Kottwitz86} factors through the abelianization map $\mathbb{H}^1(F_v,H)\rightarrow \mathbb{H}^1(F_v,H_{\bfab})$ and the induced map 
\[\mathbb{H}^1(F_v,H_{\bfab})\rightarrow \pi_0(Z(\hat{H})^{\Gamma_v})^D.\]
is a monomorphism (and an isomorphism for non-archimedean $v$), equal to the map in Lemma \ref{lem:identification_of_Kottwitz_A(H)} (the point is that both monomorphisms are induced by Tate-Nakayama duality); this also equals the monomorphism constructed by Borovoi \cite[Prop.4.1, 4.2]{Borovoi98}. 
If $H^{\der}=H^{\uc}$, this map for archimedean $v$ is also surjective, because then $H_{\bfab}$ is quasi-isomorphic to an \emph{$F$-torus} whose dual torus is $Z(\hat{H})$ (i.e. $H^{\ab}=H/H^{\der}$), so that Tate-Nakayama duality for tori applies.

(2) This lemma is mentioned, without proof, in \emph{Remarque} on p.43 of \cite{Labesse99} ($\mathbb{H}^1(\A_F/F,\tilde{I}_{\bfab})$ is the same as $\mathbb{H}^0_{\bfab}(\A_F/F,I\backslash G):=\mathbb{H}^0(\A_F/F,I_{\bfab}\rightarrow G_{\bfab})$ defined there, cf. \cite[Prop.1.8.1]{Labesse99} and (\ref{eq:DT_of_CX_of_tori})). Earlier, this was also observed in the proof of Thm. 6.6 of \cite{Kottwitz86} when $G^{\der}=G^{\uc}$, in which case, according to (1), the maps $\lambda$, $\mu$ become the natural maps: 
\[\lambda:\oplus_v \pi_0(Z(\hat{\tilde{I}})^{\Gamma_v})^D \rightarrow \pi_0(Z(\hat{\tilde{I}})^{\Gamma})^D,\quad \mu: \oplus_v \pi_0(Z(\hat{\tilde{I}})^{\Gamma_v})^D \rightarrow \oplus_v \pi_0(Z(\hat{I})^{\Gamma_v})^D\]
which are defined only in terms of $Z(\hat{\tilde{I}})$, $Z(\hat{I})$, and which are the definitions for these maps used in \cite[Thm.6.6]{Kottwitz86}.%%
\footnote{It is, however, not clear to the author whether the original version of the lemma stated in terms only of $Z(\hat{\tilde{I}})$ and $Z(\hat{I})$ continues to hold beyond the case $G^{\der}=G^{\uc}$.}
\end{rem}

Since we could not find a proof (in the general case) in literatures, we present a proof. In our proof, we work systematically with the abelianized cohomology groups $H^1(k,H_{\bfab}(C_k))$ (for $H=I,\tilde{I},G$) instead of $\pi_0(Z(\hat{H})^{\Gal(\bar{k}/k)})^D$. Also, we use Galois (hyper)cohomology of complexes of $\Q$-tori of length $2$, especially their Poitou-Tate-Nakayama local/global dualities, as expounded in \cite{Borovoi98}, \cite{Demarche11}, \cite[Appendix A]{KottwitzShelstad99}, \cite[Ch.1]{Labesse99}. 

\begin{proof} (of Lemma \ref{lem:proof_of_Kottwitz86_Thm.6.6})
We have two diagrams (the left diagram defines $\mathfrak{K}(I/F)$ by being cartesian and the right one is its dual):
\begin{equation}  \label{eq:dual_of_connecting_map_Z(hat{G})}
\xymatrix{ \pi_0(Z(\hat{\tilde{I}})^{\Gamma}) \ar[r]^(.44){\partial} & H^1(F,Z(\hat{G})) & \pi_0(Z(\hat{\tilde{I}})^{\Gamma})^D \ar@{->>}[d]_{\nu} & H^1(F,Z(\hat{G}))^D \ar[l]_{\partial^D} \ar@{->>}[d]^{i^D} \\ \mathfrak{K}(I/F) \ar@{^(->}[u] \ar[r] & \ker^1(F,Z(\hat{G})) \ar@{^(->}[u]_{i} & \mathfrak{K}(I/F)^D & \ker^1(F,Z(\hat{G}))^D \ar[l] },
\end{equation}
from which we see the equality:
\begin{align*}
\ker(\nu)&=\mathrm{im}(\partial^D)(\ker(i^D)).
\end{align*}
We need to describe the maps $\partial^D$, $i^D$ in terms of the complexes $G_{\bfab}$, $\tilde{I}_{\bfab}$. For that, we recall some facts.

For a connected reductive group $G$ over a field $k$ and a Levi $k$-subgroup $I$ of $G$, the map (\ref{eq:boundary_map_for_center_of_dual})
\[\partial:\pi_0(Z(\hat{\tilde{I}})^{\Gamma})\rightarrow H^1(F,Z(\hat{G}))\] 
is identified in a natural manner with the connecting homomorphism 
\[\widehat{\mathbb{H}}^{-1}(k'/k,\hat{\tilde{I}}_{\bfab}) \rightarrow \widehat{\mathbb{H}}^0(k'/k,\hat{G}_{\bfab}),\]
in the long exact sequence of Tate-cohomology arising from the distinguished triangle (\ref{eq:DT_of_dualCX_of_tori}), where $k'$ is a finite Galois extension of $k$ splitting $T$.
Indeed, we recall \cite[Lem. 2.2, Cor. 2.3]{Kottwitz84b} that the given map
is the connecting homomorphism $\mathrm{Ext}^1_{\Gamma}(X^{\ast}(Z(\hat{\tilde{I}})),\Z)\rightarrow \mathrm{Ext}^2_{\Gamma}(X^{\ast}(Z(\hat{G})),\Z)$ in the long exact sequence of the cohomology $\mathrm{Ext}^{\bullet}_{\Gamma}(-,\Z)$ for the exact sequence of $\Gamma$-modules: $0\rightarrow X^{\ast}(Z(\hat{\tilde{I}}))\rightarrow X^{\ast}(Z(\hat{I}))\rightarrow X^{\ast}(Z(\hat{G}))\rightarrow0$. As was shown in \textit{loc. cit.} (and using that $\pi_0(D^{\Gamma})=\widehat{\mathbb{H}}^{0}(\Gamma,D)$), one has a functorial isomorphism $\mathrm{Ext}^n_{\Gamma}(X^{\ast}(D),\Z)=\widehat{H}^{n-1}(\Gamma,D)\ (n\geq1)$ for any diagonalizable $\C$-group $D$ with $\Gamma$-action. 
Since $Z(\hat{\tilde{I}})[1]=\hat{\tilde{I}}_{\bfab}$, $Z(\hat{G})[1]=\hat{G}_{\bfab}$ (\ref{eq:center_of_complex_dual}), the claim follows.

For a connected reductive group $H$ over $F$ (number field), we $j^i_{H}$ denote the injection (arising from a suitable cohomology long exact sequence for $H_{\bfab}$):
\[j^i_{H}:\mathbb{H}^i(\A_F,H_{\bfab})/\mathbb{H}^i(F,H_{\bfab}) \hookrightarrow \mathbb{H}^i(\A_F/F,H_{\bfab}).\]

Next, there exist natural isomorphisms \cite[11.2.2, 4.2.2]{Kottwitz84b}, \cite[D.2.C]{KottwitzShelstad99}, \cite[Thm.5.13]{Borovoi98}
\begin{align} \label{eq:identification_of_TS-gps}
\ker^1(F,Z(\hat{G}))^D=\ker^1(F,G) =\ker^1(F,G_{\bfab}),
\end{align}
where for $i\geq0$, we define (cf. \cite[(C.1)]{KottwitzShelstad99}, \cite[1.4]{Labesse99})
\begin{align*}
\ker^i(F,G_{\bfab})&:=\ker(H^i(F,G_{\bfab})\rightarrow H^i(\A_F,G_{\bfab})) \\
&=\ker(H^i(F,G_{\bfab})\rightarrow \oplus_vH^i(F_v,G_{\bfab})) \ \text{ if }i\geq1
\end{align*}
(When $i\geq1$, there exists a canonical isomorphism $H^i(\A_F,G_{\bfab})= \oplus_vH^i(F_v,G_{\bfab})$, \cite[Lem.4.5]{Borovoi98}, \cite[Lem. C.1.B]{KottwitzShelstad99}, \cite[Prop.1.4.1]{Labesse99}).

%%%%%%%%%%%%%%%%%%%%
\begin{lem} \label{lem:commutativity_of_duality_diagram} 
There exists a commutative diagram induced by the global and local Tate-Nakayama dualities:
\[ \xymatrix{ \ker^1(F,G_{\bfab}) \ar[r] & \ker^2(F,\tilde{I}_{\bfab}) \\
H^1(F,Z(\hat{G}))^D \ar[r]^(0.55){\partial^D} \ar@{->>}[u]^{i^D} & \pi_0(Z(\hat{\tilde{I}})^{\Gamma})^D  \ar@{->>}[u] \\ 
\mathbb{H}^0(\A_F,G_{\bfab})/\mathbb{H}^0(F,G_{\bfab}) \ar[r]^{\bar{\xi}} \ar@{^(->}[u]^{j^0_G} & \mathbb{H}^1(\A_F,\tilde{I}_{\bfab})/\mathbb{H}^1(F,\tilde{I}_{\bfab})  \ar@{^(->}[u]_{j^1_{\tilde{I}}} } \] 
where $\bar{\xi}$ is induced by the connecting homomorphism $\xi:\mathbb{H}^0(\A_F,G_{\bfab})\rightarrow \mathbb{H}^1(\A_F,\tilde{I}_{\bfab})$ in the long exact sequence of cohomology attached to the distinguished triangle (\ref{eq:DT_of_CX_of_tori}).

The two vertical sequences are short exact sequences. 
\end{lem}

\begin{proof} In view of the discussion above and the isomorphism (\ref{eq:connecting_isom_diagonal_gp}), $\partial^D$ equals the dual of the obvious connecting homomorphism
\[\widehat{\mathbb{H}}^1(k'/k,X^{\ast}(G_{\bfab}))^D \longrightarrow \widehat{\mathbb{H}}^{0}(k'/k,X^{\ast}(\tilde{I}_{\bfab}))^D, \]
Hence, the existence and exactness of each vertical sequence is easily deduced by reading the relevant parts of the Poitou-Tate exact sequence for two-term complexes of tori \cite[Thm.6.1]{Demarche11} (take $\hat{C}=X^{\ast}(H_{\bfab})$ for $H=\tilde{I}, G$ in \textit{loc. cit.}): to identify the upper diagram, use the identification (\ref{eq:identification_of_TS-gps}) and the fact \cite[Lem. C.3.B, C.3.C]{KottwitzShelstad99}) that
\[\ker^1(F,G_{\bfab})=[\ker^0(F,\hat{G}_{\bfab})_{\mathrm{red}}]^D=[\ker^1(F,X_{\ast}(\hat{G}_{\bfab}))]^D.\]
The commutativity of the daigram follows from the compatibility of the global and local Tate-Nakayama dualities.
\end{proof}

Now, as $\mathrm{im}(\xi)=\ker(\mu)$, by Lemma \ref{lem:commutativity_of_duality_diagram} we see that
\begin{align*}
\ker(\nu)&=\mathrm{im}(\partial^D\circ j^0_G)=\mathrm{im}(\nu\circ\xi)\\
&=\mathrm{im}(\nu)(\ker\mu) 
\end{align*}
as was asserted. 
\end{proof}

%%%%%%%%%%%%%%%%%%%%
%%%%%%%%%%%%%%%%%%%%

\section{Proofs of Kottwitz formula: Cardinality of fixed point set of Frobenius-twisted Hecke correspondence}

In this section, we present two proofs of the Kottwitz formula \cite[(3.1)]{Kottwitz90}, \cite[(19.6)]{Kottwitz92}. The first proof assumes the Langlands-Rapoport conjecture, i.e. Conjecture \ref{conj:Langlands-Rapoport_conjecture_ver1}, and is merely a generalization, to arbitrary automorphic sheaves, of the arguments of \cite{LR87} in the constant coefficient case. The second proof is unconditional, relying heavily on the geometric results obtained by Kisin \cite{Kisin17}. But we do not use his version of the Langlands-Rapoprt conjecture \cite[Thm.(0.3)]{Kisin17} which is too weak to invoke the first proof for, and instead we emulate the arguments of the first proof. In both proofs, the effectivity criterion of Kottwitz triple (Thm. \ref{thm:LR-Satz5.25}, Thm. \ref{thm:LR-Satz5.25b2}) is crucial.

The main part of the Kottwitz conjecture is concerned with obtaining a group-theoretic description of the cardinality of the fixed point set of a Hecke correspondence twisted by a Frobenius automorphism acting on each subset of $\sS_{K}(\F)$ indexed by an admissible morphism $\phi$ or on each isogeny class.
In \cite[$\S$1]{Kottwitz84b}, Kottwitz discusses how one can arrive at such description from a description of the mod-$p$-point set as provided by the Langlands-Rapoport conjecture (but at that time his deduction was based on some earlier version of it suggested by Langlands \cite{Langlands76}, \cite{Langlands79}). This argument by Kottwitz was one of the motivations for the work \cite{LR87} and will also serve as a guide for our unconditional proof.

%%%%%%%%%%%%%%%%%%%%
%%%%%%%%%%%%%%%%%%%%

\subsection{Proof of Kottwitz conjecture under Langlands-Rapoport conjecture} \label{subsec:Proof_of_K-formula1}
Since in this section we will work exclusively with the mod-$\wp$ reductions of $\sS_{K^p}$ and $\sS:=\varprojlim_{H^p}\sS_{H^p}$, we denote these $\kappa(\wp)$-schemes again by $\sS_{K^p}$ and $\sS$.
Any element $g\in G(\A_f^p)$ gives rise to a Hecke correspondence (from $\sS_{K^p}$ to itself): 
\begin{equation} \label{eq:Hecke_corr_f}
\sS_{K^p}\stackrel{p_1'}{\longleftarrow} \sS_{K^p_g} \stackrel{p_2}{\longrightarrow} \sS_{K^p}.
\end{equation}
where $K^p_g:=K^p\cap gK^pg^{-1}$, the right-hand map $p_2$ is the natural projection induced by the inclusion $K^p_g\subset K^p$ and the left-hand map $p_1'$ is the composite of the natural projection $\sS_{K^p_g}\stackrel{p_1}{\rightarrow} \sS_{gK^pg^{-1}}$ (induced by the inclusion $K^p_g\subset gK^pg^{-1}$) with the (right) action by $g$: $\sS_{gK^pg^{-1}} \isom \sS_{K^p}$; we denote by $f$ this Hecke correspondence. 
We are interested in the fixed point set of the composition $\Phi^m\circ f$ of the morphism $\Phi^m$ and the correspondence $f$, namely the fixed point set of the correspondence: 
\begin{equation} \label{eq:Hecke_corr_twisted_by_Frob}
\sS_{K^p}\stackrel{p_1'}{\longleftarrow} \sS_{K^p_g} \stackrel{p_2'=\Phi^m\circ p_2}{\longrightarrow} \sS_{K^p}.
\end{equation}
By definition, a fixed point of this correspondence $\Phi^m\circ f$ is a point in $\sS_{K^p_g}(\Fpb)$ whose images in $\sS_{K^p}(\Fpb)$ under $p_1'$ and $p_2'=\Phi^m\circ p_2$ coincide. 

Now we assume that Langlands-Rapoport conjecture (Conj. \ref{conj:Langlands-Rapoport_conjecture_ver1}) holds; we also assume that $(G,X)$ satisfies the Serre condition. 
Then, it is clear that for each admissible morphism $\phi$, the correspondences $f$, $\Phi^m\circ f$ restrict to correspondences from $S_{K^p}(\phi)$ to itself: 
\[ S_{K^p}(\phi) \stackrel{p_1'}{\longleftarrow} S_{K^p_g}(\phi) \stackrel{p_2,p_2'}{\longrightarrow} S_{K^p}(\phi).\]
where $S_{K^p}(\phi)=I_{\phi}(\Q)\backslash [X_p(\phi)\times (X^p(\phi)/K^p)]$ and $S_{K^p_g}(\phi)$ is defined similarly. From now on, we assume $K^p$ to be small enough so that the following conditions hold:
\begin{align} \label{item:Langlands-conditions}
(a) &\text{ If }\epsilon\in I_{\phi}(\Q)\text{ and }\epsilon x= xg\text{ for some }x\in X^p(\phi)/K^p_g\times X_p(\phi),\text{ then }\epsilon\in Z(\Q)_K:=Z(\Q)\cap K. \\
(b) &\ I_{\phi}^{\der}\cap K\cap Z(G)(\Q)=\{1\}. \nonumber
\end{align}
This is possible by \cite[p.1171-1172]{Langlands79} (cf. \cite[1.3.7, 1.3.8]{Kottwitz84b}, \cite[Lem. 5.5]{Milne92}).

Under these conditions, an elementary argument (\cite[$\S$1.4]{Kottwitz84b}, \cite[Lem. 5.3]{Milne92}) establishes that the fixed point subset of $\Phi^m\circ f$ acting on $S_{K^p}(\phi)$ decomposes into disjoint subsets:
\begin{align} \label{eq:fixed_pt_set_of_Heck-corresp1}
S_{K^p}(\phi)^{p_1'=p_2'} &= \bigsqcup_{\epsilon} I_{\phi,\epsilon}(\Q)\backslash X(\phi,\epsilon)_{K^p_g},
\end{align}
where the index $\epsilon$ runs through a set of representatives in $I_{\phi}(\Q)$ for the conjugacy classes of $I_{\phi}(\Q)/Z(\Q)_K$, $I_{\phi,\epsilon}(\Q)$ is the centralizer of $\epsilon$ in $I_{\phi}(\Q)$ (we regard $I_{\phi,\epsilon}$ as an algebraic $\Q$-subgroup of $I_{\phi}$), and 
\[X(\phi,\epsilon)_{K^p_g} := \{\ x\in X(\phi)/K^p_g  \ \ |\ \  \epsilon p_1'(x)=\Phi^m p_2(x)\ \}= X_p(\phi,\epsilon)\times (X^p(\phi,\epsilon;g)/K^p_g) \]
with
\begin{align*}
X_p(\phi,\epsilon) &:=\{\ x_p\in X_p(\phi)  \ \ |\ \  \epsilon x_p=\Phi^m x_p\ \}, \\
X^p(\phi,\epsilon;g) &:=\{\ x^p\in X^p(\phi) \ \ |\ \  \epsilon x^pg=x^p\text{ mod }K^p\ \},
\end{align*}
where $\Phi$ is the Frobenius automorphism acting on $X_p(\phi)$ (\ref{eq:Frob_Phi_1}).
When $g=1$, this gives the description in Remark \ref{rem:admissible_pair} (with $X^p(\phi,\epsilon):=X^p(\phi,\epsilon;1)$). 

In particular, we see that if the set $I_{\phi,\epsilon}(\Q)\backslash X(\phi,\epsilon)_{K^p_g}$ is non-empty for some $\epsilon\in I_{\phi}(\Q)$, the pair $(\phi,\epsilon)$ is an admissible pair in the sense of Def. \ref{defn:admissible_pair}. 
We want a description of the set $X_p(\phi,\epsilon)$ in terms of the (equivalence class of) Kottwitz triple attached to the pair $(\phi,\epsilon)$.

We choose $u\in G(\Qpb)$ such that $\xi_p'=\Int(u)\circ\phi(p)\circ\zeta_p$ is unramified, say the inflation of a $\Qpnr/\Qp$-Galois gerb morphism $\theta^{\nr}:\fD\rightarrow\fG_{G_{\Qp}}^{\nr}$ (for example, $u\in G(\Qpb)$ with $u^{-1}\mbfK_p(\Qpnr)\in X_p(\phi)$),
which then gives $b\sigma=\theta^{\nr}(s_{\sigma})\in G(\Qpnr)\rtimes\langle\sigma\rangle$ and $\epsilon':=\Int(u)(\epsilon)\in G(\Qpnr)$ via the embedding of $\Qp$-groups
\begin{equation} \label{eq:int(u)}
\Int(u):(I_{\phi})_{\Qp}=I_{\phi(p)}:=\underline{\mathrm{Aut}}(\phi(p))\ \hookrightarrow\ \underline{\mathrm{Aut}}(\phi(p)\circ\zeta_p)\ \isom\ \underline{\mathrm{Aut}}(\xi_p')=\underline{\mathrm{Aut}}(\theta^{\nr}),
\end{equation}
where for a morphism $\theta$ of $k'/k$-Galois gerbs ($k'/k$ being a Galois extension), $\underline{\mathrm{Aut}}(\theta)$ is the $k$-algebraic group defined in (\ref{eq:Isom(phi_1,phi_2)}). Under this embedding, the $\Qp$-group $(I_{\phi})_{\Qp}$ is identified with the centralizer of $\mathrm{im}(\Int(u)\circ\phi(p)^{\Delta})$ in $\underline{\mathrm{Aut}}(\xi_p')$, since $\im(\phi(p))$ is generated by $\im(\phi(p)\circ\zeta_p)$ and $\im(\phi(p)^{\Delta})$. 
When $u^{-1}\mbfK_p(\Qpnr)\in X_p(\phi)$, multiplying by $u$ gives an identification (\ref{eqn:X_p(phi)=ADLV}) \[u:(X_p(\phi),\Phi)\isom (X(\{\mu_X\},b)_{\mbfK_p},(b\sigma)^r)\] 
of sets with operator (see (\ref{eq:Frob_Phi_1}), (\ref{eq:Frob_Phi_2})), 
where $r=[\kappa(\wp):\Fp]$. The group $I_{\phi}(\Qp)$ acts on both sides in compatible manners with the Frobenius operators: the left action is the canonical one and the right action is induced from this action via (\ref{eq:int(u)}), so that this identification is $I_{\phi}(\Qp)$-equivariant for these actions. 
Under this identification, the equation $\epsilon x_p=\Phi^mx_p$ translates to
\[\epsilon' (ux_p)=(b\sigma)^{rm}(ux_p)\]
Note that since $\epsilon'$ commutes with $\theta^{\nr}(s_{\sigma})=b\sigma$, we then have $(\epsilon')^sx_p=(b\sigma)^{rms}x_p$ for any $s\in\N$.

From the pair $(\epsilon',b)$, by Lemma \ref{lem:delta_from_b&gamma_0}, we can find $c\in G(\mfk)$ such that 
\begin{equation} \label{eq:(epsilon,b,c)->delta1}
c(\epsilon'^{-1}(b\sigma)^m)c^{-1}=\sigma^n,
\end{equation}
($n=rm$) so that one has $\delta:=cb\sigma(c^{-1})\in G(L_n)$ and $\Nm_n\delta=c\epsilon'c^{-1}$. 

We also choose $v\in X^p(\phi,\epsilon;g)$ and set $\gamma=(\gamma_l)_{l\neq p}:=v\epsilon v^{-1}\in G(\A_f^p)$.
Then, for any $\gamma_0\in G(\Q)$ stably conjugate to $\epsilon\in G(\Qb)$, the triple 
\[(\gamma_0;\gamma,\delta)\] 
is a Kottwitz triple of level $n=rm$ attached to $(\phi,\epsilon)$ which has trivial Kottwitz invariant (\autoref{subsubsec:K-triple_attached_to_adm.pair}, Prop. \ref{prop:Kottwitz_triple}).

The cardinality of the set $I_{\phi,\epsilon}(\Q)\backslash X(\phi,\epsilon)$ will be expressed in terms of this triple $(\gamma_0;\gamma,\delta)$ and $g$ (as a product of orbital and twisted-orbital integrals and a certain constant). 
This is explained in \cite[$\S$1.5]{Kottwitz84b} (cf. \cite[$\S$16]{Kottwitz92}) in the case $g=1$. 
We extend this argument to general $g\in G(\A_f^p)$ in the style similar to that of \cite{Kottwitz92} in PEL-type cases.%%
\footnote{Kottwitz's method in \textit{loc. cit.} for general $g$ uses the notion of \textit{virtual abelian variety over a finite field}. This terminology does not appear in our work, although one can say that the notion is working behind the scenes.}

Let us use the following notation: for an algebraic group $H$ over a field $k$ (of characteristic zero) and a subset $S\subset H(k)$, $H_S$ denotes the simultaneous centralizer in $H$ of the elements of $S$, 
except for the notation $T_{\epsilon}$ from Prop. \ref{prop:canonical_decomp_of_epsilon} (subgroup generated by $\epsilon$).
For a place $v$ of $\Q$, let $I_{\phi(v)}$ be the twist of the $\Qv$-group $I_{\Qv}=Z_{G_{\Qv}}(\phi(v)^{\Delta})$ by $\phi(v)$ so that $I_{\phi(v)}(\Qv)=\mathrm{Aut}(\phi(v))=\{g_v\in G(\Qvb)\ |\ \Int(g_v)\circ\phi(v)=\phi(v) \}$. 

%%%%%%%%%%%%%%%%%%%%
\begin{lem} \label{lem:isom_Int(cu)}
Assume that $\Q\operatorname{-}\mathrm{rk}(Z(G))=\R\operatorname{-}\mathrm{rk}(Z(G))$.
Let $(\phi,\epsilon)$ be an admissible pair and $(\gamma_0;\gamma,\delta)$ an associated Kottwitz triple.

(1) The $\Q$-group $I_{\phi,\epsilon}$ is an inner form of $G_{\gamma_0}$. Moreover, if $(\phi,\epsilon)$ is well-located in a maximal $\Q$-torus $T$ of $G$ that is elliptic over $\R$, there exists an inner twisting $(I_{\phi,\epsilon})_{\Qb}\isom (G_{\gamma_0})_{\Qb}$ that restricts to the identity map of $T$ and also induces an inner twisting $(I_{\phi,\epsilon}^{\mathrm{o}})_{\Qb}\isom (G_{\gamma_0}^{\mathrm{o}})_{\Qb}$.

(2) For any $v_l^{-1}\in X_l(\phi,\epsilon)$, $\Int(v_l):(I_{\phi})_{\Ql} \hookrightarrow G_{\Ql}$ induces an isomorphism of $\Ql$-groups
\[\Int(v_l):(I_{\phi,\epsilon})_{\Ql}=I_{\phi(l),\epsilon}\isom G_{\gamma_l},\] 
where $\gamma_l:=\Int(v_l)(\epsilon)\in G(\Ql)$ and $G_{\gamma_l}$ is the centralizer of $\gamma_l$ in $G_{\Ql}$. 

(3) For $u\in G(\Qpb)$ (\ref{eq:int(u)}) and $c\in G(\mfk)$ (Lemma \ref{lem:delta_from_b&gamma_0}) above,
$\Int(cu)$ induces an isomorphism of $\Qp$-groups
\begin{equation} \label{eq:Int(cu)}
\Int(cu):\ (I_{\phi,\epsilon})_{\Qp}=I_{\phi(p),\epsilon}\isom G_{\delta\sigma},
\end{equation}
where $G_{\delta\sigma}$ is the $\sigma$-centralizer of $\delta\in G(L_n)$ \cite[p.802]{Kottwitz82} ($G_{\delta\sigma}$ is a closed subgroup of $\Res_{L_n/\Qp}(G)$ such that $G_{\delta\sigma}(\Qp)=\{ y\in G(L_n)\ |\ y(\delta\sigma)=(\delta\sigma)y\}$, and one also has $(G_{\delta\sigma})_{L_n}\simeq Z_{G_{L_n}}(\Nm_n\delta)$).

(4) For every place $v$ of $\Q$, we have a natural identification:
\[(I_{\phi,\epsilon})_{\Qv}=I_{\phi(v),\epsilon}=I_{\phi(v)\circ\zeta_v,\epsilon}.\]
\end{lem}

Unlike the case $G^{\der}=G^{\uc}$, in general, the isomorphism class of $I_{\phi,\epsilon}^{\mathrm{o}}$ \emph{as inner form of $G_{\gamma_0}^{\mathrm{o}}$} (i.e. a preimage in $H^1(\Q,(G_{\gamma_0}^{\mathrm{o}})^{\ad})$ of the corresponding class in $H^1(\Q,\mathrm{Aut}(G_{\gamma_0}^{\mathrm{o}}))$) is not uniquely determined by the LR-pair $(\phi,\epsilon)$ or the associated Kottwitz triple.

\begin{proof}
Since for any $g\in G(\Qb)$, $\Int(g)$ induces $\Q$-isomorphisms 
\[I_{\phi}=\underline{\mathrm{Aut}}(\phi)\isom I_{\phi'}=\underline{\mathrm{Aut}}(\phi'),\quad  I_{\phi,\epsilon}\isom I_{\phi',\epsilon'}\] 
for $\phi':=\Int(g)\circ\phi$ and $\epsilon':=\Int(g)(\epsilon)$, in all proofs, by admissible embeddings of maximal tori (Lemma \ref{lem:LR-Lemma5.23}), we may and do assume that $\epsilon\in G(\Q)$, in which case the inner twisting $(I_{\phi})_{\Qb}\isom I_{\Qb}$ (\ref{eq:inner-twisting_by_phi}) induces a $\Q$-isomorphism $T_{\epsilon}^{\phi}\isom T_{\epsilon}^G$ between their $\Q$-subgroups generated by $\epsilon$ (Prop. \ref{prop:canonical_decomp_of_epsilon}, (2)); we simply write $T_{\epsilon}$ for this group, and let $(\pi_0,t)$ be the elements of $T_{\epsilon}(\Q)$ (for some $s\in\N$) attached to $(T_{\epsilon},\epsilon)$ by Lemma \ref{lem:canonical_decomp_of_epsilon}.

(1) To prove that $I_{\phi,\epsilon}$  is an inner form of $G_{\gamma_0}$, it suffices to show that $Z_G(\phi^{\Delta},\epsilon)$ (simultaneous centralizer in $G$ of the image of $\phi^{\Delta}$ and $\epsilon)$ equals $G_{\epsilon}$ (as $I_{\phi,\epsilon}$ and $G_{\gamma_0}$ are respectively $\Q$-inner forms of $Z_G(\phi^{\Delta},\epsilon)$ and $G_{\epsilon}$). This follows from Prop. \ref{prop:canonical_decomp_of_epsilon}, according to which we have
\[Z_G(\phi^{\Delta},\epsilon)=Z_G(\pi_0^k,\epsilon)=Z_G(\pi_0^k,T_{\epsilon})=G_{\epsilon}.\]
When $(\phi,\epsilon)$ is well-located in a maximal $\Q$-torus $T$ of $G$ that is elliptic over $\R$, since $I_{\phi,\epsilon}$ is the inner-twist of $Z_G(\phi^{\Delta},\epsilon)$ via $\phi$, we see that there exists a $T$-equivariant inner twisting $(I_{\phi,\epsilon})_{\Q}\isom (G_{\gamma_0})_{\Qb}$ which also induces an inner twisting $(I_{\phi,\epsilon}^{\mathrm{o}})_{\Qb}\isom (G_{\gamma_0}^{\mathrm{o}})_{\Qb}$.

(2) According to Prop. \ref{prop:canonical_decomp_of_epsilon}, we know that $\Int(v_l)(\pi_0^k)$ also lies in the $\Ql$-subgroup of $G_{\Ql}$ generated by $\Int(v_l)(\epsilon)$ (which equals $\Int(v_l)(T_{\epsilon})$). 
Since $v_l\in X_l(\phi)$, $\Int(v_l)$ identifies $(I_{\phi})_{\Ql}=\underline{\mathrm{Aut}}(\phi(l))$ with the centralizer in $G_{\Ql}$ of the image of $\Int(v_l)\circ\phi(l)^{\Delta}$, which in turn equals the centralizer of $\Int(v_l)(\pi_0^k)$ for any $k\gg1$ (by Prop. \ref{prop:phi(delta)=gamma_0_up_to_center} and Lemma \ref{lem:Zariski_group_closure}, (2)).
So, under $\Int(v_l)$, $(I_{\phi,\epsilon})_{\Ql}$ is identified with the simultaneous centralizer in $G_{\Ql}$ of $\Int(v_l)(\pi_0^k)$ and $\Int(v_l)(\epsilon)$, hence with the centralizer of $\Int(v_l)(\epsilon)$ alone.

(3) Recall that under the embedding $\Int (u):(I_{\phi})_{\Qp} \hookrightarrow \underline{\mathrm{Aut}}(\xi_p')$ (\ref{eq:int(u)}), the $\Qp$-group $(I_{\phi})_{\Qp}$ is identified with the centralizer of $\mathrm{im}(\Int(u)\circ\phi(p)^{\Delta})$ in $\underline{\mathrm{Aut}}(\xi_p')$.
We put $(\epsilon',\pi_0',t')=\Int(u)(\epsilon,\pi_0,t)$ (elements of $T_{\epsilon}':=\Int(u)(T_{\epsilon})$).
Since the image of $\Int(u)\circ\phi(p)^{\Delta}$ is generated by $\pi_0'^k$ for any $k\gg1$ (by Prop. \ref{prop:phi(delta)=gamma_0_up_to_center}), $\Int(u)$ induces an isomorphism 
\[\Int(u):(I_{\phi,\epsilon})_{\Qp} \isom \underline{\mathrm{Aut}}(\xi_p')_{\pi_0'^k,\epsilon'}.\]
Suppose that $\xi_p'$ is the inflation of a morphism of $\Qpnr/\Qp$-Galois gerbs, say $\theta^{\nr}:\fD\rightarrow \fG_{G_{\Qp}}^{\nr}$; then, $\underline{\mathrm{Aut}}(\xi_p')=\underline{\mathrm{Aut}}(\theta^{\nr})$.  We also choose $N\in\N$ for which $\theta^{\nr}$ factors through $\fD_N=\fD_{L_N}$ and, as such, is induced (by inflation to $\Qpnr$) from a morphism $\fD^{L_N}_{p,L_N}\rightarrow G(L_N)\rtimes \Gal(L_N/\Qp)$ of $L_N/\Qp$-Galois gerbs (cf. \autoref{subsubsec:Dieudonne_gerb}). Then, as $(\theta^{\nr})^{\Delta}=-N\nu_b$ for $b=\theta^{\nr}(s_{\sigma})$ and its Newton homomorphism $\nu_b$ (Lemma \ref{lem:Newton_hom_attached_to_unramified_morphism}), it follows from definition of $\fD_N$ that 
\[N\nu_{b}\in\Hom_{L_N}(\Gm,G),\ \text{ and }\quad \Nm_N(b)=(N\nu_{b})(p).\] 
Since $(b\sigma)^N=(N\nu_b)(p)\sigma^N$, any $g\in G(\mfk)$, if it commutes with both $b\sigma$ and $\im(\nu_b)$, also must commute with $\sigma^N$. Hence, for any $\Qp$-algebra $R$, we have 
\begin{align*} 
\underline{\mathrm{Aut}}(\theta^{\nr})(R)
&=\{ g\in Z_{G(\Qpnr\otimes R)}(\nu_b)\ |\ gb\sigma(g^{-1})=b\} \\ 
&=\{ g\in Z_{G(\mfk\otimes R)}(\nu_b)\ |\ gb\sigma(g^{-1})=b\} \\
&=J_b(R),
\end{align*}
where $J_{b}$ is the algebraic $\Qp$-group whose set of $R$-points for any $\Qp$-algebra $R$ is given by $J_b(R):=\{g\in G(\mfk\otimes R)\ |\ g(b\sigma)=(b\sigma)g\}$ (in the last equality we used the fact that $\nu_{hb\sigma(h^{-1})}=h\nu_bh^{-1}$ for any $h\in G(\mfk)$).
Therefore, it follows that $\Int(u)$ induces isomorphisms $(I_{\phi})_{\Qp}\isom \underline{\mathrm{Aut}}(\theta^{\nr})_{\pi_0'^k}=(J_b)_{\pi_0'^k}$,
\[(I_{\phi,\epsilon})_{\Qp} \isom \underline{\mathrm{Aut}}(\theta^{\nr})_{\pi_0'^k,\epsilon'} = (J_b)_{\pi_0'^k,\epsilon'}=(J_b)_{\epsilon'},\]
for any sufficiently large $k\in\N$, because the image of $\Int(u)\circ\phi(p)^{\Delta}$ (which also contains the image of $\nu_b$) is generated by $\pi_0'^k\ (k\gg1)$ (Prop. \ref{prop:phi(delta)=gamma_0_up_to_center}) and $\pi_0'$ lies in the subgroup of $G_{\mfk}$ generated by $\epsilon'$ (Prop. \ref{prop:canonical_decomp_of_epsilon}, Lemma \ref{lem:Zariski_group_closure}).

Finally, as one has $c(\epsilon'^{-1}(b\sigma)^n)c^{-1}=\sigma^n$ (\ref{eq:(epsilon,b,c)->delta1}), $\Int(c)$ induces an isomorphism
\[ \Int(c): (J_b)_{\epsilon'} \isom G_{\delta\sigma}.\]

(4) As it is clear that $(I_{\phi,\epsilon})_{\Qv}=I_{\phi(v),\epsilon}$, this is proved by the same argument as in (2): 
\[I_{\phi(v),\epsilon}=I_{\phi(v)\circ\zeta_v,\pi_0^k,\epsilon}=I_{\phi(v)\circ\zeta_v,\epsilon}. \qedhere\] 
\end{proof}

For a stable Kottwitz triple $(\gamma_0;\gamma,\delta)$ attached to an admissible pair $(\phi,\epsilon)$, any cohomology class in $H^1(\A_f^p,I_0)\oplus B(I_0)_{basic}(\cong H^1(\Qp,I_0))$ defining a Kottwitz invariant gives a class in $H^1(\Q,(G_{\gamma_0}^{\mathrm{o}})^{\ad})$ which presents $I_{\phi,\epsilon}^{\mathrm{o}}$ as an inner form of $G_{\gamma_0}^{\mathrm{o}}$.

%%%%%%%%%%%%%%%%%%%%
\begin{lem} \label{lem:local_inner-twistings_of_Z_G(gamma_0)}
Keep the assumption from Lemma \ref{lem:isom_Int(cu)}. 

(1) For a \emph{stable} Kottwitz triple $(\gamma_0;\gamma,\delta)$ of level $n$ attached to an admissible pair $(\phi,\epsilon)$, let us choose $(g_l)_{l\neq p}\in G(\bar{\A}_f^p)$ and $g_p\in G(\mfk)$ satisfying (\ref{eq:stable_g_l}), (\ref{eq:stable_g_l}), i.e. such that $g_l\gamma_0g_l^{-1}=\gamma_l$ and $g_v^{-1}\cdot{}^{\tau}g_v\in I_0(\Qvb)$ for $v\neq p$ and $\tau\in\Gamma_v$, while $g_p\gamma_0 g_p^{-1}=\Nm_n\delta$ and $b:=g_p^{-1}\delta\sigma(g_p)\in I_0$. 

Then, $I_{\phi,\epsilon}^{\mathrm{o}}$ is an inner form of $I_0:=G_{\gamma_0}^{\mathrm{o}}$ corresponding to the class $\beta\in H^1(\Q,(I_0)^{\ad})$ whose localizations are given by:
\begin{equation} \label{eq:local_inner-class}
\beta(v)=\begin{cases} [pr(g_l^{-1}\cdot{}^{\tau}g_l)] &\text{ if } v=l\neq p \\
[pr(b)] &\text{ if } v=p  
\end{cases}
\end{equation}
where $pr:I_0\rightarrow I_0^{\ad}$ is the canonical map, and $[pr(b)]$ denotes the $\sigma$-conjugacy class of $pr(b)\in (I_0)^{\ad}(\mfk)$ which in fact lies in the subset $H^1(\Qp, (I_0)^{\ad})$ since $[b]\in B( I_0)$ is basic (cf. Lemma \ref{lem:equality_of_two_Newton_maps}, \cite[4.5]{Kottwitz85}).
Also, $\beta(\infty)$ is the (unique) class corresponding to the compact inner form of $(I_0)_{\R}$.
In particular, as an inner form of $G_{\gamma_0}^{\mathrm{o}}$, $I_{\phi,\epsilon}^{\mathrm{o}}$ is uniquely determined by the cosets $G(\Ql) g_l$ and $G(L_n) g_p$.

(2) When $(\phi,\epsilon)$ is nested in a special Shimura sub-datum, there exists a choice of $(g_v)_v\in G(\bar{\A}_f^p)\times G(\mfk)$ as in (1) such that the twisting $(I_{\phi,\epsilon}^{\mathrm{o}})_{\Qb}\isom (G_{\gamma_0}^{\mathrm{o}})_{\Qb}$ in Lemma \ref{lem:isom_Int(cu)}, (1) (which is the twisting via $\phi$, cf. (\ref{eq:inner-twisting_by_phi})) fits the local descriptions (\ref{eq:local_inner-class}).
\end{lem}

\begin{proof}
(1) Since $(\gamma_0;\gamma,\delta)$ is a stable Kottwitz triple, there exist $(g_l)_{l\neq p}$ and $c\in G(\mfk)$ as in the statement.
Then, for each $l\neq p$, we have an inner-twisting
\begin{equation} \label{eq:inner-twisting_at_l_of_G(gamma_0)}
\varphi_l:(I_0)_{\Qlb} \stackrel{\Int(g_l)}{\longrightarrow} (G_{\gamma_l}^{\mathrm{o}})_{\Qlb} \stackrel{\Int(v_l^{-1})}{\longrightarrow} (I_{\phi,\epsilon}^{\mathrm{o}})_{\Qlb},
\end{equation}
where $v_l^{-1}\in X_l(\phi,\epsilon)$ (Lemma \ref{lem:isom_Int(cu)}). So, we have $\varphi_l^{-1}\cdot {}^{\tau}\varphi_l=\Int(g_l^{-1}\cdot{}^{\tau}g_l)$ for every $\tau\in\Gamma_l$.
At $p$, there exists an inner-twisting 
\begin{equation} \label{eq:inner-twisting_at_p_of_G(gamma_0)}
\varphi_p:(I_0)_{\Qpb} \stackrel{\psi_p^{-1}}{\longrightarrow} (G_{\delta\sigma}^{\mathrm{o}})_{\Qpb}   \stackrel{\Int((cu)^{-1})}{\longrightarrow} (I_{\phi,\epsilon}^{\mathrm{o}})_{\Qpb},
\end{equation}
where $u\in G(\Qpb)$ and $c\in G(\mfk)$ are as in Lemma \ref{lem:isom_Int(cu)}, and $\psi_p$, which is a morphism $(G_{\delta\sigma}^{\mathrm{o}})_{\mfk}\isom (I_0)_{\mfk}$ defined over $\mfk$, is the composite of the two isomorphisms (defined over $L_n$ and $\mfk$)
\[ (G_{\delta\sigma}^{\mathrm{o}})_{L_n}\isom G_{\Nm_n\delta}^{\mathrm{o}},\quad \Int(g_p^{-1}): (G_{\Nm_n\delta}^{\mathrm{o}})_{\mfk}\isom (I_0)_{\mfk} \]
\cite[Lem. 5.4]{Kottwitz82}. One has 
\[\psi_p\cdot {}^{\sigma}\psi_p^{-1}=\Int(b).\] 
Since $[b]\in B(I_0)$ is basic, by definition there exist $n'=nk\ (k\gg1)$ and $d\in I_0(\mfk)$ such that $d\Nm_{n'}b\sigma^{n'}(d^{-1})=\nu_{b}(p)$, thus gives a cocycle $\sigma\mapsto pr(db\sigma(d^{-1}))\in Z^1(L_{n'}/\Qp,(I_0)^{\ad})$ whose class maps to $[pr(b)]\in B((I_0)^{\ad})=H^1(W_{\Qp},(I_0)^{\ad}(\bar{\mfk}))$. Hence, we see that $[pr(b)]$ defines the inner-twisting $(I_0)_{\Qpnr} \isom (G_{\delta\sigma}^{\mathrm{o}})_{\Qpnr}$ above, which proves the claim for $\beta(p)$.

Finally, note that these local data determine an \emph{inner form of $I_0$} uniquely, by the Hasse principle for connnected (semisimple) adjoint groups over number fields.

(2) Suppose that $\phi=\psi_{T,\mu_h}$, $\epsilon\in T(\Q)$ for a special Shimura sub-datum $(T,h)$.
Recall (Remark \ref{rem:two_different_b's}) that in this case, $\gamma_l=\gamma_0\in T(\Q)$ for every $l\neq p$ and $\delta=cb\sigma(c^{-1})$ for some $c\in G(\mfk)$ and $b\sigma=\theta^{\nr}(s_{\sigma})$, where $\theta^{\nr}$ is a $\Qpnr/\Qp$-Galois gerb morphism whose inflation to $\Qpb$ equals an unramified $T(\Qpb)$-conjugate of $\xi_p=\phi(p)\circ\zeta_p$. So, $b\in T(\mfk)$, and we can take $g_l:=1$ and $g_p:=c$.
On the other hand, since $\phi(l)\circ\zeta_l$ is conjugate to the canonical trivialization $\xi_l$ under $T(\Qlb)$, we find that the local cohomology class in $H^1(\Qv,(I_0)^{\ad})$ representing the inner twist $(I_{\phi,\epsilon}^{\mathrm{o}})_{\Qvb}\isom (I_0)_{\Qvb}$ is trivial if $v=l$ and equals $[pr(b)]\in H^1(\Qpnr/\Qp,(I_0)^{\ad})$. This proves the claim. 
\end{proof}

Note that due to Hasse-principle on $\pi_0(G_{\gamma_0})$ (Prop. \ref{prop:triviality_in_comp_gp}), we also see that  $I_{\phi,\epsilon}$ is the twist of $G_{\gamma_0}$ by (the image under $H^1(\Q,G_{\gamma_0}^{\mathrm{o}})\rightarrow H^1(\Q,G_{\gamma_0}^{\ad})$ of) the same cohomology class.

%%%%%%%%%%%%%%%%%%%%
\begin{lem} \label{lem:uniqueness_of_inner-class_with_same_K-triples}
Let $(\phi,\epsilon)$ and $(\phi',\epsilon')$ be two admissible pairs. If their associated Kottwitz triples are equivalent, then the $\Q$-groups $I_{\phi,\epsilon}^{\mathrm{o}}$, $I_{\phi',\epsilon'}^{\mathrm{o}}$ are isomorphic as inner forms of $G_{\gamma_0}^{\mathrm{o}}$. 
\end{lem}

\begin{proof}
Let $(\gamma_0;\gamma,\delta)$ and $(\gamma_0';\gamma',\delta')$ be some stable Kottwitz triples attached to $(\phi,\epsilon)$ and $(\phi',\epsilon')$ respectively such that the $G(\A_f^p)$-conjugacy classes of $\gamma$, $\gamma'$ and the $\sigma$-conjugacy classes $\delta$, $\delta'$ in $G(L_n)$ are the same; then, according to Prop. \ref{prop:triviality_in_comp_gp} they are stably conjugate and thus we may assume that $\gamma_0'=\gamma_0$. 
It follows from Lemma \ref{lem:local_inner-twistings_of_Z_G(gamma_0)} that the isomorphism class of $I_{\phi,\epsilon}^{\mathrm{o}}$ as an inner twist of $I_0$ is determined by choice of elements  $g_l$ and $g_0$ as in Lemma \ref{lem:local_inner-twistings_of_Z_G(gamma_0)} by means of the corresponding local cohomology classes (\ref{eq:local_inner-class}) in $H^1(\Qv,(I_0)^{\ad})$. Then, since $\gamma'=h_l\gamma_lh_l^{-1}$ and $\delta'=d\delta\sigma(d)^{-1}$ for some $h_l\in G(\Ql)$ and $d\in G(L_n)$, we may choose the elements $(g_v)_v$, $(g_v')_v$ in such a way that those cohomology classes are the same, hence the claim follows.
\end{proof}

Let us fix $(v_l)_l\in \prod'_{l\neq\infty,p}X_l(\phi)$; then, for almost all $l\neq p$, $\Int(v_l)$ extends to an embedding of reductive $\Z_l$-group schemes $(I_{\phi})_{\Z_l} \hookrightarrow G_{\Z_l}$ (cf. \cite[p.168]{LR87}).
By taking the restricted product, over the finite primes $l\neq p$, of $\Int(v_l):I_{\phi}(\Ql)=\mathrm{Aut}(\phi(l))\rightarrow \mathrm{Aut}(\xi_l)=G(\Ql)$, this also specify isomorphisms
\begin{equation} \label{eq:Int(v)}
\Int(v):I_{\phi}(\A_f^p) \isom G(\A_f^p),\quad I_{\phi,\epsilon}(\A_f^p) \isom G(\A_f^p)_{\gamma},
\end{equation}
where $G(\A_f^p)_{\gamma}$ denotes the centralizer of $\gamma$ in $G(\A_f^p)$. 

Set $\mbfK_q:=G(L_n)\cap \mbfKt_p$ ($q=p^n$).
Define $\phi_p$ to be the characteristic function of the subset (union of double cosets) of $\mbfK_q\backslash G(L_n)/\mbfK_q$ corresponding to the following set, cf. Definition \ref{defn:mu-admissible_subset}:
\begin{equation} \label{eq:Adm_K(mu)}
\mathrm{Adm}_{\mbfK_q}(\{\mu\}):=\{w\in\mbfK_q\backslash G(L_n)/\mbfK_q\ |\ w\leq \tilde{W}_{\mbfKt_p}t^{\lambda}\tilde{W}_{\mbfKt_p}\text{ for some }\lambda\in\Lambda(\{\mu\})\}.
\end{equation}
Here, in the inequality, $w$ also denotes its image in $\mbfKt_p\backslash G(\mfk)/\mbfKt_p\simeq \tilde{W}_{\mbfKt_p}\backslash \tilde{W}/\tilde{W}_{\mbfKt_p}$ (\ref{eqn:parahoric_double_coset}) ($\tilde{W}$ is the relative Weyl group  of $G_{\mfk}$ and $\tilde{W}_{\mbfKt_p}=\tilde{W}\cap\mbfKt_p$).
Let $dy_p$ (resp. $dy^p$) denote the Haar measure on $G(L_n)$ (resp. on $G(\A_f^p)$) giving measure $1$ on $\mbfK_q$ (resp. on $K^p$). 
We also choose an (arbitrary, for the moment) Haar measure $di_p$ (resp. $di^p$) on $G_{\delta\sigma}^{\mathrm{o}}(\Qp)$ (resp. on $G_{\gamma}^{\mathrm{o}}(\A_f^p)$) that gives rational measure to compact open subgroups of $G_{\delta\sigma}^{\mathrm{o}}(\Qp)$ (resp. of $G_{\gamma}^{\mathrm{o}}(\A_f^p)$). Then, we write $d\bar{y}_p$ $d\bar{y}^p$ for the quotient of $dy_p$ by $di_p$ and that of $dy^p$ by $di^p$, respectively.

%%%%%%%%%%%%%%%%%%%%
\begin{lem} \label{lem:fixed-pt_subset_of_Frob-Hecke_corr}
Assume that $Z(G)$ has same ranks over $\Q$ and $\R$.
We take $K^p$ small enough so that conditions of (\ref{item:Langlands-conditions}) hold and $Z(G)(\Q)\cap K=\{1\}$.
Let $(\gamma_0;\gamma,\delta)$ be a Kottwitz triple attached to an admissible pair $(\phi,\epsilon)$.
Then, we have
\[|I_{\phi,\epsilon}(\Q)\backslash X(\phi,\epsilon)_{K^p_g}|= \frac{\mathrm{vol}(I_{\phi,\epsilon}^{\mathrm{o}}(\Q)\backslash I_{\phi,\epsilon}^{\mathrm{o}}(\A_f))}{[I_{\phi,\epsilon}(\Q):I_{\phi,\epsilon}^{\mathrm{o}}(\Q)]} \cdot \mathrm{O}_{\gamma}(f^p)\cdot \mathrm{TO}_{\delta}(\phi_p),\]
where $\mathrm{TO}_{\delta}(\phi_p)$ (twisted orbital integral) and $\mathrm{O}_{\gamma}(f^p)$ (orbital integral) are defined by: 
\begin{align} \label{eq:(twisted-)orbital_integral}
\mathrm{TO}_{\delta}(\phi_p)&=\int_{G_{\delta\sigma}^{\mathrm{o}}(\Qp)\backslash G(L_n)}\phi_p(y_p^{-1}\delta\sigma(y_p)) d\bar{y}_p,\\
\mathrm{O}_{\gamma}(f^p)&=\int_{G_{\gamma}^{\mathrm{o}}(\A_f^p)\backslash G(\A_f^p)}f^p((y^p)^{-1}\gamma y^p) d\bar{y}^p, \nonumber
\end{align}
Here, $f^p$ is the characteristic function of $K^pg^{-1}K^p$ in $G(\A_f^p)$, and the Haar measure $di_p$ on $G_{\delta\sigma}^{\mathrm{o}}(\Qp)$ (resp. $di^p$ on $G_{\gamma}^{\mathrm{o}}(\A_f^p)$) is obtained from a Haar measure on $I_{\phi,\epsilon}^{\mathrm{o}}(\Qp)$ (resp. on $I_{\phi,\epsilon}^{\mathrm{o}}(\A_f^p)$) that gives rational measures to compact open subgroups of $I_{\phi,\epsilon}^{\mathrm{o}}(\Qp)$ (resp. of $I_{\phi,\epsilon}^{\mathrm{o}}(\A_f^p)$) via the isomorphism $\Int(cu)$ (\ref{eq:Int(cu)}) (resp. via $\Int(v)$ (\ref{eq:Int(v)})).
\end{lem}

\begin{proof}
The argument given in \cite[$\S$1.4, $\S$1.5]{Kottwitz84b} (cf. \cite[p.432]{Kottwitz92}) works without change, since the necessary isomorphisms (1.4.7), (1.4.8) in \cite{Kottwitz84b} follow from Lemma \ref{lem:isom_Int(cu)}. 
Then, the cardinality in question becomes the triple product
\[ \mathrm{vol}(I_{\phi,\epsilon}(\Q)\backslash I_{\phi,\epsilon}(\A_f))\cdot \int_{G_{\gamma}(\A_f^p)\backslash G(\A_f^p)}f^p((y^p)^{-1}\gamma y^p) d\bar{y}^p\cdot \int_{G_{\delta\sigma}(\Qp)\backslash G(L_n)}\phi_p(y_p^{-1}\delta\sigma(y_p)) d\bar{y}_p \]
(the quotient measures $d\bar{y}^p$, $d\bar{y}_p$ being defined similarly).
The statement follows from this.
\end{proof}

%%%%%%%%%%%%%%%%%%%%
\begin{lem} \label{eq:|widetilde{Sha}_G(Q,I_{phi,epsilon})^+|}
For any admissible pair $(\phi,\epsilon)$, the set $\widetilde{\Sha}_G(\Q,I_{\phi,\epsilon})^+:=\ker[\Sha^{\infty}_G(\Q,I_{\phi,\epsilon}^{\mathrm{o}})\rightarrow H^1(\A,I_{\phi,\epsilon})]$ (cf. Thm. \ref{thm:LR-Satz5.25}) is a finite set which depends only on the associated (equivalence class of) Kottwitz triple $(\gamma_0;\gamma,\delta)$. Its cardinality $i(\gamma_0;\gamma,\delta)$ is given by
\[  |\widetilde{\Sha}_G(\Q,I_{\phi,\epsilon})^+|=|\ker[ \ker^1(\Q,I_0)\rightarrow \ker^1(\Q,G)]| \cdot |\mathfrak{D}(\gamma_0;\gamma,\delta)|,\]
with $I_0:=G_{\gamma_0}^{\mathrm{o}}$ as usual and
\[ \mathfrak{D}(\gamma_0;\gamma,\delta) :=\im[\Sha^{\infty}_G(\Q,I_0)\rightarrow H^1(\A_f,I_0)] \cap 
\ker [H^1(\A_f,G_{\gamma,\delta}^{\mathrm{o}})\rightarrow H^1(\A_f,G_{\gamma,\delta})], \]
where $H^1(\A_f,G_{\gamma,\delta}^{\mathrm{o}}):=\oplus_l H^1(\Qv,G_{\gamma_l}^{\mathrm{o}})\oplus H^1(\Qp,G_{\delta\sigma}^{\mathrm{o}})$ and $H^1(\A_f,G_{\gamma,\delta})$ is defined similarly.
\end{lem}

The intersection in the definition of $\mathfrak{D}(\gamma_0;\gamma,\delta)$ makes sense because there exist 
inner twists $(I_0)_{\Qvb}\isom (G_{\gamma_v}^{\mathrm{o}})_{\Qvb}\ (v\neq p)$, $(G_{\delta\sigma}^{\mathrm{o}})_{\Qpb}$ (Lemma \ref{lem:isom_Int(cu)}), thereby canonical identifications
\[H^1(\A_f,I_0) \cong H^1(\A_f,(I_0)_{\bfab}) \cong H^1(\A_f,(G_{\gamma,\delta}^{\mathrm{o}})_{\bfab}) \cong H^1(\A_f,G_{\gamma,\delta}^{\mathrm{o}})\]
(the middle isomorphism is independent of the choice of the inner twists just mentioned).

%%%%%%%%%%%%%%%%%%%%
\begin{rem}
When $I_{\phi,\epsilon}$ is connected (e.g., if $G^{\der}=G^{\uc}$), one has $|\mathfrak{D}(\gamma_0;\gamma,\delta)|=1$ and the constant $|\widetilde{\Sha}_G(\Q,I_{\phi,\epsilon})^+|$ depends only on the stable conjugacy class of $\gamma_0$ (rather than on the whole triple $(\gamma_0;\gamma,\delta)$). But, the author does not know the same property in the general case (for a non-connected group $H$, $H^1(\Qv,H)$ may not be invariant under inner twists of $H$). Also, we note that the constant $\ker^1(\Q,I_0)$ appears in \cite[Lem.17.2]{Kottwitz92} with a similar interpretation (see the discussion on p.441-442 of \textit{loc. cit.} for the appearance of the constant $|\ker[ \ker^1(\Q,I_0)\rightarrow \ker^1(\Q,G)]|$).
\end{rem}

\begin{proof}
This set is equal to $\ker[\Sha^{\infty}_G(\Q,I_{\phi,\epsilon}^{\mathrm{o}})\rightarrow H^1(\A_f,I_{\phi,\epsilon})]$ (no $\infty$-component in the target).
According to Lemma \ref{lem:isom_Int(cu)}, the local groups $(I_{\phi,\epsilon})_{\Qv}$ (for all places $v$ of $\Q$) are determined by the associated Kottwitz triple, and this implies by the Hasse principle for $(I_{\phi,\epsilon}^{\mathrm{o}})^{\ad}$ that the same is also true of the connected $\Q$-group $I_{\phi,\epsilon}^{\mathrm{o}}$ and the set $\widetilde{\Sha}_G(\Q,I_{\phi,\epsilon})^+$. Moreover, the two $\Q$-groups $I_{\phi,\epsilon}^{\mathrm{o}}$ and $I_{\phi,\epsilon}$ are simultaneous inner twists of $I_0=G_{\gamma_0}^{\mathrm{o}}$ and $G_{\gamma_0}$ via a single cochain in $C^1(\Q,I_0)$ which induces cocycles both in  $G_{\gamma_0}^{\ad}$ and $I_{\gamma_0}^{\ad}$.
The map $\Sha^{\infty}_G(\Q,I_{\phi,\epsilon}^{\mathrm{o}})\rightarrow H^1(\A_f,I_{\phi,\epsilon})$ factors through
\[ \Sha^{\infty}_G(\Q,I_{\phi,\epsilon}^{\mathrm{o}}) \stackrel{j}{\rightarrow} H^1(\A_f,I_{\phi,\epsilon}^{\mathrm{o}}) \rightarrow
H^1(\A_f,I_{\phi,\epsilon})\]
Hence, our set $\widetilde{\Sha}_G(\Q,I_{\phi,\epsilon})^+$ is the disjoint union of $j^{-1}(\beta)$'s with $\beta$ running through
\begin{equation} \label{eq:D(gamma_0;gamma,delta)}
\ker[H^1(\A_f,I_{\phi,\epsilon}^{\mathrm{o}})\rightarrow H^1(\A_f,I_{\phi,\epsilon})]\cap \im(j).
\end{equation}
The map $j$ is identified, via abelianizations, with $\Sha^{\infty}_G(\Q,I_0)\rightarrow H^1(\A_f,I_0)$ (recall that by definition, $\Sha^{\infty}_G(\Q,I_{\phi,\epsilon}^{\mathrm{o}})=\ker[\Sha^{\infty}(\Q,I_{\phi,\epsilon}^{\mathrm{o}})\cong \Sha^{\infty}(\Q,I_0)\rightarrow \Sha^{\infty}(\Q,G)]$, where the isomorphism is given by the isomorphism of the corresponding abelianized cohomology groups), hence it follows that the set (\ref{eq:D(gamma_0;gamma,delta)}) is equal to $\mathfrak{D}(\gamma_0;\gamma,\delta)$.
We claim that for each $\beta=j(\alpha) \in\mathfrak{D}(\gamma_0;\gamma,\delta)$, there are bijections
\[ j^{-1}(\beta)\ \isom\ \ker[H^1(\Q,I_{\phi,\epsilon}^{\mathrm{o}})_{\alpha}\stackrel{\bfab^1}{\longrightarrow} H^1(\Q,G_{\bfab})]\ \isom\ \ker[\ker^1(\Q,I_0)\rightarrow \ker^1(\Q,G)], \]
where $H^1(\Q,I_{\phi,\epsilon}^{\mathrm{o}})_{\alpha}$ denotes the subset of $H^1(\Q,I_{\phi,\epsilon}^{\mathrm{o}})$ consisting of the elements having the same image in $H^1(\A,I_{\phi,\epsilon}^{\mathrm{o}})$ as $\alpha$ and $\bfab^1$ is the composite map
\[\bfab^1:H^1(\Q,I_{\phi,\epsilon}^{\mathrm{o}})\rightarrow H^1(\Q,(I_{\phi,\epsilon}^{\mathrm{o}})_{\bfab})=H^1(\Q,(I_0)_{\bfab})\rightarrow H^1(\Q,G_{\bfab})\] 
(or its restriction $\Sha^{\infty}(\Q,I_{\phi,\epsilon}^{\mathrm{o}})=\Sha^{\infty}(\Q,(I_{\phi,\epsilon}^{\mathrm{o}})_{\bfab})\rightarrow \Sha^{\infty}(\Q,G_{\bfab})=\Sha^{\infty}(\Q,G)$). 
The first bijection is clear. The second bijection is a consequence of the existence of a bijection $H^1(\Q,I_{\phi,\epsilon}^{\mathrm{o}})_{\alpha} \isom \ker[H^1(\Q,I_a^{\mathrm{o}})\rightarrow H^1(\A,I_a^{\mathrm{o}})]$ commuting with the natural maps $\bfab^1$ into $H^1(\Q,G_{\bfab})$, where $a$  is any cocycle in $H^1(\Q,I_{\phi,\epsilon}^{\mathrm{o}})$ representing $\alpha$ and $I_a^{\mathrm{o}}$ is the inner twist of $I_{\phi,\epsilon}^{\mathrm{o}}$ via $a$: this in turn results from \cite[Prop.35bis]{Serre02}, the commutative diagram of Lemma 3.15.1 of \cite{Borovoi98} combined with the vanishing of $\alpha$ in $H^1(\Q,G_{\bfab})$, and the well-known fact that for any \emph{connected} reductive group $H$ over $\Q$, the finite abelian group $\ker^1(\Q,H)$ is unchanged under inner twists.
Note that the set $\ker[H^1(\A_f,I_{\phi,\epsilon}^{\mathrm{o}})\rightarrow H^1(\A_f,I_{\phi,\epsilon})]$, being the image of the finite set $\pi_0(I_{\phi,\epsilon})(\A_f)$ in $H^1(\A_f,I_{\phi,\epsilon}^{\mathrm{o}})$, is also finite, which implies the same property for $\mathfrak{D}(\gamma_0;\gamma,\delta)$ and $\widetilde{\Sha}_G(\Q,I_{\phi,\epsilon})^+$. The formula for $|\widetilde{\Sha}_G(\Q,I_{\phi,\epsilon})^+|$ is immediate from this discussion.
\end{proof}

Now, we assume that the level subgroup $\mbfK_p$ is hyyperspecial and give an expression for the number of fixed points of the correspondence $\Phi^m\circ f$ (\ref{eq:Hecke_corr_twisted_by_Frob}), and, more generally, a weighted sum over the same fixed point set with weight being given by the trace of that correspondence acting on the stalk of some lisse sheaf. We briefly recall the set-up. For more details, see \cite[$\S$6, 16]{Kottwitz92}.
We fix a rational prime $l\neq p$. Let $\xi$ be a finite-dimensional representation of $G$ on a vector space $W$ over a number field $L$, and let $\lambda$ be a place of $L$ lying above $l$. We consider the $\lambda$-adic representation $W_{\lambda}:=W\otimes_L{L_{\lambda}}$ of $G(\A_f^p)$ induced from the natural one of $G(\Ql)$ via the projection $G(\A_f^p)\rightarrow G(\Ql)$. Since $Z(\Q)$ is discrete in $Z(\A_f)$ by our assumption ($Z(G)$ has same ranks over $\Q$ and $\R$), it gives rise to lisse sheaves $\sF_{K^p}$ on the spaces $\sS_{K^p}$ for varying $K^p$'s: the projective limit $\sS=\varprojlim_{H^p}\sS_{H^p}$ is a Galois covering of $\sS_{K^p}$ with Galois group $K^p$ if $K^p$ is small enough such that $K\cap Z(G)(\Q)=\{1\}$ \cite[2.1.9-12]{Deligne79}, and we have $\sF_{K^p}=\sS\times W_{\lambda}/K^p$, where $k=(k_v)\in K^p$ acts on $W_{\lambda}$ via $\xi_{L_{\lambda}}(k_l^{-1})$. It is clear that there are canonical isomorphisms 
\[\cdot g:\sF_{gK^pg^{-1}}\isom (\cdot g)^{\ast}\sF_{K^p},\quad p_2^{\ast}\sF_{K^p} =\sF_{K^p_g},\quad \Phi:\Phi^{\ast}\sF_{K^p}\isom \sF_{K^p},\]
where for any $g\in G(\A_f^p)$, $\cdot g$ denotes the right action $\sS_{gK^pg^{-1}}\isom\sS_{K^p}$.
Thus, the associated correspondence $\Phi^m\circ f$ (\ref{eq:Hecke_corr_twisted_by_Frob}) extends in a natural manner to the sheaf $\sF_{K^p}$. In particular, for any fixed point $x'\in \sS_{K^p_g}(\F)$ of $\Phi^m\circ f$ and $x:=p_1'(x')=p_2'(x')\in \sS_{K^p}(\F)$, the correspondence $\Phi^m\circ f$ gives an automorphism of the stalk $\sF_x$ (we write $\sF$, $\sF'$ for $\sF_{K^p}$, $\sF_{K^p_g}$): 
\begin{equation} \label{eq:Frob-Hecke_corr_at_stalk}
\sF_x=(\Phi^{m})^{\ast}(\sF)_{p_2(x')}\  \stackrel{\Phi^m}{\longrightarrow}\ \sF_{p_2(x')} =(p_2^{\ast}\sF)_{x'}=\sF'_{x'}\ \stackrel{p_1^{\ast}(\cdot g)}{\longrightarrow}\ (p_1'^{\ast}\sF)_{x'}=\sF_x.
\end{equation}
We are interested in computing the sum:
\begin{equation} \label{eq:fixed-pt_set_of_Frob-Hecke_corr}
T(m,f):=\sum_{x'\in\mathrm{Fix}} \mathrm{tr}(\Phi^m\circ f;\sF_x),
\end{equation}
where $\mathrm{Fix}$ denotes the set of fixed points of $\Phi^m\circ f$.

Let $A_G$ denote the maximal $\Q$-split torus in the center of $G$. 
In the next theorem, we return to an arbitrary Shimura datum $(G,X)$ of Hodge type and take $G$ to be the smallest algebraic (connected reductive) $\Q$-group such that every $h\in X$ factors through $G_{\R}$ (i.e. $G$ is the Mumford-Tate group of a generic $h\in X$). In particular, $(G,X)$ satisfies both the Serre condition and the condition that $Z(G)$ has same ranks over $\Q$ and $\R$ (thus, $Z_G(\R)/A_G(\R)$ is also compact).

%%%%%%%%%%%%%%%%%%%%
%%%%%%%%%%%%%%%%%%%%
\begin{thm} \label{thm:Kottwitz_formula:LR}
Keep the notation from the previous sections and the above discussion. We assume that $G_{\Qp}$ is unramified, and $(G,X)$ is of Hodge type. Fix a hyperspecial parahoric subgroup $\mbfK_p$ and take $K^p$ to be sufficiently small such that conditions (a), (b) of (\ref{item:Langlands-conditions}) hold and $K\cap Z(G)(\Q)=\{1\}$.

Suppose that Langlands-Rapoport conjecture, Conj. \ref{conj:Langlands-Rapoport_conjecture_ver1}, holds for $Sh_{\mbfK_p}(G,X)$.

(1) We have the following expression for (\ref{eq:fixed-pt_set_of_Frob-Hecke_corr}):
\[T(m,f)=\sum_{(\gamma_0;\gamma,\delta)} c(\gamma_0;\gamma,\delta)\cdot \mathrm{O}_{\gamma}(f^p)\cdot \mathrm{TO}_{\delta}(\phi_p)\cdot \mathrm{tr}\xi(\gamma_0),\]
with
\[ c(\gamma_0;\gamma,\delta):=i(\gamma_0;\gamma,\delta)\cdot |\pi_0(G_{\gamma_0})(\Q)|^{-1} \cdot \tau(I_0)\cdot \mathrm{vol}(A_G(\R)^{\mathrm{o}}\backslash I_0(\infty)(\R))^{-1} \]
where $I_0:=G_{\gamma_0}^{\mathrm{o}}$, $i(\gamma_0;\gamma,\delta)=|\widetilde{\Sha}_G(\Q,I_{\phi,\epsilon})^+|$ (Lemma \ref{eq:|widetilde{Sha}_G(Q,I_{phi,epsilon})^+|}), $\tau(I_0)$ is the Tamagawa number of $I_0$, and $I_0(\infty)$ is the (unique) inner form of $(I_0)_{\R}$ having compact adjoint group. Also, the sum is over a set of representatives $(\gamma_0;\gamma,\delta)$ of all equivalence classes of stable Kottwitz triples of level $n=m[\kappa(\wp):\Fp]$ having trivial Kottwitz invariant.

(2) For any $f^p$ in the Hecke algebra $\mathcal{H}(G(\A_f^p)/\!\!/ K^p)$, there exists $m(f^p)\in\N$, depending on $f^p$, such that for each $m\geq m(f^p)$, we have
\begin{align*}
\sum_i(-1)^i\mathrm{tr}( & \Phi^m\times f^p | H^i_c(Sh_{K}(G,X)_{\Qb},\sF_K)) \\
& = \sum_{(\gamma_0;\gamma,\delta)} c(\gamma_0;\gamma,\delta)\cdot \mathrm{O}_{\gamma}(f^p)\cdot \mathrm{TO}_{\delta}(\phi_p) \cdot \mathrm{tr}\xi(\gamma_0)
\end{align*}
where the sum is over a set of representatives $(\gamma_0;\gamma,\delta)$ of all equivalence classes of stable Kottwitz triples of level $n=m[\kappa(\wp):\Fp]$ having trivial Kottwitz invariant. If $G^{\ad}$ is anisotropic or $f^p$ is the identity, we can take $m(f^p)$ to be $1$ (irrespective of $f^p$).
\end{thm}

We remind readers again that ``having trivial Kottwitz invariant'' means that there exist elements $(g_v)_v\in G(\bar{\A}_f^p)\times G(\mfk)$ satisfying conditions (\ref{eq:stable_g_l}), (\ref{eq:stable_g_l}) such that the associated Kottwitz invariant $\alpha(\gamma_0;\gamma,\delta;(g_v)_v)$ vanishes, and the fact that for stable Kottwitz triples, stable equivalence is the same as (geometric) equivalence (Prop. \ref{prop:triviality_in_comp_gp}).

In the definition of $c(\gamma_0;\gamma,\delta)$, the volume of (the quotient of) $\R$-groups is defined with respect to the unique Haar measure $di_{\infty}$ on $I_0(\R)$ (or equivalently on $I_0(\infty)(\R)$ by transfer of measure) such that the product of $di^p,di_p,di_{\infty}$ is the canonical measure on $I_0(\A)$ that is used to define the Tamagawa number $\tau(I_0)$  (cf. \cite[$\S$1.2]{Labesse01}). Then, as $Z_G(\R)/A_G(\R)$ is compact, one has 
\begin{equation} \label{eq:Tamagawa_number}
\mathrm{vol}(I_0(\Q)\backslash I_0(\A_f))=\tau(I_0)\cdot \mathrm{vol}(A_{G}(\R)^{\mathrm{o}}\backslash I_0(\R))^{-1}.
\end{equation}

\begin{proof}
We may assume that $f^p$ is the characteristic function of $K^pg^{-1}K^p$ of some $g\in G(\A_f^p)$.
For an admissible morphism $\phi$ and a comact open subgroup $K^p$ of $G(\A_f^p)$, the fixed point subset
\[\mathrm{Fix}_{\phi}:=S_{K^p}(\phi)^{\Phi^m\circ f=\mathrm{Id}}=\{x'\in S_{K^p_g}(\phi)\ |\ p_1'(x')=p_2'(x')\}\]
is a disjoint union $\sqcup_{\epsilon\in I_{\phi}(\Q)} I_{\phi,\epsilon}(\Q)\backslash X(\phi,\epsilon)_{K^p_g}$ (\ref{eq:fixed_pt_set_of_Heck-corresp1}). 
We claim that for any $x'\in \mathrm{Fix}_{\phi}$ and $x:=p_1'(x')=p_2'(x')\in S_{K^p}(\phi)$, there is an equality: 
\[\mathrm{tr}(\Phi^m\circ f;\sF_x)=\mathrm{tr}\xi(\gamma_0),\]
where $\gamma_0$ is any element in $G(\Q)$ stably conjugate to the $\epsilon$ such that $x'\in I_{\phi,\epsilon}(\Q)\backslash X(\phi,\epsilon)_{K^p_g}$. In particular, this trace depends only on (the stable conjugacy class of) $\gamma_0$, so on the (equivalence class of) Kottwitz triples attached to the admissible pair $(\phi,\epsilon)$.
Indeed (cf. \cite[$\S$16]{Kottwitz92}), choose $x_p\in X_p(\phi,\epsilon)$, $x^p\in X^p(\phi,\epsilon;g)$ such that $x'=[x_p,x^p]\in I_{\phi,\epsilon}(\Q)\backslash X(\phi,\epsilon)_{K^p_g}$. It also gives a point $\tilde{x}=[x_p,x^p]$ of 
\[S(\phi):=I_{\phi}(\Q)\backslash X_p(\phi)\times X^p(\phi)=\varprojlim_{H^p}(\sS_{H^p}(\F)\cap S_{H^p}(\phi))\]  
lying above $x'$, and $\epsilon x^pgk=x^p$ for some $k\in K^p$. So, one has
\[\Phi^m(\tilde{x})=[\Phi^mx_p,x^p]=[x_p,\epsilon ^{-1} x^p]=\tilde{x}gk.\] 
If we use this point $\Phi^m(\tilde{x})$ of $S(\phi)$ to identify the stalk $\sF_x$ with $W_{\lambda}$,
we have $\beta(w)=\xi(k^{-1}g^{-1})w$: the automorphism (\ref{eq:Frob-Hecke_corr_at_stalk}) becomes
\[ [\Phi^m(\tilde{x}),w] \mapsto [\tilde{x},w] \mapsto [\tilde{x}g,\xi(g^{-1})w]= [\tilde{x}gk,\xi(k^{-1}g^{-1})w], \] 
hence $\mathrm{tr}(\Phi^m\circ f;\sF_x)=\mathrm{tr}\xi(\gamma_0)$ as $k^{-1}g^{-1}=(x^p)^{-1}\epsilon x^p\in G(\A_f^p)$ is conjugate to $\gamma_0$ under $G(\bar{\A}_f^p)$.

Now, we have the following successive equalities:
\begin{align} \label{eq:T(m,f)1}
T(m,f) &=\sum_{\phi} \sum_{x'\in\mathrm{Fix}_{\phi}} \mathrm{tr}(\Phi^m\circ f;\sF_{p_1'(x')}) \\
& =\sum_{(\phi,\epsilon)} |I_{\phi,\epsilon}(\Q)\backslash X(\phi,\epsilon)_{K^p_g}| \cdot \mathrm{tr}\xi(\gamma_0)   \nonumber \\ & =\sum_{(\phi,\epsilon)} \frac{\mathrm{vol}(I_0(\Q)\backslash I_0(\A_f))}{[I_{\phi,\epsilon}(\Q):I_{\phi,\epsilon}^{\mathrm{o}}(\Q)]} \cdot \mathrm{O}_{\gamma}(f^p)\cdot \mathrm{TO}_{\delta}(\phi_p) \cdot \mathrm{tr}\xi(\gamma_0) \nonumber
\end{align}
Here, in the first line, $\phi$ runs through a set of representatives for the equivalence classes of admissible morphisms, so the first equality results from Langlands-Rapoport conjecture, Conj. \ref{conj:Langlands-Rapoport_conjecture_ver1}.
In the second line, $(\phi,\epsilon)$ runs through a set of representatives for the equivalence classes of admissible pairs. We have just seen that if $\epsilon\in I_{\phi}(\Q)$ is such that $x'\in\mathrm{Fix}_{\phi}$ belongs to the subset $I_{\phi,\epsilon}(\Q)\backslash X(\phi,\epsilon)_{K^p_g}$ in the decomposition (\ref{eq:fixed_pt_set_of_Heck-corresp1}), we have $\mathrm{tr}(\Phi^m\circ f;\sF_{p_1'(x')})=\mathrm{tr}\xi(\gamma_0)$ for any Kottwitz triple $(\gamma_0;\gamma,\delta)$ attached to the admissible pair $(\phi,\epsilon)$, which gives the second equality. The third equality is Lemma \ref{lem:fixed-pt_subset_of_Frob-Hecke_corr}.

Next, we rewrite the last expression of (\ref{eq:T(m,f)1}) using (equivalence classes of) \emph{effective} Kottwitz triples as a new summation index.
For each equivalence class of effective Kottwitz triple $(\gamma_0;\gamma,\delta)$, we fix a (well-located) admissible pair $(\phi_1,\epsilon_1)$ giving rise to it. Then, the set of admissible pairs producing the same equivalence class of Kottwitz triple is in bijection with $\Sha_G(\Q,I_{\phi_1,\epsilon_1})^+=\im \left[ \widetilde{\Sha}_G(\Q,I_{\phi_1,\epsilon_1})^+ \rightarrow H^1(\Q,I_{\phi_1,\epsilon_1}) \right]$. More explicitly, for each 
\[ [a]\in \widetilde{\Sha}_G(\Q,I_{\phi_1,\epsilon_1})^+=\ker\left[\Sha^{\infty}_G(\Q,I_{\phi_1,\epsilon_1}^{\mathrm{o}})\rightarrow H^1(\A,I_{\phi_1,\epsilon_1})\right]\] 
(class of $a\in Z^1(\Q,I_{\phi_1,\epsilon_1}^{\mathrm{o}})$), the admissible pair corresponding to the image of $[a]$ in $H^1(\Q,I_{\phi_1,\epsilon_1})$ is the twist $(\phi:=a\phi_1,\epsilon_1)$ (Lemma \ref{lem:LR-Lem5.26,Satz5.25}) and the associated groups $I_{\phi,\epsilon_1}^{\mathrm{o}}$, $I_{\phi,\epsilon_1}$ are the (simultaneous) inner twists of $I_{\phi_1,\epsilon_1}^{\mathrm{o}}$ and $I_{\phi_1,\epsilon_1}$ via $a$. Let us write $I_a^{\mathrm{o}}$ and $I_a$ for these twists (of course, their isomorphism classes as $\Q$-algebraic groups depend only on the cohomology class $[a]\in H^1(\Q,I_{\phi_1,\epsilon_1}^{\mathrm{o}})$).
Then, the last line of (\ref{eq:T(m,f)1}) becomes the first line of the following identity:
\begin{align} \label{eq:T(m,f)-2}
T(m,f) & =\sum_{(\gamma_0;\gamma,\delta)}\sum_{[a]\in \widetilde{\Sha}_G(\Q,I_{\phi_1,\epsilon_1})^+} \frac{1}{|\ker[H^1(\Q,I_a^{\mathrm{o}})\rightarrow H^1(\Q,I_a)]|} \cdot \frac{\mathrm{vol}(I_a^{\mathrm{o}}(\Q)\backslash I_a^{\mathrm{o}}(\A_f))}{[I_a(\Q):I_a^{\mathrm{o}}(\Q)]} \\  
& \qquad \qquad \qquad \qquad \qquad \qquad \qquad \qquad \qquad \qquad \cdot \mathrm{O}_{\gamma}(f^p)\cdot \mathrm{TO}_{\delta}(\phi_p) \cdot \mathrm{tr}\xi(\gamma_0) \nonumber \\
&=\sum_{(\gamma_0;\gamma,\delta)} c_1(\gamma_0)\cdot |\pi_0(G_{\gamma_0})(\Q)|^{-1} \cdot i(\gamma_0;\gamma,\delta) \cdot \mathrm{O}_{\gamma}(f^p)\cdot \mathrm{TO}_{\delta}(\phi_p) \cdot  \mathrm{tr}\xi(\gamma_0), \nonumber
\end{align}
where $c_1(\gamma_0):=\tau(I_0)\cdot \mathrm{vol}(A_{G}(\R)^{\mathrm{o}}\backslash I_0(\infty)(\R))^{-1}$.

Here, in the first line, the fist sum is over a set of representatives of the equivalence classes of \emph{effective} Kottwitz triples $(\gamma_0;\gamma,\delta)$ and in the second sum $[a]$ runs through the set $\widetilde{\Sha}_G(\Q,I_{\phi,\epsilon})^+$. Two elements $[a],[a']$ of $\widetilde{\Sha}_G(\Q,I_{\phi_1,\epsilon_1})^+$ give equivalent admissible pairs $(a\phi_1,\epsilon_1)$, $(a'\phi_1,\epsilon_1)$ if and only if $[a]$, $[a']$ map to the same element in $H^1(\Q,I_{\phi_1,\epsilon_1})$. 
So, to establish the first equality, we need to prove that for each $[a]\in \widetilde{\Sha}_G(\Q,I_{\phi_1,\epsilon_1})^+$, the set of such elements in $ \widetilde{\Sha}_G(\Q,I_{\phi_1,\epsilon_1})^+$ is in bijection with $\ker[H^1(\Q,I_a^{\mathrm{o}})\rightarrow H^1(\Q,I_a)]$; this will then also show that the latter set has the same size for all the elements $[a]$ in $ \widetilde{\Sha}_G(\Q,I_{\phi_1,\epsilon_1})^+$ that map to the same element in $H^1(\Q,I_{\phi_1,\epsilon_1})$.
First, the subset of $H^1(\Q,I_{\phi_1,\epsilon_1}^{\mathrm{o}})$ consisting of such elements is in bijection with $\ker[H^1(\Q,I_a^{\mathrm{o}})\rightarrow H^1(\Q,I_a)]$ \cite[Prop.35bis]{Serre02}. So, it suffices to show that
if $[a']\in H^1(\Q,I_{\phi_1,\epsilon_1}^{\mathrm{o}})$ has the same image in $H^1(\Q,I_{\phi_1,\epsilon_1})$ as $[a]$, then $[a']\in \widetilde{\Sha}_G(\Q,I_{\phi_1,\epsilon_1})^+$. Since $(I_{\phi_1})_{\R}$ is a subgroup of the inner form $G'$ of $G_{\R}$ that has compact adjoint group, the map $H^1(\R,I_{\phi_1,\epsilon_1}^{\mathrm{o}})\rightarrow H^1(\R,G')$ is injective \cite[4.4.5]{Kisin17}, which implies that $[a']\in \Sha^{\infty}(\Q,I_{\phi_1,\epsilon_1}^{\mathrm{o}})=\ker[ H^1(\Q,I_{\phi_1,\epsilon_1}^{\mathrm{o}})\rightarrow H^1(\R,I_{\phi_1,\epsilon_1}^{\mathrm{o}})]$. Similarly, the map $\Sha^{\infty}(\Q,I_{\phi_1,\epsilon_1}^{\mathrm{o}})\cong \Sha^{\infty}(\Q,G_{\epsilon_1}^{\mathrm{o}})\rightarrow \Sha^{\infty}(\Q,G)$ factors through $\Sha^{\infty}(\Q,I_{\phi_1})$, which implies that $[a']\in \Sha^{\infty}_G(\Q,I_{\phi_1,\epsilon_1}^{\mathrm{o}})$.

For the second equality, we use the following two facts (E1), (E2): 

(E1) There exists an equality of numbers:
\[|\ker[H^1(\Q,I_a^{\mathrm{o}})\rightarrow H^1(\Q,I_a)]|\cdot [I_a(\Q):I_a^{\mathrm{o}}(\Q)] =|\pi_0(G_{\gamma_0})(\Q)|.\] 
Indeed, we recall \cite[5.5]{Serre02} that for any (not-necessarily connected) reductive $\Q$-group $I$ (especially for $I=I_{\phi_1,\epsilon_1}$), the exact sequence $1\rightarrow I^{\mathrm{o}}\rightarrow I\rightarrow \pi_0(I) \rightarrow 1$ gives rise to a natural action $\pi_0(I)(\Q)$ on $H^1(\Q,I^{\mathrm{o}})$ which we normalize to be a left action and write $c\cdot \alpha$ for $c\in \pi_0(I)(\Q)$ and $\alpha \in H^1(\Q,I^{\mathrm{o}})$. 
One easily checks that for $c\in \pi_0(I)(\Q)$ and $[a] \in H^1(\Q,I^{\mathrm{o}})$, $c\cdot [a]$ equals the image of $c$ under the composite map 
\[ \pi_0(I)(\Q) =\pi_0(I_a)(\Q) \stackrel{\partial_a}{\rightarrow}  H^1(\Q,I_a^{\mathrm{o}}) \isom H^1(\Q,I^{\mathrm{o}}),\] 
where $\partial_a$ is the obvious coboundary map attached to the inner twist $I_a$ of $I$ via $a$ and the last bijection is defined by $[x_{\tau}]\mapsto [x_{\tau}a_{\tau}]$ (and thus sends the distinguished element to $[a_{\tau}]$) \cite[Prop.35bis.]{Serre02}.
So, the stabilizer subgroup of $[a]$ for the action of $\pi_0(I)(\Q)$ on $H^1(\Q,I^{\mathrm{o}})$ is isomorphic to $\ker(\partial_a)=I_a(\Q)/I_a^{\mathrm{o}}(\Q)$ and the orbit of $[a]$ is in bijection with $\im(\partial_a)=\ker[H^1(\Q,I_a^{\mathrm{o}})\rightarrow H^1(\Q,I_a)]$. So, we obtain the equality
\[ |\ker[H^1(\Q,I_a^{\mathrm{o}})\rightarrow H^1(\Q,I_a)]|\cdot [I_a(\Q):I_a^{\mathrm{o}}(\Q)]=|\pi_0(I)(\Q)| \]
(In particular, this product quantity is independent of the inner twist of $I$ by a cocycle in $Z^1(\Q,I^{\mathrm{o}})$.)
Our claim follows since $G_{\gamma_0}=I_a$ for some cocyle $a$ (with $I=I_{\phi_1,\epsilon_1}$) and $\pi_0(I_{\phi_1,\epsilon_1})=\pi_0(G_{\gamma_0})$.

(E2) Since for $(\phi,\epsilon):=(a\phi_1,\epsilon_1)$, $I_{\phi,\epsilon}^{\mathrm{o}}$ is an inner-twist $I_0$ and the Tamagawa number is invariant under inner twist, by (\ref{eq:Tamagawa_number}) one has the equality:
\[\mathrm{vol}(I_{\phi,\epsilon}^{\mathrm{o}}(\Q)\backslash I_{\phi,\epsilon}^{\mathrm{o}}(\A_f))= \tau(I_0)\cdot \mathrm{vol}(A_{G}(\R)^{\mathrm{o}}\backslash I_0(\infty)(\R))^{-1}.\]

Hence, the summand in the first line of (\ref{eq:T(m,f)-2}) indexed by an admissible pair $(\phi,\epsilon)$ and a class $[a]\in \widetilde{\Sha}_G(\Q,I_{\phi_1,\epsilon_1})^+$ depends only on the associated (equivalence class of) Kottwitz triple $(\gamma_0;\gamma,\delta)$, and thus the second equality holds. 
In the expression of the second line, the index $(\gamma_0;\gamma,\delta)$ originally should run through a set of representatives of \emph{effective} Kottwitz triples (i.e. arising from an admissible pair, so having trivial Kottwitz invariant). But, according to Thm. \ref{thm:LR-Satz5.25}, any Kottwitz triple $(\gamma_0;\gamma,\delta)$ with trivial Kottwitz invariant is effective if its twisted-orbital integral $\mathrm{TO}_{\delta}(\phi_p)$ is non-zero: one easily checks (cf. \cite[$\S$1.4, $\S$1.5]{Kottwitz84b}) that non-vanishing of $\mathrm{TO}_{\delta}(\phi_p)$ is equivalent to non-emptiness of the set $Y_p(\delta)$ (\ref{eq:Y_p(delta)}). Therefore, in this sum we may as well take simply (a set of representatives of) \emph{all} Kottwitz triples with trivial Kottwitz invariant. This finishes the proof of (1).

(2) The first claim follows from (1) in view of the Deligne's conjecture proved by Fujiwara \cite{Fujiwara97}, \cite{Var07}. When $\sS_K$ is proper or $f^p$ is the identity, we can simply invoke the Grothendieck-Lefschetz fixed point formula for lisse sheaves. The properness holds if $G^{\ad}$ is anisotropic since the valuative criterion holds by \cite{Lee12}.
\end{proof}

\begin{rem} \label{rem:comments_on_Milne92}
Milne \cite[Cor.7.10]{Milne92} claimed to have proved this theorem, in the original setting (i.e when the level group $\mbfKt_p$ is hyperspecial and $G^{\der}=G^{\uc}$). His proof is incomplete and flawed, in two respects. First, as was mentioned before, he misquotes the definition of \textit{admissible pair} \cite[p.189]{LR87}  (a pair $(\phi,\epsilon)$ is admissible in his sense if and only if it is admissible in the original sense and also $\mbfK_p$-effective in our sense, cf. Remark \ref{rem:admissible_pair}), so his statements in \textit{loc. cit.} using this terminology/notion require critical reading. 
Secondly and more seriously, in the proof of his Corollary 7.10, he claims that \textit{if a Frobenius triple does not satisfy the condition of (7.5) then it contributes zero to the sum on the right} (a \textit{Frobenius triple} in Milne's work is the same as a Kottwitz triple with trivial Kottwitz invariant). This non-trivial statement (which is simply an \emph{effectivity criterion of Kottwitz triple}) was never justified in \textit{loc. cit.}, nor elsewhere, until our proof of Theorem \ref{thm:LR-Satz5.21} (which is also valid in a more general setting).
Also, we remark that for \emph{general} $g$, one needs extra arguments, more than what Milne outlines based on \cite{Kottwitz84b} which was intended mainly for $g=1$ (or at best for those $g$'s lying in a compact open subgroup of $G(\A_f^p)$).
\end{rem}

%%%%%%%%%%%%%%%%%%%%
%%%%%%%%%%%%%%%%%%%%

\subsection{Unconditional proof of Kottwitz conjecture} \label{subsec:uncond_proof_K-formula}

In this subsection, we prove Kottwitz conjecture for Shimura varieties of Hodge-type with hyperspecial level. The main ingredients are as follows:
\begin{itemize}
\item[(I)] Definition/description of isogeny classes (in terms of affine Deligne-Lusztig varieties) \cite[Prop. 2.1.3]{Kisin17} and their moduli interpretation (\textit{loc. cit.} Prop. 1.4.15), and the resulting description of $\sS_{\mbfK_p}(G,X)(\Fpb)$ as disjoint union of isogeny classes;
\item[(SCM)] Strong CM-lifting theorem (\textit{loc. cit.} Cor. 2.2.5);
\item[(Tate)] Generalization of the Tate's theorem on endomorphisms of abelian varieties over finite fields (\textit{loc. cit.} Cor. 2.3.2);
\item[(T)] Twisting method of isogeny classes or CM points (\textit{loc. cit.} Prop. 4.4.8, 4.4.13);
\item[(E)] Effectivity criterion of Kottwitz triples (Thm. \ref{thm:LR-Satz5.25b2}).
\end{itemize}

Our arguments will run in parallel to those of Langlands and Rapoport in the previous subsection which derive Kottwitz conjecture from Langlands-Rapoport conjecture. For that, we have to reformulate the above geometric results of Kisin into purely group-theoretic statements.

\begin{defn}
(1) For a $\Q$-torus $T$ and a connected reductive $\Q$-group $G$ with $\Qb\operatorname{-}\mathrm{rk}(T)=\Qb\operatorname{-}\mathrm{rk}(G)$, a \emph{stable conjugacy class of $\Q$-embeddings} $T\hookrightarrow G$ is, by definition, an equivalence class of $\Q$-embeddings $T\hookrightarrow G$ with respect to stable conjugacy relation: two $\Q$-embeddings $i_T,i_T':T\hookrightarrow G$ are \emph{stably conjugate} if and only if there exists $g\in G(\Qb)$ such that $i_T'=\Int(g)\circ i_T$ (in particular, $\Int(g)|_{i_T(T)}$ induces a transfer of maximal torus $i_T(T)\isom i_T'(T)$). 

(2) For a Shimura datum $(G,X)$, a \emph{stable conjugacy class of special Shimura sub-data} $(T,h)$ is, by definition, an equivalence class of special Shimura sub-datum $(T,h)$ with respect to the following (stable conjugacy) equivalence relation: two special Shimura sub-data $(T,h)$, $(T',h')$ are \emph{stably conjugate} if and only if there exist $g\in G(\Qb)$ inducing a $\Q$-isomorphism $\Int (g)|_T:T\isom T'$ and $g_{\infty}\in G(\R)$ such that $\Int(g)|_{T_{\R}}=\Int(g_{\infty})|_{T_{\R}}$ and $h'=\Int(g_{\infty})(h)$.
\end{defn}

We remind the reader that we have fixed an embedding $\Qb\hookrightarrow \Qpb$.

%%%%%%%%%%%%%%%%%%%%
%%%%%%%%%%%%%%%%%%%%
\begin{thm} \label{thm:Kisin17_Cor.1.4.13,Prop.2.1.3,Cor.2.2.5} \cite[Cor.1.4.13, Prop.2.1.3, Cor.2.2.5]{Kisin17} 
As a set with action by $\langle\Phi\rangle\times Z_G(\Qp)\rtimes G(\A_f^p)$, $\sS_{\mbfK_p}(\Fpb)$ is a disjoint union of subsets $S(\sI)$, called \emph{isogeny classes}, endowed with an action by the same group:
\begin{equation} \label{eq:isogeny_decomp}
 \sS_{\mbfK_p}(\Fpb)\isom \bigsqcup_{\sI}S(\sI).
\end{equation}

(1) For each isogeny class $S(\sI)$, there exist a connected reductive $\Q$-group $I_{\sI}$, an element $b\in G(\mfk)$, and embeddings of group schems (over $\Qp$ and $\Ql$ for every finite $l\neq p$)
\[ i_p:(I_{\sI})_{\Qp} \hookrightarrow J_b,\quad i_l:(I_{\sI})_{\Ql} \hookrightarrow G_{\Ql}\]
such that for almost all $l\neq p$, $i_l$ extends to an embedding $i_l:(I_{\sI})_{\Z_l}\hookrightarrow G_{\Z_l}$ between reductive $\Z_l$-group schemes $(I_{\sI})_{\Z_l}$, $G_{\Z_l}$, and in terms of which, one has
\[S(\sI):=I_{\sI}(\Q)\backslash [X(\{\mu_X\},b)_{\mbfK_p}\times G(\A_f^p)].\]
Here, $I_{\sI}$ acts on $X(\{\mu_X\},b)_{\mbfK_p}\times G(\A_f^p)$ diagonally via $i_p\times i^p$, where $i^p$ denotes the restricted product $\prod_{l\neq p}'i_l$, and
$\Phi$ acts on $S(\sI)$ via its action on $X(\{\mu_X\},b)_{\mbfK_p}$ by $(b\sigma)^r$ ($r=[\kappa(\wp):\Fp]$) while $g\in G(\A_f^p)$ acts on $S(\sI)$ via its right translation of $G(\A_f^p)$.

(2) The $\Q$-group $I_{\sI}$ has the same $\Qb$-rank as $G$. There exists an embedding $Z(G)\subset I_{\sI}$ such that $(I_{\sI}/Z(G))_{\R}$ is a subgroup of a compact inner form of $G^{\ad}_{\R}$.

(3) For every maximal $\Q$-torus $T\subset I_{\sI}$, there exists a stable conjugacy class of $\Q$-embeddings
\[i_T:T\hookrightarrow G \]
with the following properties: 

For a choice of embedding $i_T$, there exist a $G(\mfk)\times G(\A_f^p)$-conjugate of the triple of (1), denoted again by $(b,i_p, i^p)$, with $b\in i_T(T)(\mfk)$, and $h\in X\cap \Hom(\dS,i_T(T)_{\R})$ such that

(i) the embeddings $i_p:I_{\sI}(\Qp)\hookrightarrow J_b(\Qp)$, $i^p:I_{\sI}(\A_f^p)\hookrightarrow G(\A_f^p)$
are $T(\Qp)\times T(\A_f^p)$-equivariant:
\begin{equation} \label{eq:(i_p,i^p)_adapted_to_i_T}
i_p|_{T(\Qp)}=i_T|_{T(\Qp)},\quad {i^p}|_{T(\A_f^p)}=i_T|_{T(\A_f^p)},
\end{equation}
($i_T|_{T(\Qp)}$ factors through $J_{b}(\Qp)\subset G(\mfk)$, as $b\in  i_T(T)(\mfk)$)
and that

(ii) the two elements $[b]$, $[b_1]$ of $B(i_T(T))$ are equal, where $b_1$ is an element of $i_T(T)(\Qpnr)$ defined by any unramifed $i_T(T)(\Qpb)$-conjugate of $\psi_{i_T(T),\mu_h}(p)\circ\zeta_p$.

Moreover, the stable conjugacy class of $\Q$-embeddings $i_T$ depends only on the $I_{\sI}(\Q)$-conjugacy class of the embedding $T\subset I_{\sI}$.

(4) For any special Shimura sub-datum $(T,h)$ of $(G,X)$, if $\sI=\sI_{T,h}$ is the isogeny class of the reduction of the special point $[h,1]\in Sh_{\mbfK_p}(G,X)(\Qb)$, there exists an $I_{\sI}(\Q)$-conjugacy class of $\Q$-embeddings 
\[j_{T,h}:T\hookrightarrow I_{\sI}\] 
with the following properties: 

(iii) for any choice of $j_{T,h}$ in the conjugacy class, the associated (by the claim of (3)) stable conjugacy class of $\Q$-embeddings $i_T:T\hookrightarrow G$ contains the inclusion $T\subset G$. 

(iv) for any maximal $\Q$-torus $T\subset I_{\sI}$ and any choice of $i_T:T\hookrightarrow G$, the $I_{\sI}(\Q)$-conjugacy class of embeddings $T\stackrel{i_T}{\rightarrow}i_T(T)\stackrel{j_{i_T(T),h}}{\hookrightarrow} I_{\sI}$ obtained from $(i_T(T),h)$ contains the inclusion $T\subset I_{\sI}$.

Moreover, the isogeny class $\sI_{T,h}$ and the conjugacy class of $j_{T,h}$ both depend only on the stable conjugacy class of $(T,h)$.
\end{thm}

%%%%%%%%%%%%%%%%%%%%
\begin{rem}
(1) The statement (3) needs some explanation.
Note that for a triple $(b,i_p,i^p)$ as in (1) and any $(g_p,g^p)\in G(\mfk)\times G(\A_f^p)$, we obtain a new triple 
\[(b':=g_pb\sigma(g_p^{-1}),\ \Int(g_p)\circ i_p:(I_{\sI})_{\Qp}\hookrightarrow J_b\isom J_{b'},\ \Int(g^p)\circ i^p),\] 
which again satisfies (1): left multiplication by $g_p\times g^p$ induces a $\langle\Phi\rangle\times Z_G(\Qp)\rtimes G(\A_f^p)$-equivariant bijection
\[I_{\sI}(\Q)\backslash X(\{\mu_X\},b)_{\mbfK_p}\times G(\A_f^p) \isom I_{\sI}(\Q)\backslash X(\{\mu_X\},b')_{\mbfK_p}\times G(\A_f^p).\]
Furthermore, as will be evident from the proof, such new triple enjoys all the other properties. Two such triples will be said to be \emph{equivalent}: we will not distinguish them.

(2) We will say that a triple $(b,i_p,i^p)$ in (1) is \textit{adapted to} a given embedding $i_T$ if it satisfies conditions (i) and (ii) of (3).
For fixed $i_T$, any two triples $(b,i_p,i^p)$ adapted to $i_T$ (in fact, satisfying just condition (i)) differ by conjugation by an element of $i_T(T)(\mfk)\times i_T(T)(\A_f^p)$ (since $\mathrm{rk}(T)=\mathrm{rk}(G)$ and $T(\Qp)$ is Zariski-dense in $T_{\mfk}$) and thus for such triples the $\sigma$-conjugacy of $b$ in $B( i_T(T))$ does not depend on the choice of the triple. 

(3) In (3), we are not asserting that there exists a unique stable conjugacy class of $\Q$-embeddings $T\hookrightarrow G$ satisfying the conditions. 
%On the contrary, as will be shown in the proof, it rather depends on the choice of some cocharacter of $T$ (for any fixed choice of $T$ in its conjugacy class) which then also gives rise to $h$. 
%%(4) According to Langlands-Rapoport conjecture, it is expected that to each isogeny class, there exists a \emph{distinguishsed} equivalence class of triples $(b,i_p,i^p)$ of Thm. \ref{thm:Kisin17_Cor.1.4.13,Prop.2.1.3,Cor.2.2.5} which is constructed in a natural manner and possesses certain additional properties.
\end{rem}

Later, we will provide further information on the group $I_{\sI}$ and the embeddings $(i_p,i^p)$. But, before moving on,
we derive a first (primitive) description of the fixed point set of the Frobenius-twsted Hecke correspondence $\Phi^m\circ f$ (\ref{eq:Hecke_corr_twisted_by_Frob}) acting on an isogeny class $S(\sI)_{K^p}:=S(\sI)/K^p\subset \sS_{K}(\Fpb)$: 
\[S(\sI)_{K}\stackrel{p_1'}{\longleftarrow} S(\sI)_{K_g} \stackrel{p_2'=\Phi^m\circ p_2}{\longrightarrow} S(\sI)_{K}.\]
By the same (elementary) argument (\cite[$\S$1.4]{Kottwitz84b}, \cite[Lem. 5.3]{Milne92}) which yielded the description (\ref{eq:fixed_pt_set_of_Heck-corresp1}), under the same assumption on $K^p$, the fixed point set decomposes into disjoint subsets (cf. \cite[1.4.3, 1.4.4]{Kottwitz84b}):
\begin{align} \label{eq:fixed_pt_set_of_Heck-corresp2}
S_{K}(\sI)^{\Phi^m\circ f=\mathrm{Id}} &= \bigsqcup_{\epsilon} I_{\sI,\epsilon}(\Q)\backslash [ X_p(\sI,\epsilon)\times X^p(\sI,\epsilon,g)/K^p_g ] ,
\end{align}
where the index $\epsilon$ runs through a set of representatives in $I_{\sI}(\Q)$ for the conjugacy classes of $I_{\sI}(\Q)/Z(\Q)_K$, $I_{\sI,\epsilon}$ is the centralizer of $\epsilon$ in $I_{\sI}$ (regarded as an algebraic $\Q$-subgroup of $I_{\sI}$), and 
\begin{align*}
X_p(\sI,\epsilon) &:=\{\ x_p\in X(\{\mu_X\},b)_{\mbfK_p}  \ \ |\ \  i_p(\epsilon) x_p=(b\sigma)^n x_p\ \}, \\
X^p(\sI,\epsilon,g) &:=\{\ x^p\in G(\A_f^p) \ |\ \  i^p(\epsilon) x^pg=x^p\text{ mod }K^p\ \}.
\end{align*}

\begin{proof} (of Thm. \ref{thm:Kisin17_Cor.1.4.13,Prop.2.1.3,Cor.2.2.5})
For the proof, we use freely the notations of \cite{Kisin17}. In this proof, all references will be to this work, unless stated otherwise. We first give a very brief review of some results in \textit{loc. cit.} that are necessary for the proof of the theorem.

(A) There exists a $\Z_{(p)}$-lattice $V_{\Z_{(p)}}$ of $V$ and a set of tensors $\{s_{\alpha}\}_{\alpha}$ on it which defines the reductive closed $\Zp$-subgroup scheme $G_{\Zp}$ of $G_{\Qp}$ giving the hyperspecial subgroup $\mbfK_p$ (i.e. $\mbfK_p=G_{\Zp}(\Zp)$) \cite[Prop.1.3.2]{Kisin10}. 
Let $\pi:\uvA\rightarrow \sS_K$ be the universal abelian scheme over $\sS_K$ (for sufficiently small $K^p$) and $\mathcal{V}=R^1\pi_{\ast}\Omega^{\bullet}$ the first relative de Rham chomology (algebraic vector bundle).
The tensors $\{s_{\alpha}\}_{\alpha}$ give rise to horizontal sections $\{s_{\alpha,\Betti}\}_{\alpha}$ on the local system $R^1\pi^{\mathrm{an}}_{\ast}(\Z_{(p)})$ over $Sh_K(G,X)$ (with $\pi^{\mathrm{an}}$ denoting the analytification of $\pi$) and sections $\{s_{\alpha,\dR}\}_{\alpha}$ on $\mathcal{V}=R^1\pi_{\ast}\Omega^{\bullet}$ which correspond to each other (for the same $\alpha$) via the de Rham isomorphism over $\C$ \cite[1.3.6]{Kisin17}.

(B) Suppose $x\in \sS_{\mbfK_p}(k)$ for a finite extension $k\subset \Fpb$ of $\kappa(\wp)$. Let $\uvA_x$ be the underlying abelian variety over $k$ and $\xb$ the $\Fpb$-point induced by $x$.
Let $H^1_{\et}(\uvA_{\xb}/\Ql)$ and $H_{\cris}^1(\uvA_x/K_0)$ be respectively the $l$-adic \'etale and the cristalline cohomology groups of $\uvA_{\xb}$ (for $l\neq p$) and $\uvA_x$, where $K_0:=W(k)[1/p]\ (\subset\Qpb)$. 
We let $H^{\et}_1(\uvA_{\xb},\Ql):=\Hom_{\Ql}(H^1_{\et}(\uvA_{\xb}/\Ql),\Ql)$ and $H^{\cris}_1(\uvA_x/K_0):=\Hom_{K_0}(H^1_{\cris}(\uvA_{x}/K_0),K_0)$ denote their linear dual (homology) groups. The latter group is equipped with the Frobenius operator $\phi$: $\phi(f)(v):=p^{-1}\cdot{}^{\sigma}f(Vv)$ for $f\in H_1^{\cris}$ and $v\in H^1_{\cris}$, where $V$ is the Verschiebung on $H^1_{\cris}$. 
Recall that the relative Frobenius morphism  $\Fr_{\uvA_{x}/k}$ of $\uvA_{x}/k$ acts on these homology spaces ($\Ql$ and $K_0$-linearly) by the geometric $p^{[k:\Fp]}$-Frobenius $\Fr_k^{-1}$ in $\Gal(\Fpb/k)$ and by the inverse of $\phi^{[k:\Fp]}$, respectively.

Then, there exist tensors $\{s_{\alpha,l,x}\}_{\alpha}$ on $H^1_{\et}(\uvA_{\xb}/\Ql)\ (l\neq p)$ and tensors $\{s_{\alpha,0,x}\}_{\alpha}$ on $H_{\cris}^1(\uvA_x/K_0)$ which are Frobenius invariant (for the geometric Frobenius acting on $H^1_{\et}(\uvA_{\xb}/\Ql)$ and the absolute Frobenius automorphism acting on $H^1_{\cris}(\uvA_x/K_0)$); for the construction of $s_{\alpha,l,x}$ and $s_{\alpha,0,x}$, see \cite[(2.2)]{Kisin10} and \cite[Prop.1.3.9, 1.3.10]{Kisin17} respectively.
Also, there exist isomorphisms matching the respective tensors for each $\alpha$:
\begin{equation} \label{eq:isom_eta}
\begin{split}
\eta_l:(V,\{s_{\alpha}\})\otimes\Ql &\lisom (H^{\et}_1(\uvA_{\xb},\Ql),\{s_{\alpha,l,x}\}) , \\
\eta_p:(V,\{s_{\alpha}\})\otimes K_0 &\lisom (H^{\cris}_1(\uvA_x/K_0),\{s_{\alpha,0,x}\}).
\end{split}
\end{equation}
Most of the time, we are just contented with an isomorphism defined over $\Qpnr$ (or even over $\mfk$):
\begin{equation} \label{eq:isom_eta_nr}
\eta_p^{\nr}:(V,\{s_{\alpha}\})\otimes \Qpnr \lisom (H^{\cris}_1(\uvA_x/\Qpnr),\{s_{\alpha,0,x}\}).
\end{equation}
For almost all $l\neq p$ (in particular, such that $G_{\Ql}$ is unramified), we may assume that the following conditions hold:
there exist a $\Z_{(l)}$-lattice $V_{\Z_{(l)}}$ of $V$ such that the tensors $\{s_{\alpha}\}_{\alpha}$ live on it and defines a reductive $\Z_l$-subgroup scheme $G_{\Z_l}$ of $G_{\Ql}$, and a similar statement holds true of the lattice $H^{\et}_1(\uvA_{\xb},\Z_l)$ and the tensors $\{s_{\alpha,l,x}\}$. Further, for these $\Z_l$-structures, there exists an $\Z_l$-isomorphism extending $\eta_l$; we denote it again by $\eta_l$.

(C) We define $\gamma_l\in G(\Ql)$ ($l\neq p$), $\gamma_p\in G(K_0)$ by
\begin{equation}  \label{eq:gamma_v}
\gamma_v^{-1}:=\Int(\eta_v^{-1})(\Fr_{\uvA_{x}/k}).
\end{equation}
These elements $\gamma_l$, $\gamma_p$ are well-defined up to conjugacy in $G(\Ql)$ and $G(K_0)$, respectively.
We also define $\delta\in \mathrm{GL}(V_{K_0})$ by 
\[\delta(1_V\otimes\sigma):=\Int(\eta_p^{-1})(\phi),\]
so one has $\gamma_p=\Nm_n\delta:=\delta\sigma(\delta)\cdots\sigma^{n-1}(\delta)$ ($n=[K_0:\Qp]$).

Let $k'\subset\Fpb$ be a finite extension of $k$. For each finite place $l\neq p$ of $\Q$, let $I_{l/k'}$ be the centralizer of $\gamma_l^{[k':k]}\in G(\Ql)$ and define $I_{p/k'}$ to be the $\sigma$-centralizer $G_{\delta\sigma}$ of $\delta$ in $G(K_0')$, where $K_0':=W[k'][\frac{1}{p}]$ \cite[p.802]{Kottwitz82}: $G_{\delta\sigma}$ is a closed subgroup of $\Res_{K_0'/\Qp}(G)$ such that $G_{\delta\sigma}(\Qp)=\{ y\in G(K_0')\ |\ y(\delta\sigma)=(\delta\sigma)y\}$. One has $(G_{\delta\sigma})_{K_0'}\simeq Z_{G_{K_0'}}(\Nm_n\delta)$.

The increasing sequence of subgroups $\{I_{l/k'}\}_{k'\subset\Fpb}$ of $G_{\Ql}$ stabilizes to a subgroup $I_l$, which also equals the centralizer of $\gamma_l^n$ in $G_{\Ql}$ for (any) sufficiently large $n\in\N$. By similar reasoning (cf. \cite[(2.1.2)]{Kisin17}), there exists a $\Qp$-subgroup $I_p$ of $J_{\delta}$ which equals $I_{p/k'}$ for all sufficiently large $k'\subset\Fpb$.
Write $\Aut_{\Q}(\uvA_{\xb})$ for the automorphism group of $\uvA_{\xb}$ in the isogeny category, and let $I_{x}\subset \Aut_{\Q}(\uvA_{\xb})$ denote the subgroup consisting of elements fixing all the tensors $\{s_{\alpha,l,x}\}_{\alpha}\ (l\neq p)$, $\{s_{\alpha,0,x}\}_{\alpha}$, regarded as an algebraic $\Q$-subgroup of $\Aut_{\Q}(\uvA_{\xb})$.

%%%%%%%%%%%%%%%%%%%%
%%%%%%%%%%%%%%%%%%%%
\begin{thm} \label{thm:Kisin17_Cor.2.3.2;Tate_isom} \cite[Cor.2.3.2]{Kisin17}
For every finite place $v$ of $\Q$, the restriction of $\Int(\eta_v)^{-1}$ if $v\neq p$ or of $\Int(\eta_p^{\nr})^{-1}$ if $v=p$ to $(I_{x})_{\Qv}$ induces an isomorphism 
\[\Int(\eta_v)^{-1}(v\neq p),\ \Int(\eta_p^{\nr})^{-1}\ :\ (I_{x})_{\Qv}\isom I_v.\] 
\end{thm}

Now, we enter into the proof of the theorem (Thm. \ref{thm:Kisin17_Cor.1.4.13,Prop.2.1.3,Cor.2.2.5}). Recall the chosen embedding $\sigma_p:\Qpb\hookrightarrow \C$ of $E(G,X)$-algebras.

(1) For each $x\in \sS_{\mbfK_p}(\Fpb)$, if $x$ is defined over a finite field $k\subset \Fpb$, 
any choice of isomorphisms $\eta_l$ (\ref{eq:isom_eta}), $\eta_p^{\nr}$ (\ref{eq:isom_eta_nr}) gives an element $b\in G(\Qpnr)$ by $\Int(\eta_p^{\nr})(b\sigma)=F$, and
embeddings defined over $\Ql$ and $\Qp$
\[ i_{x,l}:(I_{x})_{\Ql} \hookrightarrow G_{\Ql},\quad i_{x,p}:(I_{x})_{\Qp} \hookrightarrow J_b.\]
Of course, different choice of $\eta^p:=\prod'\eta_l$, $\eta_p^{\nr}$ gives rise to an equivalent triple $(b,i_{x}^p:=\prod'i_{x,l},i_{x,p})$. For almost all $l\neq p$, so that, among others, one can find a $\Z_l$-isomorphism $\eta_l$ for the $\Z_l$-structures explained above, $i_{x,l}$ extends to an embedding $(I_x)_{\Z_l}\hookrightarrow G_{\Z_l}$ of $\Z_l$-group schemes.

Then, there exists a $\langle\Phi\rangle\times Z_G(\Qp)\rtimes G(\A_f^p)$-equivariant map 
\[\iota_x:X(\{\mu_X\},b)_{\mbfK_p}\times G(\A_f^p)\rightarrow \sS_{\mbfK_p}(\Fpb),\] 
where $\langle\Phi\rangle$ is the cyclic group generated by $\Phi$ \cite[Cor. 1.4.13]{Kisin17}; the image of $\iota_x$ is the \emph{isogeny class} containing $x$. 
This gives rise to the decomposition (\ref{eq:isogeny_decomp})
\[  \bigsqcup_{x} I_{x}(\Q)\backslash [X(\{\mu_X\},b)_{\mbfK_p}\times G(\A_f^p)] \isom \sS_{\mbfK_p}(\Fpb),\]
where $x$ runs through a set of representatives of (i.e. points lying in) the isogeny classes \cite[Prop. 2.1.3]{Kisin17}.

(2) The $\Q$-group $I_{x}$ has the same $\Qb$-rank as $G$ \cite[Cor. 2.1.7]{Kisin17}.
 
(3) It is shown in the proof of \cite[Thm. 2.2.3]{Kisin17} that for every maximal $\Q$-torus $T\subset I_{x}\subset \Aut_{\Q}(\uvA_{\xb})$, and for any choice of a cocharacter $\mu_T\in X_{\ast}(T)$ satisfying the conditions of \cite[Lem.2.2.2]{Kisin17} (in particular, $\mu_X$ lies in the conjugacy class $c(G,X)$),
there exists a point $x'$ in the isogeny class of $x$ which lifts to a $K$-valued point $\tilde{x}'$ of $\sS_{\mbfK_p}$ for a finite extension $K\subset\Qpb$ of $K_0$ in such a way that the action of $T$ on $\uvA_{x'}$ (in the isogeny category) lifts to $\uvA_{\tilde{x}'}$ and $\mu_T^{-1}$ induces the Hodge filtration on $H^1_{\cris}(\uvA_{x'}/K)\cong H^1_{\dR}(\uvA_{\tilde{x}'}/K)$ (defined by $\uvA_{\tilde{x}'}$). We choose one such $\mu_T$ and denote $x'$ again by $x$, thereby assume that $x$ itself is a point lifting to $\tilde{x}$.
This implies that $T$ is a subgroup of $G$, via a choice of an isomorphism $\Q$-vector spaces endowed with a set of tensors
\begin{equation} \label{eq:Betti-isom}
\eta_{\Betti}:(H^{\Betti}_1(\uvA_{\sigma_p(\tilde{x})},\Q),\{s_{\alpha,\Betti,\sigma_p(\tilde{x})}\}) \lisom (V,\{s_{\alpha}\})
\end{equation} 
and thus $\sigma_p(\tilde{x})$ is a special (=CM) point on $Sh_{\mbfK_p}(G,X)(\Qb)$; $\eta_{\Betti}$ is well-defined up to action of $G(\Q)$ on $V$ and is provided by the moduli interpretation of $\Sh_{\mbfK}(G,X)(\C)$. 
That is, we obtain an embedding and a special Shimura sub-datum
\begin{equation} \label{eq:CM-lifting_via_T}
i_T:T\hookrightarrow G,\quad h\in X\cap \Hom(\dS,i_T(T)_{\R})
\end{equation}
(such that $\sigma_p(i_T\circ\mu_T)=\mu_h$).

For such embedding $i_T$ and $h$, we claim that there exist isomorphisms $\eta_l$ (\ref{eq:isom_eta}), $\eta_p^{\nr}$
(\ref{eq:isom_eta_nr}) which are \emph{$T$-equivariant} with respect to $i_T:T\hookrightarrow G$ and the action of $T$ on $\uvA_{x}$. 
By construction of $i_T$ via the choice of $\eta_B$ (\ref{eq:Betti-isom}), it suffices to find such $T$-equivariant isomorphisms with $(H^{\Betti}_1(\uvA_{\sigma_p(\tilde{x})},\Q),\{s_{\alpha,\Betti,\sigma_p(\tilde{x})}\})$ replacing $(V,\{s_{\alpha}\})$ (for the lifted action of $T$ on $\uvA_{\tilde{x}}$).
For $l\neq p$, this is clear since there exist \emph{canonical} $T$-equivariant isomorphisms of $\Ql$-vector spaces matching the respective tensors:
\begin{equation} \label{eq:isom_epsilon_l}
\begin{split}
\epsilon_l:(H^{\et}_1(\uvA_{\xb},\Ql),\{s_{\alpha,l,x}\}) &\lisom (H^{\et}_1((\uvA_{\tilde{x}})_{\Qpb},\Ql),\{s_{\alpha,l,\tilde{x}}\})\\
& \lisom (H^{\Betti}_1(\uvA_{\sigma_p(\tilde{x})},\Q),\{s_{\alpha,\Betti,\sigma_p(\tilde{x})}\})\otimes\Ql.
\end{split}
\end{equation}
(In fact, the tensors $s_{\alpha,l,\tilde{x}}$ are constructed by the second isomorphism, cf. \cite[(2.2)]{Kisin10}).
For $p$, we also have canonical isomorphisms of $\C$-vector spaces matching the respective tensors:
\begin{equation} \label{eq:isom_epsilon_C}
\begin{split}
(H^{\cris}_1(\uvA_{x}/K),\{s_{\alpha,0,x}\})\otimes_{\sigma_p}\C & \lisom (H^{\dR}_1(\uvA_{\tilde{x}}/K),\{s_{\alpha,\dR,\tilde{x}}\})\otimes_{\sigma_p}\C\\ 
&\lisom (H^{\Betti}_1(\uvA_{\sigma_p(\tilde{x})},\Q),\{s_{\alpha,\Betti,\sigma_p(\tilde{x})}\})\otimes\C. 
\end{split}
\end{equation}
Clearly, these isomorphisms are \emph{$T$-equivariant}. For the fact that the first isomorphism matches the respective tensors, see the proof of \cite[Prop.2.3.5]{Kisin10} (cf. \cite[Prop.1.3.9]{Kisin17}).
This implies that the functor which associates with a $K_0$-algebra $R$ the set of $T$-equivariant, tensor-matching, $R$-linear isomorphisms $H^{\cris}_1(\uvA_{x}/K_0)\otimes R \isom H^{\Betti}_1(\uvA_{\sigma_p(\tilde{x})},\Q)\otimes R$ is a $K_0$-torsor under $T_{K_0}$ (use that $T(\Q)$ is Zariski-dense in $T_{K_0}$ since $T$ is unirational). So by Steinberg's theorem $H^1(\Qpnr,T)=\{1\}$ (cf. \cite[3.2.2]{Lee16}), one can find a $\Qpnr$-isomorphism 
\begin{equation} \label{eq:isom_epsilon_ur_p} 
\epsilon_p:(H^{\cris}_1(\uvA_{x}/K_0),\{s_{\alpha,0,x}\})\otimes_{K_0}\Qpnr \lisom (H^{\Betti}_1(\uvA_{\sigma_p(\tilde{x})},\Q),\{s_{\alpha,\Betti,\sigma_p(\tilde{x})}\})\otimes\Qpnr. 
\end{equation}
taking $s_{\alpha,\dR,\tilde{x}}$ to $s_{\alpha,\Betti,\sigma_p(\tilde{x})}$ for every $\alpha$ and intertwining the two $T$-actions. 

For such $T$-equivariant $\eta_l$, $\eta_p^{\nr}$, the resulting embeddings $i_{x,p}:I_{x}(\Qp)\hookrightarrow J_b(\Qp)$, $i^p_{x}:I_{x}(\A_f^p)\hookrightarrow G(\A_f^p)$ clearly satisfy (\ref{eq:(i_p,i^p)_adapted_to_i_T}), where $b\in G(\Qpnr)$ is given by that $b\sigma$ is the absolute Frobenius element acting on $H^{\cris}_1(\uvA_x/\Qpnr)$ via $\eta_p^{\nr}$.
Note that we have $b\in i_T(T)(\Qpnr)$ because $b$ commutes with $i_T(T)(\Qp)$, thus $b\in Z_{G(\Qpnr)}(i_T(T)(\Qp))=i_T(T)(\Qpnr)$ (the equality holds since $T$ is unirational so that $T(\Qp)$ is Zariski dense in $T$).
Since $\uvA_{x}$ is the reduction of the CM point $[h,1]\in Sh_{K_T}(i_T(T),\{h\})\in Sh_K(G,X)$ (for $K_{T}:=i_T(T)(\A_f)\cap K$), property (ii) of (3) follows from Lemma \ref{lem:properties_of_psi_T,mu} and \cite[Lem.3.2.4]{Lee16}.

Next, let $j_x:T\subset I_{x}$ be a maximal $\Q$-torus and $T'=\Int(a)(T)$ for $a\in I_{x}(\Q)$. 
Then, we claim that for any $\mu_{T'} \in X_{\ast}(T')$ satisfying the conditions of \cite[Lem.2.2.2]{Kisin17}, if $\tilde{x}'$ and $i_{T'}:T'\hookrightarrow G$ are the resulting CM-lifting of a suitable point $x'$ in the isogeny class of $x$ and the embedding discussed above (which is obtained by the construction of \cite[Thm.2.2.3]{Kisin17}, via an isomorphism (\ref{eq:Betti-isom})), the two embeddings of $T$ into $G$
\[i_{T'}\circ\Int(a),\ i_T\ :\ T\hookrightarrow G\] 
are stably conjugate. 
It suffices to show that for \emph{any} two quasi-isogenies $\theta_i:\mathcal{A}_x\rightarrow \mathcal{A}_{x_i}\ (i=1,2)$ respecting weak polarizations and matching the respective e\'tale and crystalline tensors, 
and such that $x_i$ lifts to a point $\tilde{x}_i$ in characteristic zero in a way that $\theta_{i\ast}:T\isom T_i\subset I_{x_i}$ lifts to $G$ via an isomorphism (\ref{eq:Betti-isom}), the resulting embeddings $i_1,i_2:T\hookrightarrow G$ are stably conjugate.
Since the isomorphisms induced by $\theta_2\circ\theta_1^{-1}$
\[ (H^{\et}_1(\uvA_{\xb_1},\Ql),\{s_{\alpha,l,x_1}\})\isom (H^{\et}_1(\uvA_{\xb_2},\Ql),\{s_{\alpha,l,x_2}\}),\ (H^{\cris}_1(\uvA_{x_1}/\mfk),\{s_{\alpha,0,x_1}\})\isom (H^{\cris}_1(\uvA_{x_2}/\mfk),\{s_{\alpha,0,x_2}\}) \]
are compatible with $\theta_{1\ast}$, $\theta_{2\ast}$, the functor, defined on the category of $\Q$-algebras, of isomorphisms
$H^{\Betti}_1(\uvA_{\sigma_p(\tilde{x}_1)},\Q)\isom H^{\Betti}_1(\uvA_{\sigma_p(\tilde{x}_2)},\Q)$
preserving the Betti tensors and compatible with $i_1$, $i_2$ is non-empty, thus becomes a $T$-torsor by (2). In particular, it has a $\Qb$-valued point, which implies the claim.
This fact also implies that to obtain $i_T$, we may use the embedding $T\subset \Aut_{\Q}(\uvA_{\tilde{y}})$ for any point $y\in \sI$ which lifts to a special-point lifting $\tilde{y}$, regarding $T$ as a subtorus of $I_{y}$ via any quasi-isogeny $\uvA_{y}\rightarrow \uvA_{x}$.

Now, for each isogeny class $\sI$, we fix a $\Fpb$-point $x$ lying in it and define the group $I_{\sI}$ to be $I_{x}$. By choosing isomorphisms $\eta_l\ (l\neq p)$ (\ref{eq:isom_eta}), $\eta_p^{\nr}$ (\ref{eq:isom_eta}) defined over $\mfk$ such that for almost all $l\neq p$, $\eta_l$ extends over $\Z_l$, we obtain an element $b\in G(\mfk)$ and embeddings $i^p:=i_{x}^p:(I_{x})_{\A_f}\hookrightarrow G_{\A_f}$, $i_p:=i_{x,p}:(I_{x})_{\Qp}\hookrightarrow J_b$ of group schemes over $\A_f^p$ and $\Qp$. We have already shown statements (1) - (3) for $I_{x}$.

For (4), if $\sI$ is the isogeny class of the reduction $x'$ of the special point $[h,1]$ and $x$ is the prechosen point of $\sI$, the natural $\Q$-embedding $j_{x'}:T\hookrightarrow \Aut_{\Q}(\uvA_{x'})$ factors through the subgroup $I_{x'}$ defined as above for $x'$. We define $j_{T,h}$ to be the composite $T\hookrightarrow I_{x'}\isom I_{x}$ for any isomorphism $I_{x'}\isom I_{x}$ induced by a choice of an isogeny $\uvA_{x'} \rightarrow \uvA_{x}$ preserving the (\'etale and crystalline) tensors and the weak polarizations. Clearly, the $I_{x}(\Q)$-isogeny class of such embeddings does not depend on the choice of the isogeny. In view of this construction, the properties of (4) are immediate.
The last property that the isogeny class $\sI_{T,h}$ depends only on the stable conjugacy class of $(T,h)$ will be established later in Prop. \ref{prop:Kisin17_Prop.4.4.8}.
\end{proof}

%%%%%%%%%%%%%%%%%%%%
\begin{defn} \label{defn:admissible_pair2}
(cf. Def. \ref{defn:admissible_pair})
A pair $(\sI,\epsilon)$ consisting of an isogeny class $\sI\subset\sS_{\mbfK}(\Fpb)$ and an element $\epsilon$ of $I_{\sI}(\Q)$ is \emph{admissible} of level $n=m[\kappa(\wp):\Fp]$ ($m\in\N$) if 
for a triple $(b\in G(\mfk),i_p,i^p)$ in Theorem \ref{thm:Kisin17_Cor.1.4.13,Prop.2.1.3,Cor.2.2.5},
there exists $x_p\in G(\mfk)/\mbfKt_p$ such that 
\[i_p(\epsilon)x_p=(b\sigma)^nx_p\] 
(equiv. there exists  $x\in G(\mfk)$ such that $i_p(\epsilon)x\rtimes \sigma^n=(b\rtimes\sigma)^nx$, cf. Lemma \ref{lem:Kottwitz84-a1.4.9_b3.3}).
Two such pairs $(\sI,\epsilon)$, $(\sI',\epsilon')$ are said to be \emph{equivalent} if $\sI=\sI'$ and $\epsilon'=\Int(g)(\epsilon)$ for some $g\in I_{\sI}(\Q)$.
\end{defn}
Clearly, the admissibility condition in this definition does not depend on the choice of a representative $(b,i_p)$ in its equivalence class.

We have the same name ``admissible'' for two different (but closely related) definitions:
admissible in the sense of Def. \ref{defn:admissible_pair} for a pair $(\phi,\epsilon)$ consisting of a Galois gerb morphism $\phi:\fP\rightarrow \fG_G$ and an element $\epsilon\in I_{\phi}(\Q)(\subset G(\Qb))$, 
and admissible in the above sense for a pair $(\sI,\epsilon)$ consisting of an isogeny class $\sI$ and an element $\epsilon\in I_{\sI}(\Q)$. To avoid confusion, we will use the words \emph{LR-pair}, \emph{LR-admissible} in the former situation and the words \emph{K-pair}, \emph{K-admissible} in the latter situation.

Next, with any K-admissible K-pair $(\sI,\epsilon\in I_{\sI}(\Q))$, say, of level $n=mr$, we associate a Kottwitz triple imitating the recipe for admissible LR-pairs.
First, by Lemma \ref{lem:Kottwitz84-a1.4.9_b3.3} again, there exists $c\in G(\mfk)$ such that 
\begin{equation} \label{eq:(epsilon,b,c)->delta2}
c(i_p(\epsilon)^{-1}(b\rtimes\sigma)^n)c^{-1}=\sigma^n,
\end{equation}
which implies that $\delta:=cb\sigma(c^{-1})\in G(L_n)$ and $\Nm_n\delta=c\epsilon'c^{-1}$; the $\sigma$-conjugacy class of $\delta$ in $G(L_n)$ depends only on the K-pair $(\sI,\epsilon)$ (i.e. depends on the choice of neither $c$ or of a representative $(b,i_p)$ in its equivalence class.
We put $\gamma=(\gamma_l)_{l\neq p}:=i^p(\epsilon)$; its $G(\A_f^p)$-conjugacy class is also uniquely attached to $(\sI,\epsilon)$. 
To define $\gamma_0\in G(\Q)$, we choose a maximal $\Q$-torus $T\subset I_{\sI}$ with $\epsilon\in T(\Q)$, and fix an embedding $i_T:T\hookrightarrow G$ and embeddings $(i_p,i^p)$ accordingly as in Thm. \ref{thm:Kisin17_Cor.1.4.13,Prop.2.1.3,Cor.2.2.5} (especially, satisfying condition (3)). Then, we obtain a triple of elements in $G(\Q)\times G(L_n)\times G(\A_f^p)$
\begin{equation} \label{eq:K-triple_for_isogeny_adm.pair}
(\gamma_0;\gamma,\delta):=(i_T(\epsilon);i^p(\epsilon),cb\sigma(c^{-1})),
\end{equation} 
where $n$ is the level of $(\sI,\epsilon)$. 
We will show below that this triple is a stable Kottwitz-triple with trivial Kottwitz invariant  whose stable equivalence class depends only on the K-pair $(\sI,\epsilon)$ (in particular, not depending on the choices of $i_T$ and $T$); cf. \autoref{subsubsec:pre-Kottwitz_triple}, \autoref{subsubsec:Kottwitz_invariant}, \autoref{subsubsec:K-triple_attached_to_adm.pair}.
For that end and for our proof of Kottwitz conjecture coming later, the following simple observations are critical:

%%%%%%%%%%%%%%%%%%%%
\begin{lem} \label{lem:key_observations}
(1) A K-pair $(\sI,\epsilon\in I_{\sI}(\Q))$ is K-admissible (i.e. admissible in the sense of Def. \ref{defn:admissible_pair2}) if and only if for a (equiv. any) maximal $\Q$-torus $T$ of $I_{\sI}$ containing $\epsilon$, the special LR-pair $(\psi_{i_T(T),\mu_h},i_T(\epsilon))$ is LR-admissible (i.e. admissible in the sense of Def. \ref{defn:admissible_pair}), where $i_T:T\hookrightarrow G$, $h\in X\cap \Hom(\dS,i_T(T)_{\R})$, and $(b\in i_T(T)(\mfk),i_p,i^p)$ are as in Thm. \ref{thm:Kisin17_Cor.1.4.13,Prop.2.1.3,Cor.2.2.5}.

In this case, the triple $(\gamma_0;\delta,\gamma):=(i_T(\epsilon);cb\sigma(c^{-1}),i^p(\epsilon))$ (\ref{eq:K-triple_for_isogeny_adm.pair}) is stably equivalent to the stable Kottwitz triple attached to the special LR-pair $(\psi_{i_T(T),\mu_h},i_T(\epsilon))$ as constructed in Prop. \ref{prop:Kottwitz_triple}. 

Moreover, its stable equivalence class does not depend on the choice of $c$, $i_T$, $T$. We denote its equivalence class by $\kappa(\sI,\epsilon)$; we use similar notation $\kappa(\phi,\epsilon)$ for an LR-admissible LR-pair $(\phi,\epsilon)$.

(2) A special LR-pair $(\psi_{T,\mu_h},\epsilon\in T(\Q))$ is LR-admissible if and only if
the K-pair $(\sI,j_{T,h}(\epsilon))$ is K-admissible, where $\sI\subset \sS_{\mbfK_p}(\Fpb)$ is the isogeny class of the reduction of $[h,1]\in Sh_{\mbfK_p}(G,X)(\Qb)$ and $j_{T,h}:T\hookrightarrow I_{\sI}$ is any embedding in the $I_{\sI}(\Q)$-conjugacy class attached to $(T,h)$ in Thm. \ref{thm:Kisin17_Cor.1.4.13,Prop.2.1.3,Cor.2.2.5}, (4).
In this case the two Kottwitz triples $\kappa(\psi_{T,\mu_h},\epsilon)$, $\kappa(\sI,j_{T,h}(\epsilon))$ are equivalent.
\end{lem}

\begin{proof}
(1) Set $\phi:=\psi_{i_T(T),\mu_h}$. Since one has $i_T(T)(\bar{\A}_f^p)\cap X^p(\phi)\neq\emptyset$ for any \emph{special} admissible morphism $\phi$ into $\fG_{i_T(T)}$ (Lemma \ref{lem:properties_of_psi_T,mu}), to check LR-admissibility for the special LR-pair $(\psi_{i_T(T),\mu_h},i_T(\epsilon))$, we only need to consider condition (3) of Def. \ref{defn:admissible_pair} at $p$. Let us choose some unramified conjugate $\xi_p':=\Int(u)\circ\xi_p$ of $\xi_p=\phi(p)\circ\zeta_p:\fG_p\rightarrow \fG_{i_T(T)}(p)$ under $i_T(T)(\Qpb)$, and let $b'\in T(\mfk)$ be defined by $b'\sigma=\theta^{\nr}(s_{\sigma})$ when $\xi_p'$ is the inflation of a $\Qpnr/\Qp$-Galois gerb morphism $\theta^{\nr}$. Then, since we have $[b']=[b]$ in $B(i_T(T))$ and $i_p|_{T(\Qp)}=i_T|_{T(\Qp)}$ (Thm. \ref{thm:Kisin17_Cor.1.4.13,Prop.2.1.3,Cor.2.2.5}, (3)), it is immediate that the equation $\Int u(i_T(\epsilon))x\rtimes\sigma^n(=i_T(\epsilon)x\rtimes\sigma^n)=(b'\rtimes\sigma)^nx$ has a solution in $G(\mfk)$ (i.e. $(\phi,i_T(\epsilon))$ is LR-admissible) if and only if $i_p(\epsilon)x\rtimes\sigma^n=(b\rtimes\sigma)^nx$ has one (i.e. $(\sI,\epsilon)$ is K-admissible). In this case, if we choose $d\in i_T(T)(\mfk)$ with $b=db'\sigma(d^{-1})$, from $c(i_p(\epsilon)^{-1}(b\rtimes\sigma)^n)c^{-1}=\sigma^n$, we obtain $cd(i_T(\epsilon)^{-1}(b'\rtimes\sigma)^n)(cd)^{-1}=\sigma^n$. Since $(cd)b'\sigma(cd)^{-1}=cb\sigma(c^{-1})=:\delta$ and $i_T(T)(\bar{\A}_f^p)\cap X^p(\psi_{i_T(T),\mu_h})\neq\emptyset$, it follows from definition (cf. \autoref{subsubsec:K-triple_attached_to_adm.pair}) that the triple $(i_T(\epsilon);i_T(\epsilon),\delta)$ is a Kottwitz triple attached to the LR-admissible LR-pair $(\psi_{i_T(T),\mu_h},i_T(\epsilon))$. It is easy to see that this is a stable Kottwitz triple with trivial Kottwitz invariant that was constructed in Prop. \ref{prop:Kottwitz_triple} (Remark \ref{rem:two_different_b's}).

Next, we show that the stable equivalence class of $(i_T(\epsilon);i_T(\epsilon),\delta)$ does not depend on the choices of $i_T$ and $T$. It suffices to show this independence for its (geometric) equivalence class, because for stable Kottwitz triples $(\gamma_0;\gamma,\delta)$, the $G(\Qb)$-conjugacy relation for $\gamma_0$ is the same as the stable conjugacy relation (Prop. \ref{prop:triviality_in_comp_gp}). Since the $G(\A_f^p)\times G(\mfk)$-conjugacy class of $(i^p,(i_p,b\sigma))$ is completely determined by the isogeny class $\sI$, the $G(\Ql)$-conjugacy class of $\gamma$ and the $\sigma$-conjugacy class in $G(L_n)$ of $\delta$ depends only on the K-pair $(\sI,\epsilon)$. By contrast, a priori the stable conjugacy class of $\gamma_0$ depends on the choice of $T$ and $i_T$. But, as $\gamma_0=i_p(\epsilon)$, we find that the $G(\Qb)$-conjugacy class of $\gamma_0$ is also determined only by the K-pair $(\sI,\epsilon)$.

(2) This follows from (1), because the inclusion $T\subset G$ lies in the distinguished stable conjugacy class of embeddings $T\hookrightarrow G$ attached to $j_{T,h}:T\hookrightarrow I_{\sI}$ as in Thm. \ref{thm:Kisin17_Cor.1.4.13,Prop.2.1.3,Cor.2.2.5}, (3) (property (iv) of Thm.  \ref{thm:Kisin17_Cor.1.4.13,Prop.2.1.3,Cor.2.2.5}, (4)). Note that a (resp. stable) Kottwitz triple attached to an LR-admissible pair $(\phi,\epsilon)$ that is well-located (resp. nested) in a maximal $\Q$-torus $T$ of $G$ remains (resp. stably) equivalent under a transfer of maximal torus $\Int(g):T\hookrightarrow G\ (g\in G(\Qb))$.
\end{proof}

Next, we proceed with our task of describing the fixed point set $S_{K^p}(\sI)^{\Phi^m\circ f=\mathrm{Id}}$ towards the ultimate goal of expressing its cardinality in terms of the triple $(\gamma_0;\delta,\gamma)$ and $g$. First, we need an analogue of Prop. \ref{prop:canonical_decomp_of_epsilon}.

%%%%%%%%%%%%%%%%%%%%
\begin{prop} \label{prop:canonical_decomp_of_epsilon2} 
Assume that $\Q\operatorname{-}\mathrm{rk}(Z(G))=\R\operatorname{-}\mathrm{rk}(Z(G))$.
For any K-admissible K-pair $(\sI,\epsilon)$, there exist $s\in\N$ and $\pi_0, t\in T_{\epsilon}(\Q)$,
where $T_{\epsilon}$ is the subgroup (of multiplicative type) of $I_{\sI}$ generated by $\epsilon\in G(\Q)$, satisfying the properties of Lemma \ref{lem:canonical_decomp_of_epsilon}: we have 
\[(a)\ \epsilon^s=\pi_0 t\ ;\qquad (b)\ \pi_0\in K_l \text{ for all }l\neq p\ ;\qquad (c)\ t\in K_p,\] 
for the maximal compact subgroup $K_v$ of $T_{\epsilon}^{\mathrm{o}}(\Q_v)$ (for each finite place $v$). Also, the K-pair $(\sI,\pi_0^k)$ is K-admissible for every $k\gg1$.
The pair $(\pi_0,t)$ is uniquely determined by $\epsilon$, up to taking simultaneous powers. \end{prop}

\begin{proof}
Any $\Q$-torus of $I_{\sI}$ has the same ranks over $\Q$ and $\R$ by Thm. \ref{thm:Kisin17_Cor.1.4.13,Prop.2.1.3,Cor.2.2.5} and our assumption on $Z(G)$. Under that condition,
the construction of $\pi_0, t\in T_{\epsilon}(\Q)$ satisfying the properties (a) - (c) of Lemma \ref{lem:canonical_decomp_of_epsilon} uses only the $\Q$-group structure of the $\Q$-group $T_{\epsilon}$ (of multiplicative type) and $\epsilon$. 
To see the second statement, we choose a maximal $\Q$-torus $T$ of $I_{\sI}$ containing $\epsilon$, and the data $i_T:T\hookrightarrow G$, $h\in X\cap \Hom(\dS,i_T(T)_{\R})$, $(b\in i_T(T)(\mfk),i_p,i^p)$ in Thm. \ref{thm:Kisin17_Cor.1.4.13,Prop.2.1.3,Cor.2.2.5}.
Then that $(\sI,\pi_0^k)$ is K-admissible for all $k\gg1$ follows from LR-admissibility of $(\psi_{i_T(T),\mu_h},i_T(\pi_0)^k)$ which in turn is implied by that of $(\psi_{i_T(T),\mu_h},i_T(\epsilon))$  (Lemma \ref{lem:key_observations}, Prop. \ref{prop:canonical_decomp_of_epsilon}).
\end{proof}

For a group $A$, we consider the following equivalence relation $\sim$ among the set of pairs $(a,n)$ with $a\in A$, $n\in\N$: $(a,n)\sim (a',n')$ if there exists $N\in\N$ such that $a^{n'N}=(a')^{nN}$. We define a \emph{germ of an element} of $A$ to be an equivalence class of pairs $(a,n)$ for this equivalence relation. 
%\begin{enumerate} \item[(i)] \end{enumerate}
For an algebraic group $G$ over a field $k$ (say, of characteristic zero) and a germ $\pi$ of an element of $G(k)$, if $(\pi_n,n)$ is a representative of $\pi$, the Zariski closures in $G$ of the subgroups generated by $\pi_n^e\ (e\in\N)$ form a decreasing sequence (with $e$'s being ordered multiplicatively) of subgroups, thus stabilizes to a $k$-subgroup. It is then easy to see that this $k$-group depends only on the given germ $\pi$, not on the choice of representative $(\pi_n,n)$; we call this subgroup of $G$ the \emph{subgroup generated by} the germ $\pi$. Also we call the centralizer in $G$ of this subgroup the \emph{centralizer} (in $G$) of the germ $\pi$ and denote it by $G_{\pi}$.

%%%%%%%%%%%%%%%%%%%%
%%%%%%%%%%%%%%%%%%%%
\begin{thm} \label{thm:Kisin17_Cor.2.3.2} \cite[Cor.2.3.2]{Kisin17}
Assume that $\Q\operatorname{-}\mathrm{rk}(Z(G))=\R\operatorname{-}\mathrm{rk}(Z(G))$. Let $\sI\subset\sS_{\mbfK_p}(\Fpb)$ be an isogeny class.

There exists a unique germ $\pi_{\sI}$ of an element in $Z(I_{\sI})(\Q)$, called the \emph{germ of Frobenius endomorphism} of $\sI$, with the following properties:

(a) There exist a representative $i_p:(I_{\sI})_{\Qp}\hookrightarrow J_b$ of the $G(\mfk)$-conjugacy class of Thm. \ref{thm:Kisin17_Cor.1.4.13,Prop.2.1.3,Cor.2.2.5} and a representative $(\pi_N,N)$ of $\pi=\pi_{\sI}$ such that one has $b\in G(L_N)$ and $i_p(\pi_N)=\Nm_Nb$;

(b) The embeddings $i_l:(I_{\sI})_{\Ql} \hookrightarrow G_{\Ql}\ (l\neq p)$, $i_p:(I_{\sI})_{\Qp}\hookrightarrow J_b$ of Thm. \ref{thm:Kisin17_Cor.1.4.13,Prop.2.1.3,Cor.2.2.5} induce isomorphisms of group schemes
\[ i_l:(I_{\sI})_{\Ql} \isom I_{i_l(\pi)},\quad i_p:(I_{\sI})_{\Qp}\isom I_{i_p(\pi)},\] 
where $I_{i_l(\pi)}$ and $I_{i_p(\pi)}$ denote respectively the centralizer in $G_{\Ql}$ of the germ $i_l(\pi)$ and the centralizer in $J_b$ of the germ $i_p(\pi)$. Also for almost all $l\neq p$, the $\Z_l$-embedding $i_l:(I_{\sI})_{\Z_l} \hookrightarrow G_{\Z_l}$ of Thm. \ref{thm:Kisin17_Cor.1.4.13,Prop.2.1.3,Cor.2.2.5} induces an isomorphism  $(I_{\sI})_{\Z_l} \isom (G_{\Z_l})_{i_l(\pi)}$.

(c) For any representative $(\pi_n,n)$ of $\pi_{\sI}$, the K-pair $(\sI,\pi_n^k)$ is K-admissible for all $k\gg1$ and $\pi_n$ lies in a compact subgroup of $I_{\sI}(\Ql)$ for every finite place $l\neq p$.
For any K-admissible K-pair $(\sI,\epsilon)$ of level $n$ and each triple $(\pi_0,t\in T_{\epsilon}(\Q);s\in\N)$ attached to $\epsilon$ as in Prop. \ref{prop:canonical_decomp_of_epsilon2}, $(\pi_0,ns)$ represents the germ $\pi_{\sI}$.
\end{thm}

\begin{proof}
There is a natural candidate for $\pi_n\ (k\gg1)$, i.e. the $p^n$-th relative Frobenius endomorphism in $Z(I_{\sI})(\Q)$ of the isogeny class $\sI$. The uniqueness (as a germ of element in $Z(I_{\sI})(\Q)$) will follow from (c).

Property (b) is Thm. \ref{thm:Kisin17_Cor.2.3.2;Tate_isom}, i.e. Kisin's generalization of the Tate's theorem on the endomorphisms of abelian varieties over finite fields.
We choose a point $x$ in $\sI$ defined over a finite field $\F_q$ and via some $K_0:=W(\F_q)[1/p]$-linear isomorphism $V\otimes K_0 \isom H^{\cris}_1(\uvA_x/K_0)$ matching the tensors $s_{\alpha}$, $s_{\alpha,0,x}$ on both sides, we identify the Frobenius automorphism on $H_{\cris}^1(\uvA_x/K_0)$ with $\delta\sigma$ for $\delta\in G(K_0)$. From this choice, we obtain a datum $(b=\delta,i_p:(I_{\phi})_{\Qp}\hookrightarrow J_b)$ in Thm. \ref{thm:Kisin17_Cor.1.4.13,Prop.2.1.3,Cor.2.2.5}; note that in this case, we have 
\begin{equation} \label{eq:(delta,i_p)}
i_p(\pi_N)=\Nm_N\delta
\end{equation}
for any $N\in N$ with $\F_q\subset \F_{p^N}$; this establishes (a).
Now, property (b) for $l\neq p$ is immediate from Thm. \ref{thm:Kisin17_Cor.2.3.2;Tate_isom}. In the $p$-adic case, \textit{loc. cit.} says that for any $N\gg1$, $i_p$ identifies $(I_{\sI})_{\Qp}$ with the $\sigma$-centralizer $G_{\delta\sigma}(\subset \Res_{L_N/\Qp}(G))$ inside $J_{\delta}$. But, since $\Nm_N\delta\rtimes\sigma^N=(\delta\rtimes\sigma)^N$, for any $\Qp$-algebra $R$, an element $g\in G(\mfk\otimes_{\Qp} R)$ commutes simultaneously with $\delta\sigma$ and $\sigma^N$ if and only if it does so with $\delta\sigma$ and $\Nm_N\delta=i_p(\pi_N)$, thus one has
\begin{align*}
G_{\delta\sigma}(R) &=\{ g\in G(L_N\otimes_{\Qp} R)\ |\ \delta\sigma(g)=g\delta \} \\
&=\{g\in G(\mfk\otimes_{\Qp} R)\ |\ \sigma^Ng=g,\ \delta\sigma(g)=g\delta \} \\
&=J_{\delta,i_p(\pi_N)}(R),
\end{align*}
as was asserted.
The second statement on the extension of $i_l$ over $\Z_l$ is clear, since for almost all $l\neq p$, $i_l(\pi_N)\in G_{\Z_l}(\Z_l)$ for all $N\gg1$ according to (c) below.

Next, for property (c), the K-admissibility of $(\sI,\pi_N)$ (for $N\gg1$) follows from (a), since for the choice of $(b,i_p)$ there, one has $i_p(\pi_N)^{-1}(b\rtimes\sigma)^N=\sigma^N$. Also, being a relative Frobenius endomorphism,  for every $l\neq p$, the eigenvalues of $\pi_N$ are all $l$-adic units, which implies that the subgroup of $Z(I_{\sI})(\Ql)$ generated by $\pi_N$ is bounded (thus, its closure is compact). The last claim of (c) can be established by the argument of the proof of Prop. \ref{prop:phi(delta)=gamma_0_up_to_center}. Take a maximal $\Q$-torus $T$ of $I_{\sI}$ containing $\epsilon$ (which automatically also contains the germ $\pi$), and let $i_T:T\hookrightarrow G$, $h\in X\cap\Hom(\dS,i_T(T)_{\R})$ be as in Thm. \ref{thm:Kisin17_Cor.1.4.13,Prop.2.1.3,Cor.2.2.5}. 
By Lemma \ref{lem:key_observations} and Prop. \ref{prop:canonical_decomp_of_epsilon2} , we know that the LR-pair $(\phi:=\psi_{i_T(T),\mu_h},i_T(\pi_0^k))$ is LR-admissible of level $nsk$ for all $k\gg1$, hence we have $\pi_0^{k}=\psi_{i_T(T),h}(\delta_{nsk})$ (Prop. \ref{prop:phi(delta)=gamma_0_up_to_center}). 
If $A/\F_{p^{nsk}}$ is the abelian variety over a finite field $\F_{p^{nsk}}$ for $k\gg1$ corresponding to the CM point $[h,1]\in Sh_{\mbfK_p}(G,X)(\Qb)$, it follows from the theory of complex multiplication (cf. \cite[A.2.5.7, A.2.5.8]{CCO14}) that $\psi_{i_T(T),h}(\delta_{nsk})$ is the relative Frobenius of $A/\F_{p^{nsk}}$, hence equals $\pi_{nsk}$.
\end{proof}

%%%%%%%%%%%%%%%%%%%%
\begin{lem} \label{lem:Tate_thm2}
Assume that $\Q\operatorname{-}\mathrm{rk}(Z(G))=\R\operatorname{-}\mathrm{rk}(Z(G))$.
Let $(\sI,\epsilon)$ be a K-admissible K-pair. We fix a maximal $\Q$-torus $T$ of $I_{\sI}$ containing $\epsilon$, and also $(i_T,h)$, $(b,i_p,i^p)$ as in Thm. \ref{thm:Kisin17_Cor.1.4.13,Prop.2.1.3,Cor.2.2.5}. Let $(\gamma_0;\gamma,\delta)$ be an associated Kottwitz triple (\ref{eq:K-triple_for_isogeny_adm.pair}).

(1) The embedding $i^p:(I_{\sI})_{\A_f^p} \hookrightarrow G_{\A_f^p}$ induces an isomorphism of $\A_f^p$-group scheme
\[i^p:(I_{\sI,\epsilon})_{\A_f^p} \isom G_{\gamma},\] 
where $G_{\gamma}$ is the centralizer of $\gamma$ in $G_{\A_f^p}$. 

(2) If the couple $(b,i_p)$ of Thm. \ref{thm:Kisin17_Cor.1.4.13,Prop.2.1.3,Cor.2.2.5} is chosen such that $b\in i_T(T)(\mfk)$, and we fix $c\in G(\mfk)$ (\ref{eq:(epsilon,b,c)->delta2}), the embedding $\Int(c)\circ i_p:(I_{\sI})_{\Qp}\hookrightarrow J_b\isom J_{\delta}$ induces an isomorphism of $\Qp$-groups
\begin{equation} \label{eq:Int(cu)2}
\Int(c)\circ i_p:\ (I_{\sI,\epsilon})_{\Qp}\isom G_{\delta\sigma},
\end{equation}
where $G_{\delta\sigma}(\subset \Res_{L_n/\Qp}(G))$ is the $\sigma$-centralizer of $\delta\in G(L_n)$. 

(3) There exists an inner twisting $(I_{\sI,\epsilon}^{\mathrm{o}})_{\Qb}\isom (G_{\gamma_0}^{\mathrm{o}})_{\Qb}$ that restricts to $i_T:T\isom i_T(T)$.
It also induces an inner twisting $(I_{\sI,\epsilon})_{\Qb}\isom (G_{\gamma_0})_{\Qb}$.
\end{lem}

As in the case of LR-admissible LR-pairs, the isomorphism class of $I_{\sI,\epsilon}^{\mathrm{o}}$ as an inner form of $G_{\gamma_0}^{\mathrm{o}}$ is not uniquely determined by the K-pair $(\sI,\epsilon)$ or the associated Kottwitz triple (unless $G^{\der}=G^{\uc}$).

\begin{proof}
(1) By Thm. \ref{thm:Kisin17_Cor.2.3.2} and Prop. \ref{prop:canonical_decomp_of_epsilon2} (plus Lemma \ref{lem:Zariski_group_closure}), $i_l$ induces isomorphisms (for any $k\gg1$)
\[(I_{\sI,\epsilon})_{\Ql} \isom Z_{G_{\Ql}}(i_l(\pi_{nsk}),i_l(\epsilon)) \isom Z_{G_{\Ql}}(i_l(\pi_0^k),i_l(\epsilon)) \isom Z_{G_{\Ql}}(i_l(\epsilon)),\] 
where $(\pi_{nsk},nsk)$ is a representative of the germ $\pi_{\sI}$ and $(s\in\N,\pi_0\in T_{\epsilon}(\Q))$ are as in Prop. \ref{prop:canonical_decomp_of_epsilon2}  (the groups in the middle denote the simultaneous centralizers of the elements inside the round brackets).

(2) Again, as in the case $l\neq p$, it follows from Thm. \ref{thm:Kisin17_Cor.2.3.2} and Prop. \ref{prop:canonical_decomp_of_epsilon2} (plus Lemma \ref{lem:Zariski_group_closure}) that $i_p$ induces isomorphisms (for any $k\gg1$)
\begin{equation} \label{eq:i_{p,epsilon}}
(I_{\sI,\epsilon})_{\Qp} \isom Z_{J_b}(i_p(\pi_{nsk}),i_p(\epsilon)) \isom Z_{J_b}(i_p(\pi_0^k),i_p(\epsilon)) \isom Z_{J_b}(i_p(\epsilon)).
\end{equation}
Then, as one has $c(i_p(\epsilon)^{-1}(b\sigma)^n)c^{-1}=\sigma^n$ (\ref{eq:(epsilon,b,c)->delta2}), $\Int(c)$ induces an isomorphism 
\[ \Int(c): Z_{J_b}(i_p(\epsilon)) \isom G_{\delta\sigma}\]
(cf. proof of Lemma \ref{lem:isom_Int(cu)}).

(3) An inner class of a connected reductive $\Q$-group is determined by the canonical Galois action on the Dynkin diagram which is given by a homomorphism $\Gal(\Qb/\Q)\rightarrow \Aut(X_{\ast}(T))/N_G(T)$ acting on a based root datum $(X^{\ast}(T),\Delta,X_{\ast}(T),\Delta^{\vee})$ for any maximal $\Q$-torus $T$ ($\Delta\subset X^{\ast}(T)$ being a set of simple roots, as usual). Hence, using that the $\Q$-groups $I_{\sI,\epsilon}^{\mathrm{o}}$, $G_{\gamma_0}^{\mathrm{o}}$ share the same maximal $\Q$-torus $T\isom i_T(T)$, it suffices to show that the two homomorphisms $\Gal(\Qb/\Q)\rightarrow \Aut(X_{\ast}(T))/N_G(T)$ canonically attached to them are the same. By Chebotarev density theorem, it is enough to check this locally for places in a set of Dirichlet density $1$. For a place $v$, the equality of the restrictions of the homomorphisms to $\Gal(\Qvb/\Qv)$ follows from the existence of $T$-equivariant inner-twistings
\[(I_{\sI,\epsilon}^{\mathrm{o}})_{\Qvb}\isom (G_{\gamma_0}^{\mathrm{o}})_{\Qvb}.\]
For finite $v$, such inner-twisting is provided by (1) and (2): for $l\neq p$, it suffices that $\gamma_0$ and $\gamma_l$ are stably conjugate. For $p$, one uses the fact that there exists a $T$-equivariant inner-twisting $(G_{\delta\sigma})_{K_0}\isom (G_{K_0})_{\gamma_p}$, where $K_0=W(k)[1/p]$ and $\gamma_p=\Nm_{n}\delta\in i_T(T)(K_0)$ (cf. \cite[Lem.5.4]{Kottwitz82}). 
\end{proof}

Next, we prove an analogue of Thm. \ref{thm:LR-Satz5.25}. For that, adapting the idea of Kisin of twisting an isogeny class $\sI$ by cohomology classes in $\Sha^{\infty}_G(\Q,I_{\sI})$, we also twist arbitrary K-admissible K-pair $(\sI,\epsilon\in I_{\sI}(\Q))$ by certain cohomology classes in 
\[ \im[\Sha^{\infty}_G(\Q,I_{\sI,\epsilon}^{\mathrm{o}})\rightarrow H^1(\Q,I_{\sI,\epsilon})].\]
Here, the groups $\Sha^{\infty}_G(\Q,I_{\sI})$, $\Sha^{\infty}_G(\Q,I_{\sI,\epsilon}^{\mathrm{o}})$ are defined as follows. Recall $\pi_{\sI}$, the germ of Frobenius endomorphism attached to $\sI$ (Lemma \ref{thm:Kisin17_Cor.2.3.2}). Using any maximal $\Q$-torus $T\subset I_{\sI}$ and a fixed choice of embedding $i_T:T\hookrightarrow G$ as in Thm. \ref{thm:Kisin17_Cor.1.4.13,Prop.2.1.3,Cor.2.2.5}, we obtain a germ $i_T(\pi_{\sI})$ of element in $G(\Q)$. Let $I_0\subset G$ be its centralizer. Then, the fact (Thm. \ref{thm:Kisin17_Cor.2.3.2}) that $I_{\sI}$ is an inner-form of $I_0$ allows us to define (as in  (\ref{eq:Sha^{infty}_G}))
\[ \Sha^{\infty}_G(\Q,I_{\sI}):=\ker\left[\Sha^{\infty}(\Q,I_{\sI}) \isom \Sha^{\infty}(\Q,I_0) \rightarrow \Sha^{\infty}(\Q,G)\right], \]
where the first isomorphism is induced by \emph{any} inner-twisting $(I_{\sI})_{\Qb}\isom (I_0)_{\Qb}$ and the second map by the inclusion $I_0\subset G$: this kernel is independent of all the choices made, especially of $T$, (cf. \cite[(4.4.7)]{Kisin17}). The group $\Sha^{\infty}_G(\Q,I_{\sI,\epsilon}^{\mathrm{o}})$ is defined similarly.
As we will see, such twisting procedure corresponds to twisting an admissible morphism $\phi:\fP\rightarrow \fG_G$ (resp. an LR-admissible LR-pair $(\phi,\epsilon)$) by cohomology classes in $\Sha^{\infty}_G(\Q,I_{\phi})$ (resp. classes in $\Sha_G(\Q,I_{\phi,\epsilon})^+$) as explained in Lemma \ref{lem:LR-Lem5.26,Satz5.25}.

Now, suppose given an isogeny class $\sI$ and $T\subset I_{\sI}$ a maximal $\Q$-torus. We fix a $\Q$-embedding $i_T:T\hookrightarrow G$ and $h\in X\cap \Hom(\dS,i_T(T)_{\R})$ as in Thm. \ref{thm:Kisin17_Cor.1.4.13,Prop.2.1.3,Cor.2.2.5}.
Let $\tilde{\beta}\in \Sha^{\infty}(\Q,T)$ and assume that the image of $i_T(\tilde{\beta})\in \Sha^{\infty}(\Q,i_T(T))$ in $H^1(\Q,G)$ is trivial. Thus, there exists $\tilde{\omega}\in G(\Qb)$ such that the cochain $\tau\mapsto \tilde{\omega}^{-1}\cdot{}^{\tau}\tilde{\omega}$ on $\Gal(\Qb/\Q)$ belongs to $Z^1(\Q,i_T(T))$ and as such one has 
\begin{equation} \label{eq:tilde{omega}}
[\tilde{\omega}^{-1}\cdot{}^{\tau}\tilde{\omega}]=i_T(\tilde{\beta}) 
\end{equation}
in $H^1(\Q,i_T(T))$. Equivalently, $\Int(\tilde{\omega}):G_{\Qb}\isom G_{\Qb}$ gives a transfer of maximal torus $i_T(T)\hookrightarrow G$ whose base-change $\Int(\tilde{\omega})_{\R}:i_T(T)_{\R}\hookrightarrow G_{\R}$ is induced from conjugation by an element in $G(\R)$.
Therefore, for such $\tilde{\omega}$, the pair 
\begin{equation} \label{eq:stable-conj._of_special_SD}
(i_T(T)^{\tilde{\beta}}, h^{\tilde{\beta}}):=(\Int(\tilde{\omega})(i_T(T)),\Int(\tilde{\omega})(h))
\end{equation}
is another special Shimura sub-datum of $(G,X)$, which is easily seen to depend only on 
$(i_T(T),h)$ and the cohomology class $\tilde{\beta}\in \Sha^{\infty}(\Q,T)$.

%%%%%%%%%%%%%%%%%%%%
\begin{prop} \label{prop:Kisin17_Prop.4.4.8} 
(1) \cite[Prop.4.4.8]{Kisin17} Let $\sI$ be an isogeny class and $\beta_1,\beta_2\in \Sha^{\infty}_G(\Q,I_{\sI})$. 
For each $i=1,2$, choose a maximal $\Q$-torus $T_i\subset I_{\sI}$ such that $\beta_i\in H^1(\Q,I_{\sI})$ is the image of some $\tilde{\beta}_i\in H^1(\Q,T_i)$ (which exists by \cite[Thm.5.11]{Borovoi98}), and fix $i_{T_i}:T_i\hookrightarrow G$, $h_i\in X\cap \Hom(\dS,i_{T_i}(T_i)_{\R})$, and $\tilde{\omega}_i\in G(\Qb)$ as above. 
Then, the isogeny classes $\sI_i\ (i=1,2)$ of the reductions of $[\Int(\tilde{\omega}_i)(h_i),1]$ are the same if and only if $\beta_1=\beta_2$.

In particular, for any $\beta\in \Sha^{\infty}_G(\Q,I_{\sI})$, the corresponding isogeny class depends only on $\beta$ (not on the choices of auxiliary data $T$, $\tilde{\beta}$, $i_{T}$, $h$, $\tilde{\omega}$), and when we denote it by $\sI^{\beta}$, the assignment $\beta\mapsto \sI^{\beta}$ defines an inclusion of $\Sha^{\infty}_G(\Q,I_{\sI})$ into the set of isogeny classes in $\sS_{\mbfK_p}(G,X)(\Fpb)$.

(2) Let $(\sI,\epsilon)$ be a K-admissible K-pair. 
Then, for any $\beta\in \im[\Sha^{\infty}_G(\Q,I_{\sI,\epsilon}^{\mathrm{o}})\rightarrow H^1(\Q,I_{\sI,\epsilon})]$, 
there exists a K-pair $(\sI^{\beta},\epsilon^{\beta})$ with the following properties: 
\begin{itemize} \addtolength{\itemsep}{-6pt}
\item[(i)] as an isogeny class, $\sI^{\beta}$ is the twist of $\sI$ by the image of $\beta$ in $\Sha^{\infty}_G(\Q,I_{\sI})$; 
\item[(ii)] for any maximal $\Q$-torus $T$ of $I_{\sI,\epsilon}^{\mathrm{o}}$ such that $\beta$ is the image of some $\tilde{\beta}\in H^1(\Q,T)$, 
the $I_{\sI^{\beta}}(\Q)$-conjugacy class of $\epsilon^{\beta}$ contains
\[\Int(\tilde{\omega})(i_T(\epsilon))\ \in\ i_T(T)^{\tilde{\beta}}(\Q),\]
where $i_T:T\hookrightarrow G$ and $\tilde{\omega}$ are as before and $\Int(\tilde{\omega})(i_T(\epsilon))$ is regarded as an element of $I_{\sI^{\beta}}(\Q)$ via any embedding $j_{i_T(T)^{\tilde{\beta}},h^{\tilde{\beta}}}:i_T(T)^{\tilde{\beta}}\hookrightarrow I_{\sI^{\beta}}$ in Thm.  \ref{thm:Kisin17_Cor.1.4.13,Prop.2.1.3,Cor.2.2.5}, (4).
\end{itemize}

Moreover, the assignment $\beta\mapsto (\sI^{\beta},\epsilon^{\beta})$ gives a well-defined inclusion of $\im[\Sha^{\infty}_G(\Q,I_{\sI,\epsilon}^{\mathrm{o}})\rightarrow H^1(\Q,I_{\sI,\epsilon})]$ into the set of equivalence classes of K-pairs.
\end{prop}

\begin{proof}
(1) As indicated in the statement, this is \cite[Prop.4.4.8]{Kisin17}. 

(2) We use a moduli interpretation of $\sI^{\beta}$ (which was also the key ingredient of the proof of (1)). 
For any $\beta\in \im[\Sha^{\infty}_G(\Q,I_{\sI,\epsilon}^{\mathrm{o}})\rightarrow H^1(\Q,I_{\sI,\epsilon})]$,
we can find a maximal $\Q$-torus $T$ of $I_{\sI,\epsilon}^{\mathrm{o}}$ such that $\beta$ is the image of some $\tilde{\beta}\in H^1(\Q,T)$ \cite[Thm.5.11]{Borovoi98}, which then must lie in $\Sha^{\infty}_G(\Q,T)$ \cite[4.4.5]{Kisin17}. We fix $i_T:T\hookrightarrow G$, $\tilde{\omega}$ as in (1).
If $x\in \sI$ denotes the reduction of $[h,1]$, the reduction of $[\Int(\tilde{\omega})(h),1]$ is $\mathbf{i}_{\omega}(x)$ in the notation of (4.4.7) of \cite{Kisin17}. Its underlying abelian variety $\mathcal{A}_{\mathbf{i}_{\omega}(x)}$ is isomorphic to the twist $\mathcal{A}_{x}^{\mathcal{P}}$ of $\mathcal{A}_{x}$ by the $T$-torsor $\mathcal{P}$ corresponding to $\tilde{\beta}\in H^1(\Q,T)$ (\textit{loc. cit.} 4.1.6), and there exists a natural $\Qb$-isogeny between the corresponding weakly polarized abelian varieties endowed with a set of (crystalline and \'etale) cycles (i.e. $s_{\alpha,0,x}$, $s_{\alpha,l,x}\ (l\neq p)$ in the notation of the proof of Thm. \ref{thm:Kisin17_Cor.1.4.13,Prop.2.1.3,Cor.2.2.5}) 
\begin{equation} \label{eq:Qb-isogeny_theta}
\theta_{\tilde{\omega}}:\mathcal{A}_{x}^{\mathcal{P}}\otimes_{\Q}\Qb\isom \mathcal{A}_{x}\otimes_{\Q}\Qb,
\end{equation}
which is unique up to $T(\Qb)$-conjugacy, once we fix an identification $I_{\sI}(\Q)= I_x(\Q)$: in more detail, such $\Qb$-isogeny is induced by a point of $\mathcal{P}(\Qb)$ (i.e. a homomorphism $\mathcal{O}_{\mathcal{P}}\rightarrow \Qb$) from a universal isomorphism (in the isogeny category, cf. \textit{loc. cit.} 4.1.6)
\begin{equation} \label{eq:univ_isogeny_theta}
\mathcal{A}_x^{\mathcal{P}}\otimes_{\Q}\mathcal{O}_{\mathcal{P}}\isom \mathcal{A}_{x}\otimes_{\Q}\mathcal{O}_{\mathcal{P}}. 
\end{equation}
Two points $p_1,p_2\in\Hom(\mathcal{O}_{\mathcal{P}},\Qb)$ differ by a point $t\in T(\Qb)$ acting on $\mathcal{O}_{\mathcal{P}}$, and it is easy to see that then the resulting $\Qb$-isogenies $\mathcal{A}_{x}^{\mathcal{P}}\otimes_{\Q}\Qb\isom \mathcal{A}_{x}\otimes_{\Q}\Qb$ differ by composition with $t\in I_{x}(\Q)$. 

Now, when $\epsilon\in T(\Q)$, it is obvious from the definition (\textit{loc. cit.} 4.1.6) that $\epsilon$ gives an automorphism of the group functor represented by $\mathcal{A}_x^{\mathcal{P}}$, namely a self-isogeny
\[ \epsilon^{\mathcal{P}} \in \mathrm{Aut}_{\Q}(\mathcal{A}_x^{\mathcal{P}}) \]  
(as an element of $\mathrm{Aut}_{\Q}(\mathcal{A}_x^{\mathcal{P}})$, this still depends on $\mathcal{P}$, i.e. on $\tilde{\beta}$) and in fact lies in $I_{\mathbf{i}_{\omega}(x)}(\Q)$ (\textit{loc. cit.} 4.1.6, and Lemma 4.1.5, 4.1.7).
Also, the universal isomorphism (\ref{eq:univ_isogeny_theta}) takes $\epsilon^{\mathcal{P}}$ to $\epsilon$, so for any $\Qb$-isogeny $\theta_{\tilde{\omega}}$ (\ref{eq:Qb-isogeny_theta}) one has \[\epsilon^{\mathcal{P}} =\Int(\theta_{\tilde{\omega}}^{-1})(\epsilon).\]

Let $I_{\mathbf{i}_{\omega}(x),\epsilon}:=Z_{I_{\mathbf{i}_{\omega}(x)}}(\epsilon)$. 
We consider the $\Q$-scheme $\mathcal{Q}_{\epsilon}$ of isomorphisms 
\[\mathcal{A}_x^{\mathcal{P}}\lisom \mathcal{A}_{x}, \]
compatible with the weak polarizations and which fix the crystalline and the etale tensors $\{s_{\alpha,0,x}\}$, $\{s_{\alpha,l,x}\}\ (l\neq p)$ and further takes $\epsilon^{\mathcal{P}}$ to $\epsilon$. 
One readily sees (\textit{loc. cit.}, proof of Prop. 4.4.8) that this is a $I_{\sI,\epsilon}$-torsor which admits a $T$-equivariant map $\mathcal{P}\rightarrow\mathcal{Q}_{\epsilon}$, hence is isomorphic to the $I_{\sI,\epsilon}$-torsor associated with $\beta\in H^1(\Q,I_{\sI,\epsilon})$. 
This implies that the assignment $\beta\mapsto (\sI^{\mathcal{P}},\epsilon^{\mathcal{P}})$ gives a well-defined inclusion of $\im[\Sha^{\infty}_G(\Q,I_{\sI,\epsilon}^{\mathrm{o}})\rightarrow H^1(\Q,I_{\sI,\epsilon})]$ into the set of equivalence classes of K-pairs, where $\mathcal{P}$ is the $T$-torsor corresponding to any choice of $\tilde{\beta}\in H^1(\Q,T)$ mapping to $\beta$ (for a maximal $\Q$-torus $T$ of $I_{x,\epsilon}$).
In particular, the $I_{\mathbf{i}_{\omega}(x)}(\Q)$-conjugacy class of $\epsilon^{\mathcal{P}}$ depends only on $\beta$; we let $\epsilon^{\beta}$ denote any representative of the associated $I_{\sI^{\beta}}(\Q)$-conjugacy class (via some identification $I_{\sI^{\beta}}=I_{\mathbf{i}_{\omega}(x)}$) and write $(\sI^{\beta},\epsilon^{\beta})$ for $(\sI^{\mathcal{P}},\epsilon^{\mathcal{P}})$.

Next, it remains to show the equality (up to $I_{\sI^{\beta}}(\Q)$-conjugacy):
 \[j^{\tilde{\beta}}(i_T(\epsilon)^{\tilde{\beta}}) = \epsilon^{\beta},\] 
where $i_T(\epsilon)^{\tilde{\beta}}:=\Int(\tilde{\omega})(i_T(\epsilon))$ and $j^{\tilde{\beta}}:=j_{i_T(T)^{\tilde{\beta}},h^{\tilde{\beta}}}:i_T(T)^{\tilde{\beta}}\hookrightarrow I_{\sI^{\beta}}$ is any member in its conjugacy class of such embeddings in Thm. \ref{thm:Kisin17_Cor.1.4.13,Prop.2.1.3,Cor.2.2.5}, (4). Recall that we have fixed identifications $I_{\sI}=I_x$, $I_{\sI^{\beta}}=I_{\mathbf{i}_{\omega}(x)}$.
Let us fix $\tilde{\omega}\in G(\Qb)$ such that $[\tilde{\omega}^{-1}\cdot{}^{\tau}\tilde{\omega}]=\tilde{\beta}$ in $H^1(\Q,T)$. This also fixes a canonical $\Qb$-isogeny 
\[\theta_{\tilde{\omega}}:\uvA_{\sigma_p(\tilde{x})}^{\mathcal{P}}\otimes_{\Q}\Qb\isom \uvA_{\sigma_p(\tilde{x})}\otimes_{\Q}\Qb\]
which preserves the Betti tensors and is compatible with the weak polarizations; the twist $\uvA_{\sigma_p(\tilde{x})}^{\mathcal{P}}$ is the underlying abelian variety of the special point $[h^{\tilde{\beta}},1]$ ($h^{\tilde{\beta}}=\Int(\tilde{\omega})(h)$.
By reduction, this gives the $\Qb$-isogeny (\ref{eq:Qb-isogeny_theta}) (we use the same notation). The isogeny induces an inner-twisting
\[\Int(\theta_{\tilde{\omega}}^{-1}):\Aut_{\Q}(\uvA_{\sigma_p(\tilde{x})})_{\Qb} \isom \Aut_{\Q}(\uvA_{\sigma_p(\tilde{x})}^{\mathcal{P}})_{\Qb},\quad (I_{x})_{\Qb}\isom (I_{\mathbf{i}_{\omega}(x)})_{\Qb}\]
whose corresponding cocycle $(\theta_{\tilde{\omega}}\cdot {}^{\tau}\theta_{\tilde{\omega}}^{-1})_{\tau}\in Z^1(\Q, \Aut_{\Q}(\uvA_{\sigma_p(\tilde{x})}))$ equals (as cochains) $\tilde{\omega}^{-1}\cdot {}^{\tau}\tilde{\omega}\in Z^1(\Q,i_T(T))=Z^1(\Q,T)$ (\textit{loc. cit.} Lemma 4.1.2). 
Hence, by restriction, $\Int(\theta_{\tilde{\omega}}^{-1})$ induces a $\Q$-isomorphism $T\isom T^{\tilde{\beta}}$ for some maximal $\Q$-torus $T^{\tilde{\beta}}\subset  I_{\mathbf{i}_{\omega}(x)}$. 
The Hodge structure on $H^{\Betti}_1(\uvA_{\sigma_p(\tilde{x})}^{\mathcal{P}},\Q)$ is given by $h^{\tilde{\beta}}=\Int(\tilde{\omega})(h)$ via the isomorphism 
\[ H^{\Betti}_1(\uvA_{\sigma_p(\tilde{x})}^{\mathcal{P}},\Q) \isom H^{\Betti}_1(\uvA_{\sigma_p(\tilde{x})},\Q)^{\mathcal{P}} \stackrel{\eta_{\Betti}^{\mathcal{P}}}{\isom} V^{\mathcal{P}} \stackrel{\tilde{\omega}}{\isom} V, \]
where  the first isomorphism is the canonical one induced by $\theta_{\tilde{\omega}}$ (\textit{loc. cit.} Lemma 4.1.7), and the third one is the multiplication by $\tilde{\omega}$ which identifies $V^{\mathcal{P}}:=(V\otimes\mathcal{O}_{\mathcal{P}})^T\subset V_{\Qb}$ with $V$ (cf. \textit{loc. cit.} proof of Prop. 4.2.6). Also, $\eta_{\Betti}:H^{\Betti}_1(\uvA_{\sigma_p(\tilde{x})},\Q)\isom V$ is an isomorphism as in (\ref{eq:Betti-isom}) which, together with $\tilde{x}=[h,1]$, defines the embedding $i_T:T\hookrightarrow G$.
Therefore, there exists an isomorphism $i_{T^{\tilde{\beta}}}:T^{\tilde{\beta}}\isom i_T(T)^{\tilde{\beta}}$
making a commutative diagram
\begin{equation} \label{eq:new_embedding_T-beta}
\xymatrix{ i_T(T) \ar[r]^{\Int(\tilde{\omega})} & i_T(T)^{\tilde{\beta}} \\ T \ar[u]^{i_T} \ar[r]_{\Int(\theta_{\tilde{\omega}}^{-1})} & T^{\tilde{\beta}} \ar@{->}[u]_{i_{T^{\tilde{\beta}}}} } 
\end{equation}

By definition, $\mathbf{i}_{\omega}(x)$ is the reduction of the special point $[h^{\tilde{\beta}},1]$. 
So, we see that the $\Q$-embedding $i_{T^{\tilde{\beta}}}$ lies in the stable conjugacy class attached, by Thm. \ref{thm:Kisin17_Cor.1.4.13,Prop.2.1.3,Cor.2.2.5}, to the maximal $\Q$-torus $T^{\tilde{\beta}}\subset I_{\mathbf{i}_{\omega}(x)}$. 
Now, as $\epsilon^{\beta}=\Int(\theta_{\tilde{\omega}}^{-1})(\epsilon)$, we have
\[i_T(\epsilon)^{\tilde{\beta}}=i_{T^{\tilde{\beta}}}(\epsilon^{\beta}).\]
Then, since the $I_{\sI^{\beta}}(\Q)$-isogeny class of $j_{i_T(T)^{\tilde{\beta}},h^{\tilde{\beta}}}\circ i_{T^{\tilde{\beta}}}:T^{\tilde{\beta}}\hookrightarrow I_{\mathbf{i}_{\omega}(x)}$ contains the inclusion $T^{\tilde{\beta}}\subset I_{\mathbf{i}_{\omega}(x)}$ (Thm. \ref{thm:Kisin17_Cor.1.4.13,Prop.2.1.3,Cor.2.2.5}, (4)), the claim follows. 
\end{proof}

%%%%%%%%%%%%%%%%%%%%
%%%%%%%%%%%%%%%%%%%%
\begin{thm} \label{thm:LR-Satz5.25b2} 
Keep the previous notation. Let $(\gamma_0;\gamma,\delta)$ be a stable Kottwitz triple with trivial Kottwitz invariant. If $\mathrm{O}_{\gamma}(f^p)\cdot \mathrm{TO}_{\delta}(\phi_p)\neq 0$, then there exists a K-admissible pair $(\sI,\epsilon)$ giving rise to $(\gamma_0;\gamma,\delta)$, in which case the number of equivalence classes of such admissible pairs equals the cardinality of the set
\[ \Sha_G(\Q,I_{\sI,\epsilon})^+:=\im[\Sha^{\infty}_G(\Q,I_{\sI,\epsilon}^{\mathrm{o}})\rightarrow H^1(\Q,I_{\sI,\epsilon})]\cap \ker^1(\Q,I_{\sI,\epsilon}).\]
\end{thm} 

We remind the readers again that ``having trivial Kottwitz invariant'' means that there exist elements $(g_v)_v\in G(\bar{\A}_f^p)\times G(\mfk)$ satisfying conditions (\ref{eq:stable_g_l}), (\ref{eq:stable_g_l}) such that the associated Kottwitz invariant $\alpha(\gamma_0;\gamma,\delta;(g_v)_v)$ vanishes. This is our second version of effectivity criterion of Kottwitz triple.

\begin{proof}
In view of Lemma \ref{lem:key_observations}, the effectivity statement is a consequence of the effectivity statement for LR-admissible LR-pairs (Thm. \ref{thm:LR-Satz5.25}) and Lemma \ref{lem:LR-Lemma5.23}. More precisely, there exists an LR-admissible LR-pair $(\phi,\epsilon)$ giving rise to given (stable) Kottwitz triple $(\gamma_0;\gamma,\delta)$ (Thm. \ref{thm:LR-Satz5.25}). Then, this LR-pair is conjugate to a special one $(\psi_{T,\mu_h},\epsilon\in T(\Q))$ for a special Shimura sub-datum $(T,h)$ (Lemma \ref{lem:LR-Lemma5.23}), which gives rise to a K-admissible K-pair $(\sI,j_{T,h}(\epsilon))$, where $\sI$ is the isogeny class of the reduction of the special point $[h,1]\in Sh_K(G,X)(\Qb)$ and $j_{T,h}:T\hookrightarrow I_{\sI}$ is any embedding as given in Thm. \ref{thm:Kisin17_Cor.1.4.13,Prop.2.1.3,Cor.2.2.5}, (4). The conclusion follows by Lemma \ref{lem:key_observations}.

Next, we prove the statement on the number of equivalence classes of K-admissible K-pairs producing a given Kottwitz triple. Let $(\sI,\epsilon)$ be a K-admissible K-pair and $\beta\in \im[\Sha^{\infty}_G(\Q,I_{\sI,\epsilon}^{\mathrm{o}})\rightarrow H^1(\Q,I_{\sI,\epsilon})]$. We can find a maximal $\Q$-torus $T\subset I_{\sI,\epsilon}^{\mathrm{o}}$ and $\tilde{\beta}\in\Sha^{\infty}_G(\Q,T)$ that maps to $\beta$ \cite[Thm.5.11]{Borovoi98}, \cite[4.4.5]{Kisin17}. Fix $i_{T}:T\hookrightarrow G$, $h\in X\cap \Hom(\dS,i_{T}(T)_{\R})$, and $\tilde{\omega}\in G(\Qb)$ as in Prop. \ref{prop:Kisin17_Prop.4.4.8} (cf. (\ref{eq:tilde{omega}})).
Let $T_1:=T^{\tilde{\beta}}\subset I_{\sI^{\beta}}$ and $i_{T_1}:T_1\hookrightarrow G$ be as in (\ref{eq:new_embedding_T-beta}); one has 
\[\Int(\tilde{\omega})(i_T(T),i_T(\epsilon))=(i_{T_1}(T_1),i_{T_1}(\epsilon^{\beta}))\] for some representative $\epsilon^{\beta}\in I_{\sI^{\beta}}(\Q)$ in its $I_{\sI^{\beta}}(\Q)$-conjugacy class in Prop. \ref{prop:Kisin17_Prop.4.4.8}. Also, $h_1:=h^{\tilde{\beta}}=\Int(\tilde{\omega})(h)\in X$, and we have two special LR-pairs
\begin{equation} \label{eq:twisting_LR-pairs}
(\phi:=\psi_{i_T(T),\mu_h},i_T(\epsilon)),\quad (\phi_1:=\psi_{i_{T_1}(T_1),\mu_{h_1}},i_{T_1}(\epsilon^{\beta})).
\end{equation}
By Lemma \ref{lem:key_observations}, (1), LR-admissibility of the LR-pair $(\phi,i_T(\epsilon))$ follows from K-admissibility of the K-pair $(\sI,\epsilon)$. 

By conjugating the LR-pair $(\phi_1,i_{T_1}(\epsilon^{\beta}))$ back by $\Int(\tilde{\omega}^{-1})$, we obtain another LR-pair 
\[(\phi':=\Int(\tilde{\omega}^{-1})\circ\phi_1,i_T(\epsilon))\] 
which is also well-located in $i_T(T)$ and shares the same Frobenius descent element $i_T(\epsilon)$ as the original LR-pair $(\phi,i_T(\epsilon))$. One has 
\[\phi'(q_{\rho})=\tilde{\omega}^{-1}\cdot{}^{\rho}\tilde{\omega}\cdot \phi(q_{\rho})\] 
for every $\rho\in\Gal(\Qb/\Q)$ ($\rho\mapsto q_{\rho}$ is the chosen section to $\fP\twoheadrightarrow \Gal(\Qb/\Q)$). Indeed, construction of the morphism $\psi_{T,\mu}:\fP\rightarrow \fG_T$ is functorial in the pairs $(T,\mu)$ \cite[Satz.2.3]{LR87}, thus
if $\phi'(q_{\rho})=g_{\rho}'\rtimes\rho$ and $\phi(q_{\rho})=g_{\rho}\rtimes\rho$ with $g_{\rho}',g_{\rho}\in G(\Qb)$, we have $g_{\rho}'=\Int(\tilde{\omega})(g_{\rho})$ since $\Int(\tilde{\omega}):T\isom T'$ is $\Q$-rational and sends $\mu$ to $\mu'$, which implies the claim.
Therefore we see that the LR-pair $(\phi',i_{T}(\epsilon))$ is obtained from the original LR-pair $(\phi,i_T(\epsilon))$ by twisting with $i_T(\tilde{\beta}) \in \im[\Sha^{\infty}_G(\Q,T)\stackrel{i_T}{\rightarrow} H^1(\Q,I_{\phi,i_T(\epsilon)})]$ (cf. Lemma \ref{lem:LR-Lem5.26,Satz5.25}). In particular, the special LR-pair $(\phi_1,i_{T_1}(\epsilon^{\beta}))$ is LR-admissible if the image of $\beta$ in $H^1(\Qp,I_{\phi,i_T(\epsilon)})$ vanishes. 

Now, suppose further that $\beta\in \ker^1(\Q,I_{\sI,\epsilon})$; then, we also have $\beta\in \ker^1(\Q,I_{\phi,i_T(\epsilon)})$, since one already has $\beta \in \im[\Sha^{\infty}_G(\Q,T)\stackrel{i_T}{\rightarrow} \Sha^{\infty}_G(\Q,I_{\phi,i_T(\epsilon)})]$ and there exists a $T$-equivariant $\Q$-isomorphism $I_{\sI,\epsilon}\isom I_{\phi,i_T(\epsilon)}$ (proof of \cite[Prop.4.4.13]{Kisin17}).
Since the special LR-pair $(\phi_1,i_{T_1}(\epsilon^{\beta}))$ is LR-admissible, by Lemma \ref{lem:key_observations}, (2), the K-pair $(\sI^{\beta},\epsilon^{\beta})$ is also K-admissible and one has equivalences of Kottwitz triples:
\begin{equation} \label{eq:twistng_K-triples}
 \mfk(\sI^{\beta},\epsilon^{\beta}) \sim \mfk(\phi_1,i_{T_1}(\epsilon^{\beta})) \sim \mfk(\phi',i_T(\epsilon)),\quad   \mfk(\sI,\epsilon) \sim \mfk(\phi,i_T(\epsilon)).
 \end{equation}
Therefore, by Lemma \ref{lem:LR-Lem5.26,Satz5.25}, (3), one has $\mfk(\sI^{\beta},\epsilon^{\beta})\sim \mfk(\sI,\epsilon)$, and by Prop. \ref{prop:Kisin17_Prop.4.4.8}, we obtain an inclusion 
\[ \Sha_G(\Q,I_{\sI,\epsilon})^+\hookrightarrow \{\text{ K-admissible K-pairs }\}/\sim \ :\  \beta\mapsto (\sI^{\beta},\epsilon^{\beta})\]
such that the K-admissible K-pairs in the image have equivalent associated Kottwitz triples (with trivial Kottwitz invariant) as $(\sI,\epsilon)$, where 
\begin{align*}
\Sha_G(\Q,I_{\sI,\epsilon})^+&:=\im[\Sha^{\infty}_G(\Q,I_{\sI,\epsilon}^{\mathrm{o}})\rightarrow H^1(\Q,I_{\sI,\epsilon})]\cap \ker^1(\Q,I_{\sI,\epsilon}).
\end{align*}

It remains to show that the image of this inclusion exhausts all the K-admissible K-pairs whose associated Kottwitz triples are equivalent to that of $(\sI,\epsilon)$.
Let $(\sI',\epsilon')$ be such a K-admissible K-pair whose associated Kottwitz triple $(\gamma_0';\gamma',\delta')$ is conjugate to $(\gamma_0;\gamma,\delta)$. 
By construction (cf. Lemma \ref{lem:key_observations}), we may assume that the Kottwitz triple $(\gamma_0;\gamma,\delta)$ is of the form (\ref{eq:K-triple_for_isogeny_adm.pair}) defined by choice of a maximal torus $T\subset I_{\sI}$ containing $\epsilon$, and an accompanying choice of $\Q$-embedding $i_{T}:T\hookrightarrow G$ as in Thm. \ref{thm:Kisin17_Cor.1.4.13,Prop.2.1.3,Cor.2.2.5}: one has $\gamma_0=i_{T}(\epsilon)$, $\gamma=\gamma_0$, and the Kottwitz triple is stable, among others. Recall  (Thm. \ref{thm:Kisin17_Cor.1.4.13,Prop.2.1.3,Cor.2.2.5}) that $i_T$ is defined by a certain cocharacter $\mu_T$ of $T$ which produces a special point $\tilde{x}$ lifting some point $x\in \sI$ corresponding to $h\in X\cap \Hom(\dS,i_{T}(T)_{\R})$ such that $\sigma_p(i_T\circ\mu_T)=\mu_{h}$; we identify $I_{\sI}=I_{x}$.
Moreover, we may take $b\in i_T(T)(\mfk)$ to be defined by $b\sigma:=\theta^{\nr}(s_{\sigma})$, where $\theta^{\nr}:\fD\rightarrow \fG_{i_T(T)_{\Qp}}^{\nr}$ is a $\Qpnr/\Qp$-Galois gerb morphism whose inflation $\overline{\theta}^{\nr}$ is an unramified $i_T(T)(\Qpb)$-conjugate of $\psi_{i_T(T),\mu_h}(p)\circ\zeta_p$. The Kottwitz triple $(\gamma_0';\gamma',\delta')$ is similarly defined.
Let $I_0=G_{\gamma_0}^{\mathrm{o}}$ and $I_0'=G_{\gamma_0'}^{\mathrm{o}}$ be the connected centralizers of $\gamma_0$ and $\gamma_0'$ respectively.

We will construct inner twistings \[ \varphi: (I_{\sI,\epsilon}^{\mathrm{o}})_{\Qb} \isom (I_0)_{\Qb},\quad \varphi': (I_{\sI',\epsilon'}^{\mathrm{o}})_{\Qb} \isom (I_0')_{\Qb} \] 
and a $\Q$-embedding 
\[T\hookrightarrow I_{\sI',\epsilon'}^{\mathrm{o}}\] 
such that $\varphi|_{T}=i_T$ and $i_T':=\varphi'|_{T_{\Qb}}:T_{\Qb}\hookrightarrow (I_0')_{\Qb}$ is $\Q$-rational, and further \[i_T'=\Int(g)\circ i_T\] for some $g\in G(\Qb)$. We proceed in three steps.

In the first step, by Lemma \ref{lem:Tate_thm2}, the data $(T,i_T,b\in i_T(T)(\mfk))$ determine an inner twisting 
\[ \varphi: I_{\sI,\epsilon}^{\mathrm{o}}\otimes_{\Q}\Qb\isom I_0\otimes_{\Q}\Qb\] 
that is $T$-equivariant with respect to $T\subset I_{\sI,\epsilon}^{\mathrm{o}}$ and $i_T:T\hookrightarrow I_0$.

In the second step, we construct a $\Q$-embedding $T\hookrightarrow I_{\sI',\epsilon'}^{\mathrm{o}}$ with certain local properties. We choose a point $x'\in \sI'$ and an identification $I_{\sI'}=I_{x'}$. 
By construction, there exist a maximal $\Q$-torus $T'$ of $I_{x'}$ containing $\epsilon'$ and a $\Q$-embedding $i_{T'}:T'\hookrightarrow G$ giving rise to $(\gamma_0';\gamma',\delta')$ by (\ref{eq:K-triple_for_isogeny_adm.pair}).
The equivalence of the Kottwiz triples attached to $(\sI,\epsilon)$, $(\sI',\epsilon')$ implies the existence, for every place $v$ of $\Q$, of ``natural'' local isomorphisms 
\begin{equation*} \label{eq:rho_v}
\rho_v:(I_{x,\epsilon}^{\mathrm{o}})_{\Qv}\isom (I_{x',\epsilon'}^{\mathrm{o}})_{\Qv}
\end{equation*}
taking $\epsilon$ to $\epsilon'$, where $I_{x,\epsilon}$ and $I_{x',\epsilon'}$ are as usual the centralizers of $\epsilon$ and $\epsilon'$ respectively. Indeed, let $\pi_0\in T_{\epsilon}(\Q)(\subset I_{x}(\Q))$ and $\pi_0'\in T_{\epsilon'}(\Q)(\subset I_{x'}(\Q))$ be the elements as in Lemma \ref{prop:canonical_decomp_of_epsilon2}.
Then, for every finite place $l\neq p$, we have $i_{T'}(\pi_0')=\Int(h_l)(i_T(\pi_0))$ for every $h_l\in G(\Ql)$ such that $\gamma_l'=\Int(h_l)(\gamma_l)$: use functoriality of the construction of $\pi_0$, $\pi_0'$ (Lemma \ref{lem:canonical_decomp_of_epsilon}) and Lemma \ref{lem:Zariski_group_closure}. Hence, there exists an isomorphism of $\Ql$-vector spaces endowed with $k$-Frobenius automorphism action ($[k:\F_p]\gg1$) and Frobenius-invariant tensors
\begin{equation} \label{eq:x->x'_at_l}
(H^{\et}_1(\uvA_{\bar{x}},\Ql),\Fr_{\uvA_x/k},\{s_{\alpha,l,x}\}) \lisom (H^{\et}_1(\uvA_{\bar{x}'},\Ql),\Fr_{\uvA_{x'}/k},\{s_{\alpha,l,x'}\})
\end{equation}
taking $\epsilon\in I_{x}$ to $\epsilon'\in I_{x'}$, by existence of a $i_T$(or $i_{T'}$)-equivariant isomorphism (\ref{eq:isom_eta}) for $x$ and $x'$. From this we obtain a desired $\Qv$-isomorphism $\rho_l$ by \cite[Cor.2.3.2]{Kisin17} (i.e. Thm. \ref{thm:Kisin17_Cor.2.3.2;Tate_isom} here).
For an analogous statement at $p$, recall that $\delta=cb\sigma(c^{-1})\in G(L_n)$ for some $c\in G(\mfk)$ such that $c^{-1}i_T(\epsilon)^{-1}(b\sigma)^nc=\sigma^n$ and $\delta'=c'b'\sigma(c'^{-1})$ for a similar $c'$.
Then, if $\delta'=h_p\delta\sigma(h_p^{-1})$ for $h_p\in G(L_n)$, we have
$i_{T'}(\epsilon')=c'\Nm_n\delta'c'^{-1}=\Int(c'h_p)(\Nm_n\delta)=\Int(h_p')(i_T(\epsilon))$ and $b'=h_p'b\sigma(h_p'^{-1})$ for $h_p':=c'^{-1}h_pc\in G(\mfk)$.
Hence, there exists an isomorphism of isocrystals over $\Fpb$ endowed with Frobenius-invariant tensors
\begin{equation} \label{eq:x->x'_at_p}
(H^{\cris}_1(\uvA_{x}/\mfk),\phi,\{s_{\alpha,0,x}\}) \lisom (H^{\cris}_1(\uvA_{x'}/\mfk),\phi',\{s_{\alpha,0,x'}\})
\end{equation}
taking $\epsilon$ to $\epsilon'$, by existence of a $i_T$(or $i_{T'}$)-equivariant isomorphism (\ref{eq:isom_eta_nr}) for $x$ and $x'$, and again we obtain a desired $\Qp$-isomorphism $\rho_p$ by the isomorphism $(I_{x,\epsilon})_{\Qp} \isom Z_{J_b}(i_p(\epsilon))$ (\ref{eq:i_{p,epsilon}}).
At infinity, the two groups $(I_{x,\epsilon}^{\mathrm{o}})_{\R}$, $(I_{x',\epsilon'}^{\mathrm{o}})_{\R}$ are both the unique (as inner classes) inner forms of $I_0$ with compact adjoint group which gives a required $\rho_{\infty}$.
The existence of these $\rho_v$'s implies (by Chebotarev density theorem) that the canonical action of $\Gal(\Qb/\Q)$ on the Dynkin diagrams of $I_{x,\epsilon}^{\mathrm{o}}$, $I_{x',\epsilon'}^{\mathrm{o}}$ are the same, i.e. there exists an inner-twist
\[\rho:(I_{x,\epsilon}^{\mathrm{o}})_{\Qb}\isom (I_{x',\epsilon'}^{\mathrm{o}})_{\Qb}\]
whose base-change $\rho_{\Qvb}$ is conjugate to $(\rho_v)_{\Qvb}$ for every $v$.
We have $\rho(\epsilon)=\epsilon'$ as $\rho_v$ has the same property.%%
\footnote{In fact, using the argument of Lemma \ref{lem:uniqueness_of_inner-class_with_same_K-triples} one can show that the two $\Q$-groups $I_{x,\epsilon}^{\mathrm{o}}$, $I_{x',\epsilon'}^{\mathrm{o}}$ are isomorphic as inner-twists of $I_0$. But it is not clear whether such an isomorphism is conjugate over $\Qvb$ to $\rho_v$ as here which is induced by an isomorphism of vectors spaces endowed with Frobenius action and tensors.}
Next, we show the existence of $g'\in I_{x',\epsilon'}^{\mathrm{o}}(\Qb)$ such that the restriction of $\Int(g')\circ\rho$ to $T_{\Qb}$ induces a $\Q$-embedding $T\hra I_{x',\epsilon'}^{\mathrm{o}}$, in other words, that \emph{the maximal torus $T$ of $I_{x,\epsilon}$ transfers to $I_{x',\epsilon'}^{\mathrm{o}}$ with respect to the $I_{x',\epsilon'}^{\mathrm{o}}(\Qb)$-conjugacy class of the inner twisting $\rho$}. Since $T_{\R}$ is elliptic in $(I_{x,\epsilon})_{\R}$, according to \cite[Lemma 5.6]{LR87}, this follows from the condition that for each place $v$ of $\Q$, $T_{\Qv}$ transfers into $(I_{x',\epsilon'}^{\mathrm{o}})_{\Qv}$ with respect to the conjugacy class of $\rho_{\Qvb}$ (\cite[$\S$9]{Kottwitz84a}): in more detail, when there exist transfers locally everywhere, the obstruction to finding a global transfer of $T$ in $I_{x',\epsilon'}^{\mathrm{o}}$ lies in $\mathrm{ker}^2(\Q,T^{\uc})$ (locally trivial elements in $H^2(\Q,T^{\uc})$), where $T^{\uc}$ is the inverse image of $T$ under the natural map $(I_{x',\epsilon'}^{\mathrm{o}})^{\uc}\ra I_{x',\epsilon'}^{\mathrm{o}}$ \cite[9.5]{Kottwitz84a}.%%
\footnote{To apply this discussion (and \cite[Lemma 5.6]{LR87} as well), one does not need the hypothesis (as was made in the beginning of the same section $\S$9) that the inner twisting in question is an inner twisting of a \textit{quasi-split} group.}
On the other hand, $\mathrm{ker}^2(\Q,T^{\uc})$ vanishes if $T^{\uc}$ becomes anisotropic at one place (see the last part of the proof of Lemma 14.1 in \cite{Kottwitz92}). Therefore, we have shown the existence of a $\Q$-embedding  $T\hookrightarrow I_{x',\epsilon'}^{\mathrm{o}}$ such that for every place $v$ of $\Q$, its base-change to $\Qvb$ is induced from an isomorphism as in (\ref{eq:x->x'_at_l}), (\ref{eq:x->x'_at_p}), over $\Qlb$ and $\Qpb$.

In the final step, for the embedding $T\subset I_{\sI',\epsilon'}^{\mathrm{o}}$ just constructed, the \emph{same} cocharacter $\mu_h\in X_{\ast}(T)$ still satisfies the conditions of \cite[Lem.2.2.2]{Kisin17}, thus determines a special point $\tilde{x}'$ lifting some point in $\sI'$, denoted again by $x'$, thereby also an embedding 
\[ i_{T}':T\hookrightarrow G \]
(via a choice of an isomorphism (\ref{eq:Betti-isom})) and an element $b'\in i_T'(T)(\mfk)$ (cf. proof of Thm. \ref{thm:Kisin17_Cor.1.4.13,Prop.2.1.3,Cor.2.2.5}). The pair $(i_T',b')$ gives a (new) stable Kottwitz triple attached to $(\sI',\epsilon')$ (\ref{eq:K-triple_for_isogeny_adm.pair}), which is still stably equivalent to $(\gamma_0;\gamma,\delta)$, as they are (geometrically) equivalent (Prop. \ref{prop:triviality_in_comp_gp}); by abuse of notation, we continue to denote the new triple by $(\gamma_0';\gamma',\delta')$. Again, by Lemma \ref{lem:Tate_thm2}, the datum $(i_T',b'\in i_T'(T)(\mfk))$ gives rise to an inner twisting
\[\varphi':I_{\sI',\epsilon'}^{\mathrm{o}}\otimes_{\Q}\Qb\isom I_0'\otimes_{\Q}\Qb\] 
that is $T$-equivariant with respect to $T\hookrightarrow I_{\sI',\epsilon'}^{\mathrm{o}}$ and $i_T':T\hookrightarrow I_0'$. 

It remains to see that there exists $g\in G(\Qb)$ such that 
\[i_T'=\Int(g)\circ i_T.\] 
Indeed, by the property of the embedding $T\hookrightarrow I_{\sI',\epsilon'}^{\mathrm{o}}$, there exists a $\Qpb$-isomorphism
\[(H^{\cris}_1(\uvA_{x}/\Qpb),\{s_{\alpha,0,x}\},T) \lisom (H^{\cris}_1(\uvA_{x'}/\Qpb),\{s_{\alpha,0,x'}\},T)\] 
which matches the tensors and is compatible with $T$-actions and the Frobenius automorphisms (for $k$ large enough, especially such that $x$ and $x'$ are both defined over $k$).
This implies, via the existence of special points $\tilde{x}$, $\tilde{x}'$ lifting $x$, $x'$ (up to isogeny), the existence of a similar $\Qb$-morphism
\[(H^{\Betti}_1(\uvA_{\tilde{x}}/\Q),\{s_{\alpha,\tilde{x}}\},T) \lisom (H^{\Betti}_1(\uvA_{\tilde{x}'}/\Q),\{s_{\alpha,\tilde{x}'}\},T)\] 
which proves the claim.

Therefore, for every $\tau\in\Gal(\Qb/\Q)$, we have 
\[\tilde{\beta}_{\tau}:=i_T^{-1}(g^{-1}\tau(g))\in T(\Qb).\] 
We denote by $\tilde{\beta}\in H^1(\Q,T)$ its cohomology class.
When $h'\in X\cap\Hom(\dS,i_T'(T)_{\R})$ is the homomorphism defining $\tilde{x}'$ (i.e. $\sigma_p(i_T'\circ\mu_T)=\mu_{h'}$), it follows from $h'=\Int(g)(h)$ that the image of $i_T(\tilde{\beta})$ in $H^1(\R,K_{\infty})$ is trivial, where $K_{\infty}$ is the centralizer of $h$ (compact-modulo-center inner form of $G_{\R}$).
Therefore, $i_T(\tilde{\beta})$ belongs to the subgroup $\Sha^{\infty}_G(\Q,i_T(T))$ \cite[Lem. 4.4.5]{Kisin17}.
Moreover, we have $\sI'=\sI^{\beta'}$, where $\beta'$ denotes the image of $\tilde{\beta}$ in $\Sha^{\infty}_G(\Q,I_{\sI})$  (via $i_T$): if $\tilde{\beta}_{\iota}=g_{\infty}^{-1}\iota(g_{\infty})$ for $g_{\infty}\in K_{\infty}(\C)(=G(\C))$, $\sI^{\beta'}$ is the isogeny class of the reduction of the special point $[\Int(gg_{\infty}^{-1})(h),1]\in Sh_{\mbfK_p}(G,X)(\Qpb)$. Clearly, we have $(\sI',\epsilon')=(\sI^{\beta},\epsilon^{\beta})$, where $\beta$ is the image of $\tilde{\beta}$ in $\im[\Sha^{\infty}_G(\Q,I_{\sI,\epsilon}^{\mathrm{o}})\rightarrow H^1(\Q,I_{\sI,\epsilon})]$.

Now, the statement that the image $\beta$ of $\tilde{\beta}$ in $H^1(\Q,I_{\sI,\epsilon})$ vanishes in $\prod_{v\neq\infty} H^1(\Q_v,I_{\sI,\epsilon})$ follows from the similar statement for the LR-pairs 
(\ref{eq:twisting_LR-pairs}), i.e. Thm. \ref{thm:LR-Satz5.25}, in view of the relations (\ref{eq:twistng_K-triples}).
This completes the proof.
\end{proof}

%%%%%%%%%%%%%%%%%%%%
%%%%%%%%%%%%%%%%%%%%
\begin{thm} \label{thm:Kottwitz_formula:Kisin}
Let $(G,X)$ be a Shimura datum of Hodge type. Fix a hyperspecial subgroup $\mbfK_p$ and take $K^p$ to be sufficiently small such that conditions (a), (b) of (\ref{item:Langlands-conditions}) hold and $K\cap Z(G)(\Q)=\{1\}$.

(1) We have the following expression for (\ref{eq:fixed-pt_set_of_Frob-Hecke_corr}):
\[T(m,f)=\sum_{(\gamma_0;\gamma,\delta)} c(\gamma_0;\gamma,\delta)\cdot \mathrm{O}_{\gamma}(f^p)\cdot \mathrm{TO}_{\delta}(\phi_p)\cdot \mathrm{tr}\xi(\gamma_0),\]
with
\[ c(\gamma_0;\gamma,\delta):=i(\gamma_0;\gamma,\delta)\cdot |\pi_0(G_{\gamma_0})(\Q)|^{-1} \cdot \tau(I_0)\cdot \mathrm{vol}(A_G(\R)^{\mathrm{o}}\backslash I_0(\infty)(\R))^{-1} \] 
where $I_0:=G_{\gamma_0}^{\mathrm{o}}$, $i(\gamma_0;\gamma,\delta)=|\widetilde{\Sha}_G(\Q,I_{\phi,\epsilon})^+|$ (Lemma \ref{eq:|widetilde{Sha}_G(Q,I_{phi,epsilon})^+|}), $\tau(I_0)$ is the Tamagawa number of $I_0$, and $I_0(\infty)$ is the (unique) inner form of $(I_0)_{\R}$ having compact adjoint group. Also, the sum is over a set of representatives $(\gamma_0;\gamma,\delta)$ of all (stable) equivalence classes of \emph{stable} Kottwitz triples of level $n=m[\kappa(\wp):\Fp]$ having trivial Kottwitz invariant.

(2) Then, for any $f^p$ in the Hecke algebra $\mathcal{H}(G(\A_f^p)/\!\!/ K^p)$, there exists $m(f^p)\in\N$, depending on $f^p$, such that for each $m\geq m(f^p)$, we have
\begin{align} \label{eq:Lef-number1}
\sum_{i}(-1)^i\mathrm{tr}( & \Phi^m\times f^p | H^i_c(Sh_{K}(G,X)_{\Qb},\sF_K)) \\
& = \sum_{(\gamma_0;\gamma,\delta)} c(\gamma_0;\gamma,\delta)\cdot \mathrm{O}_{\gamma}(f^p)\cdot \mathrm{TO}_{\delta}(\phi_p) \cdot \mathrm{tr}\xi(\gamma_0), \nonumber
\end{align}
where the sum is over a set of representatives $(\gamma_0;\gamma,\delta)$ of \emph{all} (stable) equivalence classes of \emph{stable} Kottwitz triples of level $n$ having trivial Kottwitz invariant. If $G^{\ad}$ is anisotropic or $f^p$ is the identity, we can take $m(f^p)$ to be $1$ (irrespective of $f^p$).
\end{thm}

\begin{proof}
Given the results established in this subsection, the proof is the same as that of Thm. \ref{thm:Kottwitz_formula:LR}.
\end{proof}

The proof of this theorem also shows:

\begin{cor} \label{cor:geom_effectivity_of_K-triple} Under the same assumption as Thm. \ref{thm:Kottwitz_formula:Kisin}, a Kottwitz triple $(\gamma_0;\gamma,\delta)$ of level $n$ with trivial Kottwitz invariant is \emph{geometrically effective} in the sense that it arises from a $\F_{p^n}$-valued point of $\sS$ if and only if $\mathrm{TO}_{\delta}(\phi_p)$ is non-zero. \end{cor}

In particular, an $\R$-elliptic stable conjugacy class of $\gamma_0\in G(\Q)$ arises from an $\F_{p^n}$-valued point of $\sS$ for some $n\in\N$ if and only if there exists $\delta\in G(L_n)$ such that $\gamma_0$ is stably conjugate to $\Nm_n\delta$ and $\mathrm{TO}_{\delta}(\phi_p)\neq0$.

%%%%%%%%%%%%%%%%%%%%
%%%%%%%%%%%%%%%%%%%%
%%%%%%%%%%%%%%%%%%%%
\section{Stabilization}

In this section, we stabilize the right-hand side of the formula (\ref{eq:Lef-number1}), namely express it as a weighted sum, over the elliptic endoscopic data $\underline{H}$ of $G$, of the elliptic part of the geometric side of the stable trace formula for a suitable function on an endoscopic group $H_1(\A)$ for $\underline{H}$. We follow closely the arguments of Kottwitz \cite[$\S$4, $\S$7]{Kottwitz90}, \cite{Kottwitz10} who however worked out this process under the assumption that $G^{\der}=G^{\uc}$.
For the general case, we need to adapt some of his arguments, borrowing necessary ingredients from \cite{Labesse04}.

%%%%%%%%%%%%%%%%%%%%
\subsection{Endoscopic transfer of (twisted) orbital integrals}

Let $F$ be a local or global field of characteristic zero with a fixed $\bar{F}$, $\Gamma:=\Gal(\bar{F}/F)$, and $G$ a connected reductive group over $F$. For $L$-group, we will use the Weil form: ${}^LG=\hat{G}\rtimes W_F$.
We recall (\cite[(1.2)]{LS87}, \cite[2.1]{KottwitzShelstad99}, \cite[$\S$5]{Shelstad08}) that an \emph{endoscopic datum} of $G$ is a tuple $(H,\mathcal{H},s,\xi)$, where 
\begin{itemize}
\item[(i)] $H$ is a quasi-split connected reductive group over $F$, 
\item[(ii)] $\mathcal{H}$ is a split extension of $W_F$ by $\hat{H}$ (complex dual group of $H$) such that the action of $W_F$ on $\hat{H}$ given by any splitting coincides with an $L$-action of $W_F$ (defined by the $F$-structure of $H$) up to conjugation under $\hat{H}$,
\item[(iii)] $s$ is a semi-simple element of $\hat{G}$,
\item[(iv)] $\xi:\mathcal{H}\rightarrow {}^LG$ is an $L$-homomorphism 
such that (a) $\Int(s)\circ\xi=a\xi$ for a cocycle $a\in Z^1(W_F,Z(\hat{G}))$ that is trivial if $F$ is local, or is locally trivial if $F$ is global, and that (b) $\xi|_{\hat{H}}$ is an isomorphism of $\hat{H}$ wth the connected component $\hat{G}_s^{\mathrm{o}}$ of the centralizer of $s$ on $\hat{G}$.
\end{itemize}

An endoscopic datum $(H,\mathcal{H},s,\xi)$ is said to be \emph{elliptic} if $\xi(Z(\hat{H})^{\Gamma})^{\mathrm{o}}\subset Z(\hat{G})$.
An isomorphism from $(H,\mathcal{H},s,\xi)$ to $(H',\mathcal{H}',s',\xi')$ is a conjugation by an element $g\in \hat{G}$ such that $g\xi(\mathcal{H})g^{-1}=\xi'(\mathcal{H}')$ and $gsg^{-1}=s'$ modulo $Z(\hat{G})$; it then induces an $F$-isomorphism $H\isom H'$ dual to $\Int(g)^{-1}:\hat{H}'\isom \hat{H}$ (such isomorphism is uniquely determined from $\Int(g)^{-1}$ by requiring it to preserve some chosen $F$-splittings of $H$ and $H'$) \cite[p.16]{KottwitzShelstad99}.
We remind the readers that $\mathcal{H}$ is not necessarily an $L$-group (although one can attach an $L$-action on $\hat{H}$ by requiring it to fix some given splitting of $\hat{H}$). This weakness is compensated by the notion of a $z$-pair (\cite[2.2]{KottwitzShelstad99}, \cite[$\S$5]{Shelstad08}). A $z$-pair for an endoscopic datum $(H,\mathcal{H},s,\xi)$ is a pair $(H_1,\xi_1)$, where 
\begin{itemize}
\item[(v)] $H_1$ is a $z$-extension of $H$ \cite[$\S$1]{Kottwitz82}, i.e. an extension $1\rightarrow Z_1\rightarrow H_1\rightarrow H\rightarrow 1$, where $H_1$ is a connected reductive group over $F$ with $H_1^{\der}=H_1^{\uc}$ and $Z_1$ is an induced central torus,
\item[(vi)] $\xi_1$ is an embedding of extensions $\mathcal{H}\rightarrow {}^LH_1$ that extends the embedding $\hat{H}\rightarrow \hat{H}_1$
\end{itemize}

For an endoscopic datum $(H,\mathcal{H},s,\xi)$ and a $z$-pair $(H_1,\xi_1)$, let $\lambda_{H_1}$ be the (quasi-)character on $Z_1(\A_F)/Z_1(F)$ if $F$ is global, or on $Z_1(F)$ if $F$ is local, corresponding (via Langlands correspondence for tori%%
\footnote{Here, for the normalization for Langlands correspondence of tori, we use the Langlands's original convention which is opposite to that adopted in \cite[p.116]{KottwitzShelstad99}. See \cite[$\S$4.2]{KottwitzShelstad12} for discussion on this issue. \label{ftn:LLC_sign} }%%
) to the $L$-homomorphism $W_{F}\stackrel{c}{\rightarrow} \mathcal{H}\rightarrow {}^LH_1\rightarrow {}^LZ_1$, where $c$ is a splitting of $\mathcal{H}\rightarrow W_{F}$ as specified in the definition of endoscopic data (any two splittings define the same character).

For local $F$, let $C_{c,\lambda_{H_1}}^{\infty}(H_1(F))$ denote the space of complex-valued, smooth (i.e. $C^{\infty}$ if $F$ is archimedean, or locally constant if $F$ is nonarchimedean) functions $f^{H_1}$ on $H_1(F)$ whose supports are compact modulo $Z_1(F)$ and that satisfy $f^{H_1}(zh)=\lambda_{H_1}(z)^{-1} f^{H_1}(h)$ for all $z\in Z_1(F)$ and $h\in H_1(F)$.
For $\gamma_{H_1}\in H_1(F)$ and $f^{H_1}\in C_1^{\infty}(H_1(F))$, we define the \emph{stable orbital integral} of $f^{H_1}$ along the stable conjugacy class of $\gamma_{H_1}$ by
\[ \mathrm{SO}_{\gamma_{H_1}}(f^{H_1}):= \sum_{\gamma_{1}'} e(I_{\gamma_1'}) a(\gamma_1') \mathrm{O}_{\gamma_{1}'}(f^{H_1}), \]
Here, $\gamma_{1}'$ runs through a set of representatives for the $H_1(F)$-conjugacy classes of elements in $H_1(F)$ inside the stable conjugacy class of $\gamma_{H_1}$, $e(I_{\gamma_1'})$ is the sign attached to $I_{\gamma_1'}:=(H_1)_{\gamma_{1}'}^{\mathrm{o}}$ by Kottwitz \cite{Kottwitz83}, 
\begin{equation} \label{eq:a(gamma)}
 a(\gamma_1'):=|\ker[H^1(F,I_{\gamma_1'})\rightarrow H^1(F,H_{\gamma_1'})]|,
 \end{equation}
and $\mathrm{O}_{\gamma_{1}'}$ is the orbital integral $\int_{I_{\gamma_1'}(F)\backslash H_1(F)}f^{H_1}(\bar{x}^{-1}\gamma \bar{x}) d\bar{x}$ (\ref{eq:(twisted-)orbital_integral}) with suitable choices of Haar measures on $H_1(F)$, $I_{\gamma_1'}(F)$ being understood.

We fix an inner twist $\psi:G\isom G^{\ast}$ with quasi-split $G^{\ast}$. We recall (\cite{Langlands83} III.1, Diagram D, \cite[(1.3)]{LS87}) that an \emph{admissible} embedding of a maximal $F$-torus $T_H$ of $H$ to $G$ is a composite of two $F$-isomorphisms 
\begin{equation} \label{eq:admissible_embedding}
T_H\isom T_{\ast}\isom T,
\end{equation}
where $T_H\isom T_{\ast}$ is an $F$-embedding into $G^{\ast}$ of $G$ defined by a choice of Borel pairs $(T_H,B_H)$, $(T_{\ast},B_{\ast})$ via the associated isomorphism $\hat{T}_{\ast}\isom \hat{T}_H$ (which is also called \emph{admissible}) and the map $T\isom T_{\ast}$ is of the form $\Int(x)\circ \psi\ (x\in G^{\ast}(\bar{F}))$ (cf. \cite[(2.4)]{Shelstad82}); in this case, we say that $T$ \emph{comes from} $T_H$, and also say that a semi-simple element $\gamma_H$ of $H(F)$ \emph{comes from} or \emph{transfers to} $G(F)$ if $\gamma_H=j^{-1}(\gamma)$ for an admissible embedding $j:T_H\rightarrow T$ and some $\gamma\in G(F)$ with $\gamma\in T(F)$, in which case $\gamma_H$ is said to be a \emph{norm} (or simply an \emph{image}) of $\gamma$. A semi-simple element $\gamma_H$ of $H(F)$, if it transfers to $G(F)$, does to a unique stable conjugacy class.

We use the (global/local) Langlands-Shelstad transfer factors $\Delta$ \cite{LS87} (cf. \cite{KottwitzShelstad99}, \cite{KottwitzShelstad12}): it is a $\C$-valued map defined on the set $H(F)_{\mathrm{ss},(G,H)\operatorname{-}\mathrm{reg}}\times G(F)_{\mathrm{ss}}$ of pairs consisting of a $(G,H)$-regular (in the sense of \cite[3.1]{Kottwitz86}), semi-simple element of $H(F)$ and a semi-simple element of $G(F)$. The value $\Delta(\gamma_H,\gamma)$ depends only on the stable conjugacy class of $\gamma_H$ and the $G(F)$-conjugacy class of $\gamma$, and is zero unless $\gamma_H$ is $(G,H)$-regular and is a norm of $\gamma$.
The transfer factor is defined up to a nonzero constant: one needs to choose a reference pair $(\gamma_H',\gamma')\in H(F)\times G(F)$, where $\gamma_H'$ is a strongly $G$-regular element and is a norm of a strongly regular $\gamma'$, and what is canonically defined is the relative transfer factor $\Delta(\gamma_H,\gamma;\gamma_H',\gamma')$. Then, assigning the complex number $\Delta(\gamma_H',\gamma')$ arbitrarily, one sets
\[\Delta(\gamma_H,\gamma):=\Delta(\gamma_H',\gamma') \cdot \Delta(\gamma_H,\gamma;\gamma_H',\gamma').\]
For any $z$-pair $(H_1,\xi_1)$, the definition of $\Delta$ extends to $H_1(F)_{\mathrm{ss},(G,H_1)\operatorname{-}\mathrm{reg}}\times G(F)_{\mathrm{ss}}$ 
(by definition, an element $\gamma_{H_1}\in H_1(F)$ is $(G,H_1)$-regular if its image in $H(F)$ is $(G,R)$-regular).
As a matter of fact, for the definition of $\Delta$ for an endoscopic datum $(H,\mathcal{H},s,\xi)$ and a $z$-pair $(H_1,\xi_1)$, in this work we will use the one adopted by Kottwitz in \cite[p.178]{Kottwitz90} which, in the case $H_1=H$, $\mathcal{H}={}^LH$, is the same as the original definition of \cite{LS87} for the endoscopic datum $(H,\mathcal{H},s^{-1},\xi)$, and also is the one denoted by $\Delta'$ in \cite[(1.0.4), $\S$ 5.1]{KottwitzShelstad12}.

For two algebraic groups $H\subset G$ over a field $F$, we let $\mathfrak{D}(H,G;F)$ denote the set $\ker[H^1(F,H)\rightarrow H^1(F,G)]$. When $F$ is a number field and for $R=\A_F$, $\A_{F,f}$, we also use the notation $\mathfrak{D}(H,G;R):=\ker[H^1(R,H)\rightarrow H^1(R,G)]$.

Following \cite[$\S$7]{Kottwitz92}, we construct a function $f^{H_1}$ on $H_1(\A)$ with desired stable orbital integrals.
We assume that $H$ is unramified at $p$ and there exists an (elliptic) maximal torus of $H_{\R}$ that transfers to an elliptic maximal torus of $G_{\R}$; otherwise, we define $f^{H_1}$ to be $0$.
Under these conditions, $f^{H_1}$ will be a product of three functions $f^{H_1,p}$, $f^{H_1}_p$, $f^{H_1}_{\infty}$ on $H_1(\A_f^p)$, $H_1(\Qp)$, $H_1(\R)$, respectively, constructed now.

%%%%%%%%%%%%%%%%%%%%
\subsubsection{Untwisted endoscopy: $v\neq p,\infty$}

We state the transfer conjecture and the fundamental lemma, which were conjectured by Langlands-Shelstad \cite{LS87} and proved by Ngo \cite{Ngo10} after reduction steps of Waldspurger \cite{Waldspurger97}, \cite{Waldspurger06} (see also the references therein for related works).

%%%%%%%%%%%%%%%%%%%%
%%%%%%%%%%%%%%%%%%%%
\begin{thm} \label{thm:untwisted_endoscopy_transfer}
For every $f\in C_c^{\infty}(G(F))$, there exists an $f^{H_1}\in C_{c,\lambda_{H_1}}^{\infty}(H_1(F))$ such that for any $(G,H_1)$-regular, semi-simple element $\gamma_{H_1}$ of $H_1(F)$, the stable orbital integral $\mathrm{SO}^{H_1(F)}_{\gamma_{H_1}}(f^{H_1})$ is zero unless (the image in $H(F)$ of) $\gamma_{H_1}$ transfers to $G(F)$, in which case
\begin{equation} \label{eq:untwisted_endo-transfer1}
\mathrm{SO}^{H_1(F)}_{\gamma_{H_1}}(f^{H_1})= \sum_{\alpha\in \mathfrak{D}(I_0,G;F)} \ \langle \tilde{\alpha},s\rangle \Delta(\gamma_{H_1},\gamma_0) e(I_{\gamma_{\alpha}}) \mathrm{O}^{G(F)}_{\gamma_{\alpha}}(f),
\end{equation}
where we fix an element $\gamma_0$ of $G(F)$ whose norm is $\gamma_{H_1}$ and set $I_0:=G_{\gamma_0}^{\mathrm{o}}$, and for each $\alpha\in \mathfrak{D}(I_0,G;F):=\ker[H^1(F,I_0)\rightarrow H^1(F,G)]$, we choose an element $\gamma_{\alpha}$ of $G(F)$ whose $G(F)$-conjugacy class (in the stable conjugacy class of $\gamma_0$) corresponds to the image of $\alpha$ in $H^1(F,G_{\gamma_0})$ under (\ref{eq:C_l(gamma_0)}), and denote by $\tilde{\alpha}$ the lifting of $\alpha$ to $X^{\ast}(Z(\hat{I}_0)^{\Gamma}Z(\hat{G}))$ whose restriction to $Z(\hat{G})$ is trivial.  

Moreover, when $G$, $H$, $(H_1,\xi_1)$ are unramified in the sense of \cite[4.4]{Waldspurger08}, there exist normalizations of the transfer factor $\Delta$ and the measures used in the definition of the (stable) orbital integrals such that if $f$ is the characteristic function $\mathbf{1}_{K(\mathcal{O})}$ on a hyperspecial subgroup $K(\mathcal{O})$ of $G(F)$, then we may take $f^{H_1}$ to be the function $f_{K_1,\lambda_{H_1}}\in C_{c,\lambda_{H_1}}^{\infty}(H_1(F))$ defined by
\[ f_{K_1,\lambda_{H_1}}(x) := 
\begin{cases} \quad 0 & \text{ if }x\notin Z_1(F)K_1(\mathcal{O}) \\ \lambda_{H_1}(z)^{-1} & \text{ if } x=zk \text{ with }
z\in Z_1(F),\ k\in K_1(\mathcal{O}) \end{cases}, \]
where $K_1(\mathcal{O})$ is the hyperspecial subgroup of $H_1(F)$ attached to $K(\mathcal{O})$ as constructed in \cite[4.1-4.4]{Waldspurger08}.
\end{thm}

One says that the function $f^{H_1}$ is a \emph{transfer} of $f$ and the pair $(f,f^{H_1})$ has \emph{matching orbital integrals}.

%%%%%%%%%%%%%%%%%%%%
\begin{rem} \label{rem:untwisted_endoscopy_transfer}
(1) As the local conjecture for $\Delta$ \cite[5.6]{Kottwitz86} holds \cite[(4.2)]{LS87},
there exists a relation (cf. \cite[p.169, line -9]{Kottwitz90})
\[ \Delta(\gamma_{H_1},\gamma_{\alpha}) =\langle \tilde{\alpha},s\rangle \Delta(\gamma_{H_1},\gamma) \] 
Hence, the right-hand side of (\ref{eq:untwisted_endo-transfer1}) is also equal to
\[ \sum_{\gamma} \Delta(\gamma_{H_1},\gamma) e(I_{\gamma}) a(\gamma) \mathrm{O}^{G(F)}_{\gamma}(f),\]
where $\gamma$ runs through a set of representatives for the semi-simple $G(F)$-conjugacy classes in $G(F)$ (whose norms are $\gamma_{H_1}$): use the fact \cite[Prop.35bis]{Serre02} that for any $\alpha\in \mathfrak{D}(I_0,G;F)$, the set of $\beta$'s in $\mathfrak{D}(I_0,G;F)$ having the same image in $H^1(F,G_{\gamma_0})$ as $\alpha$ is in bijection with $\mathfrak{D}(I_{\gamma_{\alpha}},G_{\gamma_{\alpha}};F)$.

(2) A priori, this theorem is proved for strongly $G$-regular elements $\gamma_{H_1}$ of $H_1(F)$ (cf. \cite[4.8]{Waldspurger08}). Then, the identity as well as the definitions of transfer factors extend to general $(G,H_1)$-regular elements $\gamma_{H_1}$, by the argument of proof of \cite[Lemma 24.A]{LS90}: this lemma in fact does the same job for the case $H=H_1$ under the assumption $G^{\der}=G^{\uc}$, and is based on the special case of \cite[Prop.2]{Kottwitz88} that $H$ is a quasi-split inner form of $G$, but not necessarily $G^{\der}=G^{\uc}$. Also essentially the same argument is repeated in \cite[Prop. A.3.14]{Kottwitz10} for base change twisted endoscopic transfer, assuming $G^{\der}=G^{\uc}$ (one can combine the arguments of \cite[Prop.2]{Kottwitz88} and \cite[Prop. A.3.14]{Kottwitz10} to directly obtain a proof in our general set-up). 
\end{rem}

%%%%%%%%%%%%%%%%%%%%
\subsubsection{Base-change twisted endoscopy: $v=p$}

Here, we assume that $G$ and the given endoscopic datum $(H,\mathcal{H},s,\xi)$ are both unramified in the sense of \cite[4.4]{Waldspurger08}. This implies that $\mathcal{H}\isom{}^LH$ and $\xi$, as an $L$-homomorphism ${}^LH\rightarrow {}^LG$, is the identity on the inertia subgroup $I$ of $W_F$ (\textit{loc. cit.}, $\S$5.1). 

For each $n\in\N$, we let $L_n$ denote the unramified extension of $F$ in $\bar{F}$ with $[L_n:F]=n$, and $R:=\Res_{L_n/F}G$. We recall the stable norm map $\mathscr{N}$ from the set of stable $\sigma$-conjugacy classes in $G(L_n)=R(F)$ to the set of stable conjugacy classes in $G(F)$ (\cite[$\S$5]{Kottwitz86} and \autoref{subsubsec:w-stable_sigma-conjugacy}). Suppose that $\gamma_0\in G(F)$ is the stable norm of some $\delta\in G(L_n)$. Then, for $I_0:=G_{\gamma_0}^{\mathrm{o}}$, any element of $\ker[H^1(F,I_0)\rightarrow H^1(F,G)]$ (rather, its image in $H^1(F,G_{\delta\theta})$ via the canonical isomorphism $H^1(F,I_0)=H^1(F,G_{\delta\theta}^{\mathrm{o}})$) determines a unique $\sigma$-conjugacy class in $G(L_n)$ that is stably $\sigma$-conjugate to $\delta$, in view of the cohomological description (\ref{eq:sigma-conj_in_stable-sigma-conj}) of the set of such $\sigma$-conjugacy classes and the commutative diagram (\ref{eq:stable_sigma_conj_diagm}). Moreover, if we fix  $c\in G(L)$ such that $c\gamma_0c^{-1}=\Nm_n\delta$ and $b:=c^{-1}\delta\sigma(c)\in I_0$ (then, $b\in I_0(L)$ is basic), 
the composite of $\kappa_G$ (\ref{eq:kappa_G}) and $j_{[b]}^{I_0}$ (\ref{eq:j_[b]^{I_0}}) 
\[\kappa_G\circ j_{[b]}^{I_0} : \ker[H^1(F,I_0)\rightarrow H^1(F,G)] \rightarrow B(I_0) \rightarrow B(G)\rightarrow X^{\ast}(Z(\hat{G})^{\Gamma})\] 
is constant with image $\kappa_G([b])=\kappa_G([\delta])$ (cf. \autoref{subsubsec:w-stable_sigma-conjugacy}).

Let $\theta$ be the $F$-automorphism of $R$ induced by $\sigma=\sigma|_{L_n}$. There exists a natural choice of an embedding $i:{}^LG\rightarrow {}^LR$ and an automorphism $\hat{\theta}$ of $\hat{G}$. Let $\tilde{s}$ be the element of $\mathfrak{Z}$, the centralizer of $i\circ\xi(\hat{H})$ in $\hat{R}$, defined in \cite[(A.1.3.1)]{Kottwitz10}, so that the composite $i\circ\xi: \hat{H}\rightarrow\hat{G}\rightarrow \hat{R}$ identifies $\hat{H}$ with the identity component of the $\hat{\theta}$-centralizer of $\tilde{s}$ in $\hat{R}$. Further, let $\tilde{\xi}:\mathcal{H}={}^LH\rightarrow {}^LR$ be the \emph{allowed} embedding defined by that $\tilde{\xi}=i\circ\xi$ on $\hat{H}$ and $\tilde{\xi}(\tilde{\sigma})=\tilde{s}\cdot i\circ\xi(\tilde{\sigma})$ for any lift $\tilde{\sigma}\in W_F$ of $\sigma$. Then, the datum $(H,\tilde{s},\tilde{\xi})$ is a twisted endoscopic datum of $(R,\theta)$.

We may assume that $s\in Z(\hat{H})^{\Gamma}Z(\hat{G})$, by condition (iv) of the definition: there exists $z\in Z(\hat{G})$ such that $sz\in Z(\hat{H})^{\Gamma}Z(\hat{G})$. For any $(G,H)$-regular, semi-simple element $\gamma_H$ of $H(F)$ which transfers to $\gamma_0\in G(F)$, if $I_{H}:=H_{\gamma_H}^{\mathrm{o}}$, $I_0:=G_{\gamma_0}^{\mathrm{o}}$, one regards $s$ as an element of $Z(\hat{I}_0)^{\Gamma}Z(\hat{G})$ via the canonical $\Gamma$-equivariant homomorphisms $Z(\hat{H})\hookrightarrow Z(\hat{I}_H)\isom Z(\hat{I}_0)$. In this base-change situation, the twisted endoscopic transfer established by Waldspurger \cite{Waldspurger08} gives the following statement.

%%%%%%%%%%%%%%%%%%%%
%%%%%%%%%%%%%%%%%%%%
\begin{thm} \label{thm:twisted_endoscopy_transfer}
Suppose given $\mu_0\in X^{\ast}(Z(\hat{G})^{\Gamma})$.
For every $f\in C_c^{\infty}(G(L_n))$ with the property that the twisted orbital integral $\mathrm{TO}^{G(F)}_{\delta}(f)$ is zero if $\kappa_G([\delta])\neq\mu_0$, there exists an $f^{H_1}\in C_{c,\lambda_{H_1}}^{\infty}(H_1(F))$ such that for each $(G,H_1)$-regular, semi-simple element $\gamma_{H_1}$ of $H_1(F)$ with image $\gamma_H$ in $H(F)$, the stable orbital integral $\mathrm{SO}^{H_1(F)}_{\gamma_{H_1}}(f^{H_1})$ is zero unless $\gamma_H$ transfers to an element $G(F)$ which is a stable norm of an element $\delta$ of $G(L_n)$ with $\kappa([\delta])=\mu_0$, in which case, if we fix such an $\gamma_0\in G(F)$ (i.e. $\gamma_0$ is a transfer of $\gamma_H$ and $\gamma_0=\mathscr{N}\delta$), we have
\begin{equation} \label{eq:twisted_endo-transfer1}
\mathrm{SO}^{H_1(F)}_{\gamma_{H_1}}(f^{H_1}) = \sum_{\alpha\in\mathfrak{D}(I_0,G;F)} \langle \tilde{\alpha},s\rangle \Delta(\gamma_{H_1},\gamma_0) e(G_{\delta_{\alpha}\theta}^{\mathrm{o}}) \mathrm{TO}^{G(F)}_{\delta_{\alpha}}(f),
\end{equation}
where for each $\alpha\in \mathfrak{D}(I_0,G;F)$, we choose an element $\delta_{\alpha}$ of $G(L_n)$ whose stable $\sigma$-conjugacy class corresponds to the image of $\alpha$ in $H^1(F,G_{\delta\theta})$ and $\tilde{\alpha}$ is the lifting of $j_{[b]}^{I_0}(\alpha)\in B(I_0)_{basic}\cong X^{\ast}(Z(\hat{I}_0)^{\Gamma})$ to $X^{\ast}(Z(\hat{I}_0)^{\Gamma}Z(\hat{G}))$ whose restriction to $Z(\hat{G})$ is $\mu_0$.
\end{thm}

In our application, for $\mu_0$ we will take the element $\mu^{\natural}$ (\ref{eqn:mu_natural}).

%%%%%%%%%%%%%%%%%%%%
\begin{rem}
When $\gamma_{H_1}$ is strongly $G$-regular, semi-simple and $\gamma$ is strongly regular, the identity 
(\ref{eq:twisted_endo-transfer1}) is just the transfer theorem  \cite{Waldspurger08} for the base-change twisted endoscopic datum $(H,\tilde{s},\tilde{\xi})$ introduced earlier.
Indeed, first when $G_{\gamma}$ (equiv. $G_{\delta\theta}$) is connected, the right hand side of (\ref{eq:twisted_endo-transfer1}) can be rewritten as the right hand side of:
\begin{equation} \label{eq:twisted_endo-transfer2}
\mathrm{SO}^{H_1(F)}_{\gamma_{H_1}}(f^{H_1}) = \sum_{\delta} \langle \tilde{\alpha}(\gamma_0;\delta),s\rangle \Delta(\gamma_{H_1},\gamma_0) e(G_{\delta\theta}) \mathrm{TO}^{G(F)}_{\delta}(f),
\end{equation}
where $\delta$ runs through a set of $\sigma$-conjugacy classes of elements in $G(L_n)$ whose stable norm $\mathscr{N}\delta$ is $\gamma_0\in G(F)$ and such that the attached invariant $\alpha(\gamma_0;\delta)\in B(I_0)$ maps to $\mu^{\natural}$ under $\kappa_G$, and $\tilde{\alpha}(\gamma_0;\delta)$ is the extension of $\alpha(\gamma_0;\delta)$ to $X^{\ast}(Z(\hat{I}_0)^{\Gamma}Z(\hat{G}))$ whose restriction to $Z(\hat{G})$ is $\mu^{\natural}$.
Secondly, when $\gamma_{H_1}$ is strongly $G$-regular, semi-simple and $\gamma_0$ is strongly regular, one has the relation
\[\Delta(\gamma_{H_1},\delta)=\langle \tilde{\alpha}(\gamma_0;\delta),s\rangle \Delta(\gamma_{H_1},\gamma_0),\] 
between the twisted transfer factor $\Delta(\gamma_{H_1},\delta)$ and the standard transfer factor $\Delta(\gamma_{H_1},\gamma_0)$ \cite[Thm.5.6.2]{KottwitzShelstad12} (note that this result corrects the sign mistake in A.3.11.1 of \cite{Kottwitz10}).
Then, assuming that $G^{\der}=G^{\uc}$, Kottwitz extended this identity (\ref{eq:twisted_endo-transfer2}) to general $(G,H)$-regular semi-simple elements $\gamma_H$ of $H(F)$ \cite[Prop,A.3.14]{Kottwitz10}, using the arguments alluded to in (\ref{rem:untwisted_endoscopy_transfer}) and \cite[Prop.A.3.12]{Kottwitz10}.
We will see in the proof below that the same arguments continue to work for our version (\ref{eq:twisted_endo-transfer1}). Note that the right hand side of (\ref{eq:twisted_endo-transfer2}) does not make sense in our situation since the invariant $\alpha(\gamma_0;\delta)\in B(I_0)$ is not well-defined by the pair $(\gamma_0;\delta)$ unless $G_{\gamma}$ is connected. This is why we had to rewrite it as in (\ref{eq:twisted_endo-transfer1}).
\end{rem}

\begin{proof}
We keep the notation and assumption of \cite[A.3.11]{Kottwitz10}.
We only need to consider the case that $\gamma_H$ transfers to $\gamma_0\in G(F)$ which is a stable norm of $\delta\in G(L_n)$ with $\kappa([\delta])=\mu_0$. We may choose an elliptic maximal torus $T_H$ of $I_H=H_{\gamma_H}^{\mathrm{o}}$, an admissible embedding $T_H\isom T\subset G$ sending $\gamma_H$ to $\gamma_0$, and an $F$-embedding $k:T\hookrightarrow G_{\delta\theta}^{\mathrm{o}}$ (which exist since $T_H$ and $T$ are elliptic in $I_H$ and $I_0$, respectively).

Then, there exists an open neighborhood $U$ of $1$ in $T(F)$ such that for all $t\in U$,
$\Delta(\gamma_{H_1}(t),\gamma_0(t))$ is constant whenever $\gamma_{H_1}(t):=t^n\gamma_{H_1}$ is strongly $G$-regular (semisimple) and $\gamma_0(t):=t^n\gamma_0\in T(F)$ is strongly regular (semisimple) \cite[2.4]{LS87}. We set $\delta(t):=k(t)\delta$ (so, $\mathscr{N}\delta(t)=\gamma_0(t)$). Also, for strongly regular $\gamma_0(t)$ (whose centralizer is thus $T$) and for $\alpha\in \mathfrak{D}(T,G;F)$, the stable $\sigma$-conjugacy class of $\delta(t)_{\alpha}:=k(t)\delta_{\alpha}$ corresponds to $\alpha\in H^1(F,T)\isom H^1(F,G_{\delta(t)\theta})$.
Hence, for $t\neq1\in U$ such that $\gamma_0(t)$ is strongly regular, the right hand side of (\ref{eq:twisted_endo-transfer1}) for $\gamma_{H_1}(t)$ and $\gamma_0(t)$ equals $\sum_{\alpha\in\mathfrak{D}(T,G;F)} \langle \tilde{\alpha},s\rangle \Delta(\gamma_{H_1},\gamma_0) \mathrm{TO}^{G(F)}_{\delta(t)_{\alpha}}(f)$. Clearly, the pairing $\langle \tilde{\alpha},s\rangle$ and $\delta(t)_{\alpha}$ alll depend only on the image of $\alpha$ in $\mathfrak{D}(I_0,G;F)$.
Also, the map $\mathfrak{D}(T,G;F)\rightarrow \mathfrak{D}(I_0,G;F)$ is surjective \cite[10.2]{Kottwitz86}. Therefore, the degree $0$ part of its Shalika germ about $1$ equals \cite[(7.6)]{Clozel90}
\[ |\mathfrak{D}(T,I_0;F)|\cdot (-1)^{q(I_0^{\ast})}\cdot \sum_{\alpha\in\mathfrak{D}(I_0,G;F)} \langle \tilde{\alpha},s\rangle \Delta(\gamma_{H_1},\gamma_0) e(G_{\delta_{\alpha}\theta}^{\mathrm{o}}) \mathrm{TO}^{G(F)}_{\delta_{\alpha}}(f),\]
where $I_0^{\ast}$ is a (common) quasi-split inner form of $I_0$ and $G_{\delta_{\alpha}\theta}^{\mathrm{o}}$ and $q(I_0^{\ast})$ is the $F$-rank of $(I_0^{\ast})^{\der}$.
By comparing this with the degree $0$ part of the Shalika germ about $1$ of the function $t\mapsto \mathrm{SO}^{H_1(F)}_{\gamma_{H_1}(t)}(f^{H_1})$, we obtain the identity (\ref{eq:twisted_endo-transfer1}).
\end{proof}

%%%%%%%%%%%%%%%%%%%%
\subsubsection{Stabilization at $\infty$}

We fix an elliptic maximal torus $T$ of $G_{\R}$; by assumption, there exist an (elliptic) maximal torus $T_{H}$ of $(H)_{\R}$ and an admissible embedding $j_0:T_{H}\rightarrow T$ (defined over $\R$), and thus, the maximal split tori $A_{H_{\R}}$, $A_{G_{\R}}$ in the centers of $H_{\R}$ and $G_{\R}$ are canonically isomorphic.

%%%%%%%%%%%%%%%%%%%%
%%%%%%%%%%%%%%%%%%%%
\begin{thm} \cite[(7.4)]{Kottwitz92}, \cite[Rem.6.2.2]{Morel10} \label{thm:pseudo-coeff}
There exists a function $h_{\infty}$ on $H_1(\R)$, compactly supported modulo $Z_{H_1}(\R)^{\mathrm{o}}$, such that for every semisimple element $\gamma_{H_1}$ of $H_1(\R)$,
the stable orbital integral $\mathrm{SO}_{\gamma_{H_1}}(h_{\infty})$ equals $0$ unless $\gamma_{H_1}$ is $(G,H_1)$-regular and elliptic, in which case
\[ \mathrm{SO}_{\gamma_{H_1}}(h_{\infty}) = \langle \tilde{\alpha}_{\infty}(\gamma_0),s\rangle \cdot \Delta_{\infty}(\gamma_{H_1},\gamma_0) \cdot e(I_{\infty}) \cdot \mathrm{tr}\xi_{\C}(\gamma_0) \cdot \mathrm{vol}(A_{G_{\R}}(\R)^{\mathrm{o}}\backslash I_0(\infty)(\R))^{-1}, \]
where the terms on the right side are as follows (cf. \cite[p.182]{Kottwitz90}):
$\gamma_0:=j_0(\gamma_H)$, $I_0(\infty)$ is the inner form of $(I_0)_{\R}:=Z_{G}(\gamma_0)^{\mathrm{o}}_{\R}$ with anisotropic adjoint $I_0(\infty)^{\ad}$, and $e(I_{\infty})$ is the sign attached to $I_{\infty}$ by Kottwitz \cite{Kottwitz83}. The element $\tilde{\alpha}_{\infty}(\gamma_0)\in X^{\ast}(Z(\hat{I_0})^{\Gamma(\infty)}Z(\hat{G}))$ is the extension of $\alpha_{\infty}(\gamma_0)\in X^{\ast}(Z(\hat{I_0})^{\Gamma(\infty)})$ (as defined in \autoref{subsubsec:Kottwitz_invariant}) whose restriction to $Z(\hat{G})$ is $-\mu_0$. 
\end{thm}

\begin{proof}
Basically, the arguments on p. 182--186 of \cite{Kottwitz90} carry over to our general setting without essential change, and here we will be just contented with explanation of the necessary modifications which mainly concern generalization (for $(G,\mathcal{H},H_1,\xi_1)$) of certain properties of the transfer factor $\Delta_{\infty}$ (which were stated for $(G,H)$ in \textit{loc. cit.}): recall that in the Kottwitz's set-up (i.e. $G^{\der}=G^{\uc}$), the split extension $\mathcal{H}$ in the endoscopic datum is the $L$-group ${}^LH$ of $H$ and one does not need a $z$-pair $(H_1,\xi_1)$ (i.e. $H_1=H$, $\xi_1=id$). 
To state the required generalization (and fix notations), it seems unavoidable to repeat, in our setting, the discussion of Kottwitz in \textit{loc. cit.} (see (a)- (e) below). In the following discussion, as we will work exclusively with groups over $\R$ and their dual groups, we will simply write $G$, $H$, $H_1$ ... for their base-changes to $\R$.

(a) We fix an $\R$-splitting $spl_{\hat{G}}=(\mathcal{T},\mathcal{B},\{X_{\alpha}\})$ of $\hat{G}$ (as usual, $(\mathcal{T},\mathcal{B})$ is a (Borel) pair of $\hat{G}$, $\{X_{\alpha^{\vee}}\}$ is a collection of root vectors), and, assuming $s\in\mathcal{T}$, take $spl_{\hat{H}}=(\mathcal{T},\mathcal{B}_{\hat{H}}=\mathcal{B}\cap \hat{H},\{X_{\alpha^{\vee}}\})$ for an $\R$-splitting of $\hat{H}$ ($\alpha^{\vee}$ runs through $R(\mathcal{T},\hat{H})$); this also determines an $\R$-splitting $spl_{\hat{H}_1}=(\mathcal{T}_{1},\mathcal{B}_{\hat{H}_1},\cdots)$ of $\hat{H}_1$ and an embedding ${}^LH\rightarrow {}^LH_1$. Also, when $\mathfrak{B}$ denotes the set of Borel subgroups of $G_{\C}$ containing $T$ and $\mathcal{B}_{H_1}$ the similar set for $(H_1,T_{H_1})$, for $B\in \mathfrak{B}$, let $\iota_B$ denote one-half the sum of $B$-positive roots of $T$ in $G$ and $\Delta_B(-)$ the function on $T(\R)$ defined by \[\Delta_B(\gamma)=\prod_{\alpha>^B0}(1-\alpha(\gamma)^{-1}),\] with $\alpha$ running through the $B$-positive roots of $T$.

(b) To any $B\in \mathfrak{B}$ (more precisely, to the associated \emph{based} $\chi$-data) and any $\R$-splitting $spl_{\hat{G}}$ of $\hat{G}$, Langlands and Shelstad  \cite[(2.6)]{LS87} construct a canonical admissible embedding:
\[\eta_{B}:{}^LT\rightarrow {}^LG \] 
which is is uniquely determined up to $\mathcal{T}$-conjugacy.
In more detail, $\eta_{B}|_{\hat{T}}$ is the isomorphism $\hat{T}\isom \mathcal{T}$ determined by the chosen (Borel) pairs $(T,B)$, $(\mathcal{T},\mathcal{B})$, the restriction of $\eta_{B}$ to $\C^{\times}=W_{\C}\subset {}^LT$ is given by $z\mapsto (z/|z|)^{\iota_B}$, and $\eta_{B}(\tau)$ for $\tau\in\Gal(\C/\R)\subset W_{\R}$ has a certain explicit description in terms of the splitting $spl_{\hat{G}}$. 

(c) Let $J$ be the set of admissible isomorphisms $j:T_{H}\rightarrow T$ (\ref{eq:admissible_embedding}); $\Omega_G:=\Omega(T(\C),G(\C))$ acts on $J$ simply transitively \cite[Prop.2.2]{Shelstad79}. Then, any pair $(j,B)\in J\times \mathfrak{B}$ determines unique Borel subgroups $B_{H}\in \mathfrak{B}_{H}$, $B_{H_1}\in \mathfrak{B}_{H_1}$ (in the usual manner via identifications of the associated based root data), in which case we write $(j,B)\mapsto B_{H}, B_{H_1}$. Any choice of pairs $(\mathcal{T}_1,\mathcal{B}_{\hat{H}_1})$, $(T_1,B_{H_1})$ (in particular any choice of $(j,B)$) presents the (absolute) Weyl group $\Omega_{H_1}=\Omega(T_1(\C),H_1(\C))$ as a subgroup of $\Omega_G$. Given $(j,B)\in J\times \mathfrak{B}$, we can write any $\omega\in \Omega_G$ as $\omega=\omega_{H_1}\omega_{\ast}$ uniquely, where $\omega_{H_1}\in \Omega_{H_1}$ and $\omega_{\ast}$ belongs to the set $ \Omega_{\ast}$ of $\omega\in\Omega_G$ such that $(j,\omega(B))$ and $(j,B)$ have the same image under $J\times \mathfrak{B}\rightarrow \mathfrak{B}_{H_1}$.

(d) For any elliptic $L$-parameter $\varphi:W_{\R}\rightarrow{}^LG$, we may assume (by conjugation under $\hat{G}$) that it factors through $\mathcal{T}$ and $\varphi(z)=z^{\Lambda}\bar{z}^{\Lambda'}\ (z\in \C^{\times})$ for a unique $\Lambda, \Lambda' \in X_{\ast}(\mathcal{T})_{\C}$ satisfying that $\Lambda-\Lambda'\in X_{\ast}(\mathcal{T})$ and $\langle\Lambda,\alpha^{\vee}\rangle>0$ for all $\mathcal{B}$-positive roots $\alpha^{\vee}$. With this tuning, for any $B\in \mathfrak{B}$, we obtain an $L$-parameter $\varphi_B:W_{\R}\rightarrow {}^LT$ such that $\varphi=\eta_B\circ\varphi_B$. Let 
\[ \chi_B=\chi(\varphi,B)\in \Hom_{cont}(T(\R),\C^{\times})\] 
be the quasi-character on $T(\R)$ corresponding to $\varphi_B$ by the Langlands correspondence for tori: $H^1(W_{\R},\hat{T})=\Hom_{cont}(T(\R),\C^{\times})$. 

For $B_{H_1}\in \mathfrak{B}_{H_1}$ and an elliptic $L$-parameter $\varphi:W_{\R}\rightarrow{}^LH_1$, we let $\eta_{B_{H_1}}$ and $\chi_{B_{H_1}}$ be respectively the admissible embedding ${}^LT_{H_1}\rightarrow {}^LH_1$ and the associated quasi-character $\chi_{B_{H_1}}$ of $T_{H_1}(\R)$ defined as above for $(T_{H_1},B_{H_1})$. In fact, we are interested in the (elliptic) $L$-parameters for $H_1$ with certain properties.
Indeed, let $\Phi_{temp}(H_1,\lambda_1)$ denote the set of equivalence classes of tempered $L$-parameters $\varphi_{H_1}$ whose associated character on $Z_1(\R)$ equals $\lambda_{H_1}$ and $\Phi_{temp}(G)$ the set of equivalence classes of temped $L$-parameters of $G$. Then, the pair of embeddings $\xi:\mathcal{H}\rightarrow {}^LG$, $\xi_1:\mathcal{H}\rightarrow  {}^LH_1$ gives rise to a map
\begin{equation} \label{eq:endoscopic_transfer_L-paramters}
\Phi_{temp}(H_1,\lambda_1) \rightarrow \Phi_{temp}(G),
\end{equation}
cf. \cite[$\S$2]{Shelstad10} (When $H_1=H$ and $\xi_1=id$, this map simply sends $\varphi_{H_1}$ to the composite $\xi\circ\varphi_{H_1}$).
Any $L$-parameter $\varphi_{H_1}\in \Phi_{temp}(H_1,\lambda_1)$ whose image in $\Phi_{temp}(G)$ is elliptic is also elliptic.

(e) The (local) Langlands-Shelstad transfer factors \cite{LR87} are determined only up to a constant and their definition requires certain auxiliary choices, namely choices of \emph{$a$-data} and \emph{$\chi$-data} (the transfer factors themselves are independent of these choices). Given $B\in\mathfrak{B}$, for $\chi$-data we will use the \emph{based} choice \cite[$\S$9]{Shelstad08} defined by the corresponding positive system $R(T,B)$ of roots, and for $a$-data, we set 
\[ a_{\alpha}:=i=:-a_{-\alpha},\quad \forall\alpha\in R(T,B) \] 
(recall that since $T$ is elliptic, every root is imaginary so that $a_{\alpha}$ must be purely imaginary).
Then, for any choice of $(j,B)\in J\times \mathfrak{B}$, with the resulting choice of based $\chi$-data and Borel subgroup $B_{H_1}\in \mathfrak{B}_{H_1}$ (i.e. $(j,B)\mapsto (B_{H_1})$), Kottwitz and Shelstad \cite[p.38-40]{KottwitzShelstad99} construct a $1$-cocycle in $Z^1(W_{\R},\mathcal{T}_1)$ or equivalently a quasi-character on $T_1(\R)$
\[ a_{j,B}\in \Hom_{cts}(T_{H_1}(\R),\C^{\times})\]
from the associated embeddings $\xi$, $\xi_1$, $\eta_{B}$, and $\eta_{B_{H_1}}$, which, in the case $H_1=H$ and $\xi_1=id$ (so, $\xi:\mathcal{H}={}^LH\rightarrow {}^LG$), is defined by $\xi\circ\eta_{B_H}\circ\hat{j}=a_{j,B}\cdot\eta_B$ (using that $\xi\circ\eta_{B_H}\circ\hat{j}|_{\hat{T}}=\eta_B|_{\hat{T}}$). 
In \cite[p.184]{Kottwitz90} $a_{j,B}$ was denoted by $\chi_{G,H}$.
Here, in the general case, we follow the discussion of \cite[$\S$10]{Shelstad08}, where $\eta_{B}$, $\eta_{B_{H_1}}$, and $a_{j,B}$ are denoted by $\xi_T$, $\xi_{T_{H_1}}$, and $a_{T_1}$, respectively.
Then, we claim that $a_{j,B}$ satisfies the following property: 

Suppose that $\varphi_{H_1}$ is an elliptic $L$-parameter lying in $\Phi_{temp}(H_1,\lambda_1)$ and $\varphi$ is its image under the map (\ref{eq:endoscopic_transfer_L-paramters}). Then, for every (strongly) $G$-regular elliptic $\gamma_{H_1}\in T_{H_1}(\R)$ which is an image of $\gamma\in T(\R)$ (i.e. $\gamma_H$ is the image of $\gamma$ under \emph{some} admissible embedding $T_H\rightarrow T$ (\ref{eq:admissible_embedding})) and any $\omega=\omega_{1}\omega_{\ast}$ with $\omega_{1}\in \Omega_{H_1}$ and $\omega_{\ast}\in\Omega_{\ast}$, one has
\begin{equation} \label{eq:transfer_factor-III_2}
a_{j,\omega(B)}(\gamma_{H_1},\gamma) \cdot \chi_{\omega_{1}(B_{1})}(\gamma_{H_1})=\chi_{\omega(B)}(\gamma).
\end{equation}
The left-hand side factors through $T_{H_1}(\R)/Z_1(\R)=T_H(\R)$ (as a function in $\gamma_{H_1}$) and the comparison is done through $j:T_H(\R)\isom T(\R)$. Note that $(j,\omega(B))\mapsto \omega_{1}(B_{H_1})$ when $(j,B)\mapsto B_{H_1}$, so we may reduce to the situation $\omega=\omega_{H_1}=\omega_{\ast}=id$, in which case
one just needs to check the obvious relation among the corresponding Langlands parameters for the quasi-characters $\chi_{B_{H_1}}$, $a_{j,B}$, $\chi_{B}$. Such verification can be found in \cite[$\S$11]{Shelstad08}, \cite[$\S$7.b]{Shelstad10}.

Now, it is not difficult to see that the original arguments of Kottwitz (especially, those on p.186 of \cite{Kottwitz90}) continue to work in our general setting if the following statement holds for the transfer factor $\Delta_{\infty}$: 
There exists a constant $c$ (depending only on $G$, the endoscopic datum $(H,\cdots)$ plus the $z$-pair $(H_1,\xi_1)$, and $j$, but not on $B$) such that 
for all $(G,H_1)$-regular $\gamma_{H_1}\in T_{H_1}(\R)$ and $\gamma=j_1(\gamma_{H_1})$ ($j_1$ denoting the composite $T_{H_1}\rightarrow T_H\stackrel{j}{\rightarrow} T$)
\begin{equation} \label{eq:transfer_factor}
\Delta_{\infty}(\gamma_{H_1},\gamma)=c\cdot (-1)^{q(G)+q(H_1)}\cdot a_{j,B}(\gamma)\cdot \Delta_B(\gamma^{-1})\cdot  \Delta_{B_{H_1}}(\gamma_{H_1}^{-1})^{-1}.
 \end{equation}
As explained in \cite[p.186]{Kottwitz90}, by a continuity argument due to Shelstad, one only needs to establish this equality for strongly $G$-regular $\gamma_{H_1}\in T_{H_1}(\R)$, so from now on we will assume that. Recall that our transfer factor $\Delta_{\infty}$ for the endoscopic datum $(H,\mathcal{H},s,\xi)$ is equal to the transfer factor $\Delta_{\infty}^{\mathrm{LS}}$ defined by Langlands and Shelstad for the endoscopic datum $(H,\mathcal{H},s^{-1},\xi)$. We will establish the identity (\ref{eq:transfer_factor}) for $\Delta_{\infty}^{\mathrm{LS}}$ and at the same time show that $\Delta_{\infty}(\gamma_{H_1},\gamma)=C\cdot \Delta_{\infty}^{\mathrm{LS}}(\gamma_{H_1},\gamma)$ for some non-zero constant $C$ which depends only on the set of data $\{G, (H,\cdots), (H_1,\xi_1), j\}$ above, but neither on $B$ nor on the strongly $G$-regular pair $(\gamma_{H_1},\gamma)$.

The (local) Langlands-Shelstad transfer factor $\Delta_{\infty}^{\mathrm{LS}}$ \cite{LS87}, \cite{LS90}, \cite{KottwitzShelstad99} is a product of five terms:
\[ \Delta_{\infty}^{\mathrm{LS}}=\Delta_{I}\cdot \Delta_{II} \cdot \Delta_{III_1} \cdot \Delta_{III_2}\cdot \Delta_{IV}, \]
up to a constant. Recall that we need to fix a reference pair $(\gamma_{H_1}',\gamma')$ for which we take any pair such that $\gamma'=j_1(\gamma_{H_1}')$.
The transfer factor $\Delta_{\infty}(\gamma_{H_1},\gamma)$ is defined to be zero unless $\gamma_{H_1}$ is an image of $\gamma\in G(\R)$ under an admissible embedding $T_H\isom T'\subset G$ (not necessarily $j$ even if $T'=T$). When we discuss the factors $\Delta_{\bullet}$ ($\bullet=\infty,I,II,\cdots$), we will not assume that $\gamma=j_1(\gamma_{H_1})$ (but, only that $\gamma_{H_1}\in T_{H_1}(\R)$ and $\gamma\in T(\R)$), and then will use that condition for $\Delta_{III_1}$ when we establish (\ref{eq:transfer_factor}). More precisely, following Shelstad \cite[$\S$8-$\S$11]{Shelstad08}, we will group together the (relative) terms $\Delta_{II}$, $\Delta_{III_2}$, $\Delta_{IV}$ and will verify that the right-hand side of (\ref{eq:transfer_factor}) is a non-zero constant multiple of their product $\Delta_{II_+}:=\Delta_{II}\cdot \Delta_{III_2}\cdot \Delta_{IV}$ (the constant depending only on the set of data $\{G, (H,\cdots), (H_1,\xi_1), j\}$ above, but not on $B$) and that the (relative) term $\Delta_I$ also depends only on the same set of data, but not on the pair $(\gamma_{H_1},\gamma)$, while $\Delta_{III_1}(\gamma_{H_1},\gamma;\gamma_{H_1}',\gamma')=1$ whenever $\gamma=j_1(\gamma_{H_1})$, $\gamma'=j_1(\gamma_{H_1}')$. 

First, we split $\Delta_{II}$ into a product of $\Delta_{II}^{\ast}(\gamma_{H_1},\gamma):=\prod_{\alpha}\chi_{\alpha}(\alpha(\gamma)-1)$ and $\prod_{\alpha}\chi_{\alpha}(a_{\alpha})^{-1}$, where in both products, $\alpha$ runs over the $B$-positive roots of $(T,G)$ \emph{outside $H_1$} \cite[p.232]{Shelstad08} (use that $T$ is elliptic in $G$). We note that the second product, which a priori depends on the choice of $a$-data and based $\chi$-data, so on $B$, equals $(-i)^{1/2(\mathrm{dim}G-\mathrm{dm}H_1)}$. Also, from \cite[p.232, line+19]{Shelstad08}, we have 
\[ \Delta_{II}^{\ast}(\gamma_{H_1},\gamma) \Delta_{IV}(\gamma_{H_1},\gamma)=(-1)^{1/2(\mathrm{dim} G -\mathrm{dim} H_1)} \Delta_B(\gamma^{-1}) \Delta_{B_{H_1}}(\gamma_{H_1}^{-1})^{-1}.\]
Moreover, by definition \cite[$\S$10-$\S$11]{Shelstad08}, we have
\[ \Delta_{III_2}(\gamma_{H_1},\gamma)= a_{j,B}(\gamma_{H_1}) \]
if $\gamma_{H_1}$ is the image of $\gamma$ (under some admissible embedding from $T_H$ into $G$).
Therefore, we have proved that the right-hand side of (\ref{eq:transfer_factor}) is a constant multiple of $\Delta_{II_+}$. It also follows from the discussion on p. 234-245 of \cite{Shelstad08} that when we write $\Delta_{II_+}^{j,B}$ for $\Delta_{II_+}$ defined for $(j,B)\in J\times \mathfrak{B}$, for any $\omega\in \Omega_G$, one has
\[ \Delta_{II_+}^{j,\omega(B)}=\det \omega_{\ast} \cdot \Delta_{II_+}^{j,B},\]
where $\det \omega_{\ast}=\det (\omega_{\ast};X^{\ast}(T))$.
It is also evident from the definition \cite[(3.2)]{LS87} that the term $\Delta_{I}$ depends only on the admissible embedding $j:T_H\rightarrow T$, but not on $B$: strictly speaking, the discussion in \textit{loc. cit.} applies to the admissible embedding $j_{\ast}:T_H\isom T_{\ast}$ in (\ref{eq:admissible_embedding}) (which is a part of the datum of admissible embedding $j$). In more detail, $\Delta_{I}(\gamma_{H_1},\gamma;\gamma_{H_1}',\gamma')$ is the ratio $\Delta_{I}(\gamma_{H_1},\gamma)\Delta_{I}(\gamma_{H_1}',\gamma')^{-1}$ with both terms being defined by the same recipe and $\Delta_{I}(\gamma_{H_1},\gamma)=\langle \lambda(T_{\ast}^{\uc}),\mathbf{s}_{T_{\ast}}\rangle$, where both $\lambda(T_{\ast}^{\uc})\in H^1(\R,T_{\ast}^{\uc})$ and $\mathbf{s}_{T_{\ast}}\in \pi_0((\hat{T}_{\ast}/Z(\hat{G}))^{\Gamma_{\infty}})$ depend only on $j_{\ast}$ \cite[(2.3.3) and p.241]{LS87}.
Finally, for the term $\Delta_{III_1}^{\ast}$ (which is the only genuine relative term $\Delta_{III_1}(\gamma_{H_1},\gamma;\gamma_{H_1}',\gamma')$ in $\Delta_{\infty}(\gamma_{H_1},\gamma;\gamma_{H_1}',\gamma')$), it follows from the discussion in \cite[(3.4)]{LS87} that when $\gamma=j_1(\gamma_{H_1})$, $\gamma'=j_1(\gamma_{H_1}')$, the cohomology classes $v(\sigma)$, $\bar{v}(\sigma)$ (thus, $\mathrm{inv}\left(\frac{\gamma_{H_1},\gamma_G}{\bar{\gamma}_{H_1},\bar{\gamma}_G}\right)$ as well) there become trivial so that $\Delta_{III_1}(\gamma_{H_1},\gamma;\gamma_{H_1}',\gamma')=1$. 
Therefore, we see that the statement on the identity (\ref{eq:transfer_factor}) holds for the Langlands-Shelstad transfer factor $\Delta_{\infty}^{\mathrm{LS}}$. 
Furthermore, if we change $s$ to $s^{-1}$ in the endoscopic datum, the only terms being affected are $\Delta_{I}$ and $\Delta_{III_1}$: they both change to their inverses. Hence, the two transfer factors $\Delta_{\infty}(\gamma_{H_1},\gamma)$ and $\Delta_{\infty}^{\mathrm{LS}}(\gamma_{H_1},\gamma)$ are proportional.%%
\footnote{Our transfer factor $\Delta_{\infty}$ is denoted by $\Delta'$ in \cite[$\S$5.1]{KottwitzShelstad12}. We see that this is indeed given by the definition (1.0.4) of \textit{loc. cit.}, because, for standard endoscopy, one has $\Delta_{I}^{\mathrm{new}}=\Delta_{I}$, and also the normalization of Kottwitz and Shelstad for the Langlands correspondence for tori (chosen in the appendix of \textit{loc. cit.}) differs from ours by sign so that the two $\Delta_{III_1}$'s also differ by the same measure.} 
This completes the proof.
\end{proof}

%%%%%%%%%%%%%%%%%%%%
%%%%%%%%%%%%%%%%%%%%
\subsection{Stabilization of Lefschetz number formula} \label{subsec:stabilization_of_LNF}

We recall the definition of the elliptic part of the geometric side of the stable trace formula.

Let $G$ be a connected reductive group over a number field $F$ and $\mathscr{E}_{\mathrm{ell}}(G)$ a set of representatives for the isomorphism classes of elliptic endoscopic data of $G$. For each $\underline{H}=(H,\mathcal{H},s,\xi)\in \mathscr{E}_{\mathrm{ell}}(G)$, we fix a $z$-pair $(H_1,\xi_1)$.
Recall the quasi-character $\lambda_{H_1}$ on $Z_1(\A_F)/Z_1(F)$ for the central $F$-torus $Z_1$ in $H_1$.
Let $C_{c,\lambda_{H_1}}^{\infty}(H_1(\A_F))$ denote the space of $\C$-valued smooth functions on $H_1(\A_F)$ whose supports are compact modulo $Z_1(\A_F)$ and that satisfy $f^{H_1}(zh)=\lambda_{H_1}(z)^{-1} f^{H_1}(h)$ for all $z\in Z_1(\A_F)$ and $h\in H_1(\A_F)$. 
 
\begin{defn} \cite[9.2]{Kottwitz86}
Let $\underline{H}=(H,\mathcal{H},s,\xi)\in \mathscr{E}_{\mathrm{ell}}(G)$ and $(H_1,\xi_1)$ a $z$-pair of it.

(1) For a semi-simple element $\gamma_{H_1}$ of $H_1(\A_F)$ and $f^{H_1}\in C_{c,\lambda_{H_1}}^{\infty}(H_1(\A_F))$, we define the adelic stable orbital integral of $f^{H_1}$ along the adelic stable conjugacy class of $\gamma_{H_1}$ by
\[  \mathrm{SO}_{\gamma_{H_1}}(f^{H_1}):=\sum_{\alpha} e(\gamma_{\alpha}) \mathrm{O}_{\gamma_{\alpha}}(f^{H_1}),\]
where $\alpha$ runs through $\mathfrak{D}(I_0,H_1;\A_F)=\ker[H^1(\A_F,I_0)\rightarrow H^1(\A_F,H_1)]$ with $I_0=(H_1)_{\gamma_{H_1}}^{\mathrm{o}}$, $\gamma_{\alpha}$ is an element of $H_1(\A_F)$ whose $H_1(\A_F)$-conjugacy class in the adelic stable conjugacy class of $\gamma_{H_1}\in H_1(\A_F)$ corresponds to $\alpha$, and the number $e(\gamma_{\alpha})=\prod_ve(I_{\alpha,v})$ is the Kottwitz sign with $I_{\alpha,v}$ being the connected centralizer in $H_{1,F_v}$ of the $v$-component of $\gamma_{\alpha}$. The adelic orbital integral is defined with respect to the Tamagawa (or canonical) measure on $H_1(\A)$ \cite{Labesse01}.

(2) For $f^{H_1}\in C_{c,\lambda_{H_1}}^{\infty}(H_1(\A_F))$, we define the \emph{elliptic part of the geometric side of the stable trace formula} for $f^{H_1}$ to be
\begin{equation} \label{eq:epgsstf}
 \ST_{\elp}^{H_1}(f^{H_1}):=\tau(H) \sum_{\gamma_{H}\in E_{\mathrm{st}}(H)} |\pi_0(H_{\gamma_H})(F)|^{-1} \mathrm{SO}_{\gamma_{H_1}}(f^{H_1}),
 \end{equation}
where for each $\gamma_H\in E_{\mathrm{st}}(H)$, a set of representatives for the elliptic semi-simple stable conjugacy classes in $H(F)$, we fix an (elliptic) semisimple $\gamma_{H_1}\in H_1(F)$ lifting $\gamma_H$ (The stable orbital integral $\mathrm{SO}_{\gamma_{H_1}}(f^{H_1})$ depends only on $\gamma_H$, not on the choice of the lift $\gamma_{H_1}$). 
\end{defn}

For $\underline{H}=(H,\mathcal{H},s,\xi)\in \mathscr{E}_{\mathrm{ell}}(G)$, we write $\mathrm{Aut}(\underline{H})$ for its automorphism group and \[ \mathrm{Out}(\underline{H}):=\mathrm{Aut}(\underline{H})/H^{\ad}(\Q)\]  for the outer automorphism group. 
We also put
\[\mathrm{tr}(\Phi^m\times f^p | H_c(Sh_{K/\Qb},\sF_K)):=\sum_{i}(-1)^i\mathrm{tr}( \Phi^m\times f^p | H^i_c(Sh_{K}(G,X)_{\Qb},\sF_K)). \]

\begin{thm}  \label{thm:EP_GS_STF}
With the same setup as Thm. \ref{thm:Kottwitz_formula:Kisin}, 
for every $f^p$ in the Hecke algebra $\mathcal{H}(G(\A_f^p)/\!\!/ K^p)$, there exists $m(f^p)\in\N$, depending on $f^p$, with the following property: for each $m\geq m(f^p)$, there exists a function $f^{H_1}=f^{H_1}(m)\in C^{\infty}_{c,\lambda_{H_1}}(H_1(\A))$ such that one has
\[\mathrm{tr}(\Phi^m\times f^p | H_c(Sh_{K/\Qb},\sF_K))=\sum_{\underline{H}\in \mathscr{E}_{\mathrm{ell}}(G)} \iota(G,\underline{H}) \ST_{\elp}^{H_1}(f^{H_1}),\]
where
\[ \iota(G,\underline{H}):=\tau(G) \tau(H)^{-1} |\mathrm{Out}(\underline{H})|^{-1}.\]
If $G^{\ad}$ is $\Q$-anisotropic or $f^p$ is the identity, we can take $m(f^p)$ to be $1$.
\end{thm}

\begin{proof}
We will rewrite the right hand expression of (\ref{eq:Lef-number1}) as a similar sum with a different index set (in certain cohomology groups). Let $E_{\mathrm{st}}(G)$ be a set of representatives of stable conjugacy classes of $\Q$-elliptic semisimple elements in $G(\Q)$.
First, we rearrange its summation index as 
\begin{equation} \label{eq:rearrangement}
\sum_{\gamma_0}\sum_{(\gamma,\delta)},
\end{equation}
where in the first sum $\gamma_0$ runs through the subset of $\R$-elliptic elements in $E_{\mathrm{st}}(G)$ and in the second sum $(\gamma,\delta)$ runs through the set of elements $(\gamma,\delta)$ of $\prod'\mathfrak{C}_l(\gamma_0)\times \mathfrak{C}_p^{(n)}(\gamma_0)$ with trivial Kottwitz invariant, cf. (\ref{eq:C_l(gamma_0)}), Def. \ref{defn:D_p^{(n)}(gamma_0)},  Prop. \ref{prop:B(gamma_0)=D(I_0,G;Qp)}  (the restricted product $\prod'$ means that almost all $\gamma_l$'s are the distinguished elements). Here and thereafter, $\gamma$ and $\delta$ will also denote, by abuse of notation, their $G(\A_f)$-conjugacy class and $\sigma$-conjugacy class in $G(L_n)$ respectively. 
For each $\gamma_0\in E_{\mathrm{st}}(G)$, let us fix a reference element $\delta_0$ of $\mathfrak{C}_p^{(n)}(\gamma_0)$. With such choice of $\{\delta_0\}_{E_{\mathrm{st}}(G)}$,
we will lift the summation index $(\gamma,\delta)$ of (\ref{eq:rearrangement}) to the source of the surjective map
\[ \mathfrak{D}(I_0,G;\A_f)=\mathfrak{D}(I_0,G;\A_f^p)\oplus \mathfrak{D}(G_{\delta\theta}^{\mathrm{o}},G;\Qp) \stackrel{i_{\ast}^{\delta_0}}{\longrightarrow} \sideset{}{'}\prod\mathfrak{C}_l(\gamma_0)\times \mathfrak{C}_p^{(n)}(\gamma_0). \]
which sends the distinguished element to $(\gamma_0,\delta_0)$. Recall that a Kottwitz triple $(\gamma_0;\gamma,\delta)$ having trivial Kottwitz invariant means that there exists $(\alpha^p=(\alpha_l)_{\l\neq p};\alpha_p)$ in the source of this map $i_{\ast}^{\delta_0}$ such that the product $\alpha=\prod_v \tilde{\alpha}_v$ (including the $\infty$-component $\tilde{\alpha}_{\infty}$) of their extensions $\tilde{\alpha}_v\in X^{\ast}(Z(\hat{I_0})^{\Gamma_v}Z(\hat{G}))$ (\ref{eq:restriction_of_alpha_to_Z(hatG)}) is trivial as a character of $\mathfrak{K}(I_0/\Q)=\left( \bigcap_v Z(\hat{I}_0)^{\Gamma_v}Z(\hat{G}) \right)/Z(\hat{G})$. 
This suggests to take the second summation of (\ref{eq:rearrangement}) over the subset $\mathfrak{D}(I_0,G;\A_f)_1$ of $\mathfrak{D}(I_0,G;\A_f)$ consisting of the elements whose Kottwitz invariant is trivial, by taking into account the cardinality of the subset 
\[ \mathfrak{D}(\gamma_0;\alpha^p,\alpha_p):=\{ (x^p,x_p)\in \mathfrak{D}(I_0,G;\A_f)_1\ |\ i_{\ast}^{\delta_0}(x^p,x_p)=i_{\ast}^{\delta_0}(\alpha^p,\alpha_p)\}.\]

If we fix $(\alpha^p,\alpha_p)\in \mathfrak{D}(I_0,G;\A_f)_1$, an element of $(x^p,x_p)\in \mathfrak{D}(I_0,G;\A_f)$ also has trivial Kottwitz invariant if and only if their difference
$(\alpha^p-x^p,\alpha_p-x_p,0)$ (trivial component at $\infty$) lies in the kernel of the canonical map 
\[\mathfrak{D}(I_0,G;\A) \rightarrow \mathfrak{E}(I_0,G;\A) \rightarrow \mathfrak{K}(I_0/\Q)^D,\]
where $\mathfrak{E}(I_0,G;\A):=\ker[H^1(\A,(I_0)_{\bfab})\rightarrow H^1(\A,G_{\bfab})]$ and the second map comes from the canonical isomorphism 
\[\mathfrak{K}(I_0/\Q)^D=\mathrm{coker}[H^0(\A,G_{\bfab})\rightarrow H^1(\A/\Q,(\tilde{I}_0)_{\bfab})]\]
(Lemma \ref{lem:proof_of_Kottwitz86_Thm.6.6}, cf. \cite[Remarque on p.43]{Labesse99}: note that $H^0_{\bfab}(\A/\Q,I_0\backslash G)= H^1(\A/\Q,(\tilde{I}_0)_{\bfab})$).
Hence, if $(\gamma,\delta):=i_{\ast}(\alpha^p,\alpha_p)$, we have an identification
\begin{align*}
\mathfrak{D}(\gamma_0;\alpha^p,\alpha_p) \simeq &\ker[\mathfrak{D}(I_0,G;\A)\rightarrow H^1(\A_f^p,G_{\gamma_0})\oplus H^1(\Qp,G_{\delta\theta})\oplus H^1(\R,I_0)\oplus\mathfrak{K}(I_0/\Q)^D] \\
=&\im[ \ker[H^1(\Q,I_0)\rightarrow H^1(\A_f^p,G_{\gamma_0})\oplus H^1(\Qp,G_{\delta\theta})\oplus H^1(\R,I_0)\oplus H^1(\Q,G)]\rightarrow H^1(\A,I_0)]  \\
=&\im \left[\ker [\Sha^{\infty}_G(\Q, G_{\alpha}^{\mathrm{o}})\rightarrow H^1(\A,G_{\alpha})] \rightarrow H^1(\A,G_{\alpha}^{\mathrm{o}}) \right] \\
\cong&\mathfrak{D}(\gamma_0;\gamma,\delta)\quad (\text{Lemma }\ref{eq:|widetilde{Sha}_G(Q,I_{phi,epsilon})^+|}),
\end{align*}
where $G_{\alpha}$ is the inner twist of $G_{\gamma_0}$ determined by $\alpha=(\alpha_v)$; one has 
$G_{\alpha_l}\simeq G_{\gamma_l}$, $G_{\alpha_p} \simeq G_{\delta\theta}$, and
$\Sha^{\infty}_G(\Q, I_0)=\Sha^{\infty}_G(\Q, G_{\alpha}^{\mathrm{o}})$. Also, in the second equality, we used the exactness of the sequence \cite[Cor.1.8.6]{Labesse99}
\[ H^1(\Q,I_0)\rightarrow H^1(\A,I_0) \rightarrow \mathfrak{K}(I_0/\Q)^D.\]

Therefore, by Thm. \ref{thm:Kottwitz_formula:Kisin}, Lemma \ref{eq:|widetilde{Sha}_G(Q,I_{phi,epsilon})^+|} and the formula (\ref{eq:Tamagawa_number}), we have the equalities
\begin{align*}
&\mathrm{tr}(\Phi^m\times f^p | H_c(Sh_{K/\Qb},\sF_K) \\
=&\sum_{\gamma_0} \frac{\tau(I_0)}{\mathrm{vol}\cdot |\pi_0(G_{\gamma_0})(\Q)|} \sum_{\substack{(\alpha^p,\alpha_p)\in \\ \mathfrak{D}(I_0,G;\A_f)_1}}  \frac{i(\gamma_0;\gamma,\delta)}{|\mathfrak{D}(\gamma_0;\alpha^p,\alpha_p)|} \cdot \mathrm{O}_{\gamma}(f^p)\cdot \mathrm{TO}_{\delta}(\phi_p) \cdot \mathrm{tr}\xi(\gamma_0) \\
=&\sum_{\gamma_0} \frac{\tau(I_0)}{\mathrm{vol}\cdot |\pi_0(G_{\gamma_0})(\Q)|} \sum_{\substack{(\alpha^p,\alpha_p)\in\\ \mathfrak{D}(I_0,G;\A_f)_1}} |\ker[\ker^1(I_0)\rightarrow \ker^1(G)]| \cdot \mathrm{O}_{\gamma}(f^p)\cdot \mathrm{TO}_{\delta}(\phi_p) \cdot \mathrm{tr}\xi(\gamma_0),
\end{align*}
where $\mathrm{vol}:=\mathrm{vol}(A_G(\R)^{\mathrm{o}}\backslash I_0(\R))$ (recall that $\phi_p$ is the characteristic function of the set (\ref{eq:Adm_K(mu)}), cf. Lemma \ref{lem:fixed-pt_subset_of_Frob-Hecke_corr}).

Let us write $\mathrm{O}_{\alpha^p}(f^p)$ for $\mathrm{O}_{\gamma}(f^p)$ when $\gamma=i_{\ast}(\alpha^p)$, and similarly for $\mathrm{TO}_{\alpha_p}(\phi_p)$.
Then, using the relation \cite[p.395]{Kottwitz86}
\[  \tau(I_0) \cdot |\ker[\ker^1(I_0)\rightarrow \ker^1(G)]| =\tau(G) \cdot |\mathfrak{K}(I_0/\Q)|, \]
and a standard argument (from Fourier analysis on finite groups $ \mathfrak{K}(I_0/\Q)$), we see that the last line equals the first line expression of the next:
\begin{alignat}{2} \label{eq:stablization_step2}
 & \mathrm{tr}(\Phi^m\times f^p | H_c(Sh_{K/\Qb},\sF_K)) && \\
\stackrel{(a)}{=} &\ \tau(G) \sum_{\gamma_0} |\pi_0(G_{\gamma_0})(\Q)|^{-1}\cdot  \frac{\mathrm{tr}\xi(\gamma_0)}{\mathrm{vol}} \ \cdot && \sum_{(\alpha_v)_v}  \sum_{\kappa\in \mathfrak{K}(I_0/\Q)}  \prod_v \langle\alpha_v,\kappa\rangle e(I_{\alpha_v})\cdot \mathrm{O}_{\alpha^p}(f^p)\cdot \mathrm{TO}_{\alpha_p}(\phi_p) \nonumber \\
\stackrel{(b)}{=} &\ \tau(G) \sum_{\gamma_0} |\pi_0(G_{\gamma_0})(\Q)|^{-1}\cdot  \frac{\mathrm{tr}\xi(\gamma_0)}{\mathrm{vol}} \ \cdot && \sum_{\kappa\in \mathfrak{K}(I_0/\Q)} \prod_v\sum_{\alpha_v}  \langle\alpha_v,\kappa\rangle e(I_{\alpha_v})\cdot \mathrm{O}_{\alpha^p}(f^p)\cdot \mathrm{TO}_{\alpha_p}(\phi_p) \nonumber \\
\stackrel{}{=} &\ \tau(G) \sum_{\gamma_0} |\pi_0(G_{\gamma_0})(\Q)|^{-1}\cdot  \sum_{\kappa\in \mathfrak{K}(I_0/\Q)} &&  \biggl[ \biggl( \frac{\mathrm{tr}\xi(\gamma_0)}{\mathrm{vol}} \langle\alpha_{\infty},\kappa\rangle e(I_{\alpha_{\infty}}) \biggl)  \cdot  \prod_{v\neq p,\infty} \biggl( \sum_{\alpha_v} \langle\alpha_v,\kappa\rangle e(I_{\alpha_v})\mathrm{O}_{\alpha^p}(f^p) \biggl) \cdot \nonumber \\
& &&  \qquad \qquad \quad \biggl( \sum_{\alpha_p} \langle\alpha_p,\kappa\rangle e(I_{\alpha_p}) \mathrm{TO}_{\alpha_p}(\phi_p) \biggl) \biggl] \nonumber \\
\stackrel{}{=} &\ \tau(G) \sum_{(\gamma_0, \kappa)} |\pi_0(G_{\gamma_0})(\Q)|^{-1}\cdot  O^{\kappa}_{\gamma_0}(f), && \nonumber
\end{alignat}
where we write $O^{\kappa}_{\gamma_0}(f)$ for the summand that is indexed by $(\gamma_0,\kappa)$ and appears inside the bracket $[ - ]$, i.e. $(\frac{\mathrm{tr}\xi(\gamma_0)}{\mathrm{vol}} \langle\alpha_{\infty},\kappa\rangle e(I_{\alpha_{\infty}})) \cdot  \prod_{v\neq p,\infty}  ( \sum_{\alpha_v} \langle\alpha_v,\kappa\rangle e(I_{\alpha_v}) \mathrm{O}_{\alpha^p}(f^p) ) \cdot ( \sum_{\alpha_p} \langle\alpha_p,\kappa\rangle e(I_{\alpha_p}) \mathrm{TO}_{\alpha_p}(\phi_p) )$.

In (a) the first sum is over the subset of $\R$-elliptic elements in $E_{\mathrm{st}}(G)$, and in the second summation index $(\alpha_v)_v$, $\alpha_{\infty}$ is fixed to be the element $\alpha_{\infty}(\gamma_0)$ defined by $\gamma_0$ in \autoref{subsubsec:Kottwitz_invariant}, while $(\alpha^p,\alpha_p)$ runs through the entire set $\mathfrak{D}(I_0,G;\A_f)$. Also, $e(I_{\alpha_v})$ is the Kottwitz sign of the inner twist $I_{\alpha_v}$ of $I_0$ via $\alpha_v$ (cf. \cite{Kottwitz86}): by definition of $\alpha_{\infty}(\gamma_0)$, $I_{\alpha_{\infty}}$ is the inner form of $(I_0)_{\R}$ that is compact modulo the center. Here we used the fact that $\prod e(I_{\alpha_v})=1$ when the family $\{I_{\alpha_v}\}_v$ of local groups comes from a global group.
The equality (b) comes from rearranging the summations and the product. This is possible since there are only finitely many non-zero terms in the summand for any given summation index $(\gamma_0,\kappa)$, cf. \cite[$\S$7, $\S$8]{Kottwitz86}.

Next, we rewrite the last expression of  (\ref{eq:stablization_step2}) using the new index set that comprises pairs $((H,\mathcal{H},s,\xi),\gamma_{H})$ of an elliptic endoscopic datum and a $(G,H)$-regular semisimple element of $H(\Q)$. For any such pair $((H,\mathcal{H},s,\xi),\gamma_{H})$, we write 
\begin{equation} \label{eq:(H,gamma_H)->(gamma_0,kappa)}
((H,\mathcal{H},s,\xi),\gamma_{H}) \rightarrow (\gamma_0,\kappa)
\end{equation} 
if $\gamma_{H}$ transfers to an element $\gamma_0$ of $G(\Q)$ and there exists $z\in Z(\hat{G})$ such that $sz\in \bigcap_v Z(\hat{I}_0)^{\Gamma_v}Z(\hat{G})$  (via $Z(\hat{H})\hookrightarrow Z(\hat{I}_H)\isom Z(\hat{I}_0)$) and $\kappa$ is its image in $\mathfrak{K}(I_0/\Q)=\left(\bigcap_v Z(\hat{I}_0)^{\Gamma_v}Z(\hat{G})\right)/Z(\hat{G})$: recall that $\gamma_0$ is determined up to stable conjugacy (cf. \cite[6.8]{Kottwitz86}). 

Following Labesse (cf. \cite[$\S$IV.3]{Labesse04}), we call two pairs $(\gamma_0,\kappa)\in E_{\mathrm{st}}(G)\times \mathfrak{K}(I_0/\Q)$ and $(\gamma_0',\kappa')\in E_{\mathrm{st}}(G)\times \mathfrak{K}(I_0'/\Q)$ \emph{equivalent} (or \emph{isomorphic}) if there exists $g\in G(\Qb)$ such that (a) $g\gamma_0g^{-1}=\gamma_0'$ and $g^{-1}\cdot{}^{\tau}g\in I_0$ for all $\tau\in\Gamma$, and that (b) $\kappa$, $\kappa'$ correspond under the isomorphism $\Int(g)^D:\mathfrak{K}(I_0'/\Q) \isom \mathfrak{K}(I_0/\Q)$ (under the condition (a), $\Int(g)$ induces a quasi-isomorphism $I_{0\bfab}\isom I_{0\bfab}'$ of complexes of $\Q$-groups and we have $\mathfrak{K}(I_0/\Q)=H^0(\A/\Q,I_{0\bfab}\rightarrow G_{\bfab})^D$, cf. Remark \ref{eq:K_Labesse_Kottwitz}). In this case, for any maximal $\Q$-tori $T\subset I_0$, $T'\subset I_0'$, one can even find $g\in G(\Qb)$ such that 
\begin{equation} \label{eq:stable_conj_g}
 \Int(g)(T)=T',\quad \Int(g)(\gamma_0)=\gamma_0',\quad \text{and}\quad g^{-1}\cdot{}^{\tau}g\in I_0\text{ for all }\ \tau\in\Gamma
\end{equation}
(The triples $(T,\gamma_0,\kappa)$, $(T',\gamma_0',\kappa')$ are \emph{admissible triples} which are \emph{equivalent}, in the sense of \cite[p.519]{Labesse04}.)
Indeed, choose $g_1\in G(\Qb)$ such that $\Int(g_1)(\gamma_0)=\gamma_0'$, $g_1^{-1}\cdot{}^{\tau}g_1\in I_0$ for all $\tau\in\Gamma$. Then, as $T$ and $\Int(g_1^{-1})(T')$ are maximal tori of $I_0$, there exists $g_0\in I_0(\Qb)$ with $\Int(g_0)(T)=\Int(g_1^{-1})(T')$; then, $g:=g_1g_0\in G(\Qb)$ fulfills the conditions.
Also, note that if $G_{\gamma_0}=I_0$, $(\gamma_0,\kappa)$ can be equivalent to $(\gamma_0,\kappa')$ only if $\kappa=\kappa'$, because such automorphism $\Int(g)^D$ of $\mathfrak{K}(I_0/\Q)$ then must be the identity (thus, one has no use for this notion when $G^{\der}=G^{\der}$).
We also define two pairs $(\underline{H},\gamma_H) \in \mathscr{E}_{\mathrm{ell}}(G)\times H(\Q)$, $(\underline{H}',\gamma_{H'})\in \mathscr{E}_{\mathrm{ell}}(G)\times H'(\Q)$ to be \emph{equivalent} (or \emph{isomorphic}) if there exists an isomorphism $\underline{H}\isom \underline{H}'$ matching the stable conjugacy classes of $\gamma_{H}$ and $\gamma_{H'}$. 
The association (\ref{eq:(H,gamma_H)->(gamma_0,kappa)}) preserves the equivalence relations on both sides.

For $\underline{H}=(H,\mathcal{H},s,\xi)\in \mathscr{E}_{\mathrm{ell}}(G)$ and a semisimple $\gamma_H\in H(\Q)$, let $\mathrm{Aut}(\underline{H},\gamma_H)$ and $\mathrm{Out}(\underline{H},\gamma_H)$ denote respectively the subgroup of $\mathrm{Aut}(\underline{H})$ of elements sending $\gamma_H$ into a stable conjugate of itself and the quotient of $\mathrm{Aut}(\underline{H},\gamma_H)$ by the subgroup of inner automorphisms. Then, we have an equality:
\begin{equation} \label{eq:Out(H)&Out(H,gamma_H)}
|\mathrm{Out}(\underline{H})| =m(\underline{H},\gamma_H) |\mathrm{Out}(\underline{H},\gamma_H)|,
\end{equation}
where $m(\underline{H},\gamma_H)$ is the number of the stable conjugacy classes in $H(\Q)$ in the orbit of $\gamma_H$ under $\mathrm{Aut}(\underline{H})$.

The following lemma is proved in \cite[$\S$IV.3]{Labesse04} in the greater generality of \emph{twisted} endoscopy and also generalizes Lemma 9.7 of \cite{Kottwitz86} which assumes that $G^{\der}=G^{\uc}$.
We present Labesse's proof in our simple setting of standard endoscopy which then benefits from economy of notation and preparatory discussion.

%%%%%%%%%%%%%%%%%%%%
\begin{lem} \label{lem:Labesse04,IV.3}
(1) For every $(\gamma_0,\kappa)\in E_{\mathrm{st}}(G)\times \mathfrak{K}(I_0/\Q)$, there exists a pair $(\underline{H}=(H,\mathcal{H},s,\xi),\gamma_H)$ of an elliptic endoscopic datum and a $(G,H)$-regular, semisimple element such that $(\underline{H},\gamma_{H})\rightarrow (\gamma_0,\kappa)$.

(2) Suppose that $(\underline{H},\gamma_{H})\rightarrow (\gamma_0,\kappa)$ and $(\underline{H}',\gamma_{H'})\rightarrow  (\gamma_0',\kappa')$. If $(\gamma_0,\kappa)$ and $(\gamma_0',\kappa')$ are equivalent, so are $(\underline{H},\gamma_{H})$ and $(\underline{H}',\gamma_{H'})$. 

(3) If $n(\gamma_0,\kappa)$ denotes the number of elements $\kappa'\in\mathfrak{K}(I_0/\Q)$ such that
$(\gamma_0,\kappa')$ is equivalent to $(\gamma_0,\kappa)$, we have equalities
\begin{align*}
 |\pi_0(G_{\gamma_0})(\Q)| &=n(\gamma_0,\kappa) \cdot |\mathrm{Out}(\underline{H},\gamma_H)|\cdot |\pi_0(H_{\gamma_H})(\Q)|  \\
 &=n(\gamma_0,\kappa)\cdot |\mathrm{Out}(\underline{H})|\cdot m(\underline{H},\gamma_H)^{-1}\cdot |\pi_0(H_{\gamma_H})(\Q)|.
 \end{align*}
\end{lem}

\begin{proof}
(1) This is \cite[Lemma 9.7]{Kottwitz86}, where $G$ is assumed to be quasi-split, but recall that every maximal $\Q$-torus of $G$ transfers to its quasi-split inner form.

(2) This is \cite[Prop. IV.3.4]{Labesse04}. This is also proved in \cite[Lemma 9.7]{Kottwitz86} under the assumption $G^{\der}=G^{\uc}$, but the same proof can be adapted for the general case in the following way. 
By definition, there exist an admissible embedding $j:T_H\isom T\subset G$ (of maximal tori) sending $\gamma_H\in T_H(\Q)$ to $\gamma_0\in T(\Q)$ and a similar one $j':T_{H'}\isom T'\subset G$. Choose $g\in G(\Qb)$ satisfying the conditions (\ref{eq:stable_conj_g}).
Then, one can show (as in the proof of \cite[Lemma 9.7]{Kottwitz86}) that there exists a $\Qb$-isomorphism $\alpha_0:H\rightarrow H'$ extending $j'^{-1}\circ\Int(g)\circ j:T_H\isom T_{H'}$ and there exists $h\in H'(\Qb)$ such that $\alpha:=\Int(h)\circ\alpha_0$ is a $\Q$-isomorphism.
Now, if we identify $T$ and $T'$ with $T_H$ and $T_{H'}$ via $j$ and $j'$ respectively, since the composite $\Int(h)\circ \Int(g):T_H\isom T_{H'} \isom \alpha(T_{H})$ is $\Q$-rational, it follows that $g^{-1}(h^{-1}\cdot {}^{\tau}h){}^{\tau}g\in T_H(\Qb)$ for every $\tau\in\Gamma$, which implies that $h^{-1}\cdot{}^{\tau}h\in I_{H'}(\Qb)$ for all $\tau\in\Gamma$: more precisely, $j$ and $j'$ can be extended to $\Qb$-isomorphisms $J:I_H:=(H_{\gamma_H})^{\mathrm{o}}\isom I_0$, $J':I_{H'}:=(H'_{\gamma_{H'}})^{\mathrm{o}}\isom I_0'$ (which are inner twists), and $\alpha_0$ can be constructed extending $J'^{-1}\circ \Int(g)\circ J : I_H\isom I_0\isom I_0'\isom I_{H'}$, cf. \cite[$\S$3.1]{Kottwitz86}. Then, $``g^{-1}(h^{-1}\cdot {}^{\tau}h){}^{\tau}g"=\alpha_0^{-1}(h^{-1}\cdot {}^{\tau}h)\cdot J^{-1}(g^{-1}\cdot {}^{\tau}g)$.
Therefore, $\alpha(\gamma_H)=\Int(h)(\gamma_{H'})$ and $\gamma_{H'}$ are stably conjugate.

(3) This is \cite[Prop. IV.3.5]{Labesse04}.
If $\mathrm{Norm}_{G_{\gamma_0}}(T)$ and $\mathrm{Norm}_{H_{\gamma_H}}(T_H)$ denote the normalizer $\Q$-group schemes, it is known \cite[Lemma2.4]{Brion15} that $T=\mathrm{Cent}_{G_{\gamma_0}^{\mathrm{o}}}(T)= (\mathrm{Norm}_{G_{\gamma_0}}(T))^{\mathrm{o}}$ and there exists an exact sequence of $\Q$-group schemes
\[ 1\rightarrow \Omega(T,I_0) \rightarrow \Omega(T,G_{\gamma_0})\rightarrow \pi_0(G_{\gamma_0}) \rightarrow 1,\]
where $\Omega(T,I_0):=\mathrm{Norm}_{I_0}(T)/T$ and $\Omega(T,G_{\gamma_0}):=\mathrm{Norm}_{G_{\gamma_0}}(T)/T$ are the quotient (Weyl) group schemes.
Let us identify the Weyl group scheme $\Omega(T_H,H)$ with a subgroup scheme of $\Omega(T,G)$ via $j$.  Then, (via the same identification) $\Omega(T_H,I_H)=\Omega(T,I_0)$ as $\Q$-schemes. Indeed, they are isomorphic over $\Qb$, and in fact there exists an inner twist $J:I_H\isom I_0$ which extends the $\Q$-isomorphism $j:T_H\isom T$. Hence, we may find a cochain in $C^1(\Q,T_H)$ whose image in $Z^1(\Q,T_H/Z(I_H))$ gives such inner twist, and it follows that the induced map $\alpha:\Omega(T_H,I_H)\isom \Omega(T,I)$ is $\Q$-rational. 

Now, the group $\pi_0(G_{\gamma_0})(\Q)$ acts on $\mathfrak{K}(I_0/\Q)$ in a natural manner: if we fix a maximal $\Q$-torus $T\subset G$ containing $\gamma_0$, we have seen above that any $\omega\in \Omega(T,G_{\gamma_0})(\Qb)$ whose image $\bar{\omega}\in \pi_0(G_{\gamma_0})(\Qb)$ lies in $\pi_0(G_{\gamma_0})(\Q)$ acts on $\mathfrak{K}(I_0/\Q)$. As $\Omega(T,I_0)(\Qb)$ acts trivially, this induces an action of $\pi_0(G_{\gamma_0})(\Q)$ on $\mathfrak{K}(I_0/\Q)$.
So, we obtain the equality
\begin{equation} \label{eq:pi_0_acts_on_K(I_0)}
 |\pi_0(G_{\gamma_0})(\Q)|=n(\gamma_0,\kappa)\cdot |\pi_0(G_{\gamma_0})(\Q)_{\kappa}|,
\end{equation}
where $\pi_0(G_{\gamma_0})(\Q)_{\kappa}$ is the stabilizer subgroup of $\kappa$. 

Next, we construct a group homomorphism 
\[ \mathrm{Aut}(\underline{H},\gamma_H) \rightarrow \pi_0(G_{\gamma_0})(\Q),\]
whose kernel and image are $\Omega(T_H,I_H)=\Omega(T,I_0)$ and $\pi_0(G_{\gamma_0})(\Q)_{\kappa}$, respectively. For any $\alpha\in \mathrm{Aut}(\underline{H},\gamma_H)$, by the same construction of $g$ in (\ref{eq:stable_conj_g}), we can find $h\in H(\Qb)$ such that $\alpha_0:=\Int(h)\circ\alpha$ satisfies that $\alpha_0(T_H)=T_H$, $\alpha_0(\gamma_H)=\gamma_H$, and that $a_{\tau}:=h\cdot{}^{\tau}h^{-1}\in I_H=(H_{\gamma_H})^{\mathrm{o}}$ for all $\tau\in\Gamma$; thus $a_{\tau}\in \mathrm{Norm}_{I_H}(T_H)(\Qb)$. Then, the restriction of $\alpha_0$ to $T_H$ is given by an element of $\Omega(T,G)(\Qb)$, hence determines an element $\omega_{\alpha}\in \Omega(T,G_{\gamma_0})(\Qb)$. We claim that the image $\bar{\omega}_{\alpha}$ of $\omega_{\alpha}$ in $\pi_0(G_{\gamma_0})(\Qb)$ lies in the subgroup $\pi_0(G_{\gamma_0})(\Q)$. Indeed, let $n\in \mathrm{Norm}_{G_{\gamma_0}}(T)(\Qb)$ be a representative of $\omega_{\alpha}$. So, $\Int(n)|_{T}\circ j=j\circ (\Int(h)\circ\alpha)|_{T_H}$ which implies that for any $\tau\in\Gamma$, $\Int(b_{\tau} {}^{\tau}n)|_T=\Int(n)|_T$, where $b_{\tau}=J(a_{\tau})\in \mathrm{Nom}_{I_0}(T)(\Qb)$, and thus $n^{-1}b_{\tau} {}^{\tau}n\in T$, i.e. $ {}^{\tau}n\equiv n\mod \mathrm{Nom}_{I_0}(T)(\Qb)$, proving the claim.
One easily verifies that the map $\mathrm{Aut}(\underline{H},\gamma_H) \rightarrow \pi_0(G_{\gamma_0})(\Q)$ sending $\alpha$ to $\bar{\omega}_{\alpha}$ is a homomorphism having the described kernel and image. Then, since the inner automorphisms in $\mathrm{Aut}(\underline{H},\gamma_H)$ are induced by $w\in \Omega(T_H,H_{\gamma_H})$ (whose image in $\pi_0(H_{\gamma_H})(\Qb)$ lies in $\pi_0(H_{\gamma_H})(\Q)$), we have the relation
\begin{equation} \label{eq:Out(H,gamma_H)&pi_0}
 |\pi_0(G_{\gamma_0})(\Q)_{\kappa}|= |\mathrm{Out}(\underline{H},\gamma_H)|\cdot |\pi_0(H_{\gamma_H})(\Q)|. 
\end{equation}
The two equations (\ref{eq:pi_0_acts_on_K(I_0)}), (\ref{eq:Out(H,gamma_H)&pi_0}) combined give the first equality of (2). The second equality follows from it and (\ref{eq:Out(H)&Out(H,gamma_H)}). \end{proof}

Let $f^{H_1,p}\in C^{\infty}_{c,\lambda_{H_1}}(H_1(\A_f^p))$ and $f^{H_1}_p\in C^{\infty}_{c,\lambda_{H_1}}(H_1(\Qp))$ be respectively the transfers of our functions $f^p\in C^{\infty}_c(G(\A_f^p))$ and $\phi_p\in C^{\infty}_c(G(\Qp))$ given by Thm. \ref{thm:untwisted_endoscopy_transfer} (for all $v\neq p,\infty$) and Thm. \ref{thm:twisted_endoscopy_transfer}, and let $f^{H_1}_{\infty}$ be the function as in Thm. \ref{thm:pseudo-coeff}.
Put $f^{H_1}:=f^{H_1,p}f^{H_1}_pf^{H_1}_{\infty}\in  C^{\infty}_{c,\lambda_{H_1}}(H_1(\A))$. 
Then, we have that
\begin{itemize}
\item[(i)] The global transfer factor $\Delta=\prod \Delta_v$ can be normalized such that $\Delta(\gamma_{H_1},\gamma_0)$ is zero unless $\gamma_{H_1}$ transfers to $\gamma_0$, in which case $\Delta(\gamma_{H_1},\gamma_0)=1$: this is the global hypothesis for transfer factors which is verified in \cite[$\S$6.4]{LS87}.
\item[(ii)]  $\mathrm{SO}_{\gamma_{H_1}}(f^{H_1})=0$ unless $\gamma_{H_1}$ is elliptic in $H_1(\R)$ and transfers to $G(\Qv)$ for every place $v$ of $\Q$ (This is due to Thm. \ref{thm:untwisted_endoscopy_transfer}, Thm. \ref{thm:twisted_endoscopy_transfer}, and Thm. \ref{thm:pseudo-coeff} and since we assume that an (elliptic) maximal torus of $H_{\R}$ transfers to an elliptic maximal torus of $G_{\R}$).
If such two conditions hold, $\gamma_{H_1}$ transfers to $G(\Q)$ (\cite{Kottwitz90}, second paragraph on p.188).
\end{itemize}

It follows that for any pair $(\gamma_0,\kappa)\in E_{\mathrm{st}}(G)\times \mathfrak{K}(I_0/\Q)$,
if $(H_1,\gamma_{H_1})\rightarrow (\gamma_0,\kappa)$, by (i) (i.e. $\Delta(\gamma_{H_1},\gamma_0)=1$), we have
\[ O^{\kappa}_{\gamma_0}(f)=\mathrm{SO}^{H_1}_{\gamma_{H_1}}(f^{H_1}), \]
and 
since $O^{\kappa}_{\gamma_0}(f)$ depends only on the equivalence class of the pair $(\gamma_0,\kappa)$), 
the last expression of (\ref{eq:stablization_step2}) is equal to
\begin{align*}
& \sum_{(\gamma_0,\kappa)/\sim} \tau(G) n(\gamma_0,\kappa) |\pi_0(G_{\gamma_0})(\Q)|^{-1} O^{\kappa}_{\gamma_0}(f) \\
\stackrel{(c)}{=}&  \sum_{(\underline{H},\gamma_H)/\sim} \tau(G) |\mathrm{Out}(\underline{H})|^{-1} m(\underline{H},\gamma_H) |\pi_0(H_{\gamma_H})(\Q)|^{-1} \mathrm{SO}^{H_1}_{\gamma_{H_1}}(f) \\
=& \sum_{\underline{H}\in\mathscr{E}_{\mathrm{ell}}(G)} \iota(G,\underline{H})\tau(H)  \sum_{\gamma_{H}\in E_{\mathrm{st}}(H)} |\pi_0(H_{\gamma_H})(\Q)|^{-1} \mathrm{SO}^{H_1}_{\gamma_{H_1}}(f^{H_1}) =\sum_{\underline{H}\in\mathscr{E}_{\mathrm{ell}}(G)} \iota(G,\underline{H}) \ST_{\elp}^{H_1}(f^{H_1}).
\end{align*}
Here, in the first line, the summation index runs through a set of representatives of the \emph{equivalence classes} of pairs $(\gamma_0,\kappa)$ of an $\R$-elliptic element in $E_{\mathrm{st}}(G)$ and an element of the associated group $\mathfrak{K}(-/\Q)$. In the second line, the summation index runs through a set of representatives of the \emph{equivalence classes} of pairs $(\underline{H}=(H,\mathcal{H},s,\xi),\gamma_H)\in\mathscr{E}_{\mathrm{ell}}(G)\times  E_{\mathrm{st}}(H)$ and for each $\gamma_{H}\in E_{\mathrm{st}}(H)$ we choose an (elliptic) element $\gamma_{H_1}\in H_1(\Q)$ mapping to $\gamma_H$. The equality (c) follows from (ii) and Lemma \ref{lem:Labesse04,IV.3}, together with the facts that for any $\alpha\in \mathrm{\Aut}(\underline{H})$ and $\gamma_H':=\alpha(\gamma_H)$, one has $|\pi_0(H_{\gamma_H'})(\Q)|=|\pi_0(H_{\gamma_H})(\Q)|$ and that $\gamma_0$ is elliptic in $G(\R)$ if and only if $\gamma_H$ is elliptic in $H(\R)$ and some (elliptic) maximal torus of $H_{\R}$ transfers to an elliptic maximal torus of $G_{\R}$.

This completes the proof.
\end{proof}

%%%%%%%%%%%%%%%%%%%%%%%%%%%%%%%%%%%%%%%% 
%%%%%%%%%%%%%%%%%%%%%%%%%%%%%%%%%%%%%%%%

\begin{appendix}

%%%%%%%%%%%%%%%%%%%%%%%%%%%%%%%%%%%%%%%%
%%%%%%%%%%%%%%%%%%%%%%%%%%%%%%%%%%%%%%%%
\section{Existence of elliptic tori inside special maximal parahoric group schemes}  \label{sec:elliptic_tori_in_special_parahorics}

%%%%%%%%%%%%%%%%%%%%
\begin{prop} \label{prop:existence_of_elliptic_tori_in_special_parahorics}
Let $k$ be a local field with residue characteristic not equal to $2$ and $\mfk$ the completion of the maximal unramified extension $k^{\nr}$ in an algebraic closure $\overline{k}$ of $k$ with respective rings of integers $\cO_k$ and $\cO_{\mfk}$. Let $G$ be a connected reductive group over $k$ of classical Lie type. Assume that $G$ is quasi-split and that $G^{\uc}$ is a product $\prod_i \Res_{F_i/\Qp}G_i$ of simple groups each of which is the restriction of scalars $\Res_{F_i/\Qp}G_i$ of an absolutely simple group $G_i$ over a finite extension $F_i$ of $k$ such that $G_i$ splits over a tamely ramified extension of $F_i$. 

Then, for any special parahoric subgroup $K$ of $G(k)$, there exists a maximal elliptic $k$-torus $T$ of $G$ such that $T_{k^{\nr}}$ contains (equiv. is the centralizer of) a maximal ($k^{\nr}$-)split $k^{\nr}$-torus $S_1$ of $G_{k^{\nr}}$ and that the unique parahoric subgroup of $T(\mfk)$ is contained in $\tilde{K}$, the parahoric subgroup of $G(\mfk)$ corresponding to $K$.
\end{prop}

\begin{rem} \label{rem:properties_of_certain_elliptic_tori_in_special_parahorics}
(1) The unique parahoric subgroup of $T(\mfk)$ (resp. of $T(k)$) is $T(\mfk)_1:=\ker(w_{T_L})$ (resp. $T(k)_1:=\ker(w_{T_L})\cap T(k)$).

(2) With $S_1$ and $T$ as in the statement, the second property of $T$ can be translated into a statement about the Bruhat-Tits building:
any special point $\mbfv\in \mcB(G,k)$ giving $K$ (i.e. $K=\mathrm{Stab}_{G(k)}(\mbfv)\cap\ker(w_{G_{\mfk}})$) lies in the apartment $\mcA(S_1,\mfk)$ attached to $S_1$.
Indeed, $S_1(\mfk)_1\subset \tilde{K}$ if and only if $\mbfv\in\mcA(S_1,\mfk)$. This is because $\mcA(S_1,\mfk)$ is the full fixed point set of the pararhoric subgroup $S_1(\mfk)_1$, which in turn is due to \cite[3.6.1]{Tits79} since every relative root $a$ of the root datum $(G_{\mfk},(S_1)_{\mfk})$, being a non-trivial character (of a split torus), satisfies that $a(S_1(\mfk)_1)\nsubseteq 1+\pi_F$ (i.e. $\bar{a}(\overline{S_1(\mfk)_1})\neq 1$), as the residue field of $\cO_{\mfk}$ is infinite. But, $S_1(\mfk)_1\subset\tilde{K}$ if and only if $T(\mfk)_1\subset\tilde{K}$, since $T(\mfk)=Z_{G_{\mfk}}(S_1)(\mfk)$ acts on $\mcA(S_1,\mfk)$ via $w_{T_{\mfk}}\otimes\R:T(\mfk)\rightarrow X_{\ast}(T)_{\Gal(\overline{\mfk}/\mfk)}\otimes\R=X_{\ast}(S_1)_{\R}$ \cite[1.2.(1)]{Tits79}, so if $S_1(\mfk)_1\subset\tilde{K}$, $T(\mfk)_1$ fixes $\mbfv(\in \mcA(S_1,\mfk))$.

%(3) As will be evident from the proof, in fact the condition on splitting field of $G$ can be relaxed to the same condition as in Thm. \ref{thm:LR-Satz5.3}, namely to that $G^{\uc}$ is a product $\prod_i \Res_{F_i/\Qp}G_i$ of simple groups each of which is the restriction of scalars $\Res_{F_i/\Qp}G_i$ of an absolutely simple group $G_i$ over a finite extension $F_i$ of $k$ such that $G_i$ splits over a tamely ramified extension of $F_i$.
\end{rem}

\begin{proof} 
We reduce a general case to the case of absolutely almost-simple groups. First, under the isomorphism
\[\mcB(G,k)\simeq \mcB(G^{\der},k)\times X_{\ast}(A(G))_{\R},\] 
where $A(G)$ is the maximal split $F$-torus in the center $Z(G)$, a special point in $\mcA(G,k)$ corresponds to $(v,x)$ for a special vertex $v$ in $\mcB(G^{\der})$ and a point $x\in X_{\ast}(A(G))_{\R}$. This implies that we may assume that $G$ is semisimple, and by the same reasoning further that $G$ is simply-connected.
Then as $G(=G^{\der})$ is a product of almost-simple groups of the same kind (by which we mean quasi-split, tamely ramified, classical groups), we may also assume that $G$ is almost-simple. Hence, $G=\Res_{F/k}(H)$ for an absolutely-(almost)simple, quasi-split, classical, tamely ramified, semi-simple group $H$ over a finite extension $F$ of $k$. 
The building $\mcB(G,k)$ (resp. $\mcB(G,\mfk)$) is canonically isomorphic to $\mcB(H,F)$ (resp. $\prod_{\sigma\in\Hom_k(F_0,\mfk)}\mcB(H,F\otimes_{F_0,\sigma}\mfk))$, where $F_0$ is the maximal absolutely unramified subextension of $F$); this is true for any finite separable extension $F$ of $k$ \cite[2.1]{Tits79}. 
This shows that the claim for $G$ follows from the claim for $H$. Let $L$ and $\cO_L$ be the completion of the maximal unramified extension $F^{\nr}$ in an algebraic closure $\overline{F}$ of $F$ and its ring of integers.

From this point, we prove the proposition case by case. We use the complete list of isomorphism classes of quasi-split, tamely ramified, classical groups over local fields as provided in \cite{Gross12}, where Gross gives a complete list of isomorphism classes of (not necessarily quasi-split or classical) tamely ramified groups over local fields and it is fairly easy to determine the quasi-split, classical ones from that list. There are totally ten such isomorphism classes, among which the first seven are unramified ones (including four split ones). 

\begin{itemize}
\item[$A_m$]$(m\geq1)$: $H=\mathrm{SL}_m$ (split group).
\item[${}^2A'_{2m}$]$(m\geq2)$: Let $E$ be the unramified quadratic extension of $F$ and let $W$ be a non-degenerate Hermitian space of odd rank $n=2m+1$ over $E$ (its Witt-index must be $m$). Then $H=\SU(W)$ (non-split unramified group).
\item[${}^2A'_{2m-1}$]$(m\geq2)$: Let $E$ be the unramified quadratic extension of $F$ and let $W$ be a non-degenerate Hermitian space of even rank $n=2m$ over $E$ which contains an isotropic subspace of dimension $m$. Then $H=\SU(W)$ (non-split unramified group).
\item[$B_m$]$(m\geq3)$: Let $W$ be a non-degenerate orthogonal space of odd dimension $2m+1$ over $F$ which contains an isotropic subspace of dimension $n$. Then $H=\Spin(W)$ (split group).
\item[$C_m$]$(m\geq2)$: Let $W$ be a non-degenerate symplectic space of dimension $2m$ over $F$. Then $H=\Sp(W)$ (split group).
\item[$D_m$]$(m\geq4)$: Let $W$ be a non-degenerate orthogonal space of dimension $2m$ over $F$ which contains a (maximal) isotropic subspace of dimension $m$. Then, the center of the Clifford algebra is the split \'etale quadratic extension $E=F\oplus F$ of $F$, and $H=\Spin(W)$ is a split group.
\item[${}^2D_m$]$(m\geq4)$: Let $W$ be a non-degenerate orthogonal space of even dimension $2m$ over $F$ where the center of the Clifford algebra is the unramified quadratic extension $E$ of $F$. Then $H=\Spin(W)$ (non-split unramified group).
\item[$B\operatorname{-}C_m$]$(m\geq3)$: Let $E$ be a tamely ramified quadratic extension of $F$ and let $W$ be a non-degenerate Hermitian space of even rank $n=2m$ over $E$ which contains an isotropic subspace of dimension $m$. Then $H = SU(W)$ (ramified group)
\item[$C\operatorname{-}BC_m$]$(m\geq2)$: Let $E$ be a tamely ramified quadratic extension of $F$ and let $W$ be a non-degenerate Hermitian space of odd rank $n=2m+1$ over $E$. Then $H = SU(W)$ (ramified group).
\item[$C\operatorname{-}B_m$]$(m\geq2)$: Let $W$ be a non-degenerate orthogonal space of even dimension $2m$ over $F$ where the center of the Clifford algebra is a tamely ramified quadratic extension $E$ of $F$. Then $H=\Spin(W)$ (ramified group).
\end{itemize}

Here, each heading is the name for the corresponding isomorphism class that is used in \cite{Tits79}, tables 4.2 and 4.3. 

Now, when the group is unramified and given special vertex is hyperspecial, the claim is known \cite{Lee16}, Appendix 1.0.4). For convenience, we split these cases into three kinds. The first kind consists of the split groups, thus, all special vertices are automatically hyperspecial: these are $A_m$, $B_m$, $C_m$, $D_m$. The second case is when there are two special vertices and both of them are hyperspecial: they consist of ${}^2A'_{2m-1}$, ${}^2D_m$. The remaining ones constitute the last case, i.e. there is some non-hyperspecial, special vertex.
So, we only need to take care of the last case: 
\[{}^2A'_{2m},\quad B\operatorname{-}C_m,\quad C\operatorname{-}BC_m,\ \text{ and }\ C\operatorname{-}B_m.\] 
Note that except for the last one, all these are (special) unitary groups.

(1) First, we treat the special unitary groups of \emph{even} absolute rank (i.e. $H_{\overline{F}}\simeq \SU(2m+1)$ for an algebraic closure $\overline{F}$ of $F$). We will reduce the proof in this case to the special unitary groups of \emph{odd} absolute rank.
For a moment, we let $E$ be an arbitrary quadratic extension of $F$ with respective rings of integers $\cO_E$, $\cO_F$ (we assume that the residue characteristic of $\cO_F$ is not $2$). We choose a uniformizer $\pi$ of $\cO_E$ such that $\pi+\overline{\pi}=0$ for the non-trivial automorphism $\overline{\cdot}$ of $E/F$.
Let $(W,\phi:W\times W\rightarrow E)$ be a non-degenerate Hermitian space of dimension $n=2m+1\ (m\geq1)$. As is well-known, $\psi$ has maximal Witt-index $m$, so there exists a Witt basis $\{e_{-m},\cdots,e_{m}\}$, i.e. such that 
\[\phi(e_i,e_j)=\delta_{i,-j},\quad\text{for }-m\leq i,j\leq m.\] 
For $i=0,\cdots,m$, we define an $\cO_E$-lattice in $W$:
\[\Lambda_i:=\mathrm{span}_{\cO_E}\{\pi^{-1}e_{-m},\cdots,\pi^{-1}e_{-i-1},e_{-i},\cdots,e_m\}.\]
(here, $\Lambda_{m}=\mathrm{span}_{\cO_E}\{e_{-m},\cdots,\cdots,e_m\}$.)

Set $H=\SU(W,\phi)$ (algebraic group over $F$). Its (minimal) splitting field is $E$. For a non-empty subset $I$ of $\{0,\cdots,m\}$, we consider the subgroup of $H(F)$ 
\[P_I:=\{g\in\SU(W,\phi)\ |\ g\Lambda_i\subset\Lambda_i,\ \forall i\in I\}.\]
For each $i=0,\cdots,m$, $P_{\{ i\}}$ (stabilizer of a single lattice $\Lambda_i$) is also the stabilizer of a point $v_i$ of the apartment $\mcA$ of a maximal $F$-split torus $S$ of $H$, which in turn can be matched with the $(i+1)$-th vertex of the local Dynkin diagram \cite[1.15, 3.11]{Tits79}. 

From this correspondence and the information found in \textit{loc. cit.} (1.15 (9), 4.3 in the unramified case, and 1.15 (10), 4.2 in the ramified case), we deduce the following facts.

\textit{Every $P_I$ is a parahoric subgroup of $\SU(W,\phi)$ and any parahoric subgroup of $\SU(W,\phi)$ is conjugate to $P_I$ for a unique subset $I$. If $E$ is unramified (i.e. of type ${}^2A'_{2m}$), there are two special vertices, one hyperspecial and one non-hyperspecial. The group $P_{\{0\}}$ (resp. $P_{\{m\}}$) is the non-hyperspecial, special (resp. the hyperspecial) parahoric subgroup. If $E$ is ramified (i.e. of  type $C\operatorname{-}BC_m$), there are two special vertices, both non-hyperspecial, which correspond to $I=\{0\}$ and $I=\{m\}$. The corresponding parahoric subgroup $\tilde{P}_{\{i\}}$ of $H(L)$ is 
\[ \tilde{P}_{\{i\}}=\{g\in\SU(W,\phi)(L)\ |\ g\Lambda_i\otimes_{\cO_F}\cO_L\subset\Lambda_i\otimes_{\cO_F}\cO_L\}.\]}
The statement on $\tilde{P}_{\{i\}}$ follows from the fact that the stabilizer $\cO_L$-group scheme $\mcG_i$ of the vertex $v_i$ (defined by Bruhat-Tits \cite[3.4.1]{Tits79}) equals the $\cO_L$-structure on $H_L$ induced by the lattice \cite[3.11]{Tits79}, and the characterization of parahoric groups given by Haines-Rapoport  \cite{PappasRapoport08} ($\SU(W,\phi)$ being a simply-connnected semi-simple group, the Kottwitz homomorphism $w_{\SU(W,\phi)_L}$ is trivial).

(a) The case $\mathbf{{}^2A'_{2m}}$: 
Suppose that $E$ is unramified. Then, we only need to consider the non-hyperspecial, special parahoric subgroup $P_{\{0\}}$. The Hermitian space $(W,\phi)$ splits as the direct sum of two Hermitian subspaces, that is, totally isotropic subspace and anisotriopic subspace: 
\[(W,\psi)=(W_{\mathbf{iso}},\phi_{\mathbf{iso}})\oplus (W_{\mathbf{an}},\phi_{\mathbf{an}}),\]
where
\[ W_{\mathbf{iso}}:=\langle e_l\ |\ l\neq 0\rangle,\quad W_{\mathbf{an}}:=E\cdot e_{0},\]
and $\phi_{\mathbf{iso}}=\phi|_{W_{\mathbf{iso}}}$ and $\phi_{\mathbf{an}}=\phi|_{W_{\mathbf{an}}}$.
There is the corresponding lattice decomposition 
\[\Lambda_{\{0\}}=\Lambda_{\{0\}}'\oplus \Lambda_{\{0\}}'',\] 
where $\Lambda_{\{0\}}':=\mathrm{span}_{\cO_E}\{\pi^{-1}e_{-m},\cdots,\pi^{-1}e_{-1},e_{1},\cdots,e_{m}\}$ and $\Lambda_{\{0\}}'':=\cO_E\cdot e_{0}$. 
Using this decomposition, we reduce the construction of the torus looked for into construction of similar tori for the groups $\SU(W_{\bullet},\psi_{\bullet})$ ($\bullet=\mathbf{iso}$, $\mathbf{an}$).
Let us write for short $\SU_{\bullet}$ and $\mrU_{\bullet}$ for $\SU(W_{\bullet},\psi_{\bullet})$ and $\mrU(W_{\bullet},\psi_{\bullet})$ respectively ($\bullet=\mathbf{iso}$, $\mathbf{an}$).
Suppose that $T_{\mathbf{iso}}$ is an $F$-torus of $\SU_{\mathbf{iso}}$ with the property in question and let $Z_{\mathbf{iso}}$ be the (connected) center of $\mrU_{\mathbf{iso}}$ (so that the subgroup $T_{\mathbf{iso}}\cdot Z_{\mathbf{iso}}$ generated by the two groups is a maximal torus of $U_{\mathbf{iso}}$, among others). This center is isomorphic to the anisotpropic $F$-torus $\underline{E}^{\times}$, whose set of $R$-points, for an $F$-algebra $R$, is 
\[\underline{E}^{\times}_c(R):=\ker(\Nm_{E/F}:(E\otimes R)^{\times}\rightarrow (F\otimes R)^{\times}).\] 
This is also identified in a natural way with the group $\mrU_{\mathbf{an}}$. We claim that 
\[T:=S((T_{\mathbf{iso}} \cdot Z_{\mathbf{iso}}) \times \mrU_{\mathbf{an}})=(T_{\mathbf{iso}}\times\{1\})\cdot \underline{E}^{\times}_c,\] 
where $S(-)$ means the intersection of the group inside the parenthesis (subgroup of $\mrU(W,\phi)$) with $\SU(W,\phi)$ and $\underline{E}^{\times}_c$ is identified with $S(Z_{\mathbf{iso}} \times \mrU_{\mathbf{an}})$ via $x\mapsto (x,x^{-2m})$, is a maximal torus of $H=\SU(W,\psi)$ with the same required properties. First, clearly this is anisotropic, and $T_{E}$ is a split maximal torus of $H_E$.
Next, we verify that $T(L)_1$ maps into the parahoric subgroup $\tilde{P}_{\{0\}}$ of $H(L)$. By the description of $\tilde{P}_{\{0\}}$ above, we have to show that $T(L)_1$ leaves stable $\Lambda_{\{0\}}\otimes_{\cO_F}\cO_L$. But, $T(L)_1=T_{\mathbf{iso}}(L)_1\cdot \underline{E}^{\times}_c(L)_1$, and $\underline{E}^{\times}_c(L)_1$ acts on $\Lambda_{\{0\}}'\otimes\cO_L\oplus\Lambda_{\{0\}}''\otimes\cO_L$ through the map $x\mapsto (x,x^{-2m})$ above. So, clearly it suffices to checks that $\underline{E}^{\times}_c(L)_1$ leaves stable each rank-$1$ lattice $\cO_L\otimes_{\cO_F}(\cO_E\cdot e_i)$. But in the case $E$ is unramified over $F$, we have the equality $\underline{E}^{\times}_c(L)_1=\{(x,x^{-1})\in \cO_L^{\times}\times \cO_L^{\times}\}$ under the isomorphism $(E\otimes L)^{\times}=L^{\times}\times L^{\times}$, thus leaves stable $\cO_L\otimes_{\cO_F}(\cO_E\cdot e_i)=(\cO_L\oplus \cO_L)\cdot e_i$. 
Next, the fact that $T_{\mathbf{iso}}(L)_1$ leaves stable $\Lambda_{\{0\}}'\otimes\cO_L$ will be one of the defining properties of the torus $T_{\mathbf{iso}}$. Indeed, 
$\SU(W_{\mathbf{iso}},\psi_{\mathbf{iso}})$ is a group of type ${}^2A_{2m-1}'$ in the above list, and the stabilizer $P_{\{0\}}'$ of the lattice $\Lambda_{\{0\}}'$ is a hyperspecial subgroup of $\SU(W_{\mathbf{iso}},\psi_{\mathbf{iso}})$ (cf. \cite[4.3]{Tits79}). So, we already know that there exists an elliptic maximal $F$-torus $T_{\mathbf{iso}}$ of $\SU_{\mathbf{iso}}$ such that $T_{\mathbf{iso}}(L)_1$ is contained in the parahoric subgroup $\tilde{P}_{\{0\}}'$ of $\SU_{\mathbf{iso}}(L)$ corresponding to $P_{\{0\}}'$ which is the stabilizer in $\SU_{\mathbf{iso}}(L)$ of the lattice $\Lambda_{\{0\}}'\otimes\cO_L$. This finishes the proof in the case $E$ is unramified over $F$.

(b) The case $\mathbf{C\operatorname{-}BC_m}$: 
When $E$ is ramified, there are two cases: $I=\{0\}$ and $I=\{m\}$. 
Then, the same strategy just used  (i.e. for unramified unitary groups of odd absolute rank) works again, reducing proof to the ramified unitary groups of odd absolute rank $2m-1$ (the type $\mathbf{B\operatorname{-}C_m}$), which will be discussed next. We just note that in this case with the same notations as above,  the (minimal) splitting fields of $H$, $\SU_{\mathbf{iso}}$, $Z_{\mathbf{iso}}$, and $\mrU_{\mathbf{an}}$ are all $E$, and that $(\underline{E}^{\times}_c)_L$ is anisotropic, so $\underline{E}^{\times}_c(L)$ is its own parahoric subgroup and, being a subgroup of $(\cO_E\otimes_{\cO_F}\cO_L)^{\times}$, leaves stable $\cO_L\otimes_{\cO_F}(\cO_E\cdot e_i)$. 

(2) Let $E$, $W$, and $n$ be as in the previous description (1)-(b), except that the parity of $n$ is even (i fact, it can be arbitrary for a moment). Let $\cO_E$ and $\cO_F$ be the integer rings of $E$ and $F$, respectively. We fix uniformizers $\pi_F$, $\pi=\pi_E$ of $\cO_F$ and $\cO_E$ such that $\pi^2=\pi_F$ (so again $\pi+\overline{\pi}=0$). Let $\phi:W\times W\rightarrow E$ be a non-degenerate Hermitian form and put $H=\SU(W,\phi)$. Again we assume (forced by the quasi-split condition, in the even dimensional case) that $\psi$ has maximal Witt-index, namely when one writes $n=2m$ (or $n=2m+1$), it is $m$. As $E$ is ramified over $F$, the rank of $H$ is the same as that of $H_L$.
 
Following \cite{PappasRapoport08}, $\S$4, we use a different indexing in the coming discussion.
Choose a Witt basis $\{e_{1},\cdots,e_{n}\}$ such that $\phi(e_i,e_j)=\delta_{i,n+1-j}$ for $1\leq i,j\leq n$. 

Suppose that $n=2m$. For $i\in\{1,\cdots,m-2\}\cup\{m\}$, we define an $\cO_E$-lattice $\Lambda_i$:
\[\Lambda_i:=\mathrm{span}_{\cO_E}\{\pi^{-1}e_{1},\cdots,\pi^{-1}e_{i},e_{i+1},\cdots,e_n\},\]
In the place of $i=m-1$, we introduce a new lattice $\Lambda_{m'}$ defined by:
\[\Lambda_{m'}:=\mathrm{span}_{\cO_E}\{\pi^{-1}e_{1},\cdots,\pi^{-1}e_{m-1},e_{m},\pi^{-1}e_{m+1},e_{m+2},\cdots,e_{n}\}.\]
Here, $m'$ is regarded as a symbol like other numbers.

Set $H:=\SU(W,\phi)$. For a non-empty subset $I$ of $\{1,\cdots,m-2,m',m\}$, the associated stabilizer subgroup $P_I$ has the same definition as in the previous case.

When $n=2m$, the group $\SU(W,\phi)$ has the local Dynkin diagram $B\operatorname{-}C_m$ for $m\geq3$, and $C\operatorname{-}B_2$ for $m=2$ (for $m=1$, $\SU(W,\psi)\simeq\mathrm{SL}_2$). 
Then, we have a similar statement \cite{PappasRapoport08}, $\S$4), namely that

\textit{the subgroup $P_I$ is a parahoric subgroup of $\SU(W,\phi)$ and any parahoric subgroup of $\SU(W,\phi)$ is conjugate to $P_I$ for a unique subset $I$, and the special maximal parahoric subgroups are $P_{\{m\}}$, $P_{\{m'\}}$. The same description is true for the parahoric subgroup $\tilde{P}_I$ of $\SU(W,\phi)(L)$ associated with $P_I$.}

(c) The case $\mathbf{B\operatorname{-}C_m}$: We have $n=2m$. First, let us consider the case $I=\{m\}$. 
The Hermitian space $W$ is the direct sum of $m$ hyperbolic subspaces 
\[\mathbb{H}_i:=E\langle e_i,e_{n+1-i}\rangle\subset W\ (i=1,\cdots,m).\] 
Then, we claim that when we identify $\SU(\mathbb{H}_i)$ with $\SU(\mathbb{H}_i)\times\mathrm{1}_{\oplus_{j\neq i}\mathbb{H}_j}\subset\SU(W,\phi)$, 
\[\SU(\mathbb{H}_i)(F)\cap P_{\{m\}}=\{g\in \SU(\mathbb{H}_i)(F)\ |\ g(\mathbb{H}_i\cap \Lambda_{\{m\}})=\mathbb{H}_i\cap \Lambda_{\{m\}}\}\] 
is a special maximal parahoric subgroup of $\SU(\mathbb{H}_i)(F)\simeq\mathrm{SL}_2(F)$ (recall that there are two $\mathrm{SL}_2(F)$-conjugacy classes of special parahoric subgroups of $\mathrm{SL}_2(F)$, which are however conjugate under $\mathrm{GL}_2(F)$). This can be proved, e.g. using an explicit isomorphism between $\SU(\mathbb{H}_i)\simeq\mathrm{SL}_{2,F}$, one such being
\[ g=\left(\begin{array}{cc}
a& b\\ c& d
\end{array}\right) \mapsto \left(\begin{array}{cc}
a& \pi^{-1}b\\ \pi c& d
\end{array}\right)\]
(check that when $g(e_i)=ae_i+ce_{n+1-i}, g(e_{n+1-i})=be_i+de_{n+1-i}$, $\phi(gv,gw)=\phi(v,w)$ implies that $a,d\in F^{\times}$, $b+\overline{b}=c+\overline{c}=0$). But, there is another (rather indirect) way of seeing this. Let $S_i$ be a maximal ($F$-)split $F$-subtorus of $\SU(\mathbb{H}_i)$ (so that $S:=\prod_i S_I$ is a maximal ($F$-)split $F$-torus of $\SU(W,\phi)$, and also of $\mrU(W,\phi)$); $S_i$ is contained in a unique maximal torus $T_i(=Z_{\mrU(\mathbb{H}_i)}(S_i))\simeq E^{\times}$ of $\mrU(\mathbb{H}_i)$. The subgroup 
\[M_i:=\mrU(\mathbb{H}_i)\times\prod_{j\neq i}T_j\] 
of $\mrU(W,\phi)$, being the centralizer of $\{1\}\times\prod_{j\neq i}S_j$, is an $F$-Levi subgroup of $\mrU(W,\phi)$. 
For a subgroup $M$ of $\mrU(W,\phi)$, let $\mathrm{S}M$ denote the intersection $M\cap \SU(W,\phi)$. Then, as $\{1\}\times\prod_{j\neq i}S_j\subset\SU(W)$, $\mathrm{S}M_i$ is an $F$-Levi subgroup of $\SU(W)$.
Hence, by Lemma \ref{lem:specaial_parahoric_in_Levi}, (3), $\mathrm{S}M_i(F)\cap P_{\{m\}}$ is a special maximal parahoric subgroup of $M_i(F)$. But, obviously one has that (as $P_{\{m\}}\subset\SU(W)(F)$)
\[\mathrm{S}M_i(F)\cap P_{\{m\}}=(\mrU(\mathbb{H}_i)(F)\times\prod_{j\neq i}T_j(F))\cap P_{\{m\}}
=(\mrU(\mathbb{H}_i)(F)\cap P_{\{m\}})\times \prod_{j\neq i}(T_j(F)\cap P_{\{m\}}). \]
So, each $\SU(\mathbb{H}_i)(F)\cap P_{\{m\}}=\mrU(\mathbb{H}_i)(F)\cap P_{\{m\}}$ is a special maximal parahoric subgroup of $\SU(\mathbb{H}_i)(F)\simeq\mathrm{SL}_2(F)$. 

But, the two special vertices in the local Dynkin diagram of $\mathrm{SL}_{2,F}$ are hyperspecial, hence we know that there exists an anisotropic maximal torus $S_i'$ of $\SU(\mathbb{H}_i)$, splitting over $F^{\nr}$, such that the unique parahoric subgroup $S_1'(L)_1$ of $S_i'(L)$ is contained in $\SU(\mathbb{H}_i)(L)\cap \tilde{P}_{\{m\}}$, i.e. leaves stable the $\cO_E\otimes_{\cO_F}\cO_L$-lattice $(\mathbb{H}_i\cap \Lambda_{\{m\}})\otimes_{\cO_F}\cO_L=\mathrm{span}_{\cO_E\otimes_{\cO_F}\cO_L}\{\pi^{-i}e_{i},e_{n+1-i}\}$ of $\mathbb{H}_i\otimes L$. 
Therefore, for the center $Z_i\simeq \underline{E}^{\times}_c$ of $\mrU(\mathbb{H}_i)$,
\[T_i':=S_i'\cdot Z_i\] 
is an anisotropic maximal torus of $\mrU(\mathbb{H}_i)$, whose $L$-rank equals $1$ and whose group of $L$-points also leaves stable the $\cO_E\otimes_{\cO_F}\cO_L$-lattice $(\mathbb{H}_i\cap \Lambda_{\{m\}})_{\cO_L}$, as $Z_i$ remains anisotropic over $L$ and $Z_i(L)=(\cO_E\otimes_{\cO_F}\cO_L)_1$. Finally, the torus $T:=(\prod_iT_i')\cap\SU(W,\phi)$ is an anisotropic maximal torus of $\SU(W,\phi)$ with the same required property (its $L$-rank is $m$, equal to the $L$-rank of $H_L$).

The case $I=\{m'\}$ can be treated in a completely analogous way, once we switch the basis vectors $e_{m}$ and $e_{m+1}$; although such permutation does not lie in $\SU(W,\psi)$, obviously it is allowed when applying the previous argument. 

(d) We are left with the case $\mathbf{C\operatorname{-}B_{m+1}\ (m\geq1)}$. This is also similar to the above cases.
Let $E$ be a (ramified) quadratic extension of $F$ and $W=F^{\oplus m}\oplus E\oplus F^{\oplus m}$, viewed as a vector space over $F$. We consider the quadratic form on $W$ expressed in terms of a basis $\{e_{-m},\cdots,e_{m}\}\cup\{e_0\}$ by
\[q(\sum_{1\leq|i|\leq m} x_i e_i+x_0e_0)=\sum_{i=1}^{m}x_{-i} x_i+\Nm_{E/F}x_0,\quad (x_i\in F,x_0\in E).\] 
For $i=0,\cdots,m$, we define a lattice $\Lambda_i$ as before:
\[\Lambda_i:=\mathrm{span}_{\cO_F}\{\pi^{-1}e_{-m},\cdots,\pi^{-1}e_{-i-1},e_{-i},\cdots,\check{e}_0,\cdots,e_m\}\oplus\cO_Ee_0,\]
where $\check{e}_0$ means as usual that it is omitted from the list
(so, $\Lambda_{0}=\pi^{-1}\mathrm{span}_{\cO_F}\{e_{-m},\cdots,e_{-1}\}\oplus\cO_Ee_0\oplus\mathrm{span}_{\cO_F}\{e_{1},\cdots,e_{m}\}$).

It is obvious that the claim at hand holds for the special orthogonal group $\mathrm{SO}(W,q)$ if and only if it does so for the universal covering of $\mathrm{SO}(W,q)$, i.e. the spin group $\mathrm{Spin}(W,q)$. We will show that for any special maximal parahoric subgroup $K$ of $\mathrm{SO}(W,q)$, there exists an anisotropic maximal $F$-torus $T$ such that $T_L$ contains a maximal ($L$-)split $L$-torus of $\mathrm{SO}(W,q)$ and $T(L)_1$ is contained in $\tilde{K}$, the special parahoric subgroup of $\mathrm{SO}(W,q)(L)$ corresponding to $K$. Note that the $L$-rank of $H$ is $m$ as $E$ is ramified over $F$.

For a non-empty subset $I$ of $\{0,\cdots,m\}$, let $P_I$ denote the stabilizer subgroup:
\[P_I:=\{g\in \mathrm{SO}(W,q)(F)\ |\ g\Lambda_i\subset\Lambda_i,\ \forall i\in I\}.\]
We know (deduced from \cite{Tits79}, 1.16, 4.2, cf. 3.12) that

\textit{the subgroup $P_I$ is a parahoric subgroup of $\SU(W,\phi)$ and that any parahoric subgroup of $\SU(W,\phi)$ is conjugate to $P_I$ for a unique subset $I$, and that the special maximal parahoric subgroups are $P_{\{0\}}$ and $P_{\{m\}}$. The corresponding parahoric subgroup $\tilde{P}_{\{i\}}$ of $H(L)$ is 
\[ \tilde{P}_{\{i\}}=\{g\in\SO(W,\phi)(L)\ |\ g\Lambda_i\otimes_{\cO_F}\cO_L\subset\Lambda_i\otimes_{\cO_F}\cO_L\}.\]} 

The idea used above for the ramified special unitary group of odd relative rank (i.e. of type $\mathbf{B\operatorname{-}C_{m}}$) works here, too. Namely, the quadratic space $W$ decomposes into the direct sum of maximally isotropic subspaces $W_{\mathbf{iso}}=F\langle e_i\ |\ 1\leq|i|\leq m\rangle$ and the anisotropic subspace $(E\cdot e_0,\Nm_{E/F})$. Then, $\SO(W_{\mathbf{iso}})$ is a split group, so for each $j=0,m$, the parahoric subgroup $P_{\{j\}}\cap \SO(W_{\mathbf{iso}})(F)$ is a a hyperspecial subgroup. Hence, there exists an anisotropic maximal torus $T'_j$ of $\SO(W_{\mathbf{iso}})$ of $L$-rank $m$, such that $T'_j(L)_1$ is contained in $\SO(W_{\mathbf{iso}})(L)\cap \tilde{P}_{\{j\}}$; the latter means that $T'_j(L)_1$ leaves stable the $\cO_L$-lattice
\[(W_{\mathbf{iso}}\cap \Lambda_{j})\otimes\cO_L\]
of $W_{\mathbf{iso}}\otimes L$. Now, it is easy to see that the anisotropic torus $T_j:=T'_j$ is a maximal torus of $\SO(W,q)$ (which also has the $L$-rank $m$) with the same properties for $\Lambda_{j}$. 

This completes the proof of the proposition. \end{proof}

%%%%%%%%%%%%%%%%%%%%%%%%%%%%%%%%%%%%%%%%
%%%%%%%%%%%%%%%%%%%%%%%%%%%%%%%%%%%%%%%%
\section{Complexes of tori attached to connected reductive groups.} \label{sec:abelianization_complex}
Here we collect some general facts on certain complexes of tori attached to a connected reductive group and its Levi subgroups.

For a connected reductive group $H$ over a field $k$, we let $\rho_H:H^{\uc}\rightarrow H$ denote the canonical homomorphism ($H^{\uc}$ being the simply connected covering of $H^{\der}$) and choosing a maximal $k$-torus $T$ of $H$, define a two-term complex of $k$-tori by
\[H_{\bfab}:=\rho_H^{-1}(T)\rightarrow T,\]
where $\rho_H^{-1}(T)$ and $T$ are located in degree $-1$ and $0$ respectively, i.e. $H_{\bfab}$ is the mapping cone of the morphism $\rho_H^{-1}(T)\rightarrow T$ in the abelian cateogry $\mathcal{CG}_k$ of commutative algebraic $k$-group schemes.%%
\footnote{Do not confuse this with the torus $H^{\ab}:=H/H^{\der}$; of course, they become quasi-isomorphic when $H^{\der}=H^{\uc}$. If stated otherwise, every two-term complex will be concentrated in degree $-1$ and $0$.}
This complex of tori is also quasi-isomorphic to the crossed module $H^{\uc}\rightarrow H$ (again with $H$ being placed at degree $0$), and the corresponding object in the (bounded) derived category $\mathbb{D}^b(\mathcal{CG}_k)$ depends only on $H$.

Let $G$ be a connected reductive group over a field $k$ and $I$ be a $k$-subgroup of $G$ which is a $\bar{k}$-Levi subgroup.
For $\tilde{I}=\rho_G^{-1}(I)$ (a connected reductive group), there exists a map $\tilde{i}:I^{\uc}\rightarrow \tilde{I}$:
\[\xymatrix{ I^{\uc} \ar[r]^{\tilde{i}} \ar[rd]_{\rho_I} & \tilde{I} \ar@{^(->}[r] \ar[d]^{\rho_G} & G^{\uc} \ar[d]^{\rho_G} \\ & I \ar[r]^{i} & G. }\]
We choose a maximal $k$-torus $T$ of $I$ and set $T_1:=\rho_G^{-1}(T)$, $T_2:=\rho_I^{-1}(T)$; we have a commutative diagram
\[\xymatrix{ T_2 \ar[r]^{\tilde{i}} \ar[rd]_{\rho_I} & T_1 \ar[d]^{\rho_G} \\ & T }\]
The verification of the following facts are easy and thus are left to readers.
\begin{itemize}
\item[(a)] The complex $T_2\rightarrow T_1$ is quasi-isomorphic to the abelianization complex $\tilde{I}_{\bfab}:\rho_{\tilde{I}}^{-1}(T_1)\rightarrow T_1$ attached to $\tilde{I}$, where $T_1$ is located at degree $0$, and thus $i:I\hookrightarrow G$ gives rise, in a canonical manner, to a distinguished triangle (in $\mathbb{D}^b(\mathcal{CG}_k)$):
\begin{equation} \label{eq:DT_of_CX_of_tori}
\tilde{I}_{\bfab} \rightarrow I_{\bfab}  \rightarrow G_{\bfab}  \rightarrow \tilde{I}_{\bfab}[1].
\end{equation}
\item[(b)] For a two-term complex $T_{\bullet}=T_{-1}\rightarrow T_0$ of $k$-tori, let $\hat{T}_{\bullet}$ denote the complex of $\C$-tori with $\Gamma$-action:
\[\hat{T}_{\bullet}:=(\hat{T}_0\rightarrow \hat{T}_{-1}),\] 
where $\hat{T}_0$ and $\hat{T}_{-1}$ are located in degree $-1$ and $0$ respectively.
Then, for any connected reductive group $H$ and a maximal $k$-torus $T$ of it, there exists an exact sequence of diagonalizable $\C$-groups with $\Gamma$-action
\[1\rightarrow Z(\hat{H}) \rightarrow \hat{T} \rightarrow \hat{\tilde{T}} \rightarrow 1,\] 
with $\tilde{T}:=\rho_H^{-1}(T)$ and $\Gamma=\Gal(\bar{k}/k)$, i.e. 
\begin{equation} \label{eq:center_of_complex_dual}
Z(\hat{H})[1]=\hat{H}_{\bfab}.
\end{equation}
\item[(c)] The distinguished triangle (\ref{eq:DT_of_CX_of_tori}) gives rise, by taking complex dual, to a distinguished triangle in the derived category $\mathbb{D}^b(\operatorname{\Gamma-\mathcal{DG}_{\C}})$ of diagonalizable $\C$-groups with $\Gamma$-action:
\begin{align} \label{eq:DT_of_dualCX_of_tori}
\hat{G}_{\bfab} \rightarrow \hat{I}_{\bfab} \rightarrow \hat{\tilde{I}}_{\bfab} \rightarrow \hat{G}_{\bfab}[1]. 
\end{align}
\end{itemize}

\end{appendix}

\textit{Email:} machhama@gmail.com


\begin{thebibliography}{99}


\bibitem[BT65]{BT65}
A. Borel, J. Tits.
\newblock Groupes r\'eductifs. 
\newblock Inst. Hautes \'Etudes Sci. Publ. Math. No. 27 1965 55-150.


\bibitem[Bri15]{Brion15}
M. Brion.
\newblock On extensions of algebraic groups with finite quotient. 
\newblock Pacific J. Math. 279 (2015), no. 1-2, 135-153.


\bibitem[Bre94]{Breen94}
L. Breen.
\newblock Tannakian categories. 
\newblock Motives (Seattle, WA, 1991), 337-376, Proc. Sympos. Pure Math., 55, Part 1, Amer. Math. Soc., Providence, RI, 1994.


\bibitem[Bor98]{Borovoi98}
M. Borovoi. 
\newblock Abelian Galois cohomology of reductive groups. 
\newblock Mem. Amer. Math. Soc. 132 (1998), no. 626, viii+50 pp.


\bibitem[BT72]{BT72}
F. Bruhat, J. Tits.
\newblock Groupes r\'eductifs sur un corps local. I. 
\newblock Inst. Hautes \'Etudes Sci. Publ. Math. No. 41 (1972), 5-251.


\bibitem[BT84]{BT84}
F. Bruhat, J. Tits.
\newblock Groupes r\'eductifs sur un corps local. II. Sch\'emas en groupes. Existence d'une donn\'ee radicielle valu\'ee. 
\newblock Inst. Hautes \'Etudes Sci. Publ. Math. No. 60 (1984), 197-376.


\bibitem[CCO14]{CCO14}
C.-L. Chai, B. Conrad, F. Oort. 
\newblock Complex multiplication and lifting problems. 
\newblock Mathematical Surveys and Monographs, 195. American Mathematical Society, Providence, RI, 2014. x+387 pp.


\bibitem[Clo90]{Clozel90}
L. Clozel.
\newblock The fundamental lemma for stable base change. 
\newblock Duke Math. J. 61 (1990), no. 1, 255-302.


\bibitem[Del79]{Deligne79}
P. Deligne.
\newblock Vari\'{e}t\'{e}s de Shimura: interpr\'{e}tation modulaire, et techniques de construction de mod\`{e}les canoniques.
\newblock Automorphic forms, representations and $L$-functions (Proc. Sympos. Pure Math., Oregon State Univ., Corvallis, Ore., 1977), Part 2, 247-289, Proc. Sympos. Pure Math., XXXIII, Amer. Math. Soc., Providence, R.I., 1979.


\bibitem[DMOS82]{DMOS82}
P. Deligne, J.S. Milne, A. Ogus, K-y. Shih.
\newblock Hodge cycles, motives, and Shimura varieties. 
\newblock Lecture Notes in Mathematics, 900. Springer-Verlag, Berlin-New York, 1982.


\bibitem[Dem11]{Demarche11}
C. Demarche. 
\newblock Suites de Poitou-Tate pour les complexes de tores \`a deux termes. 
\newblock Int. Math. Res. Not. IMRN 2011, no. 1, 135-174.


\bibitem[Fuj97]{Fujiwara97}
K. Fujiwara. 
\newblock Rigid geometry, Lefschetz-Verdier trace formula and Deligne's conjecture. 
\newblock Invent. Math. 127 (1997), no. 3, 489-533.


\bibitem[Gir71]{Giraud71}
J. Giraud. 
\newblock Cohomologie non ab\'elienne. 
\newblock Die Grundlehren der mathematischen Wissenschaften, Band 179. Springer-Verlag, Berlin-New York, 1971.


\bibitem[Gre63]{Greenberg63}
M. Greenberg. 
\newblock Schemata over local rings, II, 
\newblock Ann. Math. 78 (1963), 256-266.


\bibitem[Gro12]{Gross12}
B. Gross.
\newblock Parahorics.
\newblock {\url{http://www.math.harvard.edu/~gross/preprints/parahorics.pdf}}


\bibitem[HR08]{HainesRapoport08}
T. Haines, M. Rapoport. 
\newblock On parahoric subgroups, Advances in Math. 219 (1), (2008), 188-198. 
\newblock appendix to: G. Pappas, M. Rapoport, Twisted loop groups and their affine flag varieties, Advances in Math. 219 (1), (2008), 118-198. 


\bibitem[HR10]{HainesRostami10}
T. Haines, S. Rostami. 
\newblock The Satake isomorphism for special maximal parahoric Hecke algebras. 
\newblock Represent. Theory 14 (2010), 264-284.


\bibitem[Hai14]{Haines14}
T. Haines.
\newblock 
The stable Bernstein center and test functions for Shimura varieties.
\newblock Automorphic Forms and Galois Representations, Durham, 2011, London Mathematical Society Lecture Notes Series: 415, vol. 2, London Mathematical Society, 2014, pp. 118-186.


\bibitem[He15]{He15}
X. He.
\newblock Kottwitz-Rapoport conjecture on unions of affine Deligne-Lusztig varieties.
\newblock {\url{http://arxiv.org/abs/1408.5838}}.


\bibitem[HR15]{HeRapoport15}
X. He, M. Rapoport.
\newblock Stratifications in the reduction of Shimura varieties.
\newblock {\url{http://arxiv.org/abs/1509.07752}}.


\bibitem[Kis10]{Kisin10}
M. Kisin.
\newblock Integral models for Shimura varieties of abelian type.
\newblock  J. Amer. Math. Soc. 23 (2010), no. 4, 967-1012.


\bibitem[Kis17]{Kisin17}
M. Kisin.
\newblock Mod $p$-points on Shimura varieties of abelian type.
\newblock  J. Amer. Math. Soc. electronically published on January 11, 2017.


\bibitem[KP15]{KisinPappas15}
M. Kisin, G. Pappas.
\newblock Integral models of Shimura varieties with parahoric level structure.
\newblock {\url{http://arxiv.org/abs/1512.01149}}.


\bibitem[Kot82]{Kottwitz82}
R. Kottwitz. 
\newblock Rational conjugacy classes in reductive groups. 
\newblock Duke Math. J. 49 (1982), no. 4, 785-806. 


\bibitem[Kot83]{Kottwitz83}
R. Kottwitz. 
\newblock Sign changes in harmonic analysis on reductive groups. 
\newblock Trans. Amer. Math. Soc. 278 (1983), no. 1, 289-297.


\bibitem[Kot84a]{Kottwitz84a}
R. Kottwitz. 
\newblock Stable Trace Formula: Cuspidal Tempered Terms, 
\newblock Duke Math. J. 51 (1984), 611-650.


\bibitem[Kot84b]{Kottwitz84b}
R. Kottwitz. 
\newblock Shimura Varieties and Twisted Orbital Integrals, 
\newblock Math. Ann. 269 (1984), 287-300.


\bibitem[Kot85]{Kottwitz85}
R. Kottwitz.
\newblock Isocrystals with additional structure.
\newblock Compositio Math. 56 (1985), no. 2, 201-220.


\bibitem[Kot86]{Kottwitz86}
R. Kottwitz, 
\newblock Stable Trace Formula: Elliptic Singular Terms,
\newblock Math. Ann. 275 (1986), 365-399.


\bibitem[Kot88]{Kottwitz88}
R. Kottwitz.
\newblock Tamagawa numbers. 
\newblock Ann. of Math. (2) 127, 1988, 629-646.


\bibitem[Kot90]{Kottwitz90}
R. Kottwitz.
\newblock Shimura varieties and $\lambda$-adic representations. 
\newblock Automorphic forms, Shimura varieties, and L-functions, Vol. I (Ann Arbor, MI, 1988), 161-209, Perspect. Math., 10, Academic Press, Boston, MA, 1990.


\bibitem[Kot92]{Kottwitz92}
R. Kottwitz.
\newblock Points on some Shimura varieties over finite fields.
\newblock  J. Amer. Math. Soc.  5  (1992),  no. 2, 373-444.


\bibitem[Kot97]{Kottwitz97}
R. Kottwitz.
\newblock Isocrystals with additional structure. II.
\newblock Compositio Math. 109 (1997), no. 3, 255-339.


\bibitem[Kot10]{Kottwitz10}
R. Kottwitz.
\newblock Comparison of two versions of twisted transfer factors
\newblock appendix to: 
S. Morel. On the cohomology of certain noncompact Shimura varieties. Annals of Mathematics Studies, 173. Princeton Univ. Press, 2010. xii+217 pp.


\bibitem[KR00]{KR00}
R.  Kottwitz, M. Rapoport. 
\newblock Minuscule alcoves for $\mathrm{GL}_n$ and $\mathrm{GSp}_{2n}$. 
\newblock Manuscripta Math. 102 (2000), no. 4, 403-428.


\bibitem[KS99]{KottwitzShelstad99}
R. Kottwitz, D. Shelstad. 
\newblock Foundations of twisted endoscopy. 
\newblock Ast\'erisque No. 255 (1999), vi+190 pp.


\bibitem[KS12]{KottwitzShelstad12}
R. Kottwitz, D. Shelstad. 
\newblock On splitting invariants and sign conventions in endoscopic transfer.
\newblock {\url{https://arxiv.org/abs/1201.5658}}


\bibitem[Koy90]{Koya90}
Y. Koya. 
\newblock A generalization of class formation by using hypercohomology. 
\newblock Invent. Math. 101 (1990), no. 3, 705-715.


\bibitem[Lab01]{Labesse01}
J.-P. Labesse.
\newblock Nombres de Tamagawa des groupes r\'eductifs quasi-connexes. 
\newblock Manuscripta Math. 104 (2001), no. 4, 407-430.


\bibitem[Lab99]{Labesse99}
J.-P. Labesse. 
\newblock Cohomologie, stabilisation et changement de base. 
\newblock Ast\'erisque No. 257 (1999), vi+161 pp.


\bibitem[Lab04]{Labesse04}
J.-P. Labesse.
\newblock Stable twisted trace formula: elliptic terms. 
\newblock J. Inst. Math. Jussieu 3 (2004), no. 4, 473-530.


\bibitem[Lan76]{Langlands76}
R.P. Langlands, 
\newblock Some contemporary problems with origins in the Jugendtraum, 
\newblock in: Mathematical developments arising from Hilbert's problems, AMS (1976), 401-418.


\bibitem[Lan79]{Langlands79}
R. P. Langlands, 
\newblock On the zeta functions of some simple Shimura varieties. 
\newblock Canad. J. Math. 31 (1979), no. 6, 1121-1216.


\bibitem[Lan83]{Langlands83}
R. P. Langlands, 
\newblock Les d\'ebuts d'une formule des traces stable, 
\newblock Publ. Math. de l'Univ. Paris VII 13, Paris 1983.


\bibitem[LR87]{LR87}
R. P. Langlands, M. Rapoport. 
\newblock Shimuravariet\"aten und Gerben. 
\newblock J. Reine Angew. Math. 378 (1987), 113-220.


\bibitem[LS87]{LS87}
R. P. Langlands, D. Shelstad. 
\newblock On the definition of transfer factors. 
\newblock Math. Ann. 278 (1987), no. 1-4, 219-271.


\bibitem[LS90]{LS90}
R. P. Langlands, D. Shelstad. 
\newblock Descent for transfer factors. 
\newblock The Grothendieck Festschrift, Vol. II, 485-563, Progr. Math., 87, Birkh\"auser Boston, Boston, MA, 1990.


\bibitem[Lee12]{Lee12}
D. U. Lee
\newblock A proof of a conjecture of Y. Morita. 
\newblock Bull. Lond. Math. Soc. 44 (2012), no. 5, 861-870.


\bibitem[Lee16]{Lee16}
D. U. Lee
\newblock Non-emptiness of Newton strata of Shimura varieties of Hodge type.
\newblock {\url{http://arxiv.org/abs/1603.04831}}, a published version to appear in \textit{Algebra Number Theory}.


\bibitem[Mil92]{Milne92}
J. S. Milne. 
\newblock The points on a Shimura variety modulo a prime of good reduction.  
\newblock The zeta functions of Picard modular surfaces, Univ. MontrŽal, Montreal, QC, 1992.


\bibitem[Mil94]{Milne94}
J. S. Milne.
\newblock Shimura varieties and motives.
\newblock Motives (Seattle, WA, 1991),  447--523, Proc. Sympos. Pure Math.,
55, Part 2, Amer. Math. Soc., Providence, RI, 1994.


\bibitem[Mil03]{Milne03}
J. S. Milne.
\newblock Gerbes and abelian motives.
\newblock manuscript available on {\url{http://www.jmilne.org/math/articles/index.html#1990b}}


\bibitem[Mil13]{Milne13}
J. S. Milne.
\newblock Class field theory.
\newblock course notes, version 4.02, 2013, available on {\url{http://www.jmilne.org/math/CourseNotes/CFT.pdf}}


\bibitem[Mor10]{Morel10}
S. Morel. 
\newblock On the cohomology of certain noncompact Shimura varieties. With an appendix by Robert Kottwitz. \newblock Annals of Mathematics Studies, 173. Princeton University Press, Princeton, NJ, 2010. xii+217 pp.


\bibitem[Ngo10]{Ngo10}
B. C. Ng\^o.
\newblock Le lemme fondamental pour les alg\`ebres de Lie.
\newblock Publ. Math. Inst. Hautes \'Etudes Sci. No. 111 (2010), 1-169.


\bibitem[NSW08]{NSW08}
J. Neukirch, A. Schmidt, K. Wingberg. 
\newblock Cohomology of number fields. Second edition. 
\newblock Grundlehren der Mathematischen Wissenschaften, 323. Springer-Verlag, Berlin, 2008.


\bibitem[PR08]{PappasRapoport08}
G. Pappas, M. Rapoport. 
\newblock Twisted loop groups and their affine flag varieties. With an appendix by T. Haines and Rapoport. 
\newblock Adv. Math. 219 (2008), no. 1, 118-198.


\bibitem[PR09]{PappasRapoport09}
G. Pappas, M. Rapoport. 
\newblock Local models in the ramified case. III. Unitary groups. 
\newblock J. Inst. Math. Jussieu 8 (2009), no. 3, 507-564.


\bibitem[PR94]{PR94}
V. Platonov, A. Rapinchuk.
\newblock Algebraic groups and number theory. 
\newblock Pure and Applied Mathematics, 139, Academic Press, Inc., Boston, MA, 1994.


\bibitem[RR96]{RR96}
M. Rapoport, M. Richartz.
\newblock On the classification and specialization of $F$-isocrystals with additional structure.
\newblock Compositio Math.  103  (1996),  no. 2, 153-181.


\bibitem[Rap05]{Rapoport05}
M. Rapoport. 
\newblock A guide to the reduction modulo p of Shimura varieties. 
\newblock Automorphic forms. I. Ast\'erisque No. 298 (2005), 271-318.


\bibitem[Rei97]{Reimann97}
H. Reimann. 
\newblock The semi-simple zeta function of quaternionic Shimura varieties. 
\newblock Lecture Notes in Mathematics, 1657. Springer-Verlag, Berlin, 1997. viii+143.


\bibitem[San81]{Sansuc81}
J.-J. Sansuc. 
\newblock Groupe de Brauer et arithm\'etique des groupes alg\'ebriques lin\'eaires sur un corps de nombres. \newblock J. Reine Angew. Math. 327 (1981), 12-80.


\bibitem[Sat67]{Satake67}
I. Satake. 
\newblock Symplectic representations of algebraic groups satisfying a certain analyticity condition. 
\newblock Acta Math. 117 1967 215-279.


\bibitem[Sat80]{Satake80}
I. Satake. 
\newblock Algebraic structures of symmetric domains. 
\newblock Iwanami Shoten, Tokyo; Princeton University Press, Princeton, N.J., 1980. xvi+321 pp.


\bibitem[Sch13]{Scholze13}
P. Scholze.
\newblock The Langlands-Kottwitz method and deformation spaces of $p$-divisible groups. 
\newblock  J. Amer. Math. Soc. 26 (2013), 227-259.


\bibitem[SS13]{ScholzeShin13}
P. Scholze, S. W. Shin. 
\newblock On the cohomology of compact unitary group Shimura varieties at ramified split places. 
\newblock J. Amer. Math. Soc. 26 (2013), no. 1, 261-294.


\bibitem[Ser02]{Serre02}
J.-P. Serre. 
\newblock Galois cohomology. 
\newblock Translated from the French by Patrick Ion and revised by the author. Corrected reprint of the 1997 English edition. Springer Monographs in Mathematics. Springer-Verlag, Berlin, 2002.


\bibitem[She79]{Shelstad79}
D. Shelstad, 
\newblock Characters and inner forms of a quasisplit group over $\R$, 
\newblock Comp. Math. 39 (1979), 11-45.


\bibitem[She82]{Shelstad82}
D. Shelstad. 
\newblock L-indistinguishability for real groups. 
\newblock Math. Ann. 259 (1982), no. 3, 385-430.


\bibitem[She08]{Shelstad08}
D. Shelstad. 
\newblock Tempered endoscopy for real groups. I. Geometric transfer with canonical factors. 
\newblock Representation theory of real reductive Lie groups, 215-246, Contemp. Math., 472, Amer. Math. Soc., 2008.


\bibitem[She10]{Shelstad10}
D. Shelstad. Tempered endoscopy for real groups. II. Spectral transfer factors. 
\newblock Automorphic forms and the Langlands program, 236-276, Adv. Lect. Math. (ALM), 9, Int. Press, Somerville, MA, 2010.


\bibitem[Ste05]{Stembridge05}
J. R. Stembridge. 
\newblock Tight quotients and double quotients in the Bruhat order. 
\newblock Electron. J. Combin. 11 (2004/06), no. 2.


\bibitem[Tat79]{Tate79}
J. Tate. 
\newblock Number theoretic background. 
\newblock Automorphic forms, representations and L-functions (Corvallis, 1977), Proc. Sympos. Pure Math., XXXIII, Part 2, pp. 3-26, Amer. Math. Soc. 1979.


\bibitem[Tit79]{Tits79}
J. Tits. 
\newblock Reductive groups over local fields. 
\newblock Automorphic forms, representations and L-functions (Proc. Sympos. Pure Math., Oregon State Univ., Corvallis, Ore., 1977), Part 1, pp. 29-69, Proc. Sympos. Pure Math., XXXIII, Amer. Math. Soc., Providence, R.I., 1979.


\bibitem[Var07]{Var07}
Y. Varshavsky. 
\newblock Lefschetz-Verdier trace formula and a generalization of a theorem of Fujiwara. 
\newblock Geom. Funct. Anal. 17 (2007), no. 1, 271-319.


\bibitem[Vas99]{Vasiu99}
A. Vasiu.
\newblock Integral canonical models of Shimura varieties of preabelian type.
\newblock Asian J. Math.  3  (1999), no. 2, 401-518.


\bibitem[Vie15]{Viehmann15}
E. Viehmann.
\newblock On the geometry of the Newton stratification.
\newblock {\url{http://arxiv.org/abs/1511.03156}}.


\bibitem[Wal97]{Waldspurger97}
J.-L. Waldspurger. 
\newblock Le lemme fondamental implique le transfert. 
\newblock Compositio Math. 105 (1997), no. 2, 153-236. 


\bibitem[Wal06]{Waldspurger06}
J.-L. Waldspurger. 
\newblock Endoscopie et changement de caractéristique.
\newblock J. Inst. Math. Jussieu 5 (2006), no. 3, 423-525.


\bibitem[Wal08]{Waldspurger08}
J.-L. Waldspurger.
\newblock L'endoscopie tordue n'est pas si tordue. 
\newblock Mem. Amer. Math. Soc. 194 (2008), no. 908, x+261 pp.


\bibitem[Zho17]{Zhou17}
R. Zhou.
\newblock Mod $p$ isogeny classes on Shimura varieties with parahoric level structure.
\newblock available on {\url{http://www.math.harvard.edu/~rzhou/LR17.pdf}}

\end{thebibliography}
\end{document}